\documentclass[a4paper,10pt]{amsart}
\usepackage{amssymb,amsmath, pinlabel, mathabx, tikz, xcolor, lineno}
\usepackage{kantlipsum} 

\usepackage{enumitem}   
\setlist[description]{labelindent=10pt,style=multiline,leftmargin=2.6cm}

\usetikzlibrary{decorations.pathreplacing}
\usetikzlibrary{shapes.geometric}
\usetikzlibrary{calc,shadows}
\definecolor{darkturquoise}{rgb}{0.0, 0.81, 0.82}
\definecolor{dodgerblue}{rgb}{0.12, 0.56, 1.0}
\definecolor{indiangreen}{rgb}{0.07, 0.53, 0.03}

  \usepackage[pagebackref]{hyperref}
  \hypersetup{
colorlinks=true,
linkcolor=cyan,
citecolor=cyan
}

\usepackage{mathtools}

\input xy 
\xyoption{all}

\usepackage{color}

\newtheorem{theorem}{Theorem}[section]
\newtheorem{lemma}[theorem]{Lemma}
\newtheorem{proposition}[theorem]{Proposition}
\newtheorem{corollary}[theorem]{Corollary}
\newtheorem{algorithm}[theorem]{Algorithm}

\theoremstyle{definition}
\newtheorem{definition}[theorem]{Definition}
\newtheorem{notation}[theorem]{Notations}
\newtheorem{remark}[theorem]{Remark}
\newtheorem{example}[theorem]{Example}

\newcommand{\gras}[1]{{\mathbb #1}} 

\newcommand{\N}{\gras{N}}
\newcommand{\Z}{\gras{Z}} 
\newcommand{\Q}{\gras{Q}} 
\newcommand{\R}{\gras{R}} 
\newcommand{\C}{\mathbb{C}}  
\newcommand{\bS}{\gras{S}}
\newcommand{\bP}{\gras{P}}

\newcommand{\cC}{\mathcal{C}}
\newcommand{\cE}{\mathcal{E}}
\newcommand{\cF}{\mathcal{F}}

\newcommand{\cN}{\mathcal{N}}
\newcommand{\cO}{\mathcal{O}}

\newcommand{\cT}{\mathcal{T}}
\newcommand{\cZ}{\mathcal{Z}}

\newcommand{\de}{\mathbf{i}}
\newcommand{\ex}{\mathbf{e}}
\newcommand{\ic}{\mathbf{c}}

\newcommand{\slp}{\mathbf{S}} 

\newcommand{\Supp}{{\mathcal{S}}}

\newcommand{\fan}{\cF}
\newcommand{\cone}{\R_+}
\newcommand{\Enriques}{\Gamma}

\def\elem(#1,#2){  \{ \frac{#1}  {\overline {\ #2\ }} \} }  

\newlength{\szer}
 \newcommand{\Teissr}[2]{%
\settowidth{\szer}{$\displaystyle\frac{#1}{#2}$}%
\setlength{\szer}{0.5\szer}%
\left\{\hspace{\szer}%
\raisebox{1.5ex}{\makebox[0pt]{$#1$}}%
\raisebox{-1.5ex}{\makebox[0pt]{$#2$}}%
\hspace{-1\szer}\rule[0.4ex]{2\szer}{0.13ex}%
\hspace{-2\szer}\rule[0.7ex]{2\szer}{0.13ex}%
\right\}%
}

\newcommand{\Teisssr}[4]{
  \setlength{\unitlength}{1ex}
  \begin{picture}(#3,3)(0,0.4)
     \put(0,1.15){\line(1,0){#3}}
     \put(0,0.85){\line(1,0){#3}}
     \put(#4,1.3){\makebox(0,0)[b]{$#1$}}
     \put(#4,0.7){\makebox(0,0)[t]{$#2$}}
  \end{picture}}

\title[The combinatorics of plane curve singularities]
{The combinatorics of plane curve singularities \\
   {\small How Newton polygons blossom into lotuses}}
\author{Evelia R. Garc\'{\i}a Barroso}
\address{Departamento de Matem\'aticas, Estad\'{\i}stica e I.O.
Secci\'on de Matem\'aticas, Universidad de La Laguna. Apartado de Correos 456.
38200 La Laguna, Tenerife, Espa\~na.}
   \email{ergarcia@ull.es}
\author{Pedro D. Gonz\'alez P\'erez} 
\address{Instituto de Matem\'atica Interdisciplinar y Departamento de \'Algebra, 
Geometr\' \i a y Topolog\'\i a,  Facultad de Ciencias Matem\'aticas,
Universidad Complutense de Madrid, Plaza de las Ciencias 3, Madrid 28040, Espa\~na.}
   \email{pgonzalez@mat.ucm.es}
\author{Patrick Popescu-Pampu}
   \address{Univ. Lille, CNRS, UMR 8524 - Laboratoire Paul Painlev\'e, F-59000 Lille, France.}
   \email{patrick.popescu-pampu@univ-lille.fr}
\date{22 June 2020}

\subjclass[2010]{14H20 (primary), 14B05, 32S05}
\keywords{Blow ups, Branch, Characteristic exponents, Dual graph, Eggers-Wall tree, 
      Embedded resolution, Enriques diagram,  Intersection numbers, Newton polygon, 
      Newton-Puiseux series, Plane curve singularity, Proximity relation, Resolution of singularities, 
      Toric geometry, Toroidal embedding, Tropicalization, Valuation, Valuative tree.}

     \makeindex

\begin{document}

{\bf To appear in the \emph{Handbook of Geometry and Topology of Singularities I}, Springer, 2020.}

\bigskip
{\em \hfill This paper is dedicated to Bernard Teissier for his 75th birthday. \smallskip}

\begin{abstract}
     This survey may be seen as an introduction to the use of toric and tropical geometry in 
    the analysis of {\em plane curve singularities}, which are germs $(C,o)$ of complex analytic curves 
    contained in a smooth complex analytic surface $S$.
    The {\em embedded topological type} of such a pair  
    $(S, C)$ is usually defined to be that of the oriented link obtained  by intersecting 
    $C$ with a sufficiently small oriented Euclidean sphere centered at the point $o$,  
    defined once a system of local coordinates $(x,y)$ was chosen on the germ $(S,o)$.  
    If one works more generally over an arbitrary algebraically closed field of characteristic zero, 
    one speaks instead of the {\em combinatorial type} of $(S, C)$. 
     One may define it by looking either at the Newton-Puiseux series associated to 
    $C$ relative to a generic local 
    coordinate system $(x,y)$, or at the set of infinitely near points which have to be blown up 
    in order to get the minimal embedded resolution of the germ $(C,o)$ or, thirdly, 
    at the preimage of this germ by the resolution. Each point of view leads to a different 
    encoding of the combinatorial type by a decorated tree: an 
    \emph{Eggers-Wall tree}, an \emph{Enriques diagram}, or a 
    \emph{weighted dual graph}. The three trees contain the same information, 
    which in the complex setting is equivalent to the knowledge of the 
    embedded topological type. 
    There are known algorithms for transforming one tree into another. In this paper 
    we explain how a special type of two-dimensional simplicial complex called a \emph{lotus} 
    allows to think geometrically about the relations between the three types of trees. 
    Namely, all of them embed in a natural lotus, their numerical decorations appearing as invariants 
    of it. This lotus is constructed from the finite set of Newton polygons 
    created during any process of resolution of $(C,o)$ by successive toric modifications. 
\end{abstract}

\maketitle

\tableofcontents


\section{Introduction}
\label{sec:intro}
\medskip

 The aim of this paper is to unify various combinatorial objects classically 
used to encode the equisingularity/combinatorial/embedded topological type of a 
plane curve singularity. 
  Often, a \emph{plane curve singularity} means a germ $(C, o)$ of algebraic or holomorphic  
 curve  defined by one equation in a  smooth complex algebraic surface. 
 In this paper we will allow the ambient 
 surface to be any germ $(S,o)$ of smooth complex algebraic or analytic surface, 
 and $C$ to be a formal germ of curve. 
  Using a local formal coordinate system $(x,y)$ on the germ $(S,o)$, 
 the global structure of $S$ disappears completely and one may suppose that $C$ 
 is formally embedded in the affine plane $\C^2$. 
 Usually one analyses in the following ways the structure of this embedding:

 \medskip 
 \noindent 
 $\bullet$ By considering the {\it Newton-Puiseux series} which express one of the variables 
          $(x,y)$ in terms of the other, whenever the equation $f(x,y)=0$ defining $C$ is satisfied. 
          Their combinatorics may be encoded  in two rooted trees, the \emph{Kuo-Lu tree} 
          and a Galois quotient of it, the \emph{Eggers-Wall tree}.

  \noindent 
 $\bullet$
  By blowing up points  starting from $o \in S$, 
        until obtaining an embedded resolution of $C$,  that is, a total transform of $C$ 
        which is a divisor with normal crossings. This blow up 
        process may be encoded in an \emph{Enriques diagram}, and the final 
        total transform of $C$ in a \emph{weighted dual graph}.

     \noindent 
 $\bullet$
    When the singularity $C$ is holomorphic, by intersecting a representative 
           of $C$ with a small enough Euclidean sphere centered at the origin, defined using 
           an arbitrary holomorphic local coordinate system $(x,y)$ on $(S,o)$. 
           This leads to an oriented link in an oriented 
          $3$-dimensional sphere.  This link is an \emph{iterated torus link}, whose 
          structure may be encoded in terms of another tree, called a \emph{splice diagram}.

 \medskip 
      Unlike the first  two  procedures, the  third one cannot be applied if the formal germ 
   $C$ is not holomorphic or if one works over an arbitrary algebraically closed field of characteristic zero.   
  For this reason, we do not develop it in this paper. 
  Let us mention only that it was  initiated in Brauner's pioneering paper \cite{B 28}, 
  whose historical background was described by Epple in \cite{E 95}. For its developments, 
  one may consult chronologically Reeve \cite{R 55},  L\^e \cite{L 72}, 
  A'Campo \cite{A 73}, Eisenbud $\&$ Neumann 
  \cite[Appendix to Chap. I]{EN 85}, Schrauwen \cite{S 90}, L\^e \cite{L 03}, Wall \cite[Chap. 9]{W 04}, 
  Weber \cite{W 08} and the present authors \cite[Chap. 5]{GBGPPP 18b}.  
  Similarly, we will not consider the discrete invariants constructed usually using the topology of 
  the Milnor fibration of a holomorphic germ $f$, as Milnor numbers, Seifert forms, 
  monodromy operators and their Zeta functions. The readers interested  in such 
  invariants may consult the 
  textbooks \cite{BK 86} of Brieskorn $\&$ Kn\"orrer and \cite{W 04} of Wall.
  
\medskip
There are algorithms  allowing to pass between the Eggers-Wall tree, the dual graph 
  and the Enriques diagram of $C$. However, they do not allow geometric representations  
  of  those passages. Our aim is to represent all these relationships 
  using a single geometric object, called a \emph{lotus},  
  which is a special type of simplicial complex of dimension at most two.

 Our approach for associating lotuses to plane curve singularities
 is done in the spirit of the papers of L\^e $\&$ Oka \cite{LO 95},  
A'Campo $\&$ Oka \cite{AO 96}, Oka \cite{O 96}, Gonz\'alez P\'erez \cite[Section 3.4]{GP 03}, and 
   Cassou Nogu\`es $\&$ Libgober \cite{CNL 14}. 
Namely, we use the fact that one may obtain an embedded resolution of $C$ 
by composing a sequence of \emph{toric} modifications determined by the successive Newton polygons 
of $C$ or of strict transforms of it, relative to suitable local coordinate systems. 
            
One may construct a lotus using the previous Newton polygons
   (see Definition \ref{def:lotustoroid}).
    Its one dimensional skeleton may be seen as a dual complex representing  
   the space-time of the evolution of the dual graph during the process 
  of blow ups of points which leads to the embedded resolution.  
   Besides the irreducible components of $C$ and the components of the exceptional divisor, 
   one takes also into account the curves 
 defined by the chosen local coordinate systems. If $A$ and $B$ are two such exceptional 
 or coordinate curves, 
 and them or their strict transforms intersect transversally at a point $p$ which is blown 
 up at some moment of the process, then a two dimensional simplex with vertices 
 labeled by $A$, $B$ and the exceptional divisor of the blow up of $p$  belongs to the lotus. 
 These simplices are called the \emph{petals} of the lotus 
  (see an example of a lotus with $18$ petals in Figure \ref{fig:lotusintrod}). 
 The Eggers-Wall tree, the Enriques diagram and the 
 weighted dual graph embed simultaneously inside the lotus, 
 and the geometry of the lotus also captures 
 the numerical decorations of the weighted dual graph and 
 the Eggers-Wall tree (see Theorem \ref{thm:repsailtor}). 
 For instance, the self-intersection number of a component of the final exceptional divisor 
 is the opposite of the number of petals containing the associated vertex of the lotus.
  The previous lotuses associated to $C$ have also valuative interpretations: they embed 
   canonically in the space of semivaluations of the completed local  ring of the germ $(S, o)$ 
   (see Remark \ref{rem:valemblot}).

\begin{figure}[h!]
     \begin{center}
\begin{tikzpicture}[scale=0.5]
\draw [fill=pink!40](0,0) -- (3,3) -- (1,-2)--(0,0);
\draw [->, very thick, red] (0,0) --(0.5,-1);
 \draw [-, very thick, red] (0.5,-1) --(1,-2);

\draw [-] (0,0)--(1.5,-0.8);
\draw [-] (1,1)--(1.5,-0.8);
\draw [-] (2,2)--(1.5,-0.8);
\draw [->] (3,3)--(3.5,5.5);
\draw [->] (2,2)--(2,5);

\draw [fill=pink!40](0,0) -- (-2,-3)--(-7,-3)--(-8,-1)--(-4,-2)--(0,0);
       \draw [->, very thick, red] (-2,-3)--(-4.5,-3);
       \draw [-, very thick, red] (-4.5,-3)--(-7,-3);
\draw [-] (-4,-2)--(-2,-3);
\draw [-] (-2,-1)--(-1,-1.5);
\draw [-] (-4,-2)--(-1,-1.5);
\draw [-] (-4,-2)--(-7,-3);
\draw [-] (-7,-3)--(-6,-1.5);
\draw [-] (-7.5,-2)--(-6,-1.5);
\draw [->] (-8,-1)--(-9,1);
\draw [->] (-6,-1.6)--(-7,1);
\draw [->] (-6,-1.6)--(-6,1);

\draw [fill=pink!40](0,0) -- (-3.5,-0.5)--(-4,3)--(-2,2)--(-1.5,4)--(0,0);
    \draw [->, very thick, red]  (0,0) -- (-1.75,-0.25);
    \draw [-, very thick, red]  (-1.75,-0.25)--(-3.5,-0.5);
\draw [-] (-3.5,-0.5)--(-2,2);
\draw [-] (-3.5,-0.5)--(-0.5,1.2);
\draw [-] (-2,2)--(-1,2.5);
\draw [-] (-2,2)--(-1,2.5);
\draw [-] (-0.5,1.2)--(-2,2);
\draw [->] (-4,3)--(-5,5);

\draw [fill=pink!40](-1.5,4) -- (-2,6)--(-3,4)--(-1.5,4);
     \draw [-, very thick, red] (-1.5,4)--(-3,4);
     \draw [->, very thick, red] (-1.5,4)--(-2.25,4);
\draw [-] (-1.5,4)--(-2.5,5);
\draw [->] (-2,6)--(-2,7);

\draw [-, ultra thick, color=orange]  (-2, -3) -- (0,0) -- (-4, -2) -- (-6, -1.5) -- (-8, -1) -- (-7,-3);
\draw [-, ultra thick, color=orange] (0,0) -- (3,3) -- (1,-2);
\draw [-, ultra thick, color=orange] (0,0) -- (-1.5,4) -- (-2,2) -- (-4,3) -- (-3.5,-0.5);
\draw [-, ultra thick, color=orange] (-1.5,4) -- (-2,6) -- (-3,4);

\node[draw,circle,inner sep=1.3pt,fill=black] at (2,2){};
\node[draw,circle,inner sep=1.3pt,fill=black] at (3,3){};
\node[draw,circle,inner sep=1.3pt,fill=black] at (1,-2){};

\node[draw,circle,inner sep=1.3pt,fill=black] at (-2,-3){};
\node[draw,circle,inner sep=1.3pt,fill=black] at (-7,-3){};
\node[draw,circle,inner sep=1.3pt,fill=black] at (-8,-1){};
\node[draw,circle,inner sep=1.3pt,fill=black] at (-6,-1.5){};

\node[draw,circle,inner sep=1.3pt,fill=black] at (-3.5,-0.5){};


\node[draw,circle,inner sep=1.3pt,fill=black] at (-3,4){};
\node[draw,circle,inner sep=1.3pt,fill=black] at (-2,6){};
\node[draw,circle,inner sep=1.3pt,fill=black] at (-1.5,4){};
\node[draw,circle,inner sep=1.3pt,fill=black] at (0,0){};

\end{tikzpicture}
\end{center}
\caption{A lotus. It is part of 
     Figure \ref{fig:lotustoroid}, which corresponds to   Example \ref{ex:lotustoroid}.}  
   \label{fig:lotusintrod}
    \end{figure}
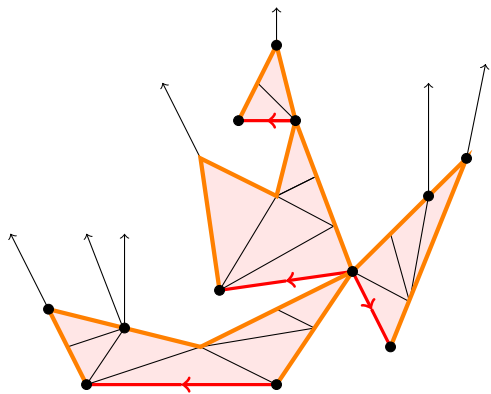

\medskip 

Let us describe the structure of the paper.

\medskip
In Section \ref{sec:basicnotex} we introduce basic notions about \emph{complex analytic varieties},  
\emph{plane curve singularities}, their \emph{multiplicities} and \emph{intersection numbers},  
\emph{normalizations},  \emph{Newton-Puiseux series},  \emph{blow ups}, 
\emph{embedded resolutions of plane curve singularities} and the associated 
\emph{weighted dual graphs}.  
The notions of \emph{Newton polygon}, 
dual \emph{Newton fan} and \emph{lotus} are first presented here on a 
  \emph{Newton non-degenerate} example. 

\medskip 
Section \ref{sec:torsurfmorph} begins with an explanation of basic notions of \emph{toric geometry}: 
\emph{fans} and their subdivisions, the associated \emph{toric varieties} and 
\emph{toric modifications} (see Subsections \ref{ssec:fans}, \ref{ssec:torsurf} and \ref{ssec:tormod}).  
In particular, we describe the \emph{toric boundary} 
of a toric variety -- the reduced divisor obtained as the 
complement of its dense torus --  in terms of the associated fan. Then we pass 
to \emph{toroidal geometry}: we introduce \emph{toroidal varieties}, which are pairs   $(\Sigma, \partial \Sigma)$ consisting 
of a normal complex analytic variety $\Sigma$ and a reduced divisor $\partial \Sigma$ on it, which are locally analytically isomorphic 
to a germ of a pair formed by a toric variety and its \emph{boundary divisor}.
A basic example of toroidal surface is that of a germ $(S,o)$ of smooth surface, 
 endowed with the divisor $L + L'$, where $(L, L')$ is a \emph{cross}, 
  that is, a pair of smooth transversal germs of curves.  A \emph{morphism 
   $\phi : (\Sigma_2, \partial \Sigma_2) \to (\Sigma_1, \partial \Sigma_1)$ of toroidal varieties}  
   is a complex analytic morphism such that 
   $\phi^{-1}(\partial \Sigma_1) \subseteq \partial \Sigma_2$ (see  Subsection \ref{ssec:toroidmod}).

\medskip 
In Section \ref{sec:tores} we explain in which way one may associate various 
morphisms of toroidal surfaces to the  plane curve singularity $C \hookrightarrow S$.
First, choose a cross $(L, L')$ on $(S,o)$, defined by a local coordinate system $(x, y)$.
   The \emph{Newton polygon}
   $\cN(f)$ of a defining function $f \in \C[[x,y]]$
  of the curve singularity $C$ depends only on $C$ and on the cross $(L, L')$.
    Its associated \emph{Newton fan} is obtained by subdividing the first quadrant along 
    the rays orthogonal to the compact edges of the Newton polygon.
    This fan defines a toric modification of $S$, 
    the \emph{Newton modification of $S$ defined by $C$ relative to the cross $(L,L')$} 
    (see Subsection \ref{ssec:NPcross}). 
    The Newton modification becomes a toroidal morphism when we endow its target $S$ 
    with the boundary divisor $\partial S : = L+L'$  
    and we define the boundary divisor of  its source to be the preimage of $L+ L'$.  
We emphasize the fact that those notions depend only on the objects 
$(S, C, (L, L'))$, in order to  insist on the underlying geometric structures. 
The strict transform of $C$ by the  previous Newton modification intersects 
the boundary divisor only at smooth points of it,  
which belong to the exceptional divisor and are smooth points of the ambient surface.
If one completes the germ of exceptional divisor into a cross at each such point $o_i$, then 
one gets again a triple of the form (\textit{surface}, \textit{curve}, \textit{cross}), 
where this time the curve is 
the germ at $o_i$ of the strict transform of $C$. Therefore one may perform again a Newton 
modification at each such point, and continue in this way until the strict transform of 
$C$ defines everywhere crosses with the exceptional divisor. 
The total transform of $C$ and of all coordinate curves introduced during previous steps 
 define the toroidal boundary $\partial \Sigma$ on the final surface $\Sigma$. 
This non-deterministic algorithm produces morphisms  
$\pi : (\Sigma, \partial \Sigma) \to (S, \partial S)$ of toroidal surfaces,
which are  \emph{toroidal pseudo-resolutions} of the plane curve singularity $C$
(see Subsection \ref{ssec:algtores}). The surface $\Sigma$ has a finite number of singular points, at which it is  locally analytically isomorphic to normal toric surfaces.  
In Subsection \ref{ssec:toremb} we show how to pass from the toroidal pseudo-resolution $\pi$  
to a  \emph{toroidal embedded resolution} 
by composing $\pi$  with the minimal resolution of these toric singularities. 
 Finally, we encode the process of successive Newton modifications in a  \emph{fan tree}, 
in terms of the Newton fans produced by the pseudo-resolution process (see Subsection \ref{ssec:fantrees}).

\medskip 
In Section \ref{sec:embres} we explain the notion of lotus. 
 A \emph{Newton lotus} associated to a fan 
encodes geometrically the continued fraction expansions of the slopes of the rays of the fan, 
as well as their common parts (see Subsection \ref{ssec:lotcf}). 
It is composed of \emph{petals}, and each petal corresponds 
to the blow up of the base point of a cross. One may clarify the subtitle of the paper 
by saying that \emph{the collection of Newton polygons 
appearing during the toroidal pseudo-resolution process blossomed into the associated lotus, each 
petal corresponding to a blow up operation}. We explain how 
to associate to the fan tree of the toroidal pseudo-resolution a \emph{lotus}, which is a $2$-dimensional 
simplicial complex obtained by gluing the \emph{Newton lotuses} associated to the 
  Newton fans of the process (see Subsections \ref{ssec:lotnf} and \ref{ssec:sailtores}). 
The lotus of
a toroidal pseudo-resolution depends on the choices of crosses made during the process 
of pseudo-resolution (see Subsection \ref{ssec:deplotus}). 
We explain then how to embed in the lotus the Enriques diagram and the dual 
graph of the embedded resolution. We conclude the section by 
defining a \emph{truncation operation} on lotuses, and we explain how it may be used 
to understand the part of the embedded resolution which does not depend on the 
supplementary curves introduced during the pseudo-resolution process (see Subsection 
\ref{ssec:trunclot}).

\medskip 
We begin Section \ref{sec:FTEW} by introducing the notion of \emph{Eggers-Wall tree} of the curve 
$C$ relative to the smooth  germ $L$ (see Subsection \ref{ssec:EW}) and by expressing  
the Newton polygon of $C$ relative to a cross $(L, L')$ in terms of the Eggers-Wall tree 
of $C + L'$ relative to $L$ (see Subsection \ref{ssec:EW-Newton}). 
Then we explain that the fan tree of the previous toroidal 
pseudo-resolution process is canonically isomorphic to the Eggers-Wall tree relative to $L$ 
of the curve obtained by adding to $C$ the projections to $S$ 
of all the crosses built during the process and how to pass from the numerical 
decorations of the fan tree to those of the Eggers-Wall tree (see Subsection \ref{ssec:transfanEW}). 
 As preliminary results, 
we prove \emph{renormalization formulae} which describe the Eggers-Wall tree of the strict 
transform of $C$ by a Newton modification, relative to the exceptional divisor, in terms 
of the Eggers-Wall tree of $C$ relative to $L$ (see Subsections \ref{ssec:renormres} and 
\ref{ssec:transfanEW-Newton}). 

\medskip 
The final Section \ref{sec:hist} begins by an overview  of the construction of a fan tree 
and of the associated lotus from the Newton fans of a toroidal pseudo-resolution process 
(see Subsection \ref{ssec:ovv}).  
Subsection \ref{ssec:persp} describes perspectives on possible applications of lotuses to 
problems of singularity theory.  The final  Subsection \ref{ssec:genterm} 
contains a list of the main notations used in the article.

\medskip

Starting from Section \ref{sec:torsurfmorph},  
 each section ends with a subsection of historical comments. 
We apologize for any omission, which 
 may result from our limited knowledge. 
One may also find historical information about various tools used to study plane curve singularities 
in Enriques and Chisini's book \cite{EC 17}, in the first chapter of Zariski's book \cite{Z 35} and 
in the final sections of the chapters of Wall's book \cite{W 04}.

\medskip

We tried to make this paper understandable to  
PhD students who have only a basic knowledge about singularities.
  Even if everything in this paper holds over an arbitrary algebraically closed field of characteristic zero, 
   we will stick to the complex  setting, in order to make things more concrete for the beginner.
 We accompany the definitions with 
examples and many figures.  Indeed, one of our objectives is to show that lotuses 
may be a great visual tool for relating the combinatorial objects  used to study plane 
curve singularities.
 There is a main example, developed throughout the paper starting from Section \ref{sec:tores} 
(see Examples \ref{ex:toroidres}, \ref{ex:constrfantree}, 
\ref{ex:dualisom}, \ref{ex:lotustoroid}, \ref{ex:Enrtreemainex}, \ref{ex:truncex}, \ref{ex:fromFTtoEW} 
and the overview 
Figure \ref{fig:overview}). We recommend to study it carefully in order to get a concrete 
feeling of the various objects manipulated in this paper.
We also recommend to those readers who are learning the subject 
to refer to the Section \ref{ssec:ovv} from time to time, in order to measure their 
understanding of the geometrical objects presented here.

\section{Basic notions and examples}
\label{sec:basicnotex}
\medskip

In this section we recall basic notions about \emph{complex varieties} and 
\emph{plane curve singularities}  
(see Subsection \ref{ssec:basicnotions}), \emph{normalization morphisms}  
(see Subsection \ref{ssec:basicnormaliz}), 
the relation between \emph{Newton-Puiseux series}  and plane curve singularities
(see Subsection \ref{ssec:NPthm}) and \emph{resolution of such singularties} by 
iteration of \emph{blow ups of points} 
(see Subsection \ref{ssec:basicblow}). We describe such a resolution for the 
\emph{semi-cubical parabola}  
(see Subsection \ref{ssec:firtsexample}). 
We give a flavor of the main construction  of this paper in Subsection \ref{ssec:redexample}. 
We show there how to transform the Newton polygon of a certain 
\emph{Newton non-degenerate} plane curve singularity with two branches into 
a \emph{lotus}, and how this lotus contains the dual graph of a resolution by blow ups of points.

\medskip

From now on, $\boxed{\N}$ 
denotes the set of non-negative integers and $\boxed{\N^*}$ the set of positive integers.

\subsection{Basic facts about plane curve singularities}
\label{ssec:basicnotions}
$\:$  
\medskip

In this subsection we recall basic vocabulary about \emph{complex analytic spaces} 
(see Definition \ref{def:complexspace}) and we explain  
the notions of \emph{plane curve singularity} (see Definition \ref{def:multcurve}), 
of \emph{multiplicity} and 
of \emph{intersection number} (see Definition \ref{def:intnumber}) for such singularities. 
Finally, we recall an important way of computing such intersection numbers 
(see Proposition \ref{prop:disymint}). 
\medskip

Briefly speaking, a complex analytic space $X$ is obtained by gluing 
model spaces, which are zero-loci of systems of analytic equations in some complex affine 
space $\C^n$. One has to prescribe also the analytic ``functions'' living on the underlying 
topological space. 
Those ``functions'' are elements of a so-called ``structure sheaf'' $\cO_X$, which may contain 
nilpotent elements. For this reason, they are not classical functions, as they are not determined 
by their values. For instance, one may endow the origin of $\C$ with the structure sheaves 
whose rings of sections are the various rings $\C[x]/(x^m)$, with $m \in \N^*$. They are 
pairwise non-isomorphic and they contain nilpotent elements whenever $m \geq 2$. 
Let us state now the formal definitions of \emph{complex analytic spaces} and of some 
special types of complex analytic spaces. 

\begin{definition}   \label{def:complexspace} 
 $\:$

       \noindent
       $\bullet$
       A {\bf model complex analytic space} is a ringed space $(X, \cO_X)$,  
              where $X$ is the zero locus of $I$ and $\cO_X = \cO_{U}/ \mathcal{I}$. 
              Here $I$ is a finitely generated ideal of the ring of 
              holomorphic functions on an open set $U$ of $\C^n$, for some $n \in \N^*$, 
              $\cO_{U}$ is the sheaf of holomorphic functions on $U$ and 
              $\mathcal{I}$ is the sheaf of ideals of $\cO_{U}$ generated by $I$.

     \noindent
     $\bullet$ A {\bf complex analytic space}\index{complex!analytic space} 
          is a ringed space locally isomorphic to a model 
         complex analytic space. 
         
     \noindent
     $\bullet$ A complex analytic space is {\bf reduced}\index{complex!analytic space} 
           if its structure sheaf $\cO_X$ is reduced, 
          that is, without nilpotent elements. In this case, one speaks also about a 
          {\bf complex variety}\index{variety!complex}\index{complex!variety}. 
          
      \noindent
       $\bullet$  A {\bf complex manifold}\index{complex!manifold} 
       is a complex variety $X$ such that any point 
           $x \in X$ has a neighborhood isomorphic to an open set 
           of $\C^n$, for some $n \in \N$. If the non-negative integer $n$ is independent of $x$, 
           then the complex manifold $X$ is called {\bf equidimensional} and $n$ 
           is its {\bf complex dimension}\index{complex!dimension}\index{dimension!complex}. 
           
     \noindent
      $\bullet$ 
     The {\bf smooth locus} of a complex variety $X$ is its open subspace whose 
          points have neighborhoods which are complex manifolds. Its 
          {\bf singular locus}\index{singular locus} 
          $\boxed{\mathrm{Sing}(X)}$ is the complement of its smooth locus.    
                   
     \noindent
      $\bullet$ A {\bf smooth complex curve} is an equidimensional complex manifold of 
           complex dimension one 
          and a {\bf smooth complex surface} is an equidimensional complex 
          manifold of complex dimension two. 
          
      \noindent
       $\bullet$  A {\bf complex curve}\index{complex!curve}\index{curve!complex} 
           is a complex variety whose smooth locus is a smooth complex curve 
           and a {\bf complex surface} is a complex variety whose smooth locus is a smooth complex 
           surface. 
\end{definition}

By construction, the singular locus $\mathrm{Sing}(X)$ of $X$ is a closed subset of $X$.  
It is a deep theorem that this subset is in fact a complex subvariety of $X$ 
(see \cite[Corollary 6.3.4]{DJP 00}).

\medskip
Let $S$ be a smooth complex surface. If $o$ is a point of $S$ and $\phi :U \to V$ is an isomorphism 
from an open neighborhood $U$ of $o$ in $S$ to an open neighborhood $V$ of the origin 
in $\C^2_{x,y}$, then the coordinate holomorphic functions $x,y : \C^2_{x,y} \to \C$ may be lifted 
by $\phi$ to two holomorphic functions on $U$, vanishing at $o$. They form a 
{\bf local coordinate system on the {\bf germ} $\boxed{(S, o)}$ of $S$ at $o$}. 
By abuse of notations, 
we still denote this local coordinate system by $(x,y)$, 
and we see it as a couple of elements of $\boxed{\cO_{S,o}}$, 
the {\bf local ring of $S$ at $o$}, equal by definition to the $\C$-algebra of 
germs of holomorphic functions defined on some neighborhood of $o$ in $S$. 
The local coordinate system $(x,y)$ establishes an isomorphism $\cO_{S,o} \simeq \C\{x,y\}$, 
where $\boxed{\C\{x,y\}}$ denotes the $\C$-algebra of convergent power series in the 
variables $x,y$. Denote by $\boxed{\C[[x,y]]}$ the $\C$-algebra of 
formal power series in the same variables. It is the completion of $\C\{x,y\}$ relative to 
its maximal ideal $(x,y)\C\{x,y\}$. One has the following fundamental theorem, 
valid in fact for any finite number of variables (see \cite[Corollary 3.3.17]{DJP 00}): 

\begin{theorem}   \label{thm:factorialalg}
     The local rings $\C\{x,y\}$ and $\C[[x,y]]$ are factorial. 
\end{theorem}

In addition to 
Definition \ref{def:complexspace}, we use also the following meaning of the term \emph{curve}: 

\begin{definition}  \label{def:curve}  
      A {\bf curve $C$ on a smooth complex surface $S$}\index{curve!on a complex surface} is  
         an effective Cartier divisor of $S$, that is, a complex  
         subspace of $S$ locally definable by the vanishing of a non-zero holomorphic function. 
\end{definition}
   
   This means that for every point $o \in C$, there exists an open  
neighborhood $U$ of $o$ in $S$ and a holomorphic function $f : U \to \C$ such that 
$C \subset U$ is the vanishing locus $\boxed{Z(f)}$  of $f$ and such that the structure sheaf 
$\boxed{\cO_{C | U}}$ of $C \subset U$ is the quotient sheaf $\cO_U / (f) \cO_U$. 
In this case, once $U$ is fixed, 
the defining function $f$ is unique up to multiplication by a holomorphic function on $U$ which 
vanishes nowhere. 

The curve $C$ is called {\bf reduced} if it is a reduced complex analytic 
space in the sense of Definition \ref{def:complexspace}. This means that any defining 
function $f : U \to \C$ as above is square-free in all local rings $\cO_{S, o}$, where $o \in U$. 
For instance, the union $C$ of coordinate axes of $\C^2$ is a reduced curve, being definable 
by the function $xy$, which is square-free in all the local rings $\cO_{\C^2,o}$, where 
$o \in C$. By contrast, the curve $D$ defined by the function $xy^2$ is not reduced.  

As results from Definition \ref{def:curve},  a complex subspace $C$ of $S$ is a curve 
on $S$ if and only if,  
for any $o \in C$, the ideal of $\cO_{S,o}$ consisting of the germs of holomorphic functions 
vanishing on the germ $(C, o)$ of $C$ at $o$  is \emph{principal}. 
We would have obtained a more general notion of \emph{curve} if we would have asked 
$C$ to be a $1$-dimensional 
complex subspace of $S$ in the neighborhood of any of its points. For instance, if 
$S = \C^2_{x,y}$, and $C$ is defined by the ideal $(x^2, xy)$ of $\C[x,y]$, then set-theoretically 
$C$ coincides with the $y$-axis $Z(x)$. But the associated structure sheaf 
$\cO_{C^2}/(x^2, xy)\cO_{C^2}$ is not the structure sheaf of an effective Cartier divisor. 
In fact the germ of $C$ at the origin cannot be defined by only one holomorphic function 
$f(x,y) \in \C\{x,y\}$. Otherwise, we would get that both $x^2$ and $xy$ are divisible by $f(x,y)$ 
in the local ring $\C\{x,y\}$. 
As this ring is factorial by Theorem \ref{thm:factorialalg}, 
we see that $f$ divides $x$ inside this ring, which implies that 
$(f) \C\{x,y\}= (x)\C\{x,y\}$. 
Therefore, $(x^2, xy)\C\{x,y\} =(x)\C\{x,y\}$ which is a contradiction, as $x$ is of order  $1$ 
and each element of the ideal $(x^2, xy)\C\{x,y\}$ is of order at least $2$. 
The notion of \emph{order} used in the previous sentence  
is defined by:

\begin{definition}  \label{def:multseries} 
    Let $f \in \C[[x,y]]$. Its {\bf order}\index{order!of a series} 
   is the smallest degree of its terms.  
\end{definition}

For instance, the maximal ideal of $\C[[x,y]]$ consists precisely of the power series of order  
at least $1$. 
It is a basic exercise to show that the order is invariant by the automorphisms of the 
$\C$-algebra $\C[[x,y]]$ and by multiplication by the elements of order $0$, which are 
the units of this algebra. Therefore, one gets a well-defined notion of \emph{multiplicity} of 
a germ of \emph{formal curve} on $S$: 

\begin{definition}   \label{def:multcurve}
   A {\bf plane curve singularity}\index{curve singularity!plane} is a germ $C$ of formal curve on a germ of 
   smooth complex surface $(S,o)$, that is, a principal ideal in the completion 
   $\boxed{\hat{\cO}_{S,o}}$ of the local ring $\cO_{S,o}$. 
   It is called a {\bf branch}\index{branch} if it is irreducible, that is, if its defining functions are 
   irreducible elements of the factorial local ring $\hat{\cO}_{S,o}$. 
   The {\bf multiplicity}\index{multiplicity!of a plane curve singularity} 
   $\boxed{m_o(C)}$ of $C$ at $o$ 
   is the order of a {\bf defining function} $f \in \hat{\cO}_{S,o}$ of $C$, 
   seen as an element of $\C[[x,y]]$ using any 
   local coordinate system $(x,y)$ of the germ $(S,o)$. 
\end{definition}

 \begin{example}   \label{ex:moncurve}
      Let $\alpha,   \beta \in \N^*$ and $f := x^{\alpha} - y^{\beta} \in \C[x,y]$. Denote by $C$ the 
      curve on $\C^2$ defined by $f$. Its multiplicity at the origin $O$  of $\mathbb{C}^2$
      is the minimum 
      of $\alpha$ and $\beta$. The curve singularity $(C,O)$ is a branch if and only if 
      $\alpha$ and $\beta$ are coprime. One implication is easy: if $\alpha$ and $\beta$ 
      have a common factor $\rho >1$, then
         $x^{\alpha} - y^{\beta} = \prod_{\omega: \: \omega^{\rho} =1} 
                 \left(x^{\alpha/\rho} - \omega y^{\beta/ \rho} \right)$, 
       the product being taken over all the complex $\rho$-th roots $\omega$ of $1$, which 
       shows that $(C,O)$ is not a branch. 
      The reverse implication results from the fact that, whenever $\alpha$ and $\beta$ 
      are coprime, $C$ is the image of the parametrization $N(t) :=  (t^{\beta}, t^{\alpha})$. 
      The inclusion $N(\C) \subseteq C$ being obvious, let us prove the reverse inclusion. 
      Let $(x,y) \in C$. As $N(0) =O$, it is enough to consider the case where 
      $xy \neq 0$. We want to show that there exists $t \in \C^*$ such that 
       $x = t^{\beta}$, $y = t^{\alpha}$. Assume the problem solved and consider also 
      a pair $(a, b) \in \Z^2$ such that $a \alpha + b \beta =1$, which exists by Bezout's theorem. 
      One gets  $t = t^{a \alpha + b \beta} = y^a x^b$.    
      Define therefore $t := y^a x^b$. Then:
          \[  
                      t^{\beta} = (y^a x^b)^{\beta} = (y^{\beta})^a x^{b \beta} = (x^{\alpha})^a x^{b \beta} = 
                             x^{a \alpha + b \beta} = x, 
                 \] 
        and similarly one shows that $t^{\alpha} = y$. 
       This proves that $C$ is indeed included in the image of $N$. 
 \end{example}

Let $C$ be a plane curve singularity on the germ of smooth surface $(S,o)$. If 
$f \in \hat{\cO}_{S,o}$ is a defining function of $C$, it may be decomposed as a product:
   \begin{equation}  \label{eq:decomp}
         f = \prod_{i \in I} f_{i}^{p_i}, 
   \end{equation}
 in which the functions $f_i$ are pairwise non-associated prime elements 
 of the local ring $\hat{\cO}_{S,o}$ 
 and $p_i \in \N^*$ for every $i \in I$.  Such a decomposition 
 is unique up to permutation of the factors $f_i^{p_i}$ and up to a replacement of each 
 function $f_i$ by an associated one (recall that two such functions are 
 \emph{associated} if one is the product of another one by a unit of the local ring).
 If $C_i \subseteq S$ is the plane curve singularity 
 defined by $f_i$, then the decomposition (\ref{eq:decomp}) gives a decomposition of 
 $C$ seen as a germ of effective divisor $C = \sum_{i \in I} p_i C_i$,  
  where each curve singularity $C_i$ is a branch. The plane curve singularity $C$ is reduced 
  if and only if $p_i =1$ for every $i \in I$.

The 
\emph{intersection number} is the simplest measure of complexity of the way 
two plane curve singularities interact at a given point:

\begin{definition}  \label{def:intnumber}
    Let $C$ and $D$ be two curve singularities on the germ of smooth surface 
    $(S, o)$ defined by functions $f$ and $g \in \hat{\cO}_{S, o}$ respectively.
    Their {\bf intersection number}\index{intersection number!of plane curve singularities} 
    $\boxed{(C \cdot D)_o}$, also denoted 
    $\boxed{C \cdot D}$ if the base point $o$ of the germ is clear from 
    the context, is defined by: 
               \[ C \cdot D := \dim_{\C} \frac{\hat{\cO}_{S, o}}{(f, g)} \in \N \cup \{\infty\},  \]
    where $(f, g)$ denotes the ideal of $\hat{\cO}_{S, o}$ generated by $f$ and $g$. 
\end{definition}

If $C$ and $D$ are two curve singularities, then one has that 
$ (C \cdot D)_o \geq  m_o(C) m_o(D) $, with equality if and 
only if the curves $C$ and $D$ are  \textbf{transversal} (see \cite[Lemma 4.4.1]{W 04}), that is, 
the tangent plane of $(S,o)$ does not contain lines which are tangent to both $C$ and $D$. 

Seen as a function of two variables, the intersection number is symmetric. 
It is moreover bilinear, in the sense that
if $C = \sum_{i \in I} p_i C_i$, then $C \cdot D = \sum_{i \in I} p_i (C_i \cdot D)$.  
Therefore, in order to compute $C \cdot D$, it is enough to find $C_i \cdot D$ 
for all the branches $C_i$ of $C$. 

One has the following useful property (see \cite[Lemma 5.1.5]{DJP 00}):

\begin{proposition}  \label{prop:disymint}
      Let $C$ be a branch and $D$ be an arbitrary curve singularity on the smooth germ 
      of smooth surface $(S,o)$. Denote by $N : (\C_t,0) \to (S,o)$ a formal parametrization of 
      degree one of $C$ and $g \in \hat{\cO}_{S,o}$ be a defining function of $D$. Then 
          \[  C \cdot D = \nu_t(g(N(t))), \]
       where $\nu_t(h)$ denotes the order of a power series $h \in \C[[t]]$. 
\end{proposition}

\begin{example}  \label{ex:intnumb}
    Let us consider two curves $C, D \subseteq \C^2_{x,y}$, defined by 
    polynomials $f := x^{\alpha} - y^{\beta}$ and $g :=  x^{\gamma} - y^{\delta}$ 
    of the type already considered in Example \ref{ex:moncurve}. 
    Assume that $\alpha$ and 
    $\beta$ are coprime. This  implies, as shown in Example \ref{ex:moncurve}, that 
    the plane curve singularity $(C, O)$ is a branch and that $N(t) :=  (t^{\beta}, t^{\alpha})$ 
    is a parametrization of degree one of it. 
    By Proposition \ref{prop:disymint}, if $C$ is not a branch of $D$,  we get:
       \[ C \cdot D =  \nu_t \left( (t^{\beta})^{\gamma}  - (t^{\alpha})^{\delta}  \right) = 
                \nu_t \left(  t^{\beta \gamma} - t^{\alpha \delta} \right) = 
                \min \{ \beta \gamma, \alpha \delta\}.    \]
\end{example}

For more details about intersection numbers of plane curve singularities, one may consult 
\cite[Sect. 6]{BK 86}, \cite[Vol. 1, Chap. IV.1]{S 94}  and \cite[Chap. 8]{F 01}.

The formal parametrizations $N : (\C_t,0) \to (S,o)$ of degree one of a branch  
appearing in the statement of Proposition \ref{prop:disymint} are exactly the \emph{normalization 
morphisms of $C$} whose sources are identified with $(\C, 0)$. Next subsection is dedicated 
to the general definition of \emph{normal complex variety} and of \emph{normalization  
morphism} in arbitrary dimension, as we will need them later also for surfaces.

\subsection{Basic facts about normalizations}
\label{ssec:basicnormaliz}
$\:$  
\medskip

In this subsection we explain basic facts about \emph{normal rings} 
(see Definition \ref{def:normar}), \emph{normal complex varieties}  (see Definition \ref{def:normgeom})
and \emph{normalization morphisms} (see Definition \ref{def:normaliz}) 
of arbitrary complex varieties. For more 
details and proofs one may consult \cite[Sections 1.5, 4.4]{DJP 00} and \cite{G 78}. 
\medskip

The following definition contains  \emph{algebraic} 
notions, concerning extensions of rings: 

\begin{definition}   \label{def:normar}  
    Let $R$ be a commutative ring and let $R \subseteq T$ be an extension of $R$. 
        \begin{enumerate} 
           \item An element of $T$ is called {\bf integral over $R$} if it satisfies a monic 
                    polynomial relation with coefficients in $R$.
           \item The extension $R \subseteq T$ of $R$ is called 
                      {\bf integral}\index{integral extension} if each element of $T$ is integral over $R$.            
            \item The {\bf integral closure}\index{integral closure} 
                of $R$ is the set of integral elements over $R$ of the total ring of 
               fractions of $R$. 
            \item $R$ is called {\bf normal}\index{normal!ring} 
               if it is reduced (without nonzero nilpotent elements) and integrally closed in its total ring 
               of fractions, that is, if it coincides with its integral closure. 
         \end{enumerate}
\end{definition}

The arithmetical notion of \emph{normal ring} allows to define the geometrical notion of 
\emph{normal variety}: 

\begin{definition}   \label{def:normgeom}  
    Let $X$ be a complex variety in the sense of Definition \ref{def:complexspace}.  
        \begin{enumerate}
            \item  If $x \in X$, then the germ $(X,x)$ of $X$ at $x$ is called {\bf normal}\index{normal!germ} 
                   if its local ring $\cO_{X,x}$ is normal. 
            \item   The complex variety $X$ is {\bf normal}\index{normal!complex variety} 
                  if all its germs are normal. 
        \end{enumerate}
\end{definition}

Normal varieties may be characterized from a more function-theoretical viewpoint as those 
complex varieties on which holds the following ``Riemann extension property'': 
\emph{every  bounded holomorphic function defined on the smooth part of an open set 
 extends to a holomorphic function on the whole open set} (see \cite[Theorem 4.4.15]{DJP 00}). 

Recall now the following 
algebraic regularity condition (see \cite[Sect. 4.3]{DJP 00}):

\begin{definition}  \label{def:regring}
    Let $\cO$ be a Noetherian local ring, with maximal ideal $\mathfrak{m}$. 
    \begin{enumerate}
       \item  The {\bf Krull dimension}\index{dimension!Krull} of $\cO$ is the maximal length 
            of its chains of prime ideals. 
       \item The {\bf embedding dimension}\index{dimension!embedding} 
          of $\cO$ is the dimension of the 
          $\cO/\mathfrak{m}$-vector space $\mathfrak{m}/\mathfrak{m}^2$. 
       \item  The local ring $\cO$ is called {\bf regular}\index{regular!local ring} 
           if its Krull dimension is equal to its 
           embedding dimension.
    \end{enumerate}
\end{definition}

The Krull dimension of $\cO$ is always less or equal to the embedding dimension. The name 
\emph{embedding dimension} may be understood by restricting to the case where $\cO$ is the local 
ring of a complex space (see \cite[Lemma 4.3.5]{DJP 00}):

\begin{proposition}   \label{prop:meaned}
   Let $(X,x)$ be a germ of complex space. Then the embedding dimension of its 
   local ring $\cO_{X,x}$ is equal to the smallest $n \in \N$ such that there exists an embedding of germs 
   $(X,x) \hookrightarrow (\C^n, 0)$. In particular, $\cO_{X,x}$ is regular if and only if $(X,x)$ is smooth, 
   that is, a germ of complex manifold. 
\end{proposition}

The normal varieties of dimension one are exactly the smooth complex curves because, 
more generally (see \cite[Thm. 4.4.9, Cor. 4.4.10]{DJP 00}):

\begin{theorem}  \label{thm:normequivreg}
    A Noetherian local ring of Krull dimension one is normal if and only if it is regular. 
\end{theorem}

There is a canonical way to construct a normal variety $\tilde{X}$ starting from any complex  
variety $X$ (see \cite[Sect. 4.4]{DJP 00}):

\begin{theorem}  \label{thm:existnorm}
    Let $X$ be a complex variety. Then there exists a finite and generically $1$ to $1$ morphism 
    $N : \tilde{X} \to X$ such that $\tilde{X}$ is normal. Moreover, such a morphism is unique 
    up to a unique isomorphism over $X$. 
\end{theorem}

Recall that a morphism between complex varieties 
is \emph{finite}\index{morphism!finite} if it is proper with finite fibers and 
that it is \emph{generically $1$ to $1$} if it is an isomorphism above the complement 
of a nowhere dense 
closed subvariety of its target space. The existence of a morphism with the properties 
stated in Theorem \ref{thm:existnorm} may be proven algebraically by considering the 
 integral closures of the rings of holomorphic functions on the open sets of $X$, and showing 
 that they are again rings of holomorphic functions on complex varieties which admit finite 
 and generically $1$ to $1$ morphisms to the starting open sets. 
 This algebraic proof 
 extends to formal germs, by showing that the integral closure in its total ring of fractions 
 of a complete ring of the form 
 $\C[[x_1, \dots, x_n]]/I$, where $n \in \N^*$ and $I$ is an ideal of $\C[[x_1, \dots, x_n]]$, 
 is a direct sum of rings of the same form. 

The canonical morphisms characterized in Theorem \ref{thm:existnorm} received a special name:

\begin{definition}  \label{def:normaliz}
       Let $X$ be a complex variety. Then a morphism $N : \tilde{X} \to X$ is called 
       a {\bf normalization morphism}\index{morphism!normalization} 
       of $X$ if it is finite, generically $1$ to $1$ and 
       $\tilde{X}$ is a normal complex variety. 
\end{definition}

Let now $(C,o)$ be a germ of complex variety of Krull dimension one, that is, an 
{\bf abstract curve singularity}\index{curve singularity!abstract}. 
Its normalization morphisms are of the form: 
$N : \bigsqcup_{i \in I} (\tilde{C}_i, o_i) \to (C,o)$,  
 where $(C_i, o)_{i \in I}$ is the finite collection of irreducible components of $(C,o)$, 
 and the restriction $N_i : (\tilde{C}_i, o_i) \to (C,o)$ of $N$ to $\tilde{C}_i$ is a normalization 
 of $(C_i, o)$. By Theorem \ref{thm:normequivreg} and Proposition \ref{prop:meaned}, we see 
 that each germ $(\tilde{C}_i, o_i)$ is smooth, that is, isomorphic to $(\C, 0)$. After precomposing 
$N$ with such isomorphisms, we see that $(C,o)$ admits a normalization morphism of the form
 $\bigsqcup_{i \in I} (\C, 0) \to (C,o)$. 
 In particular, if $(C,o)$ is irreducible,  its normalization morphism is 
 of the form $N : (\C, 0) \to (C,o)$.  
 The same construction yields a 
 \emph{formal} parametrization when the starting germ $(C,o)$ is formal. 
 This is precisely a \emph{formal parametrization of degree one} as used in the statement of 
 Proposition \ref{prop:disymint}.

\subsection{Newton-Puiseux series and the Newton-Puiseux theorem}
\label{ssec:NPthm}
$\:$  
\medskip
 
   At the end of the previous subsection we explained that normalizations of irreducible 
   germs of complex analytic or formal curves $C$ are holomorphic or formal 
   parametrizations $(\C, 0) \to C$ of degree one. In this subsection we introduce especially nice 
   parametrizations in the case of plane branches, which lead to the notion of 
   \emph{Newton-Puiseux series} (see Definition \ref{def:NewtPuiseux}). 
   The \emph{Newton-Puiseux theorem} (see Theorem \ref{thm:NPthmbasic}) 
   implies that the field of Newton-Puiseux 
   series is algebraically closed.  Another consequence of it 
   is stated in Theorem \ref{thm:NewtPuiseux} 
   below. 
   \medskip

     Let  $C$ be a branch on the smooth germ of surface 
     $(S,o)$.  Choose an arbitrary system of local coordinates 
     on $(S,o)$. If the branch $C$ is smooth, assume moreover that the germ at $o$ of the $y$-axis 
     $Z(x)$ is different from $C$. This means that for any normalization morphism 
     $N : (\C_t, 0) \to (C,o) $ of $C$, described in this coordinate system as $t \to (\xi(t), \eta(t))$, 
     where $\xi, \eta \in (t) \C[[t]]$, the power series $\xi(t)$ is not identically zero.
     We have $\xi(t) = t^n \cdot \epsilon (t)$, where $n \in \N^*$ is the order of the power  
     series $\xi(t)$ and 
     $\epsilon(t)$ is a unit in the ring  $\C[[t]]$.
     The series $\epsilon(t)$ has exactly $n$ different $n$-th roots in  $\C[[t]]$, whose 
     constant terms are the $n$-th roots of $\epsilon (0)$. 
     Pick one of them, denote it by $\epsilon^{1/n}(t)$, and set $\lambda(t) := t \epsilon^{1/n}(t)$.  
     Therefore  $ \xi(t) = \lambda(t)^n$ and $\nu_t(\lambda(t)) =1$.
      
      \begin{remark}   \label{rem:rootproblem}
          More generally, if $K$ is an algebraically closed field of characteristic zero, then any unit 
          of $K[[t]]$ has all its  $n$-th roots in $K[[t]]$. This fact is no longer true if $K$ has  
          positive characteristic. For instance, as a direct consequence of the binomial 
          formula, there is no series $\epsilon(t) \in K[[t]]$ 
          such that $\epsilon(t)^p = 1 +t$ when $K$ is of characteristic $p$. 
          For this reason, the Newton-Puiseux Theorem 
          \ref{thm:NPthmbasic} below does not always hold in positive characteristic. 
          For more details about the situation in positive characteristic, one may consult 
          \cite{P 13}.
      \end{remark}
         
         As $\nu_t(\lambda(t)) =1$, we see that the morphism $(\C_t, 0)  \to  (\C_u, 0)$, 
         which maps $t \to \lambda(t)$
              is an isomorphism of germs of smooth curves. By composing the morphism 
          $N : (\C_t, 0) \to (C,o)$ with its inverse, one gets a new normalization morphism of the form:
             \[ \begin{array}{ccc}
                      (\C_u, 0) &  \to  & (C, o)  \\
                           u        &  \to   &    (u^n, \zeta(u))  
                  \end{array} \]
             where $\zeta(u) \in \C[[u]]$. 
          Therefore, if $f(x,y) \in \C [[x,y]]$ is a defining function of $C$ in the local coordinate 
          system $(x,y)$, we have:
                \begin{equation} \label{eq:paramNP} 
                       f(u^n, \zeta(u)) =0.  
                \end{equation}
           From the equations $x= u^n$, $y = \zeta(u)$, one may deduce formally that 
            $ u = x^{1/n}$, $y = \zeta(x^{1/n})$.  
          Equation (\ref{eq:paramNP}) becomes: 
              \begin{equation} \label{eq:paramNPbis} 
                       f(x, \zeta(x^{1/n})) =0.  
                \end{equation}
                
          The composition $\zeta(x^{1/n})$ is a \emph{Newton-Puiseux series} in the following sense:  
          
\begin{definition}   \label{def:NewtPuiseux}
    The $\C$-algebra $\boxed{\C[[x^{1/ \N}]]}$ of 
    {\bf Newton-Puiseux series}\index{Newton-Puiseux!series}  consists of all the 
    formal power series of the form $\eta(x^{1/n})$, where $\eta \in \C[[t]]$ and 
    $n \in \N^*$, that is, $\C[[x^{1/ \N}]] = \bigcup_{n\in \N^*}  \C[[x^{1/ n}]]$. 
    Denote by $\boxed{\nu_x}:  \C[[x^{1/ \N}]] \to [0, \infty]$ the 
    {\bf order function}\index{function!order}, which 
    associates to every Newton-Puiseux series the smallest exponent of its terms, where 
     $\nu_x(0) := \infty$. 
\end{definition}

The function $\nu_x$ is a \emph{valuation} of the $\C$-algebra of 
Newton-Puiseux series, in the following sense:

\begin{definition}  \label{def:valuation}
  A {\bf valuation}\index{valuation} on an integral $\C$-algebra $A$
   is a function $\nu:   A      \to   \R_+  \cup \{\infty\}$   which
  satisfies the following conditions:
       \begin{enumerate}
          \item  $\nu(fg) = \nu(f) + \nu(g), \: \hbox{\em for all}\:  f, g \in A$. 
          \item  $\nu(f+g) \geq \min \{\nu(f), \nu(g)\}, \: \hbox{\em for all}\:  f, g \in A $. 
          \item   $\nu(\lambda) =0, \:  \hbox{\em for all} \: \lambda \in \C^*.$
         \item \label{condinfinity}  
              $\nu(f) = \infty$ if and only if $f =0$. 
       \end{enumerate} 
 \end{definition}

 The basic importance of the ring of Newton-Puiseux series comes from the following 
 \emph{Newton-Puiseux theorem}  
 (see Fischer \cite[Chapter 7]{F 01}, Teissier \cite[Section 1]{T 95}, \cite[Sections 3-4]{T 07}, 
de Jong $\&$ Pfister \cite[Section 5.1]{DJP 00}, Cutkosky \cite[Section 2.1]{C 04} or 
Greuel, Lossen $\&$ Shustin \cite[Thm. I.3.3]{GLS 07}):
 
 \begin{theorem} \label{thm:NPthmbasic} 
      {\bf (Newton-Puiseux theorem)}\index{Newton-Puiseux!theorem} 
       Any non-zero monic polynomial $f \in \C[[x]][y]$ such that $f(0,y) = y^d$ 
       has $d$ roots in the ring $ \C[[x^{1/ \N}]]$. 
       As a consequence, the quotient field of the ring $ \C[[x^{1/ \N}]]$ is the algebraic 
       closure of the quotient field of the ring $\C[[x]]$. 
 \end{theorem} 
 
 \begin{proof}
     
     It is immediate to reduce the proof of the first sentence of the theorem to the case 
     where $f$ is irreducible.  Assume that this is the case. 
     By equation (\ref{eq:paramNPbis}), there exists a Newton-Puiseux series 
     $\zeta(x^{1/n})$ which is a root of $f$. One has necessarily $n =d$. Indeed, by the 
     proof of equation (\ref{eq:paramNPbis}), 
     $u \to (u^n, \zeta(u))$ is a normalization of the formal branch $Z(f)$. Therefore, 
     Proposition \ref{prop:disymint} shows that: 
         \[  n = \nu_u(u^n) = Z(f) \cdot Z(x) = Z(x) \cdot Z(f) = \nu_y \left( f(0, y) \right) =d.\]
     Consider now the product: 
         \[F(x,y) := \prod_{\omega : \omega^n =1} \left(y - \zeta(\omega x^{1/n})\right) 
                       \in \C[[x^{1/ n}]][y].\]                
      It is invariant by the changes of variables $(x^{1/ n}, y) \to (\omega x^{1/ n}, y)$, 
      where $\omega$ varies among the complex $n$-th roots of $1$, which shows that 
      $F(x,y) \in \C[[x]][y]$. As $\zeta(x^{1/n})$ is a root of both $f(x,y)$ and $F(x,y)$ 
      and that $f(x,y)$ is irreducible, we see that $f$ divides $F$ in the ring $\C[[x]][y]$. 
      Both being monic and of the same degree, we get the equality $f =F$. Therefore, 
      all the roots of $f$ belong to $\C[[x^{1/ n}]]$. 
      
      The second statement of the theorem results from the first statement and 
      from \emph{Hensel's lemma} 
      (see \cite[Corollary 3.3.21]{DJP 00}), which ensures that a factorisation of 
      $f(0,y) \in \C[y]$ in pairwise coprime factors lifts to an analogous decomposition of 
      $f(x,y) \in \C[[x]][y]$. 
 \end{proof}

 The proof of Theorem \ref{thm:NPthmbasic} which we have sketched here
 also shows that the Galois group of the field extension 
associated to the ring extension $\C[[x]] \subset \C[[x^{1/n}]]$ is isomorphic 
to the cyclic group of $n$-th roots of $1$, an element $\omega$ of this group 
acting on $\zeta(x^{1/n}) \in \C[[x^{1/n}]]$ replacing it by $\zeta(\omega x^{1/n})$. 

\begin{remark} \label{rem:Gal}
 The proof of Theorem \ref{thm:NPthmbasic} which we have sketched here
 also shows that the Galois group of the field extension \index{Galois group}
associated to the ring extension $\C[[x]] \subset \C[[x^{1/n}]]$ is isomorphic 
to the cyclic group of $n$-th roots of $1$, an element $\omega$ of this group 
acting on $\zeta(x^{1/n}) \in \C[[x^{1/n}]]$ replacing it by $\zeta(\omega x^{1/n})$. 
\end{remark}
 
 \begin{remark} \label{rem:Puiseux}
Most proofs of Theorem \ref{thm:NPthmbasic}  use the Newton polygon $\cN(f)$ of $f$ 
    (see Definition \ref{def:Npolalg} below). As explained in Subsection \ref{ssec:firtsexample}, 
     the restrictions of $f$ to the compact edges of $\cN(f)$ allow to find the possible initial terms of 
     the candidate roots $\eta(x)$ of the equation $f(x,y)=0$ seen as an equation in the 
     single unknown $y$. Such proofs proceed then 
     by showing that all those terms may be extended to true roots inside $\C[[x^{1/ \N}]]$. 
\end{remark}

\begin{example}   \label{ex:NProotsex}
       Consider coprime integers $\alpha, \beta \in  \N^*$ and 
       $f(x,y) :=  x^{\alpha} - y^{\beta} \in \C[[x]][y]$, as in Example 
       \ref{ex:moncurve}. Then the Newton-Puiseux roots of $f$ are the $\beta$ 
       series $\omega \:  x^{\alpha/\beta}$, where $\omega$ varies among the 
       complex  $\beta$-th roots of $1$. 
       If $\omega'$ is another such root of $1$, it acts 
       on $\omega \: x^{\alpha/\beta}$ by sending it to $(\omega')^{\alpha}\:  \omega \: x^{\alpha/\beta}$. 
\end{example}

\subsection{Blow ups and embedded resolutions of singularities}
\label{ssec:basicblow}
$\:$  
\medskip

In this subsection we explain the notion of \emph{blow up} of $\C^2$ at the origin 
(see Definition \ref{def:blowup}) and more 
generally of a smooth complex surface at an arbitrary point of it (see Definition 
\ref{def:blowupgen}), the notion of \emph{embedded 
resolution} of a curve in a smooth surface (see Definition \ref{def:embres}) 
and the fact that an embedded resolution may 
be achieved after a finite number of blow ups of points (see Theorem \ref{thm:iterativeblowup}). 
We conclude by recalling the 
notion of \emph{weighted dual graph} of an embedded resolution 
(see Definition \ref{def:dualgraphembres}) and the way to compute 
its weights when this resolution is constructed iteratively by blowing up points. 
\medskip

Look at the complex affine plane $\C^2_{x,y}$ as a complex vector space. 
Denote by $\boxed{\bP(\C^2)_{[u:v]}}$ its {\bf projectivisation}, consisting of its vector 
subspaces of dimension one, endowed with the projective coordinates $[u:v]$ 
associated to the cartesian coordinates $(x,y)$ on $\C^2$. 

\begin{definition} \label{def:blowup}
     Consider the {\bf projectivisation map}\index{map!projectivisation} 
          \[ \begin{array}{cccc}
                \boxed{\lambda} :      &       \C^2  &  \dashrightarrow  & \bP(\C^2) \\
                                                  &      (x,y)   &  \dashrightarrow  & [x:y].
             \end{array}  \]
      associating to each point of $\C^2 \setminus \{O\}$ the line 
      joining it to the origin $O$ of $\C^2$. Let $\boxed{\Sigma}$ be the closure of its graph in 
      the product algebraic variety $\C^2 \times \bP(\C^2)$. Then the restriction 
      $\boxed{\pi} : \Sigma \to \C^2$ of the first projection $\C^2 \times \bP(\C^2)  \to \C^2$ 
      is called the {\bf blow up of $\C^2$ at the origin}\index{blow up!of the origin}. 
      By abuse of language, the 
      surface $\Sigma$ is also called in this way. The preimage $\pi^{-1}(O)$ of 
      $O$ in $\Sigma$ is called the {\bf exceptional divisor} of the blow up.  
      The restriction $\boxed{\tilde{\lambda}}: \Sigma \to \bP(\C^2)$  to $\Sigma$ 
      of the second projection 
      $\C^2 \times \bP(\C^2)  \to \bP(\C^2)$ is called the {\bf Hopf morphism}\index{Hopf morphism}.       
\end{definition}

The name ``\emph{Hopf morphism}'' is motivated by the fact that in restriction to 
the preimage $\pi^{-1}(\bS^3)$ of the unit $3$-dimensional sphere $\bS^3 \subset \C^2$, 
the morphism $\tilde{\lambda}$ becomes the ``Hopf fibration'' $\bS^3 \to \bS^2$, 
introduced by Hopf in \cite[Section 5]{H 31} (see also \cite{S 99} for historical details). 

The projectivisation \emph{map} restricts to a \emph{morphism} 
$\lambda : \C^2 \setminus \{O\} \to \bP(\C^2)$. This morphism cannot be extended 
even by continuity to the origin $O$, because $O$ belongs to the closures of all 
its level sets, which are the complex lines of $\C^2$ passing through $O$. 
Taking the closure of the graph of $\lambda$ replaces $O$ by the space $\bP(\C^2)$ of 
lines passing through $O$. This allows the lift of $\lambda$ to $\Sigma$ to extend by 
continuity, and even algebraically, to 
the whole surface $\Sigma$, becoming the Hopf morphism $\tilde{\lambda}$.  
This morphism is in fact the projection morphism of 
the total space of a line bundle, as will be shown in Proposition \ref{prop:liftprojlinebundle} below. 
Before proving it, let us explain how to describe using a simple atlas of two charts the 
blow up surface $\Sigma$.

The projective line $\bP(\C^2)_{[u:v]}$ is covered by the two affine lines $\C_{u_1}$ and 
$\C_{v_2}$, where:
    \[ \boxed{u_1} := \frac{u}{v}, \: \:  \boxed{v_2} := \frac{v}{u}. \]
Therefore, the product $\C^2 \times \bP(\C^2)$ is covered by the two affine $3$-folds 
$\C^3_{x,y, u_1}$ and $\C^3_{x,y, v_2}$. 

The surface $\Sigma$ contained in $\C^2 \times \bP(\C^2)$ 
is the zero locus $Z(xv - yu)$ of a homogeneous polynomial 
of degree one in the variables $u,v$. Its intersections with the two previous $3$-folds are therefore:
     \[ 
           \Sigma \cap \C^3_{x,y, u_1} = Z(x - y u_1), 
           \quad  \mbox{ and }  \quad 
           \Sigma \cap \C^3_{x,y, v_2} = Z(x v_2 - y). 
       \]
One recognizes in each case the equation of the graph of a function of two variables,  
those pairs of variables being $(u_1, y)$ and $(x, v_2)$ respectively. 
Therefore, by projecting on the planes 
of those two pairs of variables, one gets isomorphisms:  
       \[ 
           \Sigma \cap \C^3_{x,y, u_1} \simeq  \C^2_{u_1, y},  
           \quad \mbox{ and }  \quad 
           \Sigma \cap \C^3_{x,y, v_2} \simeq  \C^2_{x, v_2},
        \]
 which may be thought as the charts of an algebraic atlas of $\Sigma$. 
 Let us replace $y$ by $\boxed{u_2}$ in the first chart $\C^2_{u_1, y}$ and 
 $x$ by $\boxed{v_1}$ in the second chart $\C^2_{x, v_2}$. 
The blow up morphism $\pi : \Sigma \to \C^2$ gets expressed in the following way in the two 
charts: 
   \begin{equation}  \label{eq:twochartsblowup}  
         \left\{ \begin{array}{l}
                          x = u_{1} u_{2} 
                            \\
                          y = \: \: \: \: \:  u_{2}, 
                    \end{array} \right.   
                    \quad \mbox{ and } \quad 
           \left\{ \begin{array}{l}
                          x = v_{1}  
                            \\
                          y = v_{1} v_{2}.
                    \end{array} \right. 
       \end{equation}
The previous formulae show that the exceptional divisor $\pi^{-1}(O)$ becomes 
the $u_1$-axis in the chart $\C^2_{u_1, u_2}$ and the $v_2$-axis in the chart $\C^2_{v_1, v_2}$.

By composing one such morphism with the inverse of the second one, we see that 
$\Sigma$ may be obtained from the two copies $\C^2_{u_1, u_2}$ and $\C^2_{v_1, v_2}$ 
of $\C^2$ by gluing their open subsets $\C^*_{u_1} \times \C_{u_2}$ and 
$\C_{v_1} \times \C^*_{v_2}$ respectively  using the following inverse changes of variables:
     \begin{equation}   \label{eq:changechart}
             \left\{ \begin{array}{l}
                          v_1 = u_{1} u_{2} 
                            \\
                          v_2 = \: \: \: \: \:  u_{1}^{-1}
                    \end{array} \right.    
                    \quad \mbox{ and } \quad 
           \left\{ \begin{array}{l}
                          u_1 = v_{2}^{-1}  
                            \\
                          u_2 = v_{1} v_{2}.
                    \end{array} \right.  
     \end{equation}

The Hopf morphism $\tilde{\lambda} : \Sigma \to \bP(\C^2)$ becomes  
the morphisms $\C^2_{u_1, u_2} \to \C^1_{u_1}$ and $\C^2_{v_1, v_2} \to \C^1_{v_2}$ 
if one uses the charts $\C^2_{u_1, u_2}, \C^2_{v_1, v_2}$ for $\Sigma$ and 
$\C^1_{u_1}, \C^1_{v_2}$ for $\bP(\C^2)$. The fibers of these two morphisms have 
natural structures of complex lines if one identifies them with the standard complex 
line $\C$ using the parameters $u_2$ and $v_1$ respectively. As the gluing maps 
(\ref{eq:changechart}) respect those structures, we get:

\begin{proposition}  \label{prop:liftprojlinebundle}
    The Hopf morphism $\tilde{\lambda} : \Sigma \to \bP(\C^2)$ is the 
    projection morphism from the total space of a line bundle to its base $\bP(\C^2)$. 
    Its zero-section is the exceptional divisor $\pi^{-1}(O)$ of the blow up morphism 
    $\pi: \Sigma \to \C^2$. 
\end{proposition}

The fundamental numerical invariant of a complex line bundle over a projective curve, 
which characterises it up to topological isomorphims in general and up to 
algebraic isomorphisms if the curve is rational, is its {\bf degree}, defined by:

\begin{definition}   \label{def:degreelb}
    The {\bf degree}\index{degree!of a line bundle} 
    of a line bundle over a smooth connected projective curve $C$ is 
    the degree of the divisor on $C$ defined by any meromorphic section of the line bundle
    which is neither constantly $0$ nor constantly $\infty$.  
\end{definition}

In our case, we have:

\begin{proposition}   \label{prop:degHopf}
    The degree of the Hopf line bundle $\tilde{\lambda} : \Sigma \to \bP(\C^2)$ 
    is equal to $-1$. 
\end{proposition}

\begin{proof}
    Let us consider the meromorphic section $s$ of $\tilde{\lambda}$ which appears 
    as the constant function $1$ in the charts $\C^2_{u_1, u_2} \to \C^1_{u_1}$. 
    The equation of its graph is $u_2 =1$. The change of variables (\ref{eq:changechart}) 
    transform it into $v_1 v_2 =1$. Therefore, $s$ appears as the rational function 
    $v_2^{-1}$ in the charts $\C^2_{v_1, v_2} \to \C^1_{v_2}$. This shows that the 
    section $s$ has no zeros and a unique pole of multiplicity one. As a consequence, 
    the degree of the divisor defined by $s$ is equal to $-1$. 
\end{proof}

On any smooth complex algebraic or analytic surface $S$, one may define a notion 
of \emph{intersection number of two divisors} whenever at least one of them has compact support. 
This may be done \emph{algebraically}, by considering first the case when one 
divisor is a reduced compact curve $C$ on $S$, the intersection number being then the  
degree of the pullback of the line bundle defined by the second divisor to the normalization 
of $C$. Then,  one extends this definition by linearity to arbitrary not necessarily reduced or effective divisors. There is also a \emph{topological} definition, obtained by associating a homology 
class to one divisor, a cohomology class to the second one and then evaluating the 
cohomology class on the homology class. One may consult \cite[Sect. V.1]{H 77} for the 
case of algebraic surfaces and \cite[Pages 15--20]{L 71} for the case of analytic surfaces. 
It turns out that, either by definition or as a theorem, the self-intersection number of 
the zero-section of a line bundle over a smooth compact complex curve is equal to 
the degree of the line bundle. Therefore, Proposition \ref{prop:degHopf} may be also 
stated as:

\begin{corollary}   \label{cor:degHopfbis}
    The self-intersection number of the zero-section of  the Hopf line bundle 
    $\tilde{\lambda} : \Sigma \to \bP(\C^2)$  is equal to $-1$. 
\end{corollary}

Till now, we have discussed in this subsection only the blow up of the origin $O$ of $\C^2$. 
This operation may be extended to any point $o$ of a smooth complex surface $S$, by choosing 
first local coordinates $(x,y)$ in a neighborhood $U$ of that point. This allows to identify 
$U$ with an open neighborhood of $O$ in $\C^2$. Denote by $ \pi_U : \Sigma_U \to U$ the restriction 
to $U$ of the blow up morphism of $O$ in $\C^2$. This complex analytic morphism is an 
isomorphism over $U \setminus \{O\}$. Therefore, it allows to glue $\Sigma_U$ and 
$S$ along $U \setminus \{O\}$, getting a surface $\tilde{S}$ endowed with a morphism 
$\tilde{\pi} : \tilde{S} \to S$. 

\begin{definition}  \label{def:blowupgen}
    The morphism $\tilde{\pi} : \tilde{S} \to S$ constructed above is called a {\bf blow up morphism 
    of $S$ at the point $o$}. \index{blow up!of a point}
\end{definition}

It may be shown by a direct computation that the blow up morphism of $S$ at $o$ is independent 
of the choices of local coordinates and open set $U$. More precisely, given any two morphisms 
constructed in this way, there exists a unique isomorphism between their sources above $S$ 
(see \cite[Lemma 3.2.1]{W 04}). Another way to prove this uniqueness is to characterize such 
morphisms by a universal property (see \cite[Chap. II, Prop. 7.14]{H 77}):

\begin{proposition}  \label{prop:univblow}
    Let $S$ be a smooth complex surface and $\tilde{\pi} : \tilde{S} \to S$ a blow up morphism 
    of $S$ at its point $o$. Then for any morphism $f : Y \to S$ such that the ideal sheaf 
    defining $o$ on $S$ lifts to a principal ideal sheaf on $Y$, there exists a unique morphism 
    $g : Y \to \tilde{S}$ such that $f = \tilde{\pi} \circ g$. 
\end{proposition}

One may define more generally the blow up of any complex space along a closed subspace, 
and again this morphism may be characterized using an analogous 
universal property (see \cite[Pages 163--169]{H 77} for a similar study in the case 
of schemes). 

Returning to the model case of the blow up of $\C^2$ at the origin $O$, relations 
(\ref{eq:twochartsblowup}) show that the lift by $\pi$ to $\Sigma$ of the maximal 
ideal $(x,y)$ of $\C[x,y]$ defining $O$ is the principal ideal sheaf defining the exceptional 
divisor of $\pi$. This fact is an algebraic manifestation of the fact that on $\Sigma$ 
all the lines of $\C^2$ passing through $O$ get separated: they are simply the 
fibers of the Hopf morphism $\tilde{\lambda}$. Note that in order to separate them indeed, 
one does not have to lift them by taking their full preimages by 
$\pi$ (called their \emph{total transforms by $\pi$}), but 
only by taking their \emph{strict transforms}. Let us define these notions in greater 
generality:

\begin{definition}  \label{def:modiftransf}
    Let $\pi: Y \to X$ be a morphism of complex varieties and $Z \subseteq X$ a 
    closed complex subvariety of $X$. 
      \begin{enumerate}
          \item The morphism $\pi$ is a {\bf modification}\index{modification} of $X$ if it is proper and 
               bimeromorphic, that is, if it is proper and if there exists a closed nowhere 
               dense subvariety $X'$ of $X$ such that $\pi^{-1}(X')$ is a nowhere dense 
               subvariety of $Y$ and the restriction $\pi: Y \setminus \pi^{-1}(X') \to 
               X \setminus X'$ is an isomorphism. 
          \item If $X'$ is minimal with the previous properties, then $X'$ is called the 
              {\bf indeterminacy locus} of $\pi^{-1}$\index{indeterminacy locus} and $\pi^{-1}(X')$ is 
              called the {\bf exceptional locus} of $\pi$\index{exceptional locus}. 
           \item The {\bf total transform}\index{total transform} 
           $\boxed{\pi^*(Z)}$ of $Z$ by $\pi$ is the complex 
               subspace of $Y$ defined by the preimage by $\pi$ of the ideal sheaf 
               defining $Z$ in $X$. 
            \item Assume that no irreducible component of $Z$ is included in the 
              indeterminacy locus $X'$ of $\pi^{-1}$. Then the {\bf strict transform} of $Z$ by $\pi$ 
              \index{strict transform} is the closure inside $Y$ of $\pi^{-1}(Z \setminus X')$. 
      \end{enumerate}
\end{definition}

The blow up morphisms of surfaces at smooth points are examples of modifications. In the case 
of the blow up $\pi: \Sigma \to \C^2_{x,y}$ at the origin, the equations 
(\ref{eq:twochartsblowup}) show that the total transform of a line $Z(y - ax) \subseteq \C^2_{x,y}$, 
for $a \in \C^*$, may be described as $Z(u_2(1 - a u_1)) \subseteq \C^2_{u_1, u_2}$ and 
$Z(v_1(v_2 - a)) \subseteq \C^2_{u_1, u_2}$ in the two charts covering $\Sigma$. As 
$Z(u_2)$ and $Z(v_1)$  describe the exceptional divisor $\pi^{-1}(O)$ in those two charts, 
we see that the strict transform of $Z(y - ax)$ is the fiber of $\tilde{\lambda}$ whose 
equations are $u_1 = a^{-1}$ and $v_2 = a$ in those two charts. 

Assume now that $C$ is a finite sum $\sum_{i \in I} L_i$ of such lines 
$L_i$ passing through the origin in $\C^2$. 
The strict transform of $C$ by $\pi$ is the sum of the strict transforms $\tilde{L}_i$ of those lines 
and the total transform $\pi^*(C)$ is the sum $\pi^{-1}(O) + \sum_{i \in I} \tilde{L}_i$ of the exceptional 
divisor of $\pi$ and of the strict transform of $C$. Therefore, $\pi^*(C)$ is a \emph{normal crossings 
divisor} in the following sense:

\begin{definition} \label{def:normcross}
    Let $S$ be a smooth complex surface and $D$ a divisor on it. This divisor 
    is said {\bf to have normal crossings} or to be a {\bf normal crossings divisor}
    \index{divisor!normal crossings}
    if its support  is locally either a smooth curve or the union of two transversal 
    smooth curves. 
\end{definition}

Coming back to the curve $C = \sum_{i \in I} L_i$ in $\C^2$, the fact that its total transform 
$\pi^*(C)$ is a normal crossings divisor shows that the blow up morphism $\pi: \Sigma \to \C^2$ 
is an \emph{embedded resolution} of $C$, in the following sense: 

\begin{definition}   \label{def:embres}
    Let $C$ be a curve on the smooth complex surface $S$, in the sense 
    of Definition \ref{def:curve}. An {\bf embedded resolution}\index{resolution!embedded} 
    of $C$ is a modification 
    $\tilde{\pi}:   \tilde{S} \to S$ such that: 
       \begin{enumerate}
            \item $\tilde{S}$ is smooth; 
            \item the total transform $\tilde{\pi}^*(C)$ is a normal crossings divisor; 
            \item the strict transform $\tilde{C}$ of $C$ by $\tilde{\pi}$ is smooth.
        \end{enumerate}
\end{definition}

The restriction $\tilde{\pi}_C:   \tilde{C} \to C$  of an embedded resolution $\tilde{\pi}$ 
of $C$ to the strict transform $\tilde{C}$ of $C$ is a \emph{resolution} of $C$ in the 
following sense:

\begin{definition}   \label{def:resolgen}
   Let $X$ be a complex variety. A {\bf resolution}\index{resolution} of $X$ is a modification 
   $\pi : \tilde{X} \to X$ such that $\tilde{X}$ is smooth and the indeterminacy 
   locus of $\pi^{-1}$ is equal to the singular locus of $X$.  
\end{definition}

If $X$ is a complex curve, then a resolution of it is the same as a normalization morphism. 
This is no longer true in higher dimensions, as in each dimension at least $2$, there are 
normal non-smooth complex varieties. For instance, a hypersurface $X$ of $\C^n$ whose 
singular locus has codimension at least $2$ in $X$ is normal (see \cite{O 51}, \cite{A  60}).

Note that the second condition in Definition \ref{def:embres} does not imply the third one. 
For instance, if one takes the folium of Descartes $C \subset \C^2_{x,y}$ defined by the equation 
$x^3 + y^3 = 3xy$, then $C$ is a normal crossings divisor in $\C^2$ (with a single 
singular point at the origin), therefore 
the identity morphism from $\C^2$ to itself satisfies the first two conditions of 
Definition \ref{def:embres} but not the last one, because the strict transform of 
$C$ by it is not smooth, being the curve $C$ itself. 

In order to get an embedded resolution of the folium of Descartes, it is enough to 
blow up $\C^2$ at the origin $O$. More generally, if $C$ is a curve in a smooth 
complex surface $S$ such that at each point $o$ of $C$, the branches of $C$ at $o$ 
are smooth and pairwise transversal, then the morphism obtained by blowing up 
$S$ at all the singular points of $C$ is an embedded resolution of $C$. 
Conversely, as may 
be seen by working with the description (\ref{eq:twochartsblowup}) of the blow up morphism at a point 
in terms of local coordinates, this property of achieving an embedded resolution by blowing 
up distinct points of $S$ characterizes the previous kind of curves. What about curves 
with more complicated singularities? It turns out that they also have embedded resolutions, 
which may be obtained by blowing up points \emph{iteratively} (see \cite[Thm. 3.9]{H 77}, 
\cite[Pages 496-497]{BK 86}, \cite[Thm. 5.4.2]{DJP 00}, \cite[Section 3.7]{CA 00} and 
\cite[Thm. 3.4.4]{W 04}): 

\begin{theorem}   \label{thm:iterativeblowup}
   Let $C$ be a curve on the smooth complex surface $S$. Define $S_0:= S$ and 
   $\pi_0 : S_0 \to S$ to be the identity. Assume that for some $k \geq 0$ one has defined 
   a modification $\pi_k : S_k \to S$ which is not an embedded resolution of $C$. 
   Denote by $B_k \subset S_k$ 
   the set of points at which either the strict transform of $C$ is not smooth 
   or $\pi_k^*(C)$ is not a normal crossings divisor. Define 
   $\psi_k : S_{k+1} \to S_k$ to be the blow up of $S_k$ at the points of $B_k$ 
   and $\pi_{k+1} := \pi_k \circ \psi_k : S_{k+1} \to S$.  
   Then there exists $k \in \N$ such that $\pi_k$ is an embedded resolution of $C$.  
\end{theorem}

If $k$ is chosen minimal such that $\pi_k$ is an embedded resolution of $C$, then 
$\pi_k$ is called the {\bf minimal embedded resolution}\index{resolution!minimal embedded} 
of $C$. It may be shown that 
any other embedded resolution of $C$ factors through it. 

The combinatorial structure of the total transform of $C$ on a given embedded resolution 
$\tilde{\pi} : \tilde{S} \to S$ of $C$ is encoded usually by drawing its \emph{weighted dual graph}: 

\begin{definition}  \label{def:dualgraphembres}
    Let $C$ be a curve on the smooth complex surface $S$ and $\tilde{\pi} : \tilde{S} \to S$ be 
    an embedded resolution of $C$. Its {\bf weighted dual graph}\index{graph!weighted dual} 
    is a finite connected 
    graph whose vertices are labeled by the irreducible components of the total transform 
    $\tilde{\pi}^*(C)$, two vertices being connected by an edge whenever their associated curves  
    intersect on $\Sigma$. The vertices corresponding to the components of the strict transform 
    of $C$ are drawn arrowheaded. The remaining vertices are weighted by the self-intersection 
    numbers on $\Sigma$ of the associated irreducible components of the exceptional locus 
    of $\pi$. 
\end{definition}

How to compute the weights of the dual graph of the embedded resolution 
$\tilde{\pi} : \tilde{S} \to S$? 
If this resolution is obtained iteratively by the process described in Theorem 
\ref{thm:iterativeblowup}, then one may compute recursively the self-intersection numbers 
of the components of the exceptional loci of the modifications $\pi_k$ using 
Corollary \ref{cor:degHopfbis} and (see \cite[Lemma 8.1.6]{W 04}):

\begin{proposition}  \label{prop:changeselfint}
    Let $C$ be a compact curve in the smooth complex surface $S$. Let $o$ be 
    a point of $C$ of multiplicity $m \in \N$. If $\pi: \Sigma \to S$ is the blow up of $S$ at $o$, 
    then the self-intersection $\tilde{C}^2$ in $\Sigma$ of the strict transform $\tilde{C}$ of $C$ 
    by $\pi$ is related to the self-intersection $C^2$ of $C$ in $S$ by the formula $\tilde{C}^2 = C^2 - m$. 
\end{proposition}

\subsection{The minimal embedded resolution of the semicubical parabola}
\label{ssec:firtsexample}
$\:$  
\medskip

In this subsection we show how to achieve the minimal embedded resolution 
of the \emph{semicubical parabola} using the algorithm described in Theorem 
\ref{thm:iterativeblowup} and how to compute its weighted dual graph using Proposition  
\ref{prop:changeselfint}. It is an expansion of \cite[Example V.3.9.1]{H 77}. 

\medskip
  The {\bf semicubical parabola} is the curve $P \hookrightarrow \C^2_{x,y}$ defined 
  as the vanishing locus of the polynomial $p(x,y) :=  y^2 - x^3$.  
  The germ of $P$ at the origin $O$ is a branch 
  called sometimes the {\bf standard cusp}. Due to the following \emph{Jacobian criterion} 
  (see \cite[Theorem 4.3.6]{DJP 00} 
  for a generalization in arbitrary dimension and codimension), 
  the origin is the only singular point of $P$. 
  
  \begin{theorem}   \label{thm:jacrit} {\bf (Jacobian criterion)} \index{Jacobian criterion}
      Let $C$ be a reduced curve in an open set of $\C^2_{x,y}$, defined by a holomorphic 
      function $f : U \to \C$. Then the singular locus $\mathrm{Sing}(C)$ is the zero locus 
      $Z(f, \partial_x f, \partial_y f)$. 
  \end{theorem}
  
  We want to construct a sequence of blow ups which leads to an embedded resolution of $P$ 
  by following the algorithm described in Theorem \ref{thm:iterativeblowup}, whose notations we 
  use. Therefore, denote by $\boxed{\pi_1} : S_1 \to \C^2$ the blow up of  the origin 
$\boxed{O_0}:= O$ of 
  $\C^2_{x,y}$, instead of $\pi : \Sigma \to \C^2$ as in Definition \ref{def:blowup}. We use the standard 
charts $\C^2_{u_1, u_2}$ and $\C^2_{v_1, v_2}$ for computations on $S_1$, 
the blow up morphism $\pi_1$ being then described by the changes of variables 
(\ref{eq:twochartsblowup}). The total transform $\pi_1^*(P)$ 
of $P$ by $\pi_1$ is defined by the composition $p \circ \pi_1$, which is expressed as follows 
in the two charts:  
    \begin{equation} \label{eq:changefunct1} 
           p(u_1u_2, u_2) = u_2^2(1 - u_1^3 u_2) , \:  \: \: \:   p(v_1, v_1 v_2) =  v_1^2(v_2^2 - v_1). 
     \end{equation}
As the curve $P$ is smooth outside the origin, its 
strict transform $\boxed{P_1}$ by $\pi_1$ is also smooth 
outside the exceptional divisor. This strict transform intersects the exceptional divisor 
$\pi_1^{-1}(O)$ only in the chart $\C^2_{v_1, v_2}$, because its equations in the two 
charts are  $1 - u_1^3 u_2  =0$ and $v_2^2 - v_1   =0$.  
  The second equation is that of a parabola, therefore it defines a smooth curve. This shows that 
  the strict transform $P_1$ is everywhere smooth. Therefore, the restriction of the 
  morphism $\pi_1$ to  the curve $P_1$ is \emph{a resolution of $P$}, in the sense of 
  Definition \ref{def:resolgen}. But it is not an 
  \emph{embedded resolution} in the sense of Definition \ref{def:embres}, because 
  the total transform $\pi_1^*(P)$ is not a normal crossings divisor at the origin $\boxed{O_1}$ 
  of the chart $\C^2_{v_1, v_2}$. Indeed, the strict transform 
  $P_1 \cap \C^2_{v_1, v_2} = Z(v_2^2 - v_1)$ 
  and the exceptional divisor $\pi_1^{-1}(O) \cap \C^2_{v_1, v_2} = Z(v_1)$ are tangent at $O_1$. 
  
  Blow up now the point $O_1$, getting a new surface $\boxed{S_2}$. 
  Let $\boxed{\psi_1} : S_2 \to S_1$ be this blow up 
  morphism. The preimage $\psi_1^{-1}(\C^2_{v_1, v_2})$ of the chart $\C^2_{v_1, v_2}$ of 
  $S_1$ may be covered by two charts $\C^2_{w_1, w_2}$ and $\C^2_{z_1, z_2}$, in which  
  the morphism $\psi_1$ is described by the following analogs of equations (\ref{eq:twochartsblowup}): 
   \begin{equation}  \label{eq:twochartsblowupbis}  
         \left\{ \begin{array}{l}
                          v_1 = w_{1} w_{2} 
                            \\
                          v_2 = \: \: \: \: \:  w_{2},
                    \end{array} \right.   \quad  \mbox{ and } \quad 
           \left\{ \begin{array}{l}
                          v_1 = z_{1}  
                            \\
                          v_2 = z_{1} z_{2}.
                    \end{array} \right. 
       \end{equation}
  In order to cover completely the surface $S_2$, one needs also the chart $\C^2_{u_1, u_2}$ 
  of $S_1$, which is left unchanged by the blow up morphism $\psi_1$ because $O_1$ 
  does not appear in it. 
  
  Denote  $\boxed{\pi_2} := \pi_1 \circ \psi_1 : S_2 \to \C^2$. Using equations (\ref{eq:changefunct1}) 
  we see that:
     \begin{equation} \label{eq:secondtotal} 
                  p\circ \pi_2(w_{1}, w_2)  
                  = w_1^2 w_2^3(w_2  - w_1), 
                  \quad  \mbox{ and } \quad
                  p\circ \pi_2(z_1,   z_{2})  
                  = z_1^3 (z_1 z_2^2 -1).
      \end{equation}
  Therefore, the strict transform $\boxed{P_2}$ of $P_1$ by $\pi_2$ intersects again the 
  exceptional divisor only in one of those charts, namely $\C^2_{w_1, w_2}$. The total transform 
  $\pi_2^*(P) \hookrightarrow S_2$ is still not a normal crossings divisor, because its germ at 
  the origin $\boxed{O_2}$ 
  of $\C^2_{w_1, w_2}$ has three branches: $Z(w_1), Z(w_2), Z(w_2 - w_1)$, as shown by 
  equation (\ref{eq:secondtotal}). One needs 
  to blow up also this point, getting the morphisms $\boxed{\psi_2} : \boxed{S_3} \to S_2$ and 
  $\boxed{\pi_3} := \pi_2 \circ \psi_2 : S_3 \to \C^2$. 
  The blow up $\psi_2$ may be described using the following analogs of equations    
  (\ref{eq:twochartsblowup}) above the chart $\C^2_{w_1, w_2}$:
    \begin{equation}  \label{eq:twochartsblowupter}  
         \left\{ \begin{array}{l}
                          w_1 = s_{1} s_{2} 
                            \\
                          w_2 = \: \: \: \: \:  s_{2},
                    \end{array} \right.    \quad  \mbox{ and } \quad           \left\{ \begin{array}{l}
                          w_1 = t_{1}  
                            \\
                          w_2 = t_{1} t_{2}.
                    \end{array} \right. 
       \end{equation}
   Composing these changes of variables with the second equation (\ref{eq:secondtotal}), we get:
        \[  p\circ \pi_3(s_{1} , s_2) = s_1^2 s_2^6(1- s_1), \:   \: \: \:
            p\circ \pi_3(t_1,   t_{2}) = t_1^6 t_2^3(t_2- 1).  \]
    In both charts of $S_3$ the total transform $\pi_3^*(P)$ is a normal crossings divisor. 
    This being the case also in the remaining charts $\C^2_{u_1, u_2}$ and  $\C^2_{z_1, z_2}$, 
    we see that $\pi_3$ is an embedded resolution of singularities of 
    the semicubical parabola 
    $P$. By Theorem  \ref{thm:iterativeblowup}, it is the minimal such resolution.

    We illustrated the previous sequence of blow ups in Figure \ref{fig:semicubprocess}. 
    We drew whenever possible the support of the total transform of $P$ in the chart whose origin 
    is contained in the strict transform of $P$.  In the four charts the strict transforms 
    of $P$ are drawn in orange and the defining polynomial is written near it. We have used 
    systematically the same color for a point $O_i$ which is blown up by a morphism 
    $\psi_i$, for the exceptional divisor $E_i$ created by this blow up and for its strict 
    transforms $E_{i,j}$ by the next blow ups. Notice that the component 
    $E_{0,2}$ appears on the chart 
    $\C^2_{t_1, t_2}$, but it does not appear on the chart  $\C^2_{s_1, s_2}$, 
    represented on the right of Figure \ref{fig:semicubprocess}.

     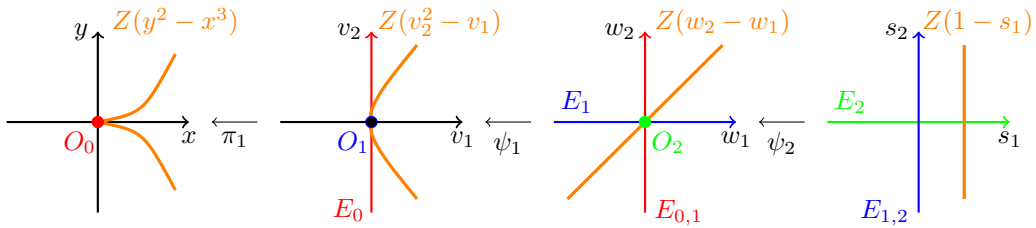
\begin{figure}[h!]
    \begin{center}
\begin{tikzpicture}[scale=0.6]

\begin{scope}[shift={(0,0)}]
     \draw [->, color=black, thick](-2,0) -- (2,0);
\draw [->, color=black, thick] (0,-2)--(0,2);

\node [below] at (2,0) {$x$};
\node [left] at (0,2) {$y$};

\draw[very thick, color=orange](0,0) .. controls (1,0.2) ..(1.7,1.5); 
\draw[very thick, color=orange](0,0) .. controls (1,-0.2) ..(1.7,-1.5); 
\node [above, color=orange] at (1.7,1.7) {$Z(y^2 - x^3)$};

\node[draw,circle, inner sep=1.5pt,color=red, fill=red] at (0,0){};
\node [below, color=red] at (-0.4,0) {$O_0$};
\end{scope}


\begin{scope}[shift={(6,0)}]
     \draw [->, color=black, thick](-2,0) -- (2,0);
\draw [->, color=red, thick] (0,-2)--(0,2);

\node [below] at (2,0) {$v_1$};
\node [left] at (0,2) {$v_2$};

\draw[very thick, color=orange](1, -1.7) .. controls (-0.38, 0) ..(1, 1.7); 
\node [above, color=orange] at (1.5, 1.7) {$Z(v_2^2- v_1)$};

\node [below, color=blue] at (-0.4,0) {$O_1$};
\node [below, color=red] at (-0.5, -1.5) {$E_0$};
\node[draw,circle, inner sep=1.5pt,color=blue, fill=black] at (0,0){};

\end{scope}


\begin{scope}[shift={(12,0)}]
     \draw [->, color=blue, thick](-2,0) -- (2,0);
\draw [->, color=red, thick] (0,-2)--(0,2);
\node [right, color=red] at (0,-2) {$E_{0,1}$};
\node [above, color=blue] at (-1.5,0) {$E_1$};

\node [below] at (2,0) {$w_1$};
\node [left] at (0,2) {$w_2$};

\draw[very thick, color=orange](-1.7,-1.7) -- (1.7, 1.7); 
\node [above, color=orange] at (1.7,1.7) {$Z(w_2- w_1)$};

\node[draw,circle, inner sep=1.5pt,color=green, fill=green] at (0,0){};
\node [below, color=green] at (0.5,0) {$O_2$};
\end{scope}


\begin{scope}[shift={(18,0)}]
     \draw [->, color=green, thick](-2,0) -- (2,0);
\draw [->, color=blue, thick] (0,-2)--(0,2);
\node [above, color=green] at (-1.5,0) {$E_2$};
\node [left, color=blue] at (0,-2) {$E_{1,2}$};

\node [below] at (2,0) {$s_1$};
\node [left] at (0,2) {$s_2$};

\draw[very thick, color=orange](1,-1.7) -- (1, 1.7); 
\node [above, color=orange] at (1.3,1.7) {$Z(1- s_1)$};
\end{scope}


      \draw[<-](2.5,0)--(3.5,0);
     \node [below] at (3,0) {$\pi_1$}; 
     
      \draw[<-](8.5,0)--(9.5,0);
     \node [below] at (9,0) {$\psi_1$}; 
     
      \draw[<-](14.5,0)--(15.5,0);
     \node [below] at (15,0) {$\psi_2$};

\end{tikzpicture}
\end{center}
 \caption{Building iteratively the minimal embedded resolution of 
       the semicubical parabola}
\label{fig:semicubprocess}
   \end{figure}

    Let us compute now the weighted dual graph of $\pi_3$. 
    For every $i \in \{0, 1, 2\}$, denote by $\boxed{E_i} \hookrightarrow S_{i+1}$ 
    the exceptional divisor of the blow up of the point $O_i \in S_i$.  
    If $0 \leq i < j \leq 2$, denote by $\boxed{E_{i, j}}$ the strict transform of 
    $E_i$ on the surface $S_{j+1}$ by the modification $\psi_{j} \circ \cdots \circ \psi_i : S_{j+1} \to S_i$. 
    By Corollary \ref{cor:degHopfbis}, one has $E_0^2 = E_1^2 = E_2^2 =-1$. Equations 
    (\ref{eq:changefunct1}) and (\ref{eq:secondtotal}) imply that 
    $O_1 \in E_0$ and $O_2 \in E_1 \cap E_{0, 1}$, because in the chart 
    $\C^2_{v_1, v_2}$ one has $E_{0,1} = Z(v_1)$, $O_1 = (0, 0)$ 
    and in the chart $\C^2_{w_1, w_2}$ one 
    has $E_1 = Z(w_2)$, $E_{0, 1}= Z(w_1)$, $O_2 = (0, 0)$. 
    Using Theorem \ref{prop:changeselfint}, we get $E_{0,2}^2 = E_0^2 -2 = -3$ 
    and $E_{1,2}^2 = E_1^2 -1 = -2$.  
    Therefore, the weighted dual graph of the minimal embedded resolution 
    $\pi_3 : S_3 \to \C^2$ of the semicubical parabola $P$ is as shown in Figure 
    \ref{fig:dualsemicub}. Near the arrowhead vertex corresponding to the strict transform 
    of $P$, we have written the defining  function of the semicubical parabola.

    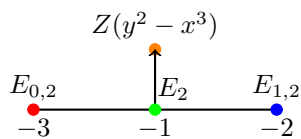
\begin{figure}[h!]
    \begin{center}
\begin{tikzpicture}[scale=0.8]
\node[draw,circle, inner sep=1.5pt,color=orange, fill=orange] at (4,1){};

\draw [-, color=black, thick](2,0) -- (6,0);
\draw [->, color=black, thick] (4,0)--(4,1);

\node [below] at (2,0) {$-3$};
\node [below] at (4,0) {$-1$};
\node [below] at (6,0) {$-2$};

\node [above] at (2,0) {$E_{0,2}$};
\node [above] at (4.3,0) {$E_2$};
\node [above] at (6,0) {$E_{1,2}$};

\node[draw,circle, inner sep=1.5pt,color=red, fill=red] at (2,0){};
\node[draw,circle, inner sep=1.5pt,color=green, fill=green] at (4,0){};
\node[draw,circle, inner sep=1.5pt,color=blue, fill=blue] at (6,0){};

\node [above] at (4,1) {$Z(y^2 - x^3)$};

\end{tikzpicture}
\end{center}
 \caption{The weighted dual graph of the minimal embedded resolution of 
       the semicubical parabola}
\label{fig:dualsemicub}
   \end{figure}

    \medskip
    The previous computations involve many charts, therefore many 
    variables and changes of variables. It is easy to get lost in them. One feels the 
    need of being able to arrive at the final result, the weighted dual graph, without 
    such manipulations. In the next subsection we show how to achieve this goal 
    by a simpler method, without working with charts. 
    We will explain the method using an apparently more complicated example, with two branches. 
    After reading it, we suggest the reader to verify that in the case  
    of the semicubical parabola, the method leads again to the weighted tree of 
    Figure \ref{fig:dualsemicub}.

\subsection{A Newton non-degenerate reducible example}
\label{ssec:redexample}
$\:$  
\medskip

In this subsection we present on a simple example of \emph{Newton non-degenerate} 
plane curve singularity the notions of \emph{Newton polygon} and 
\emph{Newton fan} of a non-zero function $f(x,y) \in \C[[x,y]]$.  
Then we introduce the associated \emph{lotus} and we show 
how to construct from it  the weighted dual graph of the minimal embedded resolution 
of the given singularity.  These notions are briefly introduced in this section to illustrate 
our second elementary example and 
will be revisited formally in Sections \ref{sec:tores} and \ref{sec:embres}. 
\medskip

Let $(C, O) \hookrightarrow (\C^2_{x,y}, O)$ be the plane curve singularity defined by the function: 
     \begin{equation} \label{eq:factoriz} 
             f(x,y):=(y^2-4x^3)(y^3-x^7).
      \end{equation}
     
 It is the sum of two branches, defined by the equations $y^2-4x^3=0$ and 
 $y^3-x^7 =0$ respectively. Thinking of them as polynomial equations in the unknown $y$, 
 as explained in Subsection \ref{ssec:NPthm}, they have degrees $2$ and $3$. 
 Their respective sets of roots are $\{\pm 2x^{3/2}  \}$ and $\{\omega x^{7/3}\}$, 
 where $\omega$ varies among the complex cubic roots of $1$. 
 We could express readily in terms of $x$ the roots of the equation $f(x,y)=0$ seen as a 
 quintic polynomial  equation in the variable $y$, because we knew a factorization 
 of $f(x,y)$ into binomial factors. Is it possible to reach the same objective if one starts instead 
 from the following expanded expression of $f$? 
      \begin{equation} \label{eq:expandedex}
              f(x,y)= y^5-4x^3y^3-x^7y^2+4x^{10}.
      \end{equation}
      
    By the Newton-Puiseux Theorem \ref{thm:NPthmbasic}, we know a priori that the roots of 
    $f(x,y)$ may be expressed as \emph{Newton-Puiseux series}. Newton's fundamental 
    insight was that one may always compute the leading terms of such series only by 
    looking at special terms of $f$ (see  the beginning of Subsection \ref{ssec:HAtoroidal}).       
    Let us explain this insight in the case of the polynomial (\ref{eq:expandedex}), 
    forgetting its factorization 
      (\ref{eq:factoriz}). Denote by $cx^{\nu}$ the leading term (that is, the term of least degree) 
      of such a series, where $c \in \C^*$ and $\nu >0$. We have 
      the equality:
         \begin{equation}   \label{eq:Newt1step}
              f(x, cx^{\nu} + o(x^{\nu}))=0.
         \end{equation}
  Using formula (\ref{eq:expandedex}), this equality may be rewritten as:
          $$     (cx^{\nu} + o(x^{\nu}))^5 -4 x^3 (cx^{\nu} + o(x^{\nu}))^3 
                             - x^7 (cx^{\nu} + o(x^{\nu}))^2 + 4 x^{10} =0,  $$
  that is, as:           
        \begin{equation} \label{eq:initerm}    
                        \left(c^5 x^{5 \nu} + o(x^{5 \nu}) \right) 
                               + \left(-4 c^3 x^{3 + 3 \nu} + o(x^{3 + 3 \nu}) \right)  
                               + \left( - c^2 x^{7 + 2 \nu} + o(x^{7 + 2 \nu}) \right) + 4 x^{10} =0.
        \end{equation}
  The left-hand side of this equation is a sum of four series, whose leading 
  exponents are $5 \nu$, $3 + 3 \nu$, $7 + 2 \nu$, $10$, since $c \neq 0$. The 
  fundamental observation of Newton was that \emph{if the sum} (\ref{eq:initerm}) \emph{vanishes, 
  then the minimal value of those four exponents is reached at least twice}. 
  
  Now, these four exponents may be expressed as the products 
  $(1 , \nu) \cdot (a,b) := a + b \nu$, 
  where $(a,b)$ varies among the exponents $(a,b) \in \N^2$ of the monomials 
  $x^a y^b$ appearing in the expanded form (\ref{eq:expandedex}) of $f(x,y)$, 
  that is, as the evaluations of the linear form $l_{\nu}(a,b) :=  a + b \nu$ on the  
  \emph{support} $\Supp(f)$ of  the series $f(x,y)$. 
  In our example  the support is finite, but it may be infinite if one allows $f$ to be a power 
  series in the variables $x,y$. It is at this point that \emph{convex geometry} enters 
  into the game, through the following property (which is a consequence of \cite[Assertion III.1.5.2]{O 97}):

   \begin{proposition}  \label{prop:infinconv}
         Let $\Supp$ be a subset of $\N^2$. If $l$ is a linear form with non-negative 
         coefficients on $\R^2$, then its restriction to $\Supp$ achieves its minimum 
         precisely on the subset of $\Supp$ lying on a face of the convex hull 
         $\mathrm{Conv}(\Supp + \R_+^2)$. 
  \end{proposition}
  
  Coming back to equation (\ref{eq:initerm}), we see that the linear form $l_{\nu} (a,b) = a + b \nu$, 
  which computes the leading exponents of the terms 
  appearing in the left-hand side of (\ref{eq:initerm}),  indeed has non-negative coefficients. 
  Therefore, the hypotheses 
  of Proposition  \ref{prop:infinconv} are satisfied. This shows that the minimal  value 
  $\min \left\{  5 \nu, 3 + 3 \nu, 7 + 2 \nu, 10  \right\}$  
   is achieved on a face of the convex hull $\mathrm{Conv}(\Supp(f) + \R_+^2)$. 
   This convex hull, called the \emph{Newton polygon} $\cN(f)$ of 
   $f \in \C[[x,y]]$ (see Definition \ref{def:Npolalg} below), is represented in Figure \ref{fig:NPfirstexample}.                                
It has three vertices, which are $(0, 5)$, $(3, 3)$, $(10, 0)$, 
corresponding to the terms $y^5, -4 x^3 y^3$ and $4 x^{10}$ of the expansion 
(\ref{eq:expandedex}). It has two compact edges
$K_1 := [(0,5), (3,3)]$ and $K_2 := [(3,3), (10, 0)]$. If the minimum is to be achieved at least twice 
on $\Supp(f)$, then it must be achieved on one of those two compact edges, because 
$\nu >0$. This means that the linear form $l_{\nu}$ 
must be orthogonal to one of those compact edges.  There are therefore two 
possibilities: 

 \noindent
$\bullet$     Either $l_{\nu}$ achieves its minimum on $K_1$, 
        which means that $(1, \nu)$ is orthogonal to it. In other words  
        $(1, \nu) \cdot (3-0, 3-5) =0$, that is, $\nu = 3/2$. Writing that the sum of the terms 
        of the left-hand side of Equation (\ref{eq:initerm}) whose leading exponents 
        achieve the minimum vanishes, one gets the 
        equation $ c^5  -4 c^3 =0$. As $c \neq 0$, this is equivalent to the equation 
        $c^2 =4$, hence  $c=\pm 2$.
    \medskip 
    
    \noindent 
   $\bullet$
     Or $l_{\nu}$ achieves its minimum on $K_2$.  In other words  
        $(1, \nu) \cdot (10-3, 0-3) =0$, that is, $\nu = 7/3$. One gets then the 
        equation $ - 4 c^3 + 4 = 0$. That is, $c$ varies now among the cubic roots of $1$. 
\medskip

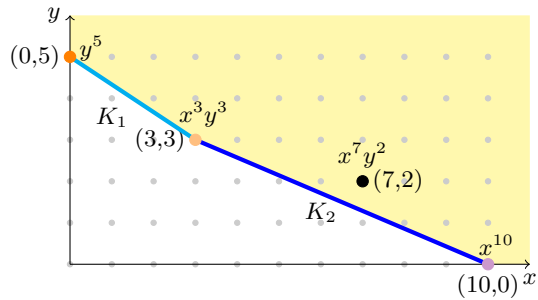
\begin{figure}[h!] 
\begin{center}
\begin{tikzpicture}[x=0.55cm,y=0.55cm] 
\tikzstyle{every node}=[font=\small]
\fill[fill=yellow!40!white] (0,5) --(3,3)-- (10,0) -- (11, 0) -- (11,6) -- (0,6) --cycle;
\foreach \x in {0,1,...,10}{
\foreach \y in {0,1,...,5}{
       \node[draw,circle,inner sep=0.7pt,fill, color=gray!40] at (1*\x,1*\y) {}; }
   }
\draw[->] (0,0) -- (11,0) node[right,below] {$x$};
\draw[->] (0,0) -- (0,6) node[above,left] {$y$};
\draw[thick] (10,0) node[below] {(10,0)};
\draw[thick] (0,5) node[left] {(0,5)};
\draw[thick] (3,3) node[left] {(3,3)};
\draw[thick] (7,2) node[right] {(7,2)};
\draw [ultra thick, color=blue](10,0) -- (3,3);
\draw [ultra thick, color=cyan](3,3) -- (0,5);
\node[draw,circle,inner sep=1.5pt,fill=violet!40, color=violet!40] at (10,0) {};
\node[draw,circle,inner sep=1.5pt,fill=orange!50,color=orange!50] at (3,3) {};
\node[draw,circle,inner sep=1.5pt,fill=orange, color=orange] at (0,5) {};
\node[draw,circle,inner sep=1.5pt,fill] at (7,2) {};
\node [above] at (10.2,0) {$x^{10}$};
\node [right] at (0,5.2) {$y^{5}$};
\node [above] at (3.2,3.1) {$x^3y^{3}$};
\node [above] at (7,2.1) {$x^7y^{2}$};
\node [below] at (1,4) {$K_1$};
\node [below] at (6,1.7) {$K_2$};

 \end{tikzpicture}
\end{center}
\caption{The Newton polygon of the series $f(x,y)=(y^2-2x^3)(y^3-x^7)$}  
   \label{fig:NPfirstexample}
    \end{figure}

It follows that the  possible leading terms  
of a Newton-Puiseux series $\eta$ in the variable $x$ such that $f(x, \eta)=0$ belong to the union  
$\{\pm 2x^{3/2} \} \cup \{ \omega x^{7/3}: \:  \omega^3 =1\}.$ One recognizes the 
roots 
from the factorization (\ref{eq:factoriz}). Newton's  
method shows that \emph{those are the leading terms of the roots $y(x)$ of the equation 
$g(x,y)=0$, for any $g \in \C[[x,y]]$ whose Newton polygon is the same as 
$\cN(f)$, and whose restrictions to the compact sides of the polygon coincide with the 
analogous restrictions for $f$}.  Any such function $g$ defines a \emph{Newton non-degenerate} 
singularity (see Definition \ref{def:newtnondegen} below), because both equations 
$c^2 =4$ and  $-4 c^3 + 4=0$ obtained by restricting $g$ 
to the compact edges of its Newton polygon have simple roots. 
Variants of Newton's previous line of thought will be followed again in the proofs 
of Propositions \ref{prop:imptrop} and \ref{prop:propstrict} below.

In general, for any series $f(x,y)$, 
once a first term $cx^{\nu}$ of a potential root of $f(x,y)=0$ is computed, 
one may perform a formal change of variables and compute a second term. 
Newton explained that one could compute as many terms as needed, but it was 
Puiseux who proved carefully that by pushing this iterative process to its limit, one gets 
true roots of the equation, which are Newton-Puiseux series.  Moreover, he 
proved that whenever one starts from a convergent function $f$, one gets only roots 
of the form $\xi(x^{1/p})$, where $\xi(t) \in \C[[t]]$ is convergent and $p \in \N^*$. 
This approach leads to a proof of the Newton-Puiseux Theorem 
\ref{thm:NPthmbasic},  different from the one given above (see Remark \ref{rem:Puiseux}).

\medskip

 Let us come back to our example.
It turns out that in this \emph{Newton non-degenerate} case, the weighted dual graph of 
the minimal embedded resolution is determined by the Newton polygon 
$\cN(f)$. In fact, one needs only 
the \emph{inclinations} of its compact edges. This information is encoded in the associated 
\emph{Newton fan}, obtained by subdividing the first quadrant along the rays orthogonal 
to the compact edges of $\cN(f)$
(see the left side of  Figure \ref{fig:NFfirstexample} and Definition \ref{def:nfan} below). 
Consider now inside the first quadrant all the triangles with vertices $f_1, f_2, f_1 + f_2$,  
where $(f_1, f_2)$ is a basis of the ambient lattice $\Z^2$. The edges 
of those triangles may be drawn recursively by starting from the segment 
$[e_1, e_2]$ which joins the elements of the canonical basis  $(e_1, e_2)$ and, each time a new segment $[f_1, f_2]$ is drawn, by drawing also the segments $[f_1,  f_1 + f_2]$ and 
$[f_2 , f_1 + f_2]$. If one performs this construction only whenever the interior of the segment 
$[f_1, f_2]$ intersects one of the rays of the Newton fan, 
one gets its associated \emph{lotus}, represented on the right side of Figure 
\ref{fig:NFfirstexample}.

\begin{figure}[h!] 
\begin{center}
\begin{tikzpicture}[x=0.6cm,y=0.6cm] 
\tikzstyle{every node}=[font=\small]

 \begin{scope}[shift={(0,0)},scale=1]

 \fill[fill=orange] (0,0) --(4,0)-- (4,6) --cycle;
  \fill[fill=orange!50] (0,0) -- (4,6) --(4,7)-- (3,7)--cycle;
    \fill[fill=violet!40] (0,0) -- (0,7) --(3,7)--cycle;
\draw [->, color=black](0,0) -- (0,7);
\draw [->, color=black](0,0) -- (4,0);
\draw [-, ultra thick, color=cyan](0,0) -- (4,6);
\draw [-, ultra thick, color=blue](0,0) -- (3,7);

\foreach \x in {1,2,...,4}{
\foreach \y in {1,2,...,7}{
       \node[draw,circle,inner sep=0.7pt,fill, color=gray!40] at (1*\x,1*\y) {}; }
   }
\end{scope}

\begin{scope}[shift={(10,0)},scale=1]
\draw [->](0,0) -- (0,7);
\draw [->](0,0) -- (4,0);

\draw[fill=pink!40](1,0) -- (0,1) -- (1,1)  --cycle;
\draw[fill=pink!40](0,1) -- (1,1) -- (1,2) --cycle;
\draw[fill=pink!40](1,1) -- (1,2) -- (2,3) --cycle;
\draw[fill=pink!40](0,1) -- (1,2) -- (1,3) --cycle;
\draw[fill=pink!40](1,2) -- (1,3) --(2,5)--cycle;
\draw[fill=pink!40](1,2) -- (2,5) --(3,7)--cycle;

\draw [-, ultra thick, color=orange](0,1) -- (3,7) -- (1,2) -- (2,3) --(1,1)--(1,0);

\foreach \x in {1,2,...,4}{
\foreach \y in {1,2,...,7}{
       \node[draw,circle,inner sep=0.7pt,fill, color=gray!40] at (1*\x,1*\y) {}; }
   }

\node[draw,circle, inner sep=1.5pt,color=red, fill=red] at (2,3){};
\node[draw,circle, inner sep=1.5pt,color=red, fill=red] at (3,7){};

\node [left] at (3,7) {$p(\frac{7}{3})$};
\node [right] at (2,3) {$p(\frac{3}{2})$};

\node [left] at (1.3,-0.3) {$e_1$};
\node [left] at (0,1) {$e_2$};
\draw [->, very thick, red] (1,0)--(0.5, 0.5);
\draw [-, very thick, red] (0.5, 0.5)--(0,1);

\end{scope}
 \end{tikzpicture}
\end{center}
\caption{The Newton fan of $f(x,y)=(y^2- 4 x^3)(y^3-x^7)$ and its associated lotus}  
   \label{fig:NFfirstexample}
    \end{figure}
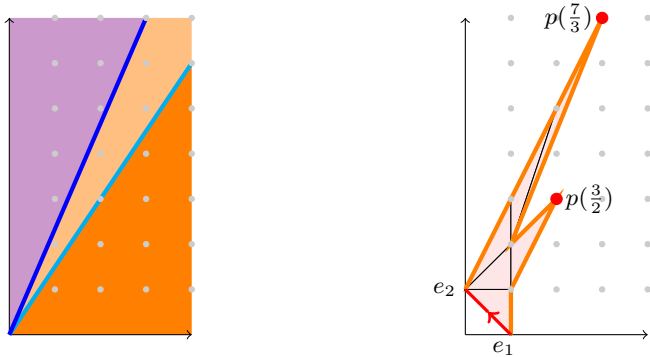

In fact, one needs to attach to it new arrowhead vertices corresponding  to the branches 
of $C$, as shown in Figure \ref{fig:Lotusex}. In this figure the lotus was redrawn as 
an abstract simplicial complex, without representing its precise embedding in the 
plane $\R^2$. This abstract simplicial structure is sufficient for seeing how it contains  
the weighted dual graph of 
the minimal embedded resolution of $Z(xy (y^2-4x^3)(y^3-x^7))$ as part of its boundary. 
The self-intersection number of an exceptional divisor is simply the opposite of the number of 
triangles containing the vertex representing this divisor (compare Figures 
\ref{fig:Lotusex} and \ref{fig:simpledualgraphs}).

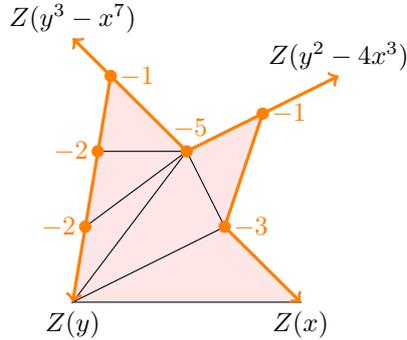
\begin{figure}[h!]
    \begin{center}
\begin{tikzpicture}[scale=0.5]
\draw[fill=pink!40](0,0) -- (6,0)--(4,2)--(5,5)--(3,4)--(1,6)--cycle;
\draw [color=black] (0,0) --(4,2)--(3,4)--cycle;
\draw [color=black] (0.33,2) --(3,4)--(0.66,4)--cycle;
\draw [->, color=orange, very thick] (1,6)--(0,0);
\draw [-, color=orange, very thick] (1,6) --(3,4);
\draw [-, color=orange, very thick] (3,4) --(5,5);
\draw [-, color=orange, very thick] (5,5)--(4,2);
\draw [->, color=orange, very thick] (4,2)--(6,0);
\draw [->, color=orange, very thick] (1,6)--(0,7);
\draw [->, color=orange, very thick] (5,5)--(7,6);
\node [below] at (0,0) {$Z(y)$};
\node [below] at (6,0) {$Z(x)$};
\node [above] at (0,6.9) {$Z(y^3-x^7)$};
\node [above] at (7,5.9) {$Z(y^2-4x^3)$};
\node[draw,circle, inner sep=1.5pt,color=orange, fill=orange] at (1,6){};
\node[draw,circle, inner sep=1.5pt,color=orange, fill=orange] at (0.66,4){};
\node[draw,circle, inner sep=1.5pt,color=orange, fill=orange] at (0.33,2){};
\node[draw,circle, inner sep=1.5pt,color=orange, fill=orange] at (3,4){};
\node[draw,circle, inner sep=1.5pt,color=orange, fill=orange] at (5,5){};
\node[draw,circle, inner sep=1.5pt,color=orange, fill=orange] at (4,2){};
\node[right] at (1,6) { \textcolor{orange}{$-1$}};
\node[left] at (0.66,4) { \textcolor{orange}{$-2$}};
\node[left] at (0.33,2) { \textcolor{orange}{$-2$}};
\node[above] at (3.1,4.1) { \textcolor{orange}{$-5$}};
\node[right] at (5,5) { \textcolor{orange}{$-1$}};
\node[right] at (4,2) { \textcolor{orange}{$-3$}};
\end{tikzpicture}
\end{center}
  \caption{ The lotus of $f(x,y)=(y^2-  4x^3)(y^3-x^7)$}
 \label{fig:Lotusex}
 \end{figure}

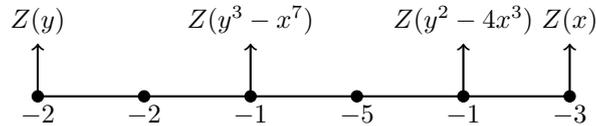
\begin{figure}[h!]
    \begin{center}
\begin{tikzpicture}[scale=0.7]
\draw [-, color=black, thick](0,0) -- (10,0);
\draw [->, color=black, thick] (4,0)--(4,1);
\draw [->, color=black, thick] (8,0)--(8,1);
\node [below] at (0,0) {$-2$};
\node [below] at (2,0) {$-2$};
\node [below] at (4,0) {$-1$};
\node [below] at (6,0) {$-5$};
\node [below] at (8,0) {$-1$};
\node [below] at (10,0) {$-3$};

\node[draw,circle, inner sep=1.5pt,color=black, fill=black] at (0,0){};
\node[draw,circle, inner sep=1.5pt,color=black, fill=black] at (2,0){};
\node[draw,circle, inner sep=1.5pt,color=black, fill=black] at (4,0){};
\node[draw,circle, inner sep=1.5pt,color=black, fill=black] at (6,0){};
\node[draw,circle, inner sep=1.5pt,color=black, fill=black] at (8,0){};
\node[draw,circle, inner sep=1.5pt,color=black, fill=black] at (10,0){};

\node [above] at (4,1) {$Z(y^3 - x^7)$};
\node [above] at (8,1) {$Z(y^2 - 4x^3)$};
 \node [above] at (0,1) {$Z(y)$};
\node [above] at (10,1) {$Z(x)$};

\draw [->, color=black, thick] (0,0)--(0,1);
 \draw [->, color=black, thick] (10,0)--(10,1);

\end{tikzpicture}
\end{center}
 \caption{The weighted dual graph of the minimal embedded resolution of 
       $Z(xy f(x,y))$}
\label{fig:simpledualgraphs}
   \end{figure}

\medskip
In the sequel we will associate lotuses to any plane curve singularity $C$ 
(see Definition \ref{def:lotustoroid}). The data needed to construct them will be  
a finite sequence of Newton  polygons generated by a \emph{toroidal pseudo-resolution algorithm} 
(see Algorithm \ref{alg:tores}). 
We will embed analogously inside them the weighted dual graphs 
of associated embedded resolutions of \emph{completions} of the curve 
(see Definition \ref{def:threeres} and Theorem \ref{thm:repsailtor}). 
We will also explain the notions of \emph{fan tree} (see Definition \ref{def:fantreetr}), 
\emph{Enriques diagram} (see Definition \ref{def:infnear}) 
and \emph{Eggers-Wall tree} (see Definition \ref{def:EW}) 
of $C$ or of an associated toroidal pseudo-resolution process 
and we will show that they embed similarly in the corresponding
lotus (see Theorem \ref{thm:repsailtor}).

\section{Toric and toroidal surfaces and their morphisms}
\label{sec:torsurfmorph}
\medskip

In this section we explain basic definitions and intuitions about toric and toroidal  varieties and 
their modifications, which will be used in the subsequent sections in the study of plane curve 
singularities.  Namely, \emph{fans} are introduced in Definition \ref{def:fan}, 
\emph{affine toric varieties} in Definition \ref{def:afftoric},  their \emph{boundaries} in 
Definition \ref{def:boundtoric}, \emph{toric morphisms} in Subsection \ref{ssec:tormod},  
in particular the toric description of $2$-dimensional blow ups in Example \ref{ex:blowupor} 
and  the \emph{category of toroidal varieties} in Subsection \ref{ssec:toroidmod}. 
Subsection \ref{ssec:HAtoric} contains historical information about the development of toric and toroidal 
geometry and about its applications to the study of singularities.


\subsection{Two-dimensional fans and their regularizations}
\label{ssec:fans}
$\:$
\medskip

In this subsection we explain the basic notions of two-dimensional convex geometry 
needed to define toric  varieties in Subsection \ref{ssec:torsurf} and toric morphisms 
in Subsection \ref{ssec:tormod}: \emph{lattices, rational cones} and \emph{fans}. 
For more details about toric geometry one may consult the standard textbooks 
\cite{O 88}, \cite{F 93}, \cite {E 96} and \cite{CLS 11}. 

\medskip

A {\bf lattice}\index{lattice} is a free $\Z$-module of finite rank.
A pair $(a,b) \in \Z^2$ may be seen as an instruction to 
build two kinds of objects: the Laurent monomial $x^a y^b$ and the parametrized monomial curve 
$t \to (t^a , t^b)$. The fact that monomials and curves 
are distinct geometrical objects indicates that it would be good 
to think also in two ways about the pairs $(a,b)$, that is, as coordinates of vectors relative to bases 
in two different  lattices. Those two  lattices are not to be chosen independently 
of each other. Indeed, given a monomial $x^a y^b$ and a parametrized monomial curve 
$t \to (t^c , t^d)$, one may substitute the parametrization in the monomial, getting a new monomial, 
this time in the variable $t$ alone: 
      \begin{equation}   \label{eq:compos}
           (x^a y^b) \circ (t^c, t^d) = t^{ac + bd}.  
      \end{equation}
This indicates that those two lattices should be seen as factors of the domain of definition of 
the unimodular $\Z$-valued bilinear form $(a, b) \cdot (c,d) := ac + bd$, 
that is, that they should be \emph{dual}  lattices. 

In order to distinguish clearly  the roles of these two lattices, 
one denotes them usually by distinct letters, instead of 
simply writing for instance $\Z^2$ and $(\Z^2)^{\vee}$. It became traditional after 
the appearance of Fulton's book \cite{F 93} to denote by $\boxed{M}$ 
the lattice whose elements are exponents of monomials in several variables, 
and by $\boxed{N}$ the dual  lattice, 
whose elements are thought as exponents of parametrized monomial curves 
in the space of the same variables. It is important to allow for changes 
of bases of those $\Z$-modules, corresponding to monomial changes of variables of the form 
$x = u^{\alpha} v^{\gamma} , y = u^{\beta} v^{\delta}$, for which the matrix of exponents is 
{\bf unimodular}\index{unimodular matrix}:
\begin{equation} \label{eq:unimod}
       \left| \begin{array}{cc}
                            \alpha & \gamma \\   \beta & \delta
                   \end{array} \right| = \pm 1 .
\end{equation}                   
This means that one does not have to fix identifications $M = \Z^2$, $N = \Z^2$,  
but instead to allow those identifications to depend on the context. 
Note also that the elements of $N$ may be seen as {\bf weights} for the variables $x,y$. That is, 
if $(c,d) \in N$, one gives the weight $c$ to $x$ and the weight $d$ to $y$, which endows 
the monomial $x^a y^b$ with the weight $ac + bd$ appearing in the equality 
(\ref{eq:compos}). For this reason, $N$ 
is called sometimes the {\bf weight lattice}\index{lattice!weight} associated to the 
{\bf monomial lattice}\index{lattice!monomial} $M$. 

We will call {\bf vectors} the elements of a lattice.
Those non-zero vectors which cannot be written as non-trivial integral multiples 
of other lattice vectors are 
called {\bf primitive}\index{primitive vector}.  
Any non-zero lattice vector $w$ may be written uniquely in the form 
$l_{\Z}(w) \: w'$, with $l_{\Z}(w) \in \N^*$ and $w'$ a primitive lattice vector. 

\begin{definition} \label{def:intlength}
    Let $N$ be a lattice and $w \in N \: \setminus \: \{0\}$. 
   The positive integer $\boxed{l_{\Z}(w)}$ is the {\bf integral length}\index{integral length} of $w$. 
   We extend this definition to the whole lattice $N$ by setting  $l_{\Z}(0) :=0$. For $w_1, w_2 \in N$, 
   the {\bf integral length} $\boxed{l_{\Z}[w_1, w_2] }\in \N$ of the segment $[w_1, w_2]$ is equal to 
   $l_{\Z}(w_2 - w_1) = l_{\Z}(w_1 - w_2)$. 
\end{definition}

If $N$ is a lattice, denote by $\boxed{N_{\R}} : = N \otimes_{\Z} \R$ the real vector space generated by 
$N$. We will say that the elements of $N$ are the {\bf integral points}\index{integral point} 
of the real vector space 
$N_{\R}$. By a {\bf cone of $N$}\index{cone!of a lattice} we will mean a convex 
rational polyhedral cone, that is, a subset of $N_{\R}$ of the form:
  \[ \boxed{\cone{\langle w_1, \dots, w_k \rangle}} := \cone w_1 + \cdots + \cone w_k, \]
 where $w_1, \dots, w_k \in N$.  If the cone does not contain a positive dimensional 
 vector subspace of $N_{\R}$, it is called {\bf strictly convex}\index{cone!strictly convex}. 
 
 If the lattice $N$ is of rank two, then the strictly convex cones are of three sorts, 
 according to their dimensions: 
  
   \noindent
   $\bullet$ The $2$-dimensional cones are of the form $\cone {\langle w_1, w_2 \rangle}$, where 
           $w_1, w_2 \in N$ are non-proportional. In classical geometric terminology, 
           they are strictly convex angles with apex at the origin of $N_{\R}$. 
           
   \noindent
   $\bullet$ The $1$-dimensional cones are the closed half-lines emanating from the origin; we will 
        call them {\bf rays}\index{ray}. 
         
    \noindent
   $\bullet$ There is only one $0$-dimensional cone: the origin of $N$. 
\medskip 

As a particular case of a terminology used in any dimension, one speaks about the 
{\bf faces}\index{face!of a cone} of a given cone $\sigma \subseteq N_{\R}$: those are the  
subsets of $\sigma$ on which the restriction to $\sigma$ of 
a linear form $l \in N_{\R}^{\vee} = M_{\R}$ 
reaches its minimum. The faces of a strictly convex $2$-dimensional cone 
$\cone {\langle w_1, w_2 \rangle}$ are the cone itself, its {\bf edges}\index{edge!of a cone} 
$\cone {w_1}$, 
$\cone {w_2}$ and the origin. The faces of a ray are the ray itself and the origin. 
Finally, the origin has only one face, which is the origin itself.

Endowing the $2$-dimensional lattice $N$ with a basis $(e_1, e_2)$ allows to speak of the 
{\bf slope}  $d / c \in \R \cup\{ \infty \}$ relative to $(e_1, e_2)$ of any vector 
$w= c \: e_1 + d \: e_2  \in N_{\R} \setminus \{0\}$ or of the associated ray 
$\cone w$. In terms of the coordinates $(c,d)$, the integral length $l_{\Z}(w)$ of $w$ 
is equal to the greatest common divisor $\mbox{gcd}(c, d)$. 

\begin{notation}   \label{not:prim}
If the basis $(e_1, e_2)$ of $N$ is fixed and clear from the context, we denote by: 
     \[ 
     \boxed{\sigma_0} :=  
     \cone {\langle e_1, e_2 \rangle} \subseteq N_{\R} \] 
the cone generated by it. 
 If $\lambda \in \Q_+ \cup \{ \infty \}$, we denote by $\boxed{p(\lambda)}$ the 
unique primitive element of the lattice $N$ contained in the cone $\sigma_0$, and  
which has slope $\lambda$.
\end{notation}

 In the sequel it will be important to work with the following special sets of cones, which 
are fundamental in toric geometry:

\begin{definition} \label{def:fan}
     A {\bf fan}\index{fan!of a lattice} of the lattice $N$ is a finite set of strictly convex 
      cones of $N$ which is closed under the operation of taking faces of its cones and 
      such that the intersection of any two of its cones is a face of each of them. 
     The {\bf support}\index{support!of a fan} $| \fan|$ of a fan $\fan$ is the union of its cones. 
     A fan $\fan$ {\bf refines}\index{refinement!of a fan} (or {\bf subdivides})\index{subdivision!of a fan} 
      another fan $\fan'$ 
     if they have the same support and if each cone of $\fan$ is contained in some 
     cone of $\fan'$. A fan {\bf subdivides a cone $\sigma$} if it subdivides the fan formed by its faces.  
    We often denote again by $\sigma$ the fan formed by the faces of a cone $\sigma$, 
    by a slight abuse of notation.
\end{definition}

Let us complete the previous definition, valid in arbitrary rank, with terminology 
and notations specific to rank two: 

\begin{definition}  \label{def:fan2} 
     Let $(e_1, e_2)$ be a basis of the lattice $N$  of rank two and $\sigma_0$ be the 
     associated cone $\cone \langle e_1, e_2 \rangle$. 
     Any fan  $\fan$  subdividing $\sigma_0$ is determined by the finite set of 
     slopes $\mathcal{E} \subset \Q_+^*$ of its rays contained in the interior of 
     $\sigma_0$. In this case we denote the fan by $\boxed{\fan(\mathcal{E})}$ 
     and we call it the {\bf fan of the set}\index{fan!of a set} $\mathcal{E}$.  
     We extend the definition of  $\fan(\mathcal{E})$ to the case where $\mathcal{E}$ 
     contains $0$ or $\infty$, by setting in this case $\fan(\mathcal{E}) := 
     \fan(\mathcal{E} \: \setminus \:  \{0, \infty\})$. If $\mathcal{E} = 
     \{\lambda_1, \dots, \lambda_p\}$, we write also 
     $\boxed{\fan(\lambda_1, \dots, \lambda_p)}$ instead of $\fan(\mathcal{E})$.  
 \end{definition}

Note that $\fan(\emptyset)$ is simply the fan consisting of the cone $\sigma_0$ and its faces.

\begin{definition}  \label{def:regular}
    A cone of a lattice $N$ is called \textbf{regular}\index{regular!cone} 
    if it can be generated by elements which 
    form a subset of a basis of $N$. A fan all of whose cones 
    are regular is called {\bf regular}\index{regular!fan}. 
\end{definition}

It is convenient to set $\cone\langle \emptyset \rangle := \{ 0 \}$. This implies that 
$\{0 \}$ is also a regular cone.

Assume that a basis $(e_1, e_2)$ of the lattice 
$N$ is fixed. If $f_1 = \alpha e_1 + \beta e_2$ and 
$f_2 = \gamma e_1 + \delta e_2$ are two primitive vectors of $N$, then the cone 
$\cone \langle f_1, f_2 \rangle $ generated by them is regular if and only if the matrix of the 
pair $(f_1, f_2)$ in the basis $(e_1, e_2)$ is unimodular, that is, the equality 
 \eqref{eq:unimod} holds.

\begin{example} \label{ex:fanex}
     If $\mathcal{E} = \left\{ 3/5, 2/1, 5/2 \right\}$, then the rays of 
     the fan $\fan(\mathcal{E})$ are represented  in Figure \ref{fig:examfan}. On each ray of the fan 
     which is distinct from the edges of the cone $\sigma_0$, we indicated by a 
     small red disc the unique primitive element of the lattice $N$ lying on it. That is, on the ray of 
     slope $\lambda \in \mathcal{E}$ is indicated the point $p(\lambda)$. The fan $\fan(\mathcal{E})$ 
     contains also $4$ cones of dimension $2$, which are 
     $\cone\langle e_1, p\left(3/5\right)\rangle$, 
     $\cone\langle p\left(3/5\right), p\left(2/1\right) \rangle$, 
     $\cone\langle p\left(2/1\right), p\left(5/2\right) \rangle$, 
     $\cone\langle p\left(5/2\right), e_2 \rangle$. 
   Using the unimodularity criterion above, we see  
   that $\cone\langle p\left(2/1\right), p\left(5/2\right) \rangle$ is the only 
   $2$-dimensional cone of the fan $\fan(\mathcal{E})$
   which  is regular.
\end{example}

\begin{figure}[h!]
     \begin{center}
\begin{tikzpicture}[scale=0.65]

\foreach \x in {0,1,...,5}{
\foreach \y in {0,1,...,5}{
       \node[draw,circle,inner sep=0.7pt,fill, color=gray!40] at (1*\x,1*\y) {}; }
   }
\draw [->](0,0) -- (0,5.5);
\draw [->](0,0) -- (5.5,0);
\draw [-, thick, color=blue](0,0) -- (2,5);
\node [left] at (2.2,5.5) {$p(\frac{5}{2})$};
\draw [-, thick, color=blue](0,0) -- (2.5,5);
\node [left] at (2.3,1.7) {$p(\frac{2}{1})$};
\node [left] at (6.5,2.8) {$p(\frac{3}{5})$};
\draw [-, thick, color=blue](0,0) -- (5,3);

\node [left] at (1.5,-0.4) {$e_1$};
\node [left] at (0,1) {$e_2$};
\node[draw,circle, inner sep=1.5pt,color=black, fill=black] at (1,0){};
\node[draw,circle, inner sep=1.5pt,color=black, fill=black] at (0,1){};
\node[draw,circle, inner sep=1.5pt,color=black, fill=black] at (0,0){};
\node[draw,circle, inner sep=1.8pt,color=red, fill=red] at (2,5){};
\node[draw,circle, inner sep=1.8pt,color=red, fill=red] at (1,2){};
\node[draw,circle, inner sep=1.8pt,color=red, fill=red] at (5,3){};
\end{tikzpicture}
\end{center}
\caption{The fan $\fan\left(3/5, 2/1, 5/2\right)$ and the points 
     $p\left(3/5\right)$, $p\left(2/1\right)$, $p\left(5/2\right)$}  
   \label{fig:examfan}
    \end{figure}

The following result is specific for lattices of rank two (see \cite[Prop. 1.19]{O 88}):

\begin{proposition}   \label{prop:existsmin}
  If the lattice $N$ is of rank two, any fan relative to $N$ has a minimal 
  regular subdivision, in the sense that any other regular subdivision refines it. 
\end{proposition}

Proposition \ref{prop:existsmin} motivates the following definition:

\begin{definition} \label{def:regulariz}
      If $\fan$ is a $2$-dimensional fan, we denote by $\boxed{\fan^{reg}}$ its 
      minimal regular subdivision, and we call it the {\bf regularization}\index{regularization!of a 
      fan} of $\fan$. 
\end{definition}

The importance of the regularization operation in our context stems from the 
fact that it allows to describe combinatorially the minimal resolution of a toric 
surface (see Proposition \ref{prop:minrestor} below). 
The regularization of a $2$-dimensional cone may be described in the following way 
(see \cite[Proposition 1.19]{O 88}):

\begin{proposition}  \label{prop:regconv} 
    Let $N$ be a lattice of rank two and 
   let $\sigma$ be a $2$-dimensional strictly convex cone of $N$. Then 
   the regularization $\sigma^{reg}$ of the fan of its faces is 
   obtained by subdividing $\sigma$ using the rays directed by the integral 
   points lying on the boundary  
   of the convex hull of the set of non-zero integral points of $\sigma$. 
   If $\fan$ is a fan of a lattice of rank two, then its regularization is the union 
   of the regularizations of its cones. 
\end{proposition}

An alternative recursive description of $\sigma^{reg}$ was  given by Mutsuo Oka
in \cite[Chap. II.2]{O 97}. 

\begin{example} \label{ex:fanexreg}
     Let us consider again the fan $\fan\left(3/5, 2/1, 5/2\right)$ 
     of Example \ref{ex:fanex}. 
     The rays of its regularization $\fan^{reg}\left(3/5, 2/1, 5/2\right)
      =\fan\left(1/2,3/5, 2/3, 1/1, 2/1, 5/2, 3/1\right)$ 
     are drawn in green in Figure \ref{fig:examfanreg}. 
     The thick orange polygonal line, on the right side of this figure, is 
     the union of compact edges of the boundaries of the convex hulls of the sets of 
     non-zero integral points of its $2$-dimensional cones.      
 \end{example}

\begin{figure}[h!]
     \begin{center}
\begin{tikzpicture}[scale=0.8]
\begin{scope}[shift={(0,0)}, scale=0.8]

\foreach \x in {0,1,...,5}{
\foreach \y in {0,1,...,5}{
       \node[draw,circle,inner sep=0.7pt,fill, color=gray!40] at (1*\x,1*\y) {}; }
   }
\draw [->](0,0) -- (0,5.5);
\draw [->](0,0) -- (5.5,0);
\draw [-, thick, color=blue](0,0) -- (2,5);
\draw [-, thick, color=blue](0,0) -- (2.5,5);
\draw [-, thick, color=blue](0,0) -- (5.5,3.3);

\draw [-,  color=black!20!green](0,0) -- (1.75,5.25);
\draw [-,  color=black!20!green](0,0) -- (5,5);
\draw [-,  color=black!20!green](0,0) -- (6,4);
\draw [-,  color=black!20!green](0,0) -- (6,3);

\node [left] at (1.3,-0.3) {$e_1$};
\node [left] at (0,1) {$e_2$};

\node[draw,circle, inner sep=1.5pt,color=black, fill=black] at (1,0){};
\node[draw,circle, inner sep=1.5pt,color=black, fill=black] at (0,1){};
\node[draw,circle, inner sep=1.5pt,color=black, fill=black] at (0,0){};
\node[draw,circle, inner sep=1.8pt,color=red, fill=red] at (2,5){};
\node[draw,circle, inner sep=1.8pt,color=red, fill=red] at (1,2){};
\node[draw,circle, inner sep=1.8pt,color=red, fill=red] at (5,3){};

\node[draw,circle, inner sep=1.8pt,color=red, fill=red] at (1,3){};
\node[draw,circle, inner sep=1.8pt,color=red, fill=red] at (1,1){};
\node[draw,circle, inner sep=1.8pt,color=red, fill=red] at (3,2){};
\node[draw,circle, inner sep=1.8pt,color=red, fill=red] at (2,1){};
\end{scope}

\begin{scope}[shift={(8,0)}, scale=0.8]

\foreach \x in {0,1,...,5}{
\foreach \y in {0,1,...,5}{
       \node[draw,circle,inner sep=0.7pt,fill, color=gray!40] at (1*\x,1*\y) {}; }
   }
\draw [->](0,0) -- (0,5.5);
\draw [->](0,0) -- (5.5,0);
\draw [-, thick, color=blue](0,0) -- (2,5);
\draw [-, thick, color=blue](0,0) -- (1,2);
\draw [-, thick, color=blue](0,0) -- (5,3);

\draw [-, ultra thick, color=orange](0,1) -- (2,5) -- (1,2) -- (1,1) -- (5,3) -- (2,1) --(1,0);
\draw [-,  color=black!20!green](0,0) -- (1,3);
\draw [-,  color=black!20!green](0,0) -- (1,1);
\draw [-,  color=black!20!green](0,0) -- (3,2);
\draw [-,  color=black!20!green](0,0) -- (2,1);

\node [left] at (1.3,-0.3) {$e_1$};
\node [left] at (0,1) {$e_2$};

\node[draw,circle, inner sep=1.5pt,color=black, fill=black] at (1,0){};
\node[draw,circle, inner sep=1.5pt,color=black, fill=black] at (0,1){};
\node[draw,circle, inner sep=1.5pt,color=black, fill=black] at (0,0){};
\node[draw,circle, inner sep=1.8pt,color=red, fill=red] at (2,5){};
\node[draw,circle, inner sep=1.8pt,color=red, fill=red] at (1,2){};
\node[draw,circle, inner sep=1.8pt,color=red, fill=red] at (5,3){};

\node[draw,circle, inner sep=1.8pt,color=red, fill=red] at (1,3){};
\node[draw,circle, inner sep=1.8pt,color=red, fill=red] at (1,1){};
\node[draw,circle, inner sep=1.8pt,color=red, fill=red] at (3,2){};
\node[draw,circle, inner sep=1.8pt,color=red, fill=red] at (2,1){};
\end{scope}

\end{tikzpicture}
\end{center}
\caption{The regularization $\fan^{reg}\left(3/5, 2/1, 5/2\right)$ 
     of the fan of Figure \ref{fig:examfan}}  
      \label{fig:examfanreg}
    \end{figure}

 \subsection{Toric varieties and their orbits}
\label{ssec:torsurf}
$\:$
 \medskip

 In this subsection we explain in which way fans determine special kinds of complex algebraic 
 varieties, called \emph{toric varieties}. Namely, every rational polyhedral cone relative to 
 a lattice determines a \emph{monoid algebra} (see Definition \ref{def:dualcone}), 
 whose maximal spectrum is  an \emph{affine toric variety} (see Definition \ref{def:afftoric}). 
 More generally, every fan determines a toric variety by gluing the affine toric varieties 
 associated to its cones (see Definition \ref{def:gentoricvar}).

\medskip
   
   One associates with a lattice $N$ of  rank $n$ the following 
   {\bf complex algebraic torus}\index{complex!algebraic torus}\index{algebraic torus!complex} of 
   dimension $n$ (that is, an algebraic group isomorphic to  $\left( (\C^*)^n, \cdot \right)$):
       \begin{equation}   \label{eq:Ntorus}
            \boxed{\cT_N} := N \otimes_{\Z} \C^*.
       \end{equation}
   Here the factors are considered as abelian groups $(N, +)$ and $(\C^*, \cdot)$, 
   therefore they are endowed 
   with canonical structures of $\Z$-modules, relative to which is taken the previous 
   tensor product. This algebraic torus may be also described in terms of the dual 
   lattice $M$ of $N$, defined by:
      $$\boxed{M} := \mbox{Hom } (N, \Z).$$
   Namely, one has:
       \begin{equation}   \label{eq:Mtorus}
            \cT_N =  \mbox{Hom}(M, \C^*).
       \end{equation}
   Equations (\ref{eq:Ntorus}) and (\ref{eq:Mtorus}) allow in turn to give the following 
   interpretations of the lattices $N$ and $M$ in terms of morphisms of algebraic groups: 
     \begin{equation} \label{eq:charcgroups} 
         \begin{array}{ccl}
             N & = &  \mbox{Hom} (\C^*, \cT_N) = \\
                 & = & \mbox{the group of {\bf one parameter subgroups} of } \cT_N; 
                      \index{one parameter subgroup}\\
                 &     &  \\
             M & = & \mbox{Hom} (\cT_N, \C^*) = \\
                 & = & \mbox{the group of {\bf characters} of } \cT_N. \index{character}
          \end{array}
     \end{equation}
   If $w \in N$ is seen as an element of the lattice $N$, we denote by $\boxed{t^w}$ 
   the same element seen as a morphism of abelian groups from $\C^*$ to $\cT_N$. 
   
   Let us explain this notation in the case when $N$ has rank $2$. 
   If $t$ is viewed as the parameter on the 
   source $\C^*$ and one identifies $\cT_N$ with $(\C^*)^2$ using the basis $(e_1, e_2)$ of $N$, 
   then the morphism becomes the following map from $\C^*$ to $(\C^*)^2$:
       \[ t \to (t^{c}, t^{d}). \]
   Here $(c, d)$ denote as before the coordinates of $w$ in the chosen basis 
   $(e_1, e_2)$ of $N$. One gets therefore a parametrized monomial curve as 
   at the beginning of Subsection \ref{ssec:fans}. 
   The advantage of seeing it as an element of $\mbox{Hom} (\C^*, \cT_N)$ is that one 
   gets a viewpoint independent of the choice of coordinates for $\cT_N$, that is, of bases 
   for $M$ or for $N$. 
   
   It is customary to say that a morphism $t^w \in  \mbox{Hom } (\C^*, \cT_N)$ 
is a \emph{one parameter subgroup} of 
   $\cT_N$, even when this morphism is 
   not injective. Note that $t^w$ is injective if and only if $w$ is a primitive element of $N$. 
   In general, when $w \in N \setminus \{0\}$, the map $t^w$ is a cyclic covering of its 
   image, of degree $l_{\Z}(w)$ (see Definition \ref{def:intlength}). 
   Note also that $t^0$ is the constant map with image 
   the unit element $1$  of the group $\cT_N$.

    We introduced the notation $t^w$ in order to be able to distinguish between  
    $N$ seen as an abstract group, and seen as the lattice of one parameter subgroups 
    of $\cT_N$. In an analogous way, if $m \in M$,  one uses the notation 
    $\boxed{\chi^m} :  \cT_N \to \C^*$ for its associated character, in order to distinguish 
    between $M$ seen as an abstract group and seen as the lattice of 
    characters of $\cT_N$. If one denotes by $\boxed{w \cdot m} \in \Z$ the result of applying 
    the canonical duality pairing 
        $N \times M \to \Z $
    to $(w, m) \in N \times M$, then the composite morphism 
    $\chi^m \circ t^w : \C^* \to \C^*$ is simply given by $t \to t^{w \cdot m}$.
    This is the intrinsic description of the composition performed in formula (\ref{eq:compos}).

   Let us see more precisely how the choice of basis $(e_1, e_2)$ of $N$ 
   determines an isomorphism $\cT_N \simeq (\C^*)^2$. To have such an isomorphism 
   amounts to choosing a special pair $(x, y)$ of regular functions on $\cT_N$, 
   which are the pull-backs of the coordinate functions on $(\C^*)^2$. This isomorphism 
   should be not only an isomorphism of algebraic surfaces, but also of groups. As the 
   coordinate functions on $(\C^*)^2$ are characters of $\left( (\C^*)^2, \cdot \right)$, 
   that is, elements of 
   $\mbox{Hom}((\C^*)^2, \C^*)$, we deduce that $x, y$ are also characters, this 
   time of $(\cT_N, \cdot)$. It means that they are elements of the lattice $M$ 
   (see  the equalities (\ref{eq:charcgroups})). In which way does the basis $(e_1, e_2)$ of $N$  
   determine a pair of elements of $M$? Well, this pair is simply the dual basis 
   $\boxed{(\epsilon_1, \epsilon_2)}$ of $(e_1, e_2)$!
        Therefore, one has $(x, y) = (\chi^{\epsilon_1}, \chi^{\epsilon_2})$ in terms of the 
    dual basis $(\epsilon_1, \epsilon_2) \in M^2$ of $(e_1, e_2) \in N^2$. 
    
    \medskip
    The choice of coordinates $(x, y)$ allows to embed  
    the torus  $\cT_N$ into the affine plane $\C^2$ with the same coordinates. 
    The coordinate ring of this affine plane is of course $\C[x, y]$. In our 
    context it is important to interpret this ring as 
    the $\C$-algebra of the commutative monoid of monomials with non-negative exponents 
    in the variables $x$ and $y$. This monoid is isomorphic (using the map 
    $m \to \chi^m$) to the monoid $\cone {\langle \epsilon_1, \epsilon_2 \rangle} \cap M$. 
    In turn, the cone $\cone {\langle \epsilon_1, \epsilon_2 \rangle}$ is in the following sense the 
    dual cone of $\sigma_0 := \cone {\langle e_1, e_2 \rangle}$:
    
    \begin{definition}  \label{def:dualcone}
          Let $\sigma$ be a cone of $N$. Its {\bf dual}\index{dual of a cone} is the cone 
          $\sigma^{\vee}$ of $M$ defined by:
              \[ \boxed{\sigma^{\vee}} := \{ m \in M_{\R}, \:  w \cdot m \geq 0 
                     \mbox{ for all } w \in \sigma \},  \]
          and its associated {\bf monoid algebra}\index{monoid algebra} is the 
          $\C$-algebra of the abelian monoid $(\sigma^{\vee} \cap M, +)$: 
           \[
                  \boxed{\C[\sigma^{\vee}  \cap M]} := 
                  \left\{ \sum_{\textrm{finite}}c_m \chi^m \, , \,  m \in \sigma^{\vee}  \cap M \mbox{ and }
                   c_m \in \C \right\}. 
             \]
    \end{definition}
    
    Note that $\sigma$ is strictly convex if and only if  the dimension of $\sigma^{\vee}$  
    is equal to the rank of the lattice $M$. The $\C$-algebra $\C[\sigma^{\vee}  \cap M]$ 
    is finitely generated, since the monoid $(\sigma^{\vee}  \cap M, +)$ 
    is finitely generated by Gordan's Lemma (see \cite[Section 1.1, Proposition 1]{F 93}).

   The set $\C^2$ with coordinates $(x,  y)$ may now be interpreted in the two following ways:
       \begin{equation}  \label{eq:twointerpr}
             \begin{array}{ccl}
                  \C^2_{x, y} & = & \mbox{the maximal spectrum of the ring } 
                                        \C[\sigma_0^{\vee}  \cap M] = \\
                          &    & \\            
                          & = &  \mbox{Hom} (\sigma_0^{\vee} \cap M, \C) . 
          \end{array}
       \end{equation}
    The last set of homomorphisms is taken in the category of abelian monoids, where 
    $\C$ is considered as a monoid with respect to multiplication. This interpretation 
    is obtained by looking at the evaluation of the monomials $\chi^m$, with 
    $m \in \sigma_0^{\vee} \cap M$, at the points of $\C^2$. 
    
    The equalities (\ref{eq:twointerpr}) may be turned into a general way to associate 
    a complex affine variety to a cone $\sigma$ of $N$, in arbitrary dimension:
    \begin{equation}  \label{eq:twointerprgen}
             \begin{array}{ccl}
                  \boxed{X_{\sigma}} & := & \mbox{the maximal spectrum of } 
                                        \C[ \sigma^{\vee} \cap M] = \\
                          & = &  \mbox{Hom} (\sigma^{\vee} \cap M, \C). 
          \end{array}
       \end{equation}
       
       The equalities (\ref{eq:twointerpr}) show that $X_{\sigma_0} = \C_{x,y}^2$, 
       if $x = \chi^{\epsilon_1}$ 
       and $y = \chi^{\epsilon_2}$.  Therefore, the affine variety $X_{\sigma_0}$ is 
       smooth. The following proposition characterizes the cones for which the associated variety 
       is smooth (see \cite[Section 2.1, Proposition 1]{F 93}):
       
       \begin{proposition}  \label{prop:smoothcritaffine}
            Let $\sigma$ be a strictly convex cone of the lattice $N$. Then the affine variety 
            $X_{\sigma}$ is smooth if and only if $\sigma$ is regular in the sense of 
            Definition \ref{def:regular}. 
       \end{proposition}

       \medskip
       In the sequel, by a {\bf stratification} of an algebraic variety we mean a finite partition 
       of it into locally closed connected smooth subvarieties, called the {\bf strata} of the 
       stratification, such that the closure of each stratum is a union of strata. 
       
     Consider the following stratification of $X_{\sigma_0} = \C_{x,y}^2$: 
         \begin{equation} \label{eq:firststrat}    
               \C^2_{x, y} = \{0\} \sqcup \left( \C^*_{x} \times \{0\}   \right) \sqcup 
               \left( \{0\} \times  \C^*_{y} \right)  \sqcup (\C^*)^2_{x, y}.
         \end{equation}
     One may interpret in the following way its strata in terms of vanishing of monomials 
     whose exponents belong to 
     $\sigma_0^{\vee} \cap M  = \N {\langle \epsilon_1, \epsilon_2 \rangle}$:
        \begin{itemize}
            \item $0$ is the only point of $\C_{x, y}^2$ at which vanish exactly the monomials with 
                   exponents in $\left(\sigma_0^{\vee}  \:  \setminus \: \{0\}\right) \cap M$.
            \item $\C^*_{x} \times \{0\}$ is the set of points of $\C^2_{x, y}$ at which vanish 
                    exactly the monomials with exponents in 
                    $\left(\sigma_0^{\vee}  \:  \setminus \: \cone \epsilon_1 \right) \cap M$.
            \item $\{0\} \times  \C^*_{y}$  is the set of points of $\C^2_{x, y}$ at which vanish 
                    exactly the monomials with exponents in 
                    $\left(\sigma_0^{\vee}  \:  \setminus \: \cone \epsilon_2 \right) \cap M$.
            \item  $ (\C^*)^2_{x, y} = \cT_N$ is the set of points of $\C^2_{x, y}$ at which vanish 
                     no monomials, that is, at which vanish exactly  the monomials with exponents in 
                     $\left(\sigma_0^{\vee}  \:  \setminus \: \sigma_0^{\vee}\right)\cap M$.
        \end{itemize}
      Note that the sets of exponents of 
      monomials appearing in the previous list are precisely those of the form  
      $\left(\sigma_0^{\vee}  \:  \setminus \: \tau \right) \cap M$, 
      where $\tau$ varies among the faces of the cone $\sigma_0^{\vee}$. It is customary in toric 
      geometry to express them in a dual way, using the following bijection between the faces 
      of $\sigma$ and of $\sigma^{\vee}$, valid in all dimensions for (not necessarily 
      rational) convex polyhedral cones $\sigma$  (see \cite[Proposition 1.2.10]{CLS 11}):
      
      \begin{proposition} \label{prop:dualfaces} 
           Let $\sigma$ be a cone of $N_{\R}$. Then the map $\rho \to \rho^{\perp} \cap 
           \sigma^{\vee}$ is an order-reversing bijection from the set of faces of $\sigma$ 
           to the set of faces of $\sigma^{\vee}$ (see Figure \ref{fig:dualfaces}). 
      \end{proposition}
      
      Here 
       $  \boxed{\rho^{\perp}} := \{m \in M_{\R}, \:  w \cdot m = 0 \mbox{ for all } w \in \rho \} $
      denotes the orthogonal of the cone $\rho$ of $N$. It is a real vector subspace of 
      $M_{\R}$, which may be characterized as the maximal vector subspace of the 
      convex cone $\rho^{\vee}$. 
      
      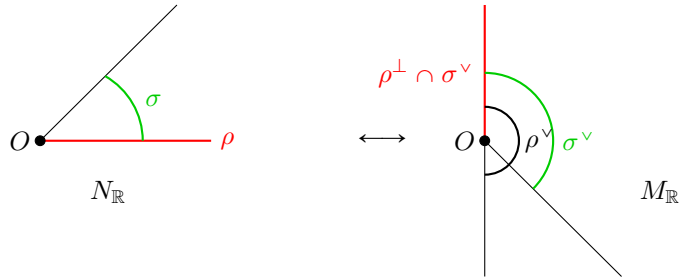
\begin{figure}[h!]
     \begin{center}
\begin{tikzpicture}[scale=0.45]
\draw [-, color=red,thick](0,0) -- (5,0);
\draw [-, color=black](0,0) -- (4,4);
\node[draw,circle,inner sep=1.3pt,fill=black] at (0,0){};
\node [left] at (0,0) {$O$};
\node [right,color=red] at (5,0) {$\rho$};
\draw [color=black!20!green, thick](3,0) arc (0:60:2.2);
\node [right,color=black!20!green] at (2.8,1.2) {$\sigma$};
\node [left] at (2.8,-1.5) {$N_{\mathbb R}$};
\node [left] at (11,0) {$\longleftrightarrow$};
\draw [-, color=red,thick](13,0) -- (13,4);
\node[draw,circle,inner sep=1.3pt,fill=black] at (13,0){};
\node [left] at (13,0) {$O$};
\draw [-, color=black](13,0) -- (13,-4);
\draw [-, color=black](13,0) -- (17,-4);
\draw [color=black!20!green, thick](13,2) arc (90:-45:2);
\node [right,color=black!20!green] at (15,0) {$\sigma^{\vee}$};
\node [left,color=red] at (13,2) {$\rho^{\perp}\cap\sigma^{\vee}$};
\draw [color=black, thick](13,1) arc (90:-90:1);
\node [right,color=black] at (13.9,0) {$\rho^{\vee}$};
\node [left] at (19,-1.5) {$M_{\mathbb R}$};
\end{tikzpicture}
\end{center}
\caption{The bijection between the faces of $\sigma$ and $\sigma^{\vee}$}  \label{fig:dualfaces}
    \end{figure}

      The stratification (\ref{eq:firststrat}) of $\C^2$ is a particular case of a  stratification 
      of any affine variety  
      of the form $X_{\sigma}$. In order to define it, one associates with  each point $p$ 
      of $X_{\sigma}$  the subset of  $\sigma^{\vee} \cap M$ formed by the exponents 
      of the monomials vanishing at $p$. This defines a function from $X_\sigma$ to 
      the power set of $\sigma^{\vee} \cap M$, whose levels are precisely the strata of 
      the stratification of $X_\sigma$. The set of strata is  in bijective 
      correspondence with the set of faces of $\sigma$, the stratum $O_{\rho}$ 
      corresponding to the face $\rho$ of $\sigma$ being:
           \begin{equation}  \label{eq:partsgen}
               \boxed{O_{\rho}} :=  \left\{ p \in  \mbox{Hom} (\sigma^{\vee} \cap M, \C), \:  
                  p^{-1}(0) = (\sigma^{\vee}  \setminus \:  \rho^{\perp}) \cap M \right\}. 
         \end{equation}
       In particular, $O_{\{0\}} = \cT_N$ is the 
       only stratum whose dimension is the same as the dimension of $X_{\sigma}$. 
       This shows that the torus $\cT_N$ embeds naturally as an affine open set in 
       the affine surface $X_{\sigma}$. For this reason, the following vocabulary 
       was introduced: 
       
       \begin{definition}  \label{def:afftoric}
           If $N$ is a lattice and $\sigma$ is a strictly convex cone of $N$, then 
           the variety $X_{\sigma}$ defined by the equalities \eqref{eq:twointerprgen} is called an 
           {\bf affine toric variety}\index{variety!affine toric}. 
       \end{definition}

       Note that for $X_{\sigma_0} = \C^2_{x, y}$, the strata are:
         \begin{itemize}
           \item $O_{\sigma_0}= \{0\}$;  
            \item $O_{\cone e_2} = \C^*_{x} \times \{0\}$;  
            \item $O_{\cone e_1}= \{0\} \times  \C^*_{y}$;                     
            \item  $O_{\{0 \}} =  (\C^*)^2_{x, y} = \cT_N$. 
         \end{itemize}
       One may feel difficult to remember the second and third equalities, a common 
       error at the time of doing computations being to permute them. A way to remember them 
       is the following: \emph{the orbit corresponding to an edge of a $2$-dimensional regular cone 
       is the complement of the origin in the axis of coordinates of $\C^2$ defined by the 
       vanishing of the dual variable}. In our case, the dual variable of the edge 
       $\cone e_1$ is $x= \chi^{\epsilon_1}$, whose $0$-locus is the axis of the 
       variable  $y$, and conversely.

       The notation $O_{\rho}$ is motivated by the fact that this subset of $X_{\sigma}$ is an  
       \emph{orbit} of a natural action of the algebraic torus $\cT_N$ on $X_{\sigma}$. 
       For $X_{\sigma_0} = \C^2_{x, y}$, case in which one may also identify  
        $\cT_N$ with $(\C^*)^2_{u, v}$, this action is given by $ (u, v) \cdot (x, y) := (u x, v y )$. 
     In general, the action of $\cT_N$ on $X_{\sigma}$ may be described  
     in intrinsic terms by: 
        \[  (M \buildrel \tau \over \to \C^*) \cdot (\sigma^{\vee} \cap M \buildrel p \over \to \C) := 
                 (\sigma^{\vee} \cap M \buildrel {\tau \cdot p} \over \longrightarrow \C) . \]
     In the previous equation we used again the interpretations of the points of $\cT_N$ 
     and $X_{\sigma}$ as morphisms of monoids (see equations (\ref{eq:Mtorus}) and 
     (\ref{eq:twointerprgen})).

     Assume now that $\fan$ is a fan of $N$, in the sense of Definition \ref{def:fan}. 
     Each affine toric variety $X_{\sigma}$, where $\sigma \in \cF$,  contains  
     the torus $\cT_N$ as an affine open set. If $\sigma$ and $\tau$ are two cones 
     of $\fan$, then one has a natural identification of their respective 
     tori, and also of their larger Zariski open subsets $X_{\sigma \cap \tau} \subset X_{\sigma}$ and 
     $X_{\sigma \cap \tau} \subset X_{\tau}$. If one glues the various affine toric varieties 
     $(X_{\sigma})_{\sigma \in \cF}$ using the previous identifications, one gets an abstract 
     separated complex algebraic variety $X_{\fan}$ 
     which still contains the torus $\cT_N$ as an affine open 
     subset (see \cite[Theorem 1.4 and 1.5]{O 88}). 
     
     \begin{definition}  \label{def:gentoricvar}
         The  {\bf toric variety}\index{variety!toric} associated with a fan $\fan$ of a lattice $N$  is 
          the variety $\boxed{X_{\fan}}$ constructed above.   
    \end{definition}

     \begin{remark}   \label{rem:toricnonnorm}
     All toric varieties constructed from fans 
     are normal  in the sense of Definition \ref{def:normgeom}
     (see \cite[Theorem 1.3.5] {CLS 11}).
     One has a more general notion of toric 
     variety, which includes some non-normal varieties as well (see 
     the paper \cite{GPT 14} of Teissier and the second author). 
     Those varieties can be described as before by gluing maximal spectra 
     of algebras of not necessarily saturated finite type submonoids of lattices, 
     the normal ones being
     precisely the toric varieties associated with a fan of Definition \ref{def:gentoricvar}. 
     \end{remark}

     As a consequence of Proposition \ref{prop:smoothcritaffine}, one has a smoothness criterion for  toric varieties:
     
     \begin{proposition}   \label{prop:smoothcritgen}
            Let $\fan$ be a fan of the lattice $N$. Then the toric variety 
            $X_{\fan}$ is smooth if and only if $\fan$ is regular in the sense of 
            Definition \ref{def:regular}. 
       \end{proposition}

     \medskip
     Let us come back to a fan $\fan$ of a weight lattice $N$. 
     When $\rho$ varies among the cones of $\fan$, 
     the actions of the torus $\cT_N$ on the affine toric varieties $X_{\rho}$ 
     glue into an action on $X_{\fan}$, 
     whose orbits are still denoted by $O_{\rho}$. 
     The conservation of the notation (\ref{eq:partsgen}) is motivated 
     by the fact that in the gluing of $X_{\sigma}$ and $X_{\tau}$, the orbits 
     denoted $O_{\rho}$ on both sides get identified, 
     for every face $\rho$ of $\sigma \cap \tau$. 
     If $\rho$ is a cone of the fan $\fan$, we denote by 
      $\boxed{\overline{O}_{\rho}}$ the closure in $X_{\fan}$ of the orbit $O_{\rho}$. 
    The orbit closure $\overline{O}_{\rho}$ has also a natural structure of 
    normal toric variety (see \cite[Chapter 3]{F 93}).
     
     The torus $\cT_N$ is identified canonically with the orbit $O_{0}$ corresponding to 
     the origin of $N_{\R}$, seen as a cone of dimension $0$. Its complement is the union 
     of all the orbits of codimension at least $1$. Let us introduce a special name and 
     notation for this complement: 
     
     \begin{definition}  \label{def:boundtoric}
          Let $X_{\fan}$ be a toric variety defined by a fan $\fan$. 
           Its {\bf boundary}\index{boundary!of a toric variety} 
           $\boxed{\partial X_{\fan}}$ is the complement of the algebraic torus 
           $\cT_N$ inside $X_{\fan}$. 
     \end{definition}
     
     The boundary $\partial X_{\fan}$ is a reduced Weil divisor inside $X_{\fan}$, whose irreducible 
     components are the orbit closures $\overline{O}_{\rho}$ corresponding to the cones 
     $\rho$ of $\fan$ which have dimension $1$, that is, to the rays of the fan $\fan$.

\subsection{Toric morphisms and toric modifications}
\label{ssec:tormod}
$\:$
 \medskip
 
 In this subsection we define the notion of \emph{toric morphism} between toric varieties 
 (see Definition \ref{def:toricmorph}) and we explain
 in which way refining a fan defines a special kind of toric morphism, 
 called a \emph{toric modification} (see Proposition \ref{prop:tormorphtypes}). 
 In Examples \ref{ex:tworegcones} and \ref{ex:blowupor} we explain how 
 to do concrete computations of toric modifications in dimension two, the second one giving 
 a toric presentation of the blow ups of the origin. Finally, Proposition \ref{prop:minrestor}  
 explains the combinatorics 
 of the minimal resolution of a normal affine toric surface.
 
 \medskip
 
 Assume that $N_1$ and $N_2$ are two weight lattices, endowed with cones $\sigma_1$ and 
 $\sigma_2$. Let $\phi: N_1 \to N_2$ be a morphism of lattices which sends the cone 
 $\sigma_1$ into the cone $\sigma_2$. Using 
     the second interpretation in the equalities (\ref{eq:twointerprgen}) of the points of affine toric varieties, 
     we see that $\phi$ induces an algebraic morphism between the associated toric varieties: 
    \begin{equation}   \label{eq:algmorph}
         \begin{array}{cccc}
               \boxed{\psi_{\sigma_2, \phi}^{\sigma_1}} 
               : &   X_{\sigma_1} &  \to   &   X_{\sigma_2}  \\
                           &    p_1 &  \to &  p_1 \circ \phi^{\vee}
         \end{array}.
       \end{equation}
    One sees immediately from the definitions that the adjoint $\phi^{\vee} : M_2 \to M_1$ of $\phi$ 
    maps $\sigma_2^{\vee}$ into $\sigma_1^{\vee}$, which shows that the composition 
    $p_1 \circ \phi^{\vee}$ belongs indeed to 
    $X_{\sigma_2} = \mbox{Hom} (\sigma_2^{\vee} \cap M_2, \C)$ whenever $p_1 \in 
    X_{\sigma_1} = \mbox{Hom} (\sigma_1^{\vee} \cap M_1, \C)$. The morphism 
    $ \psi_{\sigma_2, \phi}^{\sigma_1}$ 
    may be also described using the first interpretation in the equalities (\ref{eq:twointerprgen}), 
    as the morphism of affine schemes induced by the morphism of $\C$-algebras 
    $\C[\sigma_2^{\vee} \cap M_2] \to \C[\sigma_1^{\vee} \cap M_1]$ which sends each 
    monomial $\chi^{m_2} \in \sigma_2^{\vee} \cap M_2$ to the monomial 
    $\chi^{\phi^{\vee}(m_2)} \in \sigma_1^{\vee} \cap M_1$. 
    
    Assume now that $N_1$ and $N_2$ are endowed with fans $\fan_1$ and $\fan_2$ respectively, 
    such that $\phi$ sends each cone of $\fan_1$ into some cone of $\fan_2$. We say that 
    $\phi$ {\bf is compatible with the two fans}. It may be checked 
    formally that the previous morphisms 
    $  \psi_{\sigma_2, \phi}^{\sigma_1}  : X_{\sigma_1} \to X_{\sigma_2}$, 
    for all the pairs $(\sigma_1, \sigma_2)  \in \fan_1 \times \fan_2$ which verify that  
    $\phi(\sigma_1) \subseteq \sigma_2$, glue into an algebraic  morphism:
      $  \boxed{\psi_{\fan_2, \phi}^{\fan_1}}  : X_{\fan_1} \to X_{\fan_2}$.
     This morphism is moreover \emph{equivariant} with respect to the actions of $\cT_{N_1}$ and 
     $\cT_{N_2}$ on $X_{\fan_1}$ and    $X_{\fan_2}$ respectively. For this reason, 
     one uses the following terminology: 
     
     \begin{definition} \label{def:toricmorph}
         If the morphism of lattices $\phi: N_1 \to N_2$ sends every cone of 
         $\fan_1$ into some cone of $\fan_2$, then the morphism of algebraic varieties 
         $\psi_{\fan_2, \phi}^{\fan_1}: X_{\fan_1} \to X_{\fan_2}$ described above 
          is called the {\bf toric morphism associated with $\phi$ and the fans $\fan_1$, $\fan_2$}
          \index{morphism!toric}.
     \end{definition}     
          
     The toric morphism $\psi_{\fan_2, \phi}^{\fan_1}$  sends the 
     torus $\cT_{N_1} = X_{\fan_1} \: \setminus \: \partial \: X_{\fan_1}$ into 
     $\cT_{N_2} = X_{\fan_2} \: \setminus \: \partial \: X_{\fan_2}$. 
     This fact  implies the following property of toric morphisms 
     relative to the boundaries of their sources and targets, in the sense of 
     Definition \ref{def:boundtoric}:

     \begin{proposition} \label{prop:boundmorph}
         Let $  \psi: X_{\fan_1} \to X_{\fan_2}$ be the toric morphism
         associated with $\phi$ and the fans $\fan_1$ and  $\fan_2$. Then 
           $ \psi^{-1} (\partial \: X_{\fan_2}) \subseteq \partial \: X_{\fan_1}$. 
     \end{proposition}

     Toric morphisms have the following properties (see \cite[Theorems 1.13, 1.15]{O 88}): 
     
     \begin{proposition}  \label{prop:tormorphtypes}
          Let $N_1, N_2$ be two lattices and $\fan_1, \fan_2$ be fans of $N_1$ and $N_2$ 
          respectively. Let $\phi: N_1 \to N_2$ be a lattice morphism compatible with the two 
          fans. Then:
              \begin{enumerate}
                   \item  The morphism $\psi_{\fan_2, \phi}^{\fan_1}$ 
                       is birational if and only if $\phi$ is an isomorphism  of lattices. 
                   \item The morphism $\psi_{\fan_2, \phi}^{\fan_1}$ is proper if and only if 
                           the $\R$-linear map $\phi_{\R} : 
                       (N_1)_{\R} \to (N_2)_{\R}$ sends the support of $\fan_1$ \emph{onto} the 
                       support of $\fan_2$. 
              \end{enumerate}
        In particular, $\psi_{\fan_2, \phi}^{\fan_1}$ is a modification in the sense of 
            Definition \ref{def:modiftransf} if and only if $\phi$ is an isomorphism and, 
            after identifying $N_1$ and $N_2$ 
            using it, the fan $\fan_1$ refines the fan $\fan_2$ in the sense of Definition \ref{def:fan}.
     \end{proposition}
     
     We will consider most of the time the particular case in which 
         $N_1=N_2=N$ is a lattice of rank $2$ 
     and $\phi$ is the identity. Then, if $\sigma \subset  \sigma_0$ is a subcone of $\sigma_0$,  
     we denote by 
     $\psi^{\sigma}_{\sigma_0}$ the birational toric morphism 
     induced by the identity: 
        \begin{equation}   \label{eq:afftormorph}
               \boxed{\psi^{\sigma}_{\sigma_0}} : X_{\sigma} \to X_{\sigma_0} = \C^2_{x, y}.  
        \end{equation}

     When $\sigma$ 
      varies among all the cones of a fan $\cF$ which subdivides the cone $\sigma_0$, 
      the morphisms $\psi^{\sigma}_{\sigma_0}$ glue 
      into a single equivariant birational morphism:
          \begin{equation}   \label{eq:tormodif}  
                \boxed{\psi^{\fan}_{\sigma_0}} : X_{\fan} \to X_{\sigma_0} = \C^2_{x, y}.  
           \end{equation}

       By Proposition \ref{prop:tormorphtypes}, this morphism is also proper, because 
      $\cF$ and $\sigma_0$ have the same support. Therefore, $\psi^{\fan}_{\sigma_0}$ is a 
      modification of $\C^2_{x, y}$.

      The strict transform of $\boxed{L}:= Z(x)$ (resp. of $\boxed{L'}:= Z(y)$) 
      by the modification $ \psi^{\fan}_{\sigma_0}$ is  the orbit closure 
      $\overline{O}_{\cone e_1}$ (resp.  $\overline{O}_{\cone e_2}$) in $X_{\fan}$. 
      The preimage of $0 \in \C^2_{x,y}$, 
      called the {\bf exceptional divisor}\index{divisor!exceptional} of $ \psi^{\fan}_{\sigma_0}$,  
      and the preimage of the sum $L + L'$ of the coordinate axes, 
      which is the total transform of $L + L'$ are: 
     \begin{equation}  \label{eq:preimtor}  
       (\psi^{\fan}_{\sigma_0})^{-1}(0)  =   \overline{O}_{\rho_1} + \cdots + 
                        \overline{O}_{\rho_k}, \quad \mbox{and} \quad 
                        (\psi^{\fan}_{\sigma_0})^{-1}(L + L')        = 
                                             \overline{O}_{\cone e_1}  +  \overline{O}_{\rho_1}  + \cdots +  
                                             \overline{O}_{\rho_k}  + 
                       \overline{O}_{\cone e_2},  
     \end{equation}       
  where  $\rho_1, \dots, \rho_k$ denote the rays of $\cF$ contained in the 
  interior of $\sigma_0$, labeled as in Figure \ref{fig:examfandual}.  
  Note that $L + L' = \partial X_{\sigma_0}$ and 
    $(\psi^{\fan}_{\sigma_0})^{-1}(L + L')   = \partial X_{\cF}$, which is a particular case of 
    Proposition \ref{prop:tormorphtypes}.
    \medskip 
  
  Recall now the following classical notion of (unweighted) \emph{dual graph}, 
  which extends that of Definition \ref{def:dualgraphembres}  
  and whose historical evolution was sketched 
  by the third author in \cite{PP 18}:
  
  \begin{definition}  \label{def:dualgraph}
       A {\bf simple normal crossings curve}\index{curve!simple normal crossings} 
       is a reduced abstract complex curve 
       whose irreducible components are smooth and whose singularities are normal 
       crossings, that is, analytically isomorphic to the germ at the origin of the union of 
       coordinate axes of $\C^2$. The {\bf dual graph}\index{graph!dual} 
       of a simple normal crossings curve $D$       
       is the abstract graph whose set of vertices is associated bijectively with the set of 
       irreducible components of $D$, the edges between two vertices corresponding bijectively 
       with the intersection points of the associated components of $D$. 
       Each vertex or edge is labeled by the corresponding irreducible component or point of $D$. 
  \end{definition}

  \begin{remark} 
        Let $\sigma = \cone \langle f_1, f_2 \rangle \subset N_\R$ be a strictly convex 
        cone of dimension two, not necessarily regular. One may check that the boundary 
        $\partial X_{\sigma} =  \overline{O}_{\cone f_1} +  \overline{O}_{\cone f_2}$  
        of the affine toric surface $X_{\sigma}$ 
        is an abstract simple normal crossings curve, according to Definition 
        \ref{def:dualgraph}.
  \end{remark}

  The dual graph of the total transform $ (\psi^{\fan}_{\sigma_0})^{-1}(L + L') $ 
  may be embedded in the cone $\sigma_0 \subseteq N_{\R}$:
  
  \begin{proposition}  \label{prop:dualtot}
     Let $\cF$ be a fan which subdivides the regular cone $\sigma_0$.  
     Then the dual graph of the divisor $ (\psi^{\fan}_{\sigma_0})^{-1}(L + L') $ is 
     a segment with extremities $L$ and $L'$ and with $k$ intermediate points 
     labeled in order by $\overline{O}_{\rho_1}, \dots, \overline{O}_{\rho_k}$ from 
     $L$ to $L'$. That is, it is isomorphic to the segment $[e_1, e_2] \subset N_{\R}$, 
     marked with its intersection points with the rays of $\cF$, the 
     point $[e_1, e_2] \cap \rho_i$ being labeled by the orbit closure  $\overline{O}_{\rho_i}$. 
  \end{proposition}
  
  Therefore, the rays of the fan $\fan$  correspond bijectively to 
  the irreducible components of the 
total transform $(\psi^{\fan}_{\sigma_0})^{-1} (L + L')$ of $L + L'$. 
The $2$-dimensional cones of $\fan$ correspond to the fixed points of the torus action, 
which are the only possible singular points of the surface $X_\fan$.
The orbit closures $\overline{O}_{\rho}$ and 
$\overline{O}_{\rho'}$ intersect at a point $q \in X_\fan$ 
if and only if the cone $\rho + \rho'$ is a $2$-dimensional cone of $\fan$ and 
then $q$ is the unique orbit $O_{\rho + \rho'}$ of dimension $0$ of the affine toric surface 
 $X_{\rho + \rho'} \subset X_\fan$. The point $q$ is singular on the surface 
 $X_{\fan}$ if and only if the cone $\rho+ \rho'$ is not regular.

  \begin{example}
      For the fan $\fan\left(3/5, 2/1, 5/2\right)$ of Figure \ref{fig:examfan} 
      discussed in Example \ref{ex:fanex}, the total transform $(\psi^{\fan}_{\sigma_0})^{-1} (L + L')$ 
      and its dual graph are represented in Figure \ref{fig:examfandual}.  The $4$ singular points 
      of the total transform are also singular on the surface $X_{\cF}$, with the exception of 
      $ \overline{O}_{\rho_2} \cap \overline{O}_{\rho_3}$. Indeed, the cone $\rho_2 + \rho_3$ is 
      the only regular $2$-dimensional cone of the fan $\cF$, as may be seen in Figure 
      \ref{fig:examfanreg}.
  \end{example}
  
  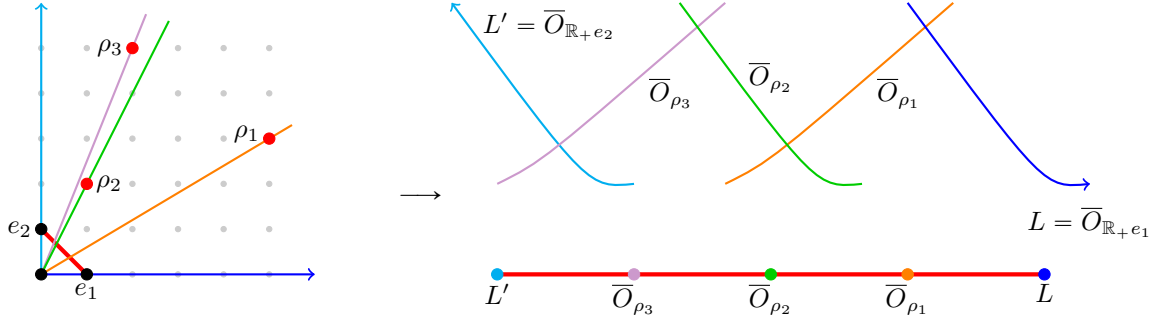
\begin{figure}[h!]
    \begin{center}
\begin{tikzpicture}[scale=0.6]
\foreach \x in {0,1,...,5}{
\foreach \y in {0,1,...,5}{
      \node[draw,circle,inner sep=0.7pt,fill, color=gray!40] at (1*\x,1*\y) {}; }
  }
\draw [->,  thick, color=cyan](0,0) -- (0,6);
\draw [->,  thick, color=blue](0,0) -- (6,0);
\draw [-, color=red, ultra thick](1,0) -- (0,1);
\node [left] at (0,1) {$e_2$};
\node [below] at (1,0) {$e_1$};
\draw [-, thick, color=violet!40](0,0) -- (2.3,5.75);
\node [left] at (2,5) {$\rho_3$};
\draw [-, thick, color=black!20!green](0,0) -- (2.8,5.6);
\node [right] at (1,2) {$\rho_2$};
\draw [-, thick, color=orange](0,0) -- (5.5,3.3);
\node [left] at (5,3.1) {$\rho_1$};
  \node[draw,circle, inner sep=1.5pt,color=black, fill=black] at (1,0){};
\node[draw,circle, inner sep=1.5pt,color=black, fill=black] at (0,1){};
\node[draw,circle, inner sep=1.5pt,color=black, fill=black] at (0,0){};
\node[draw,circle, inner sep=1.5pt,color=red, fill=red] at (2,5){};
\node[draw,circle, inner sep=1.5pt,color=red, fill=red] at (1,2){};
\node[draw,circle, inner sep=1.5pt,color=red, fill=red] at (5,3){};

\node [left] at (9,1.7) {$\longrightarrow$};
\draw [-, color=red, ultra thick](10,0) -- (22,0);
\node[draw,circle, inner sep=1.5pt,color=cyan, fill=cyan] at (10,0){};
\node [below] at (10,0) {$L'$};
\node[draw,circle, inner sep=1.5pt,color=violet!40, fill=violet!40] at (13,0){};
\node [below] at (13,0) {$\overline{O}_{\rho_3}$};
\node[draw,circle, inner sep=1.5pt,color=black!20!green, fill=black!20!green] at (16,0){};
\node [below] at (16,0) {$\overline{O}_{\rho_2}$};
\node[draw,circle, inner sep=1.5pt,color=orange, fill=orange] at (19,0){};
\node [below] at (19,0) {$\overline{O}_{\rho_1}$};
\node[draw,circle, inner sep=1.5pt,color=blue, fill=blue] at (22,0){};
\node [below] at (22,0) {$L$};

\draw[->][thick, color=cyan](13,2) .. controls (12,1.9) ..(9,6);
\node [left] at (12.8,5.5) {$L' = \overline{O}_{\cone e_2} $};
\draw[-][thick, color=violet!40](10,2) .. controls (11,2.5) ..(15,6);
\node [left] at (14.5,4) {$\overline{O}_{\rho_3}$};
\node [left] at (19.5,4) {$\overline{O}_{\rho_1}$};
\draw[-][thick, color=orange](15,2) .. controls (16,2.5) ..(20,6);
\node [above] at (16,3.8) {$\overline{O}_{\rho_2}$};
\draw[-][thick, color=black!20!green](18,2) .. controls (17,1.9) ..(14,6);
\draw[<-][thick, color=blue](23,2) .. controls (22,1.9) ..(19,6);
\node [below] at (23,1.7) {$L= \overline{O}_{\cone e_1} $};
\end{tikzpicture}
\end{center}
\caption{The dual graph of the total transform $(\psi^{\fan}_{\sigma_0})^{-1} (L + L')$} 
 \label{fig:examfandual}
   \end{figure}

    \begin{example} \label{ex:tworegcones}
     Let us explain how to describe in coordinates 
     the morphism $\psi^{\sigma}_{\sigma_0} $ of (\ref{eq:afftormorph}), 
     when $\sigma$ is a \emph{regular} subcone of $\sigma_0$. Denote by 
     $f_1, f_2$ the primitive generators of the edges of $\sigma$, ordered in such a way 
     that the bases $(e_1, e_2)$ and $(f_1, f_2)$ define the same orientation 
     of $N_{\R}$ (see Figure \ref{fig:tworegcones}). 
     Decompose $(f_1, f_2)$ in the basis $(e_1, e_2)$, writing 
      $f_1 =  \alpha e_1 +  \beta e_2 $ and $f_2 =  \gamma e_1 +  \delta e_2$. This 
      means that the unimodular matrix of change of bases from $(f_1, f_2)$ 
      to $(e_1, e_2)$ is:
          \begin{equation}  \label{eq:unimodmatr}  
          \left( \begin{array}{cc}
                    \alpha  & \gamma \\
                    \beta & \delta
          \end{array} \right) . 
       \end{equation}
     Denote by $(\varphi_1, \varphi_2) \in M^2$ the dual basis of $(f_1, f_2)$ and by 
          \begin{equation} \label{eq:newvar} 
                  \left\{ \begin{array}{l}
                      u := \chi^{\varphi_1} = x^{\delta} y ^{-\gamma} 
                               \\
                      v := \chi^{\varphi_2} = x^{-\beta} y^{\alpha},
                   \end{array} \right.
          \end{equation}
     the associated coordinates. Then,   
     in terms of the identifications $X_{\sigma} = \C^2_{u, v}$ and 
     $X_{\sigma_0} = \C^2_{x, y}$, the morphism $\psi^{\sigma}_{\sigma_0} $ 
     is given by the following monomial change of coordinates 
     (compare the disposal of exponents with the matrix (\ref{eq:unimodmatr})):
         \begin{equation}  \label{eq:unimodchange}  
          \left\{ \begin{array}{l}
                          x = u^{\alpha} v^{\gamma} 
                            \\
                          y = u^{\beta} v^{\delta}.
                    \end{array} \right.    
       \end{equation}
       
       Note that the system (\ref{eq:newvar}) implies that the expression of $v =\chi^{\varphi_2}$ 
       as a monomial in $x$ and $y$ is determined only by $f_1$, being independent 
       of the choice of $f_2$. This may be explained geometrically. Indeed, as 
       $f_1 \cdot \varphi_2 =0$, we see that $\varphi_2$ belongs to the line 
       $f_1^{\perp}$ orthogonal to $f_1$. As $\varphi_2$ may be completed into a basis of $M$, 
       it is primitive, which determines it up to sign. This sign ambiguity is lifted by the constraint 
       that the basis $(f_1, f_2)$ determines the open half-plane bounded by the line 
       $\R f_1$ on which  $\varphi_2$ has to be positive. 
       Note also that $v$ is a coordinate on the orbit $O_{\R_+ f_1}$ 
       determined by the edge $\R_+ f_1$ of $\sigma$. This 
       coordinate determines an isomorphism $O_{\R_+ f_1} \simeq \C^*_v$ of 
       complex tori, and depends only on ${\R_+ f_1}$, since 
       the orbit $O_{\R_+ f_1}$ can be realized as a subspace of the surface $X_{\R_+ f_1}$
       by formula (\ref{eq:partsgen}) above. 
  \end{example}

  \begin{figure}
     \begin{center}
\begin{tikzpicture}[scale=0.8]

 \begin{scope}[shift={(-2,0)}, scale=1]
\foreach \x in {0,1,...,3}{
\foreach \y in {0,1,...,3}{
       \node[draw,circle,inner sep=0.7pt,fill, color=gray!40] at (1*\x,1*\y) {}; }
   }
\draw [->](0,0) -- (0,4);
\draw [->](0,0) -- (4,0);
\node [left] at (1.3,-0.4) {$e_1$};
\node [left] at (0,1) {$e_2$};
\node  at (3.4,0.5) {$f_1= \alpha \: e_1 + \beta \: e_2$};
\node [above] at (1.8,2) {$f_2=\gamma e_1+\delta e_2$};
\draw [-, thick, color=blue](0,0) -- (4,2);
\draw [-, thick, color=blue](0,0) -- (3.6,2.4);
\node [right] at (2,4) {$N_{\R}$};

\node [right] at (1.5,-1) {$\left\{ \begin{array}{lc} 
                                           \sigma_0 & = \cone \langle e_1, e_2\rangle \\ 
                                           \sigma  & = \cone \langle f_1, f_2\rangle 
                                                  \end{array} \right.$};
  
\node[draw,circle, inner sep=1.5pt,color=black, fill=black] at (0,0){};
\node[draw,circle, inner sep=1.5pt,color=black, fill=black] at (1,0){};
\node[draw,circle, inner sep=1.5pt,color=black, fill=black] at (0,1){};
\node[draw,circle, inner sep=1.8pt,color=red, fill=red] at (2,1){};
\node[draw,circle, inner sep=1.8pt,color=red, fill=red] at (3,2){};

  \end{scope}

 \begin{scope}[shift={(8,0)}][scale=1]
\foreach \x in {-1, 0,1,...,2}{
\foreach \y in {-3,-2, -1, 0,1,...,3}{
       \node[draw,circle,inner sep=0.7pt,fill,color=gray!40] at (1*\x,1*\y) {}; }
   }
\draw [->](0,0) -- (0,4);
\draw [->](0,0) -- (4,0);
\node [below] at (1,0) {$\epsilon_1$};
\node [right] at (0,1) {$\epsilon_2$};
\node [right] at (2.2,-3) {$\varphi_1= \delta \epsilon_1 - \gamma \: \epsilon_2$};
\node [left] at (-1,2) {$\varphi_2 = - \beta \:  \epsilon_1 + \alpha \: \epsilon_2$};
\draw [-, thick, color=blue](0,0) -- (3,-4.5);
\draw [-, thick, color=blue](0,0) -- (-1.5,3);
\node [right] at (2,4) {$M_{\R}$};
\node[draw,circle, inner sep=1.5pt,color=black, fill=black] at (0,0){};
\node[draw,circle, inner sep=1.5pt,color=black, fill=black] at (1,0){};
\node[draw,circle, inner sep=1.5pt,color=black, fill=black] at (0,1){};
\node[draw,circle, inner sep=1.8pt,color=red, fill=red] at (2,-3){};
\node[draw,circle, inner sep=1.8pt,color=red, fill=red] at (-1,2){};

 \end{scope}

 \begin{scope}[shift={(0,-6)}, scale=0.5]
\draw [->, very thick](0,-3) -- (0,3);
\node [above] at (0,3) {$v = x^{-\beta}y^{\alpha}$};
\node [below] at (0,-3) {$O_{\cone f_1}$};
\draw [->, very thick](-3,0) -- (3,0);
\node [right] at (3,0) {$u=x^{\delta}y^{-\gamma}$};
\node [left] at (-3,0) {$O_{\cone f_2}$};
 \end{scope}

  \begin{scope}[shift={(8,-6)}, scale=0.5]
\draw [->, very thick](0,-3) -- (0,3);
\node [above] at (0,3) {$y$};
\node [below] at (0,-3) {$O_{\cone e_1}$};
\draw [->, very thick](-3,0) -- (3,0);
\node [right] at (3,0) {$x$};
\node [left] at (-3,0) {$O_{\cone e_2}$};
 \end{scope}

\draw [->, thick](3.5,-5) -- (4.5,-5);
\node [above] at (4,-5) {$\psi_{\sigma_0}^{\sigma}$};

\end{tikzpicture}
\end{center}
   \caption{The toric morphism defined by the two regular cones of Example \ref{ex:tworegcones}}
   \label{fig:tworegcones}
    \end{figure}
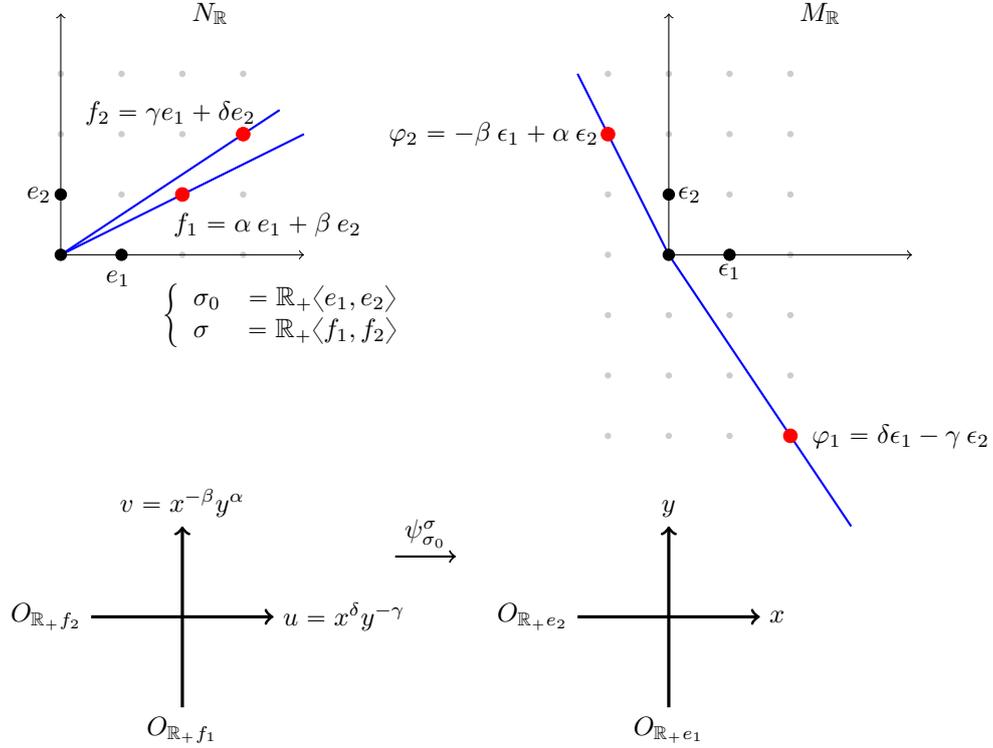

\begin{example}  \label{ex:blowupor}
 In this example we use the explanations given in Example \ref{ex:tworegcones}. 
 Let $\fan$ be the fan obtained by subdividing $\sigma_0 =  \cone { \langle e_1, e_2 \rangle } $ 
 using the half-line $\rho$ generated by $e_1 + e_2$. It has two cones of dimension $2$, 
 denoted $\sigma_1 := \cone { \langle e_1, e_1 + e_2 \rangle }$ and  
 $\sigma_2 := \cone { \langle e_1 + e_2, e_2 \rangle }$ (see Figure \ref{fig:toricblowup}). 
 Then the toric morphism $\psi^{\fan}_{\sigma_0}$ may be described by its two restrictions 
 $\psi^{\sigma_1}_{\sigma_0}$ and $\psi^{\sigma_2}_{\sigma_0}$. The matrices of change 
 of bases from $(e_1, e_1 + e_2)$ and   $(e_1 + e_2, e_2)$ to $(e_1, e_2)$  respectively are
     $
                   \left( \begin{array}{cc}
                                1  & 1 \\
                                0 & 1
                           \end{array} \right)   \mbox{ and }
                      \left( \begin{array}{cc}
                                1  & 0 \\
                                1 & 1
                           \end{array} \right)  
    $.
   Denoting by $(u_1, u_2)$ and $(v_1, v_2)$ the coordinates corresponding to the dual 
   bases of $(e_1, e_1 + e_2)$ and $(e_1 + e_2, e_2)$, the general formulas 
   (\ref{eq:unimodmatr}) and (\ref{eq:unimodchange}) 
   show that the morphisms $\psi^{\sigma_1}_{\sigma_0}$ and $\psi^{\sigma_2}_{\sigma_0}$ 
   are given by the following changes of variables:
     \begin{equation}  \label{eq:twochanges}  
         \left\{ \begin{array}{l}
                          x = u_{1} u_{2} 
                            \\
                          y = \: \: \: \: \:  u_{2}, 
                    \end{array} \right.  \quad  \mbox{ and } \quad
           \left\{ \begin{array}{l} 
                          x = v_{1}  
                            \\
                          y = v_{1} v_{2}.
                    \end{array} \right. 
       \end{equation}
We get the same expressions as in equations (\ref{eq:twochartsblowup}).
   This shows that \emph{$\psi^{\fan}_{\sigma_0}$ is a toric representative 
   of the blow up morphism of $\C^2_{x, y}$ at the origin}! \index{blow up!toric representative}
   
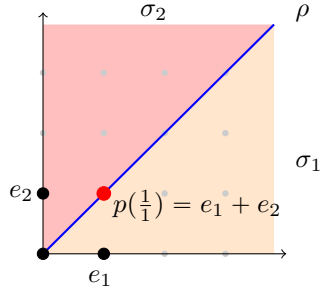
\begin{figure}
     \begin{center}
\begin{tikzpicture}[scale=0.8]
  \fill[fill=orange!20] (0,0) --(3.8,0)-- (3.8, 3.8) --cycle;
  \fill[fill=pink] (0,0) --(3.8 , 3.8)-- (0,3.8) --cycle;

\foreach \x in {0,1,...,3}{
\foreach \y in {0,1,...,3}{
       \node[draw,circle,inner sep=0.7pt,fill, color=gray!40] at (1*\x,1*\y) {}; }
   }
\draw [->](0,0) -- (0,4);
\draw [->](0,0) -- (4,0);
\draw [-, thick, color=blue](0,0) -- (3.8,3.8);
\node[draw,circle, inner sep=1.5pt,color=black, fill=black] at (0,0){};
\node[draw,circle, inner sep=1.5pt,color=black, fill=black] at (1,0){};
\node[draw,circle, inner sep=1.5pt,color=black, fill=black] at (0,1){};
\node[draw,circle, inner sep=1.8pt,color=red, fill=red] at (1,1){};

\node [right] at (1,0.8) {$p(\frac{1}{1})= e_1 + e_2$};
\node [right] at (4,4) {$\rho$};
\node [right] at (4,1.5) {$\sigma_1$};
\node [left] at (2.2,4) {$\sigma_2$};
\node [left] at (1.3,-0.4) {$e_1$};
\node [left] at (0,1) {$e_2$};

\end{tikzpicture}
\end{center}
   \caption{The subdivision of Example \ref{ex:blowupor},  
          defining the toric blow up of the origin of $\C^2$}
   \label{fig:toricblowup}
    \end{figure}
    
\end{example}

Let $\sigma$ be a non-regular cone of the weight lattice $N$ of rank two. 
By Proposition \ref{prop:smoothcritaffine}, the affine toric surface $X_{\sigma}$ 
is not smooth. In fact, it has only one singular point, the orbit $O_{\sigma}$ of dimension $0$. 
Being of dimension $2$, $X_{\sigma}$ admits a {\bf minimal resolution}\index{resolution!minimal}, 
that is, 
a resolution through which factors any other resolution (recall that this notion was 
explained in Definition \ref{def:resolgen}). It turns out that this minimal resolution may 
be given by a toric morphism, defined by the regularization of $\sigma$ in the sense 
of Definition \ref{def:regulariz} (see \cite[Proposition 1.19]{O 88}):

\begin{proposition} \label{prop:minrestor} 
   Let $\sigma$ be a non-regular cone of the weight lattice $N$ of rank two. 
   Denote by $\sigma^{reg}$ the regularization of the fan formed by the faces of 
   $\sigma$.  Then the 
   toric modification $\psi^{\sigma^{reg}}_{\sigma} : X_{\sigma^{reg}} \to X_{\sigma}$ 
   is the minimal resolution of $X_{\sigma}$. As a consequence, 
   for any fan $\fan$ of $N$, the toric modification $\psi^{\fan^{reg}}_{\fan} : X_{\fan^{reg}} \to X_{\fan}$
   is the minimal resolution of $X_{\fan}$. 
\end{proposition}

  \subsection{Toroidal varieties and modifications in the toroidal category}
\label{ssec:toroidmod}
$\:$
 \medskip

In this subsection we explain analytic generalizations of toric varieties and toric 
morphisms: the notions of \emph{toroidal variety} 
and \emph{morphism of toroidal varieties} (see Definition \ref{def:toroidobj}). 
Then we introduce the notion of \emph{cross} on a smooth germ of surface 
(see Definition \ref{def:branchcross}), and we explain how to attach to a cross 
a canonical oriented regular cone in a two-dimensional lattice  (see Definition \ref{def:ilattice}) 
and how each subdivision of this cone determines a canonical modification 
in the toroidal category (see Definition \ref{def:toroimod}). 
The toroidal pseudo-resolutions of plane curve singularities 
introduced in Subsection \ref{ssec:algtores} below will be constructed as compositions 
of such toroidal modifications.

\medskip

Toric surfaces and morphisms are not sufficient for the study of plane curve singularities 
for the following reasons. 
      One starts often from a germ of curve on a smooth complex  
          surface which does not have a preferred coordinate system. 
    It may be impossible to choose a coordinate system such that 
         the germ of curve gets resolved by only one toric modification relative to the chosen 
         coordinates (if the curve singularity is reduced and such a resolution is possible, then 
         one says that the singularity is \emph{Newton non-degenerate}, see Definition 
         \ref{def:newtnondegen} below).  
  Instead, what may be always achieved is a  \emph{morphism of toroidal surfaces}, 
  in the following sense: 
  
  \begin{definition}   \label{def:toroidobj} $\:$   
     A {\bf toroidal variety}\index{variety!toroidal}\index{toroidal!variety} 
     is a pair $(\Sigma, \partial \Sigma)$ 
        consisting of a normal complex  
        variety $\Sigma$ and a reduced divisor $\partial \Sigma$ on $\Sigma$ such that the germ 
        of $(\Sigma, \partial \Sigma)$ at any point $p \in \Sigma$ is analytically isomorphic to the germ 
        of a pair $(X_{\sigma}, \partial X_{\sigma})$ at a point of $X_{\sigma}$, where 
        $\partial X_{\sigma}$ denotes the boundary of the affine toric variety $X_{\sigma}$ 
        in the sense of Definition \ref{def:boundtoric}. Such an isomorphism is called 
        a {\bf toric chart centered at $p$}\index{toric!chart} 
        of the toroidal variety $(\Sigma, \partial \Sigma)$. 
        The divisor $\partial \Sigma$ is  {\bf the boundary of the toroidal variety}. 
        \index{boundary!of a toroidal variety}
      
      A  {\bf morphism}\index{morphism!of toroidal varieties}
          $\psi : (\Sigma_2, \partial \Sigma_2) \to (\Sigma_1, \partial \Sigma_1)$ between 
          toroidal varieties is a complex analytic morphism $\psi: \Sigma_2 \to \Sigma_1$ 
          such that $\psi^{-1}(\partial \Sigma_1) \subseteq \partial \Sigma_2$. 
          The morphism $\psi$ is a  {\bf modification}\index{modification!in the toroidal category}
       if the underlying morphism of complex varieties is a 
       modification  in the sense of Definition \ref{def:modiftransf}. 
  \end{definition}

Toroidal varieties with their
morphisms define a category, called the {\bf toroidal category}. \index{toroidal!category}
  
  The previous definition implies that if $(\Sigma, \partial \Sigma)$ is toroidal, then 
       the complement $\Sigma \: \setminus \:  \partial \Sigma$ is smooth. Indeed, the point $p$ is 
       allowed to be taken outside the boundary $\partial \Sigma$,  
       and the definition shows then that the germ of 
       $\Sigma$ at $p$ is analytically isomorphic to the germ of a toric variety at a point of the associated 
       torus, which is smooth. 
       
       If $\Sigma$ is of dimension two and if $p$ is a smooth point of  $\partial \Sigma$, 
       then $p$ is a smooth point of $\Sigma$, since  the germ of $\Sigma$ at $p$ is analytically 
       isomorphic to the germ of a normal toric surface at a point
       belonging to a $1$-dimensional orbit, which is necessarily smooth.
       
  Proposition \ref{prop:boundmorph} implies that a toric morphism 
   $\psi_{\fan_2, \phi}^{\fan_1} : X_{\fan_1} \to X_{\fan_2}$  becomes 
   an element of the toroidal category if one looks at it as a 
   complex analytic morphism from the pair $(X_{\fan_1}, \partial X_{\fan_1} )$ to 
   the pair $(X_{\fan_2}, \partial X_{\fan_2} )$, the boundaries being taken in the sense 
   of Definition \ref{def:boundtoric}.

  \begin{remark}   
     There exists also a more restrictive notion of \emph{toroidal morphism} 
     $\psi : (\Sigma_2, \partial \Sigma_2) \to (\Sigma_1, \partial \Sigma_1)$ between 
          toroidal varieties. By definition, such a morphism becomes monomial in the 
          neighborhood of any point $p$ of $\Sigma_2$, after some choice of toric 
          charts at the source and the target, centered at $p$ and $\psi(p)$ respectively. 
          Toroidal morphisms belong to the toroidal 
          category, but the converse is not true. 
     For instance,  take two copies $\C^2_{u,v}$ and $\C^2_{x,y}$ 
     of the complex affine plane and 
       the affine morphism $\psi: \C^2_{u,v} \to \C^2_{x,y}$  defined by $x = u, y = u(1 + v)$.           
       Consider the plane  $\C^2_{u,v}$ as a toroidal surface with 
       boundary equal to the union of its coordinate axes, and $\C^2_{x,y}$ as a toroidal surface with 
       boundary equal to the $y$-axis. 
        As $\psi^{-1}(\partial \C^2_{x,y}) \subseteq \partial \C^2_{u,v}$, 
       $\psi$ is a morphism of toroidal varieties. 
        But it is not a toroidal morphism. Otherwise, 
       it would become the morphism $(u,v) \to (u,u)$ after analytic changes of coordinates 
       in the neighborhoods of the origins of the two planes, which is impossible, because 
       $\psi$ is birational, therefore dominant.       
 \end{remark}

Let us come back to the case of a smooth germ of surface $(S,o)$. 

\begin{definition}  \label{def:branchcross}
    A {\bf cross}\index{cross} on the smooth germ of surface $(S,o)$ is a pair $(L, L')$ of transversal 
    smooth branches on $(S,o)$. A local coordinate system $(x,y)$ on $(S,o)$ is said 
    {\bf to define the cross} $(L, L')$ if $L = Z(x)$ and $L' = Z(y)$. 
\end{definition}

We chose the name \emph{cross} by analogy with the denomination 
\emph{normal crossings divisor} (see Definition \ref{def:normcross}). 
Note the subtle difference between the two notions: 
the pair $(L, L')$ is a cross if and only if $L + L'$ is a normal crossings divisor,  
but the knowledge of the divisor does not allow to remember the order of its branches. 

\begin{definition} \label{def:ilattice}
      Let $(L, L')$ be a cross on $(S, o)$. We associate with it the two-dimensional lattice 
    $\boxed{M_{L, L'}}$ of integral divisors supported by $L \cup L'$, called 
    the {\bf monomial lattice of the cross $(L, L')$}\index{cross!its monomial lattice}.  
    The {\bf weight lattice of the cross $(L, L')$}\index{cross!its weight lattice} 
    is the dual lattice $\boxed{N_{L, L'}}$  of $M_{L,L'}$. 
    Denote by 
    $\boxed{(\epsilon_L, \epsilon_{L'})}$ the basis  $\epsilon_L := L, \epsilon_{L'} := L'$ of  $M_{L,L'}$, 
    by $\boxed{(e_L, e_{L'})}$ the dual basis 
     of $N_{L,L'}$, and  by $\boxed{\sigma_0^{L, L'}}$ the cone $\cone { \langle e_L, e_{L'} \rangle }$. 
    When the cross $(L,L')$ is clear from the context, we often write 
    simply $\boxed{(\epsilon_1, \epsilon_2 )}$, $\boxed{(e_1, e_2)}$ and $\boxed{\sigma_0}$ instead of 
    $ (\epsilon_L, \epsilon_{L'})$, $(e_L, e_{L'})$ and $\sigma_0^{L, L'}$ respectively. 
\end{definition}

Each time we choose local coordinates $(x,y)$ defining the cross $(L, L')$, we identify 
     $M_{L, L'}$ with the lattice of exponents of monomials in those coordinates. That is, 
     $a \epsilon_1 + b \epsilon_2$ corresponds to $x^a y^b$. Such a choice of coordinates 
also identifies holomorphically a neighborhood 
of $o$ in $S$ with a neighborhood of  the origin in $\C^2$ and the cross $(L,L')$ with 
the coordinate cross in $\C^2$ at the origin. Therefore, any subdivision $\fan$ of 
$\sigma_0$ defines an analytic modification 
$\psi^{\fan}_{L, L'} : S_{\fan} \to S $ of $S$. As these modifications 
are isomorphisms over $S \:  \setminus \: \{o\}$, it is easy to see that they are independent 
of the chosen coordinate system $(x,y)$ defining $(L, L')$, up to canonical 
analytical isomorphisms above $S$. Moreover, if we define $\partial S := L + L'$ and 
$\partial S_{\fan} := (\psi^{\fan}_{L, L'})^{-1}(L + L')$, the morphism 
$\psi^{\fan}_{L, L'}$ becomes a morphism from the toroidal surface 
$(S_{\fan}, \partial S_{\fan})$ to the toroidal surface $(S, \partial S)$.

\begin{definition} \label{def:toroimod}
     If  $\fan$ is a fan subdividing the cone $\sigma_0  \subset N_{L, L'}$, then  
     the morphism of the toroidal category 
        \[  \boxed{\psi^{\fan}_{L, L'} : (S_{\fan}, \partial S_{\fan}) \to (S, L+L')} \] 
     associated with $\fan$ is the \textbf{modification 
     of $S$ associated with $\fan$ relative to the cross $(L,L')$}.
     \index{modification!associated to a fan and a cross}
\end{definition}

 When the fan $\fan$ is regular, the morphism $\psi^{\fan}_{L, L'}$ 
   between the underlying complex  
   surfaces (forgetting the toroidal structures) is a composition of blow ups of points 
   (see Definition \ref{def:blowupgen}). 
   We will explain the structure of this decomposition of $\psi^{\fan}_{L, L'}$  in Section \ref{sec:embres} 
   (see Propositions \ref{prop:lotusdecomp}, \ref{prop:dualevolution}).

\subsection{Historical comments}
\label{ssec:HAtoric}
$\:$
\medskip

       Toric varieties were called \emph{torus embeddings} at the beginning of the development of toric 
       geometry in the 1970s, following the terminology of 
        Kempf, Knudsen, Mumford and Saint-Donat's  1973 book \cite{KKMS 73},  as 
        these are 
       varieties into which an algebraic torus embeds as an affine 
       Zariski open subset. The introduction of the book \cite{KKMS 73} contains information 
       about sources 
        of  toric geometry  in papers by Demazure, Hochster, Bergman, Sumihiro and 
        Miyake $\&$ Oda. 
        Details about the development of toric geometry may be found in Cox, Little and 
        Schenck's 2011 book \cite[Appendix A]{CLS 11}. 
        
        The first applications 
        of toric geometry to the study of singularities were done by Kouchnirenko,  
        Varchenko and Khovanskii in their 1976-77 papers  \cite{K 76}, \cite{V 76} and 
        \cite{K 77} respectively. But one may see in retrospect  toric techniques in Puiseux's 1850 
        paper  \cite[Sections 20, 23]{P 50}, in Jung's 1908 paper \cite{J 08},  in Dumas' 
        1911-12 papers \cite{D 11}, \cite{D 12}, in Hodge's 1930 paper \cite{H 30}, 
        in Hirzebruch's 1953 paper \cite{H 53} and in Teissier's 1973 paper \cite{T 73}. 
        Indeed, in all those papers, monomial changes of variables 
        more general than those describing blow ups are used in an essential way. 
     For instance, in his paper \cite{H 53}, Hirzebruch described the minimal resolution 
     of an affine toric surface by gluing the toric charts of the resolved surface 
     by explicit monomial birational maps. 
     Toric surfaces appeared in Hirzebruch's paper  as normalizations of the 
     affine surfaces in $\C^3$ defined by equations of the form $z^m = x^p y^q$, 
     with $(m, p, q) \in (\N^*)^3$ globally coprime. Interesting details about 
   Hirzebruch's work \cite{H 53} are contained in Brieskorn's paper \cite{B 00}.

       The notion of \emph{toroidal variety} of arbitrary dimension 
       was introduced in a slightly different form in the same book \cite{KKMS 73} of 
       Kempf, Knudson, Mumford and Saint-Donat. The emphasis 
       was put there on a given complex manifold $V$, and one looked 
       for partial compactifications of it which were locally analytically isomorphic to embeddings 
       of an algebraic torus into a toric variety. Such partial compactifications $\overline{V}$ were 
       called \emph{toroidal embeddings} of $V$. Therefore, a toroidal embedding was a pair 
       $(\overline{V}, V)$ such that $(\overline{V}, \overline{V} \:  \setminus \: V)$ is a toroidal 
       variety in our sense. For more remarks about the toroidal category see 
       \cite[Section 1.5]{AKMW 02}.

\section{Toroidal pseudo-resolutions of plane curve singularities}
\label{sec:tores}
\medskip

In Subsection \ref{ssec:NPcross} we introduce the notions of \emph{Newton polygon} $\cN_{L, L'}(C)$, 
\emph{tropical function} $\mbox{trop}^C_{L, L'}$, 
\emph{Newton fan} $\cF_{L, L'}^C$ 
and \emph{Newton modification} $\psi^C_{L, L'}$ (see Definition \ref{def:Npolseries}) 
determined by a curve singularity $C$ on the smooth germ of surface
$(S,o)$, relative to a cross $(L, L')$. The strict transform 
of $C$ by its Newton modification is a finite set of germs. If one completes for each one of 
them the corresponding germ of exceptional divisor into a cross, one gets again a Newton 
polygon, a fan and a modification. This produces an \emph{algorithm of toroidal pseudo-resolution} 
of $C$ (see Algorithm \ref{alg:tores}). It leads only to a \emph{pseudo-resolution} morphism, 
because its source is a possibly singular surface (with toric singularities). 
In Subsection \ref{ssec:toremb} we explain how to modify Algorithm \ref{alg:tores}  
in order to get an algorithm of \emph{embedded resolution} of $C$. 
 In Subsection \ref{ssec:fantrees} we encode the combinatorics of this algorithm into a 
\emph{fan tree} (see Definition \ref{def:fantreetr}), which is a rooted tree endowed with 
a \emph{slope function}, constructed by gluing \emph{trunks} associated with the Newton fans 
generated by the process. The final Subsection \ref{ssec:HAtoroidal} contains historical information 
about Newton's and Puiseux's work on plane curve singularities, the resolution of 
such singularities by iteration of morphisms which are toric in suitable coordinates, 
and the relations with tropical geometry.

\subsection{Newton polygons, their tropicalizations, fans and modifications}
\label{ssec:NPcross}
$\:$
\medskip

This subsection begins with the definitions of the \emph{Newton polygon} $\cN(f)$ 
(see Definition \ref{def:Npolalg}), the \emph{tropicalization} (see Definition \ref{def:tropicaliz}) 
and the \emph{Newton fan}  $\fan(f)$ (see Definition \ref{def:nfan})  
associated with a non-zero germ 
$f \in \C [[x,y]]$. It turns out that they only depend on the germs $L, L', C$ defined 
by $x, y$ and $f$ respectively (see Proposition \ref{prop:geominvNP}). 
Therefore, given a cross $(L,L')$ and a plane curve singularity 
$C$ on the smooth germ $(S, o)$, one has associated Newton polygon, 
tropicalization and fan. This fan allows to introduce the \emph{Newton modification} 
of the toroidal germ $(S, L + L')$ determined by $C$ (see Definition \ref{def:Npolseries}).
\medskip

Assume that a cross $(L, L')$ is fixed on $(S,o)$ (see Definition \ref{def:branchcross}) 
and that $(x,y)$ is a local coordinate system 
defining it. This system allows to see any $f \in \hat{\cO}_{S,o}$ as a series 
in the variables $(x,y)$, that is, in toric terms, as a possibly infinite sum 
of terms of the form $\boxed{c_m(f) \: \chi^m}$, for $c_m(f) \in \C$ and
$m \in \sigma_0^{\vee} \cap M$, where 
$\boxed{M} := M_{L, L'}$ and $\boxed{\sigma_0} := \sigma_0^{L, L'}$ 
(see Definition \ref{def:ilattice}). Denote also $\boxed{N} := N_{L, L'}$. 
One has canonical identifications $M \simeq \Z^2$, $N \simeq \Z^2$, 
$\sigma_0 \simeq (\R_+)^2$, and $\sigma_0^{\vee} \simeq (\R_+)^2$.

\begin{definition}  \label{def:seriesinv}
    Let $f \in \C[[x,y]]$ be a nonzero series. 
    The {\bf support}\index{support!of a series} 
    $\Supp(f) \subseteq \sigma_0^{\vee} \cap M \simeq \N^2$ of $f$ 
    is the set of exponents of monomials with non-zero coefficients 
    in $f$. That is, if 
          \begin{equation} \label{eq:expf} 
                f = \sum_{m \in \sigma^{\vee}_0 \cap M} c_m(f) \chi^m, 
          \end{equation}
     then
   $ \boxed{\Supp(f)}  := \{ m \in \sigma^{\vee}_0 \cap M, \:   c_m(f) \neq 0\}$.
\end{definition}

 If $Y$ is a subset of a real affine space, then $\boxed{\mbox{Conv}(Y)}$ denotes its 
     {\bf convex hull}.
\begin{definition}  \label{def:Npolalg}
    Let $f \in \C[[x,y]]$. Its {\bf Newton polygon}\index{Newton!polygon} $\cN(f)$ is the following convex 
    subset of $\sigma_0^{\vee} \simeq (\R_+)^2$: 
         \[ \boxed{\cN(f)} := \mathrm{Conv}(\Supp(f) + (\sigma^{\vee}_0 \cap M)).  \]
    Its {\bf faces} are its vertices, its edges and the whole polygon itself. 
     If $K$ is a compact edge of the boundary $\boxed{\partial \cN(f)}$ of $\cN(f)$, then the 
     {\bf restriction $\boxed{f_K}$ of $f$ to $K$}\index{restriction!of a series to an edge} 
     is the sum of the terms of $f$ 
     whose exponents belong to $K$. 
\end{definition}

\begin{remark}  \label{rem:npdeplc}
In general, the Newton polygon of an element of
$\hat{\cO}_{S,o}$ depends on the choice of local coordinates. 
For instance, let us consider the change of coordinates
$(x,y)=(u,u+v)$.  The function $f(x,y):=y^2-x^3$ becomes $g(u,v):= f(u, u+v) = (u+v)^2-u^3$.  
The corresponding Newton polygons are represented in  Figure \ref{fig:Newton polygon coordinates}. 
In contrast, if the local coordinate change preserves the 
coordinate curves, then the Newton polygon remains unchanged 
(see Proposition \ref{prop:geominvNP} below).

\end{remark}

  \begin{figure}[h!]
\begin{center}
\begin{tikzpicture}[scale=0.5]
\fill[fill=yellow!40!white] (0,2.8) --(0,2)--(3,0)--(3.8,0)--(3.8,2.8)--(0,2.8)--cycle;
\draw [->, color=black](0,0) -- (0,3);
\draw [->, color=black](0,0) -- (4,0);
\draw [-, color=black, thick](0,2) -- (3,0);
\node[draw,circle, inner sep=1.5pt,color=black, fill=black] at (3,0){};
\node[draw,circle, inner sep=1.5pt,color=black, fill=black] at (0,2){};
  \node [below, color=black] at (3, 0) {$(3,0)$};
  \node [left, color=black] at (0, 2) {$(0,2)$};
   \begin{scope}[shift={(8,0)},scale=1]
   \fill[fill=yellow!40!white] (0,2.8) --(0,2)--(2,0)--(3.8,0)--(3.8,2.8)--(0,2.8)--cycle;
\draw [->, color=black](0,0) -- (0,3);
\draw [->, color=black](0,0) -- (4,0);
\draw [-, color=black, thick](0,2) -- (2,0);
\node[draw,circle, inner sep=1.5pt,color=black, fill=black] at (2,0){};
\node[draw,circle, inner sep=1.5pt,color=black, fill=black] at (3,0){};
\node[draw,circle, inner sep=1.5pt,color=black, fill=black] at (0,2){};
\node[draw,circle, inner sep=1.5pt,color=black, fill=black] at (1,1){};
  \node [below, color=black] at (1.5, 0) {$(2,0)$};
  \node [left, color=black] at (0, 2) {$(0,2)$};
   \node [below, color=black] at (3.3, 0) {$(3,0)$};
    \node [right, color=black] at (1,1) {$(1,1)$};
    \node [below] at (-6,-1) {$\cN(y^2-x^3)$};
    \node [below] at (2,-1) {$\cN((u+v)^2-u^3)$};
   \end{scope}
    \end{tikzpicture}
\end{center}
 \caption{Illustration of Remark \ref{rem:npdeplc}}  
 \label{fig:Newton polygon coordinates} 
     \end{figure}
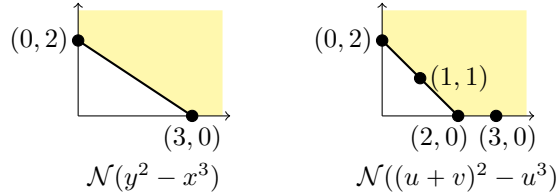

Suppose now that the variables $x$ and $y$ are weighted by non-negative real numbers. 
Denote by $c \in \R_+$ the weight of $x$ and by $d \in \R_+$ the weight of $y$. 
Therefore the pair $w := (c, d)$ may be seen as an element of the 
weight vector space $N_{\R}= (N_{L, L'})_{\R}$. More precisely, one has 
$w \in (\R_+)^2 \simeq  \sigma_0$. 
Assuming that the non-zero complex constants have weight $0$, we see that the weight 
$w(c_m(f) \chi^m)$ of a non-zero term of $f$ is simply $w \cdot m \in \R_+$. Define 
then the {\bf $w$-weight of the series}\index{weight!of a series} 
$f \in \C[[x,y]]$ as the minimal weight of its terms. 
One gets the function: 
           \begin{equation} \label{eq:defvalweight}
           \begin{array}{cccc} 
               \boxed{\nu_w }: &  \C[[x,y]]     &  \to &  \R_+ \cup \{ \infty\}  
               \\
                            &  f   &  \to  &  \min \{  w \cdot m, \:  m \in  \Supp(f)  \}
          \end{array}  .
       \end{equation}
 It is an exercise to show that $\nu_w$ is a valuation on the  
 $\C$-algebra $\C[[x,y]]$,  in the sense of Definition \ref{def:valuation}.

Instead of fixing $w$ and letting $f$ vary, let us fix now a non-zero series $f \in  \C[[x,y]]$.
Considering the $w$-weight of $f$ for every $w \in \sigma_0$ leads to the following function:

  \begin{definition}  \label{def:tropicaliz}
    The {\bf tropicalization}\index{tropicalization!of a series} 
    $\mathrm{trop}^f$ of $f  \in \C[[x,y]] \:  \setminus \: \{0\}$ 
    is  the function: 
       \begin{equation} \label{eq:deftrop}
           \begin{array}{cccc} 
               \boxed{\mathrm{trop}^f }: &   \sigma_0    &  \to &  \R_+   \\
                            &  w   &  \to  &  \min \{  w \cdot m, \:  m \in  \Supp(f)  \}
          \end{array}  .
       \end{equation}
\end{definition}

\begin{remark}  \label{rem:tropreason}
  Let us explain the name of \emph{tropicalization} used in the previous definition 
  (see also Subsection \ref{ssec:HAtoroidal}). 
  Consider the set $\overline{\R} := \R \cup \{+ \infty\}$, endowed with the operations 
  $\oplus := \min$ and $\odot:= +$. Under both operations, $\overline{\R}$ is a commutative 
  monoid, the product $\odot$ is distributive with respect to addition and the addition 
  $\oplus$ is idempotent, that is, $a \oplus a = a$, for all $a \in \overline{\R}$. One says 
  then that $(\overline{\R}, \: \oplus, \: \odot)$ is a \emph{tropical semiring}. Consider 
  now the expression defining $\mbox{trop}^f$, and compare it with the expansion 
  (\ref{eq:expf}) of $f$ as a power series. One sees that one gets formally 
  $\mbox{trop}^f$ from (\ref{eq:expf}) by replacing each constant or variable $x$, $y$ by 
  its weight, and by replacing the usual operations of sum and product by their 
  tropical analogs. 
  For further references see the textbook \cite{MS 16} on \emph{tropical geometry}. 
Foundations for the tropical study of singularities were written by Stepanov and the third author in  the paper \cite{PS 13}. 
\end{remark}

\begin{remark}  \label{rem:supfunct} 
   If $A$ is a subset of a real vector space $V$, then its \emph{support function} is the 
    function defined on the dual vector space $V^{\vee}$ and taking values in $\R \cup \{ - \infty \}$, 
    which associates to every element 
    of $V^{\vee}$ seen as a linear form on $V$, the infimum of its restriction to $A$. 
    The tropicalization $\mathrm{trop}^f$ is the restriction of the support 
    function of the subset $\Supp(f)$ of the real vector space $M_{\R}$ to the subset of 
    $M_{\R}^{\vee} \simeq N_{\R}$ on which it does not take the value $- \infty$.
   The notion of support function is an essential tool in the study 
   of convex polyhedra (see for instance Ewald's book \cite{E 96}).
\end{remark}

For every ray $\rho = \R_+ w$ included in the cone $\sigma_0$, consider the following 
closed half-plane of $M_{\R}$:
    \begin{equation}  \label{eq:closedhp}
         \boxed{H_{f, \rho}} := \{m \in M_{\R}, \:  w \cdot m \geq  \mathrm{trop}^f(w)\}. 
    \end{equation}
This definition is independent of the choice a generator  $w$ of the ray $\rho$. 

The basic reason of the importance of the Newton polygon $\cN(f)$ of $f$ in our context is the following 
strengthening of Proposition \ref{prop:infinconv}:

\begin{proposition}  \label{prop:minrealiz}
        Let the ray $\rho \subset \sigma_0$ be fixed. 
                Then the closed half-plane $H_{f, \rho}$ of $M_{\R}$
         is a {\bf supporting half-plane} of $\cN(f)$, in the sense that it contains $\cN(f)$
         and its boundary $ \{m \in M_{\R}, \: w \cdot m =  \mathrm{trop}^f(w) \}$ 
         has a non-empty intersection with the boundary $\partial \cN(f)$ of $\cN(f)$.         
\end{proposition}
    \begin{proof} Let $w$ be a generating vector of the ray $\rho$. 
       The inclusion $\cN(f) \subseteq H_{f, \rho}$ is equivalent to the property
         $w \cdot n \geq  \mathrm{trop}^f(w),  \mbox{ for all } n \in \cN(f)$.
        These inequalities result from Definition \ref{def:tropicaliz} of the tropicalization 
        function $\mathrm{trop}^f(w)$ and from the following basic equality, implied by  
        the hypothesis that $w \in \sigma_0$ (see Proposition \ref{prop:infinconv}):
            \[  \min \{  w \cdot m, \:  m \in  \Supp(f)  \} =   \min \{  w \cdot m, \:  m \in  \cN(f)  \}.  \]
        The boundary of the half-plane $H_{f, \rho}$ intersects $\cN(f)$ at its points at 
        which the restriction of the linear form $w : M_{\R} \to \R$ to $\cN(f)$ achieves its minimum, 
        that is, along its face
            $\cN(f) \cap \{m \in M_{\R}, \:  w \cdot m  =  \mathrm{trop}^f(w)\}$.       
    \end{proof}

As every closed convex subset of a real plane is the intersection of its supporting half-planes, 
one deduces that the tropicalization $\mathrm{trop}^f$ determines the Newton polygon 
$\cN(f)$ in the following way:
   \begin{equation}   \label{eq:fromtropton}
       \cN(f) = \{ m \in M_{\R}, \:  w \cdot m \geq  \mathrm{trop}^f(w), \:  \hbox{\rm for all} \: w \in \sigma_0\}.
   \end{equation}
   
  Formula (\ref{eq:fromtropton}) presents $\cN(f)$ as the intersection of an infinite set 
  of closed half-planes. In fact, as a consequence of the previous discussion, 
  a finite number of them suffices: 
  
   \begin{proposition}  \label{prop:finumb}
      Let $\fan(f)$ be the fan of $N$ obtained by subdividing the cone $\sigma_0$ using the 
      rays orthogonal to the compact edges of $\cN(f)$. Then:
        \begin{enumerate}
            \item  The tropicalization $ \mathrm{trop}^f$ is 
                  continuous and its restriction to any cone in $\fan(f)$ is linear.                 
            \item The relative interiors of the cones 
                of $\fan(f)$ may be characterized as the levels of the following map 
               from $\sigma_0$ to the set of faces of $\cN(f)$, in the sense of Definition \ref{def:Npolalg}: 
                   \[w \to \cN(f) \cap \{m \in M_{\R}, \: w \cdot m =  \mathrm{trop}^f(w) \}.   \]
            \item \label{inclrev}
                   This map realizes an inclusion-reversing bijection between  
                   $\fan(f)$ and the set of faces of $\cN(f)$. If 
                 $K_{\sigma}$ is the face of $\cN(f)$ corresponding to the cone $\sigma$ of 
                   $\fan(f)$, then:
                  $$ \mathrm{trop}^f(w) = w \cdot m, \:  \hbox{\rm for all} \: w \in \sigma,  
                          \: \hbox{\rm and for all} \: m \in K_{\sigma}.$$
             \item The Newton polygon $\cN(f)$ is the intersection of the closed half-planes 
                 $H_{f, \rho}$ defined by relation (\ref{eq:closedhp}), where 
                 $\rho$ varies among the rays of the fan $\fan(f)$.
         \end{enumerate}
  \end{proposition}
  
  The fans $\fan(f)$ appearing in the previous proposition are particularly 
  important for the sequel,  that is why they deserve a name:
  
  \begin{definition} \label{def:nfan}
      The {\bf Newton fan}\index{fan!Newton} $\boxed{\fan(f)}$ of $f \in \C[[x,y]] \: \setminus \:  \{0\}$ 
      is the fan  of $N$ obtained by subdividing 
      the cone $\sigma_0$ using the  rays orthogonal to the compact edges of the Newton polygon 
      $\cN(f) \subseteq \sigma_0^{\vee}$ of $f$, that is, by the interior normals of the 
      compact edges of $\cN(f)$. A {\bf Newton fan} in a weight lattice 
      $N$ and relative to a basis $(e_1, e_2)$ is any fan subdividing the regular cone 
      $\sigma_0= \cone \langle e_1, e_2 \rangle$.
  \end{definition}

  \begin{example}  \label{ex:Newtonobj}
      Consider the series $f \in\C[[x,y]]$ defined by:
            \[ f(x,y):= - x^{12} + x^{14} + x^7 y^2 + 2 x^5 y^3- x^{10} y^3 + x^3 y^4 + 3 x^7 y^4 + y^9.\]

             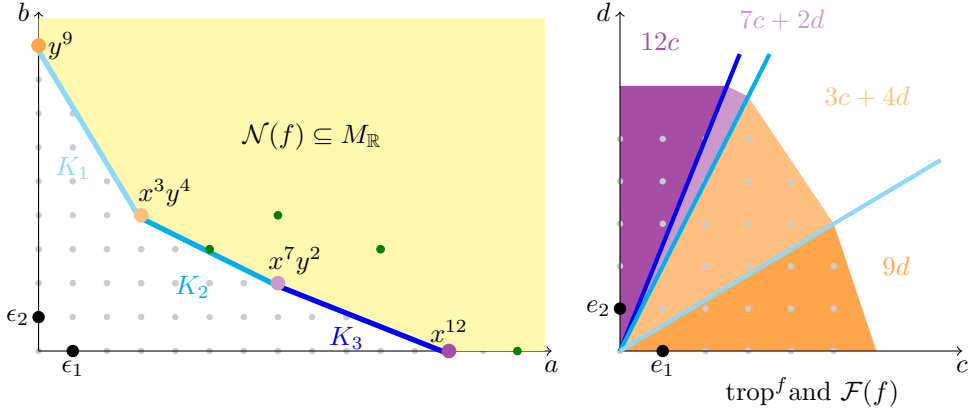
\begin{figure}[h!]
\begin{center}
\begin{tikzpicture}[scale=0.45]
\fill[fill=yellow!10] (0,9.8) --(0,9)-- (3,4) -- (7,2) -- (12,0) -- (14.8,0) -- (14.8,9.8)--cycle;
\foreach \x in {0,1,...,14}{
\foreach \y in {0,...,0}{
       \node[draw,circle,inner sep=0.7pt,fill, color=gray!40] at (1*\x,1*\y) {}; }
   }
  \foreach \x in {0,...,10}{
\foreach \y in {1,...,3}{
       \node[draw,circle,inner sep=0.7pt,fill, color=gray!40] at (1*\x,1*\y) {}; }
   } 
   \foreach \x in {0,1,...,7}{
\foreach \y in {2,...,4}{
       \node[draw,circle,inner sep=0.7pt,fill, color=gray!40] at (1*\x,1*\y) {}; }
   }
     \foreach \x in {0,1,...,2}{
\foreach \y in {5,...,5}{
       \node[draw,circle,inner sep=0.7pt,fill, color=gray!40] at (1*\x,1*\y) {}; }
   } 
     \foreach \x in {0,...,1}{
\foreach \y in {6,...,7}{
       \node[draw,circle,inner sep=0.7pt,fill, color=gray!40] at (1*\x,1*\y) {}; }
   } 
    \node[draw,circle,inner sep=0.7pt,fill, color=gray!40] at (0,8) {};
\draw [->, color=black](0,0) -- (0,10);
\draw [->, color=black](0,0) -- (15,0);
\node[draw,circle, inner sep=1.5pt,color=black, fill=black] at (1,0){};
\node[draw,circle, inner sep=1.5pt,color=black, fill=black] at (0,1){};
\draw [-,line width=4pt, color=cyan!40](0,9) -- (3,4);
   \node [below, color=cyan!40] at (1, 6) {$K_1$};
\draw [-,line width=4pt, color=cyan](3,4) -- (7,2);
    \node [below, color=cyan] at (4.5, 2.5) {$K_2$};
\draw [-,line width=4pt, color=blue](7,2) -- (12,0);
     \node [below, color=blue] at (9, 1) {$K_3$};

\fill[fill=yellow!40!white] (0,9.8) --(0,9)-- (3,4) -- (7,2) -- (12,0) -- (14.8,0) -- (14.8,9.8)--cycle;
\node [left] at (0,1) {$\epsilon_{2}$};
\node [below] at (1,0) {$\epsilon_{1}$};
\node[draw,circle, inner sep=1pt,color=black!50!green, fill=black!50!green] at (14,0){};
\node[draw,circle, inner sep=1.8pt,color=violet!70, fill=violet!70] at (12,0){};
\node [above] at (12,0) {$x^{12}$};
\node[draw,circle, inner sep=1pt,color=green!50!black, fill=green!50!black] at (10,3){};
\node[draw,circle, inner sep=1.8pt,color=violet!40, fill=violet!40] at (7,2){};
\node [right,above] at (7.5,2) {$x^{7}y^{2}$};
\node[draw,circle, inner sep=1pt,color=green!50!black, fill=green!50!black] at (7,4){};
\node[draw,circle, inner sep=1pt,color=green!50!black, fill=green!50!black] at (5,3){};
\node[draw,circle, inner sep=1.8pt,color=orange!50, fill=orange!50] at (3,4){};
\node [right,above] at (3.7,4) {$x^{3}y^{4}$};
\node[draw,circle, inner sep=1.8pt,color=orange!70, fill=orange!70] at (0,9){};
\node [right] at (0,9) {$y^{9}$};
\node [left] at (0,10) {$b$};
\node [below] at (15,0) {$a$};
\node [below] at (8,7) {${\cN}(f)\subseteq M_{\mathbb R}$};

 \begin{scope}[shift={(17,0)},scale=1.25]
 \fill[fill=orange!70] (0,0) --(6,0)-- (5,3) --cycle;
 \node [left,color=orange!70] at (7,2) {$9d$};
\fill[fill=orange!50] (0,0) --(3,6)-- (5,3) --cycle;
 \node [left,color=orange!50] at (7,6) {$3c+4d$};
\fill[fill=violet!70] (0,0) --(2.5,6.25)-- (0,6.25) --cycle;
 \node [left,color=violet!40] at (5,7.8) {$7c+2d$};
\fill[fill=violet!40] (0,0) --(2.5,6.25)-- (3,6) --cycle;
 \node [left,color=violet!70] at (1.6,7.3) {$12c$};
 \foreach \x in {0,1,...,3}{
\foreach \y in {0,1,...,5}{
       \node[draw,circle,inner sep=0.7pt,fill, color=gray!40] at (1*\x,1*\y) {}; }
   }
    \foreach \x in {4,...,4}{
\foreach \y in {0,1,...,4}{
       \node[draw,circle,inner sep=0.7pt,fill, color=gray!40] at (1*\x,1*\y) {}; }
   }
   
    \foreach \x in {5,...,5}{
\foreach \y in {0,1,...,2}{
       \node[draw,circle,inner sep=0.7pt,fill, color=gray!40] at (1*\x,1*\y) {}; }
   }
\draw [->, color=black](0,0) -- (0,8);
 \node [below] at (8,0) {$c$};
\draw [->, color=black](0,0) -- (8,0);
 \node [left] at (0,8) {$d$};
\draw [-, ultra thick, color=blue](0,0) -- (2.8,7);
\draw [-, ultra thick, color=cyan](0,0) -- (3.5,7);
\draw [-, ultra thick, color=cyan!40](0,0) -- (7.5,4.5);
 \node [below] at (4.5,-0.4) {$\mathrm{trop}^f \hbox{\rm and } {\fan}(f)$};

\node [left] at (0,1) {$e_{2}$};
\node [below] at (1,0) {$e_{1}$};
\node[draw,circle, inner sep=1.5pt,color=black, fill=black] at (1,0){};
\node[draw,circle, inner sep=1.5pt,color=black, fill=black] at (0,1){};
    \end{scope}
\end{tikzpicture}
\end{center}
 \caption{The Newton polygon, the tropicalization and the Newton fan of Example \ref{ex:Newtonobj}}  
 \label{fig:Newton polygon} 
     \end{figure}

     On the left side of Figure \ref{fig:Newton polygon} is represented its Newton polygon $\cN(f)$,  
    and on the right side are represented its tropicalization $ \mathrm{trop}^f$ 
    and its Newton fan $\fan(f)$. The support of the series $f$ is:
         \[ \Supp(f) = \{ (12,0), (14,0), (7,2), (5,3), (10,3), (3,4), (7,4), (0,9) \}.\]
       Among its elements, the vertices of $\cN(f)$ are $(12,0), (7,2), (3,4), (0,9)$. The corresponding 
       monomials are marked on the left of the figure, near the associated vertices. 
       The other elements of $\Supp(f)$ are marked as green dots. Now, each 
       vertex $(a,b)$ of $\cN(f)$ may be seen as the linear function $w=(c,d) \to ac + b d$  on $N_{\R}$. 
       The tropicalization $ \mathrm{trop}^f$ computes the minimal value of those $4$ linear 
       functions at the points of $\sigma_0$. The regular cone $\sigma_0$ 
       gets decomposed into $4$ smaller 
       $2$-dimensional subcones, according to the vertex which gives this minimum. On the 
       right side of Figure \ref{fig:Newton polygon} those subcones are represented in 
       different colors. Each such subcone has the same color as the expression 
       of the associated linear function and the vertex of $\cN(f)$ defining it. Each ray  
       separating two successive subcones is orthogonal to a compact edge of $\cN(f)$ and both are 
       drawn with the same color. 
       Denoting the compact edges by $K_1 := [(0,9), (3, 4)]$, $K_2 := [(3,4), (7, 2)]$, 
       $K_3 := [(7,2), (12, 0)]$,  the associated  restrictions of $f$ (see Definition \ref{def:Npolalg}) are:
       \[
        f_{K_1} = x^3 y^4 + y^9,   \quad f_{K_2} =  x^7 y^2 + 2 x^5 y^3 + x^3 y^4, \quad 
             \mbox{ and } f_{K_3} =  - x^{12} + x^7 y^2.
        \]
       The Newton fan of $f$ is $\fan(f)=\fan\left(3/5, 2/1, 5/2\right)$ 
       (see Definition \ref{def:fan2}  for this last notation).
  \end{example}

 If $\alpha \in \C[[t]] \setminus \{0\}$, we denote by $c_{\nu_t (\alpha)}(\alpha)$  
 the coefficient of $t^{\nu_t (\alpha)}$ in 
the series $\alpha$, and we call it the  \textbf{leading coefficient} of $\alpha$.

The following proposition shows why it is important to introduce $\mathrm{trop}^f$ when studying the germ $C$ defined by $f$:

\begin{proposition}  \label{prop:imptrop} 
        Let $f \in \C [[x,y]]$ be a non-zero series. 
        Let $t \to (\alpha(t), \beta(t))$ be a germ of formal morphism from $(\C,0)$ to 
        $(\C^2, 0)$, whose image is not contained in the union $L \cup L'$ of the coordinate axes. 
       Then one has the inequality:
          \[ \nu_t (f(\alpha(t), \beta(t))) \geq \:  \mathrm{trop}^f(\nu_t (\alpha), 
               \nu_t (\beta)), \]
       with equality if and only if $f_K( c_{\nu_t (\alpha)} (\alpha ) ,  
       c_{\nu_t (\beta)} (\beta )) \neq 0$,  
       where $K$ is the compact face of $\cN(f)$ orthogonal to  $(\nu_t (\alpha),  \nu_t (\beta)) \in N$, in the sense that its restriction 
       to $\cN(f)$ achieves its minimum on this face.
  \end{proposition}

     \begin{proof}
        The basic idea of the proof goes back to Newton's method of computing the 
        leading term of a Newton-Puiseux series $\eta(x)$ such that $f(x, \eta(x)) =0$, 
        which we explained on the example of Subsection \ref{ssec:firtsexample}, starting 
        from equation (\ref{eq:Newt1step}). 
        
        The hypothesis that the image of $t \to (\alpha(t), \beta(t))$ is not contained in the union 
        of coordinate axes means that both $\alpha$ and $\beta$ are non-zero series. Therefore, 
        they admit non-vanishing leading coefficients $c_{\nu_t (\alpha)}(\alpha)$ and 
        $c_{\nu_t (\beta)}(\beta)$ (see Definition \ref{def:NewtPuiseux}). 
        
       Using the expansion (\ref{eq:expf}), we get 
        that $f(\alpha(t), \beta(t))$ is equal to:
          \begin{equation} \label{eq:computcomp}
                 \begin{array}{l}
             \displaystyle{\sum_{(a,b) \in \Supp(f)} c_{(a,b)}(f) 
                      \left(c_{\nu_t (\alpha)} (\alpha ) t^{\nu_t(\alpha)} + 
                             o( t^{\nu_t(\alpha)})\right)^a 
                      \left(c_{\nu_t (\beta)} (\beta ) t^{\nu_t(\beta)} + 
                           o( t^{\nu_t(\beta)})\right)^b} = \\
                       =   \displaystyle{\sum_{(a,b) \in \Supp(f)} c_{(a,b)}(f) 
                               \left(c_{\nu_t (\alpha)} (\alpha )\right)^a
                                  \left(c_{\nu_t (\beta)} (\beta )\right)^b \:   
                    \left( t^{a \nu_t(\alpha) + b \nu_t(\beta)} + 
                        o(t^{a \nu_t(\alpha) + b \nu_t(\beta)} ) \right)}.
                   \end{array}    
            \end{equation}
           As a consequence:
               \[                \nu_t \left( f(\alpha(t), \beta(t))\right)  
                   \geq\displaystyle{  \min_{(a,b) \in \Supp(f)} \{ a \nu_t(\alpha) + 
                          b \nu_t(\beta) \} }   
                    = \mbox{trop}^f(\nu_t (\alpha) , \nu_t(\beta) ),  \]
              where the last equality follows from Definition \ref{def:tropicaliz}. This proves the inequality 
              stated in the proposition. 
              
              The case of equality follows from the fact,  implied by 
              the computation (\ref{eq:computcomp}), that the coefficient  
              of the term with exponent $\mbox{trop}^f(\nu_t (\alpha) , \nu_t (\beta) )$ 
              of the series $ f(\alpha(t), \beta(t))$ is 
              $f_K( c_{\nu_t (\alpha)} (\alpha ) ,  c_{\nu_t (\beta)} (\beta ))$. 
     \end{proof}

 In Proposition \ref{prop:imptrop}, $K$ may be either an edge or a vertex of $\cN(f)$. Note that 
  this statement plays with the two dual ways of defining a curve singularity on $(\C^2, 0)$, either 
  as the vanishing locus of a function or by a parametrization.  
  
 Consider now the reduced image of the morphism $t \to (\alpha(t), \beta(t))$.  The hypothesis that 
 it is not contained in $L \cup L'$ shows that it is a branch on $(S,o)$, different from 
 $L$ and $L'$. Endow it with a multiplicity equal to the degree of the morphism 
 onto its image, 
 seeing it therefore as a divisor $A$ on $(S,o)$. By Proposition \ref{prop:disymint}, the 
 orders $\nu_t(\alpha(t)), \nu_t(\beta(t))$ which appear in 
 Proposition \ref{prop:imptrop} may be interpreted as
    $\nu_t(\alpha(t))  = L \cdot A $, and $\nu_t(\beta(t)) = L' \cdot A$.
 We get the following corollary of Proposition \ref{prop:imptrop}:
   
   \begin{proposition}  \label{prop:imptropintr} 
        Let $(L, L')$ be a cross on $(S,o)$ and $C$ be a curve singularity 
     on $(S,o)$. Assume that the local coordinate system $(x,y)$ defines the cross $(L, L')$ 
     and that $f \in \hat{\cO}_{S,o}$ defines  $C$. Then, for every 
        effective divisor $A$ on $(S,o)$ supported on a branch distinct from 
        $L$ and $L'$, one has the inequality:
          \[ C \cdot A \geq \:  \mathrm{trop}^f((L \cdot A) e_1  + (L' \cdot A) e_2). \]
       Moreover, one has equality when $A$ is generic for fixed values of 
       $L \cdot A$ and $L' \cdot A$. 
  \end{proposition}
   
   One may describe the genericity condition involved in the last sentence 
   of Proposition \ref{prop:imptropintr} as follows. 
  As a consequence of the proof of Proposition 
   \ref{prop:propstrict} below, one has    
   $f_K( c_{\nu_t (\alpha)} (\alpha ) ,  c_{\nu_t (\beta)} (\beta )) \neq 0$ 
   (which is equivalent 
   to the equality $C \cdot A = \:  \mathrm{trop}^f((L \cdot A) e_1  + (L' \cdot A) e_2)$) if and 
   only if the strict transforms of $A$ and $C$ by the Newton modification $\psi^C_{L, L'}$ 
   of $S$ defined by $C$ (see Definition \ref{def:Npolseries} below) are disjoint.

As a consequence of Propositions \ref{prop:finumb} (\ref{inclrev}) and \ref{prop:imptropintr} we have: 

\begin{proposition}  \label{prop:geominvNP}
     Let $(L, L')$ be a cross on $(S,o)$ and $C$ be a curve singularity 
     on $(S,o)$. Assume that the local coordinate system $(x,y)$ defines the cross $(L, L')$ 
     and that $f \in \hat{\cO}_{S,o}$ defines  
     $C$. Then the Newton polygon $\cN(f)$, the tropicalization $\mathrm{trop}^f$ and 
     the Newton fan $\fan(f)$  
     do not depend on the choice of the defining functions $x, y , f$ of the curve germs $ L, L', C$. 
\end{proposition}

By contrast, the support of $f$ depends on the choice of the  local coordinate system 
$(x,y)$ defining a fixed cross, even if $f \in \hat{\cO}_{S,o}$ is fixed. 
For instance, the monomial $xy$ becomes 
a series with infinite support if one replaces simply $x$ by $x( 1 + x + x^2 + \cdots)$.

 Proposition \ref{prop:geominvNP} implies that the following notions are well-defined: 
 
\begin{definition}  \label{def:Npolseries}
     Let $(L, L')$ be a cross on  $(S,o)$, and 
     let $(x,y)$ be a local coordinate system defining it. Let $C$ be a curve singularity  
     on $(S,o)$, defined by a function $f \in \hat{\cO}_{S,o}$, seen 
     as a series in $\C[[x,y]]$ using the coordinate system $(x,y)$. Then: 
       
       \noindent 
       $\bullet$ 
       The  {\bf Newton polygon $\boxed{\cN_{L, L'}(C)} \subseteq M_{L, L'}$ of $C$ 
                 relative to the cross $(L, L')$}\index{polygon!Newton, relative to a cross}
                 is the Newton polygon $\cN(f)$.
        
        \noindent 
       $\bullet$ 
             The {\bf tropical function}\index{tropical!function}  
             $\boxed{\mbox{trop}^C_{L, L'}} : \sigma_{0} \to \R_+$ of $C$ 
                 {\bf relative to the cross $(L, L')$} is the tropicalization $\mbox{trop}^f$ of the series $f$.
           
            \noindent 
       $\bullet$ 
            The {\bf Newton fan $\boxed{\fan_{L, L'}(C)}$ of $C$ relative to the cross $(L, L')$} 
                \index{fan!Newton, relative to a cross} is the fan $\cF(f)$.
          
           \noindent 
       $\bullet$ 
         The {\bf Newton modification 
                 $\boxed{\psi_{L, L'}^{C}}: (S_{\cF_{L, L'}(C)} ,  \partial S_{\cF_{L, L'}(C)}) 
                      \to (S, L + L')$ of $S$ defined by $C$ relative to the cross $(L,L')$} 
                      \index{modification!Newton, relative to a cross} is the modification 
                  of $S$ associated with $\fan_{L, L'}(C)$ relative to the cross $(L, L')$, that is, 
                  $\psi_{L, L'}^{C} := \psi_{L, L'}^{\fan_{L,L'}(C)}$ (see  Definition \ref{def:toroimod}).     
                  The strict transform of $C$ by $\psi_{L, L'}^{C}$ is denoted $\boxed{C_{L, L'}}$. 
\end{definition}

Note that we consider the Newton modification $\psi_{L, L'}^{C}$  
as a morphism in the toroidal category, by
endowing $S$ with the boundary $L + L'$ and the modified surface 
$S_{\cF_{L, L'}(C)}$ with a boundary equal to the reduced total transform of $L + L'$.

\subsection{An algorithm of toroidal pseudo-resolution}
\label{ssec:algtores}
$\:$
\medskip

In this subsection we assume for simplicity that the plane curve singularity $C$ 
is reduced (see Remark \ref{rem-nonreduced}).
We explain that, once a smooth branch $L$ is fixed on the germ 
of smooth surface $(S,o)$, one may obtain a so-called \emph{toroidal pseudo-resolution} 
of $C$ on $(S,o)$ (see Definition \ref{def:threeres})
by completing the smooth branch into a cross $(L, L')$, by performing the associated 
Newton modification, and by iterating these steps at every point at which the strict transform 
of $C$ intersects the exceptional divisor of the Newton modification 
(see Theorem \ref{thm:algstop}). The algorithm stops after the first step if and 
only if $C$ is \emph{Newton non-degenerate} relative to the cross $(L, L')$ 
(see Definition \ref{def:newtnondegen}).
\medskip

The following definition formulates two notions of possibly partial 
\emph{resolution of $C$ in the toroidal category},  relative to the ambient smooth germ of surface $S$:

\begin{definition}  \label{def:threeres}
     Let $(L, L')$ be a cross in the sense of Definition \ref{def:branchcross}
     on the smooth germ of surface $(S, o)$  and let $C$ be a curve singularity on $S$. 
     Consider a modification $\pi : (\Sigma, \partial \Sigma) \to (S, L + L')$  
     of $(S, L+L')$ in the toroidal category, 
     in the sense of Definition \ref{def:toroidobj}. It is called, in decreasing generality: 
     
     \noindent
    $\bullet$  A {\bf toroidal pseudo-resolution of $C$}\index{toroidal!pseudo-resolution} 
        if the following conditions are satisfied: 
                 \begin{enumerate} 
                         \item the boundary $\partial \Sigma$ of $\Sigma$ contains the 
                                  reduction of the total transform $\pi^*(C)$ of $C$ by $\pi$;  
                         \item the strict transform of $C$ by $\pi$ (see Definition \ref{def:modiftransf})
                                  does not contain singular points of $\Sigma$. 
                 \end{enumerate}
       
       \noindent
       $\bullet$  A {\bf toroidal embedded resolution  of $C$}\index{toroidal!embedded resolution} 
           if, moreover, the surface $\Sigma$ is smooth. 
       
    If $\pi : (\Sigma, \partial \Sigma) \to (S, L + L')$ is a toroidal pseudo-resolution of $C$, 
    then the reduction of the image $\pi(\partial \Sigma)$ of $\partial \Sigma$ in $S$
    is called the {\bf completion $\boxed{\hat{C}_{\pi}}$ of $C$ relative to $\pi$}.  
    \index{completion!of a plane curve singularity}
\end{definition}

\begin{remark} \label{rem:comparisonres}
   Note that if $\pi : (\Sigma, \partial \Sigma) \to (S, L + L')$   is a toroidal pseudo-resolution of $C$, 
   then the strict transform of $C$ by $\pi$ is smooth and $\hat{C}_{\pi} \supseteq C + L + L'$. 
   If moreover $\pi$ is an embedded resolution, 
   then the total transform $\pi^*(C)$ is a normal crossings divisor in $\Sigma$ 
   (see Definition \ref{def:normcross}). 
   Note also that if $\pi : (\Sigma, \partial \Sigma) \to (S, L + L')$ is a 
   toroidal embedded resolution of $C$, 
   then $\pi : \Sigma \to S$ is an embedded 
   resolution of $C$ according to Definition \ref{def:embres}. 
   From now on, we will keep track carefully of the toroidal structures, considering 
   only  toroidal embedded resolutions in the sense of Definition \ref{def:threeres}. 
\end{remark}

\begin{remark}  \label{rem:goldin-teissier}
    If $\pi : (\Sigma, \partial \Sigma) \to (S, L + L')$ is a toroidal pseudo-resolution of $C$,  
    then the strict transform of $C$ is transversal to the critical locus of $\pi$. 
    Our choice of terminology in Definition \ref{def:threeres} is inspired by 
    Goldin and Teissier's paper \cite{GT 00}, where an analogous notion of  (embedded) 
    \emph{toric pseudo-resolution} of a subvariety of the affine space is considered.
\end{remark}

Let us look now at the strict transform $C_{L, L'}$ of $C$  by the Newton modification 
$\psi_{L, L'}^{C}$ defined by $C$ relative to the cross $(L, L')$ (see Definition \ref{def:Npolseries}). 
The following proposition describes its intersection with the boundary $\partial S_{\cF_{L, L'}(C)}$:

\begin{proposition}   \label{prop:propstrict}
    Assume that neither $L$ nor $L'$ is a branch of $C$. 
   Then the strict transform $C_{L, L'}$ of $C$ by the Newton modification $\psi_{L, L'}^{C}$ intersects  
   the boundary $\partial S_{\cF_{L, L'}(C)}$ of the toroidal surface \linebreak 
   $(S_{\cF_{L, L'}(C)} ,  \partial S_{\cF_{L, L'}(C)})$ only at smooth points of it. Moreover, 
   if $\rho$ is a ray of the Newton fan $\fan_{L, L'}(C)$ different from the edges of 
   $\sigma_0$, then $C_{L, L'}$ intersects the corresponding component $\overline{O}_{\rho}$ 
   of the exceptional divisor of $\psi_{L, L'}^{C}$ only inside the orbit ${O}_{\rho}$.
   The number of intersection points counted with multiplicities is 
   equal to the integral length of the edge of the Newton polygon $\cN_{L, L'}(C)$ which 
   is orthogonal to the ray $\rho$. 
\end{proposition}

\begin{proof} We give a detailed proof of this proposition in geometric language, in order 
      to emphasize the roles played by the fundamental combinatorial objects 
      $\cN_{L, L'}(C)$, $\mathrm{trop}^C_{L, L'}$ and $\fan_{L, L'}(C)$ associated 
      with $C$ relative to the cross $(L,L')$ (see Definition \ref{def:Npolseries}). 
       
      The orbit $O_{\rho}$ is independent of the toric surface containing it, because any two 
      such surfaces contain the affine toric surface $X_{\rho} \supset O_{\rho}$ as Zariski 
      open sets. Therefore, in order to compute the intersection of the strict transform of $C$ 
      with $O_{\rho}$, we may choose another surface than $X_{\cF_{L, L'}(C)}$. 
      
      Choose local coordinates $(x,y)$ defining the cross $(L , L')$. In this way $M_{L, L'}$ 
      gets identified with the lattice of exponents of Laurent monomials in $(x,y)$.       
       Assume that $f_1:= \alpha e_1 + \beta e_2$ is the unique primitive generator of the ray $\rho$. 
       Let us complete it in a basis $(f_1, f_2)$ of the lattice $N_{L, L'}$, such that the cone 
       $\sigma:= \cone \langle f_1, f_2\rangle$ is contained in one of the two cones of 
       dimension $2$ of $\fan_{L, L'}(C)$ adjacent to $\rho$. 
       We are now in the setting of Example \ref{ex:tworegcones}. As explained there, 
       if $(\varphi_1, \varphi_2)$ is the basis of $M_{L, L'}$ dual to the basis $(f_1, f_2)$ of 
       $N_{L, L'}$ and $u:= \chi^{\varphi_1},  v := \chi^{\varphi_2}$,   then 
       $v = x^{-\beta} y^{\alpha}$ is a coordinate 
       of the orbit $O_{\rho}$. Moreover, it realises an isomorphism  of its closure in 
       the affine toric surface $X_{\sigma} = \C^2_{u,v}$ with the affine line $\C_v$. 
              
       Let $K_{\rho}$ be the edge of the Newton polygon $\cN_{L, L'}(C)$ which is orthogonal to the 
       ray $\rho$. It is parallel to the line $\R \varphi_2$, because by definition $f_1 \cdot \varphi_2 =0$. 
       Orient $K_{\rho}$ by the vector $\varphi_2$ and denote then its vertices by $m_0$ and $m_1$,  
       such that $K_{\rho}$ is oriented from $m_0$ to $m_1$. This means that 
           $ m_1 - m_0 = L_{\rho} \: \varphi_2$,  
        where $L_{\rho}$ denotes the integral length of the segment $K_{\rho}$, in the sense of 
        Definition \ref{def:intlength}. Moreover, the points of $K_{\rho} \cap M$ are precisely 
        those of the form: 
           \begin{equation}  \label{eq:pointsonedge}
                 m := m_0 + k \: \varphi_2, \mbox{ for } k \in \{0, 1, \dots, L_{\rho}\}.
            \end{equation}
        
        Consider the smooth toric surface  
        $X_{\sigma}= \C^2_{u,v}$. The orbit $O_{\rho}$ is its pointed $v$-axis $\C^*_v$. 
        Therefore, one may compute the intersection of the strict transform of $C$ with this orbit 
        by taking the lift $(\psi_{\sigma_0}^{\sigma})^* f $ 
        of a defining function $f$ of $C$ to $\C^2_{u,v}$, by simplifying by the 
        greatest monomial in $\sigma^{\vee} \cap M$ which divides it, and then by setting $u =0$. 
        Let therefore 
            \[  f := \sum_{m \in \mathrm{\Supp}(f)} c_m(f) \chi^m   \in \C[[x,y]] \] 
        be a defining function of $C$. As the bases $(f_1, f_2)$ and 
        $(\varphi_1, \varphi_2)$ are dual of each other, we have the relation 
            $ m = (f_1 \cdot m) \varphi_1 + (f_2 \cdot m) \varphi_2$.  
         This implies that $\chi^m = u^{f_1 \cdot m} \: v^{f_2 \cdot m}$. 
         As the ray $\rho = \cone \: f_1$ is orthogonal to the edge $K_{\rho}$ of the Newton polygon 
         $\cN_{L, L'}(C) = \cN(f)$, we know that:
             \[ f_1 \cdot m \geq h_{\rho} \:  \mbox{ for all } \: m \in  \Supp(f), \]
         where $h_{\rho} := \mathrm{trop}^f(f_1)$, 
         with equality if and only if $m \in K_{\rho}$. Therefore, the highest power of $u$ 
         which divides 
               $$(\psi_{\sigma_0}^{\sigma})^* f = 
               \sum_{m \in \Supp(f)} c_m(f) \: u^{f_1 \cdot m} \: v^{f_2 \cdot m} $$ 
         is $u^{h_{\rho}}$, and it is achieved only on the edge $K_{\rho}$ of $\cN(f)$. 
         Moreover, the linear form $m \to f_2 \cdot m$ achieves its minimum 
         (at least) at the vertex  $m_0$ of $\cN(f)$, by the hypothesis that $\sigma$ 
         is contained in one of the two $2$-dimensional cones of $\fan(f)= \fan_{L, L'}(C)$ 
         which are adjacent to $\rho$.  This shows that the maximal 
         monomial in $(u,v)$ which divides $(\psi_{\sigma_0}^{\sigma})^* f$ is 
         $u^{h_{\rho}} v^{f_2 \cdot m_0}$. 
         After simplifying by it and setting $u =0$, one gets the following polynomial equation 
         in the variable $v$, describing the intersection of the strict transform of $C$ with 
         the $v$-axis:
              \begin{equation} \label{eq:restrside}
                    \sum_{m \in K_{\rho} \cap M} c_m(f) \: v^{f_2 \cdot (m - m_0)} =0.
              \end{equation}
          We recognize here the equation obtained from $f_{K_{\rho}} =0$ 
          after the change of variables from $(x,y)$ to $(u,v)$ and the simplification of the 
          highest dividing monomial. This illustrates the importance in 
           our context of the operation of restriction of $f$ to a compact edge of its Newton polygon, 
           introduced in Definition \ref{def:Npolalg}. 
          Using Equation (\ref{eq:pointsonedge}), we see that Equation 
          (\ref{eq:restrside})  becomes:
               \begin{equation} \label{eq:restrsidebis} 
                     \sum_{k=0}^{L_{\rho}} c_{m_0 + k \: \varphi_2}(f) \: v^k =0.
                \end{equation}

           The two extreme coefficients $c_{m_0}(f)$ and $c_{m_1}(f)$ of the previous 
           polynomial equation being non-zero, 
           we see that the strict transform of $C$ does not pass through the origin of $\C^2_{u,v}$ and 
           that it intersects the orbit $O_{\rho}$ in $L_{\rho}= l_{\Z} K_{\rho}$ points, counted with 
           multiplicities. The solutions of Equation (\ref{eq:restrsidebis}) are the $v$-coordinates 
           of the intersection points of the strict transform of $C$ with the orbit $O_{\rho}$. 
           
           By using the same kind of argument for all the cones of the regularization of 
           $\fan_{L, L'}(C)$, we may show also that the strict transform of $C$ misses 
           all the singular points of the boundary divisor of $X_{\fan_{L, L'}(C)}$.         
  \end{proof}

\begin{example}  \label{ex:Newtmod}
     Let us give an example of the objects manipulated in the proof of Proposition 
     \ref{prop:propstrict}. Consider the function $f \in \C[[x,y]]$ of Example \ref{ex:Newtonobj}. Let 
    $\rho$ be the ray of slope $2/1$ of $\fan(f)$. Then $K_{\rho}$ is the side 
    $K_2 := [(3,4), (7, 2)]$ of $\partial \cN(f)$ (see Figure \ref{fig:Newton polygon}). 
    One has $f_1 = e_1 + 2 e_2$. A possible choice of the vector $\varphi_2$ is 
    $\varphi_2 = -2 \epsilon_1 + \epsilon_2$. Therefore $v = x^{-2} y$. 
    Orienting $K_{\rho}$ by this vector $\varphi_2$ one gets $ m_0 = (7, 2)$ and $m_1 = (3,4)$. 
    We saw in Example \ref{ex:Newtonobj} that
     $ f_{K_{\rho}} =  x^7 y^2 + 2 x^5 y^3 + x^3 y^4 = x^3y^2 (x^4 + 2 x^2 y + y^2)$. 
    As $v = x^{-2} y$, Equation (\ref{eq:restrsidebis}) is in this case 
         $1 + 2v + v^2 =0$. 
    We see that its degree is indeed the integral length $L_\rho$ of the side $K_{\rho}$. 
    As it has a double root, the series $f$ is not Newton non-degenerate 
    (see Definition \ref{def:newtnondegen} below). 
    The strict transform of $C$ intersects $O_{\rho}$ at the single point $v= -1$. 
\end{example}

The proof of Proposition \ref{prop:propstrict} yields easily also a proof of the following 
proposition :

\begin{proposition}  \label{prop:newtnondeg}
    Let $(L, L')$ be a cross and $C$ a curve singularity on $S$. Let $f \in \C[[x,y]]$ be 
    a defining function of $C$ relative to any coordinate system $(x,y)$ defining 
    the  chosen cross. Then the following conditions are equivalent: 
       \begin{enumerate}
            \item the curve $C$ is reduced and the Newton modification $\psi_{L, L'}^{C}$ 
                 becomes a toroidal pseudo-resolution of $C$ if one replaces 
                 the boundary $\partial S_{\cF_{L, L'}(C)}$ by the total transform of the divisor 
                 $(\psi^C_{L, L'})^*(C + L + L')$; 
            \item for any ray $\rho$ of the Newton fan $\fan_{L, L'} (C)$ which is  
                orthogonal to a compact edge of $\cN_{L, L'}(C)$, 
                the polynomial equation (\ref{eq:restrsidebis}) has only simple roots; 
            \item  the defining function $f$ of $C$ has the property that  
                  all the restrictions $f_K$ of $f$ to the  compact edges $K$  of the Newton polygon 
                  $ \cN(f) = \cN_{L, L'}(C)$ define smooth curves in the torus $(\C^*)_{x,y}^2$. 
        \end{enumerate}
   \end{proposition}

   The plane curve singularities which satisfy the equivalent conditions of Proposition 
   \ref{prop:newtnondeg} received a special name:
   
   \begin{definition}   \label{def:newtnondegen}
           Let $(L, L')$ be a cross and $C$ a curve singularity on $S$. Let $f \in \C[[x,y]]$ be 
       a defining function of $C$ relative to any coordinate system associated to the 
       chosen cross. The function $f$ is called {\bf Newton non-degenerate} 
       \index{Newton non-degenerate!function} and the curve 
       $C$ is called {\bf Newton non-degenerate relative to the cross $(L, L')$} 
       \index{Newton non-degenerate!curve singularity, relative to a cross} if the 
       equivalent conditions listed in Proposition \ref{prop:newtnondeg} are satisfied.   
   \end{definition}
        
  Usually one speaks about Newton non-degenerate germs of holomorphic functions of several 
  variables. We introduce here the notion of \emph{Newton non-degenerate plane 
  curve singularity relative to a cross} in order to emphasize  the underlying geometric 
  phenomena.

\medskip
Let us come back to Proposition \ref{prop:propstrict}. 
At each point of intersection $o_i$ of the strict transform $C_{L, L'}$ with the exceptional divisor 
of $\psi_{L, L'}^{C}$, one has the following dichotomy:

\noindent
$\bullet$ Either only one branch of $C_{L, L'}$ passes through $o_i$, where it is moreover 
             smooth and transversal to the exceptional divisor. The germ $A_i$ at $o_i$ of the 
             exceptional divisor and this branch form 
             a canonical cross on $S_{\cF_{L, L'}(C)}$. Then, one
             reaches locally a toroidal pseudo-resolution of $C$ in the neighborhood of that point.

\noindent 
$\bullet$ Or one does not have a canonical cross, but only a canonical smooth 
           branch: the germ $A_i$ at $o_i$ of the exceptional divisor 
           $(\psi_{L, L'}^{C})^{-1}(o)$ itself. 
\medskip 

In the second case, one may complete  $A_i$ into a cross $(A_i, L_i)$ by the choice 
of a germ $L_i$ of smooth branch transversal to it. Then one is again in the presence of a germ 
of effective divisor (the germ of the strict transform $C_{L, L'}$ of $C$ by $\psi_{L, L'}^{C}$) 
on a germ of smooth surface endowed with a cross 
(the surface $S_{\cF_{L, L'}(C)}$ endowed with the cross $(A_i, L_i)$). One gets  
again a Newton polygon, a tropical function, a Newton fan and a Newton modification,  
and the previous construction may be iterated. This iterative process may be formulated as the following 
\emph{algorithm of toroidal pseudo-resolution} of the germ $C$:

\begin{algorithm}  \label{alg:tores}  \index{toroidal!pseudo-resolution!algorithm}
    Let $(S,o)$ be a smooth germ of surface, $L$ a smooth branch on $(S,o)$  
    and $C$ a reduced germ of curve on $(S,o)$, which does not contain the 
    branch $L$ in its support. 
  \medskip
  
  \noindent
   {\bf STEP 1.}  If $(L, C)$ is a cross, then STOP. 
   
     \noindent
    {\bf STEP 2.}  Choose a smooth branch $L'$ on $(S,o)$, possibly 
             included in $C$, such that $(L, L')$ is a cross.  
 
   \noindent 
  {\bf STEP 3.}  Let $\cF_{L, L'}(C)$ be the Newton fan of $C$ relative to the cross 
                 $(L, L')$. Consider 
                 the associated Newton modification 
                 $\psi_{L, L'}^C: (S_{\cF_{L, L'}(C)} ,  \partial S_{\cF_{L, L'}(C)}) \to (S, L + L')$
                 and the strict transform $C_{L,L'}$ 
                 of $C$ by $\psi_{L, L'}^C$ (see Definition \ref{def:Npolseries}).

    \noindent
   {\bf STEP 4.} For each point $\tilde{o}$ belonging to 
            $C_{L,L'} \cap \partial S_{\cF_{L, L'}(C)}$, denote:
                     \begin{itemize}
                          \item $L:=$ the germ of $\partial S_{\cF_{L, L'}(C)}$ at $\tilde{o}$; 
                          \item $C:=$ the germ of $C_{L,L'}$ at $\tilde{o}$; 
                          \item $ o := \tilde{o}$;
                          \item $S := S_{\cF_{L, L'}(C)}$. 
                     \end{itemize}
     
        \noindent
      {\bf STEP 5.}  GO TO STEP 1.               
\end{algorithm}

Note that one considers that only the smooth branch $L$ is 
given at the beginning, and that the second branch $L'$ of the cross $(L, L')$
is chosen when one executes STEP 2  for the first time. Note also that the 
algorithm is non-deterministic, as it involves choices of supplementary branches.

A variant of this algorithm, obtained by replacing Step 3 by a Step $3^{reg}$, 
  will be studied in Subsection \ref{ssec:toremb}.  It produces a 
   {\em toroidal embedded resolution}
  of $C$  instead of a {\em pseudo-resolution} (see Definition \ref{def:threeres}). 
  
 Proposition \ref{prop:newtnondeg} means that if $C$ is Newton non-degenerate 
relative to the cross $(L, L')$ chosen at Step 2 of Algorithm \ref{alg:tores}, then  
 this algorithm stops after performing only one Newton modification.
 More generally, a  fundamental property of Algorithm \ref{alg:tores}  is: 

\begin{theorem}   \label{thm:algstop}
   Algorithm \ref{alg:tores}  stops after a finite number of iterations. 
\end{theorem}

\begin{proof} 
  Assume that $A$ is a curve singularity on the smooth germ of surface $(S, o)$, obtained 
  after a finite number of steps of the algorithm, and that  $(L \cdot A)_o = 1$. 
  Then $(L, A)$ is a cross and the algorithm stops. Therefore, in order to show that the algorithm stops, 
  it is enough to show that after a finite number of steps all the local intersection 
  numbers of the strict transform $C_{L, L'}$ of $C$ with the exceptional divisor 
  are equal to $1$. 
  
  By the end statement of Proposition \ref{prop:propstrict}, a sequence 
  of such intersection numbers at {\em infinitely near points of $o$} 
  (see Definition \ref{def:infnear}) which dominate each other is necessarily decreasing: 
    \begin{equation} \label{eq:o-tilde4}
       (C \cdot L)_o   \geq ( C_{1} \cdot E_1 )_{o_1} \geq \cdots \geq  (C_k \cdot E_k)_{o_k} \geq \cdots.
   \end{equation}
   At the $k$-th iteration of the algorithm we are considering the strict transform 
   $C_k$ of $C$ at a point $o_k$, which belongs to the component $E_k$ 
   of the exceptional divisor.
   
   The sequence (\ref{eq:o-tilde4}) being composed of positive integers, it necessarily stabilizes. 
   If the stable value is $1$ for all choices of 
   sequence $o, o_1, o_2 , \dots$, then the algorithm stops after a finite 
   number of steps. 
   
   Let us reason by contradiction, assuming the contrary. Therefore, one may 
   find a sequence as before for which the stable intersection number is 
   $n > 1$. Let us assume without loss of generality, by starting our analysis 
   after the stabilization took place, that: 
       \begin{equation} \label{eq:o-tilde1}
           (C \cdot L)_o  = ( C_{1} \cdot E_1 )_{o_1} = \cdots =  (C_k \cdot E_k)_{o_k} = \cdots
                 = n > 1. 
   \end{equation}
 
 Therefore, for every $k \geq 1$, $(E_k, C_k)$ is not a cross at $o_k$. 
  By STEP 2 of the algorithm,  
  a smooth germ $L_k$ was chosen at $o_k$ such that $(E_k, L_k)$ is a cross at $o_k$.

\medskip 
  Let us reformulate the first equality
     \begin{equation} \label{eq:o-tilde}
         ( C_{1} \cdot E_1 )_{o_1}  = (C \cdot L)_o 
     \end{equation}
   of the sequence (\ref{eq:o-tilde1}) in terms 
  of Newton polygons. By applying again the end statement of Proposition \ref{prop:propstrict}, 
  we see that $ ( C_{1} \cdot E_1 )_{o_1}$ is less or equal to the integral length $l_{\Z} K$ of the 
  compact edge $K$ of $\cN_{L,L'}(C)$ whose orthogonal ray corresponds to the 
  prime exceptional curve $E_1$. One has equality if and only if the strict transform 
  of $C$ intersects $E_1$ at a single point. 
  In turn, the integral length $l_{\Z} K$ is less or equal to the height 
  $(C \cdot L)_o = \: n$ of $\cN_{L,L'}(C)$ (the ordinate of its lowest point on the 
  vertical axis), with equality if and only if $K$ is the only compact edge of 
  $\cN_{L,L'}(C)$ and $K = [(0, n), (m_1 n, 0)]$, with $m_1 \in \N^*$. 
  
 As a consequence, \emph{one has the equality (\ref{eq:o-tilde}) if and only if 
 $\cN_{L,L'}(C)$ has a single compact edge, of the form $[(0, n), (m_1 n, 0)]$, 
 with $m_1 \in \N^*$, 
 and the associated polynomial in one variable has only one root in $\C^*$}. 
 In terms of local coordinates $(x,y)$ on $(S,o)$ defining 
  the cross $(L,L')$ and a defining unitary polynomial $f \in \C [[x]] [y ]$ of the 
  plane curve singularity  $C$ (see Theorem \ref{thm:NewtPuiseux} below), 
 equality holds in (\ref{eq:o-tilde}) if and only if $f$ is of the form 
$f = (y - c_{1}  x^{m_1}) ^n + \cdots$, 
with $c_{1} \in \C^*$, $m_1 \in \N^*$ and where we wrote only the 
restriction $f_K$ of $f$ to the compact edge $K$ of the Newton polygon $\cN_{L, L'} (C)$, 
in the sense of Definition \ref{def:Npolalg}. 
Then, STEP 3 is performed simply by considering the morphism: 
\begin{equation} \label{eq:o-tilde2} 
     \left\{ \begin{array}{lcl}
       x & = &  x_1, \\
        y  & = &  x_1^{m_1} \, (w_1 + c_{1}), 
     \end{array}   \right.
\end{equation}
where $(x_1, w_1)$ are local coordinates at ${o}_1$ and $Z(x_1) = (E_1, o_1)$.
The hypothesis  (\ref{eq:o-tilde1})
implies that $(E_1, C_1)$ is not a cross. Denote by $L_1'$ the 
smooth branch at $o_1$ obtained by applying again STEP 2. Therefore, 
$(E_1, L_1')$ is a cross at $o_1$.   
By the formal version of the implicit function theorem, 
we can choose local coordinates $(x_1, u_1)$ defining the cross $(E_1, L_1')$ 
in such a way that $u_1 = w_1 - \phi_1(x_1)$, for some $\phi_1 \in \C[[ t ]]$ with $\phi_1(0) =0$.

Let us define $y_1 : = y - x^{m_1} (c_1 + \phi_1 (x))$ and denote  $L_1 := Z(y_1)$. 
Notice that the strict transform of $L_1$ by the  modification (\ref{eq:o-tilde2})
is equal to $L_1'$ and that  (\ref{eq:o-tilde2}) can be rewritten 
\begin{equation} \label{eq:o-tilde3} 
     \left\{ \begin{array}{lcl}
       x & = &  x_1, \\
        y_1  & = &  x_1^{m_1} \, u_1
     \end{array}   \right.
\end{equation}
with respect to the local coordinates $(x, y_1)$ and $(x_1, u_1)$. 
Let us denote by $f_1 \in \C[[x_1]] [ u_1] $ the monic polynomial defining $C_1$  
relative to the coordinates $(x_1, u_1)$ (see again Theorem \ref{thm:NewtPuiseux}). 
Reasoning as before, the hypothesis (\ref{eq:o-tilde1}) implies that the polynomial $f_1$ is of the form 
 $f_1 = (u_1 - c_{2}  x_1^{m_2}) ^n + \cdots$, 
where $c_{2} \in \C^*$,  $m_2 \in \N^*$ and the exponents of the monomials $x_1^{i} u_1^{j}$ 
which were omitted verify that $i + m_2 j > m_2 n$ and $0 \leq j < n$. 
Notice that the order of vanishing of $f$ along $E_1$ is equal to $n m_1$. 
 We recover a defining function of $C$ with respect to the coordinates $(x, y_1)$ by 
expressing, using the relation  (\ref{eq:o-tilde3}),  the monomials appearing in the product 
 $x_1^{m_1 n} \cdot f_1 (x_1, u_1)$ as monomials in $(x, y_1)$. 
 We get a defining function of $C$ of the form 
$(y_1 - c_{2}  x_1^{m_1 + m_2}) ^n + \cdots $, 
where the exponents of the  monomials $x_1^i y_1^j$ which are not written above verify that 
$i + (m_1+m_2) j > (m_1 + m_2) n$ and $ 0 \leq j < n$. 
\medskip

By induction  on $k \geq 1$, one may  show similarly that: 
\medskip

\noindent
$\bullet$ The branch $L_k' = Z (u_k) $ is the strict transform  of a smooth branch $L_k =Z (y_k)$ at $S$, where $(x, y_k)$ 
is a local coordinate system defining a cross at $o$ and 
\begin{equation} \label{eq:o-tilde6} 
        y_k  = y_{k-1} - x^{m_1+ \cdots + m_k} ( c_k + \phi_k (x)), 
\end{equation}
where $\phi_k \in \C[[t ]]$ satisfies $\phi_k (0) =0$.

\noindent
$\bullet$ The composition of the maps in the algorithm expresses as 
\begin{equation} \label{eq:o-tilde5} 
     \left\{ \begin{array}{lcl}
       x & = &  x_1, \\
        y_k  & = &  x_1^{m_1+\cdots + m_k} \, u_k,
     \end{array}   \right.
\end{equation}
with respect to the local coordinates $(x_k, u_k)$ at $o_k$ and 
the coordinates $(x, y_k)$ at $o$.

\noindent
$\bullet$
 There exists a defining function of $C$ of the form: 
\[
(y_k - c_{k}  x^{m_1 + \cdots +  m_k}) ^n + \cdots 
\]
where the exponents of monomials  $x^i y_k^j$ which are not written above verify that 
$i + (m_1+\cdots + m_k) j > (m_1 + \cdots + m_k ) n$ and $ 0 \leq j < n$. 

\medskip 
In particular, we have shown that 
the Newton polygon $\cN_{L, L_k} (C)$ has only one 
compact edge  with vertices $(0, n)$ and $(m_1 + \cdots + m_k, 0)$,  
where  $L_k \cdot C  = m_1+ \cdots + m_k$. 
When we look at the polygons  $\cN_{L, L_k} (C)$ as subsets of $\R^2$, we get a nested sequence: 
\begin{equation} \label{eq:o-tilde7} 
      \cN_{L, L'} (C) \supset  \cN_{L, L_1} (C) \supset \cdots \supset  \cN_{L, L_{k-1}} (C) 
           \supset \cN_{L, L_k} (C) .
\end{equation}

By (\ref{eq:o-tilde6}), one has that $y_k = y - \xi_k (x) $ with $\xi_k (x) \in \C[[x]]$. 
One may check, using the shape of relation 
(\ref{eq:o-tilde6}), that the sequence $(\xi_k (x))_{k\geq 1}$ 
converges to a series $\xi_{\infty} (x)$ in the complete ring $ \C[[ x]]$. 
Set  $y_\infty := y - \xi_{\infty} (x)$ and $L_{\infty} := Z( y_\infty )$. 
Then  $(L, L_{\infty})$ is a cross at $o$. 
We deduce that $L_{\infty} \cdot C =  \nu_x f ( x, \xi_{\infty} (x) ) =  + \infty$ 
and by (\ref{eq:o-tilde7}) one gets  the inclusion 
$
\cN_{L, L_\infty} (C)  \subset  \cN_{L, L_k} (C)$, 
for every $k \geq 1$. 
These two facts together imply that the  Newton polygon $\cN_{L, L_\infty} (C)$ 
has only one vertex $(0, n)$. 
Therefore, a local defining series for $C$ is 
of the form $( y_{\infty} )^n$. Since $n >1$,  $C$ would not be a reduced germ, contrary to the hypothesis. 
\end{proof}

\begin{remark}
    The argument used in the proof of Theorem \ref{thm:algstop} coincides basically with 
    one step of the proof of the Newton-Puiseux theorem 
    (see Theorems \ref{thm:NPthmbasic} and \ref{thm:NewtPuiseux}), as presented 
    in Teissier's survey \cite{T 07}.  Unlike the rest of the proof of this theorem, 
    this particular step holds without making any assumption on the characteristic of the base field. 
\end{remark}

Algorithm \ref{alg:tores} involves a finite number of choices, those of the smooth branches 
introduced in order to get crosses each time one executes STEP 2. Let us introduce the following 
notations:

\begin{notation}   \label{def:manycrosses}
    Assume that one executes Algorithm \ref{alg:tores}  
    on $(S,o)$, starting from the curve singularity $C$ and the smooth branch $L$. Then:
          \begin{enumerate}
               \item $\{\boxed{o_i}, \:  i \in \boxed{I}\}$ is the set of points at which one applies 
                    STEP 1 or STEP 2. 
                   One assumes that $\{1\} \subseteq I$ and that $o_1 := o$. 
               \item $\{\boxed{(A_i, B_i)}, \:  i \in I\}$ is the set of crosses considered each time one applies 
                   STEP 1 or STEP 2. Therefore $A_1 = L$ and 
                   for $i \in I \:  \setminus \: \{1\}$, the branch $A_i$ is included in the exceptional divisor 
                   of the Newton modification performed at the previous iteration.
               \item $\boxed{J} \subseteq I$ consists of those $j \in I$ 
                   for which one performs STEP 2 at $o_j$. Denote by $\boxed{L_j}$ the projection 
                   on $S$ of the branch $B_j$, for every 
                   $j \in J$. Therefore, $B_i$ is a strict transform of a branch of $C$ whenever 
                   $i \in I \:  \setminus \:  J$ and $B_j$ is the strict transform of $L_j$ whenever $j \in J$.
               \item $\boxed{S^{(1)}} := S$. For $k \geq 1$, the surface $S^{(k+1)}$ 
                     is obtained from $\boxed{S^{(k)}}$ by performing 
                     simultaneously the Newton modification of STEP 3 at all the points $o_j$ 
                     of $S^{(k)}$ at which one executes STEP 2. At such a point, denote by 
                     $\boxed{\fan_{A_j, B_j }(C)}$ the corresponding fan. 
                     It is the Newton fan of the germ of strict transform 
                     of $C$ at $o_j$, relative to the cross $(A_j, B_j)$. 
               \item The previous simultaneous Newton modification 
                  is denoted $\boxed{\pi^{(k)}}: S^{(k+1)} \to S^{(k)}$. We call it the {\bf $k$-th level 
                  of Newton modifications}. 
               \item The toroidal boundary $\boxed{\partial S^{(k)}}$ is by definition the total transform 
                    on $S^{(k)}$ of all the crosses which appeared in the algorithm 
                    until performing STEP 2 
                    at all the points of $S^{(k)}$. In particular, $\partial S= L + L_1$. 
                    Each morphism $\pi^{(k)}: (S^{(k+1)}, \partial  S^{(k+1)}) \to (S^{(k)}, \partial S^{(k)})$ 
                    belongs to the toroidal category, as  
                    $(\pi^{(k)})^{-1}( \partial S^{(k)}) \subseteq \partial  S^{(k+1)}$.
               \item $\boxed{\pi} := \pi^{(1)} \circ \cdots \circ \pi^{(h)}$, where $\boxed{h}$ is the number of 
                   modifications $\pi^{(k)}$ produced by the algorithm. We denote by 
                   $\boxed{\Sigma}$ the source 
                  of $\pi$. Therefore, $\pi : \Sigma \to S$ is a modification of the initial germ $S$. 
                \item $\boxed{\partial \Sigma}$ denotes $\partial  S^{(h)}$. 
                   It is the underlying reduced divisor of the total transform 
                    $\pi^*(\hat{C}_{\pi})$ of the completion $\hat{C}_{\pi} = C + \sum_{j \in J} L_j$, 
                    in the sense of Definition \ref{def:threeres}. 
         \end{enumerate}
\end{notation}

There are a lot of notations here! The only way to get used to them, to 
understand how those objects are related, and why they are important, is to look at examples. 
That is why we made a detailed one below (see Example \ref{ex:toroidres}). In fact, all the 
works which deal in a detailed way with processes of resolution of singularities introduce analogously 
plenty of notations (see for instance Zariski \cite{Z 39}, Zariski \cite{Z 71}, 
Lejeune-Jalabert \cite{LJ 95}, A'Campo and Oka \cite{AO 96}, Casas \cite{CA 00}, Wall \cite{W 04} 
or Greuel, Lossen and Shustin \cite{GLS 07}). 
This is one of the main advantages we see for the notion of \emph{lotus} 
attached below to such a resolution process (see Definition \ref{def:lotustoroid}): 
it allows to get a simultaneous global view of the previous objects.

We can state in the following way the output of Algorithm \ref{alg:tores} in terms of Definition \ref{def:threeres}:

\begin{proposition}  \label{prop:toroidres}
      The morphism $\pi: (\Sigma, \partial \Sigma) \to (S, L + L')$ is a toroidal pseudo-resolution of $C$. 
\end{proposition}

\begin{remark} \label{rem-nonreduced}
We formulated Algorithm \ref{alg:tores} 
only for \emph{reduced} curve singularities $C$. It extends readily 
to an algorithm applicable to any $C$, by agreeing that one runs it on the reduction of $C$. 
One may agree also to define the fan tree of an arbitrary curve singularity 
$C$ as the fan tree of its reduction (see Definition \ref{def:fantreetr}), 
each leaf being decorated with the multiplicity of the corresponding branch inside the divisor $C$. 
Similar conventions may be chosen in order to associate a lotus to an arbitrary 
curve singularity $C$. As we do not use those more general notions in this text, 
we will not introduce them formally. 
\end{remark}

Let us give now an example of application of  Algorithm \ref{alg:tores}. 
Instead of starting from a particular equation, we will assume that 
the algorithm involves three levels of toroidal modifications
with prescribed Newton fans and we will 
describe from them the toroidal boundary of the final surface. 
We will see in Example \ref{ex:fromFTtoEW} below how to write concrete 
    equations for branches $C_i$ and $L_j$ appearing in a toroidal resolution 
    process structured as in Example \ref{ex:toroidres}. The idea is to associate 
    to the Newton polygons of the process a \emph{fan tree} (see Definition 
    \ref{def:fantreetr}), which may be 
    transformed into an \emph{Eggers-Wall tree} (see Definition \ref{prop:slopedetindex}), 
    which in turn allows to 
    write Newton-Puiseux series defining the branches $C_i$ and $L_j$. One may take 
    as their defining functions in $\C[[x,y]]$   
    the minimal polynomials of those Newton-Puiseux series.

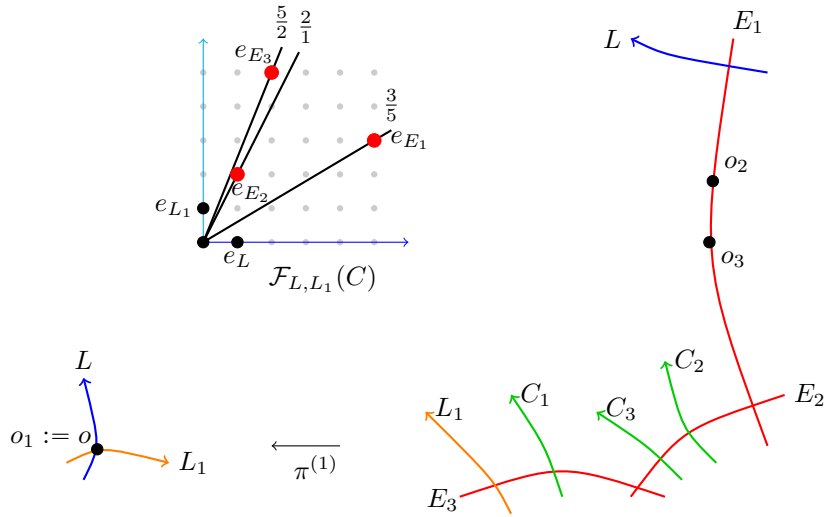
\begin{figure}[h!]
\begin{center}
\begin{tikzpicture}[scale=0.45]
\draw[->][thick, color=blue](-0.5,-1) .. controls (0,0) ..(-0.5,2);
\node [above] at (-0.5,2) {$L$}; 
\draw[->][thick, color=orange](-1,-0.5) .. controls (0,0) ..(2,-0.5);
\node [right] at (2,-0.5) {$L_1$}; 
\node[draw,circle,inner sep=1.5pt,fill=black] at (-0.1,-0.1){};
\node [left] at (0,0.2) {$o_1:=o$}; 
\draw[<-](5,0)--(7,0);
\node [below] at (6.3,0) {$\pi^{(1)}$}; 

 \begin{scope}[shift={(3,6)},scale=1]
\foreach \x in {0,1,...,5}{
\foreach \y in {0,1,...,5}{
       \node[draw,circle,inner sep=0.7pt,fill,color=gray!40] at (1*\x,1*\y) {}; }
   }
\draw [->, color=cyan](0,0) -- (0,6);
\draw [->, color=blue](0,0) -- (6,0);
\node[draw,circle, inner sep=1.5pt,color=black, fill=black] at (0,0){};
\node[draw,circle, inner sep=1.5pt,color=black, fill=black] at (1,0){};
\node[draw,circle, inner sep=1.5pt,color=black, fill=black] at (0,1){};
\node [left] at (0,1) {$e_{L_{1}}$};
\node [below] at (1,0) {$e_L$};
\node [left,above] at (1.5,5) {$e_{E_{3}}$};
\draw [-, thick](0,0) -- (2.3,5.75);
\node[draw,circle, inner sep=1.8pt,color=red, fill=red] at (2,5){};
\node [above] at (2.3,5.75) {$\frac{5}{2}$}; 
\node [below] at (1.4,2) {$e_{E_{2}}$};
\draw [-, thick](0,0) -- (2.8,5.6);
\node[draw,circle, inner sep=1.8pt,color=red, fill=red] at (1,2){};
\node [above] at (3,5.6) {$\frac{2}{1}$};
\node [right] at (5.2,3) {$e_{E_{1}}$};
\draw [-, thick](0,0) -- (5.5,3.3);
\node[draw,circle, inner sep=1.8pt,color=red, fill=red] at (5,3){};
\node [above] at (5.5,3.3) {$\frac{3}{5}$};
\node [below] at (3.5,-0.4) {$\fan_{L, L_1}(C)$};
 \end{scope}

 \begin{scope}[shift={(13,4)}]
\draw[thick, color=red](6.5,-4) .. controls (4.5,1.5) ..(5.5,8);  
\node [right] at (5.2,8.5) {$E_1$}; 
\draw[thick, color=red](7,-2.5) .. controls (4,-3.5) ..(2.5,-5.5);
\node [right] at (7,-2.5) {$E_2$};
\draw [->][thick, color=black!20!green](5,-4.5) .. controls (4,-3.5) ..(3.5,-1.5); 
\node [right] at (3.5,-1.5) {$C_2$}; 
\draw [->][thick, color=black!20!green](4,-5) .. controls (3,-4) ..(1.5,-3);
\node [right] at (1.5,-2.9) {$C_3$}; 
\draw[thick, color=red](3.5,-5.5) .. controls (0.5,-4.5) ..(-2.5,-5.5); 
\node [left, below] at (-3,-5) {$E_3$}; 
\draw [->][thick, color=black!20!green](0.5,-5.5) .. controls (0,-4) ..(-1,-2.5);
\node [right] at (-1,-2.5) {$C_1$}; 
\draw [->][thick, color=orange](-1,-6) .. controls (-1.5,-5) ..(-3.5,-3);
\node [right] at (-3.5,-2.9) {$L_1$};  
\node[draw,circle,inner sep=1.5pt,fill=black] at (4.8,2){};
\node [right, below] at (5.5,2) {$o_3$};  
\node[draw,circle,inner sep=1.5pt,fill=black] at (4.9,3.8){}; 
\node [right, above] at (5.6,3.8) {$o_2$}; 
\draw [->][thick, color=blue](6.5,7) .. controls (3.5,7.5) ..(2.5,8);
\node [left] at (2.5,8) {$L$}; 
 \end{scope}
\end{tikzpicture}
\end{center}
 \caption{First level of Newton modifications in Example \ref{ex:toroidres}}  
 \label{fig:example-first level} 
     \end{figure}

     
\begin{figure}[h!]
\begin{center}
\begin{tikzpicture}[scale=0.45]
   \draw[thick, color=black](6.5,-4) .. controls (4.5,1.5) ..(5.5,8);  
\node [right] at (5.2,8.5) {$E_1$}; 
\draw[thick, color=black](7,-2.5) .. controls (4,-3.5) ..(2.5,-5.5);
\node [right] at (7,-2.5) {$E_2$};
\draw [->][thick, color=black!20!green](5,-4.5) .. controls (4,-3.5) ..(3.5,-1.5); 
\node [right] at (3.5,-1.5) {$C_2$}; 
\draw [->][thick, color=black!20!green](4,-5) .. controls (3,-4) ..(1.5,-3);
\node [right] at (1.5,-2.9) {$C_3$}; 
\draw[thick, color=black](3.5,-5.5) .. controls (0.5,-4.5) ..(-2.5,-5.5); 
\node [left, below] at (-3,-5) {$E_3$}; 
\draw [->][thick, color=black!20!green](0.5,-5.5) .. controls (0,-4) ..(-1,-2.5);
\node [right] at (-1,-2.5) {$C_1$}; 
\draw [->][thick, color=orange](-1,-6) .. controls (-1.5,-5) ..(-3.5,-3);
\node [right] at (-3.5,-2.9) {$L_1$};  
\draw [->][thick, color=orange](6.5,2) .. controls (4.5,2) ..(1,1); 
\node [left] at (1,1) {$L_3$}; 
\draw [->][thick, color=orange](6.5,3.8) .. controls (4.6,3.8) ..(0.5,2.8); 
\node [left] at (0.5,2.8) {$L_2$}; 
\node[draw,circle,inner sep=1.5pt,fill=black] at (4.8,1.9){};
\node [right, below] at (5.5,2) {$o_3$};  
\node[draw,circle,inner sep=1.5pt,fill=black] at (4.9,3.8){}; 
\node [right, above] at (5.6,3.8) {$o_2$}; 
\draw [->][thick, color=blue](6.5,7) .. controls (3.5,7.5) ..(2.5,8);
\node [left] at (2.5,8) {$L$}; 
\draw[<-](9,0)--(11,0);
\node [below] at (10.3,0) {$\pi^{(2)}$}; 
 \begin{scope}[shift={(,12)},scale=1]
\foreach \x in {0,1,...,5}{
\foreach \y in {0,1,...,5}{
       \node[draw,circle,inner sep=0.7pt,fill, color=gray!40] at (1*\x,1*\y) {}; }
   }
\draw [->, color=cyan](0,0) -- (0,6);
\draw [->, color=blue](0,0) -- (6,0);
\node[draw,circle, inner sep=1.5pt,color=black, fill=black] at (0,0){};
\node[draw,circle, inner sep=1.5pt,color=black, fill=black] at (1,0){};
\node[draw,circle, inner sep=1.5pt,color=black, fill=black] at (0,1){};
\node [left] at (0,1) {$e_{L_{2}}$};
\node [below] at (1,0) {$e_{E_{1}}$};
\node [left,above] at (3.8,3) {$e_{E_{5}}$};
\draw [-, thick](0,0) -- (5.2,3.9);
\node[draw,circle, inner sep=1.8pt,color=red, fill=red] at (4,3){};
\node [right,below] at (3.2,2) {$e_{E_{4}}$};
\draw [-,thick](0,0) -- (5.4,3.6);
\node[draw,circle, inner sep=1.8pt,color=red, fill=red] at (3,2){};
\node [above] at (5.4,3.6) {$\frac{3}{4}$};
\node [right] at (5.4,3.6) {$\frac{2}{3}$};
\node [below] at (3.5,-0.4) {$\fan_{E_{1}, L_2}(C)$};
 \end{scope}
 \begin{scope}[shift={(13,12)},scale=1]
\foreach \x in {0,1,...,5}{
\foreach \y in {0,1,...,5}{
       \node[draw,circle,inner sep=0.7pt,fill, color=gray!40] at (1*\x,1*\y) {}; }
   }
\draw [->, color=cyan](0,0) -- (0,6);
\draw [->, color=blue](0,0) -- (6,0);
\node[draw,circle, inner sep=1.5pt,color=black, fill=black] at (0,0){};
\node[draw,circle, inner sep=1.5pt,color=black, fill=black] at (1,0){};
\node[draw,circle, inner sep=1.5pt,color=black, fill=black] at (0,1){};
\node [left] at (0,1) {$e_{L_{3}}$};
\node [below] at (1,0) {$e_{E_{1}}$};
\node [left,above] at (0.6,3) {$e_{E_{7}}$};
\draw [-, thick](0,0) -- (1.8,5.4);
\node[draw,circle, inner sep=1.8pt,color=red, fill=red] at (1,3){};
\node [right,below] at (3.4,5) {$e_{E_{6}}$};
\draw [-, thick](0,0) -- (3.3,5.5);
\node[draw,circle, inner sep=1.8pt,color=red, fill=red] at (3,5){};
\node [above] at (1.8,5.4) {$\frac{3}{1}$};
\node [above] at (3.3,5.5) {$\frac{5}{3}$};
\node [below] at (3.5,-0.4) {$\fan_{E_{1}, L_3}(C)$};
 \end{scope}
 \begin{scope} [shift={(16,0)},scale=1]
 \draw[thick, color=black](6.5,-4) .. controls (4.5,1.5) ..(5.5,8); 
\node [right] at (5.5,8) {$E_1$}; 
\draw[thick, color=black](6.5,-2.5) .. controls (4,-3.5) ..(2.5,-5.5);
\node [right] at (6.5,-2.5) {$E_2$};
\draw [->][thick, color=black!20!green](5,-4.5) .. controls (4,-3.5) ..(3.5,-1.5); 
\node [right] at (3.5,-1.5) {$C_2$}; 
\draw [->][thick, color=black!20!green](4,-5) .. controls (3,-4) ..(1.5,-3);
\node [right] at (1.5,-2.9) {$C_3$}; 
\draw[thick, color=black](3.5,-5.5) .. controls (0.5,-4.5) ..(-2.5,-5.5);  
\node [left, below] at (-2.5,-5.5) {$E_3$}; 
\draw [->][thick, color=black!20!green](0.5,-5.5) .. controls (0,-4) ..(-1,-2.5);
\node [right] at (-1,-2.5) {$C_1$}; 
\draw [->][thick, color=orange](-1,-6) .. controls (-1.5,-5) ..(-3.5,-3);
\node [right] at (-3.5,-2.9) {$L_1$}; 
\draw[thick, color=red](5.5,3) .. controls (2.5,2) ..(0.5,2.5); 
\node [right] at (5.5,3) {$E_6$}; 
\node[draw,circle,inner sep=1.5pt,fill=black] at (3,2.2){};
\node [above] at (3.2,2.3) {$o_4$}; 
\draw[thick, color=red](1.5,2.5) .. controls (0.5,2) ..(-2.5,1.5); 
\node [left] at (-2.5,1.5) {$E_7$}; 
\draw [->][thick, color=black!20!green](-1,0.5) .. controls (-1,1) ..(-1.5,2.5);
\node [right] at (-1.5,2.5) {$C_6$}; 
\draw [->][thick, color=orange](-2,2) .. controls (-2.5,0.5) ..(-3,-0.5);
\node [left] at (-3,-0.5) {$L_3$}; 
\draw [->][thick, color=blue](6.5,7) .. controls (3.5,7.5) ..(2.5,8);
\node [left] at (2.5,8) {$L$}; 
\draw[thick, color=red](6.5,5) .. controls (3.5,5.5) ..(0.5,7); 
\node [right] at (6.5,5) {$E_4$}; 
\draw[thick, color=red](2,7) .. controls (-1,6) ..(-3,6); 
\node [left] at (-3,6) {$E_5$}; 
\draw [->][thick, color=black!20!green](2.5,4.5) .. controls (3,5) ..(3.5,6.5);
\node [right] at (3.5,6.5) {$C_5$}; 
\draw [->][thick, color=black!20!green](-0.5,5) .. controls (-1,6) ..(-1.5,8);
\node [right] at (-1.5,8) {$C_4$}; 
\draw [->][thick, color=orange](-2,6.5) .. controls (-3,5) ..(-4.5,4);
\node [left] at (-4.5,4) {$L_2$}; 
 \end{scope}
\end{tikzpicture}
\end{center}
 \caption{Second level of Newton modifications in Example \ref{ex:toroidres}}  
 \label{fig:example-second level} 
     \end{figure}
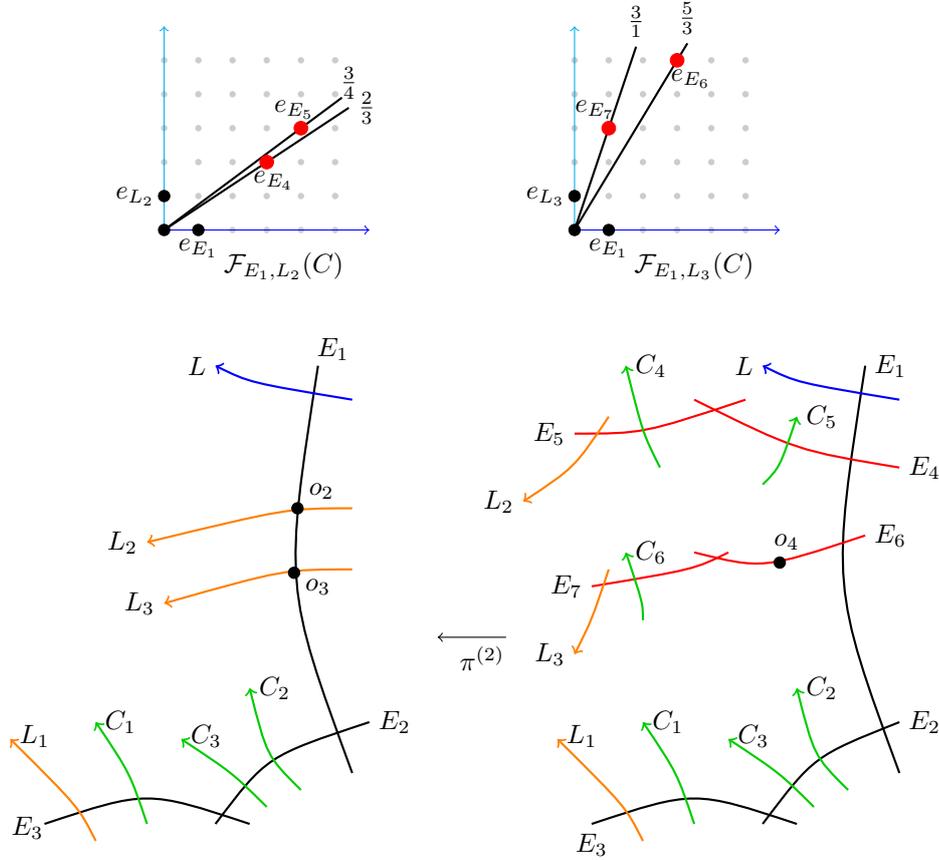

\begin{example}  \label{ex:toroidres} 
     We will use Notations \ref{def:manycrosses}, but we will denote 
    in the same way a branch and its various strict transforms by the modifications 
    produced by the algorithm. In particular, we will write $L_j$ instead of $B_j$, for any 
    $j \in J$. 
    
 Assume that, relative to the first cross $(L, L_1)$, which lives on 
    $S^{(1)}= S$, the Newton fan $\fan_{L, L_1}(C)$ 
    of the curve singularity $C$ is as represented on the top of Figure \ref{fig:example-first level}.  
    Therefore it is the same fan $\fan\left(3/5, 2/1, 5/2 \right)$ 
    as in Figure \ref{fig:examfan}. The associated Newton modification $\pi^{(1)}$ is  
    represented on the bottom of Figure \ref{fig:example-first level}. We have drawn 
    schematically the two boundaries $\partial S^{(1)} = L + L_1$ and 
    $\partial S^{(2)} = L+ E_1 + E_2 + E_3 + L_1 + C_1 + C_2 + C_3$. 
    The components $E_i$ of the exceptional divisor of $\pi^{(1)}$ correspond to the rays 
    $\cone \: e_{E_i}$ of the Newton fan $\fan_{L, L_1}(C)$. We assume that 
    there are three intersection points of the strict transform $C_{L, L_1}$ of $C$ 
    by $\pi^{(1)}$ at which the algorithm stops at STEP 1. The corresponding components 
    of $C$ are denoted $C_1, C_2, C_3$. 
    By contrast, at the points $o_2$ and $o_3$, one has to apply STEP 2 
    of Algorithm \ref{alg:tores} (which implies that $\{2, 3\} \subseteq J$). 
    
One introduces two new smooth branches $L_2$ and 
    $L_3$ passing through $o_2$ and $o_3$ respectively, transversally to the exceptional 
    divisor $E_1 + E_2 + E_3$ of $\pi^{(1)}$. Both points $o_2$ and $o_3$ belong 
    to the component $E_1$. Now one may get the second level of Newton modifications, 
    by looking at the Newton fans $\fan_{E_1, L_2}(C)$ and $\fan_{E_1, L_3}(C)$ 
    (note that we have written $(E_1, L_j)$ instead of $(A_j, L_j)$, because for 
    $j \in \{2, 3\}$, $A_j$ is the germ of $E_1$ at $o_j$). We assume that those Newton fans 
    are as represented   
    on the top of Figure \ref{fig:example-second level}. The corresponding composition $\pi^{(2)}$
    of Newton modifications at $o_2$ and $o_3$ is represented on the bottom 
    of the figure, through a schematic drawing of 
    $\partial S^{(2)} + L_2 + L_3$ and of $\partial S^{(3)}= \partial \Sigma$. 
    We assume that the process stops 
    at STEP 1 at three more points, through which pass the strict transforms of the branches 
    $C_4, C_5, C_6$ of $C$ (see the right bottom part of Figure \ref{fig:example-second level}). 
    There remains one point $o_4$, lying on the component $E_6$ of the exceptional 
    divisor $E_4 + E_5 + E_6 + E_7$ of $\pi^{(2)}$, at which one has to perform STEP 2.

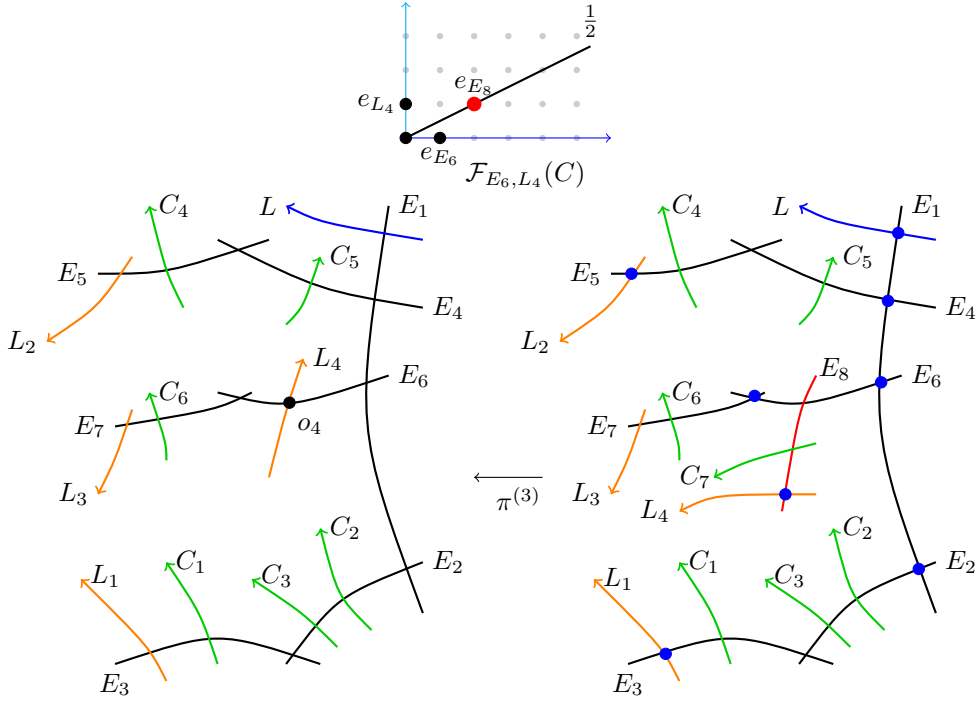
\begin{figure}[h!]
\begin{center}
\begin{tikzpicture}[scale=0.45]
\draw[thick, color=black](6.5,-4) .. controls (4.5,1.5) ..(5.5,8); 
\node [right] at (5.5,8) {$E_1$}; 
\draw[thick, color=black](6.5,-2.5) .. controls (4,-3.5) ..(2.5,-5.5);
\node [right] at (6.5,-2.5) {$E_2$};
\draw [->][thick, color=black!20!green](5,-4.5) .. controls (4,-3.5) ..(3.5,-1.5); 
\node [right] at (3.5,-1.5) {$C_2$}; 
\draw [->][thick, color=black!20!green](4,-5) .. controls (3,-4) ..(1.5,-3);
\node [right] at (1.5,-2.9) {$C_3$}; 
\draw[thick, color=black](3.5,-5.5) .. controls (0.5,-4.5) ..(-2.5,-5.5);  
\node [left, below] at (-2.5,-5.5) {$E_3$}; 
\draw [->][thick, color=black!20!green](0.5,-5.5) .. controls (0,-4) ..(-1,-2.5);
\node [right] at (-1,-2.5) {$C_1$}; 
\draw [->][thick, color=orange](-1,-6) .. controls (-1.5,-5) ..(-3.5,-3);
\node [right] at (-3.5,-2.9) {$L_1$}; 
\draw[thick, color=black](5.5,3) .. controls (2.5,2) ..(0.5,2.5); 
\node [right] at (5.5,3) {$E_6$}; 
\node [right] at (3,3.5) {$L_{4}$}; 
\draw[->][thick, color=orange](2,0) .. controls (2.5,2) ..(3,3.5); 
\node[draw,circle,inner sep=1.5pt,fill=black] at (2.6,2.2){};
\node [above] at (3.2,1) {$o_4$}; 
\draw[thick, color=black](1.5,2.5) .. controls (0.5,2) ..(-2.5,1.5); 
\node [left] at (-2.5,1.5) {$E_7$}; 
\draw [->][thick, color=black!20!green](-1,0.5) .. controls (-1,1) ..(-1.5,2.5);
\node [right] at (-1.5,2.5) {$C_6$}; 
\draw [->][thick, color=orange](-2,2) .. controls (-2.5,0.5) ..(-3,-0.5);
\node [left] at (-3,-0.5) {$L_3$}; 
\draw [->][thick, color=blue](6.5,7) .. controls (3.5,7.5) ..(2.5,8);
\node [left] at (2.5,8) {$L$}; 
\draw[thick, color=black](6.5,5) .. controls (3.5,5.5) ..(0.5,7); 
\node [right] at (6.5,5) {$E_4$}; 
\draw[thick, color=black](2,7) .. controls (-1,6) ..(-3,6); 
\node [left] at (-3,6) {$E_5$}; 
\draw [->][thick, color=black!20!green](2.5,4.5) .. controls (3,5) ..(3.5,6.5);
\node [right] at (3.5,6.5) {$C_5$}; 
\draw [->][thick, color=black!20!green](-0.5,5) .. controls (-1,6) ..(-1.5,8);
\node [right] at (-1.5,8) {$C_4$}; 
\draw [->][thick, color=orange](-2,6.5) .. controls (-3,5) ..(-4.5,4);
\node [left] at (-4.5,4) {$L_2$};  

\draw[<-](8,0)--(10,0);
\node [below] at (9.3,0) {$\pi^{(3)}$}; 
 \begin{scope}[shift={(6,10)},scale=1]
\foreach \x in {0,1,...,5}{
\foreach \y in {0,1,...,3}{
       \node[draw,circle,inner sep=0.7pt,fill,gray!40] at (1*\x,1*\y) {}; }
   }
\draw [->, color=cyan](0,0) -- (0,4);
\draw [->, color=blue](0,0) -- (6,0);
\node[draw,circle, inner sep=1.5pt,color=black, fill=black] at (0,0){};
\node[draw,circle, inner sep=1.5pt,color=black, fill=black] at (1,0){};
\node[draw,circle, inner sep=1.5pt,color=black, fill=black] at (0,1){};
\node [left] at (0,1) {$e_{L_{4}}$};
\node [below] at (1,0) {$e_{E_{6}}$};
\node [left,above] at (2,1) {$e_{E_{8}}$};
\draw [-, thick](0,0) -- (5.4,2.7);
\node[draw,circle, inner sep=1.8pt,color=red, fill=red] at (2,1){};
\node [above] at (5.4,2.7) {$\frac{1}{2}$}; 
\node [below] at (3.5,-0.4) {$\fan_{E_{6}, L_4}(C)$};
 \end{scope}
 \begin{scope}[shift={(15,0)}][scale=0.5]
\draw[thick, color=black](6.5,-4) .. controls (4.5,1.5) ..(5.5,8); 
\node [right] at (5.5,8) {$E_1$}; 
\draw[thick, color=black](6.5,-2.5) .. controls (4,-3.5) ..(2.5,-5.5);
\node [right] at (6.5,-2.5) {$E_2$};
\draw [->][thick, color=black!20!green](5,-4.5) .. controls (4,-3.5) ..(3.5,-1.5); 
\node [right] at (3.5,-1.5) {$C_2$}; 
\draw [->][thick, color=black!20!green](4,-5) .. controls (3,-4) ..(1.5,-3);
\node [right] at (1.5,-2.9) {$C_3$}; 
\draw[thick, color=black](3.5,-5.5) .. controls (0.5,-4.5) ..(-2.5,-5.5);  
\node [left, below] at (-2.5,-5.5) {$E_3$}; 
\draw [->][thick, color=black!20!green](0.5,-5.5) .. controls (0,-4) ..(-1,-2.5);
\node [right] at (-1,-2.5) {$C_1$}; 
\draw [->][thick, color=orange](-1,-6) .. controls (-1.5,-5) ..(-3.5,-3);
\node [right] at (-3.5,-3) {$L_{1}$};
\draw[thick, color=black](5.5,3) .. controls (2.5,2) ..(0.5,2.5); 
\node [right] at (5.5,3) {$E_6$}; 
\draw[thick, color=red](2,-1) .. controls (2.5,2) ..(3,3); 
\node [right] at (2.8,3.2) {$E_8$}; 
\draw [->][thick, color=black!20!green](3,1) .. controls (1,0.5) ..(0,0);
\node [left] at (0.2,0) {$C_7$}; 
\draw [->][thick, color=orange](3,-0.5) .. controls (0,-0.5) ..(-1,-1);
\node [left] at (-1,-1) {$L_4$}; 
\draw[thick, color=black](1.5,2.5) .. controls (0.5,2) ..(-2.5,1.5); 
\node [left] at (-2.5,1.5) {$E_7$}; 
\draw [->][thick, color=black!20!green](-1,0.5) .. controls (-1,1) ..(-1.5,2.5);
\node [right] at (-1.5,2.5) {$C_6$}; 
\draw [->][thick, color=orange](-2,2) .. controls (-2.5,0.5) ..(-3,-0.5);
\node [left] at (-3,-0.5) {$L_3$}; 
\draw [->][thick, color=blue](6.5,7) .. controls (3.5,7.5) ..(2.5,8);
\node [left] at (2.5,8) {$L$}; 
\draw[thick, color=black](6.5,5) .. controls (3.5,5.5) ..(0.5,7);
\node [right] at (6.5,5) {$E_4$};  
\draw[thick, color=black](2,7) .. controls (-1,6) ..(-3,6); 
\node [left] at (-3,6) {$E_5$}; 
\draw [->][thick, color=black!20!green](2.5,4.5) .. controls (3,5) ..(3.5,6.5);
\node [right] at (3.5,6.5) {$C_5$}; 
\draw [->][thick, color=black!20!green](-0.5,5) .. controls (-1,6) ..(-1.5,8);
\node [right] at (-1.5,8) {$C_4$}; 
\draw [->][thick, color=orange](-2,6.5) .. controls (-3,5) ..(-4.5,4);
\node [left] at (-4.5,4) {$L_2$}; 

\node[draw,circle,inner sep=1.5pt,fill=blue, color=blue] at (5.4,7.2){};
\node[draw,circle,inner sep=1.5pt,fill=blue, color=blue] at (6,-2.7){};
\node[draw,circle,inner sep=1.5pt,fill=blue, color=blue] at (-1.4,-5.2){};
\node[draw,circle,inner sep=1.5pt,fill=blue, color=blue] at (5.1,5.2){};
\node[draw,circle,inner sep=1.5pt,fill=blue, color=blue] at (-2.4,6){};
\node[draw,circle,inner sep=1.5pt,fill=blue, color=blue] at (4.9,2.8){};
\node[draw,circle,inner sep=1.5pt,fill=blue, color=blue] at (1.2,2.4){};
\node[draw,circle,inner sep=1.5pt,fill=blue, color=blue] at (2.1,-0.5){};
 \end{scope}
\end{tikzpicture}
\end{center}
 \caption{Third level of Newton modifications in Example \ref{ex:toroidres}}  
 \label{fig:example-third level} 
     \end{figure}

One completes then the germ $A_4$ of $E_6$ at $o_4$ into a cross $(E_6, L_4)$, 
    represented on the left bottom part of Figure \ref{fig:example-third level}. 
    We assume now that  the 
    Newton fan $\fan_{E_6, L_4}(C)$ is as drawn on the top of the figure. It has 
    only one ray distinct from the edges of the cone $\cone \langle e_{E_6}, e_{L_4} \rangle$. 
    Therefore, the corresponding Newton 
    modification, which alone gives the third level of Newton modifications $\pi^{(3)}$, 
    introduces only one more irreducible component of exceptional divisor, denoted 
    $E_8$. It is cut by the strict transform of one more branch  of $C$, denoted 
    $C_7$ and represented on the bottom right part of Figure \ref{fig:example-third level}. 
    The whole curve schematically represented here is the boundary $\partial \Sigma$. On the 
    bottom left is represented the divisor $\partial S^{(3)} + L_4$. 
    The toroidal pseudo-resolution of $C$ produced by the algorithm is the composition 
    $\pi^{(1)} \circ \pi^{(2)} \circ \pi^{(3)} : (\Sigma, \partial \Sigma) \to (S, L + L_1)$. 
    The singular points of the total surface $\Sigma := S^{(3)}$ correspond bijectively 
    to the non-regular $2$-dimensional cones of the Newton fans $\fan_{L, L_1}(C)$, 
    $\fan_{E_1, L_2}(C)$, $\fan_{E_1, L_3}(C)$ and $\fan_{E_6, L_4}(C)$ produced 
    by the algorithm.  
    We represented them as small blue discs on the bottom right of Figure 
    \ref{fig:example-third level}. 
\end{example}

\subsection{From toroidal pseudo-resolutions to embedded resolutions}
\label{ssec:toremb}
$\:$
\medskip

In this subsection, we explain how to get an embedded resolution 
of $C \hookrightarrow S$ from one of the toroidal pseudo-resolutions  
produced by Algorithm \ref{alg:tores}.
Recall first from Definition \ref{def:threeres} the difference between toroidal 
pseudo-resolutions and embedded ones: in the first ones the source of the modification may 
have toric singularities, while in the second ones the source is required to be smooth. 

\medskip
Consider a toroidal pseudo-resolution morphism 
$\pi: (\Sigma, \partial \Sigma) \to (S, L + L')$ of $C$ 
produced by Algorithm \ref{alg:tores} (we speak about ``a morphism'' instead of ``the morphism'', 
because of the choices of smooth branches $(L_j)_{j \in J}$ involved in its construction, see Definition 
\ref{def:manycrosses}). The surface $\Sigma$ has a finite number of singular 
points. As explained in Example \ref{ex:dualisom}, they correspond to the $2$-dimensional 
non-regular cones of the Newton fans  which appeared during the process. 
Proposition \ref{prop:minrestor} shows that one may 
resolve minimally those singular points by taking the regularization of each such cone. 
In fact, those regularizations glue into the regularizations of the Newton fans. 

A way to regularize all the Newton fans produced by Algorithm \ref{alg:tores} 
is to run  a variant of it, obtained 
by always replacing STEP 3 with the following ``regularized'' version of it: 

\medskip
\noindent
\emph{{\bf STEP $3^{reg}$.}  Let $\cF_{L, L'}^{reg}(C)$ be the regularized 
                 Newton fan of $C$ relative to the cross $(L, L')$ and let 
                 $\psi_{L, L'}^{C, reg}: (S_{\cF_{L, L'}^{reg}(C)} ,  
                 \partial S_{\cF_{L, L'}^{reg}(C)}) \to (S, L + L')$ be 
                 the associated Newton modification. Consider the strict transform $C_{L,L'}$ 
                 of $C$ by $\psi_{L, L'}^{C,reg}$.} \index{toroidal!embedded resolution!algorithm}
 \medskip

We did not change the notations for the successive strict transforms of $C$ from 
STEP 3 to STEP $3^{reg}$, because this variant of the algorithm 
does never modify the surfaces produced by the first algorithm 
in the neighborhood of those strict transforms.  
Indeed, the strict transforms never pass through the singular points of the 
modified surfaces $S^{(k)}$ (see Proposition \ref{prop:propstrict}  and Notations 
\ref{def:manycrosses}). 

One has the following description of the result of running the ``regularized'' algorithm:

\begin{proposition}   \label{prop:regalg}
Let $\pi: ( \Sigma, \partial \Sigma) \to (S, L+L')$ be a toroidal pseudo-resolution 
obtained by running Algorithm \ref{alg:tores}.
    Assume that one replaces always STEP 3 with STEP $3^{reg}$ above, 
    choosing the same smooth germs $(L_j)_{j \in J}$ as in the construction of $\pi$. Then one gets 
    a  morphism in the toroidal category 
    $\boxed{\pi^{reg}}: (\Sigma^{reg}, \partial \Sigma^{reg}) \to (S, L + L')$, 
    which is moreover an embedded resolution of $C$ and which  
    factors as $\pi^{reg} = \pi \circ \eta$, where 
                 $\eta: (\Sigma^{reg}, \partial \Sigma^{reg}) \to (\Sigma, \partial \Sigma)$ 
                 is a modification in the toroidal category 
                 whose underlying modification of complex surfaces 
                 is the minimal resolution of the complex surface $\Sigma$.       
\end{proposition}

Let us look at the underlying morphism of complex surfaces $\pi^{reg}: \Sigma^{reg} \to S$. Both 
surfaces  are smooth, therefore this morphism is a composition of blow ups of points, 
by the following theorem of Zariski (see \cite[Corollary 5.4]{H 77} or \cite[Vol.1, Ch. IV.3.4, Thm.5]{S 94}):

\begin{theorem}  \label{thm:composblow} 
   Let $\psi : S_2 \to S_1$ be a modification of a smooth complex surface $S_1$, 
   with $S_2$ also smooth. Then $\psi$ may be written as a composition of blow ups of points. 
\end{theorem}

In Section \ref{sec:embres} we will describe explicitly the combinatorics of the decomposition of 
$\pi^{reg}: \Sigma^{reg} \to S$ into blow ups of points.

\medskip

Let us recall the following classical terminology about objects associated to a process 
of blow ups of points, starting from $o \in S$ (see \cite{LJ 95},  \cite[Chap. 3]{CA 00}, 
\cite{PP 11} and \cite{PPPP 14}):

\begin{definition}   \label{def:infnear}
    Let $(S,o)$ be a smooth germ of surface. 
 
            \noindent
           $\bullet$
           An {\bf infinitely near point} \index{point!infinitely near} of $o$ is either $o$ or a point of the 
    exceptional divisor of a smooth modification of $(S,o)$. Two such points, on two modifications, 
     are considered to be the same, if the associated bimeromorphic map between the two 
     modifications is an isomorphism in their neighborhoods. 
             
             \noindent
           $\bullet$
             If $o_1$ and $o_2$ are two infinitely 
     near points of $o$, then one says that $o_2$ {\bf is proximate to} \index{proximity} $o_1$, written 
     $\boxed{o_2 \rightarrow o_1}$, if $o_2$ belongs to the strict transform 
     of the irreducible rational curve 
     created by blowing up $o_1$. If moreover there is no point 
     $o_3$ such that $o_2 \rightarrow o_3 \rightarrow o_1$, one says that 
     $o_1$ {\bf is the parent of $o_2$}. \index{point!parent}
            
             \noindent
           $\bullet$
              A {\bf finite constellation (above $o$)} \index{constellation} is a finite 
     set $\cC$ of infinitely near points of $o$, closed under the operation of taking 
     the parent. 
            
             \noindent
           $\bullet$
              The {\bf Enriques diagram} \index{diagram!Enriques} \index{Enriques!diagram} 
              $\boxed{\Enriques(\cC)}$ of the finite constellation 
     $\cC$ is the rooted tree with vertex set $\cC$, rooted at $o$, and such that there is an 
     edge joining each point of $\cC$ with its parent. 
\end{definition}

Note that the proximity binary relation on the set of all the infinitely near 
points of $o$ is not a partial order, as it is neither reflexive, nor transitive. For instance, if 
$o_1$ belongs to the exceptional divisor $E_0$ of the blow up of $o$ and 
$o_2$ belongs to the exceptional divisor of the blow up of $o_1$ but not to the srict transform 
of $E_0$ by this blow up, then $o_2 \rightarrow o_1 \rightarrow o$, but $o_2 \nrightarrow o$. 
Therefore, the Enriques diagram 
of a finite constellation encodes only part of the proximity binary relation on it. 
For this reason, Enriques introduced in \cite{EC 17} supplementary rules for the drawing 
of his diagrams, allowing to reconstruct completely the proximity relation. Namely, 
the edges of the Enriques diagram are moreover either \emph{straight} or \emph{curved} and 
there are \emph{breaking points} between some pairs of successive straight edges. As we 
do not insist on those aspects, we do not give the precise definitions, sending 
the interested reader to the literature cited above.

\subsection{The fan tree of a toroidal pseudo-resolution process}
\label{ssec:fantrees}
$\:$
\medskip

In this subsection we explain how to associate a \emph{fan tree} to each process of toroidal 
pseudo-resolution of  a curve singularity $C$ on the smooth germ of surface $(S,o)$ 
(see Definition \ref{def:fantreetr}).  
It is a couple formed by a rooted tree and a $[0, \infty]$-valued 
function constructed from the Newton fans created by the process. It turns out that 
it is isomorphic to the dual graph of the boundary $\partial \Sigma$ of the source 
surface $\Sigma$ of the toroidal pseudo-resolution  morphism $\pi: (\Sigma, \partial \Sigma) \to 
(S, \partial S)$ (see Proposition \ref{prop:fantreedualgr}). 
\medskip

Fan trees are constructed from \emph{trunks} associated with Newton fans. 
Let us define first  those trunks: 

\begin{definition}  \label{def:fantrunk}
    Let $N$ be a $2$-dimensional lattice 
    endowed with a basis $(e_1, e_2)$ and let $\fan$ be a Newton fan of $N$ relative to 
    this basis, in the sense of Definition \ref{def:nfan}. 
    Its {\bf trunk} \index{trunk!of a fan} $\boxed{\theta(\fan)}$ is the segment 
    $[e_1, e_2] \subseteq \sigma_0$ endowed with the {\bf slope function} \index{slope function} 
    \index{function!slope}
    $\boxed{\slp_{\fan}} :  [e_1, e_2] \to [0, \infty]$ which associates with 
    each point $w \in [e_1, e_2]$ 
    the slope in the basis $(e_1, e_2)$ of the ray $\cone w$ generated by it. 
    Its {\bf marked points} \index{point!marked, of a trunk} 
    are the points of intersection of $[e_1, e_2]$ with the  rays of $\fan$. 
    If $\cE \subseteq \Q_+ \cup \{ \infty\}$, we denote by $\boxed{\theta(\cE)}$ the 
    trunk of the fan $\fan(\cE)$ introduced in Definition \ref{def:fan2}. 
\end{definition}

Note that the slope function of a trunk is a homeomorphism. Several examples of trunks 
are represented in Figure \ref{fig:example-trunks}. 

Assume now that we apply Algorithm \ref{alg:tores} to the curve singularity $C$ living 
on the smooth germ of surface $(S,o)$. Consider the 
set $\{(A_i, B_i), i \in I\}$ of crosses produced by the algorithm, as explained in 
 Notations  \ref{def:manycrosses}. Note that we consider also the crosses at which the algorithm 
stops at  an iteration of STEP 1. Denote by $\boxed{(e_{A_i}, e_{B_i})}$ 
the basis $(e_1, e_2)$ of the 
weight lattice $N_{A_i, B_i}$. The segment $[e_{A_i}, e_{B_i}]$ 
is the trunk $\theta(\fan_{A_i, B_i}(C))$. The following definition uses  
 Notations \ref{def:manycrosses}:

\begin{definition}  \label{def:fantreetr}  \index{fan tree}
    The {\bf fan tree $\boxed{(\theta_{\pi}(C), \slp_{\pi})}$ of the toroidal pseudo-resolution  
    $\pi: (\Sigma, \partial \Sigma) \to (S, L + L')$ of $C$} is a pair formed by a rooted tree 
    $\theta_{\pi}(C)$ and a {\bf slope function} $\slp_{\pi} :  \theta_{\pi}(C) \to [0, \infty]$ 
    obtained by gluing the disjoint union of the trunks 
    $(\theta(\fan_{A_i, B_i}(C)), \slp_{\fan_{A_i, B_i}(C)})_{i \in I}$ 
    in the following way:
        \begin{enumerate}
             \item Label each marked point with the corresponding irreducible component $E_k$,  
                  $L_j$ or $C_l$ of the boundary $\partial \Sigma$ of the toroidal surface 
                  $(\Sigma, \partial \Sigma)$.  
             \item Identify all the points of $\bigsqcup_{i \in I} \theta(\fan_{A_i, B_i}(C))$ 
                   which have the same label. 
                 The result of this identification is the fan tree 
                 $\theta_{\pi}(C)$ and the images inside it of the marked  
                 points of $\bigsqcup_{i \in I} \theta(\fan_{A_i, B_i}(C))$ are its {\bf marked points}. 
                 \index{point!marked, of a fan tree} We keep for each one of 
                 them the same label as in the initial trunks.  
             \item The {\bf root} of $\theta_{\pi}(C)$ is the point labeled by the initial smooth branch $L$. 
             \item For every $i \in I$, the restriction of $\slp_{\pi}$ 
                   to every half-open trunk 
                   $\theta(\fan_{A_i, B_i}(C)) \setminus \{e_{A_i}\}= (e_{A_i}, e_{B_i}] \hookrightarrow 
                   \theta_{\pi}(C)$ is equal to $\slp_{\fan_{A_i, B_i}(C)}$. 
             \item At the root, $\slp_{\pi}(L) =  \slp_{\fan_{L, L_1}(C)}(L) = 0$. 
        \end{enumerate}
\end{definition}

As in any rooted tree, the root $L$ defines a partial order $\boxed{\preceq_L}$ on the set 
of vertices of the fan tree $\theta_{\pi}(C)$ (that is, on its set of marked points), by declaring 
that $P \preceq_L Q$ if and only if the unique segment $[L, P]$ joining $L$ and 
$P$ inside the tree is included in the analogous segment $[L, Q]$.

Note that the slope function $\slp_{\pi}$ is discontinuous at all the marked points 
of $\theta_{\pi}(C)$ resulting from the identification of points of two different trunks, 
its directional limits jumping from a positive value to $0$  when one passes from 
one trunk to another one in increasing way relative to the partial order $\preceq_L$. 
It follows that the fan tree of a toroidal pseudo-resolution
determines the trunks $(\theta(\fan_{A_i, B_i}(C)), \slp_{\fan_{A_i, B_i}(C)})_{i \in I}$.

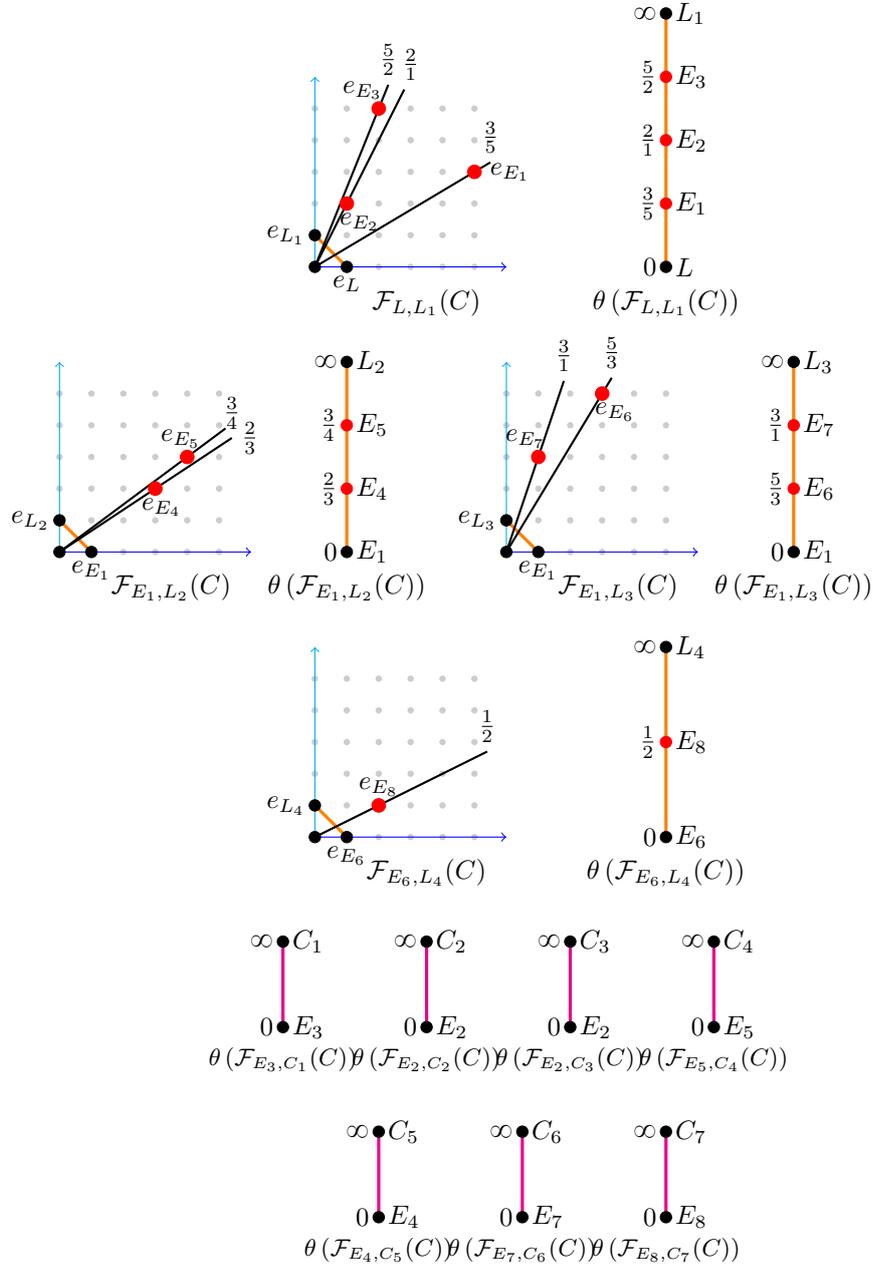
\begin{figure}
\begin{center}
\begin{tikzpicture}[scale=0.42]

 \begin{scope}[shift={(-3,0)},scale=1]
\foreach \x in {0,1,...,5}{
\foreach \y in {0,1,...,5}{
       \node[draw,circle,inner sep=0.7pt,fill, color=gray!40] at (1*\x,1*\y) {}; }
   }
\draw [->, color=cyan](0,0) -- (0,6);
\draw [->, color=blue](0,0) -- (6,0);
\draw [-, color=orange, very thick](1,0) -- (0, 1) ; 
\node[draw,circle, inner sep=1.5pt,color=black, fill=black] at (0,0){};
\node[draw,circle, inner sep=1.5pt,color=black, fill=black] at (1,0){};
\node[draw,circle, inner sep=1.5pt,color=black, fill=black] at (0,1){};
\node [left] at (0,1) {$e_{L_{1}}$};
\node [below] at (1,0) {$e_L$};
\node [left,above] at (1.5,5) {$e_{E_{3}}$};
\draw [-, thick](0,0) -- (2.3,5.75);
\node[draw,circle, inner sep=1.8pt,color=red, fill=red] at (2,5){};
\node [above] at (2.3,5.75) {$\frac{5}{2}$}; 
\node [below] at (1.4,2) {$e_{E_{2}}$};
\draw [-,thick](0,0) -- (2.8,5.6);
\node[draw,circle, inner sep=1.8pt,color=red, fill=red] at (1,2){};
\node [above] at (3,5.6) {$\frac{2}{1}$};
\node [right] at (5.2,3) {$e_{E_{1}}$};
\draw [-, thick](0,0) -- (5.5,3.3);
\node[draw,circle, inner sep=1.8pt,color=red, fill=red] at (5,3){};
\node [above] at (5.5,3.3) {$\frac{3}{5}$};
\node [below] at (3.5,-0.4) {$\fan_{L, L_1}(C)$};
   \draw [-, color=orange, very thick](11,0) -- (11, 8) ; 
   \node[draw,circle, inner sep=1.5pt,color=black, fill=black] at (11,0){};
   \node [right] at (11,0) {$L$};
   \node [left] at (11,0) {$0$};
   \node[draw,circle, inner sep=1.5pt,color=black, fill=black] at (11,8){};
   \node [right] at (11,8) {${L_1}$};
   \node [left] at (11,8) {$\infty$};
   \node[draw,circle, inner sep=1.5pt,color=red, fill=red] at (11,2){};
   \node [right] at (11,2) {${E_1}$};
   \node [left] at (11,2) {$\frac{3}{5}$};
   \node[draw,circle, inner sep=1.5pt,color=red, fill=red] at (11,4){};
   \node [right] at (11,4) {${E_2}$};
   \node [left] at (11,4) {$\frac{2}{1}$};
   \node[draw,circle, inner sep=1.5pt,color=red, fill=red] at (11,6){};
   \node [right] at (11,6) {${E_3}$};
   \node [left] at (11,6) {$\frac{5}{2}$};
   \node [below] at (11,-0.4) {$\theta\left(\fan_{L, L_1}(C)\right)$};
 \end{scope}

 \begin{scope}[shift={(-11,-9)},scale=1]
\foreach \x in {0,1,...,5}{
\foreach \y in {0,1,...,5}{
       \node[draw,circle,inner sep=0.7pt,fill, color=gray!40] at (1*\x,1*\y) {}; }
   }
\draw [->, color=cyan](0,0) -- (0,6);
\draw [->, color=blue](0,0) -- (6,0);
\draw [-, color=orange, very thick](1,0) -- (0, 1) ; 
\node[draw,circle, inner sep=1.5pt,color=black, fill=black] at (0,0){};
\node[draw,circle, inner sep=1.5pt,color=black, fill=black] at (1,0){};
\node[draw,circle, inner sep=1.5pt,color=black, fill=black] at (0,1){};
\node [left] at (0,1) {$e_{L_{2}}$};
\node [below] at (1,0) {$e_{E_{1}}$};
\node [left,above] at (3.8,3) {$e_{E_{5}}$};
\draw [-, thick](0,0) -- (5.2,3.9);
\node[draw,circle, inner sep=1.8pt,color=red, fill=red] at (4,3){};
\node [right,below] at (3.2,2) {$e_{E_{4}}$};
\draw [-, thick](0,0) -- (5.4,3.6);
\node[draw,circle, inner sep=1.8pt,color=red, fill=red] at (3,2){};
\node [above] at (5.4,3.6) {$\frac{3}{4}$};
\node [right] at (5.4,3.6) {$\frac{2}{3}$};
\node [below] at (3.5,-0.4) {$\fan_{E_{1}, L_2}(C)$};
  \draw [-, color=orange, very thick](9,0) -- (9, 6) ; 
   \node[draw,circle, inner sep=1.5pt,color=black, fill=black] at (9,0){};
   \node [right] at (9,0) {$E_1$};
    \node [left] at (9,0) {$0$};
   \node[draw,circle, inner sep=1.5pt,color=black, fill=black] at (9,6){};
   \node [right] at (9,6) {${L_2}$};
   \node [left] at (9,6) {$\infty$};
   \node[draw,circle, inner sep=1.5pt,color=red, fill=red] at (9,2){};
   \node [right] at (9,2) {${E_4}$};
   \node [left] at (9,2) {$\frac{2}{3}$};
   \node[draw,circle, inner sep=1.5pt,color=red, fill=red] at (9,4){};
   \node [right] at (9,4) {${E_5}$};
   \node [left] at (9,4) {$\frac{3}{4}$};
   \node [below] at (9,-0.4) {$\theta\left(\fan_{E_1, L_2}(C)\right)$};
 \end{scope}

 \begin{scope}[shift={(3,-9)},scale=1]
\foreach \x in {0,1,...,5}{
\foreach \y in {0,1,...,5}{
       \node[draw,circle,inner sep=0.7pt,fill, color=gray!40] at (1*\x,1*\y) {}; }
   }
\draw [->, color=cyan](0,0) -- (0,6);
\draw [->, color=blue](0,0) -- (6,0);
\draw [-, color=orange, very thick](1,0) -- (0, 1) ; 
\node[draw,circle, inner sep=1.5pt,color=black, fill=black] at (0,0){};
\node[draw,circle, inner sep=1.5pt,color=black, fill=black] at (1,0){};
\node[draw,circle, inner sep=1.5pt,color=black, fill=black] at (0,1){};
\node [left] at (0,1) {$e_{L_{3}}$};
\node [below] at (1,0) {$e_{E_{1}}$};
\node [left,above] at (0.6,3) {$e_{E_{7}}$};
\draw [-, thick](0,0) -- (1.8,5.4);
\node[draw,circle, inner sep=1.8pt,color=red, fill=red] at (1,3){};
\node [right,below] at (3.4,5) {$e_{E_{6}}$};
\draw [-, thick](0,0) -- (3.3,5.5);
\node[draw,circle, inner sep=1.8pt,color=red, fill=red] at (3,5){};
\node [above] at (1.8,5.4) {$\frac{3}{1}$};
\node [above] at (3.3,5.5) {$\frac{5}{3}$};
\node [below] at (3.5,-0.4) {$\fan_{E_{1}, L_3}(C)$};
  \draw [-, color=orange, very thick](9,0) -- (9, 6) ; 
   \node[draw,circle, inner sep=1.5pt,color=black, fill=black] at (9,0){};
   \node [right] at (9,0) {$E_1$};
    \node [left] at (9,0) {$0$};
   \node[draw,circle, inner sep=1.5pt,color=black, fill=black] at (9,6){};
   \node [right] at (9,6) {${L_3}$};
   \node [left] at (9,6) {$\infty$};
   \node[draw,circle, inner sep=1.5pt,color=red, fill=red] at (9,2){};
   \node [right] at (9,2) {${E_6}$};
   \node [left] at (9,2) {$\frac{5}{3}$};
   \node[draw,circle, inner sep=1.5pt,color=red, fill=red] at (9,4){};
   \node [right] at (9,4) {${E_7}$};
   \node [left] at (9,4) {$\frac{3}{1}$};
   \node [below] at (9,-0.4) {$\theta\left(\fan_{E_1, L_3}(C)\right)$};
 \end{scope}

 \begin{scope}[shift={(-3,-18)},scale=1]
\foreach \x in {0,1,...,5}{
\foreach \y in {0,1,...,5}{
       \node[draw,circle,inner sep=0.7pt,fill, color=gray!40] at (1*\x,1*\y) {}; }
   }
\draw [->, color=cyan](0,0) -- (0,6);
\draw [->, color=blue](0,0) -- (6,0);
\draw [-, color=orange, very thick](1,0) -- (0, 1) ; 
\node[draw,circle, inner sep=1.5pt,color=black, fill=black] at (0,0){};
\node[draw,circle, inner sep=1.5pt,color=black, fill=black] at (1,0){};
\node[draw,circle, inner sep=1.5pt,color=black, fill=black] at (0,1){};
\node [left] at (0,1) {$e_{L_{4}}$};
\node [below] at (1,0) {$e_{E_{6}}$};
\node [left,above] at (2,1) {$e_{E_{8}}$};
\draw [-, thick](0,0) -- (5.4,2.7);
\node[draw,circle, inner sep=1.8pt,color=red, fill=red] at (2,1){};
\node [above] at (5.4,2.7) {$\frac{1}{2}$}; 
\node [below] at (3.5,-0.4) {$\fan_{E_{6}, L_4}(C)$};
  \draw [-, color=orange, very thick](11,0) -- (11, 6) ; 
   \node[draw,circle, inner sep=1.5pt,color=black, fill=black] at (11,0){};
   \node [right] at (11,0) {$E_6$};
    \node [left] at (11,0) {$0$};
   \node[draw,circle, inner sep=1.5pt,color=black, fill=black] at (11,6){};
   \node [right] at (11,6) {${L_4}$};
   \node [left] at (11,6) {$\infty$};
   \node[draw,circle, inner sep=1.5pt,color=red, fill=red] at (11,3){};
   \node [right] at (11,3) {${E_8}$};
   \node [left] at (11,3) {$\frac{1}{2}$};
   \node [below] at (11,-0.4) {$\theta\left(\fan_{E_6, L_4}(C)\right)$};
 \end{scope}

\begin{scope}[shift={(5,-24)},scale=0.9]
  \draw [-, color=magenta, very thick](-10,0) -- (-10, 3) ; 
   \node[draw,circle, inner sep=1.5pt,color=black, fill=black] at (-10,0){};
   \node [right] at (-10,0) {$E_3$};
    \node [left] at (-10,0) {$0$};
   \node[draw,circle, inner sep=1.5pt,color=black, fill=black] at (-10,3){};
   \node [right] at (-10,3) {${C_1}$};
   \node [left] at (-10,3) {$\infty$};
   \node [below] at (-10,-0.4) {\small $\theta\left(\fan_{E_3, C_1}(C)\right)$};
   
   \draw [-, color=magenta, very thick](-5,0) -- (-5, 3) ; 
   \node[draw,circle, inner sep=1.5pt,color=black, fill=black] at (-5,0){};
   \node [right] at (-5,0) {$E_2$};
    \node [left] at (-5,0) {$0$};
   \node[draw,circle, inner sep=1.5pt,color=black, fill=black] at (-5,3){};
   \node [right] at (-5,3) {${C_2}$};
   \node [left] at (-5,3) {$\infty$};
   \node [below] at (-5,-0.4) {\small $\theta\left(\fan_{E_2, C_2}(C)\right)$};
   
   \draw [-, color=magenta, very thick](0,0) -- (0, 3) ; 
   \node[draw,circle, inner sep=1.5pt,color=black, fill=black] at (0,0){};
   \node [right] at (0,0) {$E_2$};
    \node [left] at (0,0) {$0$};
   \node[draw,circle, inner sep=1.5pt,color=black, fill=black] at (0,3){};
   \node [right] at (0,3) {${C_3}$};
   \node [left] at (0,3) {$\infty$};
   \node [below] at (0,-0.4) {\small $\theta\left(\fan_{E_2, C_3}(C)\right)$};
   
   \draw [-, color=magenta, very thick](5,0) -- (5, 3) ; 
   \node[draw,circle, inner sep=1.5pt,color=black, fill=black] at (5,0){};
   \node [right] at (5,0) {$E_5$};
    \node [left] at (5,0) {$0$};
   \node[draw,circle, inner sep=1.5pt,color=black, fill=black] at (5,3){};
   \node [right] at (5,3) {${C_4}$};
   \node [left] at (5,3) {$\infty$};
   \node [below] at (5,-0.4) {\small $\theta\left(\fan_{E_5, C_4}(C)\right)$};
   
    \end{scope}
   
   \begin{scope}[shift={(-10,-30)},scale=0.9]
   \draw [-, color=magenta, very thick](10,0) -- (10, 3) ; 
   \node[draw,circle, inner sep=1.5pt,color=black, fill=black] at (10,0){};
   \node [right] at (10,0) {$E_4$};
    \node [left] at (10,0) {$0$};
   \node[draw,circle, inner sep=1.5pt,color=black, fill=black] at (10,3){};
   \node [right] at (10,3) {${C_5}$};
   \node [left] at (10,3) {$\infty$};
   \node [below] at (10,-0.4) {\small $\theta\left(\fan_{E_4, C_5}(C)\right)$};
   
    \draw [-, color=magenta, very thick](15,0) -- (15, 3) ; 
   \node[draw,circle, inner sep=1.5pt,color=black, fill=black] at (15,0){};
   \node [right] at (15,0) {$E_7$};
    \node [left] at (15,0) {$0$};
   \node[draw,circle, inner sep=1.5pt,color=black, fill=black] at (15,3){};
   \node [right] at (15,3) {${C_6}$};
   \node [left] at (15,3) {$\infty$};
   \node [below] at (15,-0.4) {\small $\theta\left(\fan_{E_7, C_6}(C)\right)$};
   
    \draw [-, color=magenta, very thick](20,0) -- (20, 3) ; 
   \node[draw,circle, inner sep=1.5pt,color=black, fill=black] at (20,0){};
   \node [right] at (20,0) {$E_8$};
    \node [left] at (20,0) {$0$};
   \node[draw,circle, inner sep=1.5pt,color=black, fill=black] at (20,3){};
   \node [right] at (20,3) {${C_7}$};
   \node [left] at (20,3) {$\infty$};
   \node [below] at (20,-0.4) {\small $\theta\left(\fan_{E_8, C_7}(C)\right)$};   
 \end{scope}

\end{tikzpicture}
\end{center}
 \caption{The trunks associated to the toroidal pseudo-resolution of Example \ref{ex:toroidres}}  
 \label{fig:example-trunks} 
     \end{figure}


\begin{figure}
\begin{center}
\begin{tikzpicture}[scale=0.4]

 \begin{scope}[shift={(-8,0)},scale=1]
   \draw [-, color=orange, very thick](11,0) -- (11, 8) ; 
   \node[draw,circle, inner sep=1.5pt,color=black, fill=black] at (11,0){};
   \node [right] at (11,0) {$L$};
   \node [left] at (11,0) {$0$};
   \node[draw,circle, inner sep=1.5pt,color=black, fill=black] at (11,8){};
   \node [right] at (11,8) {${L_1}$};
   \node [left] at (11,8) {$\infty$};
   \node[draw,circle, inner sep=1.5pt,color=red, fill=red] at (11,2){};
   \node [right] at (11,2) {${E_1}$};
   \node [left] at (11,2) {$\frac{3}{5}$};
   \node[draw,circle, inner sep=1.5pt,color=red, fill=red] at (11,4){};
   \node [right] at (11,4) {${E_2}$};
   \node [left] at (11,4) {$\frac{2}{1}$};
   \node[draw,circle, inner sep=1.5pt,color=red, fill=red] at (11,6){};
   \node [right] at (11,6) {${E_3}$};
   \node [left] at (11,6) {$\frac{5}{2}$};
   \node [below] at (11,-0.4) {$\theta\left(\fan_{L, L_1}(C)\right)$};
 \end{scope}

 \begin{scope}[shift={(-3,0)},scale=1]
  \draw [-, color=orange, very thick](11,0) -- (11, 6) ; 
   \node[draw,circle, inner sep=1.5pt,color=black, fill=black] at (11,0){};
   \node [right] at (11,0) {$E_1$};
    \node [left] at (11,0) {$0$};
   \node[draw,circle, inner sep=1.5pt,color=black, fill=black] at (11,6){};
   \node [right] at (11,6) {${L_2}$};
   \node [left] at (11,6) {$\infty$};
   \node[draw,circle, inner sep=1.5pt,color=red, fill=red] at (11,2){};
   \node [right] at (11,2) {${E_4}$};
   \node [left] at (11,2) {$\frac{2}{3}$};
   \node[draw,circle, inner sep=1.5pt,color=red, fill=red] at (11,4){};
   \node [right] at (11,4) {${E_5}$};
   \node [left] at (11,4) {$\frac{3}{4}$};
   \node [below] at (11,-0.4) {$\theta\left(\fan_{E_1, L_2}(C)\right)$};
 \end{scope}

 \begin{scope}[shift={(2,0)},scale=1]
  \draw [-, color=orange, very thick](11,0) -- (11, 6) ; 
   \node[draw,circle, inner sep=1.5pt,color=black, fill=black] at (11,0){};
   \node [right] at (11,0) {$E_1$};
    \node [left] at (11,0) {$0$};
   \node[draw,circle, inner sep=1.5pt,color=black, fill=black] at (11,6){};
   \node [right] at (11,6) {${L_3}$};
   \node [left] at (11,6) {$\infty$};
   \node[draw,circle, inner sep=1.5pt,color=red, fill=red] at (11,2){};
   \node [right] at (11,2) {${E_6}$};
   \node [left] at (11,2) {$\frac{5}{3}$};
   \node[draw,circle, inner sep=1.5pt,color=red, fill=red] at (11,4){};
   \node [right] at (11,4) {${E_7}$};
   \node [left] at (11,4) {$\frac{3}{1}$};
   \node [below] at (11,-0.4) {$\theta\left(\fan_{E_1, L_3}(C)\right)$};
 \end{scope}

 \begin{scope}[shift={(7,0)},scale=1]
  \draw [-, color=orange, very thick](11,0) -- (11, 4) ; 
   \node[draw,circle, inner sep=1.5pt,color=black, fill=black] at (11,0){};
   \node [right] at (11,0) {$E_6$};
    \node [left] at (11,0) {$0$};
   \node[draw,circle, inner sep=1.5pt,color=black, fill=black] at (11,4){};
   \node [right] at (11,4) {${L_4}$};
   \node [left] at (11,4) {$\infty$};
   \node[draw,circle, inner sep=1.5pt,color=red, fill=red] at (11,2){};
   \node [right] at (11,2) {${E_8}$};
   \node [left] at (11,2) {$\frac{1}{2}$};
   \node [below] at (11,-0.4) {$\theta\left(\fan_{E_6, L_4}(C)\right)$};
 \end{scope}

\begin{scope}[shift={(13,-6)},scale=1]
  \draw [-, color=magenta, very thick](-10,0) -- (-10, 3) ; 
   \node[draw,circle, inner sep=1.5pt,color=black, fill=black] at (-10,0){};
   \node [right] at (-10,0) {$E_3$};
    \node [left] at (-10,0) {$0$};
   \node[draw,circle, inner sep=1.5pt,color=black, fill=black] at (-10,3){};
   \node [right] at (-10,3) {${C_1}$};
   \node [left] at (-10,3) {$\infty$};
   \node [below] at (-10,-0.4) {$\theta\left(\fan_{E_3, C_1}(C)\right)$};
   
   \draw [-, color=magenta, very thick](-5,0) -- (-5, 3) ; 
   \node[draw,circle, inner sep=1.5pt,color=black, fill=black] at (-5,0){};
   \node [right] at (-5,0) {$E_2$};
    \node [left] at (-5,0) {$0$};
   \node[draw,circle, inner sep=1.5pt,color=black, fill=black] at (-5,3){};
   \node [right] at (-5,3) {${C_2}$};
   \node [left] at (-5,3) {$\infty$};
   \node [below] at (-5,-0.4) {$\theta\left(\fan_{E_2, C_2}(C)\right)$};
   
   \draw [-, color=magenta, very thick](0,0) -- (0, 3) ; 
   \node[draw,circle, inner sep=1.5pt,color=black, fill=black] at (0,0){};
   \node [right] at (0,0) {$E_2$};
    \node [left] at (0,0) {$0$};
   \node[draw,circle, inner sep=1.5pt,color=black, fill=black] at (0,3){};
   \node [right] at (0,3) {${C_3}$};
   \node [left] at (0,3) {$\infty$};
   \node [below] at (0,-0.4) {$\theta\left(\fan_{E_2, C_3}(C)\right)$};
   
   \draw [-, color=magenta, very thick](5,0) -- (5, 3) ; 
   \node[draw,circle, inner sep=1.5pt,color=black, fill=black] at (5,0){};
   \node [right] at (5,0) {$E_5$};
    \node [left] at (5,0) {$0$};
   \node[draw,circle, inner sep=1.5pt,color=black, fill=black] at (5,3){};
   \node [right] at (5,3) {${C_4}$};
   \node [left] at (5,3) {$\infty$};
   \node [below] at (5,-0.4) {$\theta\left(\fan_{E_5, C_4}(C)\right)$};
   
    \end{scope}

    \begin{scope}[shift={(-5,-12)},scale=1]
   \draw [-, color=magenta, very thick](10,0) -- (10, 3) ; 
   \node[draw,circle, inner sep=1.5pt,color=black, fill=black] at (10,0){};
   \node [right] at (10,0) {$E_4$};
    \node [left] at (10,0) {$0$};
   \node[draw,circle, inner sep=1.5pt,color=black, fill=black] at (10,3){};
   \node [right] at (10,3) {${C_5}$};
   \node [left] at (10,3) {$\infty$};
   \node [below] at (10,-0.4) {$\theta\left(\fan_{E_4, C_5}(C)\right)$};
   
    \draw [-, color=magenta, very thick](15,0) -- (15, 3) ; 
   \node[draw,circle, inner sep=1.5pt,color=black, fill=black] at (15,0){};
   \node [right] at (15,0) {$E_7$};
    \node [left] at (15,0) {$0$};
   \node[draw,circle, inner sep=1.5pt,color=black, fill=black] at (15,3){};
   \node [right] at (15,3) {${C_6}$};
   \node [left] at (15,3) {$\infty$};
   \node [below] at (15,-0.4) {$\theta\left(\fan_{E_7, C_6}(C)\right)$};
   
    \draw [-, color=magenta, very thick](20,0) -- (20, 3) ; 
   \node[draw,circle, inner sep=1.5pt,color=black, fill=black] at (20,0){};
   \node [right] at (20,0) {$E_8$};
    \node [left] at (20,0) {$0$};
   \node[draw,circle, inner sep=1.5pt,color=black, fill=black] at (20,3){};
   \node [right] at (20,3) {${C_7}$};
   \node [left] at (20,3) {$\infty$};
   \node [below] at (20,-0.4) {$\theta\left(\fan_{E_8, C_7}(C)\right)$};   
 \end{scope}
 
 
  \draw [->, thick](9,-14) -- (9, -16.5) ; 
   \node [right, color=black] at (9,-15.2) {gluing};

  \begin{scope}[shift={(9,-30)},scale=1.5]
    \draw [-, color=orange, very thick](0,0) -- (0, 8) ; 
    \draw [-, color=orange, very thick](0, 2) -- (6, 2) ; 
     \draw [-, color=orange, very thick](0, 2) -- (6, 8) ; 
     \draw [-, color=orange, very thick](2,4) -- (2, 8) ; 
     \draw [-, color=magenta, very thick](0,6) -- (-2, 6) ; 
     \draw [-, color=magenta, very thick](2,2) -- (2, 0) ; 
     \draw [-, color=magenta, very thick](4,6) -- (6, 6) ; 
     \draw [-, color=magenta, very thick](2,7) -- (4, 7) ; 
     \draw [-, color=magenta, very thick](4,2) -- (4, 0) ; 
     \draw [-, color=magenta, very thick](0,4) -- (-2, 4) ; 
     \draw [-, color=magenta, very thick](0,4) -- (-2, 2) ;
   \node[draw,circle, inner sep=1.5pt,color=black, fill=black] at (0,0){};
   \node [right] at (0,0) {$L$};
   \node [left] at (0,0) {$0$};
   \node[draw,circle, inner sep=1.5pt,color=black, fill=black] at (0,8){};
   \node [right] at (0,8) {${L_1}$};
   \node [left] at (0,8) {$\infty$};
   \node[draw,circle, inner sep=1.5pt,color=red, fill=red] at (0,2){};
   \node [right] at (0,1.5) {${E_1}$};
   \node [left] at (0,1.5) {$\frac{3}{5}$};
   \node[draw,circle, inner sep=1.5pt,color=red, fill=red] at (0,4){};
   \node [right] at (0,4) {${E_2}$};
   \node [left] at (0,4.5) {$\frac{2}{1}$};
   \node[draw,circle, inner sep=1.5pt,color=red, fill=red] at (0,6){};
   \node [right] at (0,6) {${E_3}$};
   \node [left] at (0,6.5) {$\frac{5}{2}$};
   \node[draw,circle, inner sep=1.5pt,color=black, fill=black] at (-2,6){};
   \node [below] at (-2,6) {${C_1}$};
   \node [above] at (-2,6) {$\infty$};   
   \node[draw,circle, inner sep=1.5pt,color=black, fill=black] at (-2,4){};
   \node [below] at (-2,4) {${C_2}$};
   \node [above] at (-2,4) {$\infty$}; 
   \node[draw,circle, inner sep=1.5pt,color=black, fill=black] at (-2,2){};
   \node [below] at (-2,2) {${C_3}$};
   \node [above] at (-2,2) {$\infty$}; 
   \node[draw,circle, inner sep=1.5pt,color=black, fill=black] at (6,2){};
   \node [below] at (6,2) {${L_2}$};
   \node [above] at (6,2) {$\infty$};
   \node[draw,circle, inner sep=1.5pt,color=red, fill=red] at (2,2){};
   \node [below] at (2.5,2) {${E_4}$};
   \node [above] at (2,2) {$\frac{2}{3}$};
   \node[draw,circle, inner sep=1.5pt,color=red, fill=red] at (4,2){};
   \node [below] at (4.5,2) {${E_5}$};
   \node [above] at (4,2) {$\frac{3}{4}$};
   \node[draw,circle, inner sep=1.5pt,color=black, fill=black] at (2,0){};
   \node [right] at (2,0) {${C_5}$};
   \node [left] at (2,0) {$\infty$}; 
   \node[draw,circle, inner sep=1.5pt,color=black, fill=black] at (4,0){};
   \node [right] at (4,0) {${C_4}$};
   \node [left] at (4,0) {$\infty$}; 
   \node[draw,circle, inner sep=1.5pt,color=black, fill=black] at (6,8){};
   \node [below] at (6,7.8) {${L_3}$};
   \node [above] at (6,8.2) {$\infty$};
   \node[draw,circle, inner sep=1.5pt,color=red, fill=red] at (4,6){};
   \node [below] at (4,5.8) {${E_7}$};
   \node [left] at (4,6.2) {$\frac{3}{1}$};
   \node[draw,circle, inner sep=1.5pt,color=red, fill=red] at (2,4){};
   \node [right] at (2.2,4) {${E_6}$};
   \node [left] at (2,4.2) {$\frac{5}{3}$};
     \node[draw,circle, inner sep=1.5pt,color=black, fill=black] at (6, 6){};
   \node [below] at (6, 6) {${C_6}$};
   \node [above] at (6, 6) {$\infty$};   
   \node[draw,circle, inner sep=1.5pt,color=black, fill=black] at (2,8){};
   \node [right] at (2,8) {${L_4}$};
   \node [left] at (2,8) {$\infty$};
   \node[draw,circle, inner sep=1.5pt,color=red, fill=red] at (2,7){};
   \node [right] at (2,6.5) {${E_8}$};
   \node [left] at (2,6.5) {$\frac{1}{2}$};
   \node[draw,circle, inner sep=1.5pt,color=black, fill=black] at (4, 7){};
   \node [below] at (4, 7) {${C_7}$};
   \node [above] at (4, 7) {$\infty$};
  \end{scope}
 \node [below] at (9, -32) {$\theta_{\pi}(C)$}; 
\end{tikzpicture}
\end{center}
 \caption{Construction of the fan tree of the toroidal pseudo-resolution of Example \ref{ex:toroidres}}  
 \label{fig:example-fantree} 
     \end{figure}
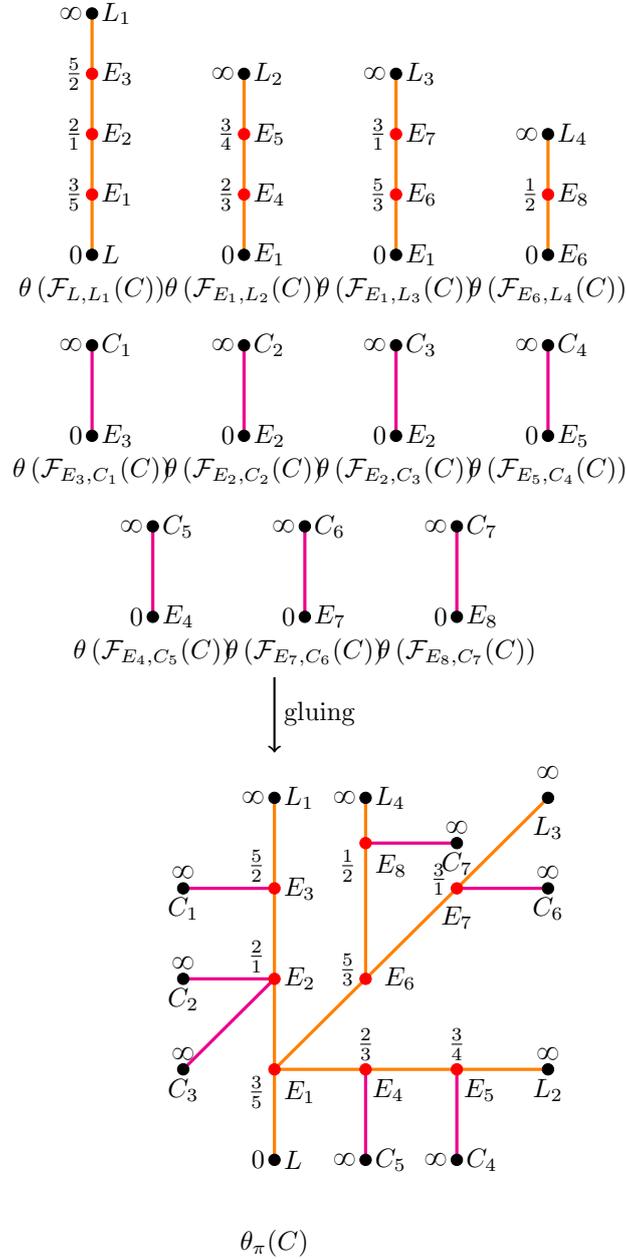


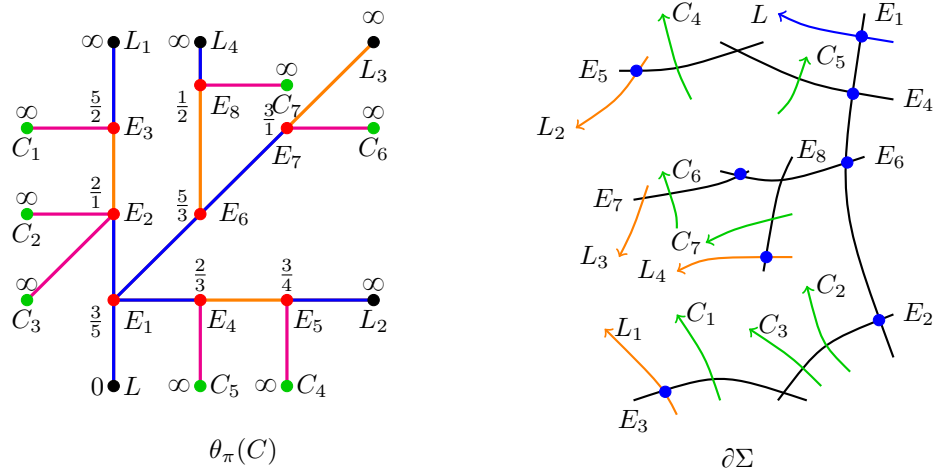
\begin{figure}
\begin{center}
\begin{tikzpicture}[scale=0.38]

\begin{scope}[shift={(18,1.5)}][scale=0.5]
\draw[thick, color=black](6.5,-4) .. controls (4.5,1.5) ..(5.5,8); 
\node [right] at (5.5,8) {$E_1$}; 
\draw[thick, color=black](6.5,-2.5) .. controls (4,-3.5) ..(2.5,-5.5);
\node [right] at (6.5,-2.5) {$E_2$};
\draw [->][thick, color=black!20!green](5,-4.5) .. controls (4,-3.5) ..(3.5,-1.5); 
\node [right] at (3.5,-1.5) {$C_2$}; 
\draw [->][thick, color=black!20!green](4,-5) .. controls (3,-4) ..(1.5,-3);
\node [right] at (1.5,-2.9) {$C_3$}; 
\draw[thick, color=black](3.5,-5.5) .. controls (0.5,-4.5) ..(-2.5,-5.5);  
\node [left, below] at (-2.5,-5.5) {$E_3$}; 
\draw [->][thick, color=black!20!green](0.5,-5.5) .. controls (0,-4) ..(-1,-2.5);
\node [right] at (-1,-2.5) {$C_1$}; 
\draw [->][thick, color=orange](-1,-6) .. controls (-1.5,-5) ..(-3.5,-3);
\node [right] at (-3.5,-3) {$L_{1}$};
\draw[thick, color=black](5.5,3) .. controls (2.5,2) ..(0.5,2.5); 
\node [right] at (5.5,3) {$E_6$}; 
\draw[thick, color=black](2,-1) .. controls (2.5,2) ..(3,3); 
\node [right] at (2.8,3.2) {$E_8$}; 
\draw [->][thick, color=black!20!green](3,1) .. controls (1,0.5) ..(0,0);
\node [left] at (0.2,0) {$C_7$}; 
\draw [->][thick, color=orange](3,-0.5) .. controls (0,-0.5) ..(-1,-1);
\node [left] at (-1,-1) {$L_4$}; 
\draw[thick, color=black](1.5,2.5) .. controls (0.5,2) ..(-2.5,1.5); 
\node [left] at (-2.5,1.5) {$E_7$}; 
\draw [->][thick, color=black!20!green](-1,0.5) .. controls (-1,1) ..(-1.5,2.5);
\node [right] at (-1.5,2.5) {$C_6$}; 
\draw [->][thick, color=orange](-2,2) .. controls (-2.5,0.5) ..(-3,-0.5);
\node [left] at (-3,-0.5) {$L_3$}; 
\draw [->][thick, color=blue](6.5,7) .. controls (3.5,7.5) ..(2.5,8);
\node [left] at (2.5,8) {$L$}; 
\draw[thick, color=black](6.5,5) .. controls (3.5,5.5) ..(0.5,7);
\node [right] at (6.5,5) {$E_4$};  
\draw[thick, color=black](2,7) .. controls (-1,6) ..(-3,6); 
\node [left] at (-3,6) {$E_5$}; 
\draw [->][thick, color=black!20!green](2.5,4.5) .. controls (3,5) ..(3.5,6.5);
\node [right] at (3.5,6.5) {$C_5$}; 
\draw [->][thick, color=black!20!green](-0.5,5) .. controls (-1,6) ..(-1.5,8);
\node [right] at (-1.5,8) {$C_4$}; 
\draw [->][thick, color=orange](-2,6.5) .. controls (-3,5) ..(-4.5,4);
\node [left] at (-4.5,4) {$L_2$}; 

\node[draw,circle,inner sep=1.5pt,fill=blue, color=blue] at (5.4,7.2){};
\node[draw,circle,inner sep=1.5pt,fill=blue, color=blue] at (6,-2.7){};
\node[draw,circle,inner sep=1.5pt,fill=blue, color=blue] at (-1.4,-5.2){};
\node[draw,circle,inner sep=1.5pt,fill=blue, color=blue] at (5.1,5.2){};
\node[draw,circle,inner sep=1.5pt,fill=blue, color=blue] at (-2.4,6){};
\node[draw,circle,inner sep=1.5pt,fill=blue, color=blue] at (4.9,2.8){};
\node[draw,circle,inner sep=1.5pt,fill=blue, color=blue] at (1.2,2.4){};
\node[draw,circle,inner sep=1.5pt,fill=blue, color=blue] at (2.1,-0.5){};

\node [left] at (2,-7.5) {$\partial \Sigma$}; 
 \end{scope}
 
 \begin{scope}[shift={(-2.5,-3.5)},scale=1.5]
    \draw [-, color=orange, very thick](0,0) -- (0, 8) ; 
    \draw [-, color=orange, very thick](0, 2) -- (6, 2) ; 
     \draw [-, color=orange, very thick](0, 2) -- (6, 8) ; 
     \draw [-, color=orange, very thick](2,4) -- (2, 8) ; 
     \draw [-, color=magenta, very thick](0,6) -- (-2, 6) ; 
     \draw [-, color=magenta, very thick](2,2) -- (2, 0) ; 
     \draw [-, color=magenta, very thick](4,6) -- (6, 6) ; 
     \draw [-, color=magenta, very thick](2,7) -- (4, 7) ; 
     \draw [-, color=magenta, very thick](4,2) -- (4, 0) ; 
     \draw [-, color=magenta, very thick](0,4) -- (-2, 4) ; 
     \draw [-, color=magenta, very thick](0,4) -- (-2, 2) ;
      \draw [-, color=blue, very thick](0,0) -- (0, 4) ; 
      \draw [-, color=blue, very thick](0,6) -- (0, 8) ; 
      \draw [-, color=blue, very thick](0,2) -- (2, 2) ; 
      \draw [-, color=blue, very thick](4,2) -- (6, 2) ; 
      \draw [-, color=blue, very thick](0,2) -- (4,6) ;
      \draw [-, color=blue, very thick](2,7) -- (2,8) ;  

   \node[draw,circle, inner sep=1.5pt,color=black, fill=black] at (0,0){};
   \node [right] at (0,0) {$L$};
   \node [left] at (0,0) {$0$};
   \node[draw,circle, inner sep=1.5pt,color=black, fill=black] at (0,8){};
   \node [right] at (0,8) {${L_1}$};
   \node [left] at (0,8) {$\infty$};
   \node[draw,circle, inner sep=1.5pt,color=red, fill=red] at (0,2){};
   \node [right] at (0,1.5) {${E_1}$};
   \node [left] at (0,1.5) {$\frac{3}{5}$};
   \node[draw,circle, inner sep=1.5pt,color=red, fill=red] at (0,4){};
   \node [right] at (0,4) {${E_2}$};
   \node [left] at (0,4.5) {$\frac{2}{1}$};
   \node[draw,circle, inner sep=1.5pt,color=red, fill=red] at (0,6){};
   \node [right] at (0,6) {${E_3}$};
   \node [left] at (0,6.5) {$\frac{5}{2}$};
   \node[draw,circle, inner sep=1.5pt,color=black!20!green, fill=black!20!green] at (-2,6){};
   \node [below] at (-2,6) {${C_1}$};
   \node [above] at (-2,6) {$\infty$};   
   \node[draw,circle, inner sep=1.5pt,color=black!20!green, fill=black!20!green] at (-2,4){};
   \node [below] at (-2,4) {${C_2}$};
   \node [above] at (-2,4) {$\infty$}; 
   \node[draw,circle, inner sep=1.5pt,color=black!20!green, fill=black!20!green] at (-2,2){};
   \node [below] at (-2,2) {${C_3}$};
   \node [above] at (-2,2) {$\infty$}; 
   \node[draw,circle, inner sep=1.5pt,color=black, fill=black] at (6,2){};
   \node [below] at (6,2) {${L_2}$};
   \node [above] at (6,2) {$\infty$};
   \node[draw,circle, inner sep=1.5pt,color=red, fill=red] at (2,2){};
   \node [below] at (2.5,2) {${E_4}$};
   \node [above] at (2,2) {$\frac{2}{3}$};
   \node[draw,circle, inner sep=1.5pt,color=red, fill=red] at (4,2){};
   \node [below] at (4.5,2) {${E_5}$};
   \node [above] at (4,2) {$\frac{3}{4}$};
   \node[draw,circle, inner sep=1.5pt,color=black!20!green, fill=black!20!green] at (2,0){};
   \node [right] at (2,0) {${C_5}$};
   \node [left] at (2,0) {$\infty$}; 
   \node[draw,circle, inner sep=1.5pt,color=black!20!green, fill=black!20!green] at (4,0){};
   \node [right] at (4,0) {${C_4}$};
   \node [left] at (4,0) {$\infty$}; 
   \node[draw,circle, inner sep=1.5pt,color=black, fill=black] at (6,8){};
   \node [below] at (6,7.8) {${L_3}$};
   \node [above] at (6,8.2) {$\infty$};
   \node[draw,circle, inner sep=1.5pt,color=red, fill=red] at (4,6){};
   \node [below] at (4,5.8) {${E_7}$};
   \node [left] at (4,6.2) {$\frac{3}{1}$};
   \node[draw,circle, inner sep=1.5pt,color=red, fill=red] at (2,4){};
   \node [right] at (2.2,4) {${E_6}$};
   \node [left] at (2,4.2) {$\frac{5}{3}$};
     \node[draw,circle, inner sep=1.5pt,color=black!20!green, fill=black!20!green] at (6, 6){};
   \node [below] at (6, 6) {${C_6}$};
   \node [above] at (6, 6) {$\infty$};   
   \node[draw,circle, inner sep=1.5pt,color=black, fill=black] at (2,8){};
   \node [right] at (2,8) {${L_4}$};
   \node [left] at (2,8) {$\infty$};
   \node[draw,circle, inner sep=1.5pt,color=red, fill=red] at (2,7){};
   \node [right] at (2,6.5) {${E_8}$};
   \node [left] at (2,6.5) {$\frac{1}{2}$};
   \node[draw,circle, inner sep=1.5pt,color=black!20!green, fill=black!20!green] at (4, 7){};
   \node [below] at (4, 7) {${C_7}$};
   \node [above] at (4, 7) {$\infty$};
 \node [below] at (3, -1) {$\theta_{\pi}(C)$}; 
\end{scope}
\end{tikzpicture}
\end{center}
 \caption{The fan tree $\theta_{\pi}(C)$ is isomorphic to the dual graph of the 
     toroidal boundary $\partial \Sigma$}  
 \label{fig:example-isomfantree} 
     \end{figure} 
     \begin{example}  \label{ex:constrfantree}
      Consider again the toroidal pseudo-resolution process of Example \ref{ex:toroidres}. The 
      construction of the trunks associated to its Newton fans is represented in Figure 
      \ref{fig:example-trunks} for all the crosses at which one applies STEP 2 of the 
      algorithm, that is, for the crosses $(A_i, B_i)$ with $i \in J$.  The remaining crosses  
      are those at which the algorithm stops while executing STEP 1. The corresponding 
      trunks are represented on the bottom line of Figure \ref{fig:example-trunks}. 
      Figure \ref{fig:example-fantree} shows the construction of the fan tree 
      from the previous collection of trunks. In order to make clear the process of gluing 
      of points with the same label, the upper part of the figure shows again the 
      whole collection of trunks, as well as the labels of its marked points. 
\end{example}

The following proposition is an easy consequence of Definition \ref{def:fantreetr}  
and of Proposition \ref{prop:dualtot} (recall that the notion of \emph{dual graph} of an abstract simple normal crossings curve was explained in Definition \ref{def:dualgraph}): 

\begin{proposition}  \label{prop:fantreedualgr}
   The fan tree $\theta_{\pi}(C)$ is isomorphic to the dual graph of 
   the boundary $\partial \Sigma$ of the source of the toroidal pseudo-resolution 
   $\pi: (\Sigma, \partial \Sigma) \to (S, L + L')$ of the curve singularity $C$,  
   by an isomorphism which respects the labels. 
\end{proposition}

\begin{example} \label{ex:dualisom}
   Proposition \ref{prop:fantreedualgr} is illustrated in Figure 
   \ref{fig:example-isomfantree}  with the fan tree of the bottom of Figure 
   \ref{fig:example-fantree}  and the boundary $\partial \Sigma$ of the 
   bottom right of Figure \ref{fig:example-third level}. Both of them correspond 
   to the toroidal pseudo-resolution process of Example \ref{ex:toroidres}. 
   The singular points of $\Sigma$ may be found out from the knowledge of 
   the slope function on the trunks composing the fan tree. 
   Indeed, consider the slopes $\beta/ \alpha$ and $\delta/ \gamma$ 
   of two consecutive vertices of the trunk of one of the Newton fans of the pseudo-resolution process.  
   Then the matrix 
   $ \left( \begin{array}{cc}
                   \alpha & \gamma \\
                    \beta & \delta 
               \end{array}  \right)$  
    is of determinant $\pm 1$ if and only if  the intersection point $o_i$ of the 
    irreducible components of $\partial \Sigma$ which corresponds to this edge 
    is non-singular on $\Sigma$. Moreover, the surface singularity $(\Sigma, o_i)$ is 
    analytically isomorphic to the germ at its orbit of dimension $0$ of the 
    affine toric surface generated by the cone 
    $\cone \langle \alpha \: e_1 + \beta \: e_2, \gamma \:  e_1 + \delta \:  e_2 \rangle$ 
    and the lattice $N = \Z\langle e_1, e_2 \rangle$. As in Figure  \ref{fig:example-third level}, 
    the singular points on $\partial \Sigma$ are indicated by small blue discs. 
    The corresponding edges of the fan trees are represented also in blue. Note that 
    in the previous explanation it was important to say that one has to work with the slope 
    function on the individual trunks, instead of the slope function of the fan tree. For 
    instance,  if one looks at the intersection point of the components $E_1$ and $E_4$, 
    the corresponding slopes are to be read on the trunk $\theta(\fan_{E_1, L_2}(C))$ 
    (they are therefore $0/1$ and $2/3$, and the associated matrix 
    $ \left( \begin{array}{cc}
                   1 & 3 \\
                    0 & 2 
               \end{array}  \right)$  is not unimodular), not on the fan tree $\theta_{\pi}(C)$ 
    (which would give the slopes $3/5$ and $2/3$, whose associated matrix 
    $ \left( \begin{array}{cc}
                   5 & 3 \\
                   3 & 2 
               \end{array}  \right)$ is unimodular). 
\end{example}

\subsection{Historical comments}
\label{ssec:HAtoroidal}
$\:$
\medskip

The oldest method to study a plane curve singularity $C$, imagined  
by Newton around 1665, but published only in 1736 as \cite{N 36}, 
is to express it first in local coordinates $(x,y)$ as the vanishing locus 
of a power series $f(x,y)$ satisfying $f(0,0)=0$ and $f(0,y) \neq 0$, then to 
compute iteratively a formal power series $\eta(x)$ with \emph{rational} positive exponents  
 such that $f(x, \eta(x)) =0$. Whenever $\dfrac{\partial f}{\partial y}(0,0) \neq 0$,    
 there is only one such series $\eta(x)$ which has moreover only integral exponents. 
 This series is simply the Taylor expansion at the origin 
 of the explicit function $y(x)$ whose existence is ensured by the implicit function theorem 
 applied to the function $f(x,y)$ in the neighborhood of $(0,0)$. But, if 
 $\dfrac{\partial f}{\partial y}(0,0) = 0$, then there are at least two such series, their 
 number being equal to the order in $y$ of the series $f(0,y)$. 

As explained on the example studied in Subsection \ref{ssec:redexample}, 
the first step of Newton's iterative method consists in finding the possible leading terms 
$c\:  x^{\alpha}$ of the series $\eta(x)$. His main insight was that if one substitutes 
$y := c\:  x^{\alpha}$ in the series $f(x,y)$, getting a formal power series 
with rational exponents in the variable $x$, then \emph{there are at least two 
terms of this series with minimal exponent, and the sum of all 
such terms vanishes}. This fact has two consequences. First,  there is a finite number 
of possible exponents $\alpha$, which are  the slopes of the rays orthogonal to the
compact edges of the \emph{Newton polygon} of $f(x,y)$. 
Secondly, for a fixed exponent $\alpha_K$ corresponding to the compact edge $K$, there is 
a finite number of values of the leading coefficient $c$, given by the roots of the algebraic equation 
$f_K(x, c \: x^{\alpha_K}) = 0$, where $f_K$ is the restriction of $f$ to $K$ in the sense of 
Definition \ref{def:Npolalg}.

\begin{figure}
  \centering 
 \includegraphics[scale=0.9]{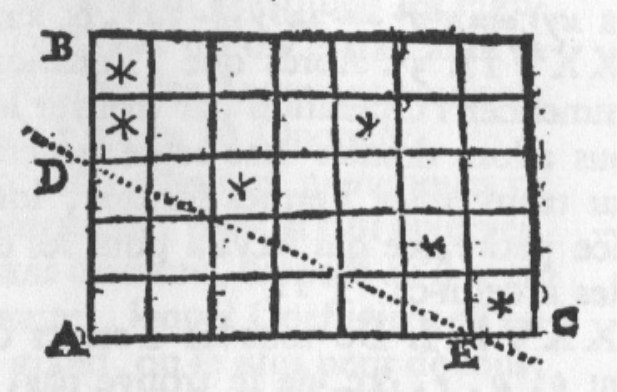} 
  \caption{The first Newton polygon} 
 \label{fig:Newton}
\end{figure} 
 
 \begin{figure}
  \centering 
  \includegraphics[scale=0.9]{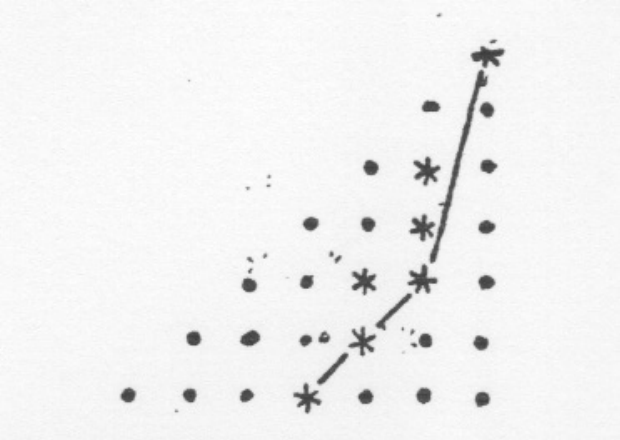} 
  \caption{The compact sides of a Newton polygon, as represented by Cramer} 
  \label{fig:Cramer}
  \end{figure}

Newton's explanations were much developed in Cramer's 1750 
book \cite[Chapter VII]{C 50}, which seems also 
interesting to us in this context for its interpretation of the weights of the variables $x$ and $y$ as 
orders of magnitude for infinitely small quantities. 

Figures \ref{fig:Newton} and \ref{fig:Cramer} 
are extracted from \cite[I, Section XXX]{N 36} and \cite[Section 103]{C 50} respectively. 
The first one  represents the only drawing of Newton polygon in Newton's book. 
Strictly speaking, what we call the \emph{Newton polygon}  
of a series in two variables was not formally introduced in the book. Newton explained  
only how to move a ruler in order to get a first bounded edge of the polygon (see Figure 
\ref{fig:Newtruler}). More details about Newton's and Cramer's ideas on this subject may be found 
in Ghys' 2017 book \cite[Pages 43--68]{G 17}. 

\begin{figure}
 \centering 
 \includegraphics[scale=0.5]{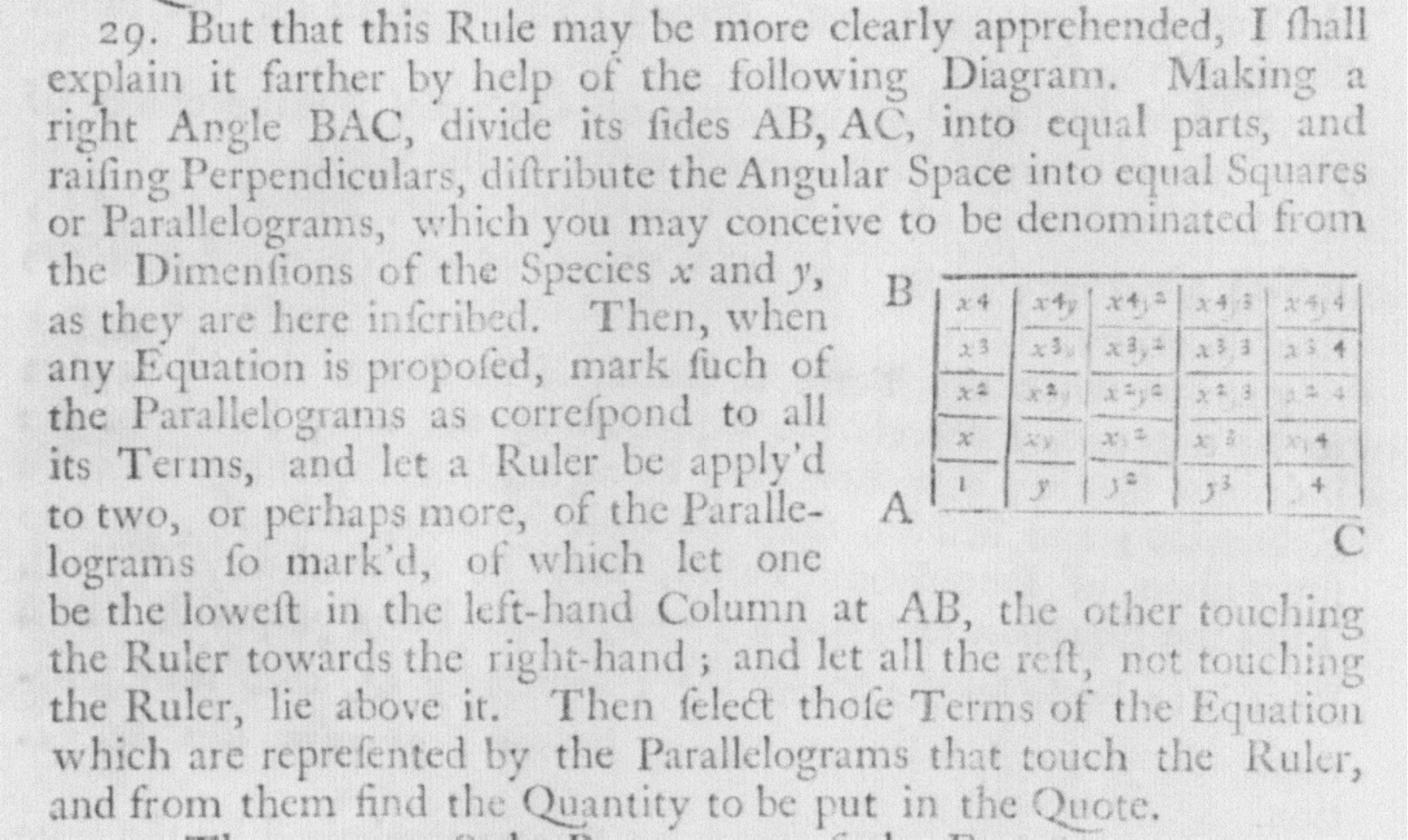} 
 \caption{Newton's ruler} 
 \label{fig:Newtruler}
 \end{figure}

Newton wrote that his procedure may be performed iteratively in order to compute as many terms 
of the series $\eta(x)$ as desired. He also sketched in \cite[Ch. I.LII]{N 36} 
an explanation of the fact that, whenever $f(x,y)$ converges, 
the formal series with rational exponents $\eta(x)$ obtained by continuing forever the procedure also converge and satisfy indeed, all of them, the relation $f(x, \eta(x))=0$. 
But it was Puiseux, in his 1850 paper 
\cite{P 50}, who proved rigorously that one gets indeed as many series as the order 
in $y$ of  $f(0,y)$, that all of them are obtained by substituting some root $x^{1/n}$ 
of the variable $x$ into formal power series with integral exponents, and that those 
formal power series are in fact 
convergent in a neighborhood of the origin. In order to honor his work, 
the formal or convergent power series in a variable $x$ of the form 
$\xi(x^{1/n})$, where $\xi(x)$ is a usual power series and $n \in \N^*$ 
are called nowadays \emph{Puiseux series} or \emph{Newton-Puiseux series}.

Puiseux's approach to the proofs of the existence and the convergence of these series 
avoided the use of roots $x^{1/n}$, by performing 
changes of variables of the form $x = x_1^q$, $y = c_1 x_1^p + y_1$ or of the form
$x = x_1^q$, $y = x_1^p (c_1 + y_1)$, where $c_1$ is a non-zero constant and 
$p/q$ is the irreducible expression of one of the 
exponents $\alpha_K$ given by the Newton polygon of $f$. 
Both changes of variables are compositions of a birational change of variables and 
of the monomial change of variables $x = u^q$, $y = u^p \: v$. This monomial 
change of variables is birational only when $q =1$, that is, when $\alpha_K \in \N^*$. 
Therefore Puiseux's changes of variables are in general not birational. Nevertheless, by  
Lemma \ref{lem:Newton-map} below, such a map can be seen as  
the local analytical expression of a birational map, with respect to 
a particular choice of local coordinates.

Zariski saw this non-birationality as a drawback, and in his 1939 paper \cite{Z 39} 
he introduced alternative changes of variables of the form $x = x_1^q (c_1 + y_1)^{q_1}$, 
$y = x_1^p (c_1 + y_1)^{p_1}$, where $(p_1, q_1) \in \N^* \times \N^*$ and 
$p_1 q - q_1 p =1$. This last condition means that Zariski's changes of variables are  
birational.

\medskip

Let us discuss now the toric approach to the study of plane curve singularities. 
 Note that the changes of variables used by Puiseux and by Zariski are 
compositions of affine morphisms and of toric ones. This fact became clear after 
the development of  toric geometry  (see Subsection \ref{ssec:HAtoric}). 

The systematic study of plane curve singularities using sequences of toric 
modifications began with Mutsuo Oka's 1995-96 papers \cite{LO 95}, \cite{AO 96}, 
  \cite{O 96}, the first one written in collaboration with L\^e and the second one with A'Campo 
  (see also Oka's 1997 book \cite[Ch. III, Sect. 4]{O 97}). 
  Oka gave an introduction to this approach in his 2010 paper \cite{O 10}, through 
  the detailed examination of the case of one branch. 
The second author generalized this approach to \emph{quasi-ordinary} hypersurface 
singularities of arbitrary dimension in his 2003 paper \cite{GP 03} and applied it  
to the study of deformations of real plane curve singularities in the 2010 papers 
\cite{GP  10} and \cite{GPR 10}, the second one written in collaboration with Risler. 

 Also during the 1990s, Pierrette Cassou-Nogu\`es started studying  plane curve singularities 
  using Pui\-seux's non-birational toric morphisms, called \emph{Newton maps}. 
  References to her early works on the subject, 
 done partly in collaboration, may be found in her 2011 paper \cite{CNP 11} with 
    P\l oski, her 2014-15 papers \cite{CNV 14, CNV 15} with Veys, her 2014 
    paper with Libgober \cite{CNL 14} 
    and her 2018 paper with Raibaut  \cite{CNR 18}.   
    
    In his 1997 paper \cite{V 97}, Veys considered the  \emph{log-canonical model} 
    of a plane curve singularity,  obtained by contracting certain exceptional divisors on its minimal 
    embedded resolution, in order to study associated zeta functions. 
  The modification from the log-canonical model to the ambient germ of the plane curve 
  singularity may be seen as a morphism associated with a toroidal pseudo-resolution of this singularity. 
  A toroidal pseudo-resolution algorithm for  plane curve singularities was described by the 
  second author in  \cite[Section 3.4]{GP 03}. A more general algorithm was given by 
  Cassou-Nogu\`es and Libgober in \cite[Section 3]{CNL 14}.  Our  
  Algorithm \ref{alg:tores} of toroidal pseudo-resolution generalizes them, since it 
  does not depend on the choice of special kinds of coordinates.
  
  There are several approaches for the search of \emph{optimal choices} of smooth branches 
in  STEP 2 of Algorithm \ref{alg:tores}. 
\emph{Assume first that $C$ is a branch}, that $f \in \C [[ x ]] [y]$ is the monic polynomial of
degree $n$ defining $C$ in the local coordinate system $(x,y)$ 
and that the line $L = Z(x)$ is transversal to $C$. Let $a$ be a divisor of $n$.
The  \emph{$a$-th approximate root} $h \in  \C[[x]][y]$ of $f$ is the unique 
monic polynomial of degree $a$ 
such that the degree in $y$ of $f- h^{n/a}$ is smaller than $n-a$.  The importance of 
approximate roots for 
the study of plane curve singularities and of the algebraic embeddings of $\C$ in $\C^2$ 
was emphasized by Abhyankar and Moh in their 1973--75 papers \cite{AM 73} and \cite{AM 75}. 
Certain approximate roots of $f$, called \emph{characteristic approximate roots},  
have the property that their strict transforms can be chosen at
 STEP 2 of Algorithm \ref{alg:tores}, providing in this way a toroidal pseudo-resolution of 
 $C$ with the minimal number of Newton modifications. This number is precisely the number of 
 \emph{characteristic exponents} of $C$ with respect to $x$ (see Section \ref{sec:FTEW}). 
 This approach was explained by A'Campo and Oka in their 1996 paper \cite{AO 96}. 

Some properties of the approximate roots may fail when working with 
a base field of positive characteristic.  By contrast, the more general 
combinatorial notion of \emph{semiroot/maximal contact curve} 
can be defined over fields of arbitrary characteristic and plays a similar role   
(see the papers \cite{LJ 72} of Lejeune-Jalabert and  \cite{GBP 15} 
of the first author and P\l oski). For details on applications of approximate roots 
and semiroots to the study of plane curve singularities, see the paper 
\cite{GP 95} of Gwo\'zdziewicz and P\l oski and \cite{PP 03} of the third author. 
Proposition \ref{prop:fantreedualgr} above implies that if  $\pi: (\Sigma, \partial \Sigma) \to (S, L+L')$ 
is a toroidal embedded resolution of $C$ which defines its minimal resolution, 
then the irreducible components of the associated completion $\hat{C}_\pi = \pi ( \partial \Sigma )$  
may be thought as generalizations of the notion of semiroot to plane curve singularities 
with an arbitrary number of branches (see also 
the final comments in Example \ref{ex:continex} below). 

Assume now that $C$ is an arbitrary plane curve singularity. 
The minimal number of Newton modifications involved in the construction of a  
toroidal pseudo-resolution $C$ was characterized by L\^e and Oka in \cite{LO 95} 
in terms of properties of the dual graph of its minimal embedded resolution.

      Another toric approach to the study 
    of plane curve singularities was initiated in Goldin and Teissier's 2000 paper 
    \cite{GT 00}, in the case of branches. 
    They first reembedded in a special way the initial germ of surface in a higher 
    dimensional space, then they resolved the branch by just one toric modification of that space. 
    Their approach 
    was done in the spirit of the philosophy of Teissier's 1973 paper \cite{T 73}, in which he saw 
    all equisingular plane branches as deformations of a single branch 
    of higher embedding dimension, 
    the germ at the origin of their common \emph{monomial curve}. 
    A generalization of some of the results in \cite{GT 00} to the case of 
    quasi-ordinary hypersurface singularities was obtained by the second author in \cite{GP 03}. 
    The theoretical possibility 
    of studying analogously singularities of any dimension was established by Tevelev 
    in his 2014 paper \cite{Tev 14}. See Teissier's comments in \cite[Section 11]{T 14} 
    for more details about his toric approach to the study of singularities. 
    
    The notions of \emph{Newton non-degenerate polynomials} and \emph{series} 
    were introduced by Kouchnirenko in his 1976 paper \cite{K 76}, using the last characterization 
    of Proposition \ref{prop:newtnondeg}. A version of the first characterization was 
    essential in Varchenko's theorem in \cite{V 76} about the monodromy of Newton non-degenerate 
    holomorphic series. Then Khovanskii introduced in \cite{K 77} 
    \emph{Newton non-degenerate complete intersection singularities}, a notion which was 
    much studied by Mutsuo Oka in a series of papers, which were the basis of his 1997 
    monograph \cite{O 97}. 
    Characterizations of Newton non-degenerate singularities, analogous to those of
     Proposition \ref{prop:newtnondeg}, are in fact true for 
   complete intersection singularities 
    (see  Oka's book \cite{O 97}  or Teissier's paper \cite[Section 5]{T 04}). 
  This last paper    contains interesting comments about 
     the evolution of the notion of Newton non-degeneracy, and an extension of 
     it to arbitrary singularities, which are not necessarily complete intersections. 
     This extension was further studied in  Fuensanta Aroca,  G\'{o}mez-Morales 
     and  Shabbir's paper \cite{AGS 05}.

\medskip

     Let us discuss now the notion of \emph{tropicalization} $\mathrm{trop}^f$ introduced 
     in Definition \ref{def:tropicaliz}.
     The union of the rays of the Newton fan $\fan(f)$ which intersect the interior of the 
     regular cone $\sigma_0$ is the 
     \emph{tropical zero-locus} of the function $\mathrm{trop}^f$, as defined in 
     tropical geometry, that is, the locus of non-differentiability of the continuous piecewise linear 
     function $\mathrm{trop}^f$. It is also part 
     of the \emph{local tropicalization} of the zero locus $Z(f)\hookrightarrow (\C^2,0)$ 
     of $f$, as defined by Stepanov and the third author  in \cite{PS 13} for complex 
     analytic singularities of arbitrary dimension embedded in germs of affine toric varieties. 
     The local tropicalization contains also portions at infinity, in a partial compactification 
     of the cone defining the ambient toric variety, in order to keep track of the intersections 
     of the singularity with all the toric orbits. 
     
     A precursor of the notion of local tropicalisation 
     was introduced under the name of ``\emph{tropism of an ideal}'' 
     by Maurer in his 1980 article \cite{M 80}, which was unknown to the authors  
     of \cite{PS 13} when they wrote that paper. In our case, the tropism of the ideal 
     $(f) \subseteq \C[[x,y]]$ is the set of lattice points lying on the rays of $\fan(f)$ 
     which are different from the edges of the cone $\sigma_0$. The term ``\emph{tropism}'' 
     had been used before by Lejeune-Jalabert and Teissier in their 1973 paper \cite{LJT 73}, 
     in the expression ``\emph{tropisme critique}''. They saw this notion as a measure of anisotropy, 
     as explained by Teissier in \cite[Footnote to Sect. 1]{JMM 08}: 
     
     \begin{quote}
     ``\emph{ {\small As far as I know the term 
     did not exist before. We tried to convey the idea that giving different weights to some variables 
     made the space ``anisotropic'', and we were intrigued by the structure, for example, of 
     anisotropic projective spaces (which are nowadays called weighted projective spaces). 
     From there to ``tropismes critiques'' was a quite natural linguistic movement. Of course 
     there was no ``tropical'' idea around, but as you say, it is an amusing coincidence. The 
     Greek ``Tropos'' usually designates change, so that ``tropisme critique'' is well adapted 
     to denote the values where the change of weights becomes critical for the computation 
     of the initial ideal. The term ``Isotropic'', apparently due to Cauchy, refers to the property 
     of presenting the same (physical) characters in all directions. Anisotropic is, of course, its 
     negation. The name of Tropical geometry originates, as you probably know, from tropical 
     algebra which honours the Brazilian computer scientist Imre Simon living close to the tropics, 
     where the course of the sun changes back to the equator. In a way the tropics of Capricorn 
     and Cancer represent, for the sun, critical tropisms.}}''
  \end{quote}

\section{Lotuses}
\label{sec:embres}
\medskip

Throughout this section, we will assume that $C$ is \emph{reduced}. 
We explain the notion of \emph{Newton lotus} (see Definition \ref{def:deflot}), 
its relation with continued fractions (see Subsection \ref{ssec:lotcf}) and how 
to construct a more general \emph{lotus} from the fan tree of a toroidal 
pseudo-resolution process (see Definition \ref{def:lotustoroid}). It is a 
special type of simplicial complex of dimension $2$, built from the Newton lotuses 
associated with the Newton fans generated by the process, by gluing them in the same way 
one glued the corresponding trunks into the fan tree. It allows to visualize the combinatorics 
of the decomposition of the embedded resolution morphism into point blow ups, as well as 
the associated Enriques diagram and the final dual graphs (see Theorem \ref{thm:repsailtor}). 
We show by two examples that its structure depends on the choice 
of auxiliary curves introduced each time one executes STEP 2 of Algorithm \ref{alg:tores}, 
that is, on the choice of completion $\hat{C}_{\pi}$ of $C$ (see Subsection \ref{ssec:deplotus}). 
In Subsection \ref{ssec:trunclot}  we define an operation of \emph{truncation} of the lotus of a toroidal 
pseudo-resolution and we explain some of its uses. In the final Subsection \ref{ssec:HAembres} 
we give historical information about other works in which appeared objects similar to the notion 
of lotus.



\subsection{The lotus of a Newton fan}
\label{ssec:lotnf}
$\:$
\medskip

In this subsection, whose content is very similar to that of \cite[Section 5]{PP 11}, 
we give a first level of explanation of the subtitle of this article, 
a second level being described in Subsection \ref{ssec:sailtores}. 
Namely, we introduce the notion of \emph{lotus} $\Lambda(\fan)$ 
of a Newton fan $\fan$ (see Definition \ref{def:deflot}). 
If the fan originates from a Newton polygon $\cN(f)$, 
that is, if $\fan = \fan(f)$ (see Definition \ref{def:nfan}), 
we imagine $\Lambda(\fan)$  as a \emph{blossoming} of $\cN(f)$. 
The lotus of a Newton fan $\fan$ allows to understand 
visually the decomposition into blow ups of the toric modification 
defined by the regularized fan $\fan^{reg}$. For instance, 
the dual graph of the final exceptional divisor, 
the Enriques diagram and the graph of the proximity relation of the 
associated constellation embed naturally in it, as subcomplexes of its $1$-skeleton 
(see Propositions \ref{prop:dualevolution},    \ref{prop:lotusinterprenriques} and  \ref{prop:lotusinterprproxim}).  
\medskip

\begin{figure}[h!]
    \begin{center}
\begin{tikzpicture}[scale=0.8]
   \draw [dashed, gray] (0,0) grid (3,3);
\node [below] at (2.5,0) {$e_{1}     \leftarrow \hbox{\rm basic vertex}$};
\node [left] at (0,1) {$\hbox{\rm basic vertex} \rightarrow e_{2}$};
\node [left] at (0,0) {$0$};

\draw [fill=pink](1,0) -- (0,1)--(1,1)--cycle;
\draw [->, very thick, red] (1,0)--(0.5, 0.5);
\draw [-, very thick, red] (0.5, 0.5)--(0,1);

\draw[->][thick, color=black](1,-1) .. controls (0.5,-0.5) ..(0.7,0.7);
\node [below] at (1,-1) {$\delta(e_{1},e_{2})$};

\draw[->][thick, color=black](1.9,1.2) .. controls (1.5,0.5) ..(1,0.5);
\node [right] at (1.8,1.2) {$\hbox{\rm lateral edge}$};
\draw[->][thick, color=black](1.1,2.1) .. controls (0.5,1.5) ..(0.5,1);
\node [right] at (1.1,2.1) {$\hbox{\rm lateral edge}$};


\begin{scope}[shift={(7,0)},scale=1]
   \draw [dashed, gray] (0,0) grid (3,3);
\node [below] at (1,0) {$e_{1}$};
\node [left] at (0,1) {$e_{2}$};
\node [left] at (0,0) {$0$};

\draw [fill=pink](1,0) -- (0,1)--(1,1)--cycle;
\draw [fill=pink!40](1,0) -- (1,1)--(2,1)--cycle;
\draw [fill=pink!40](0,1) -- (1,1)--(1,2)--cycle;
\draw [->, very thick, red] (1,0)--(0.5, 0.5);
\draw [-, very thick, red] (0.5, 0.5)--(0,1);

\draw[->][thick, color=black](1,-1) .. controls (0.5,-0.5) ..(0.7,0.7);
\node [below] at (1,-1) {$\delta(e_{1},e_{2})$};

\draw[->][thick, color=black](2.1,1.5) .. controls (1.5,1.5) ..(1.2,0.7);
\node [right] at (2.1,1.5) {$\delta(e_{1},e_{1}+e_{2})$};

\draw[->][thick, color=black](1.1,2.3) .. controls (0.5,1.7) ..(0.5,1.2);
\node [right] at (1.1,2.3){$\delta(e_{1}+e_{2}, e_2)$};
\end{scope}
  \end{tikzpicture}
\end{center}
 \caption{Vocabulary and notations about petals}
 \label{fig:Vocapetals}
   \end{figure}
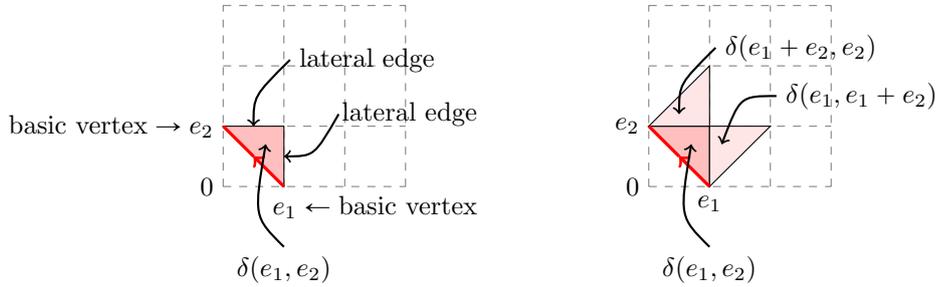

Lotuses are built from \emph{petals}, which are triangles with supplementary structure 
(see Figure \ref{fig:Vocapetals}):   

\begin{definition} \label{def:petal}
       Let $N$ be a 2-dimensional lattice and let $(e_1, e_2)$ be a basis of it.
      Denote by $\boxed{\delta(e_1, e_2)}$ the  convex and  compact triangle 
       with vertices $e_1, e_2, e_1 + e_2$, contained in the real plane $N_{\R}$. 
      It is the {\bf petal associated with the basis} \index{petal} $(e_1, e_2)$. Its {\bf base} is the 
      segment $[e_1, e_2]$, oriented from $e_1$ to $e_2$. The points $e_1$ and $e_2$ 
      are called the {\bf basic vertices} \index{vertex!basic, of a petal} of the petal.  
      Its  {\bf lateral edges} \index{edge!lateral, of a petal} 
      are the segments $[e_i, e_1 + e_2]$, for each $i \in \{1, 2\}$.  
    \end{definition}

Once the petal $\delta(e_1, e_1 + e_2)$ is constructed, 
the construction may be repeated starting from each one of the bases 
$(e_1, e_1 +e_2)$ and $(e_1 +e_2,e_2)$ of $N$, getting two new petals 
$\delta(e_1, e_1 + e_2)$ and $\delta(e_1 + e_2, e_2)$, and so on.  Note that the bases 
produced by this process are ordered such as to define always the same orientation 
of the real plane $N_{\R}$ -- we say that they are {\bf positive bases}.  
In this way, one progressively constructs an infinite simplicial complex embedded in the cone 
 $ \sigma_0$: at the $n$-th  step, one adds 
$2^n$ petals to those already constructed. Each petal, with the exception of the first one 
$\delta(e_1, e_2)$, has a common edge -- its base -- 
with exactly one of the petals constructed at the previous step, called its {\bf parent}.  
\index{petal!parent}

The pairs of vectors $(f_1, f_2) \in N^2$ which appear as bases of petals 
$\delta(f_1, f_2)$ during the previous process may be characterized in the following way 
(see \cite[Remarque 5.1]{PP 11}): 

\begin{lemma}  \label{lem:base}
     A segment $[f_1, f_2]$, oriented from $f_1$ to $f_2$, 
     is the base of a petal $\delta(f_1, f_2)$  constructed during the previous process 
     if and only if $(f_1, f_2)$ is a positive basis of the lattice $N$ contained in the cone 
     $\sigma_0$. 
     Said differently, if a positive basis $(f_1, f_2)$ of $N$ is contained in the cone 
     $\sigma_0$ and is different from $(e_1, e_2)$, then there exists a unique permutation 
     $(i,j)$ of $(1,2)$ such that $f_j - f_i \in \sigma_0 \cap N$.
\end{lemma} 

We are ready to define the simplest kinds of lotuses:

\begin{definition} \label{def:univlot}
     The simplicial complex obtained as the union of all the petals  
    constructed by the previous process starting from the basis $(e_1, e_2)$ 
    of $N$, is called {\bf the universal lotus $\boxed{\Lambda(e_1, e_2)}$ relative to $(e_1, e_2)$} 
    \index{lotus!universal} (see Figure \ref{fig:Lotus}).        
     A {\bf lotus $\Lambda$ relative to} \index{lotus} $(e_1, e_2)$ is either the segment $[e_1, e_2]$ 
     or the union of a non-empty set of petals of the universal lotus 
     $\Lambda(e_1, e_2)$, stable under the operation of taking the parent of a petal.  
     The segment $[e_1, e_2]$ is called the {\bf base} \index{base!of a lotus} of $\Lambda$. If 
     $\Lambda$ is of dimension $2$, then the petal $\delta(e_1, e_2)$ is 
     called its {\bf base petal}. \index{petal!base}
    The point $e_1$ is called the {\bf first basic vertex} \index{vertex!basic, of a lotus}
    and $e_2$  the {\bf second basic vertex}    
    of the lotus. The lotus is oriented by restricting to it the orientation of 
   $N_{\R}$ induced by the basis $(e_1, e_2)$. 
 \end{definition}
 
 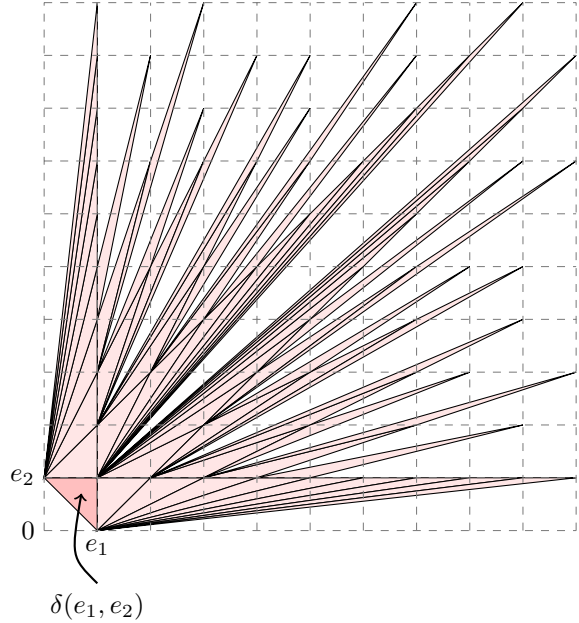
\begin{figure}[h!]
     \begin{center}
\begin{tikzpicture}[scale=0.7]

\draw [fill=pink](1,0) -- (0,1)--(1,1)--cycle;
\draw [fill=pink!40](1,0) -- (1,1)--(2,1)--cycle;
\draw [fill=pink!40](1,0) -- (2,1)--(3,1)--cycle;
\draw [fill=pink!40](1,0) -- (3,1)--(4,1)--cycle;
\draw [fill=pink!40](1,0) -- (4,1)--(5,1)--cycle;
\draw [fill=pink!40](1,0) -- (5,1)--(6,1)--cycle;
\draw [fill=pink!40](1,0) -- (6,1)--(7,1)--cycle;
\draw [fill=pink!40](1,0) -- (7,1)--(8,1)--cycle;
\draw [fill=pink!40](1,0) -- (8,1)--(9,1)--cycle;
\draw [fill=pink!40](1,0) -- (9,1)--(10,1)--cycle;

\draw [fill=pink!40](1,1) -- (2,1)--(3,2)--cycle;
\draw [fill=pink!40](1,1) -- (3,2)--(4,3)--cycle;
\draw [fill=pink!40](1,1) -- (4,3)--(5,4)--cycle;
\draw [fill=pink!40](1,1) -- (5,4)--(6,5)--cycle;
\draw [fill=pink!40](1,1) -- (6,5)--(7,6)--cycle;
\draw [fill=pink!40](1,1) -- (7,6)--(8,7)--cycle;
\draw [fill=pink!40](1,1) -- (8,7)--(9,8)--cycle;
\draw [fill=pink!40](1,1) -- (9,8)--(10,9)--cycle;

\draw [fill=pink!40](2,1) -- (3,2)--(5,3)--cycle;
\draw [fill=pink!40](2,1) -- (5,3)--(7,4)--cycle;
\draw [fill=pink!40](2,1) -- (7,4)--(9,5)--cycle;

\draw [fill=pink!40](2,1) -- (3,1)--(5,2)--cycle;
\draw [fill=pink!40](2,1) -- (5,2)--(7,3)--cycle;
\draw [fill=pink!40](2,1) -- (7,3)--(9,4)--cycle;

\draw [fill=pink!40](3,2) -- (5,3)--(8,5)--cycle;

\draw [fill=pink!40](3,2) -- (4,3)--(7,5)--cycle;
\draw [fill=pink!40](3,2) -- (7,5)--(10,7)--cycle;

\draw [fill=pink!40](4,3) -- (5,4)--(9,7)--cycle;

\draw [fill=pink!40](3,1) -- (5,2)--(8,3)--cycle;
\draw [fill=pink!40](3,1) -- (4,1)--(7,2)--cycle;
\draw [fill=pink!40](3,1) -- (7,2)--(10,3)--cycle;

\draw [fill=pink!40](4,1) -- (5,1)--(9,2)--cycle;

\draw [fill=pink!40](0,1) -- (1,1)--(1,2)--cycle;
\draw [fill=pink!40](0,1) -- (1,2)--(1,3)--cycle;
\draw [fill=pink!40](0,1) -- (1,3)--(1,4)--cycle;
\draw [fill=pink!40](0,1) -- (1,4)--(1,5)--cycle;
\draw [fill=pink!40](0,1) -- (1,5)--(1,6)--cycle;
\draw [fill=pink!40](0,1) -- (1,6)--(1,7)--cycle;
\draw [fill=pink!40](0,1) -- (1,7)--(1,8)--cycle;
\draw [fill=pink!40](0,1) -- (1,8)--(1,9)--cycle;
\draw [fill=pink!40](0,1) -- (1,9)--(1,10)--cycle;

\draw [fill=pink!40](1,1) -- (1,2)--(2,3)--cycle;
\draw [fill=pink!40](1,1) -- (2,3)--(3,4)--cycle;
\draw [fill=pink!40](1,1) -- (3,4)--(4,5)--cycle;
\draw [fill=pink!40](1,1) -- (4,5)--(5,6)--cycle;
\draw [fill=pink!40](1,1) -- (5,6)--(6,7)--cycle;
\draw [fill=pink!40](1,1) -- (6,7)--(7,8)--cycle;
\draw [fill=pink!40](1,1) -- (7,8)--(8,9)--cycle;
\draw [fill=pink!40](1,1) -- (8,9)--(9,10)--cycle;

\draw [fill=pink!40](1,2) -- (2,3)--(3,5)--cycle;
\draw [fill=pink!40](1,2) -- (3,5)--(4,7)--cycle;
\draw [fill=pink!40](1,2) -- (4,7)--(5,9)--cycle;

\draw [fill=pink!40](1,2) -- (1,3)--(2,5)--cycle;
\draw [fill=pink!40](1,2) -- (2,5)--(3,7)--cycle;
\draw [fill=pink!40](1,2) -- (3,7)--(4,9)--cycle;

\draw [fill=pink!40](2,3) -- (3,5)--(5,8)--cycle;

\draw [fill=pink!40](2,3) -- (3,4)--(5,7)--cycle;
\draw [fill=pink!40](2,3) -- (5,7)--(7,10)--cycle;

\draw [fill=pink!40](3,4) -- (4,5)--(7,9)--cycle;

\draw [fill=pink!40](1,3) -- (2,5)--(3,8)--cycle;
\draw [fill=pink!40](1,3) -- (1,4)--(2,7)--cycle;
\draw [fill=pink!40](1,3) -- (2,7)--(3,10)--cycle;

\draw [fill=pink!40](1,4) -- (1,5)--(2,9)--cycle;

\draw [dashed, gray] (0,0) grid (10,10);
\node [below] at (1,0) {$e_{1}$}; 
\node [left] at (0,1) {$e_{2}$}; 
\node [left] at (0,0) {$0$}; 

\draw[->][thick, color=black](1,-1) .. controls (0.5,-0.5) ..(0.7,0.7);  
\node [below] at (1,-1) {$\delta(e_{1},e_{2})$}; 
   \end{tikzpicture}
\end{center}
  \caption{Partial view of the universal lotus $\Lambda(e_{1},e_{2})$ relative to $(e_1, e_2)$}  
  \label{fig:Lotus} 
    \end{figure}

 A lotus may be associated with any set $\mathcal{E} \subseteq [0, \infty]$ or with any 
 Newton fan:

 \begin{definition}  \label{def:deflot} 
      Let $N$ be a lattice of rank $2$, endowed with a basis $(e_1, e_2)$. 
        
        \noindent
        $\bullet$  If $\lambda \in  (0, \infty)$, then its  {\bf lotus}, denoted 
                  $\boxed{\Lambda(\lambda)}$, is the union of petals of  the universal lotus  
                   $\Lambda(e_1, e_2)$ whose interiors intersect the ray of slope $\lambda$.  
             If $\lambda \in \{ 0, \infty \}$, then its lotus $\Lambda(\lambda)$ is just 
                      $[e_1, e_2]$.

             \noindent
             $\bullet$ 
             If $\cE \subseteq  [0, \infty]$, then its {\bf lotus} $ \boxed{\Lambda (\cE)}$ 
                   is  the union $\bigcup_{\lambda \in \cE} \Lambda( \lambda)$ of the 
                   lotuses of its elements.      \index{lotus!of a set}
             
             \noindent
             $\bullet$ 
              If $\fan$ is a Newton fan and $\fan  = \fan(\cE)$  in the sense of 
                   Definition \ref{def:fan2}, we say that $\boxed{\Lambda(\fan)} := 
                     \Lambda(\cE)$ is the {\bf lotus of the fan $\fan$}.  \index{lotus!of a Newton fan}

              \noindent
             $\bullet$
              A {\bf Newton lotus}  is the lotus of a Newton fan. \index{lotus!Newton} \index{Newton!lotus}
                   That is, it is a lotus relative to $(e_1, e_2)$ with a finite number of petals. 
\end{definition}

We could have called the lotuses relative to $(e_1, e_2)$ \emph{finite lotuses} 
instead of \emph{Newton lotuses}. We chose the second terminology because 
in Definition \ref{def:lotustoroid} below we will introduce a more general kind of lotuses
with a finite number of petals, and we want to distinguish the class of 
lotuses of Newton fans inside that more general class of lotuses. 

 A lotus $\Lambda(\cE)$, for $\cE \subseteq  [0, \infty]$, 
is a Newton lotus if and only $\mathcal{E}$ is a finite set of non-negative rational numbers. 
Note that, as illustrated for instance by Example \ref{ex:lotex} below, the  
structure of the lotus $\Lambda(\cE)$ 
does not allow to reconstruct the initial set $\cE$. For this reason, 
we enrich $\Lambda(\cE)$ with several \emph{marked} points, whose knowledge allows to 
reconstruct $\cE$ unambiguously:

\begin{definition} \label{def:lotus-point} 
Fix a Newton lotus $\Lambda$. 

  \noindent
             $\bullet$
If $\Lambda \neq [e_1, e_2]$, we denote by  $\boxed{\partial_+ \Lambda}$  
             the compact and connected polygonal line defined as the 
            complement  of the open segment $(e_1, e_2)$ in the boundary of the lotus    
            $\Lambda$. If $\Lambda = [e_1, e_2]$,  we set $\partial_+ \Lambda  :=[e_1, e_2]$.
         The polygonal line $\partial_+ \Lambda  \subseteq \Lambda$ is  called 
         the {\bf lateral boundary}  of the lotus $\Lambda$. 
          \index{boundary!lateral}  
      
       \noindent
             $\bullet$
 We denote by  $\boxed{p_{\Lambda} }$ the homeomorphism 
           $p_{\Lambda} : [0, \infty]  \to  \partial_+ \Lambda $ 
           which associates with any $\lambda \in [0, \infty] $ 
           the unique point $p_{\Lambda} (\lambda)  \in \partial_+ \Lambda$ of slope $\lambda$. 
          If $\Lambda = \Lambda(\cE)$ where $\cE \subseteq \Q_+ \cup \{\infty\}$ is finite 
            and $\lambda \in \cE$, then we call 
          $p_{\Lambda(\cE)} (\lambda)$ the {\bf marked point of} $\lambda$  
          \index{point!marked, of a lotus} 
          (or of the ray of slope $\lambda$) 
             in the lotus $\Lambda(\cE)$. We consider $\Lambda(\cE)$ as a {\bf marked lotus} 
             \index{lotus!marked}
           using those marked points. 
 \end{definition}
 
 \begin{remark}
      Notice that  if $\lambda \in \mathcal{E}$,  then $p_{\Lambda(\cE)} (\lambda) $
      is by construction the unique  primitive element $p( \lambda)$ 
      of the lattice $N$, which has slope $\lambda$ relative to the basis $(e_1, e_2)$. 
      Therefore, it is independent of  the remaining elements of the set $\mathcal{E}$.
\end{remark}

\medskip

 We distinguish also by geometric properties several vertices  of a Newton lotus: 
 
 \begin{definition} \label{def:defpinch} 
    Assume that $\Lambda$ is a Newton lotus.  A vertex of $\Lambda$ 
    different from $e_1$ and $e_2$  is called a {\bf pinching point} \index{point!pinching}
      of the lotus $\Lambda$ if it belongs to a unique petal of it.
      If the lotus $\Lambda$ is two-dimensional, then the lattice point which is connected to $e_2$     
      (resp. to $e_1$) inside the lateral boundary $\partial_+ \Lambda$ of 
      $\Lambda$ is called the {\bf last interior point} (resp.  {\bf first interior point}) 
      \index{point!first interior} \index{point!last interior}
      of the lateral boundary.
\end{definition}

\begin{remark} \label{pinch-lotus}
    The pinching points of a Newton lotus $\Lambda(\cE)$ are part of its marked points. 
    Two Newton lotuses $\Lambda(\cE_1)$ and $\Lambda(\cE_2)$ coincide as 
    unmarked simplicial complexes if and only if their sets of pinching points coincide. 
\end{remark}

\begin{figure}[h!]
    \begin{center}
\begin{tikzpicture}[scale=0.6]
\draw [->](0,0) -- (0,6);
\draw [->](0,0) -- (6,0);

\draw[fill=pink!40](1,0) -- (0,1) -- (1,1)  --cycle;
\draw[fill=pink!40](1,0) -- (1,1) -- (2,1) --cycle;
\draw[fill=pink!40](1,1) -- (2,1) -- (3,2) --cycle;
\draw[fill=pink!40](2,1) -- (3,2) -- (5,3) --cycle;

\draw [-, ultra thick, color=orange](0,1) --  (1,1) -- (5,3) -- (2,1) --(1,0);
\foreach \x in {0,1,...,5}{
\foreach \y in {0,1,...,5}{
      \node[draw,circle,inner sep=0.7pt,fill, color=gray!40] at (1*\x,1*\y) {}; }
  }
\node[draw,circle, inner sep=1.5pt,color=red, fill=red] at (5,3){};
\node [left] at (6.5,3) {$p(\frac{3}{5})$};

\node [left] at (1.3,-0.3) {$e_1$};
\node [left] at (0,1) {$e_2$};
\draw [->, very thick, red] (1,0)--(0.5, 0.5);
\draw [-, very thick, red] (0.5, 0.5)--(0,1);


\begin{scope}[shift={(8,0)},scale=1]
\draw [->](0,0) -- (0,6);
\draw [->](0,0) -- (6,0);

\draw[fill=pink!40](1,0) -- (0,1) -- (1,1)  --cycle;
\draw[fill=pink!40](1,1) -- (1,2) -- (0,1) --cycle;
\draw[fill=pink!40](1,2) -- (1,3) -- (0,1) --cycle;
\draw[fill=pink!40](1,2) -- (1,3) -- (2,5) --cycle;
\draw[fill=pink!40](1,0) -- (1,1) -- (2,1) --cycle;
\draw[fill=pink!40](1,1) -- (2,1) -- (3,2) --cycle;
\draw[fill=pink!40](2,1) -- (3,2) -- (5,3) --cycle;

\draw [-, ultra thick, color=orange](0,1) -- (2,5) -- (1,2) -- (1,1) -- (5,3) -- (2,1) --(1,0);

\foreach \x in {0,1,...,5}{
\foreach \y in {0,1,...,5}{
      \node[draw,circle,inner sep=0.7pt,fill, color=gray!40] at (1*\x,1*\y) {}; }
  }

\node[draw,circle, inner sep=1.5pt,color=red, fill=red] at (2,5){};
\node[draw,circle, inner sep=1.5pt,color=red, fill=red] at (1,2){};
\node[draw,circle, inner sep=1.5pt,color=red, fill=red] at (5,3){};

\node [left] at (2.4,5.5) {$p(\frac{5}{2})$};
\node [left] at (2.5,2) {$p(\frac{2}{1})$};
\node [left] at (6.5,3) {$p(\frac{3}{5})$};

\node [left] at (1.3,-0.3) {$e_1$};
\node [left] at (0,1) {$e_2$};
\draw [->, very thick, red] (1,0)--(0.5, 0.5);
\draw [-, very thick, red] (0.5, 0.5)--(0,1);
\end{scope}

\end{tikzpicture}
\end{center}
\caption{The Newton lotuses  $\Lambda\left(3/5\right)$,  
     $\Lambda\left(3/5, 2/1, 5/2\right)$ and their marked points}
     \label{fig:exlot}
   \end{figure}
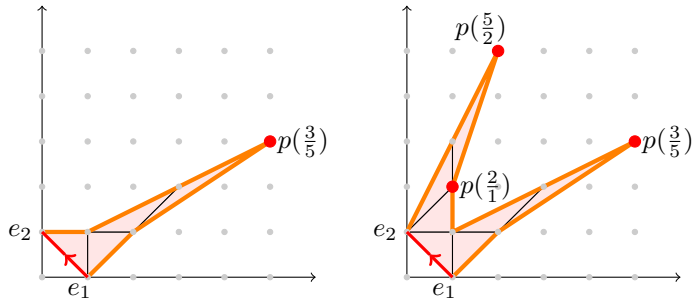

\begin{example} \label{ex:lotex}
  In Figure \ref{fig:exlot} are represented the lotuses 
  $\Lambda\left(3/5\right)$ and $\Lambda(\mathcal{E})$, 
  where $\mathcal{E} = \left\{ 3/5, 2/1, 5/2 \right\}$ is the set 
  whose fan $\fan(\cE)$ was drawn in Figure \ref{fig:examfan}. 
  The lotus $\Lambda\left(3/5\right)$ has only one pinching point, 
  which is $p\left(3/5\right)$. 
  The pinching points of $\Lambda(\mathcal{E})$ are $p\left(3/5\right)$ and $p\left(5/2\right)$. 
  Its marked points are $p\left(3/5\right)$,  $p\left(2/1\right)$ and 
  $p\left(5/2\right)$. This differentiates it from the lotus 
  $\Lambda(3/5,  5/2 ) := \Lambda(\left\{ 3/5,  5/2 \right\})$, which is the same simplicial 
  complex if one forgets their respective marked points. The first interior point of $\Lambda(\mathcal{E})$ 
  is $p\left( 1/2 \right)$ and its last interior point is $p\left(3/1\right)$.
  
  \begin{figure}[h!]
     \begin{center}
\begin{tikzpicture}[scale=0.6]

\begin{scope}[shift={(0,0)},scale=1]
\foreach \x in {0,1,...,5}{
\foreach \y in {0,1,...,5}{
       \node[draw,circle,inner sep=0.7pt,fill, color=gray!40] at (1*\x,1*\y) {}; }
   }
\draw [->](0,0) -- (0,6);
\draw [->](0,0) -- (6,0);
\draw [-, ultra thick, color=orange](0,1) -- (2,5) -- (1,2) -- (1,1) -- (5,3) -- (2,1) --(1,0);
\draw [-, thick, color=blue](0,0) -- (2.3,5.75);
\draw [-, thick, color=blue](0,0) -- (2.8,5.6);
\draw [-, thick, color=blue](0,0) -- (5.5,3.3);
\draw [-,  color=black!20!green](0,0) -- (2,6);
\draw [-,  color=black!20!green](0,0) -- (5.5,5.5);
\draw [-,  color=black!20!green](0,0) -- (6,4);
\draw [-,  color=black!20!green](0,0) -- (6,3);

\node[draw,circle, inner sep=1.5pt,color=black, fill=black] at (1,0){};
\node[draw,circle, inner sep=1.5pt,color=black, fill=black] at (0,1){};
\node[draw,circle, inner sep=1.5pt,color=black, fill=black] at (0,0){};
\node[draw,circle, inner sep=1.8pt,color=red, fill=red] at (2,5){};
\node[draw,circle, inner sep=1.8pt,color=red, fill=red] at (1,2){};
\node[draw,circle, inner sep=1.8pt,color=red, fill=red] at (5,3){};

\node [left] at (1.3,-0.3) {$e_1$};
\node [left] at (0,1) {$e_2$};
\end{scope}


\begin{scope}[shift={(9,0)},scale=1]
\draw [->](0,0) -- (0,6);
\draw [->](0,0) -- (6,0);

\draw[fill=pink!40](1,0) -- (0,1) -- (1,1)  --cycle;
\draw[fill=pink!40](1,1) -- (1,2) -- (0,1) --cycle;
\draw[fill=pink!40](1,2) -- (1,3) -- (0,1) --cycle;
\draw[fill=pink!40](1,2) -- (1,3) -- (2,5) --cycle;
\draw[fill=pink!40](1,0) -- (1,1) -- (2,1) --cycle;
\draw[fill=pink!40](1,1) -- (2,1) -- (3,2) --cycle;
\draw[fill=pink!40](2,1) -- (3,2) -- (5,3) --cycle;

\draw [-, ultra thick, color=orange](0,1) -- (2,5) -- (1,2) -- (1,1) -- (5,3) -- (2,1) --(1,0);

\foreach \x in {0,1,...,5}{
\foreach \y in {0,1,...,5}{
       \node[draw,circle,inner sep=0.7pt,fill, color=gray!40] at (1*\x,1*\y) {}; }
   }

\node[draw,circle, inner sep=1.5pt,color=red, fill=red] at (2,5){};
\node[draw,circle, inner sep=1.5pt,color=red, fill=red] at (1,2){};
\node[draw,circle, inner sep=1.5pt,color=red, fill=red] at (5,3){};

\node [left] at (2.4,5.5) {$p(\frac{5}{2})$};
\node [left] at (2.5,2) {$p(\frac{2}{1})$};
\node [left] at (6.5,3) {$p(\frac{3}{5})$};

\node [left] at (1.3,-0.3) {$e_1$};
\node [left] at (0,1) {$e_2$};
\draw [->, very thick, red] (1,0)--(0.5, 0.5);
\draw [-, very thick, red] (0.5, 0.5)--(0,1);

\end{scope}

\end{tikzpicture}
\end{center}
\caption{The regularized fan $\fan^{reg}\left(3/5, 2/1, 5/2\right)$ 
     and the Newton lotus $\Lambda\left(3/5, 2/1, 5/2\right)$}  
      \label{fig:examfanlotus}
    \end{figure}
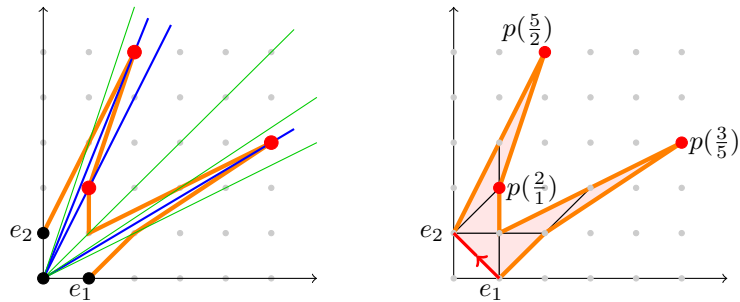

  By comparing Figures \ref{fig:exlot} and \ref{fig:examfanreg}, which we combined 
    in Figure \ref{fig:examfanlotus}, 
    one sees that the lateral boundary of the lotus 
    $\Lambda(3/5, 2/1, 5/2 )$ is exactly the polygonal line constructed 
    when one performed the regularization of the fan $\fan(3/5, 2/1, 5/2 )$ 
    (see Proposition \ref{prop:regconv}). This is a general 
    phenomenon, as shown by the following proposition. 
\end{example}

\begin{proposition}  \label{prop:lotusdecomp}
    Let $\fan$ be a fan subdividing the cone $\sigma_0$. Then the regularization 
    $\fan^{reg}$ of $\fan$ is obtained by subdividing $\sigma_0$ using the rays generated 
    by all the lattice points lying along the lateral boundary $\partial_+ \Lambda(\cF)$ of the lotus 
    $\Lambda(\fan)$. 
\end{proposition}

   \begin{proof}
        Consider two successive marked points $p(\lambda)$ and $p(\mu)$ of the lateral 
        boundary $\partial_+ \Lambda(\cF)$. They are primitive elements of the ambient lattice 
        $N$. Denote by $p(\lambda) + \R_+ p(\lambda) $ 
        the closed half line originating from the point $p(\lambda)$ and generated by the 
        vector $p(\lambda)$. Consider analogously the half-line $p(\mu) + \R_+ p(\mu)$. 
        Let $P(\lambda, \mu)$ be the polygonal line joining the points $p(\lambda)$ and 
        $p(\mu)$ inside $\partial_+ \Lambda(\cF)$. Consider the union of the three previous 
        polygonal lines: 
        $Q(\lambda, \mu)  := \left( p(\lambda) + \R_+ p(\lambda) \right) \cup P(\lambda, \mu)  
            \cup   \left( p(\mu) + \R_+ p(\mu) \right)$.
         
        As the pinching points of $\Lambda(\cF)$ belong to the marked points, this shows that 
        there are no pinching points in the interior of the polygonal line $P(\lambda, \mu)$. 
        Therefore, $Q(\lambda, \mu)$ is the boundary of a closed convex set 
        $\hat{Q}(\lambda, \mu)$ contained in the cone $\R_+ \langle p(\lambda), p(\mu) \rangle$.  
        The complement $\R_+ \langle p(\lambda), p(\mu) \rangle \:  \setminus \: \hat{Q}(\lambda, \mu)$  
        is contained in the union of the complement 
        $\Lambda(\fan)  \: \setminus \: \partial_+ \Lambda(\cF)$ and the convex hull 
        of the points $0, e_1, e_2$ deprived of the segment $[e_1, e_2]$. 
        Therefore, the origin $0$ is the only point of $N$ contained in 
        $\R_+ \langle p(\lambda), p(\mu) \rangle \:  \setminus \: \hat{Q}(\lambda, \mu)$. 
        As all the vertices of $Q(\lambda, \mu)$ belong to $N$, this shows that 
        $\hat{Q}(\lambda, \mu)$ is the convex hull of the set 
        $\R_+ \langle p(\lambda), p(\mu) \rangle \cap (N\:  \setminus \: \{0\})$. 
        One concludes using Proposition \ref{prop:regconv}.        
   \end{proof}

Consider again Figure \ref{fig:examfanlotus}. As shown by Proposition \ref{prop:dualtot}, 
the polygonal line on the left side gives a concrete embedding of the dual graph of 
the boundary $\partial X_{\fan^{reg}}$. But it does not show the order in which were 
performed the blow ups into which the associated modification 
$\psi_{\sigma_0}^{\fan}: X_{\fan} \to X_{\sigma_0}$ decomposes (see Theorem \ref{thm:composblow}). 
It turns out that this order is indicated by the lotus on the right side of Figure \ref{fig:examfanlotus}. 
To understand this fact, recall first the combinatorial description of the blow up of the 
orbit of dimension $0$ of the smooth affine toric surface $X_{\sigma_0}$, 
explained in Example \ref{ex:blowupor}: one gets it by subdividing the cone $\sigma_0$ using 
the ray generated by $e_1 + e_2$. In terms of the associated bases of $N$, one 
replaces the basis $(e_1, e_2)$ by the pair of bases $(e_1 , e_1 + e_2)$ and 
$(e_1 + e_2, e_2)$. Graphically, this may be understood as the passage from the base 
$[e_1, e_2]$ of the petal $\delta(e_1, e_2)$ seen as the simplest $2$-dimensional 
lotus (see Definition \ref{def:petal}) to its lateral 
boundary $[e_1 , e_1 + e_2] \cup [e_1 + e_2, e_2]$. Again by Proposition \ref{prop:dualtot}, 
we may see this passage as the replacement of the dual graph of $\partial X_{\sigma_0}$ by  
the dual graph of the boundary of the blown up toric surface. Now, each new petal in the lotus 
$\Lambda(\fan)$ corresponds to the blow up of an orbit of dimension $0$ of the previous 
toric surface. Its base may be seen as the dual graph of the  irreducible components of the 
boundary meeting at that point. One gets:

\begin{proposition}  \label{prop:dualevolution}
      Let $\fan$ be a Newton fan. Then: 
                 
             \noindent
             $\bullet$ The lateral boundary $\partial_+ \Lambda(\fan)$ of the lotus $\Lambda(\fan)$ 
                       is the dual graph of the boundary 
                       $\partial X_{\fan^{reg}}$ of the smooth toric surface $X_{\fan^{reg}}$. Two 
                       vertices of it are joined by an edge of the lotus $\Lambda(\fan)$ if and only if 
                       the corresponding orbits have intersecting closures at some moment of the process 
                       of creation of $\partial X_{\fan^{reg}}$ by blow ups of orbits of dimension $0$, 
                       which are particular infinitely near points of $O_{\sigma_0} \in X_{\sigma_0}$. 
              
              \noindent
             $\bullet$
              If one associates with each orbit of dimension $0$ the corresponding 
                        petal of $\Lambda(\fan)$, then the parent map on the set of petals induces 
                        on the previous set of $0$-dimensional orbits the restriction of the parent 
                        relation defined on the set of infinitely near points of $O_{\sigma_0}$ 
                        (see Definition \ref{def:infnear}). 
\end{proposition}

Let us set a notation for the constellation created during a
toric blow up process (see  Definition \ref{def:infnear}):

\begin{definition}  \label{def:toricconst}
    Let $\fan$ be a Newton fan. 
    Denote by $\boxed{\cC_{\fan}}$ the finite constellation above $O_{\sigma_0}$ consisting of the 
    $0$-dimensional orbits $O_{\sigma}$, where $\sigma$ varies among the 
    regular $2$-dimensional cones of the blow up process leading to 
    the smooth toric surface $X_{\fan^{reg}}$. It is {\bf the constellation of the fan $\fan$}. 
    \index{constellation!of a fan}
\end{definition}

Let $\sigma$ be one of the cones mentioned in Definition \ref{def:toricconst}. It is of the form 
$\cone \langle f_1, f_2 \rangle$, where $(f_1, f_2)$ is a positive basis 
of the lattice $N$. Proposition \ref{prop:dualevolution} shows that one 
may represent the $0$-dimensional orbit $O_{\sigma}$  
either by the edge $[f_1, f_2]$ of the lotus $\Lambda(\fan)$ or by the petal $\delta(f_1, f_2)$. 
How to understand the Enriques diagram of the constellation $\cC_{\fan}$ using the lotus 
$\Lambda(\fan)$? 
It turns out that this may be done easily using the representing edges $[f_1, f_2]$. 
In order to explain it, let  us introduce first the following definition (see Figures \ref{fig:Univenriques} and 
\ref{fig:specEnriques}):

\begin{definition}  \label{def:Enriquesst}
    Let $\delta(f_1, f_2)$ be a petal of the universal lotus $\Lambda(e_1, e_2)$. 
    Assume that it is different from $\delta(e_1, e_2)$, which means that 
   there exists a unique permutation $(i,j)$ of $(1,2)$ such that 
     $f_j - f_i \in \sigma_0 \cap N$
    (see Lemma \ref{lem:base}).  Then its {\bf Enriques edge} 
    is its lateral edge $[f_j, f_1 + f_2]$, that is, its unique lateral edge which extends 
    an edge of its parent petal. The {\bf Enriques tree} of a lotus $\Lambda$ is:
          \begin{itemize}
          \item
          the union of the Enriques edges of all its petals different from $\delta(e_1, e_2)$,  
                     rooted at its vertex $e_1 + e_2$, whenever $\Lambda$ is of dimension $2$;
         
         \item
          the vertex $e_1+ e_2$ of $\delta(e_1, e_2)$, if $\Lambda =[e_1, e_2]$.  
       \end{itemize}
    The {\bf extended Enriques tree} of a lotus $\Lambda$ is:
        \begin{itemize}
            \item the union of the Enriques subtree and of the lateral edge 
                 $[e_1, e_1+ e_2]$ of the base petal $\delta(e_1, e_2)$ 
                  of $\Lambda$, whenever $\Lambda$ is of dimension $2$;
             \item the lateral edge $[e_1, e_1+ e_2]$ of $\delta(e_1, e_2)$, if $\Lambda =[e_1, e_2]$.  
        \end{itemize}  
\end{definition}

 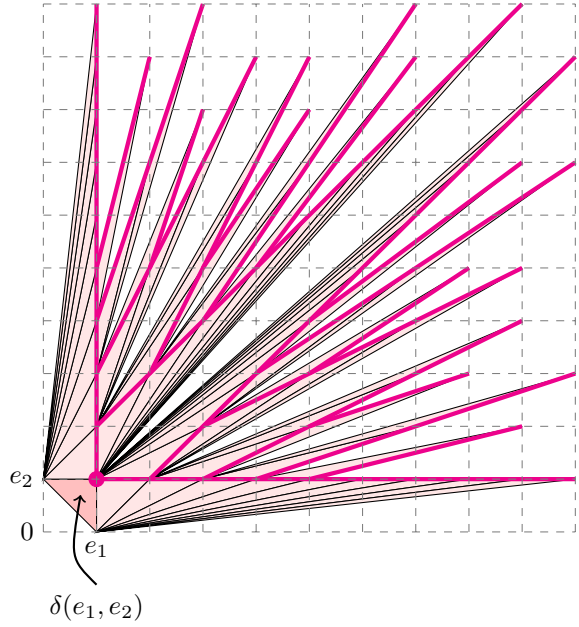
\begin{figure}[h!]
     \begin{center}
\begin{tikzpicture}[scale=0.7]

\draw [fill=pink](1,0) -- (0,1)--(1,1)--cycle;
\draw [fill=pink!40](1,0) -- (1,1)--(2,1)--cycle;
\draw [fill=pink!40](1,0) -- (2,1)--(3,1)--cycle;
\draw [fill=pink!40](1,0) -- (3,1)--(4,1)--cycle;
\draw [fill=pink!40](1,0) -- (4,1)--(5,1)--cycle;
\draw [fill=pink!40](1,0) -- (5,1)--(6,1)--cycle;
\draw [fill=pink!40](1,0) -- (6,1)--(7,1)--cycle;
\draw [fill=pink!40](1,0) -- (7,1)--(8,1)--cycle;
\draw [fill=pink!40](1,0) -- (8,1)--(9,1)--cycle;
\draw [fill=pink!40](1,0) -- (9,1)--(10,1)--cycle;

\draw [fill=pink!40](1,1) -- (2,1)--(3,2)--cycle;
\draw [fill=pink!40](1,1) -- (3,2)--(4,3)--cycle;
\draw [fill=pink!40](1,1) -- (4,3)--(5,4)--cycle;
\draw [fill=pink!40](1,1) -- (5,4)--(6,5)--cycle;
\draw [fill=pink!40](1,1) -- (6,5)--(7,6)--cycle;
\draw [fill=pink!40](1,1) -- (7,6)--(8,7)--cycle;
\draw [fill=pink!40](1,1) -- (8,7)--(9,8)--cycle;
\draw [fill=pink!40](1,1) -- (9,8)--(10,9)--cycle;

\draw [fill=pink!40](2,1) -- (3,2)--(5,3)--cycle;
\draw [fill=pink!40](2,1) -- (5,3)--(7,4)--cycle;
\draw [fill=pink!40](2,1) -- (7,4)--(9,5)--cycle;

\draw [fill=pink!40](2,1) -- (3,1)--(5,2)--cycle;
\draw [fill=pink!40](2,1) -- (5,2)--(7,3)--cycle;
\draw [fill=pink!40](2,1) -- (7,3)--(9,4)--cycle;

\draw [fill=pink!40](3,2) -- (5,3)--(8,5)--cycle;

\draw [fill=pink!40](3,2) -- (4,3)--(7,5)--cycle;
\draw [fill=pink!40](3,2) -- (7,5)--(10,7)--cycle;

\draw [fill=pink!40](4,3) -- (5,4)--(9,7)--cycle;

\draw [fill=pink!40](3,1) -- (5,2)--(8,3)--cycle;
\draw [fill=pink!40](3,1) -- (4,1)--(7,2)--cycle;
\draw [fill=pink!40](3,1) -- (7,2)--(10,3)--cycle;

\draw [fill=pink!40](4,1) -- (5,1)--(9,2)--cycle;

\draw [fill=pink!40](0,1) -- (1,1)--(1,2)--cycle;
\draw [fill=pink!40](0,1) -- (1,2)--(1,3)--cycle;
\draw [fill=pink!40](0,1) -- (1,3)--(1,4)--cycle;
\draw [fill=pink!40](0,1) -- (1,4)--(1,5)--cycle;
\draw [fill=pink!40](0,1) -- (1,5)--(1,6)--cycle;
\draw [fill=pink!40](0,1) -- (1,6)--(1,7)--cycle;
\draw [fill=pink!40](0,1) -- (1,7)--(1,8)--cycle;
\draw [fill=pink!40](0,1) -- (1,8)--(1,9)--cycle;
\draw [fill=pink!40](0,1) -- (1,9)--(1,10)--cycle;

\draw [fill=pink!40](1,1) -- (1,2)--(2,3)--cycle;
\draw [fill=pink!40](1,1) -- (2,3)--(3,4)--cycle;
\draw [fill=pink!40](1,1) -- (3,4)--(4,5)--cycle;
\draw [fill=pink!40](1,1) -- (4,5)--(5,6)--cycle;
\draw [fill=pink!40](1,1) -- (5,6)--(6,7)--cycle;
\draw [fill=pink!40](1,1) -- (6,7)--(7,8)--cycle;
\draw [fill=pink!40](1,1) -- (7,8)--(8,9)--cycle;
\draw [fill=pink!40](1,1) -- (8,9)--(9,10)--cycle;

\draw [fill=pink!40](1,2) -- (2,3)--(3,5)--cycle;
\draw [fill=pink!40](1,2) -- (3,5)--(4,7)--cycle;
\draw [fill=pink!40](1,2) -- (4,7)--(5,9)--cycle;

\draw [fill=pink!40](1,2) -- (1,3)--(2,5)--cycle;
\draw [fill=pink!40](1,2) -- (2,5)--(3,7)--cycle;
\draw [fill=pink!40](1,2) -- (3,7)--(4,9)--cycle;

\draw [fill=pink!40](2,3) -- (3,5)--(5,8)--cycle;

\draw [fill=pink!40](2,3) -- (3,4)--(5,7)--cycle;
\draw [fill=pink!40](2,3) -- (5,7)--(7,10)--cycle;

\draw [fill=pink!40](3,4) -- (4,5)--(7,9)--cycle;

\draw [fill=pink!40](1,3) -- (2,5)--(3,8)--cycle;
\draw [fill=pink!40](1,3) -- (1,4)--(2,7)--cycle;
\draw [fill=pink!40](1,3) -- (2,7)--(3,10)--cycle;

\draw [fill=pink!40](1,4) -- (1,5)--(2,9)--cycle;


\draw [-, ultra thick, color=magenta] (1,1) -- (1,10);
\draw [-, ultra thick, color=magenta] (1,1) -- (10,1);
\draw [-, ultra thick, color=magenta] (2,1) -- (10,9);
\draw [-, ultra thick, color=magenta] (1,2) -- (9,10);
\draw [-, ultra thick, color=magenta] (3,1) -- (9,4);
\draw [-, ultra thick, color=magenta] (1,3) -- (4,9);
\draw [-, ultra thick, color=magenta] (4,1) -- (10,3);
\draw [-, ultra thick, color=magenta] (1,4) -- (3,10);
\draw [-, ultra thick, color=magenta] (5,1) -- (9,2);
\draw [-, ultra thick, color=magenta] (1,5) -- (2,9);
\draw [-, ultra thick, color=magenta] (3,2) -- (9,5);
\draw [-, ultra thick, color=magenta] (2,3) -- (5,9);
\draw [-, ultra thick, color=magenta] (4,3) -- (10,7);
\draw [-, ultra thick, color=magenta] (3,4) -- (7,10);
\draw [-, ultra thick, color=magenta] (5,4) -- (9,7);
\draw [-, ultra thick, color=magenta] (4,5) -- (7,9);
\draw [-, ultra thick, color=magenta] (5,3) -- (8,5);
\draw [-, ultra thick, color=magenta] (3,5) -- (5,8);
\draw [-, ultra thick, color=magenta] (5,2) -- (8,3);
\draw [-, ultra thick, color=magenta] (2,5) -- (3,8);
\node[draw,circle, inner sep=2pt,color=magenta, fill=magenta] at (1,1){};

\draw [dashed, gray] (0,0) grid (10,10);
\node [below] at (1,0) {$e_{1}$}; 
\node [left] at (0,1) {$e_{2}$}; 
\node [left] at (0,0) {$0$}; 

\draw[->][thick, color=black](1,-1) .. controls (0.5,-0.5) ..(0.7,0.7);  
\node [below] at (1,-1) {$\delta(e_{1},e_{2})$}; 
   \end{tikzpicture}
\end{center}
  \caption{Partial view of the Enriques subtree of  the universal lotus $\Lambda(e_{1},e_{2})$}  
  \label{fig:Univenriques} 
    \end{figure}


\begin{figure}[h!]
    \begin{center}
\begin{tikzpicture}[scale=0.6]

\draw [->](0,0) -- (0,5);
\draw [->](0,0) -- (5,0);

\draw[fill=pink!40](1,0) -- (0,1) -- (1,1)  --cycle;
\draw[fill=pink!40](1,1) -- (1,2) -- (0,1) --cycle;
\draw[fill=pink!40](1,2) -- (1,3) -- (0,1) --cycle;
\draw[fill=pink!40](1,2) -- (1,3) -- (2,5) --cycle;
\draw[fill=pink!40](1,0) -- (1,1) -- (2,1) --cycle;
\draw[fill=pink!40](1,1) -- (2,1) -- (3,2) --cycle;
\draw[fill=pink!40](2,1) -- (3,2) -- (5,3) --cycle;

\draw [-, ultra thick, color=orange](0,1) -- (2,5) -- (1,2) -- (1,1) -- (5,3) -- (2,1) --(1,0);

\foreach \x in {0,1,...,5}{
\foreach \y in {0,1,...,5}{
      \node[draw,circle,inner sep=0.7pt,fill, color=gray!40] at (1*\x,1*\y) {}; }
  }

\node [left] at (1.3,-0.3) {$e_1$};
\node [left] at (0,1) {$e_2$};
\draw [->, very thick, red] (1,0)--(0.5, 0.5);
\draw [-, very thick, red] (0.5, 0.5)--(0,1);

\draw [-, line width=2.5pt, color=magenta] (1,1) -- (2,1) -- (3,2) -- (5,3);
\draw [-, line width=2.5pt, color=magenta] (1,1) -- (1,3) -- (2,5);
\node[draw,circle, inner sep=3pt,color=magenta, fill=magenta] at (1,1){};

\begin{scope}[shift={(8,0)}]
\draw [->](0,0) -- (0,5);
\draw [->](0,0) -- (5,0);

\draw[fill=pink!40](1,0) -- (0,1) -- (1,1)  --cycle;
\draw[fill=pink!40](1,1) -- (1,2) -- (0,1) --cycle;
\draw[fill=pink!40](1,2) -- (1,3) -- (0,1) --cycle;
\draw[fill=pink!40](1,2) -- (1,3) -- (2,5) --cycle;
\draw[fill=pink!40](1,0) -- (1,1) -- (2,1) --cycle;
\draw[fill=pink!40](1,1) -- (2,1) -- (3,2) --cycle;
\draw[fill=pink!40](2,1) -- (3,2) -- (5,3) --cycle;

\draw [-, ultra thick, color=orange](0,1) -- (2,5) -- (1,2) -- (1,1) -- (5,3) -- (2,1) --(1,0);

\foreach \x in {0,1,...,5}{
\foreach \y in {0,1,...,5}{
      \node[draw,circle,inner sep=0.7pt,fill, color=gray!40] at (1*\x,1*\y) {}; }
  }

\node [left] at (1.3,-0.3) {$e_1$};
\node [left] at (0,1) {$e_2$};
\draw [->, very thick, red] (1,0)--(0.5, 0.5);
\draw [-, very thick, red] (0.5, 0.5)--(0,1);

\draw [-, line width=2.5pt, color=magenta] (1,1) -- (2,1) -- (3,2) -- (5,3);
\draw [-, line width=2.5pt, color=magenta] (1,1) -- (1,3) -- (2,5);

\draw [-, line width=3pt, color=magenta!70!] (1,0) -- (1,1);

\end{scope}
\end{tikzpicture}
\end{center}
\caption{The Enriques tree and the extended Enriques tree of the lotus 
                   $\Lambda\left(3/5, 2/1, 5/2\right)$}
     \label{fig:specEnriques}
   \end{figure}
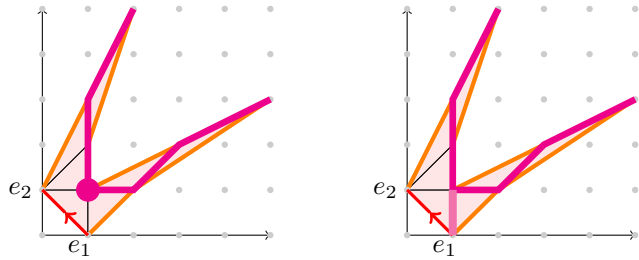

One has the following interpretation of the Enriques diagram of the 
constellation of the fan $\fan$ using the lotus $\Lambda(\fan)$. It 
allows to understand for which reason we defined the Enriques tree of a lotus reduced 
to the base $[e_1, e_2]$ in the previous way:

\begin{proposition}   \label{prop:lotusinterprenriques}
    Let $\fan$ be a Newton fan. Then the Enriques diagram 
    $\Enriques(\cC_{\fan})$ of the constellation $\cC_{\fan}$ of $\fan$ 
(see Definition \ref{def:toricconst}) is isomorphic to the Enriques 
    subtree of the lotus $\Lambda(\fan)$. This isomorphism sends each orbit 
    $O_{\sigma}$ belonging to $\cC_{\fan}$ onto the point $f_1 + f_2$, if 
    $\sigma = \cone \langle f_1, f_2 \rangle$. 
\end{proposition}

\begin{proof}
   The basic idea is that we have a bijection between the set of infinitely near points of 
$O_{\sigma_0}$ 
    and the set of prime exceptional divisors created by blowing them up. Therefore, 
    the parent binary relation may be thought as a binary relation on the set of those 
    prime exceptional divisors. In this proposition, we restrict to the divisors which 
    are the orbit closures 
    $\overline{O}_{\rho}$, where $\rho$ varies among the rays of the regularization $\fan^{reg}$ 
    of $\fan$ which are distinct from the edges of $\sigma_0$. Each such a ray is generated by a 
    lateral vertex of $\Lambda(\fan)$, therefore the parent binary relation among 
    those orbit closures may be also seen as a binary relation among those lateral 
    vertices. One may prove by induction on this number of rays, that is, on the number 
    of petals of the associated lotus $\Lambda(\fan)$, that the pairs of related vertices are 
    precisely those which are connected by an edge in the Enriques tree of $\Lambda(\fan)$. 
      
    The case $\fan= \sigma_0$ corresponds to a constellation formed by 
   $O_{\sigma_0}$ alone. 
    In this case one looks at the prime divisor created by blowing it up, that is, at 
    $\overline{O}_{\cone \langle e_1 + e_2 \rangle}$. This explains why we defined 
     $\Enriques(\cC_{\sigma_0})$ as the vertex $e_1 + e_2$ of the petal $\delta(e_1, e_2)$.
\end{proof}

\begin{remark}  \label{rem:whyEET}
The reason why we introduced also the notion of \emph{extended Enriques tree} 
in Definition \ref{def:Enriquesst}, in addition to that of \emph{Enriques tree}, will 
become clear after understanding point (\ref{point:globalE}) of Theorem \ref{thm:repsailtor}. 
Briefly speaking, the constellations associated to the toroidal pseudo-resolution 
processes have associated lotuses which are glued from lotuses of Newton fans. 
An analog of Proposition \ref{prop:lotusinterprenriques} is also true for them. The corresponding 
Enriques tree contains the Enriques trees of the Newton fans created by the 
toroidal process, but also other edges. Those supplementary edges are precisely the edges  
which have to be added to the Enriques tree of a Newton fan in order to get the 
corresponding extended Enriques tree (see Definition \ref{def:lotustoroid} below). 
\end{remark}

The lotus $\Lambda(\fan)$ contains also the \emph{graph of the proximity binary relation} 
on the constellation $\cC_{\fan}$, whose set of vertices 
is the given constellation, two points being joined by an edge if and only if one of them is proximate 
to the other one (see Definition \ref{def:infnear}):

\begin{proposition}  \label{prop:lotusinterprproxim}
  Let $\fan$ be a fan refining the regular cone $\sigma_0$. Then the graph 
  of the proximity binary relation on the finite constellation $\cC_{\fan}$ 
  is isomorphic to the union of the edges of the lotus $\Lambda(\fan)$ which do not contain 
  the vertices $e_1$ and $e_2$. 
\end{proposition}

The proof of this proposition is based on the same principles as the proof of 
Proposition \ref{prop:lotusinterprenriques} and is left to the reader.

\subsection{Lotuses and continued fractions}
\label{ssec:lotcf}
$\:$
\medskip

In this subsection we explain a way to build, up to isomorphism, the lotus of a finite set of positive 
rational numbers in the sense of Definition \ref{def:deflot},  
starting from the continued fraction expansions of its elements. Namely, given a positive 
rational number $\lambda$, we show how to construct an \emph{abstract lotus} 
$\Delta(\lambda)$ starting from the continued fraction expansion of $\lambda$ 
(see Definition \ref{def:decomponenumber}) and we explain that $\Delta(\lambda)$ is isomorphic 
to the lotus $\Lambda(\lambda)$. Then we show how to glue two abstract lotuses 
$\Delta(\lambda)$ and $\Delta(\mu)$ in order to get a simplicial complex isomorphic to 
the lotus $\Lambda(\lambda, \mu)$ (see Proposition \ref{prop:lotus-wedge}). 
This extends readily to arbitrary finite sets of positive rationals. 
\medskip

Recall first the following classical notion:

\begin{definition}  \label{def:contfrac}
  Let $ k \in \N^*$ and let $a_1, \dots , a_k$ be natural numbers 
  such that $a_1 \geq 0$ and $a_j >0$ if $j \in  \{ 2, \dots, k \}$. 
  The {\bf continued fraction} \index{continued fraction} \index{term!of a continued fraction} 
  with {\bf terms} $a_1, \dots , a_k$ is the non-negative rational number: 
   $$\boxed{[a_1, a_2, \dots, a_k]} : = 
          a_1 + \cfrac{1}{a_2 + \cfrac{1}{ \cdots + \cfrac{1}{a_k}}}.$$ 
\end{definition}

Any $\lambda \in \Q_+^*$ may be written uniquely as a continued fraction 
$[a_1, a_2, \dots, a_k]$ if one imposes the constraint that $a_k>1$ whenever $\lambda \neq 1$. 
One speaks then of the {\bf continued fraction expansion} of $\lambda$. Note that 
its  first term  $a_1$ vanishes if and only if $\lambda \in (0, 1)$. \index{continued fraction!expansion}

\begin{definition}  \label{def:decomponenumber}
   Let $\lambda \in \Q_+^*$. Consider its continued fraction expansion 
   $\lambda =  [a_1, a_2, \dots, a_k]$. Its {\bf abstract lotus}  \index{lotus!abstract}
   $\boxed{\Delta(\lambda)}$ is the simplicial complex constructed as follows: 
   \medskip
   
   \noindent
   $\bullet$
   Start from an affine triangle $[A_1,A_2, V]$, with vertices $A_1, A_2, V$. 
       
       \noindent
   $\bullet$
        Draw a polygonal line $P_0 P_1 P_2 \dots P_{k-1}$ whose vertices 
              belong alternatively to the sides $[A_1, V]$, $[A_2, V]$, and such that  $P_0 := A_2$ and    
            $$ \left\{ \begin{array}{l} 
                        P_1 \in [A_1, V), \mbox{ with } P_1 =A_1 \mbox{ if and only if }  a_1 =0,  \\
                          P_i \in (P_{i-2}, V) \mbox{ for any }  i \in \{2, \dots, k-1\}.
                    \end{array}  \right. $$ 
                 By convention, we set also $P_{-1} := A_1, P_k := V$. The resulting subdivision  
                 of the triangle $[A_1, A_2, V]$ into $k$ triangles 
                 is the {\bf zigzag decomposition associated with $\lambda$}. 
                 \index{zigzag decomposition}
       
       \noindent
   $\bullet$
        Decompose then each segment $[P_{i -1}, P_{i + 1}]$ (for $i \in \{0, \dots, k-1 \}$) 
               into $a_{i+1}$ segments, and join the interior points of $[P_{i -1}, P_{i + 1}]$ 
               created in this way to $P_i$. One obtains then a new triangulation of the 
               initial triangle $[A_1,A_2,V]$, which is by definition the abstract lotus $\Delta(\lambda)$. 
               
    \medskip 
    The {\bf base} of the abstract lotus \index{base!of an abstract lotus} 
    $\Delta(\lambda)$ is the segment $[A_1, A_2]$, oriented from $A_1$ to $A_2$. 
    One orients also the other edges of $\Delta(\lambda)$ in the following way: 
      \medskip 
      
         \noindent
   	$\bullet$
          $[P_{i-1}, P_i]$ is oriented from $P_i$ to $P_{i-1}$, for each $i \in \{1, \dots, k-1\}$. 
          
          \noindent
   	$\bullet$
   	An edge joining $P_i$ to a point of the open segment $(P_{i-1}, P_{i+1})$ is oriented 
              towards $P_i$. 
          
          \noindent
   	$\bullet$
   	An edge contained in a segment $[V, A_j]$ is oriented towards $A_j$, for 
             each $j \in \{1, 2\}$.
 \end{definition}

 The abstract lotus $\Delta(\lambda)$ of $\lambda \in  \Q_+^*$ is a simplicial complex 
 of pure dimension $2$, isomorphic to a convex polygon triangulated by 
 diagonals intersecting only at vertices and with a distinguished oriented base. It is well-defined, up to 
 combinatorial isomorphism of polygons triangulated by diagonals intersecting only at vertices, 
 respecting the bases and their orientations. The orientations of its other edges are in fact 
 determined by the orientation of the base. Those orientations will not be important in the 
 sequel, excepted in Proposition \ref{prop:countpaths} below. For this reason we do not 
 draw them in our examples of abstract lotuses.

\begin{example} \label{Simplelot}
   Figures \ref{fig:Exlot1} and \ref{fig:Exlot2} represent the previous constructions 
applied to the numbers $\lambda = [4,2,5]$ and $\mu = [3,2,1,4]$. On the left are shown  
the initial zigzag decompositions and on the right the final abstract lotuses  
$\Delta(\lambda)$ and $\Delta(\mu)$.
\end{example}

    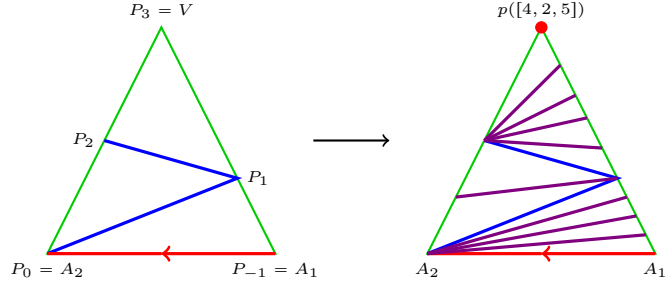
\begin{figure}
    \begin{center}
\begin{tikzpicture}[scale=0.5]
\draw [color=black!20!green, thick] (0,0) -- (6,0)--(3,6)--cycle;
\draw [-,color=blue, very thick] (0,0) -- (5,2)--(1.5,3);
\draw [->, color=red, very thick] (6,0)--(3,0);
\draw [-, color=red, very thick] (3,0)--(0,0);
\node [below] at (0,0) {{\tiny $P_{0}=A_{2}$}};
\node [below] at (6,0) {{\tiny $P_{-1}=A_{1}$}};
\node [right] at (5,2) {{\tiny $P_{1}$}};
\node [left] at (1.5,3) {{\tiny $P_{2}$}};
\node [above] at (3,6) {{\tiny $P_{3}=V$}};

 \draw [->, thick](7,3) -- (9, 3) ;

\begin{scope}[shift={(10,0)}]
\draw [color=black!20!green, thick] (0,0) -- (6,0)--(3,6)--cycle;
\draw [-,color=blue, very thick] (0,0) -- (5,2)--(1.5,3);
\draw [->, color=red, very thick] (6,0)--(3,0);
\draw [-, color=red, very thick] (3,0)--(0,0);
\draw [-, color=violet, very thick] (0,0)--(23/4,0.5);
\draw [-, color=violet, very thick] (0,0)--(11/2,1);
\draw [-, color=violet, very thick] (0,0)--(21/4,1.5);
\draw [-, color=violet, very thick] (5,2)--(0.75,1.5);
\draw [-, color=violet, very thick] (1.5,3)--(4.6,2.8);
\draw [-, color=violet, very thick] (1.5,3)--(4.2,3.6);
\draw [-, color=violet, very thick] (1.5,3)--(3.9,4.2);
\draw [-, color=violet, very thick] (1.5,3)--(3.5,5);
\node[draw,circle, inner sep=1.5pt,color=red, fill=red] at (3,6){};
\node [above] at (3,6) {{\tiny $p([4,2,5])$}};
\node [below] at (0,0) {{\tiny $A_{2}$}};
\node [below] at (6,0) {{\tiny $A_{1}$}};
\end{scope}
\end{tikzpicture}
\end{center}
 \caption{The construction of the abstract lotus $\Delta([4,2,5])$}
\label{fig:Exlot1}
   \end{figure}

\begin{figure}
    \begin{center}
\begin{tikzpicture}[scale=0.5]
\draw [color=black!20!green, thick] (0,0) -- (6,0)--(3,6)--cycle;
\draw [-,color=blue, very thick] (0,0) -- (5,2)--(1.5,3)--(4,4);
\draw [->, color=red, very thick] (6,0)--(3,0);
\draw [-, color=red, very thick] (3,0)--(0,0);
\node [below] at (0,0) {{\tiny $P_{0}=A_{2}$}};
\node [below] at (6,0) {{\tiny $P_{-1}=A_{1}$}};
\node [right] at (5,2) {{\tiny $P_{1}$}};
\node [left] at (1.5,3) {{\tiny $P_{2}$}};
\node [right] at (4,4) {{\tiny $P_{3}$}};
\node [above] at (3,6) {{\tiny $P_{4}=V$}};

\draw [->, thick](7,3) -- (9, 3) ;

\begin{scope}[shift={(10,0)}]
\draw [color=black!20!green, thick] (0,0) -- (6,0)--(3,6)--cycle;
\draw [-,color=blue, very thick] (0,0) -- (5,2)--(1.5,3)--(4,4);
\draw [->, color=red, very thick] (6,0)--(3,0);
\draw [-, color=red, very thick] (3,0)--(0,0);
\draw [-, color=violet, very thick] (0,0)--(5.65,0.7);
\draw [-, color=violet, very thick] (0,0)--(5.3,1.4);
\draw [-, color=violet, very thick] (5,2)--(0.75,1.5);
\draw [-, color=violet, very thick] (4,4)--(1.875,3.75);
\draw [-, color=violet, very thick] (4,4)--(2.25,4.5);
\draw [-, color=violet, very thick] (4,4)--(2.625,5.25);
\node[draw,circle, inner sep=1.5pt,color=red, fill=red] at (3,6){};
\node [above] at (3,6) {{\tiny $p([3,2,1,4])$}};
\node [below] at (0,0) {{\tiny $A_{2}$}};
\node [below] at (6,0) {{\tiny $A_{1}$}};
\end{scope}
\end{tikzpicture}
\end{center}
 \caption{The construction of the abstract lotus $\Delta([3,2,1,4])$}
 \label{fig:Exlot2}
   \end{figure}
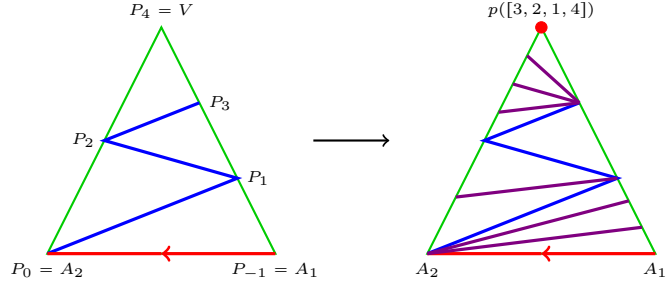

 The abstract lotus of a positive rational number is isomorphic with its lotus:

\begin{proposition} \label{prop:drawlotone} 
   There is a unique isomorphism between the lotus $\Lambda(\lambda)$ and 
   the abstract lotus $\Delta(\lambda)$, seen as simplicial complexes with 
   a marked point and an oriented base. 
\end{proposition}

\begin{proof} The isomorphism sends $A_i$ to $e_i$ for $i= 1,2$.
  The proof may be done by induction on $k$, the number of  terms in the continued 
  fraction expansion of $\lambda$. We leave the details to the reader. 
\end{proof}

The previous isomorphism does not always send the orientations of the edges 
  of $\Lambda(\lambda)$ as chosen after Definition \ref{def:petal} onto the orientations of the 
  edges of $\Delta(\lambda)$ as fixed in Definition \ref{def:decomponenumber}. The possibility 
  of defining various canonical orientations on the edges of a lotus of the form 
  $\Lambda(\lambda)$ may be useful in applications. 

The rational number $\lambda >0$ may be recovered in the following way from the 
structure of the corresponding abstract lotus:

  \begin{proposition} \label{prop:countpaths}
      Assume that $\lambda = p_2/ p_1$ with $p_1, p_2 \in \N^*$ coprime. Then, for each 
      $j \in \{1, 2\}$, the positive integer $p_j$ is equal to the number of oriented paths 
      not containing the base $[A_1, A_2]$ and going from $V$ to $A_j$ inside the $1$-skeleton 
      of $\Delta(\lambda)$, oriented as in Definition \ref{def:decomponenumber}. 
  \end{proposition}
  
  This proposition may be easily proved by induction on the number of petals of 
$\Delta(\lambda)$. It shows a way in which the numbers leading to the construction of 
a Newton lotus may be interpreted as combinatorial invariants of the lotus, seen purely as a 
marked simplicial complex with oriented base.

\begin{example}
   In Figure \ref{fig:Illustr-4.26} is represented the case $(p_1, p_2) = (2,3)$ of Proposition \ref{prop:countpaths}.  We have drawn twice the lotus $\Delta(3/2) = 
\Delta([1,2])$. On the right are drawn the $2$ oriented paths starting from $V$ and arriving at $A_1$. 
On the left are drawn the $3$ oriented paths starting from $V$ and arriving at $A_2$. We see that 
the constraint not to contain the base is necessary, otherwise one would obtain $2$ more paths 
from $V$ to $A_2$ by adding the base to the paths from $V$ to $A_1$.

\begin{figure}
    \begin{center}
\begin{tikzpicture}[scale=0.5]
\draw [color=black!20!green, thick] (0,0) -- (6,0)--(3,6)--cycle;
\draw [-,color=black!20!green, very thick] (0,0) -- (5,2)--(1.5,3);
\draw [->, color=red, very thick] (6,0)--(3,0);

\draw [->,color=blue, very thick] (2.8,6)--(1.8,4);
\draw [->,color=blue, very thick] (1.8,4)--(0.55,1.5);
\draw [-,color=blue, very thick] (1.3,3)--(-0.2,0);

\draw [->,color=blue, very thick] (3.2,6)--(3.95,4.5);
\draw [-,color=blue, very thick] (3.95,4.5)--(5.2,2);
\draw [->,color=blue, very thick] (5.2,2)--(2.7,1);
\draw [-,color=blue, very thick] (2.7,1)--(0.2,0);

\draw [->,color=white!70!blue, very thick] (3,5.5)--(2.4,4.35);
\draw [-,color=white!70!blue, very thick] (2.4,4.35)--(1.719,2.938);
\draw [->,color=white!70!blue, very thick] (1.719,2.938)--(3.4595,2.469);
\draw [-,color=white!70!blue, very thick] (3.4595,2.469)--(4.7,2.1);

\draw [->,color=white!70!blue, very thick] (4.7,2.1)--(2.25,1.15);
\draw [-,color=white!70!blue, very thick] (2.25,1.15)--(0.2,0.35);

\draw [-, color=red, very thick] (3,0)--(0,0);
\node [below] at (0,0) {{\tiny $A_{2}$}};
\node [below] at (6,0) {{\tiny $A_{1}$}};
\node [above] at (3,6) {{\tiny $V$}};
\node [below] at (3,-0.8) {$p_2=3$};

\begin{scope}[shift={(10,0)}]
\draw [color=black!20!green, thick] (0,0) -- (6,0)--(3,6)--cycle;
\draw [-,color=black!20!green, very thick] (0,0) -- (5,2)--(1.5,3);
\draw [->, color=red, very thick] (6,0)--(3,0);

\draw [->,color=blue, very thick] (3.2,6)--(3.95,4.5);
\draw [-,color=blue, very thick] (3.95,4.5)--(5.2,2);
\draw [->,color=blue, very thick] (5.2,2)--(5.7,1);
\draw [-,color=blue, very thick] (5.7,1)--(6.2,0);

\draw [->,color=white!70!blue, very thick] (3,5.5)--(2.4,4.35);
\draw [-,color=white!70!blue, very thick] (2.4,4.35)--(1.719,2.938);
\draw [->,color=white!70!blue, very thick] (1.719,2.938)--(3.4595,2.469);
\draw [-,color=white!70!blue, very thick] (3.4595,2.469)--(4.8,2.1);
\draw [->,color=white!70!blue, very thick] ((4.8,2.1)--(5.325,1.05);
\draw [-,color=white!70!blue, very thick] (5.325,1.05)--(5.85,0);

\draw [-, color=red, very thick] (3,0)--(0,0);
\node [below] at (0,0) {{\tiny $A_{2}$}};
\node [below] at (6,0) {{\tiny $A_{1}$}};
\node [above] at (3,6) {{\tiny $V$}};

\node [below] at (3,-0.8) {$p_1=2$};
\end{scope}
\end{tikzpicture}
\end{center}
\caption{An illustration of  Proposition \ref{prop:countpaths} for $p_{2}/p_{1}=3/2$ }
\label{fig:Illustr-4.26}
   \end{figure}
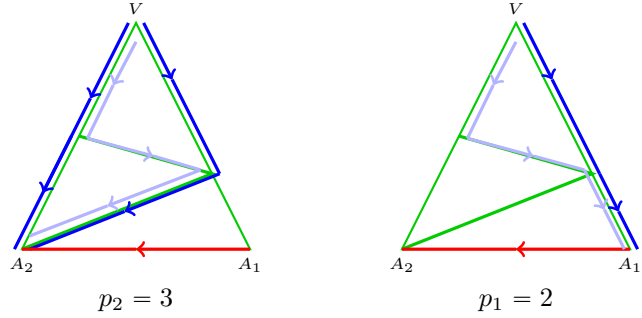

\end{example}

Suppose now that one has two numbers $\lambda, \mu \in \Q_+^*$. If 
$\lambda =  [a_1,\dots, a_k]$ and 
$\mu = [b_1,\dots, b_l]$, let $j\in \{0, \dots,  \min\{k,l\}\}$ be maximal such that  
$a_i = b_i$ for all $i \in \{1,\dots, j\}$.  We may assume, up 
to permutation of $\lambda$ and $\mu$, that $k =j$ or $a_{j+1} < b_{j+1}$. Define then:
 \begin{equation} \label{eq:defwedge}  
     \boxed{\lambda \wedge \mu} = \boxed{\mu \wedge \lambda}: = \left\{ 
       \begin{array}{l}
             [a_1, \dots , a_j],  \mbox{ if }    k = j,    \\
                \left[a_1, \dots, a_j, a_{j+1} \right],    \mbox{ if }  k =  j + 1,  \\
               \left[a_1, \dots, a_j, a_{j+1}  + 1 \right],   \mbox{ if }  k > j + 1.
       \end{array}  \right.
  \end{equation}
    
     Next proposition explains that the symmetric binary operation $\wedge$ on $\Q_+^*$ 
     allows to describe the intersection of two lotuses of the form $\Lambda(\lambda)$: 

\begin{proposition} \label{prop:lotus-wedge} 
   For any $\lambda, \mu \in \Q_+^*$, one has:
       \[ \Lambda(\lambda) \cap \Lambda(\mu) = \Lambda( \lambda \wedge \mu).  \]
   Therefore, the lotus $\Lambda(\lambda, \mu)$ is isomorphic as a simplicial complex with an oriented 
   base to the triangulated polygon obtained by gluing 
   $\Delta(\lambda)$ and $\Delta(\mu)$ along $\Delta(\lambda \wedge \mu)$.
\end{proposition}

\begin{proof}
      Assume that $\lambda =  [a_1,\dots, a_k]$. Proposition \ref{prop:drawlotone} shows 
      in particular that the lotus $\Lambda(\lambda)$ 
   has $n:= a_1 + \cdots + a_k$ petals. Denote by $(\lambda_i)_{1 \leq i \leq n}$ the sequence 
   of positive rationals such that the successive non-basic vertices of the petals of 
   $\Lambda([a_1, \dots, a_k])$  are the primitive vectors $p(\lambda_1) , \dots, p(\lambda_n)$. 
   The sequence of continued fraction expansions of $(\lambda_i)_{1 \leq i \leq n}$ is:
     \begin{equation} \label{eq-cf}
      [1], [2], \dots, [a_1], [a_1, 1], [a_1, 2], \dots, [a_1, a_2], [a_1, a_2, 1], \dots,  
          [a_1, \dots, a_k]. 
    \end{equation}
   One may prove this fact at the same time as Proposition \ref{prop:drawlotone}, by making 
   now an induction on the number $n$ of petals of $\Lambda([a_1, \dots, a_k])$, 
   instead of the number $k$ of terms of the continued fraction. 
 
   The proposition results then by combining the previous fact with formula (\ref{eq:defwedge}). 
\end{proof}

\begin{example}
  Let us consider the two rational numbers $\lambda = [4,2,5]$ and $\mu = [3,2,1,4]$ of 
 Example \ref{Simplelot}. Then $j =0, k= 3, l=4$, therefore $j + 1 < \min \{k,l\}$ and 
  $\lambda \wedge \mu = [ 3 + 1] = 4$. The lotus $\Lambda(\lambda, \mu)$ is 
 therefore isomorphic to the triangulated polygon with an oriented base of 
 the right side of Figure \ref{fig:Doublelot}. 
\end{example}

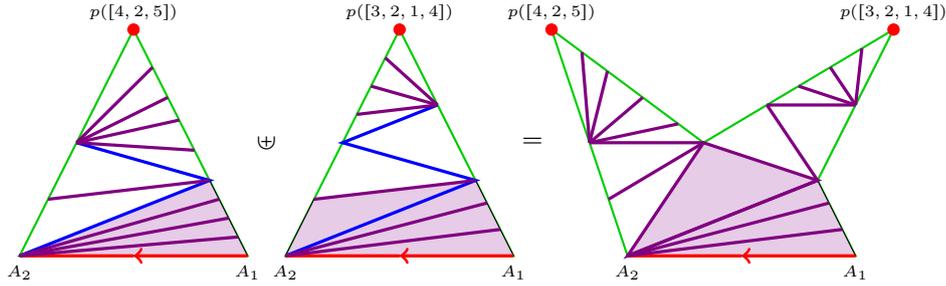
\begin{figure}[h!]
    \begin{center}
\begin{tikzpicture}[scale=0.5]
\draw [color=black!20!green, thick] (0,0) -- (6,0)--(3,6)--cycle;
\draw[fill=violet!20](0,0) -- (6,0) -- (5,2)  --cycle;
\draw [-,color=blue, very thick] (0,0) -- (5,2)--(1.5,3);
\draw [->, color=red, very thick] (6,0)--(3,0);
\draw [-, color=red, very thick] (3,0)--(0,0);
\draw [-, color=violet, very thick] (0,0)--(23/4,0.5);
\draw [-, color=violet, very thick] (0,0)--(11/2,1);
\draw [-, color=violet, very thick] (0,0)--(21/4,1.5);
\draw [-, color=violet, very thick] (5,2)--(0.75,1.5);
\draw [-, color=violet, very thick] (1.5,3)--(4.6,2.8);
\draw [-, color=violet, very thick] (1.5,3)--(4.2,3.6);
\draw [-, color=violet, very thick] (1.5,3)--(3.9,4.2);
\draw [-, color=violet, very thick] (1.5,3)--(3.5,5);
\node[draw,circle, inner sep=1.5pt,color=red, fill=red] at (3,6){};
\node [above] at (3,6) {{\tiny $p([4,2,5])$}};
\node [below] at (0,0) {{\tiny $A_{2}$}};
\node [below] at (6,0) {{\tiny $A_{1}$}};
\node at (6.5,3) {{\large $\uplus$}};

\begin{scope}[shift={(7,0)}]
\draw [color=black!20!green, thick] (0,0) -- (6,0)--(3,6)--cycle;
\draw[fill=violet!20](0,0) -- (6,0) -- (5,2)  --(0.75,1.5)--cycle;
\draw [-,color=blue, very thick] (0,0) -- (5,2)--(1.5,3)--(4,4);
\draw [->, color=red, very thick] (6,0)--(3,0);
\draw [-, color=red, very thick] (3,0)--(0,0);
\draw [-, color=violet, very thick] (0,0)--(5.65,0.7);
\draw [-, color=violet, very thick] (0,0)--(5.3,1.4);
\draw [-, color=violet, very thick] (5,2)--(0.75,1.5);
\draw [-, color=violet, very thick] (4,4)--(1.875,3.75);
\draw [-, color=violet, very thick] (4,4)--(2.25,4.5);
\draw [-, color=violet, very thick] (4,4)--(2.625,5.25);
\node[draw,circle, inner sep=1.5pt,color=red, fill=red] at (3,6){};
\node [above] at (3,6) {{\tiny $p([3,2,1,4])$}};
\node [below] at (0,0) {{\tiny $A_{2}$}};
\node [below] at (6,0) {{\tiny $A_{1}$}};
\end{scope}

\node at (13.5,3) {{\large $=$}};

\begin{scope}[shift={(16,0)}]
\draw [color=black!20!green, thick] (0,0) -- (-2,6) --(2,3)--(7,6)--(5,2)--(6,0)--cycle;
\draw[fill=violet!20](0,0) -- (2,3) --(2,3)--(5,2)--(6,0)--cycle;
\draw [->, color=red, very thick] (6,0)--(3,0);
\draw [-, color=red, very thick] (3,0)--(0,0);
\draw [-, color=violet, very thick] (0,0)--(5.65,0.7);
\draw [-, color=violet, very thick] (0,0)--(5.3,1.4);
\draw [-, color=violet, very thick] (0,0)--(5,2);
\draw [-, color=violet, very thick] (0,0)--(2,3)--(5,2)--cycle;
\draw [-, color=violet, very thick] (2,3)--(-0.5,1.5);
\draw [-, color=violet, very thick] (2,3)--(-1,3);
\draw [-, color=violet, very thick] (-1,3)--(4/3,3.5);
\draw [-, color=violet, very thick] (-1,3)--(0.4,4.2);
\draw [-, color=violet, very thick] (-1,3)--(-0.53,4.9);
\draw [-, color=violet, very thick] (-1,3)--(-1.2,5.4);
\draw [-, color=violet, very thick] (5,2)--(11/3,4);
\draw [-, color=violet, very thick] (6,4)--(11/3,4);
\draw [-, color=violet, very thick] (6,4)--(13.5/3,4.5);
\draw [-, color=violet, very thick] (6,4)--(16/3,5);
\draw [-, color=violet, very thick] (6,4)--(37/6,5.5);
\node[draw,circle, inner sep=1.5pt,color=red, fill=red] at (-2,6){};
\node [above] at (-2,6) {{\tiny $p([4,2,5])$}};
\node[draw,circle, inner sep=1.5pt,color=red, fill=red] at (7,6){};
\node [above] at (7,6) {{\tiny $p([3,2,1,4])$}};
\node [below] at (0,0) {{\tiny $A_{2}$}};
\node [below] at (6,0) {{\tiny $A_{1}$}};
\end{scope}
\end{tikzpicture}
\end{center}
  \caption{The abstract lotus $\Delta([4,2,5], [3,2,1,4])$}
 \label{fig:Doublelot}
 \end{figure}

Iterating the gluing operation, one may construct an {\bf abstract lotus} \index{lotus!abstract} 
$\boxed{\Delta(\lambda_1,\dots, \lambda_k)}$ 
combinatorially equivalent to any given Newton lotus 
$\Lambda(\lambda_1,\dots, \lambda_k)$, seen as a triangulated polygon with 
marked points and oriented base. One gets an abelian monoid 
of (abstract) lotuses, the monoid operation $\boxed{\uplus}$ generalizing the gluing operation of 
Figure \ref{fig:Doublelot}. Namely, if $\cE_1$ and $\cE_2$ are finite subsets 
of $\Q_+ \cup \{\infty\}$, then:
  \begin{equation} \label{eq:monoidlotus}  
      \boxed{\Delta(\cE_1) \uplus \Delta(\cE_2)} :=   \Delta(\cE_1 \cup \cE_2). 
  \end{equation}  
The neutral element of this monoid is the segment $[A_1, A_2]= \Delta(\emptyset) 
= \Delta(0) = \Delta(\infty) = \Delta(\{0, \infty\})$.

\subsection{The lotus of a toroidal  pseudo-resolution}
\label{ssec:sailtores}
$\:$
\medskip

In this subsection we reach a second level of explanation of the subtitle of this article, 
the first level having been reached in Subsection \ref{ssec:lotnf} above. Namely, we define  
a new kind of lotus by gluing the lotuses associated to the Newton fans produced by  
Algorithm \ref{alg:tores} (see Definition \ref{def:lotustoroid}).  
We illustrate this definition by 
our recurrent example (see Example \ref{ex:lotustoroid}) and by the case of an arbitrary branch 
(see Example \ref{ex:onebranch}).  Finally,
we show how this lotus allows 
to visualize many objects associated to the regularized algorithm and 
with the decomposition into blow ups of points of the embedded resolution produced by it 
(see Theorem \ref{thm:repsailtor}). 
\medskip

Consider again a reduced curve singularity $C$ on the smooth germ of surface $(S,o)$. Fix a smooth 
branch $L$ on $(S,o)$, and run Algorithm \ref{alg:tores}. Denote as before by 
$\pi: (\Sigma, \partial \Sigma) \to (S, L + L')$ a resulting toroidal pseudo-resolution of $C$. 
We associated to it a fan tree $(\theta_{\pi}(C), \slp_{\pi})$, as explained 
in Definition \ref{def:fantreetr}.  One may associate an analogous fan tree 
$(\theta_{\pi^{reg}}(C), \slp_{\pi^{reg}})$ to the toroidal resolution 
$\pi^{reg}: (\Sigma^{reg}, \partial \Sigma^{reg}) \to (S, L + L')$ defined in Subsection 
\ref{ssec:toremb} (see Proposition \ref{prop:regalg}). One sees that the 
trunks used in the two constructions  are the same, as well as the gluing rules. 
What changes is that $\theta_{\pi^{reg}}(C)$ has more vertices than 
$\theta_{\pi}(C)$, those labeled by the irreducible components of the exceptional 
divisor of the modification $\eta : \Sigma^{reg} \to \Sigma$ which resolves the 
singularities of the surface $\Sigma$. Therefore:

\begin{proposition}   \label{prop:morelabels}
    Seen as rooted trees endowed with $[0, \infty]$-valued functions, the 
    fan trees $(\theta_{\pi}(C), \slp_{\pi})$ and $(\theta_{\pi^{reg}}(C), \slp_{\pi^{reg}})$ 
    coincide. The second one contains more vertices than the first one, labeled by the 
    irreducible components of the exceptional divisor of the minimal resolution 
    $\eta : \Sigma^{reg} \to \Sigma$. The fan tree $\theta_{\pi^{reg}}(C)$
    of the toroidal resolution $\pi^{reg}$ is isomorphic to 
    the dual graph of the boundary $\partial  \Sigma^{reg}$ by an isomorphism which respects the labels 
    of the irreducible components.     
\end{proposition}

The disadvantage of the fan tree $(\theta_{\pi^{reg}}(C), \slp_{\pi^{reg}})$  is that 
one cannot see on it at a glance the partial order of the blow ups leading to the resolution 
$\pi^{reg} : \Sigma \to S$ of $C$. We explained in Subsection \ref{ssec:lotnf} that this order 
may be visualized by using the notion of lotus, for each Newton modification of the 
regularized algorithm obtained by replacing STEP 3 with STEP $3^{reg}$. In order 
to visualize the blow up structure of the resolution process leading to the 
modification $\pi^{reg}: (\Sigma^{reg}, \partial \Sigma^{reg}) \to (S, L + L')$, 
we glue those lotuses using the same rules as those allowing to construct the fan tree 
from its trunks (see Definition \ref{def:fantreetr}):

\begin{definition}  \label{def:lotustoroid}
     Let $C$ be a reduced curve singularity and $(L, L')$ be a cross on the smooth germ $(S,o)$. 
    The {\bf lotus $\boxed{\Lambda_{\pi}(C)}$ of the toroidal pseudo-resolution  
    \index{lotus!of a toroidal pseudo-resolution} 
    $\pi: (\Sigma, \partial \Sigma) \to (S, L + L')$ of $C$} is a 
    simplicial complex of dimension $2$ endowed with a marked oriented edge 
    called its {\bf base}. \index{base!of a lotus} It is obtained by gluing the disjoint union of the lotuses  
    $(\Lambda(\fan_{A_i, B_i}(C)))_{i \in I}$ in the following way:
        \begin{enumerate}
             \item Label each vertex of those lotuses with the corresponding 
                  irreducible component $E_k$,  
                  $L_j$ or $C_l$ of the boundary $\partial \Sigma^{reg}$ of the smooth toroidal surface 
                  $(\Sigma^{reg}, \partial \Sigma^{reg})$.  
             \item Identify all the vertices of $\bigsqcup_{i \in I} \Lambda(\fan_{A_i, B_i}(C))$ 
                  which have the same label. 
                 The result of this identification is $\Lambda_{\pi}(C)$ and the images inside it of the labeled 
                 points of $\bigsqcup_{i \in I} \Lambda(\fan_{A_i, B_i}(C))$ are its vertices. 
                 We keep for each one of them the same label as in the initial lotuses.  
        \end{enumerate}
      Introduce the following terminology for the anatomy of $\Lambda_{\pi}(C)$: 
      \medskip 

              \noindent
              $\bullet$
              The {\bf petals} of $\Lambda_{\pi}(C)$ are the images by the gluing morphism 
                  of the petals of the initial lotuses $(\Lambda(\fan_{A_i, B_i}(C)))_{i \in I}$. 
             
             \noindent
              $\bullet$
              Its {\bf base} is the edge labeled by the 
                  initial cross $(L, L_1)$  and its {\bf basic petal} is the petal having it as base. 

              \noindent
              $\bullet$
              Its {\bf basic vertices} are the images inside it of the basic vertices of 
                 the $2$-dimensional lotuses \linebreak 
                 $(\Lambda(\fan_{A_j, B_j}(C)))_{j \in J}$ 
                 which were not identified with other vertices. 
              
              \noindent
              $\bullet$
              Its {\bf lateral boundary} $\boxed{\partial_+ \Lambda_{\pi}(C)}$ is the image 
                 by the gluing morphism of the union 
                 of the lateral boundaries $(\partial_+ \Lambda(\fan_{A_i, B_i}(C)))_{i \in I}$ 
                 in the sense of Definition \ref{def:lotus-point}. \index{boundary!lateral, of a lotus}
              
              \noindent
              $\bullet$
              Its {\bf lateral vertices} are the vertices of  $\Lambda_{\pi}(C)$ which are not basic.   
              \index{vertex!lateral, of a lotus} 
              
              \noindent
              $\bullet$
              Its {\bf membranes} are the images inside it of the lotuses  $\Lambda(\fan_{A_i, B_i}(C))$  
                 used to construct it. \index{membrane} 
             
              \noindent
              $\bullet$
              Its {\bf Enriques tree} is the union of the Enriques tree of 
                 $\Lambda(\fan_{A_1, B_1}(C))$ (remember that $(A_1, B_1) = (L, L')$) and of the 
                 extended Enriques trees of the other Newton fans $\Lambda(\fan_{A_i, B_i}(C))$ 
               (see Definition \ref{def:Enriquesst}). \index{Enriques!tree}
\end{definition}

We introduce the notion of \emph{Enriques tree} of a lotus in order to be able 
to state point (\ref{point:globalE}) of Theorem \ref{thm:repsailtor} below.  
See also Remark \ref{rem:whyEET}.

\begin{remark}  \label{rem:structlotoroid} 
        The lateral boundary $\partial_+ \Lambda_{\pi}(C)$ is a \emph{covering subtree}
             of the $1$-skeleton of the lotus $\Lambda_{\pi}(C)$, that is, a subtree containing all of its 
             vertices. 
         The membranes of $\Lambda_{\pi}(C)$ may be obtained by removing all the vertices 
             of $\Lambda_{\pi}(C)$ and by taking the closures inside $\Lambda_{\pi}(C)$ of 
             the connected components of the resulting topological space.        
       The lotus $\Lambda_{\pi}(C)$ is a \emph{flag complex}, that is, it may be reconstructed 
            from its $1$-skeleton by filling each complete subgraph with $k$ vertices by 
            a $(k-1)$-dimensional simplex. It turns out that there are such complete subgraphs 
            only for $k \in \{1, 2\}$, for which values of $k$ the filling process adds nothing new, 
            and for $k=3$, for which one gets all the petals of the lotus. 
 \end{remark}

 \begin{figure}
\begin{center}
\begin{tikzpicture}[scale=0.42]

 \begin{scope}[shift={(-7,0)},scale=1]
\foreach \x in {0,1,...,5}{
\foreach \y in {0,1,...,5}{
       \node[draw,circle,inner sep=0.7pt,fill, color=gray!40] at (1*\x,1*\y) {}; }
   }
\draw [->, color=cyan](0,0) -- (0,6);
\draw [->, color=blue](0,0) -- (6,0);
\draw [-, color=orange, very thick](1,0) -- (0, 1) ; 
\node[draw,circle, inner sep=1.5pt,color=black, fill=black] at (0,0){};
\node[draw,circle, inner sep=1.5pt,color=black, fill=black] at (1,0){};
\node[draw,circle, inner sep=1.5pt,color=black, fill=black] at (0,1){};
\node [left] at (0,1) {$e_{L_{1}}$};
\node [below] at (1,0) {$e_L$};
\node [left,above] at (1.5,5) {$e_{E_{3}}$};
\draw [-, thick](0,0) -- (2.3,5.75);
\node[draw,circle, inner sep=1.8pt,color=red, fill=red] at (2,5){};
\node [above] at (2.3,5.75) {$\frac{5}{2}$}; 
\node [below] at (1.4,2) {$e_{E_{2}}$};
\draw [-,thick](0,0) -- (2.8,5.6);
\node[draw,circle, inner sep=1.8pt,color=red, fill=red] at (1,2){};
\node [above] at (3,5.6) {$\frac{2}{1}$};
\node [right] at (5.2,3) {$e_{E_{1}}$};
\draw [-, thick](0,0) -- (5.5,3.3);
\node[draw,circle, inner sep=1.8pt,color=red, fill=red] at (5,3){};
\node [above] at (5.5,3.3) {$\frac{3}{5}$};
\node [below] at (3.5,-0.8) {$\fan_{L, L_1}(C)$};
\end{scope}

\begin{scope}[shift={(-6,0)},scale=1]
   \draw [-, color=orange, very thick](11,0) -- (11, 8) ; 
   \node[draw,circle, inner sep=1.5pt,color=black, fill=black] at (11,0){};
   \node [right] at (11,0) {$L$};
   \node [left] at (11,0) {$0$};
   \node[draw,circle, inner sep=1.5pt,color=black, fill=black] at (11,8){};
   \node [right] at (11,8) {${L_1}$};
   \node [left] at (11,8) {$\infty$};
   \node[draw,circle, inner sep=1.5pt,color=red, fill=red] at (11,2){};
   \node [right] at (11,2) {${E_1}$};
   \node [left] at (11,2) {$\frac{3}{5}$};
   \node[draw,circle, inner sep=1.5pt,color=red, fill=red] at (11,4){};
   \node [right] at (11,4) {${E_2}$};
   \node [left] at (11,4) {$\frac{2}{1}$};
   \node[draw,circle, inner sep=1.5pt,color=red, fill=red] at (11,6){};
   \node [right] at (11,6) {${E_3}$};
   \node [left] at (11,6) {$\frac{5}{2}$};
   \node [below] at (11,-0.8) {$\theta\left(\fan_{L, L_1}(C)\right)$};
 \end{scope}  
 
 \begin{scope}[shift={(10,0)},scale=1.3]
 \draw [->](0,0) -- (0,6);
\draw [->](0,0) -- (6,0);

\draw[fill=pink!40](1,0) -- (0,1) -- (1,1)  --cycle;
\draw[fill=pink!40](1,1) -- (1,2) -- (0,1) --cycle;
\draw[fill=pink!40](1,2) -- (1,3) -- (0,1) --cycle;
\draw[fill=pink!40](1,2) -- (1,3) -- (2,5) --cycle;
\draw[fill=pink!40](1,0) -- (1,1) -- (2,1) --cycle;
\draw[fill=pink!40](1,1) -- (2,1) -- (3,2) --cycle;
\draw[fill=pink!40](2,1) -- (3,2) -- (5,3) --cycle;

\draw [-, ultra thick, color=orange](0,1) -- (2,5) -- (1,2) -- (1,1) -- (5,3) -- (2,1) --(1,0);
\foreach \x in {0,1,...,5}{
\foreach \y in {0,1,...,5}{
       \node[draw,circle,inner sep=0.7pt,fill, color=gray!40] at (1*\x,1*\y) {}; }
   }

\node[draw,circle, inner sep=1.5pt,color=red, fill=red] at (2,5){};
\node[draw,circle, inner sep=1.5pt,color=red, fill=red] at (1,2){};
\node[draw,circle, inner sep=1.5pt,color=red, fill=red] at (5,3){};

\node [left] at (2.4,5.5) {$E_{3}$};
\node [left] at (2.1,2) {$E_{2}$};
\node [left] at (6.2,3) {$E_{1}$};

\node [left] at (1.3,-0.3) {$L$};
\node [left] at (0,1) {$L_{1}$};
\draw [->, very thick, red] (1,0)--(0.5, 0.5);
\draw [-, very thick, red] (0.5, 0.5)--(0,1);
\node [below] at (3,-0.4) {$\Lambda\left(\dfrac{3}{5}, \dfrac{2}{1}, \dfrac{5}{2}\right)$};
 \end{scope}      
     
 \begin{scope}[shift={(-7,-9.5)},scale=1]
\foreach \x in {0,1,...,5}{
\foreach \y in {0,1,...,5}{
       \node[draw,circle,inner sep=0.7pt,fill, color=gray!40] at (1*\x,1*\y) {}; }
   }
\draw [->, color=cyan](0,0) -- (0,6);
\draw [->, color=blue](0,0) -- (6,0);
\draw [-, color=orange, very thick](1,0) -- (0, 1) ; 
\node[draw,circle, inner sep=1.5pt,color=black, fill=black] at (0,0){};
\node[draw,circle, inner sep=1.5pt,color=black, fill=black] at (1,0){};
\node[draw,circle, inner sep=1.5pt,color=black, fill=black] at (0,1){};
\node [left] at (0,1) {$e_{L_{2}}$};
\node [below] at (1,0) {$e_{E_{1}}$};
\node [left,above] at (3.8,3) {$e_{E_{5}}$};
\draw [-, thick](0,0) -- (5.2,3.9);
\node[draw,circle, inner sep=1.8pt,color=red, fill=red] at (4,3){};
\node [right,below] at (3.2,2) {$e_{E_{4}}$};
\draw [-, thick](0,0) -- (5.4,3.6);
\node[draw,circle, inner sep=1.8pt,color=red, fill=red] at (3,2){};
\node [above] at (5.4,3.6) {$\frac{3}{4}$};
\node [right] at (5.4,3.6) {$\frac{2}{3}$};
\node [below] at (3.5,-0.8) {$\fan_{E_{1}, L_2}(C)$};
\end{scope}

 \begin{scope}[shift={(-6,-9.5)},scale=1]
  \draw [-, color=orange, very thick](11,0) -- (11, 6) ; 
   \node[draw,circle, inner sep=1.5pt,color=black, fill=black] at (11,0){};
   \node [right] at (11,0) {$E_1$};
    \node [left] at (11,0) {$0$};
   \node[draw,circle, inner sep=1.5pt,color=black, fill=black] at (11,6){};
   \node [right] at (11,6) {${L_2}$};
   \node [left] at (11,6) {$\infty$};
   \node[draw,circle, inner sep=1.5pt,color=red, fill=red] at (11,2){};
   \node [right] at (11,2) {${E_4}$};
   \node [left] at (11,2) {$\frac{2}{3}$};
   \node[draw,circle, inner sep=1.5pt,color=red, fill=red] at (11,4){};
   \node [right] at (11,4) {${E_5}$};
   \node [left] at (11,4) {$\frac{3}{4}$};
   \node [below] at (11,-0.8) {$\theta\left(\fan_{E_1, L_2}(C)\right)$};
 \end{scope}
 
 \begin{scope}[shift={(10,-9.5)},scale=1]
 \draw [->](0,0) -- (0,6);
\draw [->](0,0) -- (6,0);
\draw[fill=pink!40](1,0) -- (0,1) -- (1,1)  --cycle;
\draw[fill=pink!40](1,0) -- (1,1) -- (2,1) --cycle;
\draw[fill=pink!40](1,1) -- (2,1) -- (3,2) --cycle;
\draw[fill=pink!40](1,1) -- (3,2) -- (4,3) --cycle;
\draw [-, ultra thick, color=orange](0,1) -- (1,1) -- (4,3) -- (1,0);
\foreach \x in {0,1,...,5}{
\foreach \y in {0,1,...,5}{
       \node[draw,circle,inner sep=0.7pt,fill, color=gray!40] at (1*\x,1*\y) {}; }
   }
\node[draw,circle, inner sep=1.5pt,color=red, fill=red] at (3,2){};
\node[draw,circle, inner sep=1.5pt,color=red, fill=red] at (4,3){};
\node [right] at (3,2) {$E_{4}$};
\node [right] at (4,3) {$E_{5}$};
\node [left] at (1.3,-0.4) {$E_1$};
\node [left] at (0,1) {$L_2$};
\draw [->, very thick, red] (1,0)--(0.5, 0.5);
\draw [-, very thick, red] (0.5, 0.5)--(0,1);
\node [below] at (3,-0.4) {$\Lambda\left(\dfrac{2}{3}, \dfrac{3}{4}\right)$};
 \end{scope}      
     

 \begin{scope}[shift={(-7,-19)},scale=1]
\foreach \x in {0,1,...,5}{
\foreach \y in {0,1,...,5}{
       \node[draw,circle,inner sep=0.7pt,fill, color=gray!40] at (1*\x,1*\y) {}; }
   }
\draw [->, color=cyan](0,0) -- (0,6);
\draw [->, color=blue](0,0) -- (6,0);
\draw [-, color=orange, very thick](1,0) -- (0, 1) ; 
\node[draw,circle, inner sep=1.5pt,color=black, fill=black] at (0,0){};
\node[draw,circle, inner sep=1.5pt,color=black, fill=black] at (1,0){};
\node[draw,circle, inner sep=1.5pt,color=black, fill=black] at (0,1){};
\node [left] at (0,1) {$e_{L_{3}}$};
\node [below] at (1,0) {$e_{E_{1}}$};
\node [left,above] at (0.6,3) {$e_{E_{7}}$};
\draw [-, thick](0,0) -- (1.8,5.4);
\node[draw,circle, inner sep=1.8pt,color=red, fill=red] at (1,3){};
\node [right,below] at (3.4,5) {$e_{E_{6}}$};
\draw [-, thick](0,0) -- (3.3,5.5);
\node[draw,circle, inner sep=1.8pt,color=red, fill=red] at (3,5){};
\node [above] at (1.8,5.4) {$\frac{3}{1}$};
\node [above] at (3.3,5.5) {$\frac{5}{3}$};
\node [below] at (3.5,-0.8) {$\fan_{E_{1}, L_3}(C)$};
\end{scope}

 \begin{scope}[shift={(-6,-19)},scale=1]
  \draw [-, color=orange, very thick](11,0) -- (11, 6) ; 
   \node[draw,circle, inner sep=1.5pt,color=black, fill=black] at (11,0){};
   \node [right] at (11,0) {$E_1$};
    \node [left] at (11,0) {$0$};
   \node[draw,circle, inner sep=1.5pt,color=black, fill=black] at (11,6){};
   \node [right] at (11,6) {${L_3}$};
   \node [left] at (11,6) {$\infty$};
   \node[draw,circle, inner sep=1.5pt,color=red, fill=red] at (11,2){};
   \node [right] at (11,2) {${E_6}$};
   \node [left] at (11,2) {$\frac{5}{3}$};
   \node[draw,circle, inner sep=1.5pt,color=red, fill=red] at (11,4){};
   \node [right] at (11,4) {${E_7}$};
   \node [left] at (11,4) {$\frac{3}{1}$};
   \node [below] at (11,-0.8) {$\theta\left(\fan_{E_1, L_3}(C)\right)$};
 \end{scope}

 \begin{scope}[shift={(10,-19)},scale=1]
 \draw [->](0,0) -- (0,6);
\draw [->](0,0) -- (6,0);
\draw[fill=pink!40](1,0) -- (0,1) -- (1,1)  --cycle;
\draw[fill=pink!40](0,1) -- (1,1) -- (1,2) --cycle;
\draw[fill=pink!40](1,1) -- (1,2) -- (1,3) --cycle;
\draw[fill=pink!40](0,1) -- (1,2) -- (1,3) --cycle;
\draw[fill=pink!40](1,1) -- (1,2) -- (2,3) --cycle;
\draw[fill=pink!40](1,2) -- (2,3) -- (3,5) --cycle;
\draw [-, ultra thick, color=orange](0,1) -- (1,3) -- (1,2) -- (3,5) -- (1,1) --(1,0);
\foreach \x in {0,1,...,5}{
\foreach \y in {0,1,...,5}{
       \node[draw,circle,inner sep=0.7pt,fill, color=gray!40] at (1*\x,1*\y) {}; }
   }
\node[draw,circle, inner sep=1.5pt,color=red, fill=red] at (3,5){};
\node[draw,circle, inner sep=1.5pt,color=red, fill=red] at (1,3){};
\node [left] at (3.4,5.6) {$E_{6}$};
\node [left] at (1.1,3) {$E_{7}$};
\node [below] at (1.3,0) {$E_{1}$};
\node [left] at (0,1) {$L_{3}$};
\draw [->, very thick, red] (1,0)--(0.5, 0.5);
\draw [-, very thick, red] (0.5, 0.5)--(0,1);
\node [below] at (3,-0.4) {$\Lambda\left(\dfrac{5}{3}, \dfrac{3}{1}\right)$};
 \end{scope}

 \begin{scope}[shift={(-7,-28.5)},scale=1]
\foreach \x in {0,1,...,5}{
\foreach \y in {0,1,...,5}{
       \node[draw,circle,inner sep=0.7pt,fill, color=gray!40] at (1*\x,1*\y) {}; }
   }
\draw [->, color=cyan](0,0) -- (0,6);
\draw [->, color=blue](0,0) -- (6,0);
\draw [-, color=orange, very thick](1,0) -- (0, 1) ; 
\node[draw,circle, inner sep=1.5pt,color=black, fill=black] at (0,0){};
\node[draw,circle, inner sep=1.5pt,color=black, fill=black] at (1,0){};
\node[draw,circle, inner sep=1.5pt,color=black, fill=black] at (0,1){};
\node [left] at (0,1) {$e_{L_{4}}$};
\node [below] at (1,0) {$e_{E_{6}}$};
\node [left,above] at (2,1) {$e_{E_{8}}$};
\draw [-, thick](0,0) -- (5.4,2.7);
\node[draw,circle, inner sep=1.8pt,color=red, fill=red] at (2,1){};
\node [above] at (5.4,2.7) {$\frac{1}{2}$}; 
\node [below] at (3.5,-0.8) {$\fan_{E_{6}, L_4}(C)$};
\end{scope}

 \begin{scope}[shift={(-6,-28.5)},scale=1]
  \draw [-, color=orange, very thick](11,0) -- (11, 6) ; 
   \node[draw,circle, inner sep=1.5pt,color=black, fill=black] at (11,0){};
   \node [right] at (11,0) {$E_6$};
    \node [left] at (11,0) {$0$};
   \node[draw,circle, inner sep=1.5pt,color=black, fill=black] at (11,6){};
   \node [right] at (11,6) {${L_4}$};
   \node [left] at (11,6) {$\infty$};
   \node[draw,circle, inner sep=1.5pt,color=red, fill=red] at (11,3){};
   \node [right] at (11,3) {${E_8}$};
   \node [left] at (11,3) {$\frac{1}{2}$};
   \node [below] at (11,-0.8) {$\theta\left(\fan_{E_6, L_4}(C)\right)$};
 \end{scope}
 
\begin{scope}[shift={(10,-28.5)},scale=1]
 \draw [->](0,0) -- (0,6);
\draw [->](0,0) -- (6,0);
\draw[fill=pink!40](1,0) -- (0,1) -- (1,1)  --cycle;
\draw[fill=pink!40](1,0) -- (1,1) -- (2,1) --cycle;

\draw [-, ultra thick, color=orange](0,1) -- (2,1) -- (1,0);
\foreach \x in {0,1,...,5}{
\foreach \y in {0,1,...,5}{
       \node[draw,circle,inner sep=0.7pt,fill, color=gray!40] at (1*\x,1*\y) {}; }
   }
\node[draw,circle, inner sep=1.5pt,color=red, fill=red] at (2,1){};

\node [right] at (2,1) {$E_{8}$};
\node [left] at (0,1) {$L_{4}$};
\node [below] at (1,0) {$E_{6}$};
\draw [->, very thick, red] (1,0)--(0.5, 0.5);
\draw [-, very thick, red] (0.5, 0.5)--(0,1);
\node [below] at (3,-0.4) {$\Lambda\left(\dfrac{1}{2}\right)$};
 \end{scope}      

\end{tikzpicture}
\end{center}
 \caption{The $2$-dimensional Newton lotuses of Example \ref{ex:lotustoroid}}  
 \label{fig:example-constrlotustoroid} 
     \end{figure}
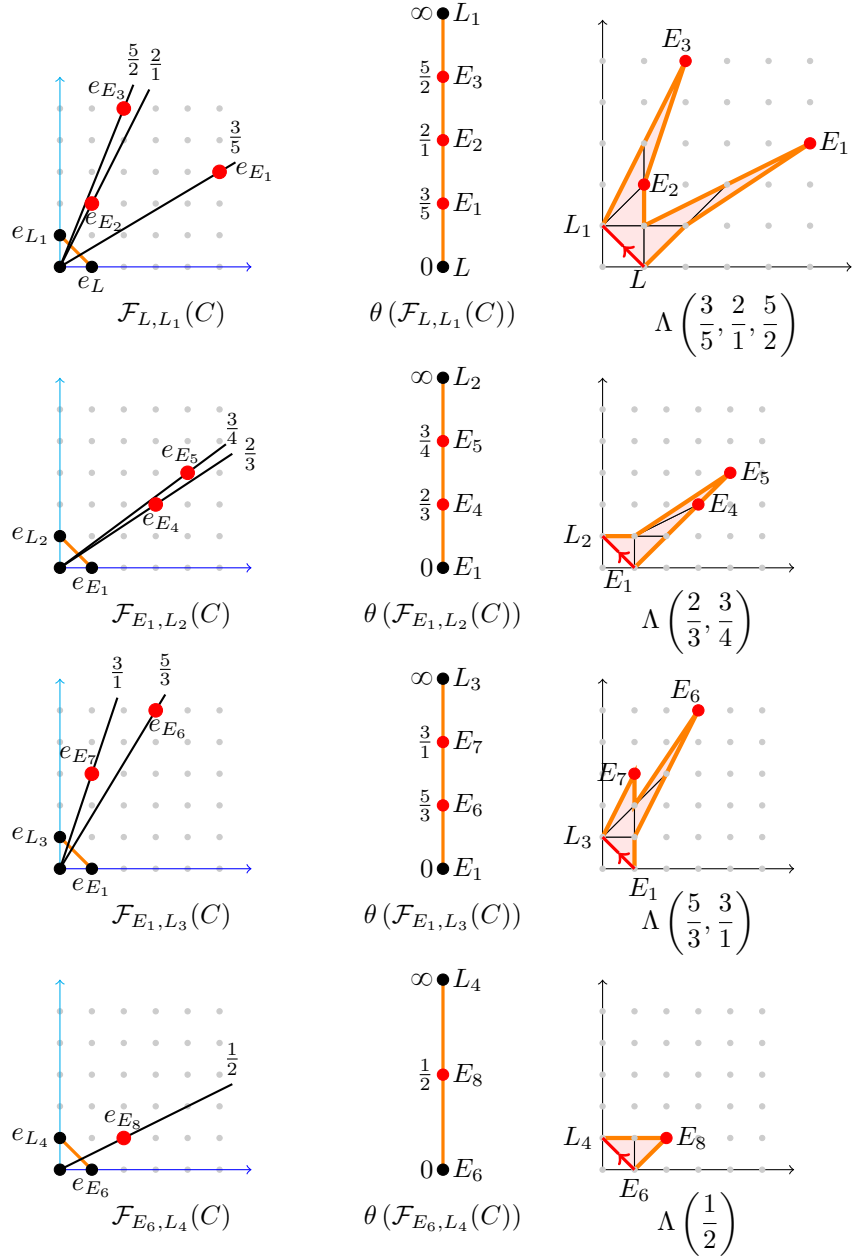


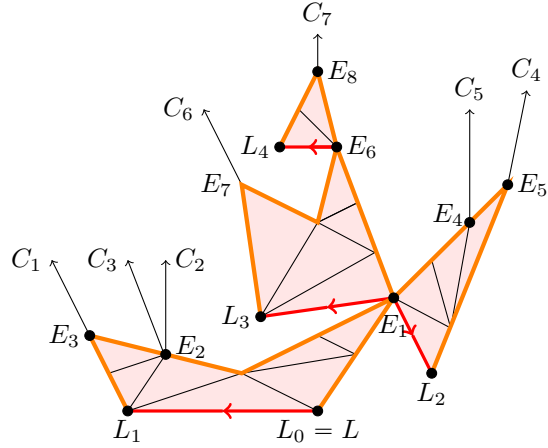
\begin{figure}
     \begin{center}
\begin{tikzpicture}[scale=0.5]
\draw [fill=pink!40](0,0) -- (3,3) -- (1,-2)--(0,0);
\draw [->, very thick, red] (0,0) --(0.5,-1);
 \draw [-, very thick, red] (0.5,-1) --(1,-2);

\draw [-] (0,0)--(1.5,-0.8);
\draw [-] (1,1)--(1.5,-0.8);
\draw [-] (2,2)--(1.5,-0.8);
\draw [->] (3,3)--(3.5,5.5);
\node [above] at (3.5,5.5) {$C_4$};
\draw [->] (2,2)--(2,5);
\node [above] at (2,5) {$C_5$};
\node [below] at (0,-0.2) {$E_1$};

\draw [fill=pink!40](0,0) -- (-2,-3)--(-7,-3)--(-8,-1)--(-4,-2)--(0,0);
       \draw [->, very thick, red] (-2,-3)--(-4.5,-3);
       \draw [-, very thick, red] (-4.5,-3)--(-7,-3);
\draw [-] (-4,-2)--(-2,-3);
\draw [-] (-2,-1)--(-1,-1.5);
\draw [-] (-4,-2)--(-1,-1.5);
\draw [-] (-4,-2)--(-7,-3);
\draw [-] (-7,-3)--(-6,-1.5);
\draw [-] (-7.5,-2)--(-6,-1.5);
\draw [->] (-8,-1)--(-9,1);
\node [left] at (-9,1) {$C_1$};
\draw [->] (-6,-1.6)--(-7,1);
\node [left] at (-7,1) {$C_3$};
\draw [->] (-6,-1.6)--(-6,1);
\node [right] at (-6,1) {$C_2$};

\draw [fill=pink!40](0,0) -- (-3.5,-0.5)--(-4,3)--(-2,2)--(-1.5,4)--(0,0);
    \draw [->, very thick, red]  (0,0) -- (-1.75,-0.25);
    \draw [-, very thick, red]  (-1.75,-0.25)--(-3.5,-0.5);
\draw [-] (-3.5,-0.5)--(-2,2);
\draw [-] (-3.5,-0.5)--(-0.5,1.2);
\draw [-] (-2,2)--(-1,2.5);
\draw [-] (-2,2)--(-1,2.5);
\draw [-] (-0.5,1.2)--(-2,2);
\draw [->] (-4,3)--(-5,5);
\node [left] at (-5,5) {$C_6$};
\node [left] at (-4,3) {$E_7$};

\draw [fill=pink!40](-1.5,4) -- (-2,6)--(-3,4)--(-1.5,4);
     \draw [-, very thick, red] (-1.5,4)--(-3,4);
     \draw [->, very thick, red] (-1.5,4)--(-2.25,4);
\draw [-] (-1.5,4)--(-2.5,5);
\node [left] at (-3,4) {$L_4$};
\draw [->] (-2,6)--(-2,7);
\node [above] at (-2,7) {$C_7$};

\draw [-, ultra thick, color=orange]  (-2, -3) -- (0,0) -- (-4, -2) -- (-6, -1.5) -- (-8, -1) -- (-7,-3);
\draw [-, ultra thick, color=orange] (0,0) -- (3,3) -- (1,-2);
\draw [-, ultra thick, color=orange] (0,0) -- (-1.5,4) -- (-2,2) -- (-4,3) -- (-3.5,-0.5);
\draw [-, ultra thick, color=orange] (-1.5,4) -- (-2,6) -- (-3,4);

\node[draw,circle,inner sep=1.3pt,fill=black] at (2,2){};
\node [left] at (2.1,2.2) {$E_4$};
\node[draw,circle,inner sep=1.3pt,fill=black] at (3,3){};
\node [right] at (3,3) {$E_5$};
\node[draw,circle,inner sep=1.3pt,fill=black] at (1,-2){};
\node [below] at (1,-2) {$L_2$};

\node[draw,circle,inner sep=1.3pt,fill=black] at (-2,-3){};
\node [below] at (-2,-3) {$L_0=L$};
\node[draw,circle,inner sep=1.3pt,fill=black] at (-7,-3){};
\node [below] at (-7,-3) {$L_1$};
\node[draw,circle,inner sep=1.3pt,fill=black] at (-8,-1){};
\node [left] at (-8,-1) {$E_3$};
\node[draw,circle,inner sep=1.3pt,fill=black] at (-6,-1.5){};
\node [right] at (-6,-1.3) {$E_2$};

\node[draw,circle,inner sep=1.3pt,fill=black] at (-3.5,-0.5){};
\node [left] at (-3.5,-0.5) {$L_3$};


\node[draw,circle,inner sep=1.3pt,fill=black] at (-3,4){};
\node[draw,circle,inner sep=1.3pt,fill=black] at (-2,6){};
\node [right] at (-2,6) {$E_8$};
\node[draw,circle,inner sep=1.3pt,fill=black] at (-1.5,4){};
\node [right] at (-1.5,4) {$E_6$};
\node[draw,circle,inner sep=1.3pt,fill=black] at (0,0){};

\end{tikzpicture}
\end{center}
\caption{The lotus of the toroidal pseudo-resolution of Example \ref{ex:lotustoroid}}  
   \label{fig:lotustoroid}
    \end{figure}


    \begin{figure}
\begin{center}
\begin{tikzpicture}[scale=0.5]


 \begin{scope}[shift={(-6,-8)},scale=1.5]
  
 \draw [-, color=orange, very thick](0,0) -- (0, 8) ; 
    \draw [-, color=orange, very thick](0, 2) -- (6, 2) ; 
     \draw [-, color=orange, very thick](0, 2) -- (6, 8) ; 
     \draw [-, color=orange, very thick](2,4) -- (2, 8) ; 
     \draw [-, color=magenta, very thick](0,6) -- (-2, 6) ; 
     \draw [-, color=magenta, very thick](2,2) -- (2, 0) ; 
     \draw [-, color=magenta, very thick](4,6) -- (6, 6) ; 
     \draw [-, color=magenta, very thick](2,7) -- (4, 7) ; 
     \draw [-, color=magenta, very thick](4,2) -- (4, 0) ; 
     \draw [-, color=magenta, very thick](0,4) -- (-2, 4) ; 
     \draw [-, color=magenta, very thick](0,4) -- (-2, 2) ;
      \draw [-, color=blue, very thick](0,0) -- (0, 4) ; 
      \draw [-, color=blue, very thick](0,6) -- (0, 8) ; 
      \draw [-, color=blue, very thick](0,2) -- (2, 2) ; 
      \draw [-, color=blue, very thick](4,2) -- (6, 2) ; 
      \draw [-, color=blue, very thick](0,2) -- (4,6) ;
      \draw [-, color=blue, very thick](2,7) -- (2,8) ;  

   \node[draw,circle, inner sep=1.5pt,color=black, fill=black] at (0,0){};
   \node [right] at (0,0) {$L$};
   \node [left] at (0,0) {$0$};
   \node[draw,circle, inner sep=1.5pt,color=black, fill=black] at (0,8){};
   \node [right] at (0,8) {${L_1}$};
   \node [left] at (0,8) {$\infty$};
   \node[draw,circle, inner sep=1.5pt,color=red, fill=red] at (0,2){};
   \node [right] at (0,1.5) {${E_1}$};
   \node [left] at (0,1.5) {$\frac{3}{5}$};
   \node[draw,circle, inner sep=1.5pt,color=red, fill=red] at (0,4){};
   \node [right] at (0,4) {${E_2}$};
   \node [left] at (0,4.5) {$\frac{2}{1}$};
   \node[draw,circle, inner sep=1.5pt,color=red, fill=red] at (0,6){};
   \node [right] at (0,6) {${E_3}$};
   \node [left] at (0,6.5) {$\frac{5}{2}$};
   \node[draw,circle, inner sep=1.5pt,color=black!20!green, fill=black!20!green] at (-2,6){};
   \node [below] at (-2,6) {${C_1}$};
   \node [above] at (-2,6) {$\infty$};   
   \node[draw,circle, inner sep=1.5pt,color=black!20!green, fill=black!20!green] at (-2,4){};
   \node [below] at (-2,4) {${C_2}$};
   \node [above] at (-2,4) {$\infty$}; 
   \node[draw,circle, inner sep=1.5pt,color=black!20!green, fill=black!20!green] at (-2,2){};
   \node [below] at (-2,2) {${C_3}$};
   \node [above] at (-2,2) {$\infty$}; 
   \node[draw,circle, inner sep=1.5pt,color=black, fill=black] at (6,2){};
   \node [below] at (6,2) {${L_2}$};
   \node [above] at (6,2) {$\infty$};
   \node[draw,circle, inner sep=1.5pt,color=red, fill=red] at (2,2){};
   \node [below] at (2.5,2) {${E_4}$};
   \node [above] at (2,2) {$\frac{2}{3}$};
   \node[draw,circle, inner sep=1.5pt,color=red, fill=red] at (4,2){};
   \node [below] at (4.5,2) {${E_5}$};
   \node [above] at (4,2) {$\frac{3}{4}$};
   \node[draw,circle, inner sep=1.5pt,color=black!20!green, fill=black!20!green] at (2,0){};
   \node [right] at (2,0) {${C_5}$};
   \node [left] at (2,0) {$\infty$}; 
   \node[draw,circle, inner sep=1.5pt,color=black!20!green, fill=black!20!green] at (4,0){};
   \node [right] at (4,0) {${C_4}$};
   \node [left] at (4,0) {$\infty$}; 
   \node[draw,circle, inner sep=1.5pt,color=black, fill=black] at (6,8){};
   \node [below] at (6,7.8) {${L_3}$};
   \node [above] at (6,8.2) {$\infty$};
   \node[draw,circle, inner sep=1.5pt,color=red, fill=red] at (4,6){};
   \node [below] at (4,5.8) {${E_7}$};
   \node [left] at (4,6.2) {$\frac{3}{1}$};
   \node[draw,circle, inner sep=1.5pt,color=red, fill=red] at (2,4){};
   \node [right] at (2.2,4) {${E_6}$};
   \node [left] at (2,4.2) {$\frac{5}{3}$};
     \node[draw,circle, inner sep=1.5pt,color=black!20!green, fill=black!20!green] at (6, 6){};
   \node [below] at (6, 6) {${C_6}$};
   \node [above] at (6, 6) {$\infty$};   
   \node[draw,circle, inner sep=1.5pt,color=black, fill=black] at (2,8){};
   \node [right] at (2,8) {${L_4}$};
   \node [left] at (2,8) {$\infty$};
   \node[draw,circle, inner sep=1.5pt,color=red, fill=red] at (2,7){};
   \node [right] at (2,6.5) {${E_8}$};
   \node [left] at (2,6.5) {$\frac{1}{2}$};
   \node[draw,circle, inner sep=1.5pt,color=black!20!green, fill=black!20!green] at (4, 7){};
   \node [below] at (4, 7) {${C_7}$};
   \node [above] at (4, 7) {$\infty$};
 \node [below] at (3, -1) {$\theta_{\pi}(C)$}; 
\end{scope}

\begin{scope}[shift={(18,-4)},scale=1.2]
     \draw [fill=pink!40](0,0) -- (3,3) -- (1,-2)--(0,0);
\draw [->, very thick, red] (0,0) --(0.5,-1);
 \draw [-, very thick, red] (0.5,-1) --(1,-2);
\draw [-] (0,0)--(1.5,-0.8);
\draw [-] (1,1)--(1.5,-0.8);
\draw [-] (2,2)--(1.5,-0.8);
\draw [->] (3,3)--(3.5,5.5);
\node [above] at (3.5,5.5) {$C_4$};
\draw [->] (2,2)--(2,5);
\node [above] at (2,5) {$C_5$};

\draw [fill=pink!40](0,0) -- (-2,-3)--(-7,-3)--(-8,-1)--(-4,-2)--(0,0);
       \draw [->, very thick, red] (-2,-3)--(-4.5,-3);
       \draw [-, very thick, red] (-4.5,-3)--(-7,-3);
\draw [-] (-4,-2)--(-2,-3);
\draw [-] (-2,-1)--(-1,-1.5);
\draw [-] (-4,-2)--(-1,-1.5);
\draw [-] (-4,-2)--(-7,-3);
\draw [-] (-7,-3)--(-6,-1.5);
\draw [-] (-7.5,-2)--(-6,-1.5);
\draw [->] (-8,-1)--(-9,1);
\node [left] at (-9,1) {$C_1$};
\draw [->] (-6,-1.6)--(-7,1);
\node [left] at (-7,1) {$C_3$};
\draw [->] (-6,-1.6)--(-6,1);
\node [right] at (-6,1) {$C_2$};

\draw [fill=pink!40](0,0) -- (-3.5,-0.5)--(-4,3)--(-2,2)--(-1.5,4)--(0,0);
    \draw [->, very thick, red]  (0,0) -- (-1.75,-0.25);
    \draw [-, very thick, red]  (-1.75,-0.25)--(-3.5,-0.5);
\draw [-] (-3.5,-0.5)--(-2,2);
\draw [-] (-3.5,-0.5)--(-0.5,1.2);
\draw [-] (-2,2)--(-1,2.5);
\draw [-] (-2,2)--(-1,2.5);
\draw [-] (-0.5,1.2)--(-2,2);
\draw [->] (-4,3)--(-5,5);
\node [left] at (-5,5) {$C_6$};

\draw [fill=pink!40](-1.5,4) -- (-2,6)--(-3,4)--(-1.5,4);
     \draw [-, very thick, red] (-1.5,4)--(-3,4);
     \draw [->, very thick, red] (-1.5,4)--(-2.25,4);
\draw [-] (-1.5,4)--(-2.5,5);
\node [left] at (-3,4) {$L_4$};
\draw [->] (-2,6)--(-2,7);
\node [above] at (-2,7) {$C_7$};

\draw [-, ultra thick, color=orange]  (-2, -3) -- (0,0) -- (-4, -2) -- (-6, -1.5) -- (-8, -1) -- (-7,-3);
\draw [-, ultra thick, color=orange] (0,0) -- (3,3) -- (1,-2);
\draw [-, ultra thick, color=orange] (0,0) -- (-1.5,4) -- (-2,2) -- (-4,3) -- (-3.5,-0.5);
\draw [-, ultra thick, color=orange] (-1.5,4) -- (-2,6) -- (-3,4);

\draw [-, ultra thick, color=blue]  (-2, -3) -- (0,0) -- (-4, -2) -- (-6, -1.5); 
\draw [-, ultra thick, color=blue] (-8, -1) -- (-7,-3);
    \draw [-, ultra thick, color=blue] (0,0) -- (2,2);
    \draw [-, ultra thick, color=blue] (3,3) -- (1,-2);
\draw [-, ultra thick, color=blue] (0,0) -- (-1.5,4) -- (-2,2) -- (-4,3); 
     \draw [-, ultra thick, color=blue] (-2,6) -- (-3,4);

\node[draw,circle,inner sep=1.3pt,color=red,fill=red] at (2,2){};
\node [left] at (2.1,2.2) {$E_4$};
\node[draw,circle,inner sep=1.3pt,color=red,fill=red] at (3,3){};
\node [right] at (3,3) {$E_5$};
\node[draw,circle,inner sep=1.3pt,fill=black] at (1,-2){};
\node [below] at (1,-2) {$L_2$};
\node[draw,circle,inner sep=1.3pt,fill=black] at (-2,-3){};
\node [below] at (-2,-3) {$L_0=L$};
\node[draw,circle,inner sep=1.3pt,fill=black] at (-7,-3){};
\node [below] at (-7,-3) {$L_1$};
\node[draw,circle,inner sep=1.3pt,color=red, fill=red] at (-8,-1){};
\node [left] at (-8,-1) {$E_3$};
\node[draw,circle,inner sep=1.3pt,color=red,fill=red] at (-6,-1.5){};
\node [right] at (-6,-1.3) {$E_2$};
\node[draw,circle,inner sep=1.3pt,fill=black] at (-3.5,-0.5){};
\node [left] at (-3.5,-0.5) {$L_3$};
\node[draw,circle,inner sep=1.3pt,fill=black] at (-3,4){};
\node[draw,circle,inner sep=1.3pt,color=red,fill=red] at (-2,6){};
\node [right] at (-2,6) {$E_8$};
\node[draw,circle,inner sep=1.3pt,color=red,fill=red] at (-1.5,4){};
\node [right] at (-1.5,4) {$E_6$};
\node[draw,circle,inner sep=1.3pt,fill=black] at (0,0){};
\node[draw,circle,inner sep=1.3pt,color=red,fill=red] at (-4,3){};
\node [left] at (-4,3) {$E_7$};
\node[draw,circle,inner sep=1.3pt,color=red,fill=red] at (0,0) {};
\node [below] at (-0.1,-0.3) {$E_1$};

 \node [above] at (-3, -5.5) {$\Lambda_{\pi}(C)$}; 
\end{scope}

    \end{tikzpicture}
\end{center}
 \caption{Comparison of the fan tree and the lotus of Example \ref{ex:lotustoroid}}  
 \label{fig:example-fantree-lotus} 
     \end{figure}
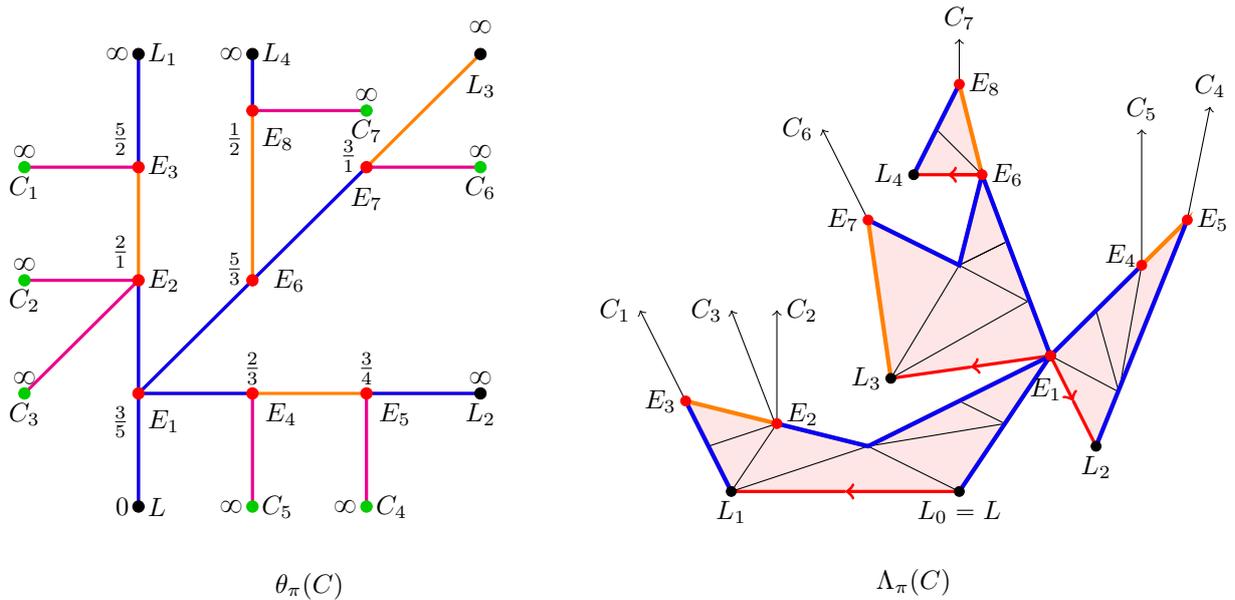

\begin{example}   \label{ex:lotustoroid}
     Consider the toroidal pseudo-resolution process of Example \ref{ex:toroidres}. The 
     construction of the corresponding fan tree was explained in Example \ref{ex:dualisom} 
     and illustrated in Figure \ref{fig:example-fantree}. The left column of 
     Figure \ref{fig:example-constrlotustoroid} represents the Newton fans produced 
     each time one runs STEP 2 of Algorithm \ref{alg:tores}. The middle column shows the 
     associated trunks and the right column the corresponding lotuses. 

     The associated lotus $\Lambda_{\pi}(C)$
     is represented in Figure \ref{fig:lotustoroid}. It has $4$ membranes of dimension $2$ 
     and $7$ membranes of dimension $1$. The oriented base of each lotus 
     $\Lambda(\fan_{A_i, B_i}(C))$ 
     used to construct it is indicated in red. The base of $\Lambda_{\pi}(C)$ is the 
     oriented edge whose vertices are labeled by $L$ and $L_1$. The basic vertices of 
     $\Lambda_{\pi}(C)$ are those labeled by $L, L_1, L_2, L_3, L_4$. The part of the  
     lateral boundary $\partial_+ \Lambda_{\pi}(C)$ contained in the $2$-dimensional lotuses 
     $(\Lambda(\fan_{A_j, B_j}(C)))_{j \in J}$  is represented in orange. In order to get the 
     whole lateral boundary, one has to add the $1$-dimensional lotuses of the fans associated 
     to the crosses at which one stops at STEP 1, that is, the segments 
     $[E_3, C_1]$, $[E_2, C_2]$, $[E_2, C_3]$, $[E_5, C_4]$, $[E_4, C_5]$, $[E_7, C_6]$ and 
     $[E_8, C_7]$. 
     
     In Figure \ref{fig:example-fantree-lotus}  are represented side by side the fan tree $\theta_{\pi}(C)$ 
      and the lotus $\Lambda_{\pi}(C)$. Note that the fan tree is homeomorphic (forgetting the values 
      of the slope function at its vertices) with the lateral boundary $\partial_+ \Lambda_{\pi}(C)$, 
      by a homeomorphism which preserves the labels. This is a general fact, as formulated in 
      point (\ref{point:latboundary}) of Theorem \ref{thm:repsailtor} below. This homeomorphism 
      is not an isomorphism of trees because some of the edges of the fan tree -- the blue ones -- 
      get subdivided in the lateral boundary of the lotus. Those are precisely the edges which 
      correspond to the singular points of the surface $\Sigma$. One may see on the lateral 
      boundary the structure of the exceptional divisor of the minimal resolution of each such 
      point. 
      
      For instance, the intersection point of the curves $E_1$ and $E_6$ on $\Sigma$ gets resolved 
      by replacing that point with an exceptional divisor with two components. Their 
      self-intersection numbers in the smooth surface $\Sigma^{reg}$ are $-4$ and $-3$, as 
      results from point (\ref{point:selfint}) of Theorem \ref{thm:repsailtor}. 
\end{example}

Here comes the announced visualization  of the structure of the decomposition of 
the modification $\pi^{reg}: \Sigma^{reg} \to S$ 
into blow ups of points in terms of the anatomy of the lotus $\Lambda_{\pi}(C)$ 
(see Definition \ref{def:lotustoroid}): 

\begin{theorem} \label{thm:repsailtor}
     Let $C$ be a reduced curve singularity on the smooth germ of surface $(S,o)$. 
     Consider a toroidal pseudo-resolution $\pi: (\Sigma, \partial \Sigma) \to (S, L + L')$ 
     of $C$ produced by Algorithm \ref{alg:tores}. Its lotus $\Lambda_{\pi}(C)$ 
    represents the following aspects of the associated embedded resolution  
     $\pi^{reg}: (\Sigma^{reg}, \partial \Sigma^{reg}) \to (S, L + L')$:   
      \begin{enumerate}
          \item \label{point:basicedges} 
                Its basic edges represent the crosses with respect to which  
               STEP 2 of Algorithm \ref{alg:tores} was applied. 
          \item \label{point:basicvert} 
                     Its basic vertices represent the branches $(L_j)_{j \in J}$ of the crosses 
                   used during the process, which were introduced each time one 
                   executed STEP 2.
           \item \label{point:latvert}
               Its lateral vertices represent the irreducible components $E_k$ of the exceptional 
               divisor $(\pi^{reg})^{-1}(o)$ of the smooth modification $\pi^{reg}:  \Sigma^{reg}  \to S$. 
           \item \label{point:latboundary}
                Its lateral boundary  $\partial_+ \Lambda_{\pi}(C)$
                is the dual graph of the boundary divisor $\partial \Sigma^{reg}$ 
              and is homeomorphic with the fan tree $\theta_{\pi^{reg}}(C)$, 
              by a homeomorphism which respects the labels.
           \item \label{point:selfint} 
                The opposite of the number of petals 
                of $\Lambda_{\pi}(C)$ containing a given lateral vertex is the 
                self-intersection number of the irreducible component of $(\pi^{reg})^{-1}(o)$ 
                represented by that lateral vertex.
           \item \label{point:edgeslot}
               The edges of $\Lambda_{\pi}(C)$ represent the affine charts used in the 
               decomposition of $\pi$ 
              into a composition of blow ups of points, and the pairs of 
              irreducible components of $(\pi^{reg})^{-1}(\sum_{j \in J} L_j )$ 
              which are strict transforms of 
              crosses used at some stage of the composition of blow ups.    
             \item \label{point:graphprox}
                The graph of the proximity binary relation on  the constellation which is blown up 
                is the full subgraph of the $1$-skeleton 
                of the lotus $\Lambda_{\pi}(C)$ on its set of non-basic vertices.
            \item   \label{point:globalE}
                The Enriques tree of $\Lambda_{\pi}(C)$ is the 
               Enriques diagram of the constellation of infinitely near points  
               at which are based the crosses introduced during the blow up process leading 
               to the boundary $\partial \Sigma^{reg}$. 
      \end{enumerate}
\end{theorem}
\begin{proof}
        Points \eqref{point:basicedges} and \eqref{point:basicvert} 
        result from Proposition \ref{prop:propstrict}. 
        Points \eqref{point:latvert} and \eqref{point:latboundary} result from 
           Propositions \ref{prop:fantreedualgr}, \ref{prop:lotusdecomp} and 
           \ref{prop:morelabels}.  
         Point \eqref{point:selfint} results from 
         Corollary \ref{cor:degHopfbis} and Proposition \ref{prop:changeselfint}. 
         A prototype of this result had been stated in \cite[Thm. 6.2]{PP 11}.  
         Points \eqref{point:edgeslot} and \eqref{point:graphprox} result  from Proposition 
         \ref{prop:lotusinterprproxim}. 
         Point \eqref{point:globalE} results from Proposition \ref{prop:lotusinterprenriques}. 
\end{proof}

\begin{figure}
    \begin{center}
\begin{tikzpicture}[scale=0.45]

\draw [thick, color=black!20!green] (0,0) -- (5,0)--(1,5)--cycle;
\draw [-,color=blue, very thick] (0,0) -- (3.5,1.9)--(0.7,3.7);
\node [below] at (0,0) {{\tiny $L_{1}$}};
\node[draw,circle, inner sep=1pt,color=black!20!green, fill=black!20!green] at (0,0){};
\node [below] at (5,0) {{\tiny $L$}};
\node[draw,circle, inner sep=1pt,color=black!20!green, fill=black!20!green] at (5,0){};
\draw [-,color=violet] (0,0) -- (3.9,1.4);
\draw[-, thick, color=violet, dotted] (3.5,0.6)--(3.5,1.1);
\draw [-,color=violet] (0,0) -- (4.5,0.6);
\draw [-,color=violet] (3.5,1.9) -- (0.5,2.3);
\draw[-, thick, color=violet, dotted] (1,1.6)--(1,2.1);
\draw [-,color=violet] (3.5,1.9) -- (0.27,1.3);

\draw [-,color=blue, very thick] (0.7,3.7) -- (1.1,3.8);
\draw[-, thick, color=blue, dotted] (1.1,3.8)--(1.7,3.905);

\draw [decorate,decoration={brace,amplitude=10pt},xshift=-4pt,yshift=0pt]
(0,0) --(0.7,3.7) node [black,midway,xshift=-0.6cm]
{$q_1$};

\draw [decorate,decoration={brace,amplitude=10pt, mirror, raise=4pt},yshift=-4pt,xshift=0pt]
(5,0) -- (3.5,1.9) node  [black,midway,yshift=0.6cm]
{$\;$};
\node [right] at (4.8,1.3) {{ $p_{1}$}};

\begin{scope}[shift={(-4,5)}]
\draw [thick, color=black!20!green] (0,0) -- (5,0)--(1,5)--cycle;
\draw [-,color=blue, very thick] (0,0) -- (3.5,1.9)--(0.7,3.7);
\node [below] at (0,0) {{\tiny $L_{2}$}};
\node[draw,circle, inner sep=1pt,color=black!20!green, fill=black!20!green] at (0,0){};
\node[draw,circle, inner sep=1pt,color=black!20!green, fill=black!20!green] at (5,0){};
\draw [-,color=violet] (0,0) -- (3.9,1.4);
\draw[-, thick, color=violet, dotted] (3.5,0.6)--(3.5,1.1);
\draw [-,color=violet] (0,0) -- (4.5,0.6);
\draw [-,color=violet] (3.5,1.9) -- (0.5,2.3);
\draw[-, thick, color=violet, dotted] (1,1.6)--(1,2.1);
\draw [-,color=violet] (3.5,1.9) -- (0.27,1.3);

\draw [-,color=blue, very thick] (0.7,3.7) -- (1.1,3.8);
\draw[-, thick, color=blue, dotted] (1.1,3.8)--(1.7,3.905);

\draw [decorate,decoration={brace,amplitude=10pt},xshift=-4pt,yshift=0pt]
(0,0) --(0.7,3.7) node [black,midway,xshift=-0.6cm]
{$q_2$};

\draw [decorate,decoration={brace,amplitude=10pt, mirror, raise=4pt},yshift=-4pt,xshift=0pt]
(5,0) -- (3.5,1.9) node  [black,midway,yshift=0.6cm]
{$\;$};
\node [right] at (4.8,1.3) {{ $p_{2}$}};
\draw[-, thick, color=black!20!green, dotted] (1,5)--(0,6.25);

\end{scope}

\begin{scope}[shift={(-9,11.25)}]
\draw [thick, color=black!20!green] (0,0) -- (5,0)--(1,5)--cycle;
\draw [-,color=blue, very thick] (0,0) -- (3.5,1.9)--(0.7,3.7);
\node [below] at (0,0) {{\tiny $L_{k}$}};
\node[draw,circle, inner sep=1pt,color=black!20!green, fill=black!20!green] at (0,0){};
\node[draw,circle, inner sep=1pt,color=black!20!green, fill=black!20!green] at (5,0){};
\draw [-,color=violet] (0,0) -- (3.9,1.4);
\draw[-, thick, color=violet, dotted] (3.5,0.6)--(3.5,1.1);
\draw [-,color=violet] (0,0) -- (4.5,0.6);
\draw [-,color=violet] (3.5,1.9) -- (0.5,2.3);
\draw[-, thick, color=violet, dotted] (1,1.6)--(1,2.1);
\draw [-,color=violet] (3.5,1.9) -- (0.27,1.3);

\draw [-,color=blue, very thick] (0.7,3.7) -- (1.1,3.8);
\draw[-, thick, color=blue, dotted] (1.1,3.8)--(1.7,3.905);

\draw [decorate,decoration={brace,amplitude=10pt},xshift=-4pt,yshift=0pt]
(0,0) --(0.7,3.7) node [black,midway,xshift=-0.6cm]
{$q_k$};

\draw [decorate,decoration={brace,amplitude=10pt, mirror, raise=4pt},yshift=-4pt,xshift=0pt]
(5,0) -- (3.5,1.9) node  [black,midway,yshift=0.6cm]
{$\;$};
\node [right] at (4.8,1.3) {{ $p_{k}$}};
\draw[->, thick, color=black!20!green] (1,5)--(1,5.8);
\node [right] at (1,5.6) {{ $C$}};
\end{scope}
\end{tikzpicture}
\end{center}
 \caption{The lotus of toroidal pseudo-resolution for one branch from Example \ref{ex:onebranch}}
 \label{fig:branchlotus}
   \end{figure}
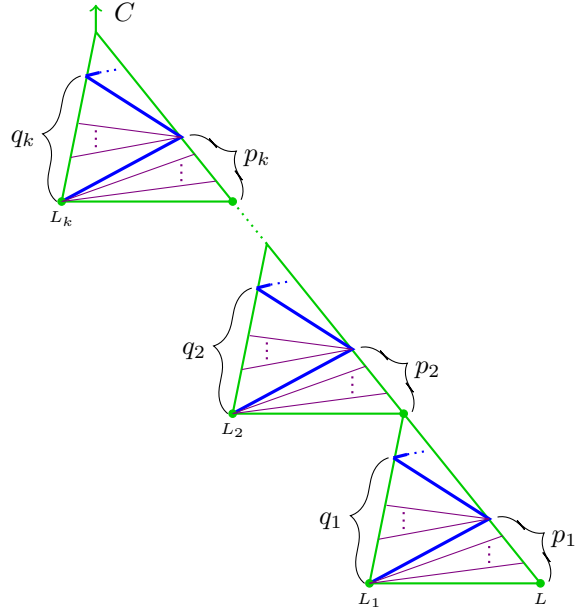

\begin{example}  \label{ex:onebranch}
   Assume that $C$ is a branch.  Its fan tree $\theta_{\pi}(C)$ is a segment $[L, C]$. Denote 
   its interior vertices by $P_1 \prec_L \cdots  \prec_L P_k =P$, with $k \geq 1$. Here 
   $\preceq_L$ denotes the total order on $\theta_{\pi}(C)$ induced by the root $L$. 
    Consider the continued fraction expansions of their slopes
    $\slp_{\pi}(P_j) =   [p_j, q_j, \dots]$,  for all $ j \in \{1, \dots, k \}$. 
     Then the lotus $\Lambda_{\pi}(C)$ is represented in Figure \ref{fig:branchlotus}. 
     We explain in Examples  \ref{ex:branchPuiseux} and \ref{ex:continex} below 
     how to give examples 
     of branches which admit a pseudo-resolution process with such a lotus. 
\end{example}

  \begin{figure}
\begin{center}
\begin{tikzpicture}[scale=0.58]


  \draw [->](0,0) -- (0,6);
\draw [->](0,0) -- (6,0);

\draw[fill=pink!40](1,0) -- (0,1) -- (1,1)  --cycle;
\draw[fill=pink!40](1,1) -- (1,2) -- (0,1) --cycle;
\draw[fill=pink!40](1,2) -- (1,3) -- (0,1) --cycle;
\draw[fill=pink!40](1,2) -- (1,3) -- (2,5) --cycle;
\draw[fill=pink!40](1,0) -- (1,1) -- (2,1) --cycle;
\draw[fill=pink!40](1,1) -- (2,1) -- (3,2) --cycle;
\draw[fill=pink!40](2,1) -- (3,2) -- (5,3) --cycle;

\draw [-, ultra thick, color=orange](0,1) -- (2,5) -- (1,2) -- (1,1) --
(5,3) -- (2,1) --(1,0);
\foreach \x in {0,1,...,5}{
\foreach \y in {0,1,...,5}{
        \node[draw,circle,inner sep=0.7pt,fill, color=gray!40] at
(1*\x,1*\y) {}; }
    }
\node [left] at (2.4,5.5) {$E_{3}$};
\node [left] at (2.1,2) {$E_{2}$};
\node [left] at (6.2,3) {$E_{1}$};

\node [left] at (1.3,-0.3) {$L$};
\node [left] at (0,1) {$L_{1}$};
\draw [->, very thick, red] (1,0)--(0.5, 0.5);
\draw [-, very thick, red] (0.5, 0.5)--(0,1);
\node [below] at (3,-0.4) {$\Lambda\left(\dfrac{3}{5}, \dfrac{2}{1},
\dfrac{5}{2}\right)$};

\draw [-, line width=2.5pt, color=magenta] (1,1) -- (2,1) -- (3,2) --
(5,3);
\draw [-, line width=2.5pt, color=magenta] (1,1) -- (1,3) -- (2,5);
\node[draw,circle, inner sep=3pt,color=magenta, fill=magenta] at
(1,1){};
\node[draw,circle, inner sep=1.5pt,color=red, fill=red] at (2,5){};
\node[draw,circle, inner sep=1.5pt,color=red, fill=red] at (1,2){};
\node[draw,circle, inner sep=1.5pt,color=red, fill=red] at (5,3){};


\begin{scope}[shift={(10,0)},scale=1]
  \draw [->](0,0) -- (0,6);
\draw [->](0,0) -- (6,0);

\draw[fill=pink!40](1,0) -- (0,1) -- (1,1)  --cycle;
\draw[fill=pink!40](1,1) -- (1,2) -- (0,1) --cycle;
\draw[fill=pink!40](1,2) -- (1,3) -- (0,1) --cycle;
\draw[fill=pink!40](1,2) -- (1,3) -- (2,5) --cycle;
\draw[fill=pink!40](1,0) -- (1,1) -- (2,1) --cycle;
\draw[fill=pink!40](1,1) -- (2,1) -- (3,2) --cycle;
\draw[fill=pink!40](2,1) -- (3,2) -- (5,3) --cycle;

\draw [-, ultra thick, color=orange](0,1) -- (2,5) -- (1,2) -- (1,1) --
(5,3) -- (2,1) --(1,0);
\foreach \x in {0,1,...,5}{
\foreach \y in {0,1,...,5}{
        \node[draw,circle,inner sep=0.7pt,fill, color=gray!40] at
(1*\x,1*\y) {}; }
    }

\node [left] at (2.4,5.5) {$E_{3}$};
\node [left] at (2.1,2) {$E_{2}$};
\node [left] at (6.2,3) {$E_{1}$};

\node [left] at (1.3,-0.3) {$L$};
\node [left] at (0,1) {$L_{1}$};
\draw [->, very thick, red] (1,0)--(0.5, 0.5);
\draw [-, very thick, red] (0.5, 0.5)--(0,1);
\node [below] at (3,-0.4) {$\Lambda\left(\dfrac{3}{5}, \dfrac{2}{1},
\dfrac{5}{2}\right)$};

  \draw [-, line width=2.5pt, color=magenta] (1,1) -- (2,1) -- (3,2) --
(5,3);
\draw [-, line width=2.5pt, color=magenta] (1,1) -- (1,3) -- (2,5);
  \draw [-, line width=3pt, color=magenta!70!] (1,0) -- (1,1);
  \node[draw,circle, inner sep=1.5pt,color=red, fill=red] at (2,5){};
\node[draw,circle, inner sep=1.5pt,color=red, fill=red] at (1,2){};
\node[draw,circle, inner sep=1.5pt,color=red, fill=red] at (5,3){};
  \end{scope}

\begin{scope}[shift={(0,-6.5)},scale=1]
  \draw [->](0,0) -- (0,4);
\draw [->](0,0) -- (4,0);
\draw[fill=pink!40](1,0) -- (0,1) -- (1,1)  --cycle;
\draw[fill=pink!40](1,0) -- (1,1) -- (2,1) --cycle;
\draw[fill=pink!40](1,1) -- (2,1) -- (3,2) --cycle;
\draw[fill=pink!40](1,1) -- (3,2) -- (4,3) --cycle;
\draw [-, ultra thick, color=orange](0,1) -- (1,1) -- (4,3) -- (1,0);
\foreach \x in {0,1,...,5}{
\foreach \y in {0,1,...,3}{
        \node[draw,circle,inner sep=0.7pt,fill, color=gray!40] at
(1*\x,1*\y) {}; }
    }
\node [right] at (3.2,2) {$E_{4}$};
\node [right] at (4,3) {$E_{5}$};
\node [left] at (1.3,-0.4) {$E_1$};
\node [left] at (0,1) {$L_2$};
\draw [->, very thick, red] (1,0)--(0.5, 0.5);
\draw [-, very thick, red] (0.5, 0.5)--(0,1);
\node [below] at (3,-0.4) {$\Lambda\left(\dfrac{2}{3},
\dfrac{3}{4}\right)$};

  \draw [-, line width=2.5pt, color=magenta] (1,1) -- (2,1) -- (3,2) --
(4,3);
\node[draw,circle, inner sep=1.5pt,color=red, fill=red] at (3,2){};
\node[draw,circle, inner sep=1.5pt,color=red, fill=red] at (4,3){};
  \end{scope}

\begin{scope}[shift={(10,-6.5)},scale=1]
  \draw [->](0,0) -- (0,4);
\draw [->](0,0) -- (4,0);
\draw[fill=pink!40](1,0) -- (0,1) -- (1,1)  --cycle;
\draw[fill=pink!40](1,0) -- (1,1) -- (2,1) --cycle;
\draw[fill=pink!40](1,1) -- (2,1) -- (3,2) --cycle;
\draw[fill=pink!40](1,1) -- (3,2) -- (4,3) --cycle;
\draw [-, ultra thick, color=orange](0,1) -- (1,1) -- (4,3) -- (1,0);
\foreach \x in {0,1,...,5}{
\foreach \y in {0,1,...,3}{
        \node[draw,circle,inner sep=0.7pt,fill, color=gray!40] at
(1*\x,1*\y) {}; }
    }
\node [right] at (3.2,2) {$E_{4}$};
\node [right] at (4,3) {$E_{5}$};
\node [left] at (1.3,-0.4) {$E_1$};
\node [left] at (0,1) {$L_2$};
\draw [->, very thick, red] (1,0)--(0.5, 0.5);
\draw [-, very thick, red] (0.5, 0.5)--(0,1);
\node [below] at (3,-0.4) {$\Lambda\left(\dfrac{2}{3},
\dfrac{3}{4}\right)$};

  \draw [-, line width=2.5pt, color=magenta] (1,1) -- (2,1) -- (3,2) --
(4,3);
  \draw [-, line width=3pt, color=magenta!70!] (1,0) -- (1,1);
  \node[draw,circle, inner sep=1.5pt,color=red, fill=red] at (3,2){};
\node[draw,circle, inner sep=1.5pt,color=red, fill=red] at (4,3){};
  \end{scope}

%

  \begin{scope}[shift={(0,-15)},scale=1]
  \draw [->](0,0) -- (0,6);
\draw [->](0,0) -- (6,0);
\draw[fill=pink!40](1,0) -- (0,1) -- (1,1)  --cycle;
\draw[fill=pink!40](0,1) -- (1,1) -- (1,2) --cycle;
\draw[fill=pink!40](1,1) -- (1,2) -- (1,3) --cycle;
\draw[fill=pink!40](0,1) -- (1,2) -- (1,3) --cycle;
\draw[fill=pink!40](1,1) -- (1,2) -- (2,3) --cycle;
\draw[fill=pink!40](1,2) -- (2,3) -- (3,5) --cycle;
\draw [-, ultra thick, color=orange](0,1) -- (1,3) -- (1,2) -- (3,5) --
(1,1) --(1,0);
\foreach \x in {0,1,...,5}{
\foreach \y in {0,1,...,5}{
        \node[draw,circle,inner sep=0.7pt,fill, color=gray!40] at
(1*\x,1*\y) {}; }
    }
\node [left] at (3.4,5.5) {$E_{6}$};
\node [left] at (1.1,3) {$E_{7}$};
\node [below] at (1.3,0) {$E_{1}$};
\node [left] at (0,1) {$L_{3}$};
\draw [->, very thick, red] (1,0)--(0.5, 0.5);
\draw [-, very thick, red] (0.5, 0.5)--(0,1);
\node [below] at (3,-0.4) {$\Lambda\left(\dfrac{5}{3},
\dfrac{3}{1}\right)$};

  \draw [-, line width=2.5pt, color=magenta] (1,1) -- (1,3) ;
  \draw [-, line width=2.5pt, color=magenta] (1,2) -- (2,3)-- (3,5);
  \node[draw,circle, inner sep=1.5pt,color=red, fill=red] at (3,5){};
\node[draw,circle, inner sep=1.5pt,color=red, fill=red] at (1,3){};
  \end{scope}

  %

  \begin{scope}[shift={(10,-15)},scale=1]
  \draw [->](0,0) -- (0,6);
\draw [->](0,0) -- (6,0);
\draw[fill=pink!40](1,0) -- (0,1) -- (1,1)  --cycle;
\draw[fill=pink!40](0,1) -- (1,1) -- (1,2) --cycle;
\draw[fill=pink!40](1,1) -- (1,2) -- (1,3) --cycle;
\draw[fill=pink!40](0,1) -- (1,2) -- (1,3) --cycle;
\draw[fill=pink!40](1,1) -- (1,2) -- (2,3) --cycle;
\draw[fill=pink!40](1,2) -- (2,3) -- (3,5) --cycle;
\draw [-, ultra thick, color=orange](0,1) -- (1,3) -- (1,2) -- (3,5) --
(1,1) --(1,0);
\foreach \x in {0,1,...,5}{
\foreach \y in {0,1,...,5}{
        \node[draw,circle,inner sep=0.7pt,fill, color=gray!40] at
(1*\x,1*\y) {}; }
    }
\node [left] at (3.4,5.5) {$E_{6}$};
\node [left] at (1.1,3) {$E_{7}$};
\node [below] at (1.3,0) {$E_{1}$};
\node [left] at (0,1) {$L_{3}$};
\draw [->, very thick, red] (1,0)--(0.5, 0.5);
\draw [-, very thick, red] (0.5, 0.5)--(0,1);
\node [below] at (3,-0.4) {$\Lambda\left(\dfrac{5}{3},
\dfrac{3}{1}\right)$};

  \draw [-, line width=2.5pt, color=magenta] (1,1) -- (1,3) ;
  \draw [-, line width=2.5pt, color=magenta] (1,2) -- (2,3)-- (3,5);
  \draw [-, line width=3pt, color=magenta!70!] (1,0) -- (1,1);

\node[draw,circle, inner sep=1.5pt,color=red, fill=red] at (3,5){};
\node[draw,circle, inner sep=1.5pt,color=red, fill=red] at (1,3){};
  \end{scope}

  \begin{scope}[shift={(0,-19.5)},scale=1]
  \draw [->](0,0) -- (0,3);
\draw [->](0,0) -- (3,0);
\draw[fill=pink!40](1,0) -- (0,1) -- (1,1)  --cycle;
\draw[fill=pink!40](1,0) -- (1,1) -- (2,1) --cycle;

\draw [-, ultra thick, color=orange](0,1) -- (2,1) -- (1,0);
\foreach \x in {0,1,...,2}{
\foreach \y in {0,1,...,2}{
        \node[draw,circle,inner sep=0.7pt,fill, color=gray!40] at
(1*\x,1*\y) {}; }
    }
\node [right] at (2,1) {$E_{8}$};
\node [left] at (0,1) {$L_{4}$};
\node [below] at (1,0) {$E_{6}$};
\draw [->, very thick, red] (1,0)--(0.5, 0.5);
\draw [-, very thick, red] (0.5, 0.5)--(0,1);
\node [below] at (3,-0.4) {$\Lambda\left(\dfrac{1}{2}\right)$};
\draw [-, line width=2.5pt, color=magenta] (1,1) -- (2,1) ;
\node[draw,circle, inner sep=1.5pt,color=red, fill=red] at (2,1){};
  \end{scope}

  \begin{scope}[shift={(10,-19.5)},scale=1]
  \draw [->](0,0) -- (0,3);
\draw [->](0,0) -- (3,0);
\draw[fill=pink!40](1,0) -- (0,1) -- (1,1)  --cycle;
\draw[fill=pink!40](1,0) -- (1,1) -- (2,1) --cycle;

\draw [-, ultra thick, color=orange](0,1) -- (2,1) -- (1,0);
\foreach \x in {0,1,...,2}{
\foreach \y in {0,1,...,2}{
        \node[draw,circle,inner sep=0.7pt,fill, color=gray!40] at
(1*\x,1*\y) {}; }
    }
\node [right] at (2,1) {$E_{8}$};
\node [left] at (0,1) {$L_{4}$};
\node [below] at (1,0) {$E_{6}$};
\draw [->, very thick, red] (1,0)--(0.5, 0.5);
\draw [-, very thick, red] (0.5, 0.5)--(0,1);
\node [below] at (3,-0.4) {$\Lambda\left(\dfrac{1}{2}\right)$};

\draw [-, line width=2.5pt, color=magenta] (1,1) -- (2,1) ;
\draw [-, line width=3pt, color=magenta!70!] (1,0) -- (1,1);
\node[draw,circle, inner sep=1.5pt,color=red, fill=red] at (2,1){};
\end{scope}

\end{tikzpicture}
\end{center}
\caption{The Enriques trees and the extended Enriques trees in Example
\ref{ex:Enrtreemainex}}
\label{fig:Enrtree}
      \end{figure}
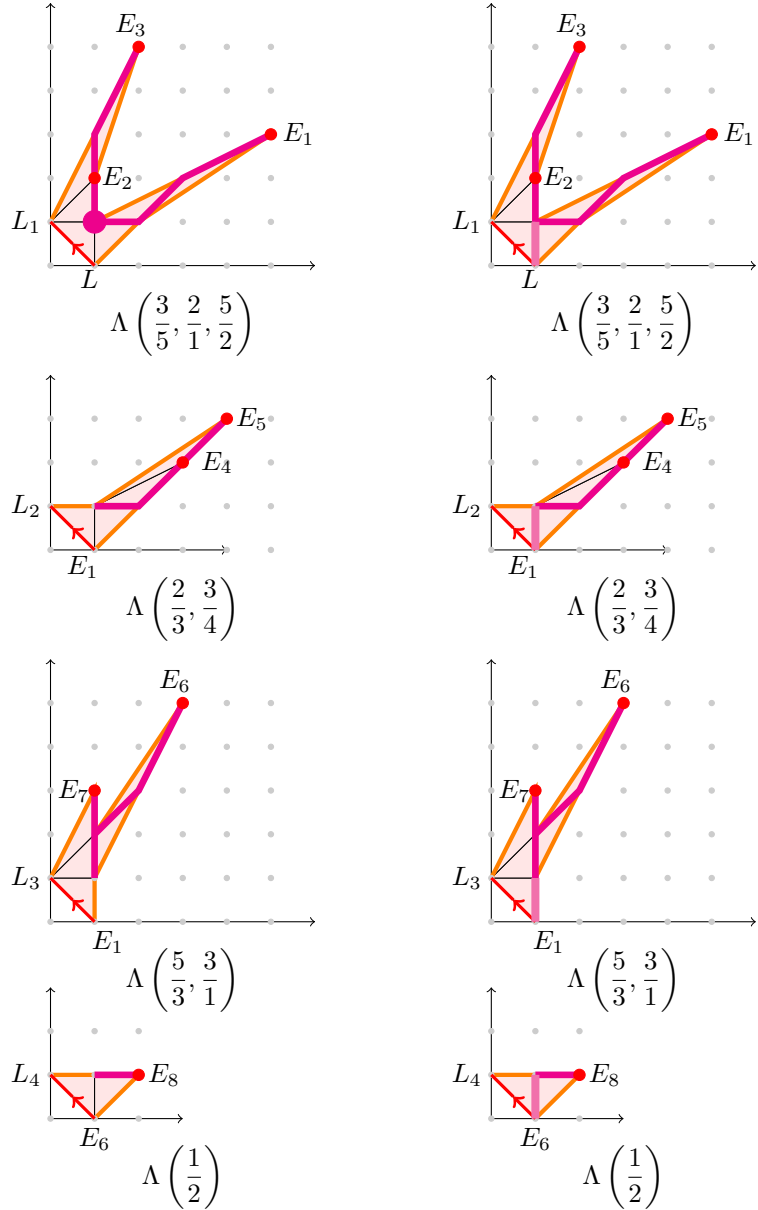

  \begin{figure}
    \begin{center}
\begin{tikzpicture}[scale=0.58]
\draw [fill=pink!40](0,0) -- (3,3) -- (1,-2)--(0,0);
\draw [->, very thick, red] (0,0) --(0.5,-1);
\draw [-, very thick, red] (0.5,-1) --(1,-2);
\draw [-] (0,0)--(1.5,-0.8);
\draw [-] (1,1)--(1.5,-0.8);
\draw [-] (2,2)--(1.5,-0.8);
\draw [-] (3,3)--(3.5,5.5);
\node [above] at (3.5,5.5) {$C_4$};
\draw [-] (2,2)--(2,5);
\node [above] at (2,5) {$C_5$};
\node [above] at (0.2,0.3) {$E_1$};

\draw [fill=pink!40](0,0) -- (-2,-3)--(-7,-3)--(-8,-1)--(-4,-2)--(0,0);
\draw [->, very thick, red] (-2,-3)--(-4.5,-3);
\draw [-, very thick, red] (-4.5,-3)--(-7,-3);
\draw [-] (-4,-2)--(-2,-3);
\draw [-] (-2,-1)--(-1,-1.5);
\draw [-] (-4,-2)--(-1,-1.5);
\draw [-] (-4,-2)--(-7,-3);
\draw [-] (-7,-3)--(-6,-1.5);
\draw [-] (-7.5,-2)--(-6,-1.5);
\draw [-] (-8,-1)--(-9,1);
\node [left, above] at (-9,1) {$C_1$};
\draw [-] (-6,-1.6)--(-7,1);
\node [left, above] at (-7,1.2) {$C_3$};
\draw [-] (-6,-1.6)--(-6,1);
\node [right, above] at (-6,1.2) {$C_2$};

\draw [fill=pink!40](0,0) -- (-3.5,-0.5)--(-4,3)--(-2,2)--(-1.5,4)--(0,0);
\draw [->, very thick, red]  (0,0) -- (-1.75,-0.25);
\draw [-, very thick, red]  (-1.75,-0.25)--(-3.5,-0.5);
\draw [-] (-3.5,-0.5)--(-2,2);
\draw [-] (-3.5,-0.5)--(-0.5,1.2);
\draw [-] (-2,2)--(-1,2.5);
\draw [-] (-2,2)--(-1,2.5);
\draw [-] (-0.5,1.2)--(-2,2);
\draw [-] (-4,3)--(-5,5);
\node [left] at (-5,5) {$C_6$};
\node [left] at (-4,3) {$E_7$};

\draw [fill=pink!40](-1.5,4) -- (-2,6)--(-3,4)--(-1.5,4);
\draw [-, very thick, red] (-1.5,4)--(-3,4);
\draw [->, very thick, red] (-1.5,4)--(-2.25,4);
\draw [-] (-1.5,4)--(-2.5,5);
\node [left] at (-3,4) {$L_4$};
\draw [-] (-2,6)--(-2,8);
\node [above] at (-2,8) {$C_7$};

\draw [-, ultra thick, color=orange]  (-2, -3) -- (0,0) -- (-4, -2) -- (-6, -1.5) -- (-8, -1) -- (-7,-3);
\draw [-, ultra thick, color=orange] (0,0) -- (3,3) -- (1,-2);
\draw [-, ultra thick, color=orange] (0,0) -- (-1.5,4) -- (-2,2) -- (-4,3) -- (-3.5,-0.5);
\draw [-, ultra thick, color=orange] (-1.5,4) -- (-2,6) -- (-3,4);

\node [left] at (2.1,2.2) {$E_4$};
\node [right] at (3,3) {$E_5$};
\node [below] at (1,-2) {$L_2$};
\node [below] at (-2,-3) {$L_0=L$};
\node [below] at (-7,-3) {$L_1$};
\node [left] at (-8,-1) {$E_3$};
\node [right] at (-6,-1.3) {$E_2$};
\node [left] at (-3.5,-0.5) {$L_3$};
\node [right] at (-2,6) {$E_8$};
\node [right] at (-1.5,4) {$E_6$};
\node[draw,circle,inner sep=1.3pt,fill=black] at (0,0){};


\draw [-, line width=2.5pt, color=magenta] (-8,-1) -- (-7.5,-2)--(-6,-1.5)--(-4,-2)--(-1,-1.5)--(-2,-1)--(0,0)--(1.45,-0.8)--(1,1)--(3,3);
\draw [-, line width=2.5pt, color=magenta] (0,0)--(-0.5,1.2)--(-2,2)--(-4,3);
\draw [-, line width=2.5pt, color=magenta] (-2,2)--(-1,2.5)--(-1.5,4)--(-2,6);

\draw [-] (-2,6)--(-2,8)--(-1,8.5);
\draw [fill=pink!40](-2,6) -- (-2,8) -- (-1,8.5)--cycle;
\draw [-, line width=2.5pt, color=magenta] (-2,6)--(-1,8.5);

\draw [-] (-4,3)--(-5,5)--(-4.5,5.5)--cycle;
\draw [fill=pink!40](-4,3)--(-5,5)--(-4.5,5.5)--cycle;
\draw [-, line width=2.5pt, color=magenta] (-4,3)--(-4.5,5.5);

\draw [-] (3,3)--(3.5,5.5)--(5,6)--cycle;
\draw [fill=pink!40] (3,3)--(3.5,5.5)--(5,6)--cycle;
\draw [-, line width=2.5pt, color=magenta](3,3)--(5,6);

\draw [-] (2,2)--(2,5)--(3,5.5)--cycle;
\draw [fill=pink!40] (2,2)--(2,5)--(3,5.5)--cycle;
\draw [-, line width=2.5pt, color=magenta](2,2)--(3,5.5);

\draw [-] (-8,-1)--(-9,1)--(-7.7,1.5)--cycle;
\draw [fill=pink!40](-8,-1)--(-9,1)--(-7.7,1.5)--cycle;
\draw [-, line width=2.5pt, color=magenta](-8,-1)--(-7.7,1.5);

\draw [-] (-6,-1.5)--(-7,1)--(-6.5,1.5)--cycle;
\draw [fill=pink!40] (-6,-1.5)--(-7,1)--(-6.5,1.5)--cycle;
\draw [-, line width=2.5pt, color=magenta](-6,-1.5)--(-6.5,1.5);

\draw [-] (-6,-1.5)--(-6,1)--(-5.5,1.5)--cycle;
\draw [fill=pink!40] (-6,-1.5)--(-6,1)--(-5.5,1.5)--cycle;
\draw [-, line width=2.5pt, color=magenta] (-6,-1.5)--(-5.5,1.5);

\node[draw,circle,inner sep=2.5pt,color=magenta,fill=magenta] at (-4,-2){};
\end{tikzpicture}
\end{center}
\caption{The Enriques tree of the toroidal pseudo-resolution of Example \ref{ex:Enrtreemainex}}
  \label{fig:Enriqtoroid}
   \end{figure}
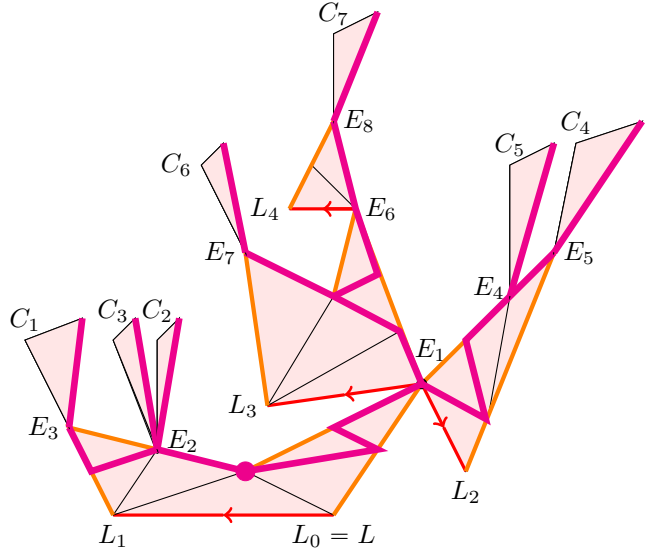

\begin{example}   \label{ex:Enrtreemainex}
   Let us consider again our recurrent example of toroidal pseudo-resolution. 
   Its associated lotus was represented in Figure \ref{fig:lotustoroid}. In 
   Figure \ref{fig:Enrtree} are represented the Enriques trees and extended 
   Enriques trees of its membranes of dimension $2$. Finally, in Figure 
   \ref{fig:Enriqtoroid} is represented its full Enriques tree. In this figure 
   we have also represented the petals associated to the pairs $(E_i, C_j)$, in order to 
   draw the end edges of the Enriques tree.
\end{example}

\subsection{The dependence of the lotus on the choice of completion}
\label{ssec:deplotus}
$\:$
\medskip

In this subsection we show using two examples that the lotus 
$\Lambda_{\pi}(C)$ of a toroidal pseudo-resolution 
process $\pi$ of a plane curve singularity $C \hookrightarrow S$ depends on the choice of auxiliary 
curves added each time one executes STEP 2 of Algorithm \ref{alg:tores}, 
that is, on the choice of completion $\hat{C}_{\pi}$ of $C$ (see Definition \ref{def:threeres}). 
\medskip

In the following two examples \ref{ex: two-lotuses2} and 
\ref{ex: two-lotuses},  we build 
the lotuses $\Lambda_\pi (C)$ associated 
with two distinct embedded resolutions  $\pi: (\Sigma, \partial \Sigma) \to (S, \partial S)$ 
of the curve singularity $C=Z(f)$,  defined by the power series $ f := y^2 - 2 x y + x^2 - x^3 \in \C[[x,y]]$, 
relative to local coordinates $(x, y)$  on the germ $(S, o)$. 
These examples illustrate the fact that the associated lotus $\Lambda_\pi (C)$ 
(see Definition \ref{def:lotustoroid}), 
which is based on the toroidal structure of $\Sigma$, depends 
on the choices of auxiliary curves done at STEP 2 of the Algorithm \ref{alg:tores}, 
that is, on the choice of \emph{completion} $\hat{C}_{\pi}$ of $C$ (see Definition \ref{def:threeres}). 
In both examples we run Algorithm \ref{alg:tores} with $L = Z(x)$,  
replacing STEP 3 by STEP $3^{reg}$ as we explained in Subsection  \ref{ssec:toremb}, 
and taking different choices of auxiliary curves.
The output, which determines the toroidal boundary on 
$\Sigma$, provides two different lotuses. 
On both of them we recognize  the same weighted dual graph of the final 
total transform of $C$, thanks to point (\ref{point:latboundary}) of Theorem \ref{thm:repsailtor}.

\begin{example} \label{ex: two-lotuses2} 
We start the algorithm by  choosing  $L_1 := Z(y-x)$.  The cross 
$(L,L_1)$ at $o$  is defined by the local coordinate system $(x, y_1:= y -x)$.
Relative to these coordinates, $C$ has local equation $y_1^2 - x^3 =0$. 
The Newton polygon $\cN_{L, L_1}(C)$ 
has only one edge and its orthogonal ray has slope $3/2$, hence 
$\cF_{L,L'} (C) \simeq \fan (3/2)$. The first trunk is just the segment 
$[e_{L}, e_{L_1}]$ with its point of slope  $3/2$ marked. 

The Newton modification $\pi := \psi_{L, L_1}^{C, \,reg} : (\Sigma, \partial \Sigma) \to (S, \partial S)$ 
has three exceptional divisors 
$E_1$, $E_2$ and $E_3$ which correspond to the rays 
of the regularization $\fan^{reg} (3/2) = \fan(1, 2, 3/2)$ of the fan $\fan (3/2)$
of slopes $1$ and $2$ and $3/2$ respectively. 
In this case, the strict transform $C_{L, L_1}$ of $C$ is smooth and 
intersects transversally the component $E_3$ of the exceptional divisor, 
that is, the Newton modification $\pi$  is an embedded resolution of $C$. 
Note that when running the algorithm \ref{alg:tores}, we include the cross 
$(E_3, C_{L, L_1} )$ in the toroidal structure of the boundary of $\Sigma$. 

The lotus $\Lambda_{\pi}(C)$ is built by gluing 
the lotus $\Lambda_{L, L_1} (C) = \Lambda( 3/2 )$ 
with the lotus $[e_{E_3}, C]$ associated to the cross $(E_3, C_{L, L_1} )$, 
identifying the points labeled by $E_3$ (see Figure \ref{fig:two-lotuses2}).
\end{example}

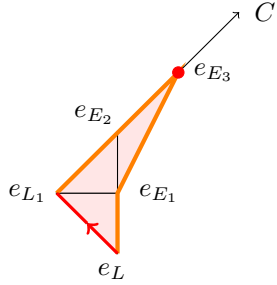
\begin{figure} 
     \begin{center}
\begin{tikzpicture}[scale=0.8]
\draw[fill=pink!40](1,0) -- (0,1) -- (1,1)  --cycle;
\draw[fill=pink!40](1,1) -- (1,2) -- (0,1) --cycle;
\draw[fill=pink!40](1,1) -- (1,2) -- (2,3) --cycle;
\draw [-, ultra thick, color=orange](0,1) -- (1,2) -- (2,3) -- (1,1) -- (1,0);
\draw [->] (2,3)--(3, 4);
\node[draw,circle, inner sep=1.5pt,color=red, fill=red] at (2,3){};

\node [right] at (3.1, 4.0) {$C$};
\node [right] at (2.1, 3.0) {$e_{E_3}$};
\node [left] at (1.1,2.3) {$e_{E_2}$};
\node [right] at (1.2,1) {$e_{E_1}$};
\node [left] at (1.3,-0.3) {$e_{L}$};
\node [left] at (0,1) {$e_{L_1}$};

\draw [->, very thick, red] (1,0)--(0.5, 0.5);
\draw [-, very thick, red] (0.5, 0.5)--(0,1);

\end{tikzpicture}
\end{center}
\caption{The lotus $\Lambda_{\pi} (C)$ of Example  \ref{ex: two-lotuses2}}
\label{fig:two-lotuses2}
\end{figure}

\begin{example} \label{ex: two-lotuses}
We start the algorithm by  choosing  $L_1 := Z(y)$ and  the cross 
$(L,L_1)$ on $(S,o)$. 
The Newton polygon $\cN_{L, L_1}(C)$ 
has only one edge and its orthogonal ray has slope $1$, hence 
$\cF_{L,L_1} (C) \simeq \fan (1)$. 
The first trunk is the segment $[e_{L}, e_{L_1}]$ with its midpoint marked. 
The first lotus is just the petal $\Lambda_1 := \Lambda(1) =\delta ( e_{L}, e_{L_1} )$
with base $[e_{L}, e_{L_1}]$.  

The Newton modification $\psi_{L, L_1}^C $ is the usual blow up of the point $o$. 
We  restrict it to the chart $\C^2_{v_1, v_2}$, where $x = v_{1},   y = v_{1} v_{2}$. 
The strict transform $C_1:= C_{L, L_1}$ is defined in this chart  by the equation 
$v_2^2 - 2 v_2 +1 - v_1 =0$. 
The exceptional divisor $E_1 := Z(v_1)$ intersects the strict transform $C_1$ 
at the point $o_1$ defined by $v_2 =1$. 
When running the algorithm,  
we have to choose a smooth branch $B_2$ such that $(E_1, B_2)$ defines a cross at $o_1$. 
We set $B_2:= Z( v_2 -1)$ and $u_1 := v_2 -1$.  Then, the local coordinates $(v_1, u_1)$ 
define the cross $(E_1, B_2)$.
We denote by $L_2$ the projection to $S$ of the line $B_2= Z(u_1)$, which is parametrized by 
$v_1 = t$ and $v_2 = 1$. 
One gets that $L_2$, which is parametrized by $x= t$, $y=t$, has local equation $y -  x =0$. 

The strict transform $C_1$ has local equation $u_1^2 - v_1 =0$. 
The Newton polygon $\cN_{E_1, B_2} (C_1)$ 
has only one edge and its orthogonal ray has slope $1/2$, hence its associated fan is 
$\cF_{E_1,B_2} (C_1) \simeq \fan (1/2)$. 
The second trunk is just the segment 
$[e_{E_1}, e_{L_2}]$ with a marked point of slope $1/2$. 
The modification $\psi_{E_1, B_2}^{C_1, reg}$ defined by the regularization of this fan 
has two exceptional divisors $E_2$ and $E_3$ corresponding to the rays of 
the  regularization of the fan $\fan (1/2)$ of slopes $1$ and $1/2$ respectively. 
When we consider the regularization of the fan $\cF_{E_1,B_2} (C_1)$,  
we have to mark an additional point of slope $1$ in the second trunk $[e_{E_1}, e_{L_2}]$. 
The associated lotus is $\Lambda_2 := \Lambda(1/2)$, 
with base $[e_{E_1}, e_{L_2}]$.

In this example, 
the composition $ \pi:= \psi_{E_1, B_2}^{C_1, reg}  \circ \psi_{L, L_1} ^C: 
(\Sigma, \partial \Sigma) \to (S, \partial S)$ 
is an embedded resolution of $C$, since the strict transform $C_2$ of $C$ 
is smooth and intersects transversally the exceptional divisor of $\pi$ at a point $o_2 \in E_3$.
Notice that when running the algorithm, we have to consider also
the cross $(E_3, C_2)$  at $o_2$. 
Its trunk coincides with its associated lotus.  It is just the segment 
$\Lambda_3:= [e_{E_3}, C]$, with no marked points.

The lotus $\Lambda_\pi(C)$ is represented in Figure \ref{fig:two-lotuses}.
It is obtained from $\Lambda_1$, $\Lambda_2$ and 
$\Lambda_3$ (see  Figure \ref{fig:three-lotuses}) by identifying the points with the same label. 
\end{example}

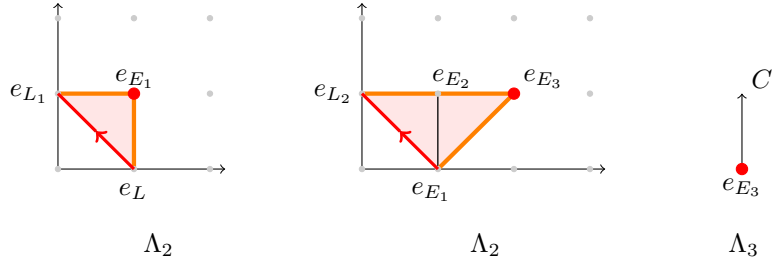
\begin{figure} 
     \begin{center}
\begin{tikzpicture}[scale=1]

\draw [->](0,0) -- (0,2.2);
\draw [->](0,0) -- (2.2,0);
\draw[fill=pink!40](1,0) -- (0,1) -- (1,1)  --cycle;
\draw [-, ultra thick, color=orange](0,1) -- (1,1) -- (1,0);

\foreach \x in {0,1,...,2}{
\foreach \y in {0,1,...,2}{
       \node[draw,circle,inner sep=0.7pt,fill, color=gray!40] at (1*\x,1*\y) {}; }
   }

\node[draw,circle, inner sep=1.5pt,color=red, fill=red] at (1,1){};
\node [left] at (1.4,1.2) {$e_{E_1}$};
\node [left] at (1.3,-0.3) {$e_{L}$};
\node [left] at (0,1) {$e_{L_1}$};

\draw [->, very thick, red] (1,0)--(0.5, 0.5);
\draw [-, very thick, red] (0.5, 0.5)--(0,1);
    \node [right] at (1,-1) {$\Lambda_2$};

\begin{scope}[shift={(4,0)},scale=1]

\draw [->](0,0) -- (0,2.2);
\draw [->](0,0) -- (3.2,0);
\draw[fill=pink!40](1,0) -- (0,1) -- (1,1)  --cycle;
\draw[fill=pink!40](1,0) -- (1,1) -- (2,1) --cycle;
\draw [-, ultra thick, color=orange](0,1) -- (1,1) -- (2,1) -- (1,0);

\foreach \x in {0,1,...,3}{
\foreach \y in {0,1,...,2}{
       \node[draw,circle,inner sep=0.7pt,fill, color=gray!40] at (1*\x,1*\y) {}; }
   }

\node[draw,circle, inner sep=1.5pt,color=red, fill=red] at (2,1){};
\node [right] at (0.8,1.2) {$e_{E_2}$};
\node [right] at (2.,1.2) {$e_{E_3}$};
\node [left] at (1.3,-0.3) {$e_{E_1}$};
\node [left] at (0,1) {$e_{L_2}$};

\draw [->, very thick, red] (1,0)--(0.5, 0.5);
\draw [-, very thick, red] (0.5, 0.5)--(0,1);
   \node [right] at (1.3,-1) {$\Lambda_2$};
\end{scope}

\begin{scope}[shift={(8,0)},scale=1]
\draw [->](1,0) -- (1,1);
\node[draw,circle, inner sep=1.5pt,color=red, fill=red] at (1,0){};
\node [below] at (1, 0) {$e_{E_3}$};
\node [right] at (1,1.2) {$C$};
  \node [right] at (0.7,-1) {$\Lambda_3$};
\end{scope}

\end{tikzpicture}
\end{center}
\caption{The Newton lotuses $\Lambda_1$, $\Lambda_2$ and $\Lambda_3$ 
    of Example \ref{ex: two-lotuses}}
 \label{fig:three-lotuses}
\end{figure}


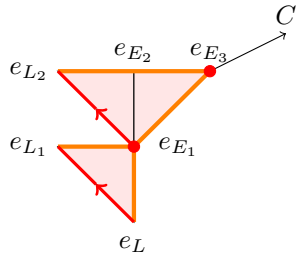
\begin{figure} 
     \begin{center}
\begin{tikzpicture}[scale=1]

\draw[fill=pink!40](1,0) -- (0,1) -- (1,1)  --cycle;
\draw[fill=pink!40](1,1) -- (0,2) -- (1,2) --cycle;
\draw[fill=pink!40](1,1) -- (1,2) -- (2,2) --cycle;

\draw [-, ultra thick, color=orange](0,1) -- (1,1) -- (1,0);
\draw [-, ultra thick, color=orange](1,1) -- (2,2) -- (1,2) -- (0,2);

\node[draw,circle, inner sep=1.5pt,color=red, fill=red] at (1,1){};
\node [above] at (1,2) {$e_{E_2}$};
\node [right] at (1.2,1.0) {$e_{E_1}$};
\node [left] at (1.3,-0.3) {$e_{L}$};
\node [left] at (0,1) {$e_{L_1}$};
\node [left] at (0,2) {$e_{L_2}$};
\node [above] at (2,2) {$e_{E_3}$};
\draw [->](2,2) -- (3,2.5);
\node[draw,circle, inner sep=1.5pt,color=red, fill=red] at (2,2){};
\node [above] at (3,2.5) {$C$};

\draw [->, very thick, red] (1,0)--(0.5, 0.5);
\draw [-, very thick, red] (0.5, 0.5)--(0,1);
\draw [->, very thick, red] (1,1)--(0.5, 1.5);
\draw [-, very thick, red] (0.5, 1.5)--(0,2);
\end{tikzpicture}
\end{center}
\caption{The lotus $\Lambda_{\pi} (C)$ of Example  \ref{ex: two-lotuses}}
\label{fig:two-lotuses}
\end{figure}

 \begin{remark}  \label{rem:valemblot}
    The lotus $\Lambda_{\pi}(C)$ may be embedded canonically into the set of \textbf{semivaluations} 
    of the local $\C$-algebra $\hat{\cO}_{S,o}$ (semi-valuations are defined 
    similarly to valuations, but dropping the last condition from Definition 
    \ref{def:valuation}). Indeed, its base membrane 
    $\Lambda(\fan_{L, L_1}(C))$ embeds into the regular cone 
    $\sigma_0^{L, L_1}$ of Definition \ref{def:ilattice}, which may 
    be interpreted valuatively by associating to each $w \in \sigma_0^{L, L_1}$ 
    the valuation $\nu_w$ defined by Equation (\ref{eq:defvalweight}). Each 
    other membrane may be similarly interpreted valuatively, and one may show
    that one gets in this way an embedding. Details may be found in 
    \cite[Section 7]{PP 11}. 
 \end{remark}

\subsection{Truncated lotuses} 
\label{ssec:trunclot}
$\:$ 
\medskip

In this subsection we introduce an operation of \emph{truncation} of 
the lotus of a toroidal pseudo-resolution of a plane curve singularity $C$, 
and we explain how to use it in order to visualize the dual graph of the 
total transform of $C$ on the associated embedded resolution, as well 
as the Enriques diagram of the constellation of infinitely near points 
blown up for creating this resolution, in a way different from that formulated 
in point (\ref{point:globalE}) of Theorem \ref{thm:repsailtor}. 
\medskip

Recall first from Definition \ref{def:lotustoroid} the construction of the lotus 
$\Lambda_{\pi}(C)$ of a toroidal pseudo-resolution $\pi : (\Sigma, \partial \Sigma) \to (S, L +L')$ 
of the curve singularity $C \hookrightarrow S$. 
As stated in point (\ref{point:latboundary}) of Theorem \ref{thm:repsailtor}, 
its lateral boundary $\partial_+ \Lambda_{\pi}(C)$ 
is isomorphic to the dual graph of the boundary divisor $\partial \Sigma^{reg}$. 
Here $\Sigma^{reg}$ denotes the minimal resolution of $\Sigma$, and 
$\partial \Sigma^{reg}$ is the total transform on it of the boundary divisor 
$\partial \Sigma$ of the toroidal surface $(\Sigma, \partial \Sigma)$. The divisor 
$\partial \Sigma^{reg}$ is also the total transform of the completion $\hat{C}_{\pi}$ of $C$ 
relative to $\pi$, that is, the sum of the total transform of $C$ by the 
smooth modification $\pi^{reg}  : \Sigma^{reg} \to S$ and of the strict transforms of the 
branches $L_j$ introduced while running Algorithm \ref{alg:tores}. 

How to get the dual graph of the total transform of $C$ on $\Sigma^{reg}$ 
from the lateral boundary $\partial_+ \Lambda_{\pi}(C)$? 
One has simply to remove the ends of $\partial_+ \Lambda_{\pi}(C)$ which are labeled 
by the branches $L_j$, as well as the edges which connect them to other vertices of 
$\partial_+ \Lambda_{\pi}(C)$. This \emph{truncation operation} 
performed on the tree $\partial_+ \Lambda_{\pi}(C)$ may be seen as the  
restriction of a similar operation performed on the whole lotus $\Lambda_{\pi}(C)$. 
Let us explain this truncation operation on $\Lambda_{\pi}(C)$, as well as some of its uses. 

Consider first a petal $\delta(e_1, e_2)$ associated to a base $(e_1, e_2)$ of a lattice $N$ 
(see Definition \ref{def:petal}). Its {\bf axis} is the median $[(e_1 + e_2)/2, e_1 + e_2]$ of the 
petal, joining the vertex $e_1 + e_2$ to the midpoint of the opposite edge. This axis decomposes 
the petal into two {\bf semipetals}. \index{semipetal}

\begin{figure}[h!]
    \begin{center}
\begin{tikzpicture}[scale=1.4]
   \draw [dashed, gray] (0,0) grid (1,1);
\node [above] at (1.2,1) {$e_{1}+ e_2$};
\node [below] at (1,0) {$e_{1}$};
\node [left] at (0,1) {$e_{2}$};
\node [left] at (0,0) {$0$};

\draw [fill=pink](1,0) -- (0,1)--(1,1)--cycle;
\draw [fill=orange](1,0) -- (0.5,0.5)--(1,1)--cycle;
\draw [very thick, blue] (1,1)--(0.5, 0.5);

\draw [->, very thick, red] (1,0)--(0.5, 0.5);
\draw [-, very thick, red] (0.5, 0.5)--(0,1);

\draw[->][thick, color=black](1,-0.7) .. controls (0.5,-0.5) ..(0.7,0.7);
\node [below] at (1,-0.7) {$\hbox{\rm axis}$};

\draw[->][thick, color=black](1.7,1) .. controls (1.5,0.5) ..(1,0.5);
\node [right] at (1.7,1) {$\hbox{\rm first semipetal}$};
\draw[->][thick, color=black](1,1.8) .. controls (0.5,1.5) ..(0.5,1);
\node [right] at (1,1.8) {$\hbox{\rm second semipetal}$};


  \end{tikzpicture}
\end{center}
 \caption{The two semipetals and the axis of the petal $\delta(e_1, e_2)$}
 \label{fig:Semipetalsaxis}
   \end{figure}
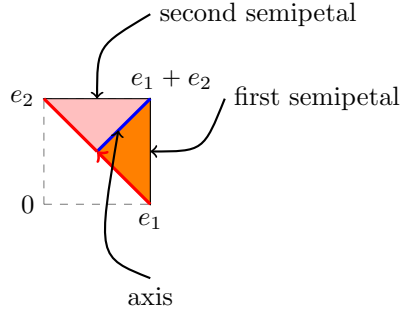

The {\bf semipetals} of a lotus are the semipetals of all its petals. Using this 
vocabulary, as well as that introduced in Definition \ref{def:lotustoroid} about the anatomy 
of lotuses of toroidal pseudo-resolutions, we may define now the operation of truncation of 
such a lotus:

\begin{definition}  \label{def:trunclot}
    Let $\Lambda_{\pi}(C)$ be the lotus of a toroidal pseudo-resolution $\pi$ 
    of the plane curve singularity $C \hookrightarrow (S, o)$. 
    Its {\bf truncation} $\boxed{\Lambda_{\pi}^{tr}(C)}$ \index{truncation!of a lotus} 
    is the union of the axis of its basic petal, of all the semipetals  
    which do not contain basic vertices and of all the membranes which are segments, that is, 
    of the edges of $\Lambda_{\pi}(C)$ 
    which have an extremity labeled by a branch of $C$. The {\bf lateral boundary} 
    \index{boundary!of a truncated lotus} 
    $\boxed{\partial_+ \Lambda_{\pi}^{tr}(C)}$ of $\Lambda_{\pi}^{tr}(C)$ 
    is the part of the lateral boundary of 
    $\Lambda_{\pi}(C)$ which remains in $\Lambda_{\pi}^{tr}(C)$.      
\end{definition}

Truncating the lotus $\Lambda_{\pi}(C)$ corresponds to forgetting 
   its points whose corresponding semivaluations depend on the choice of the branches 
   $L_j$. One keeps only those semivaluations determined by the given curve singularity 
   $C$ and by the infinitely near points through which pass its strict transforms during 
   the blow up process (see Remark \ref{rem:valemblot}). In fact, the third author 
   had introduced truncated lotuses 
   in \cite{PP 11} -- under the name of \emph{sails} -- as objects 
   which represent the combinatorial type of a blow up process of a finite constellation, 
   without considering any supplementary branches passing through the points of the 
   constellation.

By construction, the lateral boundary $\partial_+ \Lambda_{\pi}^{tr}(C)$ is isomorphic 
to the dual graph of the total transform $(\pi^{reg})^*(C)$. One may read again 
the self-intersection number of each irreducible component of the exceptional divisor 
of $\pi^{reg}$ as the opposite of the number of petals, semi-petals and axis containing 
the vertex representing it.

        Note that both lotuses of Figures \ref{fig:two-lotuses2} and \ref{fig:two-lotuses} have the 
        same truncations. The reason is that their associated toroidal pseudo-resolutions 
        lead to the same 
        embedded resolution of $C$ by regularization and that the truncated lotus is a combinatorial 
        object encoding the decomposition of this resolution into blow ups of points.

\begin{example}  \label{ex:truncex}
For instance, in Figure \ref{fig:lotoroidtrunc} is shown the truncation of the 
lotus of Figure \ref{fig:lotustoroid}. Its lateral boundary is emphasized using thick 
orange segments. 
The component of the exceptional divisor 
represented by the unique vertex of the lotus contained in the axis has 
self-intersection number $-4$, as this vertex is contained in the axis, in two semi-petals 
and in one petal of $\Lambda_{\pi}^{tr}(C)$. 
\end{example}


\begin{figure}
     \begin{center}
\begin{tikzpicture}[scale=0.57]
\draw [fill=pink!40](0,0) -- (3,3) -- (1.5,-0.8)--(0,0);
\draw [fill=pink!40](0,0) --(0.5,-1) -- (1.5,-0.8)--(0,0);
\draw [dashed](0.5,-1) -- (1.5,-0.8)--(1,-2)--(0.5,-1) ;
\draw[-, pink!90](0.5,-1)--(1.5,-0.8);
\draw [-] (0,0) --(0.5,-1);
\draw [-, dashed] (1,-2)--(0.5,-1);
\draw[-, very thick, pink!90](-4,-2)--(-4.5,-3);  
\node[draw,rectangle, inner sep=1pt,color=black,fill=white] at (0.5,-1){};

\draw [-] (0,0)--(1.5,-0.8);
\draw [-] (1,1)--(1.5,-0.8);
\draw [-] (2,2)--(1.5,-0.8);
\draw [->] (3,3)--(3.5,5.5);
\draw [->] (2,2)--(2,5);

\draw [fill=pink!40](0,0) -- (-4,-2)--(-3,-2.5)--(-1,-1.5)--(0,0);
\draw [fill=pink!40](-6,-1.5)--(-6.5,-2.25)--(-7.5,-2)--(-8,-1);
\draw [fill=pink!40](-6,-1.5)--(-4,-2)--(-5.5,-2.5);
\draw[-, pink!90](-3,-2.5)--(-1,-1.5);
       \draw [-, dashed] (-2,-3)--(-4.5,-3);
       \draw [-, dashed] (-7,-3)--(-4.5,-3);
       \node[draw,rectangle, inner sep=1pt,color=black,fill=white] at (-4.5,-3){};
       \draw [-, dashed] (-7,-3)--(-7.5,-2);
       \draw [-, dashed] (-7,-3)--(-6.5,-2.25);
        \draw[-, pink!90](-7.5,-2)--(-6.5,-2.25); 
         \node[draw,rectangle, inner sep=1pt,color=black,fill=white] at (-6.5,-2.25){};

\draw [-] (-4,-2)--(-3,-2.5);
\draw [-,dashed] (-2,-3)--(-3,-2.5);
 \node[draw,rectangle, inner sep=1pt,color=black,fill=white] at (-3,-2.5){};
\draw [-,dashed] (-2,-3)--(-1,-1.5);
\draw [-] (-2,-1)--(-1,-1.5);
\draw [-] (-4,-2)--(-1,-1.5);

\draw [-] (-4,-2)--(-5.5,-2.5);
\draw [-, dashed] (-7,-3)--(-5.5,-2.5);
\draw [-, pink!90] (-6,-1.5)--(-5.5,-2.5);
 \node[draw,rectangle, inner sep=1pt,color=black,fill=white] at (-5.5,-2.5){};

\draw [-] (-6,-1.5)--(-6.5,-2.25);
\node[draw,rectangle, inner sep=1pt,color=black,fill=white] at (-6.5,-2.25){};
\draw [-] (-7.5,-2)--(-6,-1.5);
\draw [->] (-8,-1)--(-9,1);
\draw [->] (-6,-1.6)--(-7,1);
\draw [->] (-6,-1.6)--(-6,1);

\draw [fill=pink!40](0,0) -- (-1.75,-0.25)--(-0.5,1.2)--(-2,2)--(-1,2.5)--(-1.5,4);
\draw [fill=pink!40](-2,2)--(-1,2.5)--(-1.5,4);
\draw [fill=pink!40](-2,2)--(-4,3)--(-2.75,0.75);
\draw [fill=pink!40](-2,2)--(-2,0.35)--(-0.5,1.2);

    \draw [-,pink!90] (-2.75,0.75)--(-4,3);
    \draw [-]  (0,0) -- (-1.75,-0.25);
    \draw [-,pink!90] (-1.75,-0.25)--(-0.5,1.2);
    \draw [-, dashed]  (-3.5,-0.5)--(-1.75,-0.25);
    \node[draw,rectangle, inner sep=1pt,color=black,fill=white] at (-1.75,-0.25){};

\draw [-, dashed] (-3.5,-0.5)--(-2.75,0.75);
\draw [-] (-2,2)--(-2.75,0.75);
\node[draw,rectangle, inner sep=1pt,color=black,fill=white] at (-2.75,0.75){};
\draw [-, dashed] (-3.5,-0.5)--(-4,3);
\draw [-, dashed] (-3.5,-0.5)--(-2,0.35);
\draw [-] (-0.5,1.2)--(-2,0.35);
\draw [-,pink!90] (-2,2)--(-2,0.35);
\node[draw,rectangle, inner sep=1pt,color=black,fill=white] at (-2,0.35){};

\draw [-] (-2,2)--(-1,2.5);
\draw [-] (-2,2)--(-1,2.5);
\draw [-] (-0.5,1.2)--(-2,2);
\draw [->] (-4,3)--(-5,5);

\draw [fill=pink!40](-1.5,4) -- (-2,6)--(-2.5,5)--(-1.5,4);
\draw [fill=pink!40](-2.5,5)--(-2.25,4)--(-1.5,4)--(-2.5,5);
     \draw [-] (-1.5,4)--(-2.25,4);
     \draw [-, dashed] (-3,4)--(-2.25,4);
     \draw [-, pink!90] (-2.5,5)--(-2.25,4);
     \node[draw,rectangle, inner sep=1pt,color=black,fill=white] at (-2.25,4){};
      \draw [-, dashed] (-3,4)--(-2.5,5); 
\draw [-] (-1.5,4)--(-2.5,5);
\draw [->] (-2,6)--(-2,7);
\draw [-, ultra thick, color=orange]  (-1, -1.5) -- (0,0) -- (-2,-1)--(-4, -2) -- (-6, -1.5) --(-8,-1)-- (-7.5, -2);
\draw [-, ultra thick, color=orange] (0,0) -- (3,3) -- (1.5,-0.8);
\draw [-, ultra thick, color=orange] (0,0) -- (-0.5,1.2)--(-1,2.5)--(-1.5,4) -- (-2,2) -- (-4,3);
\draw [-, ultra thick, color=orange] (-1.5,4) -- (-2,6) -- (-2.5,5);

\end{tikzpicture}
\end{center}
\caption{The truncation of the lotus of Figure \ref{fig:lotustoroid} (see Example \ref{ex:truncex})}  
   \label{fig:lotoroidtrunc}
    \end{figure}
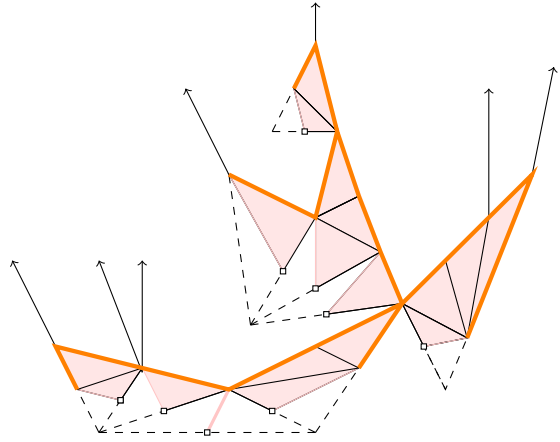

    Consider now the Enriques tree of the toroidal pseudo-resolution $\pi$. 
Its edges are certain lateral edges of the $2$-dimensional petals of $\Lambda_{\pi}(C)$ 
and of the $2$-dimensional petals constructed from the $1$-dimensional petals 
of $\Lambda_{\pi}(C)$ as bases (see Definition \ref{def:lotustoroid}). For 
each edge $[A,B]$ of the Enriques tree, one may consider instead the 
homothetic segment $(1/2)[A, B]$, joining the points $(1/2)A$ and $(1/2)B$. 
This homothety is well-defined if one interprets the elements of the segment 
$[A,B]$ as valuations (see Remark \ref{rem:valemblot}). If both $A$ and $B$ 
are vertices of the same petal, then the segment $(1/2)[A, B]$ joins two 
edge midpoints of this petal. Otherwise, the interior points of the segment  $(1/2)[A, B]$ 
are disjoint from the lotus $\Lambda_{\pi}(C)$. 

The union of such segments $(1/2)[A, B]$ -- which were called \emph{ropes} by the 
third author in \cite{PP 11} -- is isomorphic to the Enriques tree of $\pi$. 
Therefore it is another representation of the Enriques diagram of the constellation 
whose blow up creates the resolution $\pi^{reg}$.


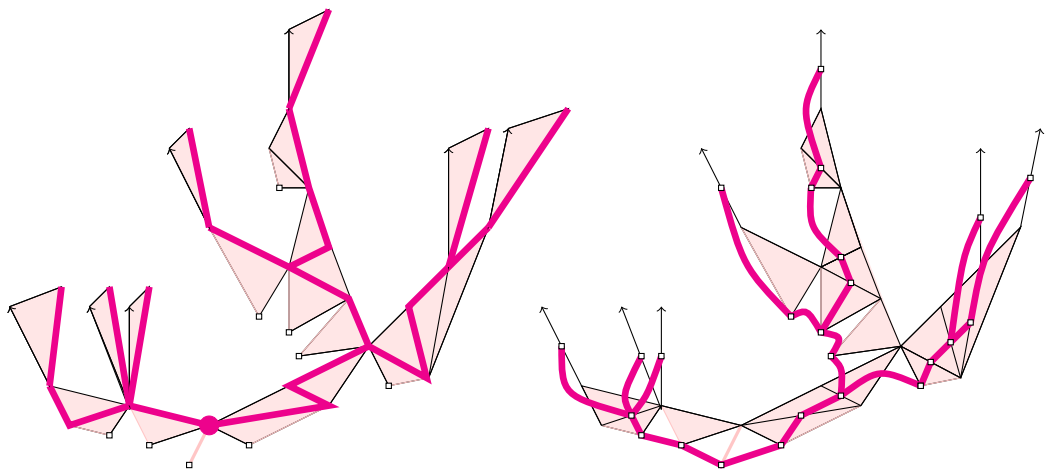
\begin{figure}[h!]
     \begin{center}
\begin{tikzpicture}[scale=0.7]


\begin{scope}[shift={(8,0)},scale=0.75]
\draw [fill=pink!40](0,0) -- (3,3) -- (1.5,-0.8)--(0,0);
\draw [fill=pink!40](0,0) --(0.5,-1) -- (1.5,-0.8)--(0,0);
\draw[-, pink!90](0.5,-1)--(1.5,-0.8);
\draw [-] (0,0) --(0.5,-1);

\draw [-] (0,0)--(1.5,-0.8);
\draw [-] (1,1)--(1.5,-0.8);
\draw [-] (2,2)--(1.5,-0.8);
\draw [->] (3,3)--(3.5,5.5);
\draw [->] (2,2)--(2,5);

\draw [fill=pink!40](0,0) -- (-4,-2)--(-3,-2.5)--(-1,-1.5)--(0,0);
\draw [fill=pink!40](-6,-1.5)--(-6.5,-2.25)--(-7.5,-2)--(-8,-1);
\draw [fill=pink!40](-6,-1.5)--(-4,-2)--(-5.5,-2.5);
\draw[-, pink!90](-3,-2.5)--(-1,-1.5);
 
\draw[-, very thick, pink!90](-4,-2)--(-4.5,-3);   
\draw[-, pink!90](-7.5,-2)--(-6.5,-2.25);  
   
\draw [-] (-4,-2)--(-3,-2.5);
\draw [-] (-2,-1)--(-1,-1.5);
\draw [-] (-4,-2)--(-1,-1.5);

\draw [-] (-4,-2)--(-5.5,-2.5);
\draw [-, pink!90] (-6,-1.5)--(-5.5,-2.5);

\draw [-] (-6,-1.5)--(-6.5,-2.25);
\draw [-] (-7.5,-2)--(-6,-1.5);
\draw [->] (-8,-1)--(-9,1);
\draw [->] (-6,-1.6)--(-7,1);
\draw [->] (-6,-1.6)--(-6,1);

\draw [fill=pink!40](0,0) -- (-1.75,-0.25)--(-0.5,1.2)--(-2,2)--(-1,2.5)--(-1.5,4);
\draw [fill=pink!40](-2,2)--(-1,2.5)--(-1.5,4);
\draw [fill=pink!40](-2,2)--(-4,3)--(-2.75,0.75);
\draw [fill=pink!40](-2,2)--(-2,0.35)--(-0.5,1.2);

    \draw [-,pink!90] (-2.75,0.75)--(-4,3);
    \draw [-]  (0,0) -- (-1.75,-0.25);
    \draw [-,pink!90] (-1.75,-0.25)--(-0.5,1.2);
    \draw [-] (-2,2)--(-2.75,0.75);

\draw [-] (-0.5,1.2)--(-2,0.35);
\draw [-,pink!90] (-2,2)--(-2,0.35);

\draw [-] (-2,2)--(-1,2.5);
\draw [-] (-2,2)--(-1,2.5);
\draw [-] (-0.5,1.2)--(-2,2);
\draw [->] (-4,3)--(-5,5);

\draw [fill=pink!40](-1.5,4) -- (-2,6)--(-2.5,5)--(-1.5,4);
\draw [fill=pink!40](-2.5,5)--(-2.25,4)--(-1.5,4)--(-2.5,5);
     \draw [-] (-1.5,4)--(-2.25,4);
     \draw [-, pink!90] (-2.5,5)--(-2.25,4);

\draw [-] (-1.5,4)--(-2.5,5);
\draw [->] (-2,6)--(-2,8);

\draw [-, color=black]  (-1, -1.5) -- (0,0) -- (-2,-1)--(-4, -2) -- (-6, -1.5) --(-8,-1)-- (-7.5, -2);
\draw [-, color=black] (0,0) -- (3,3) -- (1.5,-0.8);
\draw [-, color=black] (0,0) -- (-1.5,4) -- (-2,2) -- (-4,3);
\draw [-, color=black] (-1.5,4) -- (-2,6) -- (-2.5,5);


\draw [-, line width=2.5pt, color=magenta] (-8,-1) -- (-7.5,-2)--(-6,-1.5)--(-4,-2)--(-1,-1.5)--(-2,-1)--(0,0)--(1.45,-0.8)--(1,1)--(3,3);
\draw [-, line width=2.5pt, color=magenta] (0,0)--(-0.5,1.2)--(-2,2)--(-4,3);
\draw [-, line width=2.5pt, color=magenta] (-2,2)--(-1,2.5)--(-1.5,4)--(-2,6);

\draw [fill=pink!40](-2,6) -- (-2,8) -- (-1,8.5)--cycle;
\draw [->] (-2,6)--(-2,8);
\draw [-] (-2,8)--(-1,8.5);
\draw [-, line width=2.5pt, color=magenta] (-2,6)--(-1,8.5);

\draw [fill=pink!40](-4,3)--(-5,5)--(-4.5,5.5)--cycle;
\draw [->] (-4,3)--(-5,5);
\draw [-] (-5,5)--(-4.5,5.5);
\draw [-, line width=2.5pt, color=magenta] (-4,3)--(-4.5,5.5);

\draw [fill=pink!40] (3,3)--(3.5,5.5)--(5,6)--cycle;
\draw [->] (3,3)--(3.5,5.5);
\draw [-] (3.5,5.5)--(5,6);
\draw [-, line width=2.5pt, color=magenta](3,3)--(5,6);

\draw [fill=pink!40] (2,2)--(2,5)--(3,5.5)--cycle;
\draw [->] (2,2)--(2,5);
\draw [-] (2,5)--(3,5.5);
\draw [-, line width=2.5pt, color=magenta](2,2)--(3,5.5);

\draw [fill=pink!40](-8,-1)--(-9,1)--(-7.7,1.5)--cycle;
\draw [->] (-8,-1)--(-9,1);
\draw [-] (-9,1)--(-7.7,1.5);
\draw [-, line width=2.5pt, color=magenta](-8,-1)--(-7.7,1.5);

\draw [fill=pink!40] (-6,-1.5)--(-7,1)--(-6.5,1.5)--cycle;
\draw [->] (-6,-1.5)--(-7,1);
\draw [-] (-7,1)--(-6.5,1.5);
\draw [-, line width=2.5pt, color=magenta](-6,-1.5)--(-6.5,1.5);

\draw [fill=pink!40] (-6,-1.5)--(-6,1)--(-5.5,1.5)--cycle;
\draw [->] (-6,-1.5)--(-6,1);
\draw [-] (-6,1)--(-5.5,1.5);
\draw [-, line width=2.5pt, color=magenta] (-6,-1.5)--(-5.5,1.5);

\node[draw,circle, inner sep=2.5pt,color=magenta,fill=magenta] at (-4,-2){};
\node[draw,rectangle, inner sep=1pt,color=black,fill=white] at (-4.5,-3){};
\node[draw,rectangle, inner sep=1pt,color=black,fill=white] at (-5.5,-2.5){};
\node[draw,rectangle, inner sep=1pt,color=black,fill=white] at (-6.5,-2.25){};
\node[draw,rectangle, inner sep=1pt,color=black,fill=white] at (-3,-2.5){};
\node[draw,rectangle, inner sep=1pt,color=black,fill=white] at (0.5,-1){};
\node[draw,rectangle, inner sep=1pt,color=black,fill=white] at (-2,0.35){};
\node[draw,rectangle, inner sep=1pt,color=black,fill=white] at (-2.75,0.75){};
\node[draw,rectangle, inner sep=1pt,color=black,fill=white] at (-2.24,4){};
\node[draw,rectangle, inner sep=1pt,color=black,fill=white] at (-1.75,-0.25){};
\end{scope}


\begin{scope}[shift={(18,0)},scale=0.75]
\draw [fill=pink!40](0,0) -- (3,3) -- (1.5,-0.8)--(0,0);
\draw [fill=pink!40](0,0) --(0.5,-1) -- (1.5,-0.8)--(0,0);
\draw[-, pink!90](0.5,-1)--(1.5,-0.8);
\draw [-] (0,0) --(0.5,-1);

\draw [-] (0,0)--(1.5,-0.8);
\draw [-] (1,1)--(1.5,-0.8);
\draw [-] (2,2)--(1.5,-0.8);
\draw [->] (3,3)--(3.5,5.5);
\draw [->] (2,2)--(2,5);

\draw [fill=pink!40](0,0) -- (-4,-2)--(-3,-2.5)--(-1,-1.5)--(0,0);
\draw [fill=pink!40](-6,-1.5)--(-6.5,-2.25)--(-7.5,-2)--(-8,-1);
\draw [fill=pink!40](-6,-1.5)--(-4,-2)--(-5.5,-2.5);
\draw[-, pink!90](-3,-2.5)--(-1,-1.5);
   
\draw[-, very thick, pink!90](-4,-2)--(-4.5,-3);   
\draw[-, pink!90](-7.5,-2)--(-6.5,-2.25);  
   
\draw [-] (-4,-2)--(-3,-2.5);
\draw [-] (-2,-1)--(-1,-1.5);
\draw [-] (-4,-2)--(-1,-1.5);

\draw [-] (-4,-2)--(-5.5,-2.5);
\draw [-, pink!90] (-6,-1.5)--(-5.5,-2.5);

\draw [-] (-6,-1.5)--(-6.5,-2.25);
\draw [-] (-7.5,-2)--(-6,-1.5);
\draw [->] (-8,-1)--(-9,1);
\draw [->] (-6,-1.6)--(-7,1);
\draw [->] (-6,-1.6)--(-6,1);

\draw [fill=pink!40](0,0) -- (-1.75,-0.25)--(-0.5,1.2)--(-2,2)--(-1,2.5)--(-1.5,4);
\draw [fill=pink!40](-2,2)--(-1,2.5)--(-1.5,4);
\draw [fill=pink!40](-2,2)--(-4,3)--(-2.75,0.75);
\draw [fill=pink!40](-2,2)--(-2,0.35)--(-0.5,1.2);

    \draw [-,pink!90] (-2.75,0.75)--(-4,3);
    \draw [-]  (0,0) -- (-1.75,-0.25);
    \draw [-,pink!90] (-1.75,-0.25)--(-0.5,1.2);
    \draw [-] (-2,2)--(-2.75,0.75);

\draw [-] (-0.5,1.2)--(-2,0.35);
\draw [-,pink!90] (-2,2)--(-2,0.35);

\draw [-] (-2,2)--(-1,2.5);
\draw [-] (-2,2)--(-1,2.5);
\draw [-] (-0.5,1.2)--(-2,2);
\draw [->] (-4,3)--(-5,5);

\draw [fill=pink!40](-1.5,4) -- (-2,6)--(-2.5,5)--(-1.5,4);
\draw [fill=pink!40](-2.5,5)--(-2.25,4)--(-1.5,4)--(-2.5,5);
     \draw [-] (-1.5,4)--(-2.25,4);
     \draw [-, pink!90] (-2.5,5)--(-2.25,4);
       
\draw [-] (-1.5,4)--(-2.5,5);
\draw [->] (-2,6)--(-2,8);

\draw [-, color=black]  (-1, -1.5) -- (0,0) -- (-2,-1)--(-4, -2) -- (-6, -1.5) --(-8,-1)-- (-7.5, -2);
\draw [-, color=black] (0,0) -- (3,3) -- (1.5,-0.8);
\draw [-, color=black] (0,0) -- (-0.5,1.2)--(-1,2.5)--(-1.5,4) -- (-2,2) -- (-4,3);
\draw [-, color=black] (-1.5,4) -- (-2,6) -- (-2.5,5);


\draw [-, line width=2.5pt, color=magenta] (-1.5,-1.25)--(-2.5,-1.75)--(-3,-2.5)--(-4.5,-3)--(-5.5,-2.5)--(-6.5,-2.25)--(-6.75,-1.75);
\draw[-][line width=2.5pt, color=magenta](-1.5,-1.25) .. controls (-0.55,-0.55) ..(0.5,-1);
\draw [-, line width=2.5pt, color=magenta](0.5,-1)--(0.75,-0.4)--(1.25,0.1)--(1.75,0.6);

\draw[-][line width=2.5pt, color=magenta](-6.75,-1.75) .. controls (-6.35,-1.25) ..(-6,-0.25);
\draw[-][line width=2.5pt, color=magenta](-6.75,-1.75) .. controls (-7,-1.25) ..(-6.5,-0.25);
\draw[-][line width=2.5pt, color=magenta](-6.75,-1.75) .. controls (-8.5,-1.25) ..
(-8.5,0);

\draw[-][line width=2.5pt, color=magenta](1.25,0.1) .. controls (1.55,2.3) ..
(2,3.25);
\draw[-][line width=2.5pt, color=magenta](1.75,0.6) .. controls (2.05,2.3) ..
(3.25,4.25);

\draw[-][line width=2.5pt, color=magenta](-1.5,-1.25) .. controls (-1.45,-0.55) ..(-1.75,-0.25);

\draw[-][line width=2.5pt, color=magenta](-1.75,-0.25).. controls (-1.45,0.25) ..(-2,0.35);

\draw[-][line width=2.5pt, color=magenta](-2,0.35).. controls (-2.35,0.95) ..(-2.75,0.75);

\draw[-][line width=2.5pt, color=magenta](-2.75,0.75).. controls (-3.95,2) ..(-4.5,4);

\draw[-][line width=2.5pt, color=magenta](-2,0.35)--(-1.25,1.6)--(-1.5,2.25);

\draw[-][line width=2.5pt, color=magenta](-1.5,2.25).. controls (-2.3,3) ..(-2.25,4);

\draw[-][line width=2.5pt, color=magenta](-2.25,4)--(-2,4.5);

\draw[-][line width=2.5pt, color=magenta](-2,4.5).. controls (-2.5,6) ..(-2,7);

\node[draw,rectangle, inner sep=1pt,color=black,fill=white] at (-4.5,-3){};
\node[draw,rectangle, inner sep=1pt,color=black,fill=white] at (-5.5,-2.5){};
\node[draw,rectangle, inner sep=1pt,color=black,fill=white] at (-6.5,-2.25){};
\node[draw,rectangle, inner sep=1pt,color=black,fill=white] at (-3,-2.5){};
\node[draw,rectangle, inner sep=1pt,color=black,fill=white] at (0.5,-1){};
\node[draw,rectangle, inner sep=1pt,color=black,fill=white] at (-2,0.35){};
\node[draw,rectangle, inner sep=1pt,color=black,fill=white] at (-2.75,0.75){};
\node[draw,rectangle, inner sep=1pt,color=black,fill=white] at (-2.24,4){};
\node[draw,rectangle, inner sep=1pt,color=black,fill=white] at (-1.75,-0.25){};
\node[draw,rectangle, inner sep=1pt,color=black,fill=white] at (-1.5,-1.25){};
\node[draw,rectangle, inner sep=1pt,color=black,fill=white] at (-2.5,-1.75){};
\node[draw,rectangle, inner sep=1pt,color=black,fill=white] at (0.5,-1){};
\node[draw,rectangle, inner sep=1pt,color=black,fill=white] at (0.75,-0.4){};
\node[draw,rectangle, inner sep=1pt,color=black,fill=white] at (1.25,0.1){};
\node[draw,rectangle, inner sep=1pt,color=black,fill=white] at (1.75,0.6){};
\node[draw,rectangle, inner sep=1pt,color=black,fill=white] at (-6.75,-1.75){};
\node[draw,rectangle, inner sep=1pt,color=black,fill=white] at (-6,-0.25){};
\node[draw,rectangle, inner sep=1pt,color=black,fill=white] at (-6.5,-0.25){};
\node[draw,rectangle, inner sep=1pt,color=black,fill=white] at (-8.5,0){};
\node[draw,rectangle, inner sep=1pt,color=black,fill=white] at (2,3.25){};
\node[draw,rectangle, inner sep=1pt,color=black,fill=white] at (3.25,4.25){};

\node[draw,rectangle, inner sep=1pt,color=black,fill=white] at (-2,7){};
\node[draw,rectangle, inner sep=1pt,color=black,fill=white] at (-4.5,4){};
\node[draw,rectangle, inner sep=1pt,color=black,fill=white] at (-1.25,1.6){};
\node[draw,rectangle, inner sep=1pt,color=black,fill=white] at (-1.5,2.25){};
\node[draw,rectangle, inner sep=1pt,color=black,fill=white] at (-2,4.5){};
\end{scope}

\end{tikzpicture}
\end{center}
\caption{Two ways of visualizing the Enriques tree on a truncated lotus}  
   \label{fig:lotoroidtruncenriq}
    \end{figure}
  
  It is convenient to draw in a same picture both the truncation $\Lambda_{\pi}^{tr}(C)$ 
and the union of the ropes. For instance, for the lotus of Figure \ref{fig:lotustoroid} 
this union is represented on the right side of Figure \ref{fig:lotoroidtruncenriq}. 
For comparison, the Enriques tree is represented on the left side. An advantage 
of the right-side drawing is that the ropes whose interiors lie outside the truncation are 
exactly the ropes which were represented by Enriques as curved arcs. One may 
similarly determine from this drawing which edges go straight in Enriques' convention. 
For details, one may consult \cite[Thm. 6.2]{PP 11}. Note that the 
\emph{kites} of the title of \cite{PP 11} (in French \emph{cerf-volants}) 
were the unions of truncated lotuses and of their ropes, as represented 
on the right side of Figure \ref{fig:lotoroidtruncenriq}.

        Assume now that the combinatorial type of a plane curve singularity is given either 
        using the dual graph of its total transform by an embedded resolution, weighted by 
        the self-intersection numbers of the components of its exceptional divisor, or using 
        the Enriques diagram of the decomposition of the resolution morphism into blow ups 
        of infinitely near points of $o$. How to get a series $f \in \C[[x,y]]$ defining a curve 
        singularity with the given combinatorial type? 
        
        One may apply the following steps:
                \medskip 
                
                \noindent
                $\bullet$
                 Pass from the given tree to the associated truncated lotus. If the 
                   given tree is an Enriques diagram, it may be more convenient for 
                   drawing purposes to think about it as the union of ropes of the 
                   truncated lotus which is searched for. 
                
                \noindent
                $\bullet$ Complete the truncated lotus into a lotus having it as truncation. 
                     This step is not canonical, as shown by the comparison of Figures 
                     \ref{fig:two-lotuses2} and \ref{fig:two-lotuses} above. 
                
                \noindent
                $\bullet$ Proceed as in Example \ref{ex:fromFTtoEW} below, by constructing 
                     the fan tree of the lotus, then the associated Eggers-Wall tree and writing 
                     finally a finite set of Newton-Puiseux series whose associated Eggers-Wall tree 
                     is isomorphic with this one. 

    \pagebreak

\subsection{Historical comments}
\label{ssec:HAembres}
$\:$
\medskip

The study of plane curve singularities by using sequences of blow ups of points 
was initiated by Max Noether in his 1875 paper \cite{N 75}, 
and became common in the meantime, as shown by 
the works \cite{N 90} of Noether, \cite{EC 17} of Enriques and Chisini, \cite{V 36} of 
Du Val and  \cite[Sections I.2, II.2]{Z 35}, \cite{Z 38} of Zariski. 

Nowadays, 
 a modification of $\C^2$ obtained as a sequence of blow ups of points is 
studied most of the time through the structure of its exceptional divisor. One encodes the 
incidences between its components, as well as their self-intersection numbers in 
a \emph{weighted dual graph}, which is a tree (see \cite{PP 18} for a description of the 
development of this idea). When one looks at an 
\emph{embedded resolution} of the plane curve singularity $C$,  one adds 
new vertices to this graph, corresponding to the strict 
transforms of the branches of $C$.  

\begin{figure}[h!] 
  \centering 
  \includegraphics[scale=0.33]{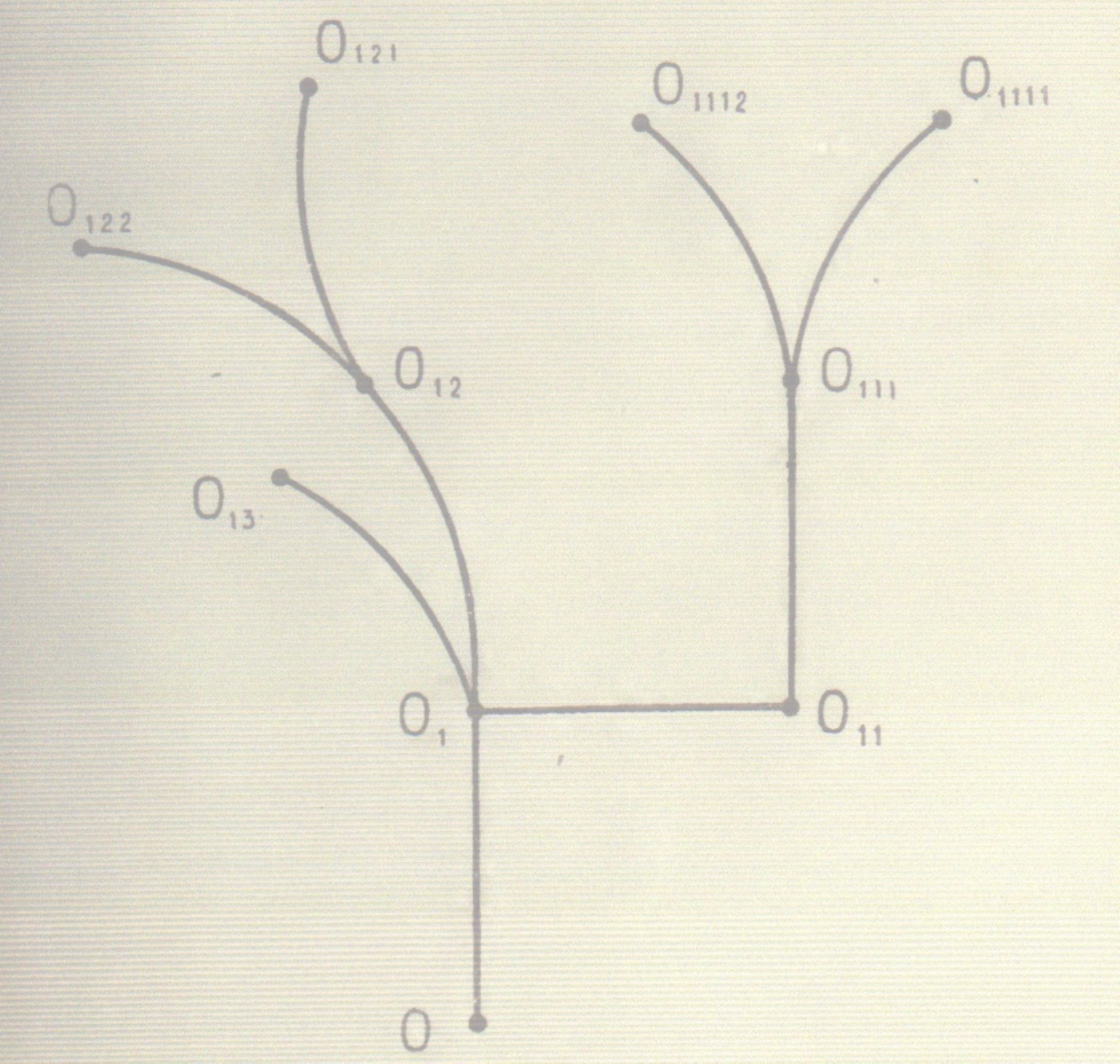} 
 \caption{An Enriques diagram} 
  \label{fig:Enriques-diag}
  \end{figure} 

The dual trees of exceptional divisors were not the first graphs associated with a 
process of blow ups of points. Another kind of tree, an \emph{Enriques 
diagram},  encoding the proximity relation between the infinitely near points which are 
blown up in the process (see Definition \ref{def:infnear}), 
was associated with such a process in the 1917 book \cite{EC 17} of Enriques 
and Chisini. An example of an Enriques diagram, 
extracted from \cite[Page 383]{EC 17}, may be seen in Figure \ref{fig:Enriques-diag}.  
Details about the notion of Enriques diagram may be found in 
Casas' book \cite{CA 00} or in the third author's papers \cite{PP 11}, \cite{PPPP 14}, 
the second written in collaboration with Pe Pereira. 
      The proximity relation was extended to higher dimensions by Semple 
      in his 1938 paper \cite{S 38}. Details about this generalization and about other approaches 
      to the study of curve singularities of higher embedding dimension may be found in 
      Campillo and Castellanos' 2005 book \cite{CC 05}.

   In order to understand the 
relation between the Enriques diagram of a finite constellation 
and the dual graph of the blow up of the constellation, the third author 
introduced the notion of \emph{kite} in his 2011 paper \cite{PP 11}. A kite was defined by gluing 
\emph{lotuses} into a \emph{sail}, and attaching then \emph{ropes} to this sail. 
The ropes were lying inside each lotus as the veins in a leaf, and they allowed to visualize the 
Enriques diagram. In turn, the dual graph could be visualised as the lateral boundary of the 
sail. A sail was composed not only of \emph{petals}, but also 
of \emph{axes} and \emph{semi-petals}. 
The lotuses were also used in Castellini's thesis 
\cite{C 15}, written under the supervision of the third author. 
Castellini was able to do everything with petals, eliminating the use of axes, semi-petals and ropes, 
as what we call here the \emph{Enriques 
tree} of a lotus proved to be more convenient to visualize the Enriques diagram.
Also, the terminology was simplified, the gluing of
lotuses resulting again in lotuses, instead of sails, 
as we do in the present paper.

\begin{figure}[h!] 
 \centering 
 \includegraphics[scale=0.55]{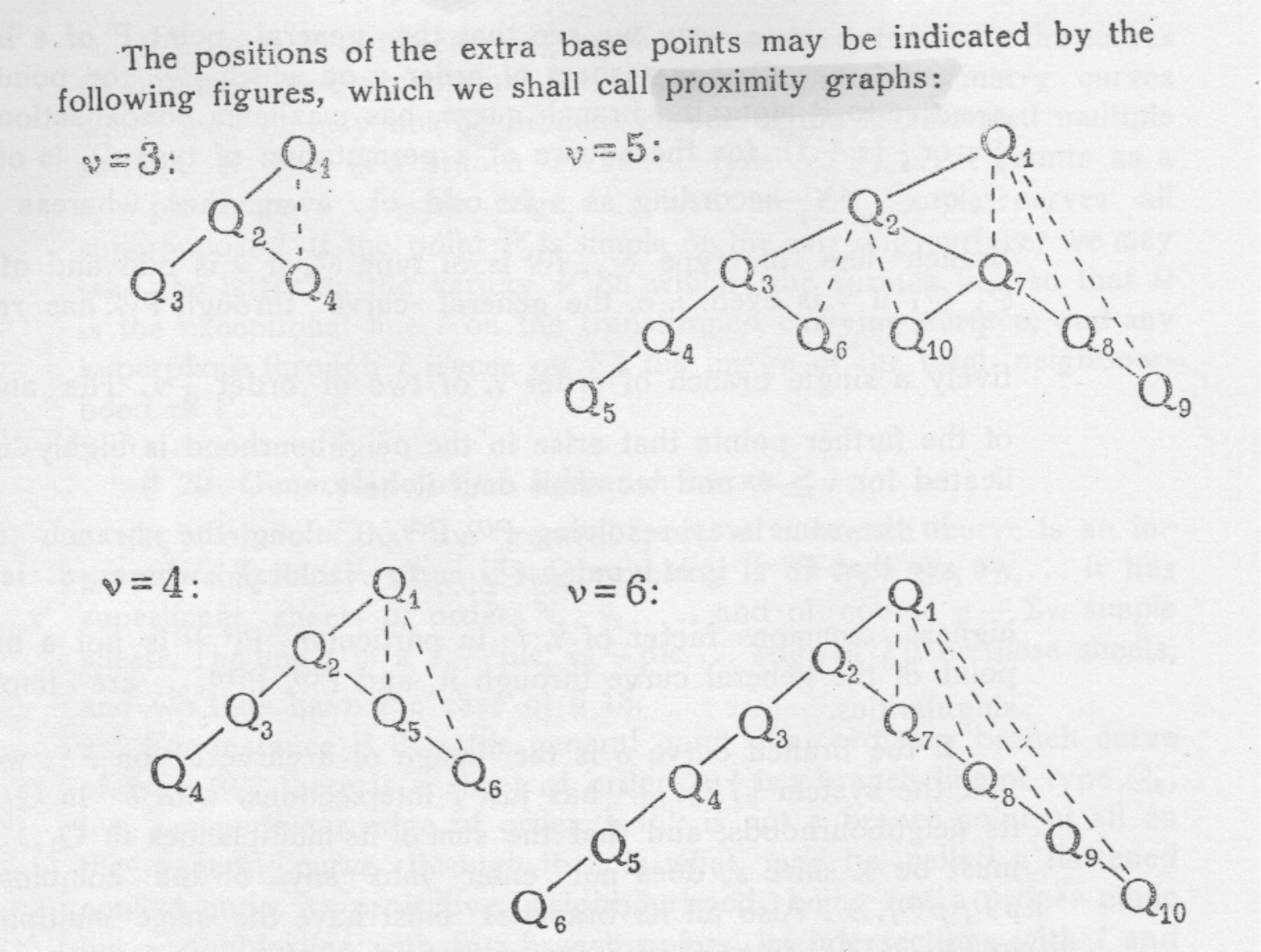} 
 \vspace*{1mm} 
 \caption{Du Val's ``proximity graphs''} 
 \label{fig:DV1-graph}
 \end{figure}

 \begin{figure}[h!] 
 \centering 
 \includegraphics[scale=0.45]{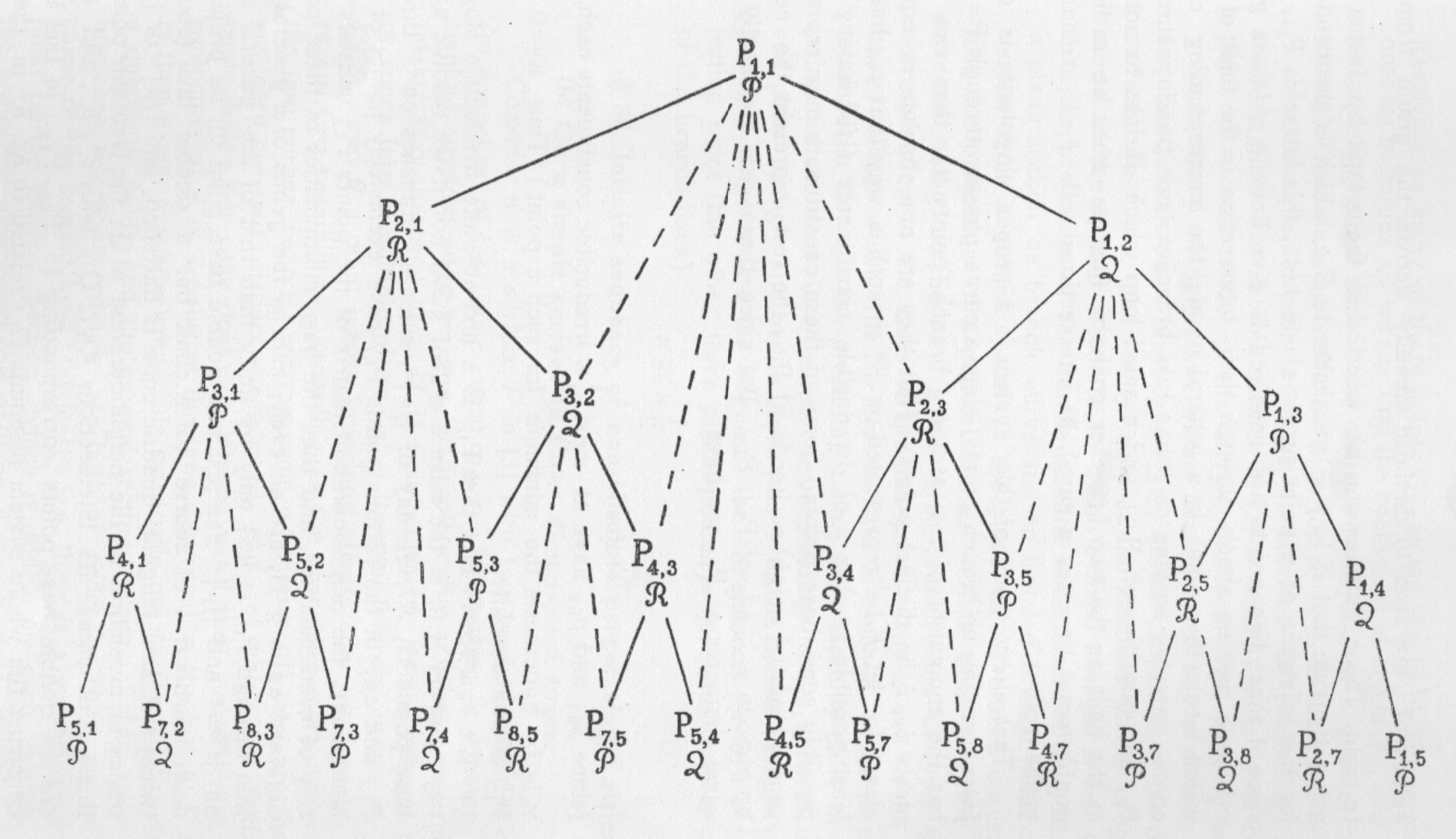} 
 \vspace*{1mm} 
 \caption{Du Val's version of universal lotus} 
 \label{fig:DV2-graph}
 \end{figure}

\medskip
It turns out that lotuses already appeared in disguise before the paper \cite{PP 11}. Their 
oldest ancestor is probably the \emph{proximity relation}, 
defined in Enriques and Chisini's book \cite[Page 381]{EC 17}.  Indeed  
(see Theorem \ref{thm:repsailtor} \eqref{point:graphprox}), 
the graph of the proximity relation among all the points whose blow up 
composes the embedded resolution produced by the second algorithm described 
in our paper may be identified with the full subgraph of the $1$-skeleton of the associated lotus 
on the set of vertices which are not basic. The oldest drawings of such proximity graphs 
 seem to be those of Du Val's 1944 paper \cite{V 44} 
(see Figures \ref{fig:DV1-graph} and \ref{fig:DV2-graph}, in which one may also 
recognize what we call the ``Enriques tree'' of a lotus, drawn with continuous segments). Before, 
the proximity binary relation was related to the exceptional divisor of the associated 
blow up process in Barber and Zariski's  1935 paper \cite{BZ 35} and 
Du Val's 1936 paper \cite{V 36}.  Du Val introduced the notion of 
\emph{proximity matrix}, equivalent to that of proximity binary relation. 
In his 1939 paper \cite{Z 39}, Zariski began a new ideal-theoretical and valuation-theoretical 
trend in the study of infinitely near points. A geometrical  
presentation of the previous approaches of study of infinitely near points 
was given by Lejeune-Jalabert in her 1995 paper \cite{LJ 95}.

The graph of the proximity relation was mentioned again by Deligne in his 1973 paper \cite{D 73}, 
by Morihiko Saito in his 2000 paper \cite{S 00} and by Wall in his 2004 book 
\cite[Sections 3.5, 3.6]{W 04}. One may find drawings of 
simple such graphs only in the first and the third reference.  

 \begin{figure}[h!] 
 \centering 
 \includegraphics[scale=0.65]{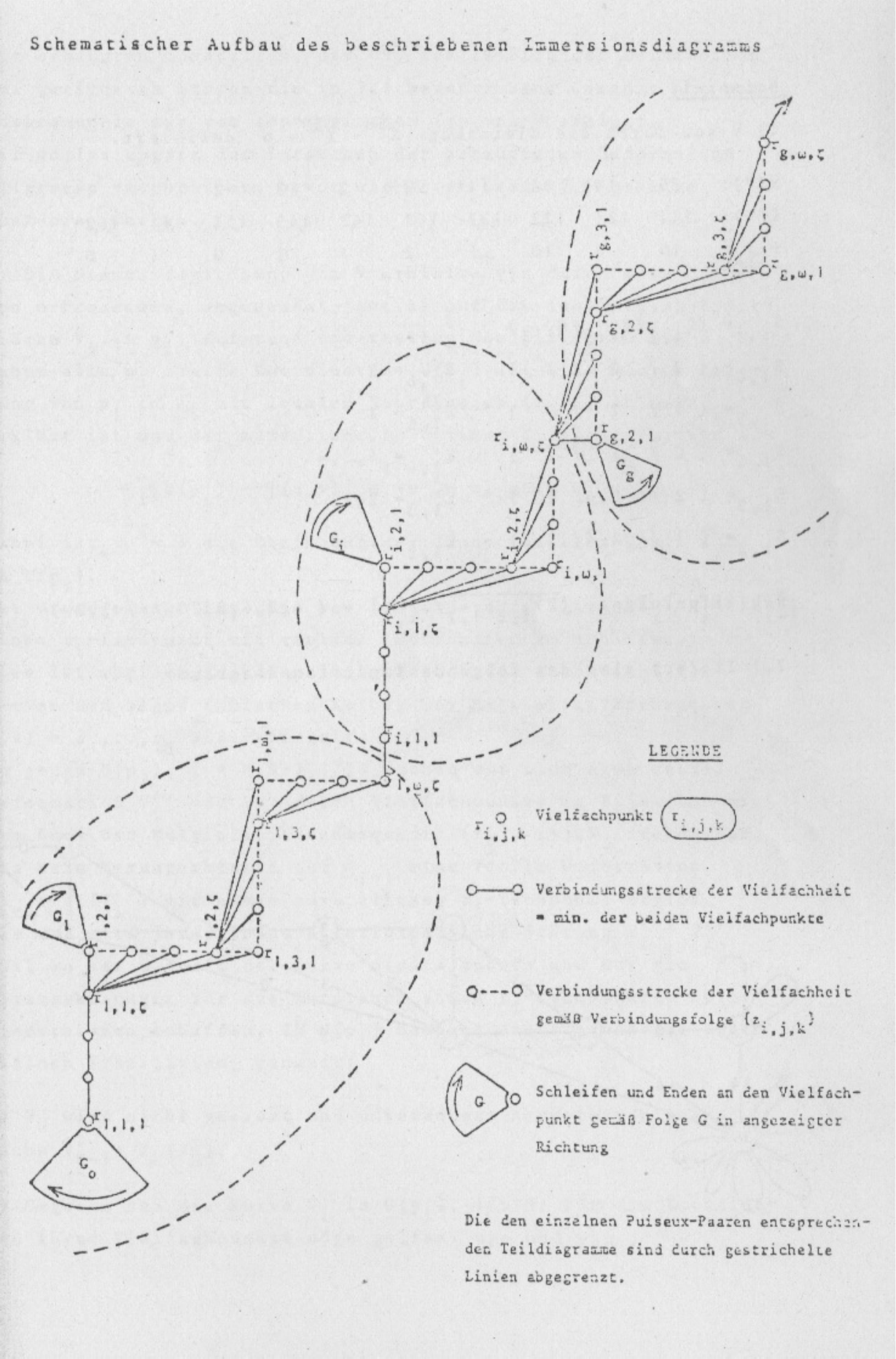} 
 \caption{The general shape of Schulze-R\"obbecke's divides} 
 \label{fig:SR-graph}
 \end{figure} 

Another occurrence of lotuses in disguise may be found in Schulze-R\"obbecke's 1977 
Diplomarbeit \cite{SR 77} written under the supervision of Brieskorn. 
In that paper are described particular divides (generic immersions of segments in a disc) 
obtained by applying to \emph{branches} A'Campo's method of constructing 
$\delta$-constant deformations explained in the 1974-75 papers \cite{A 74} and \cite{A 75}. 
The diagram of Figure \ref{fig:SR-graph}, extracted from page 
57 of \cite{SR 77}, indicates the general shape of the divides constructed in that paper. 
One may recognize 
inside it part of the lotus associated to a toroidal resolution process of a branch. 
In his already mentioned 2015 PhD thesis \cite{C 15}, Castellini could extend Schulze-R\"obbecke's description to arbitrary plane curve singularities, using in a crucial way 
the notion of lotus of a blow up process.

     \medskip
      Let us discuss now the relation of the universal lotus introduced in 
      Definition \ref{def:univlot} with other objects and constructions.
      The Enriques tree of the universal lotus $\Lambda(e_1, e_2)$ is an embedding 
        into the cone $\sigma_0$ of almost all the \emph{Stern-Brocot tree} 
        defined by Graham, Knuth and Patashnik in \cite[Page 116]{GKP 94}, 
         in reference to the 1858 paper \cite{S 58} of Stern and the 1860 paper \cite{B 60} of 
        Brocot. This tree represents the successive generation 
        of the positive rational numbers starting from the sequence $(0/1, 1/0)$. At each step of the 
        generating process, one performs the \emph{Farey addition} 
        $(a/b, c/d) \to  (a+c)/ (b+d)$ on the pairs of successive terms of the 
        increasing sequence of rationals obtained at the previous steps. The vertices of the 
        Stern-Brocot tree correspond bijectively with the positive rationals. For each Farey addition 
        $(a/b, c/d) \to  (a+c)/ (b+d)$ in which $c/d$ was created after $a/b$, one joins the vertices 
        corresponding to $c/d$ and to $ (a+c)/ (b+d)$. The embedding of the 
        Stern-Brocot tree represented in Figure \ref{fig:Univenriques} is obtained by sending 
        each vertex corresponding to $\lambda \in \Q \cap (0, \infty)$ to the primitive vector 
        $p(\lambda) \in N \cap \sigma_0$ (see Notations \ref{not:prim}) and each edge to a  
        Euclidean segment.  Another embedding in the cone 
        $\sigma_0$ of the same part of the Stern-Brocot tree 
        as above was described in \cite[Rem. 5.7]{PP 11}. That embedding may be obtained 
        from the embedding of Figure \ref{fig:Univenriques} by applying a homothety of factor $1/2$. 
        
The sequence of continued fractions \eqref{eq-cf}  
    appearing in the proof of Proposition \ref{prop:lotus-wedge} was called the 
    \emph{slow approximation} (``\emph{approximation lente}'') of 
   $[a_1, \dots, a_k]$ in L\^e, Michel and Weber's  paper \cite[Appendice]{LMW 89}. 
   They used such sequences in order to describe the construction of the dual graph of the 
   minimal embedded resolution of a plane curve singularity starting from the 
   generic characteristic exponents of its branches and the orders of coincidence between 
   such branches.

   The \emph{zigzag decompositions} introduced in Definition \ref{def:decomponenumber}
   are a variant of the \emph{zigzag diagrams} 
   of the third's author 2007 paper \cite[Section 5.2]{PP 07}. Those diagrams 
   allow to relate geometrically the usual continued fractions to the so-called 
   \emph{Hirzebruch-Jung continued fractions}. 
  Those Hirzebruch-Jung continued fractions are the traditional tool, going back to 
  Jung's 1908 paper \cite{J 08} and Hirzebruch's 1953 paper \cite{H 53}, to describe 
  the regularization of a $2$-dimensional strictly convex cone. They are also crucial 
  for the understanding of \emph{lens spaces}, which becomes obvious once one 
  sees that those $3$-manifolds are exactly the links of  toric surface  
  singularities. See Weber's survey \cite{W 18} for more details and historical explanations 
about the relations between lens spaces and complex surface singularities.

        In \cite[Section 9.1]{NW 05}, Neumann and Wahl described a method for 
        reconstructing the dual graph of the minimal resolution of a complex normal surface 
        singularity whose link is an integral homology sphere from the so-called 
        \emph{splice diagram}
        of the link. This method is based on the construction of a finite sequence 
        of rationals interpolating between two given positive rational numbers $\lambda$ and 
        $\mu$. It may be described in the following way using lotuses of sequences of positive 
        rational numbers:

          \noindent
          $\bullet$
          Construct by successive additions of petals the lotus $\Lambda(\lambda, \mu)$ 
                   as the union of $\Lambda(\lambda)$ and $\Lambda(\mu)$. 
          
             \noindent
          $\bullet$
          Consider the increasing sequence of slopes of vertices of $\Lambda(\lambda, \mu)$  
                  lying between $\lambda$ and $\mu$, that is, of vertices of the lateral boundary 
                  $\partial_+ \Lambda(\lambda, \mu)$ (see Definition \ref{def:lotus-point}) 
                  lying on the arc joining the primitive vectors 
                  $p(\lambda)$ and $p(\mu)$ of $N$. 
\medskip 

       In \cite[Section 2.2]{FG 16}, Fock and Goncharov described the \emph{tropical boundary 
       hemisphere 
       of the Teichm\"uller space of the punctured torus} as an infinite simplicial complex with integral 
       vertices embedded in the real affine space associated to a two-dimensional affine lattice.  
       This simplicial complex is a union of universal lotuses (see \cite[Fig. 1]{FG 16}).

\section{Relations of fan trees and lotuses with Eggers-Wall trees}
\label{sec:FTEW}
\medskip

In Subsection \ref{ssec:EW} we explain how to associate an \emph{Eggers-Wall} tree 
$\Theta_{L}(C)$ to a plane 
curve singularity $C \hookrightarrow (S,o)$, relative to a smooth branch $L$. 
It is a rooted tree endowed with three structure functions, the \emph{index} $\de_L$, 
the \emph{exponent} $\ex_L$ and the \emph{contact complexity} $\ic_L$. 
 In Subsection \ref{ssec:EW-Newton}  we express the Newton polygon of $C$ 
relative to a cross $(L, L')$ in terms of the Eggers-Wall tree $\Theta_L(C + L')$ of $C+L'$ 
relative to $L$ (see Corollary \ref{cor:Newton}). 
In Subsection \ref{ssec:transfanEW} we prove that the fan tree 
$\theta_{\pi}(C)$ associated with a toroidal 
pseudo-resolution process of $C$ is canonically isomorphic   
with the Eggers-Wall tree $\Theta_L(\hat{C}_{\pi})$ of the completion of $C$ relative to this 
process (see Theorem \ref{thm:isomfantreeEW}), and we explain how to compute 
the triple $(\de_L, \ex_L, \ic_L)$ of functions starting 
from the slope function of the fan tree (see Proposition \ref{prop:slopedetindex}). 
As a prerequisite, in Subsections \ref{ssec:renormres} and 
\ref{ssec:transfanEW-Newton} we prove 
\emph{renormalization formulae}, which compare the Eggers-Wall tree of $C$ 
relative to $L$ and those of its strict transform relative to the exceptional 
divisor of a Newton modification.

\subsection{Finite Eggers-Wall trees and the universal Eggers-Wall tree}
\label{ssec:EW}
$\:$
\medskip

In this subsection we define the \emph{Eggers-Wall tree} $\Theta_L(C)$ of a \emph{reduced} 
plane curve singularity $C \hookrightarrow (S,o)$ relative to a smooth branch $L$ 
(see Notations \ref{not:finotEW}). It is constructed 
from the Newton-Puiseux series of $C$ relative to a local coordinate system $(x,y)$ such that 
$L = Z(x)$ (see Definition \ref{def:EW}), 
but it is independent of this choice (see Proposition \ref{prop:indepx}). It is a rooted tree 
whose root is labeled by $L$ and whose leaves are labeled by the branches of $C$. 
It is endowed with three functions, the \emph{index} $\de_L$, the \emph{exponent} $\ex_L$ 
and the \emph{contact complexity} $\ic_L$, which allow to compute the characteristic exponents 
of the Newton-Puiseux series mentioned above and the intersection numbers of the 
branches of $C$ (see Proposition \ref{prop:intcomp}). 
Finally, we introduce the \emph{universal Eggers-Wall 
tree} of $(S,o)$ relative to $L$ (see Definition \ref{def:univEW}), as the projective limit of the 
Eggers-Wall trees of the plane curve singularities contained in $S$.
For more details and proofs one may consult our papers 
\cite[Subsection 4.3]{GBGPPP 18a} and \cite[Section 3]{GBGPPP 18b}. 
\medskip

Let $L$ be a smooth branch on $(S,o)$. Assume in the whole subsection that $C$ is \emph{reduced}. 
Let $(x,y)$ be a local coordinate system on $(S,o)$, such that $L = Z(x)$, and let $f \in \C[[x,y]]$ 
be a defining function of $C$ in this coordinate system. 
As a consequence of the Newton-Puiseux Theorem \ref{thm:NPthmbasic}, one has: 

\begin{theorem}  \label{thm:NewtPuiseux}
    Assume that $C$ does not contain $L$, that is, that $x$ does not divide $f(x,y)$. 
   Then there exists a finite set $\boxed{\cZ_x(f)}$ of Newton-Puiseux series of $\C[[x^{1/ \N}]]$ 
   and a unit $u(x,y)$ of the local ring $\C[[x,y]]$, such that:
       \begin{equation}
       \label{fmla:NewtPuiseux}
                 f(x,y) = u(x,y) \prod_{\eta(x) \in \cZ_x(f)} (y - \eta(x)). 
       \end{equation}
\end{theorem}

The set $\cZ_x(f)$ is obviously independent of the defining function $f$ of $C$. 
For this reason, we will denote it instead $\cZ_x(C)$. It is the disjoint union of the 
sets $\cZ_x(C_l)$, when $C_l$ varies among the branches of $C$. It allows to associate 
to $f$ the following objects: 

\begin{definition}  \label{def:charcont}
    Let $(x,y)$ be a local coordinate system of $(S,o)$ such that $L = Z(x)$ and let $C$ be a reduced 
    curve singularity on $(S,o)$ not containing $L$. 
              
              \medskip 
              \noindent 
              $\bullet$
              The finite subset $\boxed{\cZ_x(C)}:= \cZ_x(f)$ from the statement of Theorem 
                 \ref{thm:NewtPuiseux} is called the set of {\bf Newton-Puiseux roots} of 
                 $C$ relative to $x$. \index{Newton-Puiseux!root}

              \noindent 
              $\bullet$
              The {\bf order of coincidence} $\boxed{k_x(\xi, \xi')}$ of two Newton-Puiseux series 
                  $\xi, \xi'$ is equal to $\nu_x(\xi - \xi')$. \index{order!of coincidence}
              
              \noindent 
              $\bullet$
              The {\bf order of coincidence} $\boxed{k_x(C_l, C_m)}$ of two distinct 
                 branches $C_l$ and $C_m$ of $C$ is the maximal order of coincidence 
                 of Newton-Puiseux roots of the two branches: 
                 $\max\{ k_x(\xi, \xi') , \:  \xi \in \cZ_x(C_l), \: \xi' \in \cZ_x(C_m) \}$.

              \noindent 
              $\bullet$
              The {\bf set of characteristic exponents} $\boxed{{\mathrm{Ch}}_x(C_l)}$ 
                 \index{exponent!characteristic}
                   of a branch $C_l$ of $C$ relative to the variable $x$ is 
                  the set of orders of coincidence of pairs of distinct Newton-Puiseux roots of it:
                  $\{ k_x(\xi, \xi') , \:  \xi, \xi'  \in \cZ_x(C_l), \: \xi \neq \xi'  \}$.
                                     
\end{definition}

This shows that for each 
$\xi \in \cZ_x(C_l)$, there exists some $\xi' \in \cZ_x(C_m)$ such that 
$\nu_x(\xi - \xi') = k_x(C_l, C_m)$. Therefore, knowing a Newton-Puiseux root 
of $C_l$ determines some Newton-Puiseux root of $C_m$ until their order of 
coincidence $k_x(C_l, C_m)$. 
This fact motivates the following construction of a rooted tree endowed 
with two functions:

\medskip
\begin{definition} \label{def:EW}
    Let $(x,y)$ be a local coordinate system such that $L = Z(x)$ and $C$ be a reduced 
    curve singularity on $(S,o)$.

              \noindent 
              $\bullet$
		The {\bf Eggers-Wall tree} \index{Eggers-Wall tree} $\boxed{\Theta_x(C_l)}$ of a branch 
              $C_l \neq L$ of $C$ 
              relative to $x$ is a compact segment endowed with a homeomorphism 
              $\boxed{\ex_x}: \Theta_x(C_l) \to [0, \infty]$ called the {\bf exponent function},  
              \index{function!exponent} \index{exponent!function}
              and with {\bf marked points}, \index{point!marked, of an Eggers-Wall tree} 
              which are the preimages by the exponent function 
              of the characteristic exponents of $C_l$ relative to $x$. The point 
              $(\ex_x)^{-1}(0)$ is labeled by $L$ and $(\ex_x)^{-1}(\infty)$ is labeled 
              by $C_l$. The {\bf index function}   \index{function!index} 
              $\boxed{\de_x}: \Theta_x(C_l) \to \N^*$ 
              whose value $\de_x(P)$ on a point $P \in \Theta_x(C_l)$ is equal to the lowest common 
                multiple of the denominators of the exponents of the marked points belonging 
                to the half-open segment $[L, P)$.

              \noindent 
              $\bullet$
            The {\bf Eggers-Wall tree} $\boxed{\Theta_x(L)}$ is reduced to a point  
             labeled by $L$, at which $\ex_x(L) = 0$ and $\de_x(L) = 1$.

              \noindent 
              $\bullet$
             The {\bf Eggers-Wall tree} $\boxed{\Theta_x(C)}$ of $C$ relative to $x$ 
            is obtained from the disjoint union of the Eggers-Wall trees $\Theta_x(C_l)$ 
            of its branches  by identifying, for each pair of distinct branches $C_l$ and $C_m$ 
            of $C$, their points with equal exponents not greater than the order of 
            coincidence $k_x(C_l, C_m)$. Its {\bf marked points} are 
            its ramification points and the images 
            of the marked points of the trees $\Theta_x(C_l)$ by the 
            identification map. Its {\bf labeled points} \index{point!labeled, of an Eggers-Wall tree} 
            are analogously the images of the 
            labeled points of the trees $\Theta_x(C_l)$, the identification map being 
            label-preserving. The tree is rooted at the point labeled by $L$. 
            It is endowed with an {\bf exponent function} 
            $\boxed{\ex_x}: \Theta_x(C) \to [0, \infty]$ and an index function 
            $\boxed{\de_x}: \Theta_x(C) \to \N^*$
            obtained by gluing the exponent functions  and index functions  
            on the trees $\Theta_x(C_l)$ respectively. 
\end{definition}

Note that, by construction, the exponent function is surjective in restriction to 
every segment $[L, C_l] = \Theta_x(C_l)$ of $\Theta_x(C)$ such that $C_l \neq L$
and that the ends of $\Theta_x(C)$ 
are labeled by the branches of $C$ and by the smooth reference branch $L$. 
The marked points of $\Theta_x(C)$ which are images of marked points 
of the subtrees $\Theta_x(C_l)$ may be recovered from the knowledge 
of the index function, as its set of points of discontinuity. Therefore, the index function is constant 
on each open edge between two consecutive marked points of $\Theta_x(C)$. 
Moreover, it is continuous from below relative to the 
partial order $\preceq_L$ defined by the root $L$ of $\Theta_x(C)$.

The Eggers-Wall tree allows to determine visually the characteristic exponents of each branch $C_l$. 
One has simply to follow the segment going from the root to the leaf representing the branch 
 and to read all the vertex weights of the discontinuity points of the index function. 
 In particular, if an internal vertex of such a segment is not a ramification vertex of the tree,  
 then its exponent is necessarily a characteristic exponent of $C_l$.

\begin{figure}[h!]
\begin{center}
\begin{tikzpicture}[scale=0.6]
 \draw [-, color=black, thick](0,0) -- (0,6) ;
 \draw [-, color=black, thick](0,2) -- (2,3) ;
  \node[draw,circle, inner sep=1.5pt,color=black, fill=black] at (0,0){};
  \node [left] at (0,0) {$0$};
   \node [below] at (0,-0.5) {$Z( x )$};
  \node[draw,circle, inner sep=1.5pt,color=black, fill=black] at (0,6){};
  \node [above] at (0,6) {$Z(y^3-x^7)$};
  \node[draw,circle, inner sep=1.5pt,color=black, fill=black] at (0,2){};
  \node [left] at (0,2) {$\frac{3}{2}$};
  \node[draw,circle, inner sep=1.5pt,color=black, fill=black] at (0,4){};
  \node [left] at (0,4) {$\frac{7}{3}$};
    \node[draw,circle, inner sep=1.5pt,color=black, fill=black] at (2,3){};
 \node [right] at (2,3) {$Z(y^2-4x^3)$};
   \node[right] at (0,1) {$1$};
    \node[right] at (0,3) {$1$};
\node[right] at (0,5) {$3$};
\node[below] at (1,2.5) {$2$};

\begin{scope}[shift={(11,0)}]
\draw [-, color=black, thick](0,0) -- (0,6) ;
 \draw [-, color=black, thick](0,2) -- (2,3) ;
 \draw [-, color=black, thick](0,4) -- (-2,5) ;
  \node[draw,circle, inner sep=1.5pt,color=black, fill=black] at (0,0){};
  \node [left] at (0,0) {$0$};
  \node [below] at (0,-0.5) {$Z( x )$};
  \node[draw,circle, inner sep=1.5pt,color=black, fill=black] at (0,6){};
  \node [above] at (0,6) {$Z(y^3-x^7)$};
  \node[draw,circle, inner sep=1.5pt,color=black, fill=black] at (0,2){};
  \node [below,left] at (0,1.7) {$\frac{3}{2}$};
   \node[draw,circle, inner sep=1.5pt,color=black, fill=black] at (-2,5){};
    \node [left] at (-2,5) {$Z(y)$};
  \node[draw,circle, inner sep=1.5pt,color=black, fill=black] at (0,4){};
  \node [left] at (0,3.5) {$\frac{7}{3}$};
    \node[draw,circle, inner sep=1.5pt,color=black, fill=black] at (2,3){};
 \node [right] at (2,3) {$Z(y^2-4x^3)$};
   \node[right] at (0,1) {$1$};
    \node[right] at (0,3) {$1$};
      \node[below] at (0-1,4.5) {$1$};
\node[right] at (0,5) {$3$};
\node[below] at (1,2.5) {$2$};
\end{scope}
 \end{tikzpicture}
\end{center}
\caption{The Eggers-Wall trees of  $Z(f(x,y))$ and $Z(xyf(x,y))$ 
      from Example \ref{ex:EWexamplenndeg}}
\label{fig:EWfirstexample}
\end{figure}
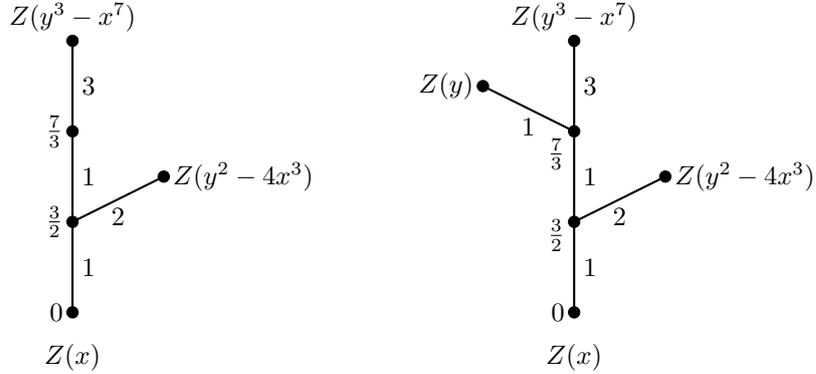

\begin{example} \label{ex:EWexamplenndeg}
    Consider again the plane curve singularity 
     $C = Z(f(x,y))$ of Subsection \ref{ssec:redexample}. 
     That is, $f(x,y) = (y^2-4x^3)(y^3-x^7)$.
     Its Eggers-Wall tree  is drawn on the 
left side of Figure \ref{fig:EWfirstexample}. On the right side is drawn the 
Eggers-Wall tree of the singularity $Z(xy(y^2-4x^3)(y^3-x^7))$, which is the sum of 
$C$ and of the coordinate axes. 

 Look  at the segment joining the root to the branch $Z(y^3-x^7)$, on the left side of Figure \ref{fig:EWfirstexample}. It contains 
two internal vertices, with exponents $3/2$ and $7/3$. The vertex of exponent $7/3$ is not 
a ramification vertex of the tree, 
therefore $7/3$ is a characteristic exponent of this branch. 
In turn, $3/2$ is not a characteristic exponent 
of this branch, as the value of the index function  
does not increase when crossing the corresponding vertex. 
Note that, by contrast, it increases when crossing the same vertex on the segment joining 
the root to the leaf corresponding to the branch $Z(y^2-4x^3)$, which shows that 
$3/2$ is a characteristic exponent of that branch. 

We have represented both the Eggers-Wall tree of $C$ and of its union with the coordinate 
axes in order to show that \emph{the second one is homeomorphic to 
the dual graph of the total transform of the union by its minimal embedded resolution}, 
while our example shows that this is not true if one looks at the total transform of $C$ alone 
(see Figure \ref{fig:simpledualgraphs}). 
The previous homeomorphism is a general phenomenon, valid for any plane curve singularity, 
as seen by  combining Proposition \ref{prop:fantreedualgr} and 
Theorem \ref{thm:isomfantreeEW} below. Note that in full generality one needs to add to $C$ 
more branches than simply the coordinate axes, considering a 
\emph{completion} in the sense of Definition \ref{def:threeres}.

\end{example}

 \begin{example} \label{ex:EWexample}
      Consider a plane curve singularity $C$ whose branches 
     $C_i$, $1\leq i\leq 3$,  are defined by the Newton-Puiseux series $\xi_i$, where:
      \[
       \xi_1 = x^{7/2} -x^{4}+ 2 x^{17/4} + x^{14/3}, \quad 
       \xi_2 =  x^{5/2} + x^{8/3}, \quad 
       \xi_3 = x^2.
       \]    
       The sets of characteristic exponents of the branches are   
    ${\mathrm {Ch}}_x(C_1)=\left\{ 7/2, 17/4, 14/3 \right\}$, 
    ${\mathrm {Ch}}_x(C_2)=\left\{5/2, 8/3 \right\}$,
    ${\mathrm {Ch}}_x(C_3)=\emptyset$. 
      One has $k_{x}(C_1, C_2) = 5/2$,  $k_{x}(C_1, C_3) = k_{x}(C_2, C_3) =2$.  
    The Eggers-Wall trees $\Theta_x(C_1)$ and $\Theta_x(C)$ relative to $x$ are 
    drawn in Figure \ref{fig:EWfive}. 
    We represented the value of the corresponding exponent near each marked or labeled point, 
    and the value of the corresponding index function near each edge. 
 \end{example}

\begin{figure}[h!] 
\begin{center}
\begin{tikzpicture}[scale=0.5]
     \draw [-, color=black, thick](-10,0) -- (-10, 12) ; 
   \node[draw,circle, inner sep=1.5pt,color=black, fill=black] at (-10,0){};
   \node [right] at (-10,0) {$L$};
   \node [left] at (-10,0) {$\mathbf 0$};
   \node[draw,circle, inner sep=1.5pt,color=black, fill=black] at (-10,12){};
   \node [right] at (-10,12) {${C_{1}}$};
   \node [left] at (-10,12) {$\infty$};
   \node[draw,circle, inner sep=1.5pt,color=black, fill=black] at (-10,6){};
   \node [left] at (-10,6) {$\mathbf {\frac{7}{2}}$};
    \node[draw,circle, inner sep=1.5pt,color=black, fill=black] at (-10,8){};
   \node [left] at (-10,8) {$\mathbf {\frac{17}{4}}$};
     \node[draw,circle, inner sep=1.5pt,color=black, fill=black] at (-10,10){};
   \node [left] at (-10,10) {$\mathbf {\frac{14}{3}}$};
    \node[right] at (-10,3) {$1$};
     \node[right] at (-10,7) {$2$};
     \node[right] at (-10,9) {$4$};
     \node[right] at (-10,11) {$12$};
     
     \node[right] at (-11,-1) {$\Theta_x(C_1)$};
     \node[right] at (-1,-1) {$\Theta_x(C)$};

  \draw [-, color=black, thick](0,0) -- (0, 12) ; 
   \node[draw,circle, inner sep=1.5pt,color=black, fill=black] at (0,0){};
   \node [right] at (0,0) {$L$};
   \node [left] at (0,0) {$\mathbf 0$};
   \node[draw,circle, inner sep=1.5pt,color=black, fill=black] at (0,12){};
   \node [right] at (0,12) {${C_{1}}$};
   \node [left] at (0,12) {$\infty$};
   \node[draw,circle, inner sep=1.5pt,color=black, fill=black] at (0,2){};
   \node [left] at (0,2) {$\mathbf 2$};
   \node[draw,circle, inner sep=1.5pt,color=black, fill=black] at (0,4){};
   \node [left] at (0,4) {$\mathbf {\frac{5}{2}}$};
   \node[draw,circle, inner sep=1.5pt,color=black, fill=black] at (0,6){};
   \node [left] at (0,6) {$\mathbf {\frac{7}{2}}$};
    \node[draw,circle, inner sep=1.5pt,color=black, fill=black] at (0,8){};
   \node [left] at (0,8) {$\mathbf {\frac{17}{4}}$};
     \node[draw,circle, inner sep=1.5pt,color=black, fill=black] at (0,10){};
   \node [left] at (0,10) {$\mathbf {\frac{14}{3}}$};
 \draw [-, color=black, thick](0,2) -- (2, 3) ; 
  \node[draw,circle, inner sep=1.5pt,color=black, fill=black] at (2,3){};
  \node [below] at (2,3) {${C_{3}}$};
   \node [above] at (2,3) {$\infty$};  
    \node [below] at (1,2.5) {$1$};
 \draw [-, color=black, thick](0,4) -- (4,6) ;   
    \node[draw,circle, inner sep=1.5pt,color=black, fill=black] at (4,6){};
  \node [below] at (4,6) {${C_{2}}$};
   \node [above] at (4,6) {$\infty$}; 
   \node[draw,circle, inner sep=1.5pt,color=black, fill=black] at (2,5){};  
     \node [below] at (2,5) {$\mathbf {\frac{8}{3}}$};
     \node [below] at (1,4.5) {$2$};
     \node [below] at (3,5.5) {$6$};
     \node[right] at (0,1) {$1$};
     \node[right] at (0,3) {$1$};
     \node[right] at (0,5) {$1$};
     \node[right] at (0,7) {$2$};
     \node[right] at (0,9) {$4$};
     \node[right] at (0,11) {$12$};
\end{tikzpicture}
\end{center}
\caption{The Eggers-Wall tree of the curve singularities $C_1$ and $C$ of Example      
         \ref{ex:EWexample}} 
\label{fig:EWfive}
\end{figure}
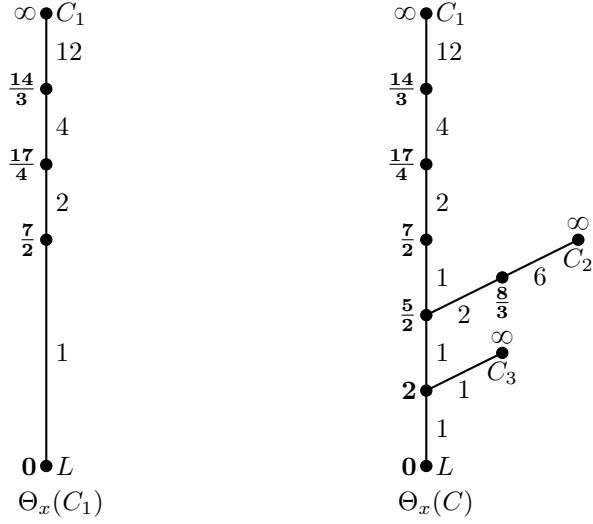

In fact, the objects introduced in Definition \ref{def:EW}
depend only on $C$ and $L$, not on the coordinate system 
$(x,y)$ such that $L = Z(x)$ (see \cite[Proposition 103]{GBGPPP 18a}):

\begin{proposition}  \label{prop:indepx}
    Let $(x,y)$ be a local coordinate system such that $L = Z(x)$ and $C$ be a reduced 
    curve singularity on $(S,o)$. Then the tree $\Theta_x(C)$ endowed with the pair 
    of functions $(\de_x, \ex_x)$ is independent of the choice of local coordinate system 
    such that $L = Z(x)$. 
 \end{proposition}
 
 Proposition \ref{prop:indepx}  motivates us to introduce the following notations:
 
 \begin{notation}  \label{not:finotEW}
     Let $L$ be a smooth branch and $C$ be a reduced curve singularity on $(S,o)$. 
     We denote 
     $\boxed{(\Theta_L(C), \de_L, \ex_L )} := (\Theta_x(C), \de_x, \ex_x )$,
     for any local coordinate system $(x,y)$ on $(S, o)$ such that $L = Z(x)$. 
     We say that this rooted tree endowed with two structure functions is the 
     {\bf Eggers-Wall tree of $C$ relative to $L$}. \index{Eggers-Wall tree}
 \end{notation}

\begin{remark}
Let $L$ be a  smooth branch and $C$ be a reduced curve singularity on $(S,o)$. 
Then for any point $Q\in \Theta_L (C)$, we have: 
\begin{equation} \label{index-refor}
     \de_L (Q) = \min \{ \de_L (A) \, , \, A \mbox { is a branch on } S \mbox{ such that } Q \preceq_L A \}, 
\end{equation}
where $Q \preceq_L A$ has a meaning in the Eggers-Wall-tree 
$\Theta_L (C+A) \supseteq \Theta_L(C)$.
Indeed, if $Q \preceq_L C_l$ for a branch $C_l$ of $C$, 
and if  $B$ is a branch on $S$ parametrized by the truncation of 
a Newton-Puiseux series  $\xi \in Z_x(C_l)$,
obtained by keeping only the terms of $\xi$ of exponent $< \ex_L (Q)$,  
then $Q \preceq_L B$ and $\de_L (Q) = \de_L (B)$.
\end{remark}

The exponent function and the index function determine a third function on the tree 
$\Theta_{L}(C)$, the \emph{contact complexity} function \index{function!contact complexity} 
 (see \cite[Def. 3.19]{GBGPPP 18b}): 

\begin{definition} \label{def:concom}
   Let $C$ be a reduced  curve singularity on $(S,o)$. 
    The {\bf contact complexity function} $\boxed{\ic_L} : \Theta_L(C) \to [0, \infty]$ 
                is defined by the formula: 
                   \[\ic_L(P) := \int_L^P \frac{d \ex_L}{\de_L}. \]
      \end{definition}
      
Note that in restriction to a segment $[L, C_l] = \Theta_L(C_l)$ of $\Theta_L(C)$, 
the contact complexity function  is a bijection $[L, C_l] \to [0, \infty]$.

\begin{remark}  \label{rem:detexp}
It follows immediately from Definition \ref{def:concom} that the contact complexity 
function together with the index function determine  the 
exponent function by the following formula: 
\begin{equation} \label{exp-int}
      \ex_L (P) = \int_L^P \de_L d \ic_L.
\end{equation}
\end{remark}

The importance of the contact complexity function stems from the following property, which 
in different formulation goes back at least to Smith \cite[Section 8]{S 74}, 
Stolz \cite[Section 9]{S 79} and Max Noether \cite{N 90}:  

\begin{proposition}  \label{prop:intcomp}
     Let $L$  be a smooth branch and $C$ be a reduced 
    curve singularity on $(S,o)$, not containing $L$.  
    Let $A$ and $B$ be two distinct branches of $C$.     
     Denote by $\boxed{A \wedge_L B}$ the infimum of the points of $\Theta_L(C)$ labeled 
     by $A$ and $B$, relative to the partial order $\preceq_L$ defined by the root $L$.   Then:
      \begin{equation} \label{f-intcomp}
             {\ic}_L(A \wedge_L B) = \frac{A \cdot B}{(L \cdot A) \cdot (L \cdot B)}.
      \end{equation}      
\end{proposition}

\begin{proof}   One may find a proof of Proposition \ref{prop:intcomp} 
  in \cite[Thm. 4.1.6]{W 04}. Let us just sketch 
 the main idea.    Fix a local coordinate system $(x,y)$ on $(S,o)$, such that $L = Z(x)$.
    Start from a normalization of the branch $A$ of the form 
    $u \to (u^n, \zeta(u))$ (see the explanations leading to formula (\ref{eq:paramNP})). 
    Therefore, $\zeta(x^{1/n})$ is a Newton-Puiseux root of $A$. By 
    Theorem \ref{thm:NewtPuiseux}, one has a defining function of the branch $B$ of the form
       $  \prod_{\eta(x) \in \cZ_x(B)} (y - \eta(x))$. 
   Proposition \ref{prop:disymint} implies that:
        \[  A \cdot B = \nu_u \left( \prod_{\eta(x) \in \cZ_x(B)} (\zeta(u) - \eta(u^n)) \right) = 
                 \sum_{\eta(x) \in \cZ_x(B)}   \nu_u \left( \zeta(u) - \eta(u^n) \right)  .\]
     The finite multi-set of rational numbers whose elements are summed may be expressed in terms of 
     the characteristic exponents of $A$ and $B$ which are not greater than 
     the order of coincidence of $A$ and $B$. A little computation finishes the proof.  
\end{proof}

\medskip

If $C$ and $D$ are two reduced plane curve singularities on $(S,o)$, with $C \subseteq D$, 
then by construction one has a natural embedding of rooted trees $\Theta_L(C) 
\subseteq \Theta_L(D)$. The uniqueness of the segment joining two points of a tree 
allows to define a canonical retraction $\Theta_L(D) \to \Theta_L(C)$. One may 
consider then either the \emph{direct limit of the previous embeddings}, 
or the \emph{projective limit of the previous retractions}, for varying $C$ and $D$.  
Both limits have natural topologies. The direct limit, which may be thought simply as the 
union of all Eggers-Wall trees $(\Theta_L(C))_C$, is not compact, but the projective 
limit is compact. It is in fact a compactification of the direct limit. For this reason, 
the projective limit is more suitable in many applications. Let us introduce a special notation 
for this notion, which will be used in Subsection \ref{ssec:renormres} below. 

\begin{definition}  \label{def:univEW}
    Let $L$ be a smooth branch on $(S, o)$. \index{Eggers-Wall tree!universal} 
   The {\bf universal Eggers-Wall tree $\boxed{\Theta_L}$ of $(S,o)$ relative to $L$} is the projective limit 
   of the Eggers-Wall trees $\Theta_L(C)$ of the various reduced curve singularities $C$ 
   on $(S,o)$, relative to the natural retraction maps $\Theta_L(D) \to \Theta_L(C)$ 
   associated to the inclusions $C \subseteq D$. 
\end{definition}

\subsection{From Eggers-Wall trees to Newton polygons} 
\label{ssec:EW-Newton}  $\,$

\medskip

In this subsection we explain how the Newton polygon $\cN_{L, L'} (C)$
of a plane curve singularity $C$ relative to the cross $(L,L')$ (see Definition \ref{def:Npolseries}) 
may be determined from the Eggers-Wall tree $\Theta_L (C+ L')$ (see Corollary \ref{cor:Newton}). 
\medskip

The {\bf Minkowski sum}  \index{Minkowski sum} 
$\boxed{K_1 + K_2}$ of two subsets of a real vector space 
is the set of sums $v_1 + v_2$, where each $v_i$ varies independently among the 
elements of $K_i$. It is a commutative  and associative operation. 
When both subsets are convex, their Minkowski sum is again convex.

The following property is classical and goes back at least to Dumas' 1906 paper 
\cite[Section 3]{D 06} (where it was formulated in a slightly different, $p$-adic, context): 

\begin{proposition}  \label{prop:sgmorph}
    If $C$ and $D$ are germs of effective divisors on $(S, o)$, then: 
         \[\mathcal{N}_{L,L'} (C + D) = \mathcal{N}_{L,L'} (C) + \mathcal{N}_{L,L'} (D). \]
\end{proposition}

\begin{proof}
   This is a direct consequence of formula (\ref{eq:fromtropton}) and 
      Proposition \ref{prop:imptropintr}.
\end{proof}

\medskip

One may extend the notion of Newton polygon to series  in two variables with non-negative 
  rational exponents whose denominators are bounded. They have again only a 
  finite number of edges. The simplest Newton polygons are those with at 
  most one compact edge:

\begin{figure}[h!]
\begin{center}
\begin{tikzpicture}[scale=0.6]
\tikzstyle{every node}=[font=\small]
\foreach \x in {0,1,...,2}{
\foreach \y in {0,1,...,3}{
      \node[draw,circle,inner sep=0.7pt,fill, color=gray!40] at (1*\x,1*\y) {}; }
  }

\begin{scope}[shift={(7,0)}]
\foreach \x in {0,1,...,2}{
\foreach \y in {0,1,...,4}{
      \node[draw,circle,inner sep=0.7pt,fill, color=gray!40] at (1*\x,1*\y) {}; }
  }
\end{scope}

\begin{scope}[shift={(14,0)}]
\foreach \x in {0,1,...,4}{
\foreach \y in {0,1,...,2}{
      \node[draw,circle,inner sep=0.7pt,fill, color=gray!40] at (1*\x,1*\y) {}; }
  }
\end{scope}

\fill[fill=yellow!40!white] (0,3) -- (0,4.9) -- (4.9,4.9) -- (4.9,0) -- (2,0) -- cycle;
\fill[fill=yellow!40!white] (14,3) -- (14,4.9) -- (18.9,4.9) -- (18.9,3) -- cycle;
\fill[fill=yellow!40!white]  (8.9,0) -- (8.9,4.9) -- (11.9,4.9) -- (11.9,0)--cycle;

\draw[->] (0,0) -- (5,0) node[right,below] {$x^a$};
\draw[->] (0,0) -- (0,5) node[above,left] {$y^b$};

\draw[->] (7,0) -- (12,0) node[right,below] {$x^a$};
\draw[->] (7,0) -- (7,5) node[above,left] {$y^b$};

\draw[->] (14,0) -- (19,0) node[right,below] {$x^a$};
\draw[->] (14,0) -- (14,5) node[above,left] {$y^b$};

\draw[thick] (2,0) node[below] {$(a,0)$};
\node[draw,circle,inner sep=1.5pt,fill] at (2,0) {};
\draw[thick] (0,3) node[left] {$(0,b)$};
\node[draw,circle,inner sep=1.5pt,fill] at (0,3) {};
\draw[thick] (14,3) node[left] {$(0,b)$};
\node[draw,circle,inner sep=1.5pt,fill] at (14,3) {};
\draw[thick] (8.9,0) node[below] {$(a,0)$};
\node[draw,circle,inner sep=1.5pt,fill] at (8.9,0) {};
\draw [ultra thick](2,0) -- (0,3);
\draw [ultra thick](2,0) -- (4.9,0);
\draw [ultra thick](0,3) -- (0,4.9);

\draw [ultra thick](8.9,0) -- (8.9,4.9);
\draw [ultra thick](8.9,0) -- (11.9,0);

\draw [ultra thick](14,3) -- (14,4.9);
\draw [ultra thick](14,3) -- (18.9,3);
\draw (2.5,-1.5) node {$\Bigl\{\Teisssr{a}{b}{3}{1.5}\Bigr\}:={\cN}(x^a+y^b)$};
\draw (9.5,-1.5) node {$\Bigl\{\Teisssr{{a}}{{\infty}}{3}{1.5}\Bigr\}:={\cN}(x^a)$};
\draw (16.5,-1.5) node {$\Bigl\{\Teisssr{\infty}{b}{3}{1.5}\Bigr\}:={\cN}(y^b)$};
\end{tikzpicture}
\end{center}
\caption{The elementary Newton polygons $\elem(a,b)$, $\elem(a, \infty)$, $\elem(\infty, b)$  }
\label{fig:Elemnew}
\end{figure}
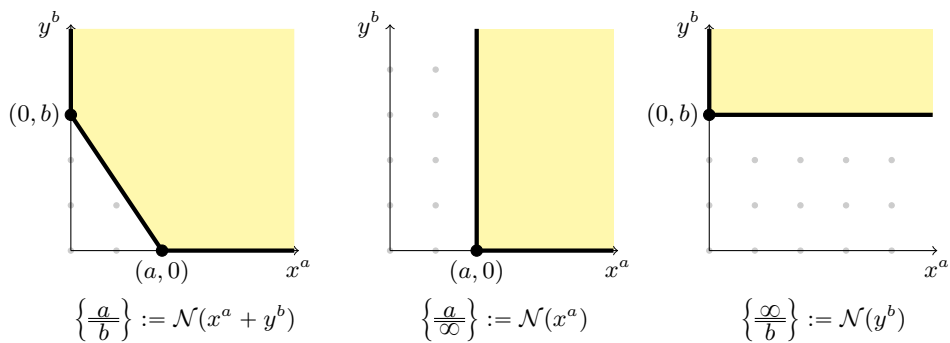

\begin{definition} \label{def:elempolyg}
   Assume that $a,b \in \Q_+^*$. One associates them the following 
   {\bf elementary Newton polygons} \index{polygon!elementary Newton} 
   (see Figure \ref{fig:Elemnew}):
     \[ 
            \boxed{ \Teissr{a}{b}}  
              := \cN (x^a + y^b),  \quad 
            \boxed{\Teissr{a}{ \infty}}  
               := \mathcal{N}  (x^a) , \quad 
            \boxed{\Teissr{\infty}{b}} 
              := \mathcal{N} (y^b). 
   \]
  The quotient  $a/b$ is the {\bf inclination} \index{inclination!of an elementary Newton polygon} 
  of  the elementary Newton polygon $ \Teissr{a}{b}$. 
 \end{definition}

Note that for any $a \in \Q_+^* \cup \{ \infty \}$, $b \in \Q_+^*$ and any $d \in \N^*$, one has:
    $ d \Teissr{a}{b}= \Teissr{da}{db}$,  
where the left-hand side is the Minkowski sum of $\Teissr{a}{b}$ with itself $d$ times. 
This allows to write:
  \begin{equation} \label{eq:multelem} 
       \Teissr{a}{b} =  b  \Teissr{a/b}{1}
  \end{equation}
whenever $b \in \N^*$. 
 The elementary Newton polygons are generators of the semigroup of 
 Newton polygons, with respect to Minkowski sum. In fact one has more:

 \begin{proposition} \label{uniqsum}
   Each Newton polygon $\mathcal{N}$ may be written in a unique 
   way, up  to permutations of the terms, as a Minkowski sum of elementary 
   Newton polygons with pairwise 
   distinct inclinations. Their compact edges are translations of the compact 
   edges of $\mathcal{N}$. 
\end{proposition}

    \begin{proof}
       This is a consequence of the following property, which in turn may be proved by 
       induction on $p \in \N^*$: \emph{If $\cN_1, \cN_2, \dots, \cN_p$ are elementary 
       Newton polygons with finite non-zero strictly increasing inclinations, then their Minkowski 
       sum $\cN$ has exactly $p$ compact edges which are translations of the compact edges 
       of $\cN_1, \cN_2, \dots, \cN_p$. Moreover, they are met in this order when one lists 
       them starting from the unique vertex of $\cN$ lying on the vertical axis.}
    \end{proof}

The next proposition explains how to compute the Newton polygon of a branch $C$ 
relative to a cross $(L, L')$, starting from the Eggers-Wall tree of $C+ L'$ relative to $L$: 

\begin{lemma} \label{lem:EWN}
     Let $(L, L')$ be a cross  and let $C \neq L$  be a branch on $(S, o)$.
     Then 
    the Newton polygon  $\mathcal{N}_{L,L'} (C)$ may be expressed as follows 
      in terms of the Eggers-Wall tree $(\Theta_L(C + L'), \ex_L, \de_L)$:
      \[ \cN_{L,L'} (C) = {\de}_{L}(C)  \Teissr{ {{\ex}_{L}(C \wedge_L L' )} }{1}.\] 
  The fan $\cF_{L, L'} (C)$ has a unique ray in the interior of the cone $\sigma_0$, 
      and  its slope is equal to ${\ex}_{L}(C \wedge_L L' )$. That is:
       \[  \cF_{L, L'}(C) =   \fan\left(  {\ex}_{L}(C \wedge_L L' )  \right).\]
\end{lemma} 

\begin{proof}
This is a consequence of Theorem \ref{thm:NewtPuiseux}. 
Indeed, let $f \in \C [[x]] [y ]$ be a defining function for $C$  relative to a local coordinate system 
$(x,y)$ defining the cross $(L, L')$. 
We know that  its set of Newton-Puiseux roots $\cZ_x(f)$ 
has $C\cdot L = {\de}_L (C)$ elements.
All of them have the same support, since $C$ is a branch, which implies that they 
form a single orbit under the Galois action of multiplication of $x^{1/{\de}_L (C)}$ 
by the group of ${\de}_L (C)$-th roots of $1$. 
The order of any such series  is equal to $k_x (L', C ) = {\ex}_L ( C \wedge_L L' )$. 
We deduce from Proposition \ref{prop:sgmorph}  
that the Newton polygon $\mathcal{N}_{L,L'} (C)$ is equal to the 
Minkowski sum of the factors of $f$ in  formula (\ref{fmla:NewtPuiseux}). 
The first assertion follows since the Newton polygon of $y - \eta(x)$ 
is equal to  $\Teissr{ {{\ex}_{L}(C \wedge_L L' )} }{1}$, for any series $\eta(x) \in \cZ_x(f)$,
and then by taking into account formula (\ref{eq:multelem}). 
The second assertion is an immediate consequence of the first one. 
\end{proof}

As a corollary we get the announced expression of 
the Newton polygon relative to $(L, L')$ of a reduced curve singularity $C$ in terms 
of the Eggers-Wall tree $\Theta_L(C + L')$ of $C+L'$ relative to $L$:

\begin{corollary}   \label{cor:Newton} 
   Let $(L, L')$ be a cross and let $C$  be a reduced curve singularity on $(S, o)$ 
   not containing the branch $L$. 
   The Newton polygon $\cN_{L,L'}  (C)$  of the germ $C$ with respect to 
   the cross $(L, L')$ is equal to the Minkowski sum:
    \begin{equation}  \label{eq:EWN}
         \sum_l   {\de}_{L}(C_l)\Teissr{ {{\ex}_{L}(C_l \wedge_L L' )} }{1},
    \end{equation}
    where $C_l$ runs through the branches of $C$. 
\end{corollary}
\begin{proof}
    By Proposition \ref{prop:sgmorph}, the Newton polygon $\cN_{L, L'}(C)$ is the 
    Minkowski sum of the Newton polygons of its branches. One uses then  
    Lemma \ref{lem:EWN} for each such branch. 
\end{proof}

Note that the previous result extends to not necessarily reduced curve singularities $C$
if one defines their Eggers-Wall tree as the Eggers-Wall tree of their reduction, each 
leaf being endowed with the multiplicity of the corresponding branch in the divisor $C$. 
Then, in the right-hand side of Equation (\ref{eq:EWN}), each branch $C_l$ has to be 
counted with its multiplicity.

\subsection{Renormalization of Eggers-Wall trees}
\label{ssec:renormres}
$\:$
\medskip

   Let $(L, L')$ be a cross on $(S,o)$. In this subsection we will denote  sometimes by 
   $\boxed{\Theta_{o, L}(C)}$ the Eggers-Wall tree denoted before by $\Theta_{L}(C)$, 
   in order to emphasize the point at which it is based. Indeed, we want 
   to compare the previous tree with the Eggers-Wall tree 
   $\Theta_{o_w, E_w}(C_w)$ of the germ $(C_w, o_w)$ of the strict transform $C_w$ of $C$ 
   at a smooth point $o_w$ of 
   the exceptional divisor  $E_w$ of a Newton modification relative to the cross $(L, L')$, 
   with respect to the germ at $o_w$ of the exceptional divisor $E_w$ itself. 
   Notice that if $C$ is a reduced curve, then the strict transform $C_w$ may consist of 
   several germs of curves, one for each point of intersection of $C_w$ with $E_w$.      
   We show that the universal Eggers-Wall tree $\Theta_{o_w, E_w}$ in the sense 
   of Definition \ref{def:univEW} 
   embeds naturally in the universal Eggers-Wall tree $\Theta_{o, L}$ 
   and we explain how to relate their triples of structure functions (index, exponent and 
   contact complexity).  We conceive the passage from $\Theta_{o, L}(C)$  to 
   $\Theta_{o_w, E_w}(C_w)$ as a \emph{renormalization operation}, which explains 
   the title of this subsection. \index{renormalization!of Eggers-Wall trees} 
   We will give another proof of the renormalization Proposition \ref{prop:renorm} 
   in Section \ref{ssec:transfanEW-Newton}, in terms of Newton-Puiseux series.
\medskip

Let us fix a cross $(L,L')$ on $(S,o)$. 
Fix also a weight vector $\boxed{w} = \boxed{c_w} e_1 + 
   \boxed{d_w} e_2  \in \sigma_0 \cap N_{L, L'}$.  
Denote by $\boxed{\pi_w : (S_w, \partial S_w) \to (S, L + L')}$
the modification obtained by subdividing $\sigma_0$ along the ray 
$\cone w$. If $A$ is a branch on $S$, we denote by $\boxed{A_w}$ the strict 
transform of $A$ by $\pi_w$. We look at the modification $\pi_w$ 
in the toroidal category, relative to the boundaries $\partial S := L+ L'$ and 
$\partial S_w := L_w + E_w + L_w'$, where $\boxed{E_w}$ is the exceptional 
divisor of the morphism $\pi_w$.

Denote  by $\boxed{W}$ the point of $\Theta_L(L')$ corresponding to $w$, that is, the unique point of $\Theta_L(L')$ whose exponent is the slope of the ray $\cone w$ in the basis $(e_1, e_2)$: 
  \begin{equation}   \label{eq:expu}
       \ex_L(W) = \frac{d_w}{c_w}. 
  \end{equation}
 Since $(L, L')$ is a cross on $(S,o)$ and $W \in \Theta_L (L')$, one has that 
$\de_L (W) =1$. Therefore, by Definition \ref{def:concom},  the contact complexity of $W$ is: 
 \begin{equation}   \label{eq:contactu}
     \ic_L (W) =  \frac{d_w}{c_w}.
 \end{equation}

Recall that $A \wedge_L B$ denotes the infimum of the points $A$ and $B$ 
of the universal Eggers-Wall tree $\Theta_{o, L}$ 
relative to  the partial order $\preceq_L$ induced by the root $L$. 
We need the following  lemma:

 \begin{lemma} \label{lem:EWN2}
  Let  $A$ be a branch on $(S,o)$ different from $L, L'$. 
     The following properties are equivalent: 
     \begin{enumerate}
         \item The strict transform $A_w$ of $A$ by $\pi_w$ intersects 
              $E_w \:  \setminus \:  (L_w \cup L'_w)$.
         \item The fan $\fan_{L, L'} (A)$ is the subdivision of $\sigma_0$ along the ray $\cone w$.
          \item $A \wedge_ L L' = W$.
      \end{enumerate}
      In addition, if these properties hold, then
     the order of vanishing of $A$ along $E_w$ is equal to $d_w  \de_L (A)$
     and  the intersection number $E_w \cdot A_w$ is $\de_L (A)/  c_w$.
\end{lemma} 

\begin{proof} The equivalence of these three properties is immediate from Propositions \ref{prop:propstrict} and  \ref{lem:EWN}.
Recall that the order of vanishing  $ \mbox{ord}_{E_w}(A) $  is by definition  
the multiplicity of $E_w$ in the divisor $(\pi_{w}^* L)$, that is, the 
value taken by the divisorial valuation $ \mbox{ord}_{E_w}$ defined by 
$E_w$ on a defining function $f$ of $A$. 
Thanks to Proposition \ref{prop:propstrict}, this value is equal to 
$\mbox{trop}^A_{L, L'} (w)$, which may be written $d_w  \de_L (A)$ by Lemma \ref{lem:EWN}.
By Proposition  \ref{prop:propstrict}, $E_w \cdot A_w$ is equal to the integral length of the compact 
edge of the Newton polygon $\cN_{L,L'}  (A)$. The equality 
$E_w \cdot A_w =\de_L (A)/ c_w$ follows by using Lemma \ref{lem:EWN} again.
\end{proof}

\begin{lemma} \label{lem:intersfund}
Let $A$ and $B$ be two branches on $(S,o)$.  Consider 
the point $W \in \Theta_L (L') $ fixed above, determined by relation (\ref{eq:expu}). 
Assume that $W = A \wedge_L L' = B \wedge_L L' $ inside the universal Eggers-Wall tree $\Theta_L$. 
Then the following conditions are equivalent: 
\begin{enumerate}
    \item $A \wedge_L B = W$. 
     \item $A\cdot B = \dfrac{d_w}{c_w} (L \cdot A) (L \cdot B)$.
     \item $A_w \cap E_w \ne B_w \cap E_w$.
\end{enumerate}
\end{lemma}

\begin{proof}$\:$ 

\noindent
  {\bf Proof of {\it (1}) $\Rightarrow$ {\it (2)}.} 
      This implication is a consequence of Formulae (\ref{f-intcomp}) and \eqref{eq:contactu}. 
      \medskip

\noindent
  {\bf Proof of  {\it (2)} $\Rightarrow$ {\it (1)}.} Let us denote by $W'$ the point $A \wedge_L B$. 
       The assumption $W \preceq_L A$, $W \preceq_L B$ implies that $W \preceq_L  W'$. 
        By Formula (\ref{f-intcomp}), we get $ \ic_L (W') =  d_w/ c_w = \ic_L (W)$. 
      Since the function $\ic_L$ is strictly increasing on $[L, A]$, the inequalities 
        $L \preceq_L W \preceq_L  W' \preceq_L A$ imply that  $W = W'$. 
     \medskip
       
\noindent
    {\bf Proof of {\it (1)} $\Leftrightarrow$ {\it (3)}.}  
          Let $(x, y)$ be a system of local coordinates defining the cross $(L, L')$. 
          Denote by $f_A$ a defining function of $A$ with respect to this system and 
          by $K_A$ the compact edge of the Newton polygon  $\cN_{L,L'}  (A)$. 
         By the proof of Lemma \ref{lem:EWN}, if $\alpha_A$ is the coefficient of  
          $x^{d_w / c_w }$ in a fixed Newton-Puiseux series of $A$,
          then the restriction of $f_A$ to the compact edge $K_A$ is equal to: 
          \[
       \left( 
         \prod_{\gamma^{c_w} =1}  (y - \alpha_A \,   \gamma  \,  
            x^{d_w/c_w})
             \right)^{\de_L (A) / c_w }   =  
                (y^{c_w} - \alpha_A^{c_w} x^{d_w})^{\de_L (A) / c_w }.
          \]
          We consider similar notations for the branch $B$. 
          By Proposition  \ref{prop:propstrict},  
          the point of intersection of the strict transform of $A_w$ with $E_w$ 
          is parametrized by the coefficient  $\alpha_A^{c_w}$.
          The desired equivalence follows since 
          $\alpha_A^{c_w} \ne \alpha_B^{c_w}  $ if and only if 
          for every $\gamma \in \C$ with $\gamma^{c_w} = 1$, 
          one has that 
          $\alpha_A \ne \gamma \cdot  \alpha_B$, 
          which is also equivalent to 
          the property  $k_L ( A, B) = d_w/ c_w$
          by the definition of the order of coincidence (see Definition \ref{def:charcont}).                 
\end{proof}

\begin{proposition} \label{prop:intersfund}
   Let $A$ and $B$ be two branches on $S$. Consider the point $W \in \Theta_L (L') $ fixed above. 
   Assume that $W = A \wedge_L L' = B \wedge_L L' $.
       Then:
          \begin{enumerate} 
           \item \label{point2} 
                      $ L \cdot A = c_w (E_w \cdot A_w)$. 
           \item \label{point1} 
                     $ A \cdot B  = A_w \cdot B_w +  \dfrac{d_w}{c_w} (L \cdot A) (L \cdot B)$. 
            \item \label{point3} 
                    $A_w \cdot B_w >  0 $ if and only if $W \prec_L A \wedge_L B $. 
            \item \label{point4} 
                     $  \ic_L (A \wedge_L B) =  \dfrac{1}{c_w^2}  \ic_{E_w} (A_w \wedge_{E_w} B_w)  + 
                         \dfrac{d_w}{c_w}.$              
          \end{enumerate}
\end{proposition}

\begin{proof} Notice first that the hypothesis and Lemma \ref{lem:EWN2} imply that 
     the strict transforms $A_w$, $B_w$ of $A$ and $B$ by $\pi_w$ 
      intersect  $E_w \: \setminus \: (L_w \cup L'_w)$.  
  If $C$ is a branch on $(S,o)$, denote by $\boxed{(\pi_{w}^* C)_{ex}}$ 
     the \emph{exceptional part} of the 
      total transform divisor $(\pi_{w}^* C) = (\pi_{w}^* C)_{ex} + C_w$ of $C$ on $S_w$.
    \medskip
    
    \noindent
    {\bf Proof of  {\it (\ref{point2})}}. We have:
       \[ \begin{array}{ll} 
             L \cdot A &  \stackrel{(i)}{=\joinrel=}  (\pi_{w}^* L) \cdot (\pi_{w}^* A)  =  \\
              & \stackrel{(ii)}{=\joinrel=}  (\pi_{w}^* L) \cdot A_w =  \\
                       &   \stackrel{(iii)}{=\joinrel=}  (\pi_{w}^* L)_{ex} \cdot A_w = \\
                      & \stackrel{(iv)}{=\joinrel=}  \mbox{ord}_{E_w}(L)  (E_w  \cdot A_w) = \\
                      &  \stackrel{(v)}{=\joinrel=} \nu_w(\chi^{\epsilon_1})  (E_w  \cdot A_w) = \\
                      &  \stackrel{(vi)}{=\joinrel=}  ( (c_w e_1 + d_w e_2)\cdot\epsilon_1) (E_w  \cdot A_w) = \\
                      &  \stackrel{(vii)}{=\joinrel=} c_w (E_w  \cdot A_w) . 
          \end{array}  \]

          Let us explain each one of the previous equalities: 
           \begin{itemize}
                \item Equality $(i)$ results from the birational invariance of the
                 intersection product, if one works with total transforms of divisors. 
                \item Equality $(ii)$ is a consequence of the equality 
                     $(\pi_{w}^* L) \cdot (\pi_{w}^* A)_{ex}=0$, 
                    which results from the \emph{projection formula} (see \cite[Appendix A1]{H 77}), 
                    applied to the divisors $L$ on $S$, 
                    $(\pi_{w}^* A)_{ex}$ on $S_w$ and to the proper morphism $\pi_w$. 
                \item Equality $(iii)$ follows from the hypothesis $L_w \cdot A_w =0$ and 
                the bilinearity of the intersection product.
                \item  Equality $(iv)$ is a consequence of the equality $(\pi_{w}^* L)_{ex} = 
                        \mbox{ord}_{E_w}(L)  E_w$. 
                \item Equality $(v)$ results from the equalities  $ \mbox{ord}_{E_w} = \nu_w$  
                   (see Equation \ref{eq:defvalweight}) and $x= \chi^{\epsilon_1}$. 
                \item Equality $(vi)$ results from the fact that $w = c_w e_1 + d_w e_2$. 
                \item Equality $(vii)$ results from the fact that $(\epsilon_1, \epsilon_2)$ is the dual 
                   basis of $(e_1, e_2)$. 
            \end{itemize}

   \begin{figure}[h!]
\begin{center}
\begin{tikzpicture}[scale=0.6]
\draw[-][thick, color=orange](0,0) .. controls (2.5,0.3) ..(7,0);
\draw[->][thick, color=darkturquoise](0.5,-1.5) .. controls (0.4,0.1) ..(0.1,1.5);
\node [above, color=darkturquoise] at (0.5,-2.2) {$L'_w$};
\draw[->][thick, color=darkturquoise](6.5,-1.5) .. controls (6.6,0.1) ..(6.9,1.5);
\node [above, color=darkturquoise] at (6.5,-2.2) {$L_w$};
\draw[-][thick, color=darkturquoise](1.5,0.8) .. controls (1.8,0.5) ..(1.8,0.2);

\draw[->][thick, color=indiangreen] (2.5,1).. controls (2.2,0.2) and (0.8,-0.2)..(0.8,-0.6);
\node [below, color=indiangreen] at (1,-0.6) {$B_w$};

\draw[->][thick, color=darkturquoise](1.8,0.2) .. controls (2.2,1.2) ..(2.5,1.4);
\node [right,above, color=darkturquoise] at (2.5,1.4) {$A_w$};

\begin{scope}[shift={(3.6,-0.1)},scale=1]
\draw[-][thick, color=dodgerblue](1.5,0.8) .. controls (1.8,0.5) ..(1.8,0.2);
\draw[->][thick, color=dodgerblue](1.8,0.2) .. controls (2.2,1.2) ..(2.5,1.4);
\node [right,above, color=dodgerblue] at (2.5,1.5) {$A'_w$};

\end{scope}
\node [right, color=orange, thick] at (7,0) {$E_{w}$};

\begin{scope}[shift={(14,0)},scale=1]
\draw[-][thick, color=darkturquoise](0,-2) -- (0,2.7);
\node [below, thick, color=darkturquoise] at (0,-2) {$L$};
\node [above, thick, color=darkturquoise] at (0,2.8) {$L'$};
\node [above, thick, color=dodgerblue] at (-1.5,2.8) {$A$};
\node [above, thick, color=darkturquoise] at (1.5,2.8) {$A'$};
\node [above, thick, color=indiangreen] at (-2.4,2.8) {$B$};
\node [left, thick, color=orange] at (0,0) {$W$};

\draw[-][thick, color=darkturquoise](0,0) -- (1.5,2.7);
\draw[-][thick, color=darkturquoise](0,0) -- (-1.5,2.7);
\draw[-][thick, color=darkturquoise] (-0.95,1.7) -- (-2.4,2.7);

\node[draw,circle, inner sep=1.3pt,color=darkturquoise, fill=darkturquoise] at (0,-2){};
\node[draw,circle, inner sep=1.3pt,color=darkturquoise, fill=darkturquoise] at (0,2.7){};     
\node[draw,circle, inner sep=1.7pt,color=orange, fill=orange] at (0,0){};
\node[draw,circle, inner sep=1.7pt,color=darkturquoise, fill=darkturquoise] at (-1.5,2.7){};
\node[draw,circle, inner sep=1.7pt,color=dodgerblue, fill=dodgerblue] at (1.5,2.7){};
\node[draw,circle, inner sep=1.7pt,color=indiangreen, fill=indiangreen] at (-2.4,2.7){};

\end{scope}
\end{tikzpicture}
\end{center}
\caption{The choice of branch $A'$ in the proof of Proposition \ref{prop:intersfund}  (\ref{point1})}
  \label{fig:toricmodifchoice}
   \end{figure}
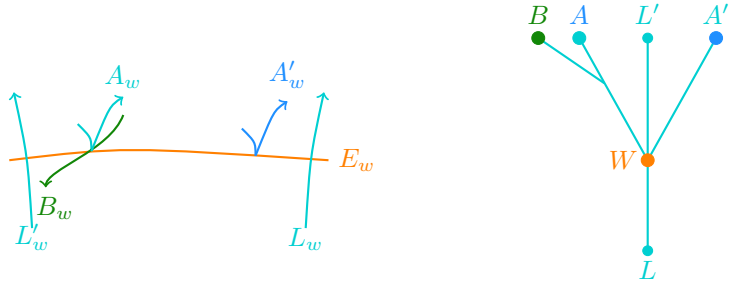
            
   \medskip
   
   \noindent
    {\bf Proof of {\it (\ref{point1})}}. Let us choose a branch $A'$ on $(S,o)$ such that:
    \begin{equation} \label{eq:aprima}
           \de_L (A) = \de_L (A') \mbox{ and }    W = A \wedge_L L'  = A'  \wedge_L  L'. 
    \end{equation}
    Using Lemma \ref{lem:intersfund}, we can translate this hypothesis in terms of the total 
    transform of the branches $A, A'$ by $\pi_w$.
    On the left side of Figure \ref{fig:toricmodifchoice} is represented the total transform 
      of $L+L'+A+A'+B$ by $\pi_w$ and on its right side is represented the Eggers-Wall tree 
      $\Theta_L(L+L'+A+A'+B)$, for some branch $B$. 
      Then:
      \[ \begin{array}{ll} 
             A \cdot B  &  \stackrel{(i)}{=\joinrel=}  (\pi_{w}^* A) \cdot (\pi_{w}^* B)  =  \\
                       & \stackrel{(ii)}{=\joinrel=}  (\pi_{w}^* A) \cdot B_w = \\
                      & \stackrel{(iii)}{=\joinrel=} A_w \cdot B_w +  (\pi_{w}^* A)_{ex} \cdot B_w = \\
                      &  \stackrel{(iv)}{=\joinrel=} A_w \cdot B_w +  (\pi_{w}^* A')_{ex} \cdot B_w  = \\
                      &  \stackrel{(v)}{=\joinrel=} A_w \cdot B_w +  (\pi_{w}^* A') \cdot B_w  = \\
                      &  \stackrel{(vi)}{=\joinrel=} A_w \cdot B_w +  A' \cdot B = \\
                      &  \stackrel{(vii)}{=\joinrel=} A_w \cdot B_w + (L \cdot A') (L \cdot B) \, \ic_L(W) = \\
                      &  \stackrel{(viii)}{=\joinrel=} A_w \cdot B_w + (L \cdot A) (L \cdot B) \dfrac{d_w}{c_w}. 
          \end{array}  \]

          Let us explain each one of the previous equalities: 
            \begin{itemize}
                \item Equalities $(i)$ and $(ii)$  are  analogs of the equalities $(i)$ and $(ii)$ 
                    in the proof of point (\ref{point2}) above.
                       
                \item  Equality $(iii)$ results from the bilinearity of the intersection product. 
                \item Equality $(iv)$ results from the hypothesis (\ref{eq:aprima}) and 
                      Lemma \ref{lem:EWN2}, which imply that 
                       $\mbox{ord}_{E_w} (A)=  \mbox{ord}_{E_w}(A')$.  
                       Then one concludes using 
                       the equality $(\pi_{w}^* C)_{ex} \cdot B_w = \mbox{ord}_{E_w} (C) (E_w \cdot B_w)$,    
                      for each $C \in \{A, A' \}$.
                \item  Equality $(v)$ results from the fact that, by construction, $A'_w$ and $B_w$ are 
                   disjoint. 
                \item Equality $(vi)$ results from the projection formula. 
                \item  Equality $(vii)$ results from Lemma \ref{lem:intersfund}.
                \item Equality $(viii)$ results from Equation (\ref{eq:contactu}) and from 
                   the equality $L \cdot A = L \cdot A'$, which is a consequence of 
                   the hypothesis 
                   \eqref{eq:aprima} and the equality $L \cdot C = \de_L (C)$ for each $C \in \{A, A' \}$. 
            \end{itemize}

     \medskip
     \noindent
    {\bf Proof of {\it (\ref{point3})}.} 
    By hypothesis,  the strict transforms  $A_w$ and $B_w$ intersect the set 
     $E_w \setminus (L_w \cup L'_w)$, which is equal to the torus orbit $O_{\cone w}$. 
     By the proof of Proposition \ref{prop:propstrict}, this implies that 
    $w$ is orthogonal to the compact edges of the Newton polygons
     $\cN_{L, L'} (A)$ and $\cN_{L, L'} (B)$.  
     Lemma \ref{lem:EWN} implies that  
      $\ex_L(W)= \ex_L( A \wedge_L L') =  \ex_L( B \wedge_L L')$.
       As the three points $W, A \wedge_L L', B \wedge_L L'$ belong to the segment 
       $[L, L']$ and that $\ex_L$ is strictly increasing on it, we get the equalities 
         $  W=  A \wedge_L L' = B \wedge_L L'$.    
    This implies that $W \preceq_L A$, $W \preceq_L B$.
     The claim follows from point (\ref{point1}) by using Lemma \ref{lem:intersfund}.

     \medskip
     \noindent
     {\bf Proof of {\it (\ref{point4})}.}  Dividing both sides of the formula of point (\ref{point1}) by the product
$(L\cdot A) (L\cdot B)$,  we get: 
\[
\frac{A \cdot B}{(L \cdot A) \cdot (L \cdot B)} =  \frac{A_w \cdot B_w }{(L \cdot A) \cdot (L \cdot B)} + \frac{d_w}{c_w}.
\]
Using  point (\ref{point2}), we get:
\[
\frac{A_w \cdot B_w }{(L \cdot A) \cdot (L \cdot B)} = \frac{1}{c_w^2} 
\frac{A_w \cdot B_w}{(E_w \cdot A_w) \cdot (L_w \cdot B_w)}.
\]
By applying formula (\ref{f-intcomp}) twice we obtain the desired formula: 
\begin{equation} \label{eq:doseq}
     \ic_L (A \wedge_L B) =  \frac{1}{c_w^2}  \ic_{E_w} (A_w \wedge_{E_w} B_w)  + \frac{d_w}{c_w}.
\end{equation}

\end{proof}

Let us define in combinatorial terms a natural embedding of the universal Eggers-Wall 
tree $\Theta_{o_w, E_w}$ into the universal Eggers-Wall tree $\Theta_{o, L}$ 
(see \ref{def:univEW}): 

\begin{definition} \label{def-EW-emb}
    Let $A_w$  be a branch on the germ of surface $(S_w, o_w)$. 
    Denote by $A$ its image by the modification $\pi_w$. 
    The {\bf natural embedding} \index{natural embedding!of Eggers-Wall trees} 
    of the universal Eggers-Wall 
     tree $\Theta_{o_w, E_w}$ into the universal Eggers-Wall tree $\Theta_{o, L}$ 
     is defined by sending each 
    point $Q$ of the Eggers-Wall segment $\Theta_{o_w, E_w} (A_w)$ 
    to the unique point $Q' $ of $\Theta_{o, L} (A)$ which satisfies: 
   \begin{equation} \label{eq:doseq-2}
         \ic_L (Q') =  \frac{1}{c_w^2}  \ic_{E_w} (Q)  + \frac{d_w}{c_w}.
   \end{equation}
\end{definition}

If $(C_{w}, o_w)$ is a reduced curve on $(S_w, o_w)$, then  the embedding of the Eggers-Wall tree 
$\Theta_{o_w, E_w} (C_w)$ in $\Theta_{o, L} (C)$
is well-defined thanks to Formula (\ref{eq:doseq}) 
applied to any pair $A_w$, $B_w$ of branches of $(C_{w}, o_w)$. 
That is,   the embeddings of the Eggers-Wall segments of its branches 
    glue into an embedding of $\Theta_{o_w, E_w} (C_w)$ in $\Theta_{o, L} (C)$.    
   Notice that  the  root $E_w$ of  $\Theta_{o_w, E_w} (C_w)$ 
       corresponds to the point $W \in \Theta_{o,L} (L')$ defined by relation (\ref{eq:expu})
       and that the leaf of $\Theta_{o_w, E_w}$  labeled by $A_w$ 
      corresponds to the leaf  of $\Theta_{o, L}$ labeled by $A$.

\medskip

The following proposition describes how to pass from the functions 
$(\de_{E_w},  \ex_{E_w})$ 
on the tree $\Theta_{o_w, E_w} (C_w)$ 
to the functions $(\de_L, \ex_L)$ on $\Theta_{o, L} (C)$:

\begin{proposition} \label{prop:renorm}
   Let $(C_{w}, o_w)$ be a reduced curve singularity on $(S_w, o_w)$. 
   Identify the tree $\Theta_{o_w, E_w} (C_w) $ with the subtree of  $\Theta_{o, L} (C)$ 
   defined by the natural embedding of Definition \ref{def-EW-emb}.  One has the 
   following relations in restriction to this subtree: 
          \begin{enumerate} 
               \item   \label{uno}  $\de_L = c_w \: \de_{E_w}  $. \\
               \item   \label{tres}  $\ex_L  = \dfrac{1}{c_w} \ex_{E_w}   + \dfrac{d_w}{c_w}. $
          \end{enumerate}        
\end{proposition}

 \begin{proof} $\,$
 
 \noindent
 {\bf Proof of {\it (\ref{uno})}.} 
We show first the assertion for an end of $\Theta_{o_w, E_w} (C_w) $ 
corresponding to a branch  $B_w$ of $C_w$. 
By the definition of the index function, we have the equalities 
$ \de_L (B) = L\cdot B$ and $\de_{E_w} (B_w) = E_w \cdot B_w$.  
Combining these equalities with  point (\ref{point2}) of Proposition \ref{prop:intersfund}, we get:
\begin{equation} \label{eq:index}
       \de_L (B) = c_w \de_{E_w} (B_w). 
\end{equation}
    Let $Q \ne E_w$ be any rational point of  $\Theta_{o_w, E_w} (C_w)$. 
By the equality (\ref{index-refor}), there exists a branch 
$A_w$ on the germ of surface $(S_w, o_w)$ such that $\de_{E_w} (A_w) = \de_{E_w} (Q)$. 
We get:
\[
    \de_L (Q) \stackrel{(\ref{index-refor})}{\leq} \de_L (A) \stackrel{(\ref{eq:index})}{=}  
         c_w \de_{E_w} (A_w) = c_w \de_{E_w} (Q).
\]
This implies that $\de_L (Q) \leq c_w \de_{E_w} (Q)$. 
Analogously, using again equality (\ref{index-refor}), there exists a branch 
$B$ on the germ $(S,o)$ such that $W \prec_L Q \prec_L B$ and  $\de_{L} (B) = \de_{L} (Q)$. 
By Definition \ref{def-EW-emb} of the  natural
embedding of $\Theta_{o_w, E_w}$ in $\Theta_{o,L}$,  
this implies that $Q \prec_{E_w} B$. Therefore: 
\[
\de_L (Q) =  \de_L (B) \stackrel{(\ref{eq:index})}{=}  c_w \de_{E_w} (B_w) \stackrel{(\ref{index-refor})}{\geq}  c_w \de_{E_w} (Q).
\]
It follows  that   $\de_L (Q) = c_w \de_{E_w} (Q)$.
We have shown that the equality in point (\ref{uno}) 
holds in restriction to the rational points of  $\Theta_{o_w, E_w} (C_w)$, 
and by the continuity properties of the index functions, it holds for every point of $\Theta_{o_w, E_w} (C_w)$.

\medskip
\noindent
{\bf Proof of {\it (\ref{tres})}.}  
    Let $P$ be a point of $\Theta_{o_w, E_w} (C_w)$. This implies that $W \preceq_L P$. 
     By the integral formula (\ref{exp-int}), we get: 
\[
\ex_L (P) = \int_L^P \de_L d \ic_L = \int_L^W \de_L d \ic_L + \int_W^P \de_L d \ic_L.
\]
Using again equation (\ref{exp-int}), we have:  
\begin{equation} \label{tres1}
 \int_L^W \de_L d \ic_L = \ex_L (W) = \frac{d_w}{c_w}.
\end{equation} 

We compute the second integral $\int_W^P \de_L d \ic_L$ by making a change of variable.  
Differentiating formula (\ref{eq:doseq-2}), we get
$   d \ic_L = (1/c_w^2) d \ic_{E_w}$.  
Using the expression for $\de_L$ of point (\ref{uno}), we obtain:  
\begin{equation} \label{tres2}
\int_W^P \de_L d \ic_L = \frac{1}{c_w} \int_W^P \de_{E_w} \, d \ic_{E_w} =
\frac{1}{c_w} \ex_{E_w} (P),
\end{equation}
where we have used again the integral formula (\ref{exp-int}).
We end the proof by combining the equalities (\ref{tres1}) and (\ref{tres2}): 
\[
\ex_L (P) = \frac{d_w}{c_w} + \frac{1}{c_w} \ex_{E_w} (P). 
\]
\end{proof}

 \begin{remark}
    Identify the tree $\Theta_{o_w, E_w} $ with the subtree of the 
    universal Eggers-Wall tree  $\Theta_{o, L}$ defined by 
    the embedding of Definition \ref{def-EW-emb}.     
    As a consequence of Proposition \ref{prop:renorm}, the two formulae stated in it 
    also hold on $\Theta_{o_w, E_w} $.
 \end{remark}

\subsection{Renormalization in terms of Newton-Puiseux series}
\label{ssec:transfanEW-Newton}
$\:$
\medskip

We give a different proof of Proposition \ref{prop:renorm} by using Newton-Puiseux series. 
This proof relates the Newton modifications  in the toroidal category of Definition 
\ref{def:Npolseries} with the \textit{Newton maps}, which appear 
sometimes in the algorithmic construction of Newton-Puiseux series 
(see Subsection \ref{ssec:HAEWtrees}). 

\medskip

We keep the notations introduced at the beginning of Section \ref{ssec:renormres}. 
Let $A$ be a branch on $(S,o)$ such that $A_w$ intersects $E_w$ at a point 
$o_w \in E_w \setminus (L_w \cup L_w')$. Consider local coordinates $(x, y)$ 
defining the cross $(L, L')$. Recall from Definition \ref{def:charcont} that $\cZ_x(A)$ denotes the 
set of Newton-Puiseux roots of $A$ relative to $x$.
Let us choose $\eta \in \cZ_x (A)$. It may be expressed as: 
\begin{equation}
       \eta  = \sum_{k \geq m}  \alpha_k x^{k/n},
\end{equation}
where $n = A\cdot L$, $m = A\cdot L'$. Hence $\alpha_{m} \ne 0$. 
All the series in $\cZ_x(A)$ have the same support, since they form a single orbit under the 
Galois action of multiplication of $x^{1/n}$ by the complex $n$-th roots of $1$ 
(see Remark \ref{rem:Gal}).

Let us denote $p := \mathrm{gcd}(n, m)$. Our hypothesis that $A_w$ meets 
    $E_w \setminus (L_w \cup L'_w)$ implies that:
\begin{equation} \label{clave}
      n = c_w \cdot p, \quad m = d_w \cdot p.
\end{equation}
The branch $A$ is defined by $f=0$, where: 
\begin{equation} \label{f-factor}
      f = \prod_{\gamma^{n} =1} (y - (\gamma \cdot \eta) (x)) = ( y^{c_w} -  \alpha_{m}^{c_w}     
       x^{d_w})^{p} + \dots . 
\end{equation}
We have only written on the right-hand side of \eqref{f-factor} 
the restriction of $f$ to the unique compact edge of the Newton polygon of $f(x,y)$.

\begin{lemma} \label{lem:Newton-map}
    There exist local coordinates $(x_1, y_1)$ on the germ $(S_w,o_w)$ such that $E_w = Z(x_1)$ and 
     the map $\pi_w$ is defined by: 
   \begin{equation} \label{Newton-map}
         \left\{ \begin{array}{lcl}
              x&=& x_1^{c_w},
              \\
             y&= & x_1^{d_w} (\alpha_{m} + y_1).
            \end{array}  \right.
     \end{equation}
\end{lemma}

\begin{proof}
Consider a vector $w' = a_w e_1 + b_w e_2$ such that: 
\begin{equation} \label{det}
     b_w c_w  - a_w d_w = 1.
\end{equation}
Therefore the cone $\sigma = \cone \langle w, w' \rangle$ is regular and included in 
one cone of the fan $\fan_{L,L'}(A)$. 
As explained in the proof of Proposition \ref{prop:propstrict}, we can 
look at the intersection of $A_w$ 
with the orbit $O_{\R_+ w} = E_w \setminus (L_w \cup L_w')$ in the open subset corresponding 
to this orbit on the toric surface $X_{\sigma} = \C^2_{u,v}$. 
The  toric morphism $\psi^\sigma_{\sigma_0}$  is the monomial map defined by 
\[  \left\{
\begin{array}{lcl}
      x&=& u^{c_w} v^{a_w}
        \\
       y&= & u^{d_w} v^{b_w}
\end{array}  \right.
\]
(see Example \ref{ex:tworegcones}).
The orbit $O_{\R_+ w}$, seen on the surface $\C^2_{u,v}$,  is the pointed axis $\C_v^*$. 
The maximal monomial in $(u,v)$ which divides $(\psi^\sigma_{\sigma_0})^*f $ is equal to 
$(u^{c_w d_w} v^{a_w d_w})^{p}$. 
After factoring out this monomial and setting $u=0$ we get: 
\begin{equation} \label{eq:change2}
    (v^{a_w c_w - b_w d_w} - \alpha_{m}^{c_w}  ) ^{p} 
    \stackrel{(\ref{det})}{=} (v - \alpha_{m}^{c_w}  ) ^{p}. 
\end{equation}
This shows that the point $o_w$ has coordinates $(u, v) = (0, \alpha_{m}^{c_w} )$. 
The formulae 
\begin{equation} \label{eq:change}
 \left\{  \begin{array}{lcl}
     u&=& x_1 ( y_1 + \alpha_{m})^{-a_w},
      \\
     v&= & ( y_1 + \alpha_{m})^{c_w},
\end{array}  \right.
\end{equation}
define local coordinates $(x_1, y_1)$ at $o_w$, since the jacobian determinant of 
$(u, v - \alpha_{m}^{c_w})$ with respect to $(x_1, y_1)$ does not vanish at $(0,0)$. 
Notice also that $Z(x_1) = Z(u) = E_w$.
By (\ref{det}) we get: 
\[  \left\{
\begin{array}{lcl}
x&=& x_1^{c_w} ( y_1 + \alpha_{m})^{-a_w c_w } ( y_1 + \alpha_{m})^{a_w c_w } 
     = x_1^{c_w}
\\
y &= & x_1^{d_w} ( y_1 + \alpha_{m})^{b_wc_w - d_w a_w} =  x_1^{d_w}  (\alpha_{m} + y_1).
\end{array}  \right.
\]
\end{proof}

\begin{proposition} \label{Newton-eta}
     With respect to the coordinates $(x_1, y_1)$ introduced in Lemma \ref{lem:Newton-map}, 
     the series 
       \[
         \eta_w  := \sum_{k > m} \alpha_m  x_1^{(k - m)/p},
         \]
     is a Newton-Puiseux series parametrizing the branch $A_w$ on $(S_w, o_w)$.  
\end{proposition}

\begin{proof}
  By formula (\ref{eq:change2}),  we have that $(A_w \cdot E_w)_{o_w} = p$. It follows that 
the Newton-Puiseux series in $\cZ_{x_1} (A_w)$ must have exponents in $(1/p)\N^*$. 
By composing \eqref{Newton-map} with 
the change of variable 
\begin{equation} \label{eq:x1}
     x_1  =   x_2^{p}, 
\end{equation}
we get: 
 \begin{equation} \label{eq:x2}
 \left\{ \begin{array}{lcl}
              x & = & x_2^{n}, 
              \\
              y&= & x_2^{d_w p } (\alpha_{m} + y_1).
             \end{array}  \right. 
\end{equation}
Apply the substitution \eqref{eq:x2}
to the factor 
$y - (\gamma \cdot \eta) (x^{1/n})$, 
using that $x_2 = x^{1/n}$ by definition,
and factor out the monomial $x_2^{d_w p}$. We get the series 
\begin{equation} \label{factor2}
      (\alpha_{m} + y_1) -  \alpha_{m} \gamma^{m}  - \sum_{k> m}  \alpha_k \gamma^{k}  
x_2^{k - m} \in \C[[x_2, y_1]].
\end{equation}
This series has vanishing constant term
if and only if $\gamma^{m} = 1$. Since $\gamma^{n} =1$ and 
$\mathrm{gcd} (n, m) = p$, one may check  that 
this condition holds if and only if $\gamma^{p} = 1$, and in this case 
for any $k > m$ one has $\gamma^k = \gamma^{k-m}$. 
It follows that the series (\ref{factor2}) which are non-units are
precisely the conjugates of the series $y_1 - \eta_w (x_1^{1/p}) $ under the Galois action, since 
 $x_2 = x_1^{1/p}$ by definition \eqref{eq:x1}.  
Therefore, the product of all the conjugates of $y_1 - \eta_w (x_1^{1/p}) $ 
under the Galois action defines a polynomial in $\C[[x_1]] [y_1]$ which 
divides the strict transform of $f$ by the map (\ref{Newton-map}). 
The remaining factor is a series with nonzero constant term, 
and must belong to the ring  $\C[[x_1, y_1]]$  since it is invariant under the Galois action.
\end{proof}

\begin{corollary}   \label{cor:renormformbis}
Let $A, B$ be two branches on $(S,o)$ such that $ o_w \in A_w \cap B_w \cap E_w$. 
Then: 
\[
k_x (A, B) = \frac{d_w}{c_w} + c_w \cdot k_{x_1} (A_w, B_w).
\]
\end{corollary}
\begin{proof}
By point (\ref{point3}) of  Proposition \ref{prop:intersfund},  
the  inequality $A_w \cdot B_w >0$ 
(which results from the hypothesis that $o_w \in A_w \cap B_w$)
implies that  $k_x (A, B) > d_w /c_w$. 
It follows that if we fix a Newton-Puiseux series 
$\eta \in \cZ_x (A)$, then there exists $\xi \in \cZ_x(B)$ with 
the same order and the same leading coefficient. We can 
apply Lemma \ref{lem:Newton-map}, using this leading coefficient, 
to define suitable local coordinates $(x_1, y_1)$ at the point $o_1$. 
The formula results from Proposition \ref{Newton-eta} by taking 
into account the facts that $\eta_w \in \cZ_{x_1} (A_w)$ and $\xi_w \in \cZ_{x_1} (B_w)$.
\end{proof}

Corollary \ref{cor:renormformbis} implies readily Proposition \ref{prop:renorm}.

\subsection{From fan trees to Eggers-Wall trees}
\label{ssec:transfanEW}
$\:$
\medskip

In this subsection we assume that $C$ is \emph{reduced}. 
We explain that there exists a canonical isomorphism from  
the fan tree $\theta_{\pi}(C)$ of a toroidal pseudo-resolution $\pi$ of $C$ produced by 
running Algorithm \ref{alg:tores},  to the Eggers-Wall 
tree of the completion $\hat{C}_{\pi}$ of $C$ (see Theorem \ref{thm:isomfantreeEW}). 
We also explain how to compute the index, exponent and contact 
complexity functions on the Eggers-Wall tree from the slope function on the fan tree 
(see Proposition \ref{prop:slopedetindex}). 
\medskip

Let $L$ be  a smooth branch on the germ $(S,o)$. 
Assume that we run Algorithm \ref{alg:tores}, arriving at a toroidal pseudo-resolution 
$\pi: (\Sigma, \partial \Sigma) \to (S, L + L')$. Consider the corresponding completion 
$\hat{C}_{\pi}$, in the sense of Definition \ref{def:threeres}. There are two trees associated 
with this setting which have their ends labeled by the branches of $\hat{C}_{\pi}$, 
the fan tree $\theta_{\pi}(C)$ and the Eggers-Wall tree $\Theta_L(\hat{C}_{\pi})$. 
How are they related? It turns out that they are isomorphic: 

\begin{theorem}  \label{thm:isomfantreeEW}
   There is a unique isomorphism from the fan tree $\theta_{\pi}(C)$ to the 
   Eggers-Wall tree $\Theta_L(\hat{C}_{\pi})$, which preserves the labels of the 
   ends of both trees by the branches of $\hat{C}_{\pi}$. 
\end{theorem}

\begin{proof} 
At the first step of Algorithm \ref{alg:tores}, one chooses a smooth branch $L'$ 
such that $(L, L')$ is a cross on $(S,o)$. By definition, the branch $L'$ is a component of 
the completion $\hat{C}_{\pi}$.   Let us consider the segment 
$[L, L']$ of $\Theta_L (\hat{C}_{\pi})$
and the first trunk 
$\theta_{\fan_{L, L'} ( C )} = [e_L, e_{L'}]$. 
We have a homeomorphism 
\[
    \Psi_o: [e_L, e_{L'}] \to [L, L'] = \Theta_L(L')
\]
sending a vector $w \in [e_L, e_{L'}]$  to the unique point $W \in [L, L']$ whose  
exponent $\ex_L ( W )$ is equal to the slope of $w$ with respect to 
the basis $(e_L, e_{L'})$ of $N_{L, L'}$. 
By Corollary \ref{cor:Newton}, the map $\Psi_o$ defines also a bijection 
between the set of marked points of 
the trunk, according to Definition \ref{def:fantreetr},  and the set of 
the marked points of the tree 
$\Theta_L (\hat{C}_{\pi})$
which belong to the segment $[L,L']$ 
according to Definition \ref{def:EW}.

Let $o_i$ be a point of $\partial S_{\fan_{L, L'} (C)}$, lying on the strict transform of $C$. 
The point $o_i$ is considered at the fourth step of Algorithm \ref{alg:tores}.
Let $A_i$ denote the germ of $\partial S_{\fan_{L, L'} (C)}$ at $o_i$ and 
let $(A_i, B_i)$ be the cross at $o_i$ 
chosen when one passes again through the first and second steps of Algorithm \ref{alg:tores}.
By definition, $L_i := \pi_{L, L'} (B_i)$ is a branch of $\hat{C}_{\pi}$.   
We denote by $\boxed{\hat{C}_{\pi, o_i}}$ (resp. $\boxed{C_{o_i}}$) the germ of the strict transform of $\hat{C}_{\pi}$ (resp. $C$) 
at the point $o_i$.  We use the Notations \ref{def:manycrosses}.
Let us consider the segment $[A_i, L_i]$ of the Eggers-Wall tree 
$\Theta_{o_i, A_i} ( \hat{C}_{\pi, o_i} )$ 
and the trunk $\theta_{\fan_{A_i, B_i} ( C_{o_i}  )} = [e_{A_i}, e_{B_i}]$. 
Arguing as before, we obtain a homeomorphism 
$
   \Psi_{o_i} : [e_{A_i}, e_{B_i}] \to [A_i, L_i] 
$
which sends
$w \in [e_{A_i}, e_{B_i}]$ to the unique point $W \in [A_i, L_i] $
such that 
$\ex_{A_i} (W)$ is equal to the slope of $w$ with respect to 
the basis $(e_{A_i}, e_{B_i})$ of the lattice $N_{A_i, B_i}$. 
In addition, we get also that the homeomorphism $\Psi_{o_i}$ defines 
a bijection between the marked points of 
the trunk $\theta_{\fan_{A_i, B_i} ( C_{o_i}  )}$
and the marked points of $\Theta_{o_i, A_i} ( \hat{C}_{\pi, o_i} )$ 
on the segment $[A_i, L_i]$. 
By Proposition \ref{prop:renorm}, we have an embedding of the Eggers-Wall tree 
$\Theta_{o_i, A_i} ( \hat{C}_{\pi, o_i} )$ such that 
the root $A_i$ of this tree is sent to the marked point 
$L' \wedge_L L_i$ of $\Theta_{L} (\hat{C}_{\pi})$. 
By Definition \ref{def:fantreetr}, the point $e_{A_i}$ of the trunk 
$\theta ( \fan_{A_i, B_i}(C) )$ is identified with the marked point labeled by 
$A_i$ on $\theta ( \fan_{L, L'}(C) )$, during the construction of 
the fan tree $\theta_{\pi} (C)$. 

If $\cT$ is a tree and $P_1, \dots, P_s \in \cT$, we denote by 
$[P_1, \dots, P_s]$ the smallest subtree of $\cT$ containing $P_1, \dots, P_s$. 
We apply 
this notation for the subtree $[e_{L}, e_{L'}, e_{B_j} ]$ of $\theta_\pi (C)$ 
and the subtree $ [L, L', L_j]$  of $\Theta_L (\hat{C}_\pi )$. 
The previous discussion implies that 
the homeomorphisms $\Psi_o$ and 
$\Psi_{o_i}$ can be glued into a homeomorphism
\[ [e_{L}, e_{L'}, e_{B_j} ] \to [L, L', L_j], \]
which sends the ramification vertex 
$e_{A_i}$ of the tree $[e_{L}, e_{L'}, e_{B_j} ]$
to the ramification vertex $L' \wedge_L L_j$ of $[L, L', L_j]$. 
We repeat this construction each time we pass through 
a cross at the first and second steps during the iterations of Algorithm \ref{alg:tores}.
By induction, we get a finite number of homeomorphisms 
$\Psi_{o_j}$, which glue into a homeomorphism
$\Psi: \theta_{\pi}(C) \to \Theta_L(\hat{C}_{\pi})$
which respects the labelings of the ends of both trees 
by the branches of $\hat{C}_{\pi}$. 
\end{proof}

 Identify the two rooted trees $\theta_{\pi}(C)$ and $\Theta_L(\hat{C}_{\pi})$ by the 
    isomorphism of Theorem \ref{thm:isomfantreeEW}. For every point 
    $P \in \theta_{\pi}(C)$, define the set $\boxed{\delta_P} \subset [L, P)$ as the finite subset 
    of discontinuity points of the restriction of the slope function $\slp_{\pi}$ to the segment 
    $[L, P)$.  If $\lambda \in \Q^*$, denote by $\boxed{\mathrm{den}(\lambda)}$  
              the denominator $q$ of $\lambda$, when one writes it in the form $p/q$, with 
              $(p, q) \in \Z \times \N^*$, and $p$, $q$ coprime. 
The fan tree $\theta_{\pi}(C)$ comes endowed with only one function, the {\em slope 
function} $\slp_{\pi}$, while the Eggers-Wall tree is endowed with 
the \emph{index} $\de_L$, the \emph{exponent} $\ex_L$ 
and the \emph{contact complexity} $\ic_L$ functions. These functions are related by: 
 
\begin{proposition}  \label{prop:slopedetindex}
    For every $P \in \theta_{\pi}(C)$, one has:
       \begin{enumerate}
           \item \label{i}
                $ \displaystyle{ \de_L(P) = \prod_{Q \in \delta_P} \mathrm{den}(\slp_{\pi}(Q))}$.         
           \item   \label{ii}
                $ \displaystyle{\ex_L(P) = \int_{L}^P \frac{1}{\de_L} \:  d \slp_{\pi}}$. 
           \item \label{iii}
               $ \displaystyle{\ic_L(P) = \int_{L}^P \frac{1}{\de_L^2} \:  d \slp_{\pi}}$. 
       \end{enumerate}
\end{proposition}

\begin{proof}   
In order to follow the proof, one has to keep in mind the isomorphism of 
the fan tree with the Eggers-Wall tree built in Theorem \ref{thm:isomfantreeEW}.
If the set $\delta_P$ is empty, that is, if the slope function $\slp_\pi$ is continuous in 
restriction to $[L, P)$, then $P$ belongs to the first trunk  $[L, L']$.
By definition, for any $Q \in [L, L']$ we have: 
\begin{equation} \label{r1}
     \de_L(Q) =1, \quad  \ex_L(Q) = \slp_\pi (Q). 
\end{equation}
Hence the equalities (\ref{i}), (\ref{ii}) and  (\ref{iii}) hold trivially for $P$. 

We prove the assertions (\ref{i}) and (\ref{ii}) by induction on the number of elements of 
the set $\delta_P$ of discontinuity points.
Assume that $\delta_P = \{ W = W_1, W_2, \dots, W_k \}$ with $k \geq 1$, and 
$W \prec_L W_2  \prec_L  \cdots \prec_L W_k \prec_L P$.  
By construction, the point $W$ belongs to the first trunk of $\theta_{\pi}(C)$.
Then, using the notation (\ref{eq:expu}), we have
$ \ex_L (W) = d_w / c_w = \slp_\pi (W)$, with $c_w = \mathrm{den}(\slp_\pi (W))$. 
We decompose the integral of the second member of equality (\ref{ii}) in the form: 
\[
    \int_L^P \frac{1}{\de_L} d \slp_\pi = \int_L^W \frac{1}{\de_L} d \slp_\pi + 
     \int_W^P \frac{1}{\de_L} d \slp_\pi.
\]
By (\ref{r1}), one has: 
\begin{equation} \label{eq:p1}
       \int_L^W \frac{1}{\de_L} d \slp_\pi = \ex_L (W) =  \frac{d_w}{c_w}.
\end{equation}
With the notations of Section \ref{ssec:renormres}, we consider 
the reduced curve $C_w$ at $(S_w,o_w)$, 
consisting of those branches $A_w$ which are the strict transforms of
branches $A$ of $C$ such that $W \prec_L P \prec_L A$ (see point (\ref{point3}) of Proposition \ref{prop:intersfund}). 
Proposition \ref{prop:renorm} implies that: 
\begin{equation} \label{eq:p3}
     \de_L (Q)  = c_w \de_{E_w} (Q), \mbox{ for } Q \in [W, P] \subset \Theta_{E_w} (C_w).
\end{equation}
Hence: 
\begin{equation} \label{eq:p2}
      \int_W^P \frac{1}{\de_L} d \slp_\pi = \frac{1}{c_w} \int_W^P \frac{1}{\de_{E_w}} d \slp_\pi 
       =  \frac{1}{c_w} \ex_{E_w} (P). 
\end{equation}
To understand the last equality of (\ref{eq:p2}), apply the induction hypothesis to the integral 
$\int_W^P (1 / \de_{E_w} )d \slp_\pi $, 
with respect to the set $ \{ W_2, \dots, W_k\}$ 
of discontinuity points of the restriction of the slope function $\slp_{\pi}$ to  $[W, P)$. 
The equality (\ref{ii}) follows from (\ref{eq:p1}), (\ref{eq:p2}) and point (\ref{tres}) of Proposition \ref{prop:renorm}. 

\medskip 
The equality (\ref{i})  follows similarly by (\ref{eq:p3}) and 
the induction hypothesis applied to $\de_{E_w} (P)$.

\medskip 

Let us prove the equality (\ref{iii}). 
By point (\ref{ii}) one has $d \ex_L = (1/\de_L) d \slp_\pi$.  
Therefore:
\[
     \ic_L (P)  = \int_L^P \frac{1}{\de_L} d \ex_L = \int_L^P \frac{1}{\de_L^2 } d \slp_\pi.
\]
\end{proof}

\begin{example}    \label{ex:fromFTtoEW}
  Consider the toroidal pseudo-resolution process of Example \ref{ex:toroidres}. 
  Figure  \ref{fig:example-fan tree-EW tree} shows the fan tree $\theta_{\pi}(C)$ 
  and the corresponding Eggers-Wall tree 
  $\Theta_L(\hat{C}_{\pi})$, for which are indicated the values of the exponent 
  and the index functions. 
  We computed them using Proposition \ref{prop:slopedetindex}.  For instance, we have 
      $\de_L(E_6)=1\cdot 5=5$, $\ex_L(E_6) =\dfrac{3}{5}+\dfrac{1}{5}\cdot \dfrac{5}{3}=
                   \dfrac{14}{15}$, 
  $\de_L(E_8)=1\cdot 5 \cdot 3=15$ and $\ex_L(E_8) =\dfrac{14}{15}+ \dfrac{1}{15}\cdot \dfrac{1}{2}
             =\dfrac{29}{30}.$
\medskip 

       Proposition \ref{prop:slopedetindex} allows us to 
      define a concrete reduced curve singularity $C$ which 
     admits the toroidal resolution process described in Example \ref{ex:toroidres}, 
     whose lotus was represented in Figure  \ref{fig:lotustoroid} and whose Enriques tree 
     was represented in Figure \ref{fig:Enriqtoroid}. Namely, we fix local coordinates $(x,y)$ 
     and we choose Newton-Puiseux series $\eta_1(x), \dots, \eta_7(x)$ defining branches 
     $C_1, \dots , C_7$, then we take supplementary series $\lambda_1(x), \dots , \lambda_4(x)$ 
     defining branches $L_1, \dots L_4$, such that the Eggers-Wall tree 
     $\Theta_L(C_1 + \cdots + C_7 + L_1 + \dots + L_4)$ is that on the right side of 
     Figure \ref{fig:example-fan tree-EW tree}. 
     For instance, one may choose:
    \[ 
    \begin{array}{c}
 \eta_1(x) :=  x^{5/2},  
 \quad   
 \eta_2(x) := x^2, 
 \quad   
 \eta_3(x) :=  -x^2,  
 \quad
 \eta_4(x) :=  x^{3/5} + x^{3/4},   
\\
\eta_5(x) :=  x^{3/5} + x^{11/15}
\quad
\eta_6(x) :=  2 x^{3/5} + x^{6/5}, 
\quad  
\eta_7(x) :=  2 x^{3/5} + x^{14/15} + x^{29/30}, 
 \\
 \lambda_1(x) :=  0, 
 \quad
 \lambda_2(x) :=  x^{3/5}, 
 \quad
 \lambda_3(x) :=  2 x^{3/5},  
 \quad
 \lambda_4(x) :=   2 x^{3/5} + x^{14/15}.
       \end{array}   
       \]   
   \end{example}

\begin{figure}
\begin{center}
\begin{tikzpicture}[scale=0.42]

 \begin{scope}[shift={(0,-26)},scale=1.5]
  
 \draw [-, color=orange, very thick](0,0) -- (0, 8) ; 
    \draw [-, color=orange, very thick](0, 2) -- (6, 2) ; 
     \draw [-, color=orange, very thick](0, 2) -- (6, 8) ; 
     \draw [-, color=orange, very thick](2,4) -- (2, 8) ; 
     \draw [-, color=magenta, very thick](0,6) -- (-2, 6) ; 
     \draw [-, color=magenta, very thick](2,2) -- (2, 0) ; 
     \draw [-, color=magenta, very thick](4,6) -- (6, 6) ; 
     \draw [-, color=magenta, very thick](2,7) -- (4, 7) ; 
     \draw [-, color=magenta, very thick](4,2) -- (4, 0) ; 
     \draw [-, color=magenta, very thick](0,4) -- (-2, 4) ; 
     \draw [-, color=magenta, very thick](0,4) -- (-2, 2) ;
      \draw [-, color=blue, very thick](0,0) -- (0, 4) ; 
      \draw [-, color=blue, very thick](0,6) -- (0, 8) ; 
      \draw [-, color=blue, very thick](0,2) -- (2, 2) ; 
      \draw [-, color=blue, very thick](4,2) -- (6, 2) ; 
      \draw [-, color=blue, very thick](0,2) -- (4,6) ;
      \draw [-, color=blue, very thick](2,7) -- (2,8) ;  

   \node[draw,circle, inner sep=1.5pt,color=black, fill=black] at (0,0){};
   \node [right] at (0,0) {$L$};
   \node [left] at (0,0) {$0$};
   \node[draw,circle, inner sep=1.5pt,color=black, fill=black] at (0,8){};
   \node [right] at (0,8) {${L_1}$};
   \node [left] at (0,8) {$\infty$};
   \node[draw,circle, inner sep=1.5pt,color=red, fill=red] at (0,2){};
   \node [right] at (0,1.5) {${E_1}$};
   \node [left] at (0,1.5) {$\frac{3}{5}$};
   \node[draw,circle, inner sep=1.5pt,color=red, fill=red] at (0,4){};
   \node [right] at (0,4) {${E_2}$};
   \node [left] at (0,4.5) {$\frac{2}{1}$};
   \node[draw,circle, inner sep=1.5pt,color=red, fill=red] at (0,6){};
   \node [right] at (0,6) {${E_3}$};
   \node [left] at (0,6.5) {$\frac{5}{2}$};
   \node[draw,circle, inner sep=1.5pt,color=black!20!green, fill=black!20!green] at (-2,6){};
   \node [below] at (-2,6) {${C_1}$};
   \node [above] at (-2,6) {$\infty$};   
   \node[draw,circle, inner sep=1.5pt,color=black!20!green, fill=black!20!green] at (-2,4){};
   \node [below] at (-2,4) {${C_2}$};
   \node [above] at (-2,4) {$\infty$}; 
   \node[draw,circle, inner sep=1.5pt,color=black!20!green, fill=black!20!green] at (-2,2){};
   \node [below] at (-2,2) {${C_3}$};
   \node [above] at (-2,2) {$\infty$}; 
   \node[draw,circle, inner sep=1.5pt,color=black, fill=black] at (6,2){};
   \node [below] at (6,2) {${L_2}$};
   \node [above] at (6,2) {$\infty$};
   \node[draw,circle, inner sep=1.5pt,color=red, fill=red] at (2,2){};
   \node [below] at (2.5,2) {${E_4}$};
   \node [above] at (2,2) {$\frac{2}{3}$};
   \node[draw,circle, inner sep=1.5pt,color=red, fill=red] at (4,2){};
   \node [below] at (4.5,2) {${E_5}$};
   \node [above] at (4,2) {$\frac{3}{4}$};
   \node[draw,circle, inner sep=1.5pt,color=black!20!green, fill=black!20!green] at (2,0){};
   \node [right] at (2,0) {${C_5}$};
   \node [left] at (2,0) {$\infty$}; 
   \node[draw,circle, inner sep=1.5pt,color=black!20!green, fill=black!20!green] at (4,0){};
   \node [right] at (4,0) {${C_4}$};
   \node [left] at (4,0) {$\infty$}; 
   \node[draw,circle, inner sep=1.5pt,color=black, fill=black] at (6,8){};
   \node [below] at (6,7.8) {${L_3}$};
   \node [above] at (6,8.2) {$\infty$};
   \node[draw,circle, inner sep=1.5pt,color=red, fill=red] at (4,6){};
   \node [below] at (4,5.8) {${E_7}$};
   \node [left] at (4,6.2) {$\frac{3}{1}$};
   \node[draw,circle, inner sep=1.5pt,color=red, fill=red] at (2,4){};
   \node [right] at (2.2,4) {${E_6}$};
   \node [left] at (2,4.2) {$\frac{5}{3}$};
     \node[draw,circle, inner sep=1.5pt,color=black!20!green, fill=black!20!green] at (6, 6){};
   \node [below] at (6, 6) {${C_6}$};
   \node [above] at (6, 6) {$\infty$};   
   \node[draw,circle, inner sep=1.5pt,color=black, fill=black] at (2,8){};
   \node [right] at (2,8) {${L_4}$};
   \node [left] at (2,8) {$\infty$};
   \node[draw,circle, inner sep=1.5pt,color=red, fill=red] at (2,7){};
   \node [right] at (2,6.5) {${E_8}$};
   \node [left] at (2,6.5) {$\frac{1}{2}$};
   \node[draw,circle, inner sep=1.5pt,color=black!20!green, fill=black!20!green] at (4, 7){};
   \node [below] at (4, 7) {${C_7}$};
   \node [above] at (4, 7) {$\infty$};
 \node [below] at (3, -1) {$\theta_{\pi}(C)$}; 
\end{scope}

 \begin{scope}[shift={(12,-1)},scale=1]
  \begin{scope}[shift={(-8,-25)},scale=1]
   \draw [-, color=black, very thick](11,0) -- (11, 10) ; 
        \draw [-, color=black, very thick](11,0) -- (11, 6) ; 
         \draw [-, color=black, very thick](11,8) -- (11, 10) ; 
   \node[draw,circle, inner sep=1.5pt,color=black, fill=black] at (11,0){};
   \node [right] at (11,0) {$L$};
   \node [left] at (11,0) {$0$};
   \node[draw,circle, inner sep=1.5pt,color=black, fill=black] at (11,10){};
   \node [right] at (11,10) {${L_1}$};
   \node [left] at (11,10) {$\infty$};
   \node[draw,circle, inner sep=1.5pt,color=black, fill=black] at (11,3){};
   \node [right] at (11,2.3) {${\bf \frac{3}{5}}$};
   \node [left] at (11,1.25) {$1$};
      \node [left] at (14,1.25) {$15$};
      \node [left] at (17,1.25) {$20$};
   \node[draw,circle, inner sep=1.5pt,color=black, fill=black] at (11,6){};
   \node [right] at (11,6) {${\bf  \frac{2}{1}}$};
  \node [left] at (11,4.5) {$1$};
   \node[draw,circle, inner sep=1.5pt,color=black, fill=black] at (11,8){};
   \node [right] at (11,8) {${\bf  \frac{5}{2}}$};
  \node [left] at (11,7) {$1$};
   \node [left] at (11,9) {$1$};
  
   \node [above] at (10,8) {$2$};
      \node [above] at (10,6) {$1$};
      \node [above] at (9.5,5) {$1$};
 \end{scope}     
 
 \begin{scope}[shift={(3,-33)},scale=1]
  \draw [-, color=black, very thick](0, 11) -- (9, 11) ; 
      \draw [-, color=black, very thick](0,11) -- (3, 11) ; 
       \draw [-, color=black, very thick](6,11) -- (9, 11) ; 
   \node[draw,circle, inner sep=1.5pt,color=black, fill=black] at (9,11){};
   \node [below] at (9,11) {${L_2}$};
   \node [above] at (9,11) {$\infty$};
   \node[draw,circle, inner sep=1.5pt,color=black, fill=black] at (3,11){};
   \node [below] at (3.5,11) {${\bf \frac{11}{15}}$};
  \node [above] at (1.75,11) {$5$};
   \node[draw,circle, inner sep=1.5pt,color=black, fill=black] at (6,11){};
   \node [below] at (6.5,11) {${\bf \frac{3}{4}}$};
    \node [above] at (5,11) {$5$};
    \node [above] at (8,11) {$5$};
 \end{scope}

 \begin{scope}[shift={(3,-22)},scale=1]
  \draw [-, color=black, very thick](0,0) -- (9, 6) ; 
       \draw [-, color=black, very thick](0,0) -- (6, 4) ; 
   \node[draw,circle, inner sep=1.5pt,color=black, fill=black] at (9,6){};
   \node [below] at (9,5.8) {${L_3}$};
   \node [above] at (9,6.2) {$\infty$};
   \node[draw,circle, inner sep=1.5pt,color=black, fill=black] at (3,2){};
   \node [below] at (3,2) {${\bf \frac{14}{15}}$};
   \node [left] at (2,1.8) {$5$};
   \node[draw,circle, inner sep=1.5pt,color=black, fill=black] at (6,4){};
   \node [below] at (6,3.8) {${\bf \frac{6}{5} }$};
   \node [above] at (4.9,3.2) {$5$};
   \node [above] at (7.5,5) {$5$};
   \node [above] at (7.25,3.9) {$5$};
 \end{scope}
 

 \begin{scope}[shift={(-5,-20)},scale=1]
  \draw [-, color=black, very thick](11,0) -- (11, 8) ; 
      \draw [-, color=black, very thick](11,6) -- (11, 8) ; 
   \node[draw,circle, inner sep=1.5pt,color=black, fill=black] at (11,8){};
   \node [right] at (11,8) {${L_4}$};
   \node [left] at (11,8) {$\infty$};
   \node[draw,circle, inner sep=1.5pt,color=black, fill=black] at (11,6){};
   \node [right] at (11,5.3) {${\bf \frac{29}{30}}$};
   \node [left] at (11,4) {$15$};
    \node [above] at (12.5,6) {$30$};
      \node [left] at (11,7) {$15$};
 \end{scope}

\begin{scope}[shift={(3,-17)},scale=0.9]
  \draw [-, color=black, very thick](0,0) -- (-3, 0) ; 
   \node[draw,circle, inner sep=1.5pt,color=black, fill=black] at (-3,0){};
   \node [below] at (-3,0) {${C_1}$};
   \node [above] at (-3,0) {$\infty$};
 \end{scope}
   
   \begin{scope}[shift={(3,-19)},scale=0.9]
   \draw [-, color=black, very thick](0,0) -- (-3, 0) ; 
   \node[draw,circle, inner sep=1.5pt,color=black, fill=black] at (-3,0){};
   \node [below] at (-3,0) {${C_2}$};
   \node [above] at (-3,0) {$\infty$};
  \end{scope}  
   
   \begin{scope}[shift={(3,-19)},scale=0.9]
   \draw [-, color=black, very thick](0,0) -- (-3, -2) ; 
   \node[draw,circle, inner sep=1.5pt,color=black, fill=black] at (-3,-2){};
   \node [below] at (-3,-2) {${C_3}$};
   \node [above] at (-3,-2) {$\infty$};
  \end{scope}
   
   \begin{scope}[shift={(9,-22)},scale=0.9]
   \draw [-, color=black, very thick](0,0) -- (0, -3) ; 
   \node[draw,circle, inner sep=1.5pt,color=black, fill=black] at (0,-3){};
   \node [right] at (0,-3) {${C_4}$};
   \node [left] at (0,-3) {$\infty$};
   \end{scope}
   
   \begin{scope}[shift={(6,-22)},scale=0.9]
   \draw [-, color=black, very thick](0,0) -- (0, -3) ; 
   \node[draw,circle, inner sep=1.5pt,color=black, fill=black] at (0,-3){};
   \node [right] at (0,-3) {${C_5}$};
   \node [left] at (0,-3) {$\infty$};
   \end{scope}
   
   \begin{scope}[shift={(9,-18)},scale=0.9]
    \draw [-, color=black, very thick](0,0) -- (3, 0) ; 
     \node[draw,circle, inner sep=1.5pt,color=black, fill=black] at (3, 0){};
   \node [below] at (3, 0) {${C_6}$};
   \node [above] at (3, 0) {$\infty$};
    \end{scope}
   
   \begin{scope}[shift={(6,-14)},scale=0.9]
    \draw [-, color=black, very thick](0,0) -- (3, 0) ; 
   \node[draw,circle, inner sep=1.5pt,color=black, fill=black] at (3, 0){};
   \node [below] at (3, 0) {${C_7}$};
   \node [above] at (3, 0) {$\infty$};
 \end{scope}
 
 \node [above] at (8, -28) {$\Theta_L(\hat{C}_{\pi})$}; 
\end{scope}

\end{tikzpicture}
\end{center}
 \caption{The fan tree $\theta_{\pi}(C)$ and the corresponding Eggers-Wall tree  
       $\Theta_L(\hat{C}_{\pi})$  in Example \ref{ex:fromFTtoEW}}  
 \label{fig:example-fan tree-EW tree} 
     \end{figure}
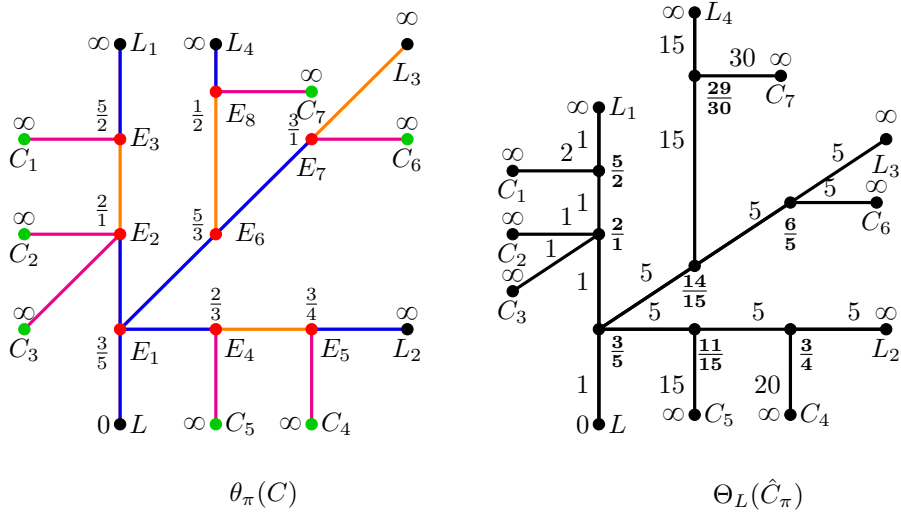

  \begin{remark}
      The right part of Figure \ref{fig:example-fan tree-EW tree} shows the Eggers-Wall 
      tree of the completion of a plane curve singularity generated by a toroidal 
      pseudo-resolution process. One may verify that it satisfies the following property 
      which characterizes the Eggers-Wall trees of such completions: 
      \emph{each vertex which is not an end of the tree is contained in the interior 
      of a segment in restriction to which the index function is constant} (in particular, 
      such an Eggers-Wall tree has no vertices of valency $2$). When one has such an 
      Eggers-Wall tree, it originates from a fan tree as described in Proposition 
      \ref{prop:slopedetindex}. But this fan tree is not unique. One has to determine 
      first which segments of the Eggers-Wall tree are trunks of the fan tree, and there 
      may be different choices. For instance, in Figure \ref{fig:example-fan tree-EW tree} 
      one could decide that the segment 
      $[L, C_2]$ is a trunk, instead of $[L, L_1]$. Once the trunks are chosen, the 
      sets $\delta_P$ are determined for every point $P$ of the tree. This allows 
      to compute the slope function  $\slp_{\pi}$ by integrating the differential relation 
      $d \slp_{\pi} = \de_L d \ex_L$, which is a consequence of Proposition 
      \ref{prop:slopedetindex} (\ref{ii}). 
  \end{remark}

Proposition \ref{prop:slopedetindex} may be written more explicitly as follows:

\begin{corollary} \label{prop:renormexpl}
     Let $P$ be a vertex of $\theta_{\pi}(C)= \Theta_L(\hat{C}_{\pi})$, different from the root $L$. 
     Assume that when one moves on the segment 
     $[L, P]$ from $L$ to $P$, one meets successively the vertices $P_1, \dots, P_{k}=P$ of 
     $\delta_P \cup \{ P\}$. 
     Denote $\slp_{\pi}(P_j) = d_j / c_j$ with coprime $c_j, d_j \in \N^*$, 
     for all $j \in \{1, \dots, k \}$ (with $c_k = 1$ and 
     $d_k = \infty$ if $P$ is a leaf of the tree). Then:
        \begin{enumerate}
             \item  $\de_L(P) = c_1 \cdots c_{k-1}$. 
             \item $\ic_L(P) = \dfrac{d_1}{c_1} + \dfrac{d_2}{c_1^2 c_2} + \dfrac{d_3}{c_1^2 c_2^2 c_3} + 
                         \cdots + \dfrac{d_k}{c_1^2 \cdots c_{k-1}^2 c_k}$. 
             \item $\ex_L(P) = \dfrac{d_1}{c_1} + \dfrac{d_2}{c_1c_2} + \dfrac{d_3}{c_1 c_2 c_3} + 
                         \cdots + \dfrac{d_k}{c_1 \cdots c_k}$. 
        \end{enumerate}
\end{corollary}

\begin{example}  \label{ex:branchPuiseux}
   Let us specialize Corollary \ref{prop:renormexpl} to the case where $P$ is a 
   leaf of $\theta_{\pi}(C)= \Theta_L(\hat{C}_{\pi})$, labeled by a branch $C$.
   Therefore the characteristic exponents of a Newton-Puiseux series of 
    $C$ relative to $L$ are:
      \begin{equation} \label{eq:fromNtoP}  
             \dfrac{m_j}{n_1 \cdots n_j} := \dfrac{d_1}{c_1} + \dfrac{d_2}{c_1c_2}  + 
                         \cdots + \dfrac{d_j}{c_1 \cdots c_j}, 
       \end{equation}
     for all $j \in \{1, \dots, k \}$. 
     Here the positive integers $(m_1, \dots, m_k)$ and $(n_1, \dots, n_k)$ are chosen 
     such that $m_j$ and $n_j$ are coprime for all $j \in \{1, \dots , k\}$. 
     The relations (\ref{eq:fromNtoP}) 
     may be reexpressed in the following way: 
      \begin{equation} \label{eq:NpPp} 
            (c_j, d_j) = ( n_j, m_j - n_j \cdot m_{j-1}),  
        \end{equation} 
     for all $j \in \{1, \dots, k \}$ (with the convention $m_0 := 0$). 
     Sometimes  
     the couples $(m_j, n_j)$ are called the \emph{Puiseux pairs} and 
     the couples $(d_j, c_j)$ are called the \emph{Newton pairs} of the given 
     Newton-Puiseux series. The importance of using both sequences of pairs 
     in the topological study of plane curve singularities was emphasized  by 
     Eisenbud and Neumann in their book \cite[Page 6]{EN 85}. More details 
     may be found in Weber's survey \cite[Section 6.1]{W 08}. 
     \end{example}

  \begin{example}     \label{ex:continex}
     This is a continuation of Example \ref{ex:branchPuiseux}. 
     Consider pairs of coprime integers $(n_j, m_j) \in \N^*\times \N^*$ 
     with $n_j >1$, for $j=1, \dots, k$ and  the Newton-Puiseux series 
     \[ x^{m_1/n_1} + x^{m_2/(n_1n_2)} + \cdots + x^{m_k/(n_1\cdots n_k)}, \]
     defining a branch $C$. 
     We can build a toroidal pseudo-resolution $\pi$ of $C$ with respect to $L=Z(x)$, 
     such that 
     $\hat{C}_\pi = L + C + \sum_{j=1}^k L_j$ and the branches $L_1, \ldots, L_k$ are 
     defined by the Newton-Puiseux series: 
    
        \[  
              0,  \quad
               x^{m_1/n_1} , \quad x^{m_1/n_1} + x^{m_2/(n_1n_2)},
                   \quad \dots , \quad 
               x^{m_1/n_1} + x^{m_2/(n_1n_2)}+ \cdots + 
                    x^{m_{k-1}/(n_1\cdots n_{k-1})}. 
        \]
          
      Then the associated lotus is as represented in Figure \ref{fig:branchlotus}. 
      Using formula (\ref{eq:NpPp}) and the notations introduced in Example \ref{ex:onebranch}, 
      we have:
         \[\frac{m_j}{n_j} -  m_{j-1} = [p_j, q_j, \ldots], \]
      for all $j \in \{1, \ldots, k\}$. 
      In fact, one gets the same lotus whenever $C$ is an arbitrary branch with the previous 
      characteristic exponents relative to $L$ 
      and the branches $L_j$ are \emph{semiroots} of $C$ (see \cite[Corollary 5.6]{PP 03}). 
      This shows that our notion of \emph{completion of a reduced curve singularity $C$ 
      relative to a toroidal pseudo-resolution process is a generalization of the operation 
      which adds to a branch a complete system of semiroots relative to $L$ (see 
      \cite[Definition 6.4]{PP 03}).   }
      \end{example}

\subsection{Historical comments}
\label{ssec:HAEWtrees}
$\:$
\medskip

Historical information about the notion of \emph{characteristic exponent} may be found 
in our paper \cite[Introduction, Rem. 2.9]{GBGPPP 17}. 

 In addition to the older Enriques diagrams and dual graphs of exceptional divisors 
of embedded resolutions, Kuo and Lu associated a third kind of tree to 
a curve singularity $C= Z(f(x,y))$ in their 1977 
paper \cite{KL 77}. An example of such a tree, extracted from their paper,  
is shown in Figure \ref{fig:Eggers-tree}. Their trees  
were rooted and their sets of leaves were in bijection with the set of Newton-Puiseux series 
$\eta(x)$ associated with the corresponding plane curve singularity $C$. 
They used their trees in order to relate the structure of  $C$ to that of its \emph{polar curve} 
defined by the equation $\dfrac{\partial f}{\partial y}=0$.

\begin{figure}[h!] 
 \centering 
 \includegraphics[scale=0.40]{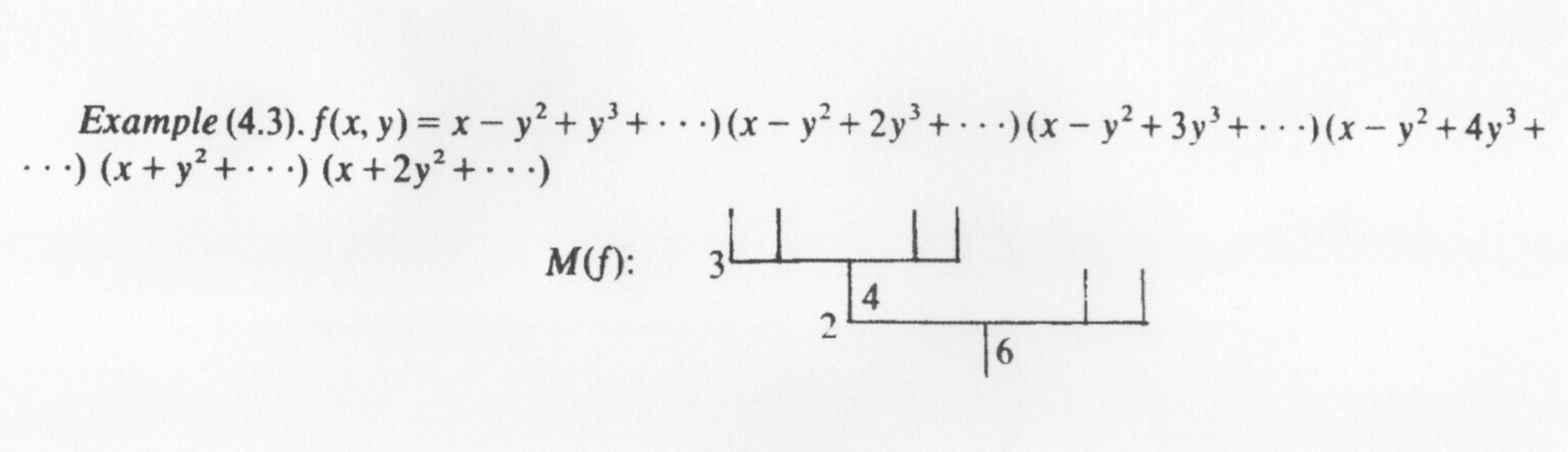} 
 \caption{A Kuo-Lu tree} 
 \label{fig:KL-tree}
 \end{figure}

In his 1983 paper \cite{E 83}, 
Eggers showed that a kind of Galois quotient of  the Kuo-Lu tree of $f$
was more convenient for this 
purpose. Figure \ref{fig:Eggers-tree} shows the first example given in \cite{E 83}.  
A variant of the Eggers tree, better suited for computations, was introduced 
by Wall \cite{W 03} and presented in more details in his textbook 
\cite[Sections 4.2 and 9.4]{W 04}.

\begin{figure}
 \centering 
 \includegraphics[scale=0.8]{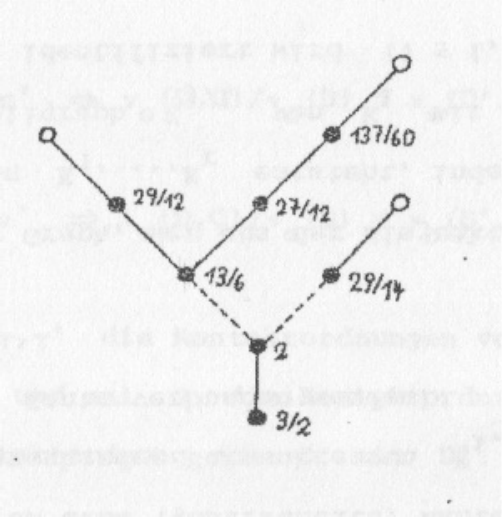} 
 \caption{An Eggers tree} 
 \label{fig:Eggers-tree}
 \end{figure}

The third author coined in his 2001 thesis \cite{PP 01} the name \emph{Eggers-Wall tree} 
for Wall's version of Eggers' tree. He proved in \cite[Section 4.4]{PP 01} 
that the Eggers-Wall  tree of $C$ relative to generic coordinates could almost always 
be embedded in the dual graph of the minimal embedded resolution of $C$ as 
the convex hull of its vertices representing the branches of $C$. He discovered 
this fact experimentally, by applying in many examples the first author's algorithm  
described in her 1996 thesis \cite[Section 1.4.6]{GB 96}, 
for the passage from Eggers' tree to the dual graph. 
Another proof of this embedding result was obtained in terms of 
certain toroidal-pseudo resolutions introduced by the second author in \cite[Section 3.4]{GP 03}. 
Wall improved the description of this embedding in his 2004 book \cite{W 04}, and 
Favre and Jonsson explained it differently from their valuative viewpoint in their 2004 book 
\cite[Appendix D2]{FJ 04}. 
Recently, we gave a new viewpoint on this embedding 
result in \cite[Theorem 112]{GBGPPP 18a}, in the framework of Eggers-Wall  
trees defined relative to arbitrary coordinate systems. 
It is important to consider the Eggers-Wall tree 
of $C$ relative to coordinate systems which are not necessarily generic relative to $C$. 
Indeed, this freedom is essential when one 
wants to compare the Eggers-Wall tree of $C$ with that of its strict transform by a 
blow up or a more complicated toric modification, because after such a modification 
the natural coordinate $x$ defines the exceptional divisor, and is not necessarily 
 generic with respect to the strict transform. In his paper \cite{PP 04}, 
 extracted from his thesis \cite{PP 01}, the third author did not consider any 
genericity hypothesis,  in order to extend the definition 
of this kind of tree to higher dimensional \emph{quasi-ordinary hypersurface singularities}.  
This generalized notion of Eggers-Wall tree was further developed in connexion 
with the study of the associated polar hypersurfaces in the 2005 paper 
\cite{GBGP 05} of the first and second authors. 
 In turn, the notion of Kuo-Lu tree was extended to quasi-ordinary 
hypersurface singularities by the first author and 
Gwo\'zdziewicz in their 2015 paper \cite{GBG 15} and used again by them in \cite{GBG 19},  
in order to study the structure of higher order polars of such singularities.

The notations for elementary Newton polygons 
described in Definition \ref{def:elempolyg} were introduced by Teissier 
 in his 1977 paper \cite[Section 3.6]{T 77}, where he restricted them to $a,b \in \N^* \cup \{\infty\}$.  
 Allowing the two numbers in Definition \ref{def:elempolyg} 
 to be rational is convenient in order to express Newton polygons in terms of  
 Eggers-Wall trees (see Corollary \ref{cor:Newton}).

  \medskip
     Let us consider now the valuative aspects of Eggers-Wall trees. 
     Favre and Jonsson proved in their 2004 book \cite{FJ 04} that the set of semivaluations 
     of the local $\C$-algebra $\hat{\cO}_{S,o}$ which are normalized by the constraint 
     that a defining function $x$ of the smooth germ $L$ has value $1$,  
     has a natural structure of rooted real tree, which they 
     called \emph{the valuative tree}. In his 2015 survey \cite{J 15}, Jonsson revisited 
     part of the theory of \cite{FJ 04} with a more geometric 
     approach which is valid for algebraically closed fields of arbitrary characteristic.
     Favre and Jonsson gave several descriptions of its tree structure. 
     In our paper \cite[Theorem 8.34]{GBGPPP 18b} we gave a new description of it, as the 
     universal Eggers-Wall tree of Definition \ref{def:univEW}. Namely, 
     we proved that the valuative tree could also be obtained as a projective limit of Eggers-Wall trees. 
     The main point of our proof is that 
     $\Theta_L(C)$ embeds naturally in the valuative tree, for any $C$.  We showed also 
     in  \cite[Theorem 8.18]{GBGPPP 18b} that 
     the triple $(\de_L, 1 + \ex_L, \ic_L)$ is the pullback by this embedding of a triple 
     of three natural functions on the valuative tree: the \emph{multiplicity}, the 
     \emph{log-discrepancy} and the \emph{self-interaction}. 
     
     An advantage of the identification 
     of $\Theta_L$ with the valuative tree is that it allows to get an interpretation of the points 
     of $\Theta_L$ which do not belong to any $\Theta_L(C)$ as special \emph{infinitely singular} 
     semivaluations, in the language of \cite{FJ 04} and \cite{J 15}. 

Another advantage is obtained
    when the base algebraically closed field has positive characteristic. 
   Let us  \emph{define} the functions $\de_L$, $\ic_L$ and $\ex_L$ 
  on $\theta_\pi (C)$ by the equalities appearing in Proposition \ref{prop:slopedetindex}.
  This provides a \emph{definition} of a notion of Eggers-Wall tree 
  in positive characteristic,  where Newton-Puiseux series are not enough for the 
  study of plane curve singularities (see Remark \ref{rem:rootproblem}). 
  The approach of Section \ref{ssec:renormres} may be 
  generalized to prove that in restriction to $\theta_\pi (C)$, 
  the multiplicity function relative to $L$ is equal to $\de_L$, 
  the contact complexity function relative to $L$ is equal to $\ic_L$ 
  and the log-discrepancy function relative to $L$ is equal to $1 + \ex_L$. 
  This abstract Eggers-Wall tree may be associated with the \emph{ultrametric distance} 
  on the branches of $C$, as described in our paper \cite{GBGPPP 18a}. It
  may be seen also as a generalization of the notion of characteristic exponents 
  in positive characteristic  introduced in Campillo's book \cite{Campillo}, 
  where the author computes these exponents using {\em Hamburger-Noether expansions} 
  (see \cite[Section 3.3]{Campillo}), infinitely near points (see \cite[Remark 3.3.8]{Campillo}) 
   or Newton polygons (see \cite[Section 3.4]{Campillo}).

\medskip
Assume now that the germ $C$ is holomorphic. Then
  the Enriques diagram and the weighted dual graph of the minimal embedded resolution,  
 as well as the Eggers-Wall tree relative to generic coordinates encode the same information, 
which is equivalent to the embedded topological type of $C$. Proofs of this 
fundamental fact may be found in Wall's book \cite[Propositions 4.3.8 and 4.3.9]{W 04}. 

A basic problem is then to find methods to transform one kind of tree into the two other kinds. 
Noether described in \cite{N 90} how to pass from the \emph{characteristic 
exponents} of an irreducible curve singularity $C$ to the structure of the blow up process leading 
to an embedded resolution. Enriques and Chisini generalized this approach in 
\cite[Libr. IV, Cap. I]{EC 17} to the case when $C$ is not necessarily irreducible. 
Namely, they showed how 
to pass from the characteristic exponents of its branches and the orders of coincidence 
of pairs of branches in generic coordinates to the associated Enriques diagram. 

Zariski and Lejeune-Jalabert  
proved by different methods in their 1971 paper \cite{Z 71} and 1972 thesis \cite{LJ 72} 
respectively, that the characteristic 
exponents of the branches of $C$ and the intersection numbers of its pairs of branches 
determine the embedded topological type of $C$
and the combinatorics of its minimal embedded 
resolution. This may be seen as a proof of the fact that the weighted dual graph of the minimal 
embedded resolution is equivalent to the generic Eggers-Wall tree. Methods to 
pass from the knowledge of the characteristic exponents and intersection numbers 
to the dual graph were explained by Eisenbud and Neumann \cite[Appendix to Ch. 1]{EN 85},  
Brieskorn and Kn\"orrer \cite[Section 8.4]{BK 86}, Michel and Weber \cite{MW 85},    
de Jong and Pfister \cite[Section 5.4]{DJP 00} and an algorithm was described by the first author in 
\cite[Sect. 1.4.6]{GB 96}. 

\medskip
Let us mention now several other trees which were associated to plane curve 
singularities. 

As explained in Subsection \ref{ssec:HAtoroidal}, the changes of variables considered by 
Puiseux (called sometimes \emph{Newton maps}) were compositions 
of affine and of toric ones, which in general were not birational. Nevertheless, an algorithm 
of abstract resolution and of computation of Newton-Puiseux series may be developed also 
using them. A variant of the fan trees, adapted to this context and called 
\emph{Newton trees}, was used by Cassou-Nogu\`es 
in her papers mentioned in Subsection \ref{ssec:HAtoroidal}, written alone or in collaboration. 
The Newton trees encode also the toroidal pseudo-resolution processes
described in the paper \cite{CNL 14} of Cassou-Nogu\`es  and Libgober.
We refer the reader especially to the papers  \cite{CNP 11} and \cite{CNV 14} 
for more details about this approach.  The changes 
of coordinates (\ref{eq:change}), which are very similar to Newton maps, 
were also used in the paper \cite{KM 10} of Kennedy and McEwan to study the 
monodromy of  holomorphic plane curve singularities.

Newton maps and Newton trees have  been used to study the singularities of quasi-ordinary hypersurfaces by Artal, Cassou-Nogu\`es, Luengo and Melle Hern\'andez (see  for instance 
\cite{ACLM 05} and \cite{ACLM 13}).  In their 2014 paper \cite{GPGV 14}, the second author 
and Gonz\'alez Villa compared the Newton maps with the toric morphisms 
appearing in a toroidal pseudo-resolution of an irreducible germ of quasi-ordinary hypersurface. 

Newton trees are algebraic variants of the \emph{splice diagrams} associated 
by Eisenbud and Neumann in their 1986 book \cite{EN 85} to any oriented graph link in an 
integral homology sphere, extending a graphical convention introduced by Siebenmann 
in his 1980 paper \cite{S 80}. In our recent paper \cite[Section 5]{GBGPPP 18b}, we explained how to 
pass from the Eggers-Wall tree of a  holomorphic plane
curve singularity $C$ relative to a smooth branch $L$ to 
the splice diagram of the oriented link of $L + C$ in $\bS^3$. 

In his 1993 papers \cite{K 93} and \cite{K 93bis}, Kapranov associated a  
version of Kuo and Lu's trees  to finite sets of formal 
power series with complex and real coefficients respectively. He called them \emph{Bruhat-Tits trees}.

A version of Kuo and Lu's trees was used recently by Ghys in his book \cite{G 17} 
about the topology of \emph{real} plane curve singularities. He associated two 
such trees, one for $x > 0$ and another one for $x < 0$ to any germ whose branches are smooth 
and transversal to the reference branch $x=0$, and studied  their relation, 
describing all the possible couples of such trees. In a theorem proved with 
Christopher-Lloyd Simon (see \cite[Page 266]{G 17}), Ghys extended this analysis 
to all plane curve singularities with 
only real branches. For this more general problem, it was not any more a variant of 
Kuo and Lu's tree which was crucial, but a real version of the dual graph of the 
associated minimal resolution. A different real version of the dual resolution graph was 
introduced  before by Castellini in \cite[Chap. 3]{C 15}. 

Ghys' version of Kuo and Lu's trees was also used by Sorea in her study 
\cite{S 18} of curve singularities defined over $\R$ but without any real branch, that is, 
singularities of real analytic functions $f(x,y)$ in the neighborhood of a local 
maximum or minimum. Those trees were related in this work with another kind of 
tree, defined using Morse theory, the so-called \emph{Poincar\'e-Reeb tree} of 
the function $f$ relative to $x$. 

Versions of our fan tree were considered by Weber in his 2008 survey  
\cite{W 08} about the embedded topological type of holomorphic plane curve singularities, 
based on the earlier 1985 preprint \cite{MW 85} of Michel and Weber, 
which contained also many examples. 
The reading of Weber's survey \cite{W 08} should facilitate the interpretations of 
the objects manipulated in this paper in terms of the 
embedded topological type of $C$.

\section{Overview and perspectives}
\label{sec:hist}
\medskip

We begin this final section by an overview of the content of the paper. Then we 
formulate a few remarks about perspectives of development 
of the use of lotuses in the study of singularities. The final 
Subsection \ref{ssec:genterm}  contains a list of notations  used in this paper.

\subsection{Overview}
\label{ssec:ovv}
$\:$
\medskip

In this subsection we give an overview 
of the construction of the fan tree and of the associated lotus 
from the Newton fans generated by a toroidal pseudo-resolution process of a plane curve 
singularity. It helps us to understand the relations between Newton polygons, 
Newton-Puiseux series, iterations of blow ups, final exceptional divisor 
and the associated Enriques diagrams, dual graphs and Eggers-Wall trees. 

\begin{figure}[h!]
\begin{center}
\begin{tikzpicture}[scale=0.19]

 \begin{scope}[shift={(-7,0)},scale=1]
\foreach \x in {0,1,...,5}{
\foreach \y in {0,1,...,5}{
       \node[draw,circle,inner sep=0.7pt,fill, color=gray!40] at (1*\x,1*\y) {}; }
   }
\draw [->, color=cyan](0,0) -- (0,6);
\draw [->, color=blue](0,0) -- (6,0);
\draw [-, color=orange, very thick](1,0) -- (0, 1) ; 
\draw [-, thick](0,0) -- (2.3,5.75);
\draw [-,thick](0,0) -- (2.8,5.6);
\draw [-, thick](0,0) -- (5.5,3.3);
\node[draw,circle, inner sep=1.2pt,color=black, fill=black] at (1,0){};
\node[draw,circle, inner sep=1.2pt,color=black, fill=black] at (0,1){};
\node[draw,circle, inner sep=1.2pt,color=red, fill=red] at (2,5){};
\node[draw,circle, inner sep=1.2pt,color=red, fill=red] at (1,2){};
\node[draw,circle, inner sep=1.2pt,color=red, fill=red] at (5,3){};
\end{scope}

\begin{scope}[shift={(-6,0)},scale=1]
   \draw [-, color=orange, very thick](11,0) -- (11, 8) ; 
   \node[draw,circle, inner sep=1.5pt,color=black, fill=black] at (11,0){};
   \node[draw,circle, inner sep=1.5pt,color=black, fill=black] at (11,8){};
   \node[draw,circle, inner sep=1.5pt,color=red, fill=red] at (11,2){};
   \node[draw,circle, inner sep=1.5pt,color=red, fill=red] at (11,4){};
   \node[draw,circle, inner sep=1.5pt,color=red, fill=red] at (11,6){};
 \end{scope}  
 
 \begin{scope}[shift={(10,0)},scale=1.3]
 \draw [->](0,0) -- (0,6);
\draw [->](0,0) -- (6,0);

\draw[fill=pink!40](1,0) -- (0,1) -- (1,1)  --cycle;
\draw[fill=pink!40](1,1) -- (1,2) -- (0,1) --cycle;
\draw[fill=pink!40](1,2) -- (1,3) -- (0,1) --cycle;
\draw[fill=pink!40](1,2) -- (1,3) -- (2,5) --cycle;
\draw[fill=pink!40](1,0) -- (1,1) -- (2,1) --cycle;
\draw[fill=pink!40](1,1) -- (2,1) -- (3,2) --cycle;
\draw[fill=pink!40](2,1) -- (3,2) -- (5,3) --cycle;

\draw [-, ultra thick, color=orange](0,1) -- (2,5) -- (1,2) -- (1,1) -- (5,3) -- (2,1) --(1,0);
\foreach \x in {0,1,...,5}{
\foreach \y in {0,1,...,5}{
       \node[draw,circle,inner sep=0.7pt,fill, color=gray!40] at (1*\x,1*\y) {}; }
   }

\node[draw,circle, inner sep=1.2pt,color=red, fill=red] at (2,5){};
\node[draw,circle, inner sep=1.2pt,color=red, fill=red] at (1,2){};
\node[draw,circle, inner sep=1.2pt,color=red, fill=red] at (5,3){};

\draw [->, very thick, red] (1,0)--(0.5, 0.5);
\draw [-, very thick, red] (0.5, 0.5)--(0,1);
 \end{scope}      
     
 \begin{scope}[shift={(-7,-10)},scale=1]
\foreach \x in {0,1,...,5}{
\foreach \y in {0,1,...,5}{
       \node[draw,circle,inner sep=0.7pt,fill, color=gray!40] at (1*\x,1*\y) {}; }
   }
\draw [->, color=cyan](0,0) -- (0,6);
\draw [->, color=blue](0,0) -- (6,0);
\draw [-, thick](0,0) -- (5.2,3.9);
\draw [-, thick](0,0) -- (5.4,3.6);
\draw [-, color=orange, very thick](1,0) -- (0, 1) ; 
\node[draw,circle, inner sep=1.2pt,color=black, fill=black] at (1,0){};
\node[draw,circle, inner sep=1.2pt,color=black, fill=black] at (0,1){};
\node[draw,circle, inner sep=1.2pt,color=red, fill=red] at (4,3){};
\node[draw,circle, inner sep=1.2pt,color=red, fill=red] at (3,2){};

\end{scope}

 \begin{scope}[shift={(-6,-10)},scale=1]
  \draw [-, color=orange, very thick](11,0) -- (11, 6) ; 
   \node[draw,circle, inner sep=1.5pt,color=black, fill=black] at (11,0){};
   \node[draw,circle, inner sep=1.5pt,color=black, fill=black] at (11,6){};
   \node[draw,circle, inner sep=1.5pt,color=red, fill=red] at (11,2){};
    \node[draw,circle, inner sep=1.5pt,color=red, fill=red] at (11,4){};
 \end{scope}
 
 \begin{scope}[shift={(10,-10)},scale=1]
 \draw [->](0,0) -- (0,6);
\draw [->](0,0) -- (6,0);
\draw[fill=pink!40](1,0) -- (0,1) -- (1,1)  --cycle;
\draw[fill=pink!40](1,0) -- (1,1) -- (2,1) --cycle;
\draw[fill=pink!40](1,1) -- (2,1) -- (3,2) --cycle;
\draw[fill=pink!40](1,1) -- (3,2) -- (4,3) --cycle;
\draw [-, ultra thick, color=orange](0,1) -- (1,1) -- (4,3) -- (1,0);
\foreach \x in {0,1,...,5}{
\foreach \y in {0,1,...,5}{
       \node[draw,circle,inner sep=0.7pt,fill, color=gray!40] at (1*\x,1*\y) {}; }
   }
\node[draw,circle, inner sep=1.2pt,color=red, fill=red] at (3,2){};
\node[draw,circle, inner sep=1.2pt,color=red, fill=red] at (4,3){};
\draw [->, very thick, red] (1,0)--(0.5, 0.5);
\draw [-, very thick, red] (0.5, 0.5)--(0,1);
 \end{scope}      
     

 \begin{scope}[shift={(-7,-20)},scale=1]
\foreach \x in {0,1,...,5}{
\foreach \y in {0,1,...,5}{
       \node[draw,circle,inner sep=0.7pt,fill, color=gray!40] at (1*\x,1*\y) {}; }
   }
\draw [->, color=cyan](0,0) -- (0,6);
\draw [->, color=blue](0,0) -- (6,0);
\draw [-, thick](0,0) -- (1.8,5.4);
\draw [-, thick](0,0) -- (3.3,5.5);
\draw [-, color=orange, very thick](1,0) -- (0, 1) ; 
\node[draw,circle, inner sep=1.2pt,color=black, fill=black] at (1,0){};
\node[draw,circle, inner sep=1.2pt,color=black, fill=black] at (0,1){};
\node[draw,circle, inner sep=1.2pt,color=red, fill=red] at (1,3){};
\node[draw,circle, inner sep=1.2pt,color=red, fill=red] at (3,5){};
\end{scope}

 \begin{scope}[shift={(-6,-20)},scale=1]
  \draw [-, color=orange, very thick](11,0) -- (11, 6) ; 
   \node[draw,circle, inner sep=1.5pt,color=black, fill=black] at (11,0){};
   \node[draw,circle, inner sep=1.5pt,color=black, fill=black] at (11,6){};
   \node[draw,circle, inner sep=1.5pt,color=red, fill=red] at (11,2){};
   \node[draw,circle, inner sep=1.5pt,color=red, fill=red] at (11,4){};
 \end{scope}

 \begin{scope}[shift={(10,-20)},scale=1]
 \draw [->](0,0) -- (0,6);
\draw [->](0,0) -- (6,0);
\draw[fill=pink!40](1,0) -- (0,1) -- (1,1)  --cycle;
\draw[fill=pink!40](0,1) -- (1,1) -- (1,2) --cycle;
\draw[fill=pink!40](1,1) -- (1,2) -- (1,3) --cycle;
\draw[fill=pink!40](0,1) -- (1,2) -- (1,3) --cycle;
\draw[fill=pink!40](1,1) -- (1,2) -- (2,3) --cycle;
\draw[fill=pink!40](1,2) -- (2,3) -- (3,5) --cycle;
\draw [-, ultra thick, color=orange](0,1) -- (1,3) -- (1,2) -- (3,5) -- (1,1) --(1,0);
\foreach \x in {0,1,...,5}{
\foreach \y in {0,1,...,5}{
       \node[draw,circle,inner sep=0.7pt,fill, color=gray!40] at (1*\x,1*\y) {}; }
   }
\node[draw,circle, inner sep=1.2pt,color=red, fill=red] at (3,5){};
\node[draw,circle, inner sep=1.2pt,color=red, fill=red] at (1,3){};
\draw [->, very thick, red] (1,0)--(0.5, 0.5);
\draw [-, very thick, red] (0.5, 0.5)--(0,1);
 \end{scope}

 \begin{scope}[shift={(-7,-30)},scale=1]
\foreach \x in {0,1,...,5}{
\foreach \y in {0,1,...,5}{
       \node[draw,circle,inner sep=0.7pt,fill, color=gray!40] at (1*\x,1*\y) {}; }
   }
\draw [->, color=cyan](0,0) -- (0,6);
\draw [->, color=blue](0,0) -- (6,0);
\draw [-, thick](0,0) -- (5.4,2.7);
\draw [-, color=orange, very thick](1,0) -- (0, 1) ;
\node[draw,circle, inner sep=1.2pt,color=black, fill=black] at (1,0){};
\node[draw,circle, inner sep=1.2pt,color=black, fill=black] at (0,1){};
\node[draw,circle, inner sep=1.2pt,color=red, fill=red] at (2,1){};
\end{scope}

 \begin{scope}[shift={(-6,-30)},scale=1]
  \draw [-, color=orange, very thick](11,0) -- (11, 6) ; 
   \node[draw,circle, inner sep=1.5pt,color=black, fill=black] at (11,0){};
   \node[draw,circle, inner sep=1.5pt,color=black, fill=black] at (11,6){};
   \node[draw,circle, inner sep=1.5pt,color=red, fill=red] at (11,3){};
 \end{scope}
 
\begin{scope}[shift={(10,-30)},scale=1]
 \draw [->](0,0) -- (0,6);
\draw [->](0,0) -- (6,0);
\draw[fill=pink!40](1,0) -- (0,1) -- (1,1)  --cycle;
\draw[fill=pink!40](1,0) -- (1,1) -- (2,1) --cycle;

\draw [-, ultra thick, color=orange](0,1) -- (2,1) -- (1,0);
\foreach \x in {0,1,...,5}{
\foreach \y in {0,1,...,5}{
       \node[draw,circle,inner sep=0.7pt,fill, color=gray!40] at (1*\x,1*\y) {}; }
   }
\node[draw,circle, inner sep=1.5pt,color=red, fill=red] at (2,1){};

\draw [->, very thick, red] (1,0)--(0.5, 0.5);
\draw [-, very thick, red] (0.5, 0.5)--(0,1);
 \end{scope}

 \draw [->, very thick] (18,-12) -- (22, -12);
 

  \begin{scope}[shift={(32,-20)},scale=2]

 \draw [-, color=orange, very thick](0,0) -- (0, 8) ; 
    \draw [-, color=orange, very thick](0, 2) -- (6, 2) ; 
     \draw [-, color=orange, very thick](0, 2) -- (6, 8) ; 
     \draw [-, color=orange, very thick](2,4) -- (2, 8) ; 
     \draw [-, color=magenta, very thick](0,6) -- (-2, 6) ; 
     \draw [-, color=magenta, very thick](2,2) -- (2, 0) ; 
     \draw [-, color=magenta, very thick](4,6) -- (6, 6) ; 
     \draw [-, color=magenta, very thick](2,7) -- (4, 7) ; 
     \draw [-, color=magenta, very thick](4,2) -- (4, 0) ; 
     \draw [-, color=magenta, very thick](0,4) -- (-2, 4) ; 
     \draw [-, color=magenta, very thick](0,4) -- (-2, 2) ;
      \draw [-, color=blue, very thick](0,0) -- (0, 4) ; 
      \draw [-, color=blue, very thick](0,6) -- (0, 8) ; 
      \draw [-, color=blue, very thick](0,2) -- (2, 2) ; 
      \draw [-, color=blue, very thick](4,2) -- (6, 2) ; 
      \draw [-, color=blue, very thick](0,2) -- (4,6) ;
      \draw [-, color=blue, very thick](2,7) -- (2,8) ;  

   \node[draw,circle, inner sep=1.5pt,color=black, fill=black] at (0,0){};
   \node [below] at (0,0) {$L$};
   \node[draw,circle, inner sep=1.5pt,color=black, fill=black] at (0,8){};
   \node[draw,circle, inner sep=1.5pt,color=red, fill=red] at (0,2){};
   \node[draw,circle, inner sep=1.5pt,color=red, fill=red] at (0,4){};
   \node[draw,circle, inner sep=1.5pt,color=red, fill=red] at (0,6){};
   \node[draw,circle, inner sep=1.5pt,color=black!20!green, fill=black!20!green] at (-2,6){};  
   \node[draw,circle, inner sep=1.5pt,color=black!20!green, fill=black!20!green] at (-2,4){};
   \node[draw,circle, inner sep=1.5pt,color=black!20!green, fill=black!20!green] at (-2,2){};
   \node[draw,circle, inner sep=1.5pt,color=black, fill=black] at (6,2){};
   \node[draw,circle, inner sep=1.5pt,color=red, fill=red] at (2,2){};
   \node[draw,circle, inner sep=1.5pt,color=red, fill=red] at (4,2){};
   \node[draw,circle, inner sep=1.5pt,color=black!20!green, fill=black!20!green] at (2,0){};
   \node[draw,circle, inner sep=1.5pt,color=black!20!green, fill=black!20!green] at (4,0){};
   \node[draw,circle, inner sep=1.5pt,color=black, fill=black] at (6,8){};
   \node[draw,circle, inner sep=1.5pt,color=red, fill=red] at (4,6){};
   \node[draw,circle, inner sep=1.5pt,color=red, fill=red] at (2,4){};
     \node[draw,circle, inner sep=1.5pt,color=black!20!green, fill=black!20!green] at (6, 6){};
   \node[draw,circle, inner sep=1.5pt,color=black, fill=black] at (2,8){};
   \node[draw,circle, inner sep=1.5pt,color=red, fill=red] at (2,7){};
   \node[draw,circle, inner sep=1.5pt,color=black!20!green, fill=black!20!green] at (4, 7){};
\end{scope}

\begin{scope}[shift={(62,-12)},scale=1.5]
     \draw [fill=pink!40](0,0) -- (3,3) -- (1,-2)--(0,0);
\draw [->, very thick, red] (0,0) --(0.5,-1);
 \draw [-, very thick, red] (0.5,-1) --(1,-2);
\draw [-] (0,0)--(1.5,-0.8);
\draw [-] (1,1)--(1.5,-0.8);
\draw [-] (2,2)--(1.5,-0.8);
\draw [->] (3,3)--(3.5,5.5);
\draw [->] (2,2)--(2,5);

\draw [fill=pink!40](0,0) -- (-2,-3)--(-7,-3)--(-8,-1)--(-4,-2)--(0,0);
       \draw [->, very thick, red] (-2,-3)--(-4.5,-3);
       \draw [-, very thick, red] (-4.5,-3)--(-7,-3);
\draw [-] (-4,-2)--(-2,-3);
\draw [-] (-2,-1)--(-1,-1.5);
\draw [-] (-4,-2)--(-1,-1.5);
\draw [-] (-4,-2)--(-7,-3);
\draw [-] (-7,-3)--(-6,-1.5);
\draw [-] (-7.5,-2)--(-6,-1.5);
\draw [->] (-8,-1)--(-9,1);
\draw [->] (-6,-1.6)--(-7,1);
\draw [->] (-6,-1.6)--(-6,1);

\draw [fill=pink!40](0,0) -- (-3.5,-0.5)--(-4,3)--(-2,2)--(-1.5,4)--(0,0);
    \draw [->, very thick, red]  (0,0) -- (-1.75,-0.25);
    \draw [-, very thick, red]  (-1.75,-0.25)--(-3.5,-0.5);
\draw [-] (-3.5,-0.5)--(-2,2);
\draw [-] (-3.5,-0.5)--(-0.5,1.2);
\draw [-] (-2,2)--(-1,2.5);
\draw [-] (-2,2)--(-1,2.5);
\draw [-] (-0.5,1.2)--(-2,2);
\draw [->] (-4,3)--(-5,5);

\draw [fill=pink!40](-1.5,4) -- (-2,6)--(-3,4)--(-1.5,4);
     \draw [-, very thick, red] (-1.5,4)--(-3,4);
     \draw [->, very thick, red] (-1.5,4)--(-2.25,4);
\draw [-] (-1.5,4)--(-2.5,5);
\draw [->] (-2,6)--(-2,8);

\draw [-, ultra thick, color=orange]  (-2, -3) -- (0,0) -- (-4, -2) -- (-6, -1.5) -- (-8, -1) -- (-7,-3);
\draw [-, ultra thick, color=orange] (0,0) -- (3,3) -- (1,-2);
\draw [-, ultra thick, color=orange] (0,0) -- (-1.5,4) -- (-2,2) -- (-4,3) -- (-3.5,-0.5);
\draw [-, ultra thick, color=orange] (-1.5,4) -- (-2,6) -- (-3,4);

\draw [-, ultra thick, color=blue]  (-2, -3) -- (0,0) -- (-4, -2) -- (-6, -1.5); 
\draw [-, ultra thick, color=blue] (-8, -1) -- (-7,-3);
    \draw [-, ultra thick, color=blue] (0,0) -- (2,2);
    \draw [-, ultra thick, color=blue] (3,3) -- (1,-2);
\draw [-, ultra thick, color=blue] (0,0) -- (-1.5,4) -- (-2,2) -- (-4,3); 
     \draw [-, ultra thick, color=blue] (-2,6) -- (-3,4);

\node[draw,circle,inner sep=1.3pt,fill=red, color=red] at (2,2){};
\node[draw,circle,inner sep=1.3pt,fill=red, color=red] at (3,3){};
\node[draw,circle,inner sep=1.3pt,fill=black] at (1,-2){};
\node[draw,circle,inner sep=1.3pt,fill=black] at (-2,-3){};
\node [below] at (-2,-3) {$L$};
\node[draw,circle,inner sep=1.3pt,fill=black] at (-7,-3){};
\node[draw,circle,inner sep=1.3pt,fill=red, color=red] at (-8,-1){};
\node[draw,circle,inner sep=1.3pt,fill=red, color=red] at (-6,-1.5){};
\node[draw,circle,inner sep=1.3pt,fill=black] at (-3.5,-0.5){};
\node[draw,circle,inner sep=1.3pt,fill=black] at (-3,4){};
\node[draw,circle,inner sep=1.3pt,fill=red,color=red] at (-2,6){};
\node[draw,circle,inner sep=1.3pt,fill=red, color=red] at (-1.5,4){};
\node[draw,circle,inner sep=1.3pt,fill=red, color=red] at (0,0){};
\node[draw,circle,inner sep=1.3pt,color=red,fill=red] at (-4,3){};
\end{scope}

\end{tikzpicture}
\end{center}
 \caption{Overview of the constructions of the paper}  
 \label{fig:overview} 
     \end{figure}
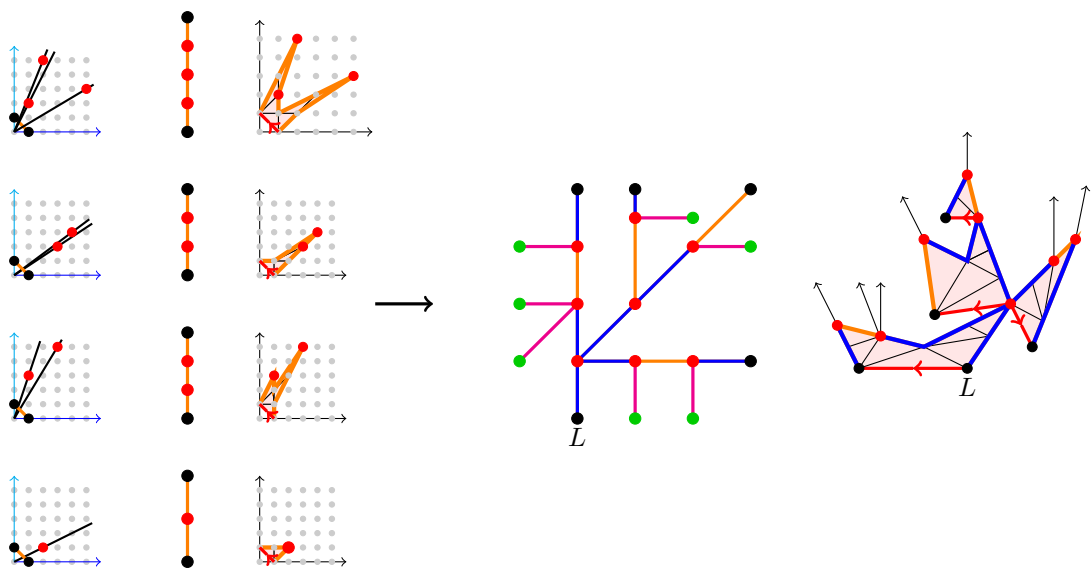

We invite the reader to look at Figure \ref{fig:overview}, which combines Figure 
\ref{fig:example-constrlotustoroid} and Figure \ref{fig:example-fantree-lotus}, 
but without their labels. Let us recall briefly  the names and main properties of the objects 
presented in this drawing,  which is our way to encode the combinatorics 
of Algorithm \ref{alg:tores}, 
and how they allow to visualize the relations between Enriques 
diagrams, dual graphs and Eggers-Wall trees (see Theorem \ref{thm:repsailtor}):

\begin{enumerate} 
    \item  Given a curve singularity $C$ embedded in a smooth germ of surface $S$, 
         study it using a \emph{cross} 
         $(L, L')$  (see Definition \ref{def:branchcross}). 
    \item  Construct the \emph{Newton fan} $\fan_{L, L'}(C)$ from the associated 
       \emph{Newton polygon} $\cN_{L, L'}(C)$ 
          (see Definitions \ref{def:Npolalg} and \ref{def:Npolseries}). 
   \item Draw the \emph{trunk} $\theta_{\fan_{L, L'}(C)}$  (see Definition \ref{def:fantrunk}) 
        and the \emph{lotus} $\Lambda(\fan_{L, L'}(C))$ (see Definitions 
        \ref{def:deflot}  and \ref{def:lotus-point}) of the Newton fan. 
    \item  As a simplicial complex, the lotus of a Newton fan is determined by the continued 
        fraction expansions of the slopes of the fan's rays (see Subsection \ref{ssec:lotcf}). 
    \item  Make the \emph{Newton modification}  (see Definition \ref{def:Npolseries}) 
      determined by the Newton fan and look at 
      the germs of the strict transform of $C$ at all its intersection points with the exceptional 
      divisor. All those points are smooth on the reduced total transform of $L + L'$. 
      For each such germ of the strict transform of $C$, 
      complete locally the exceptional divisor into a cross.
     \item  Each new cross allows to construct again a trunk and a lotus associated 
       to the corresponding germ of the strict transform of $C$. Combining the corresponding 
       Newton modifications, one gets a new level of Newton modifications. 
     \item  One iterates these constructions until reaching a toroidal surface $\Sigma$ 
       (see Definition \ref{def:toroidobj}) on which 
       the total transform of $C$ and of all the crosses used during the process is an abstract
       normal crossings curve, forming the boundary divisor  $\partial \Sigma$ 
       of a \emph{toroidal pseudo-resolution} $\pi$ of $C$ 
        (see Definition \ref{def:threeres}).  The map $\pi$ is also 
       a toroidal pseudo-resolution of the \emph{completion} $\hat{C}_{\pi}= \pi ( \partial \Sigma )$ 
       of $C$ relative to $\pi$  (see Definition \ref{def:threeres}), 
       which is a curve singularity containing 
       the branches of $C$ and all the branches  
       whose strict transforms are chosen to define crosses at certain steps of 
       Algorithm \ref{alg:tores}.       
     \item  In order to get a global combinatorial view, 
        one constructs the associated \emph{fan tree} $(\theta_{\pi}(C), \slp_{\pi})$ 
        (see Definition \ref{def:fantreetr}), by gluing the trunks generated by the   
        toroidal pseudo-resolution  process. 
        The function $\slp_{\pi} : \theta_{\pi}(C) \to [0, \infty]$ is called the \emph{slope function}.
     \item  The fan tree does not allow to visualize the decomposition of 
       the regularization $\pi^{reg}$ of $\pi$ (see Proposition \ref{prop:regalg}) into blow ups 
       of points. In order to get such a vision, one constructs the \emph{lotus} 
       $\Lambda_{\pi}(C)$ of the process  (see Definition \ref{def:lotustoroid})
       by gluing the Newton lotuses  (see Definition \ref{def:deflot})  of the 
       strict tranforms of $C$ relative to all the crosses used during the process. 
      \item   The edges of the lotus correspond bijectively to the crosses created during 
         the toroidal embedding resolution process by blow ups of points 
         (see Theorem \ref{thm:repsailtor} \eqref{point:edgeslot}). Therefore, one may see the lotus  
         as the space-time of the evolution of the dual graphs of the toroidal surfaces 
         appearing during this process.
      \item  The graph of the proximity binary relation (see Definition \ref{def:infnear}) 
                on  the constellation which is blown up 
                is the full subgraph of the $1$-skeleton 
                of the lotus $\Lambda_{\pi}(C)$ on its set of non-basic vertices (see 
                Theorem \ref{thm:repsailtor}  \eqref{point:graphprox}).
     \item  The Enriques diagram  (see Definition \ref{def:infnear})
       of the constellation of infinitely near points 
        blown up in order to decompose  $\pi^{reg}$, which are 
       the base points of the crosses appearing in
       the algorithm, is isomorphic with the Enriques tree 
        (see Definition \ref{def:lotustoroid}) of the lotus $\Lambda_{\pi}(C)$.
     \item   There is a second way of visualizing the Enriques diagram, using a 
        \emph{truncated lotus} $\Lambda_{\pi}^{tr}(C)$ (see Subsection \ref{ssec:trunclot}). 
     \item  The fan tree $\theta_{\pi}(C)$ is homeomorphic with the lateral boundary 
         $\partial_+ \Lambda_{\pi}(C)$   (see Definition \ref{def:lotustoroid}) 
          of the lotus generated by running Algorithm \ref{alg:tores}. 
     \item  The lateral boundary $\partial_+ \Lambda_{\pi}(C)$ is isomorphic with the 
         dual graph (see Definition \ref{def:dualgraph}) of the boundary divisor $\partial \Sigma$. 
         There is a simple combinatorial rule for reading on the lotus the self-intersection numbers 
         of the components of the exceptional divisor of the modification $\pi^{reg}$ 
         (see Theorem \ref{thm:repsailtor} \eqref{point:selfint}).  
     \item  The fan tree $\theta_{\pi}(C)$ is also isomorphic with the Eggers-Wall tree 
       $\Theta_L(\hat{C}_{\pi})$  (see Definition \ref{def:EW}) 
       of the completion of $C$ relative to the toroidal 
       modification $\pi$  (see Theorem \ref{thm:isomfantreeEW}). The triple of functions 
       (\emph{index $\de_L$, exponent $\ex_L$, contact complexity $\ic_L$}) defined 
       on $\Theta_L(\hat{C}_{\pi})$ is determined by the 
       \emph{slope function} $\slp_{\pi}$ on the fan tree through explicit formulae 
       (see Proposition \ref{prop:slopedetindex}).    
    \item  If $(L, L')$ is a cross on $S$, then the Eggers-Wall tree $\Theta_L(C + L')$ 
        determines the Newton polygon $\cN_{L, L'}(C)$ (see Corollary \ref{cor:Newton}).   
\end{enumerate}

\subsection{Perspectives}
\label{ssec:persp}
$\:$
\medskip

In this subsection we give a few perspectives on possible uses of lotuses. 
We believe that the lotuses of plane curve singularities may be useful in the following 
 research topics:

 \begin{enumerate}
     \item
     \emph{In the study of the topology of $\delta$-constant deformations of such 
       singularities}. As mentioned in Subsection \ref{ssec:HAEWtrees}, 
       Castellini's work \cite{C 15} gives a first step in 
       this direction. An important advantage of lotuses in this context is that 
       the lotuses of the singularities appearing in the deformations constructed 
       in \cite{C 15}  by A'Campo's 
       method embed in the lotus of the original singularity. This embedding relation is much 
       more difficult to express in terms of classical tree invariants of plane curve 
       singularities. A crucial question is to 
       understand whether this embedding property is specific to A'Campo type deformations, 
       or if it extends to other kinds of $\delta$-constant deformations.      
     \item 
    
    \emph{In the analogous study for \emph{real} plane curve singularities}. 
      One should probably describe real variants of the lotuses, embedded canonically up to isotopy 
       in an oriented real plane. 
       Again, Castellini's work \cite[Sect. 3.3.2]{C 15} gives a first step in this direction. 
     \item 

\emph{In the extension of the distributive lattice structures described by Pe Pereira and 
          the third author in \cite{PPPP 14} to arbitrary finite 
       constellations, and in the application of those structures to the problem of adjacency 
       of plane curve singularities}. The natural \emph{operad structure} on the set of finite lotuses 
       associated to toroidal pseudo-resolution processes (defined by gluing the base of one lotus 
       to an edge of the lateral boundary of another lotus) could be also useful in this direction. 
     \item 
   
   \emph{In the study of complex surface singularities through the Hirzebruch-Jung method} 
      (see \cite{PP 11bis}). 
      This method starts from a finite projection to a germ of smooth surface, and considers 
      then an embedded resolution of the discriminant curve. The lotuses of such discriminant 
      curves could be used as supports for encoding information about the initial finite projection, 
      from which one could read invariants of the surface singularity. 
 \end{enumerate}

\subsection{List of notations} 
\label{ssec:genterm}
$\:$
\medskip 

In order to help browsing through the text, we list the notations used for the main objects met in it:

\begin{description}
  \item[$\Bigl\{\Teisssr{a}{b}{3}{1.5}\Bigr\}$] 
      Elementary Newton polygon
       (see Definition \ref{def:elempolyg}). 
  \item[{$[a_1, \dots, a_k]$}] 
Continued fraction with terms 
      $a_1, \dots, a_k$ (see Definition \ref{def:contfrac}). 
  \item[{$c_m(f)$}]   Coefficient of the monomial $\chi^m$ in the series $f$ 
      (see Definition \ref{def:seriesinv}). 
  \item[{$\ic_L$}]  Contact complexity function (see Definition \ref{def:concom}). 
  \item[{$C_{L, L'}$}]   Strict transform of $C$ by the Newton modification $\psi_{L, L'}^{C}$ 
         (see Definition \ref{def:Npolseries}).
  \item[{$\hat{C}_{\pi}$}]   Completion of $C$ relative to the toroidal 
    pseudo-resolution $\pi$ (see Definition \ref{def:threeres}). 
  \item[{$\mathrm{Conv}(Y)$}]   Convex hull of a subset 
        $Y$ of a real affine space. 
  \item[{$\chi^m$}]   Monomial with exponent $m \in M$ (see the beginning 
      of Subsection \ref{ssec:torsurf}). 
  \item[{$\partial X$}]   Toric boundary of the toric variety $X$ 
        (see Definition \ref{def:boundtoric}), or 
      toroidal boundary of the toroidal variety $X$ (see Definition \ref{def:toroidobj}). 
  \item[{$\partial_+ \Lambda_{\pi}(C)$}]   Lateral boundary of the lotus 
      $\Lambda_{\pi}(C)$ (see Definition \ref{def:lotus-point}). 
   \item[{$\ex_L$}]   Exponent function  (see Definition \ref{def:EW} and Notations \ref{not:finotEW}). 
  \item[{$f_K$}]   Restriction of $f$ to the compact edge $K$ 
     of its Newton polygon (see Definition \ref{def:Npolalg}). 
  \item[{$\fan(f)$}]   Newton fan of the non-zero series $f \in \C[[x,y]]$
       (see Definition \ref{def:nfan}). 
  \item[{$\fan_{L, L'}(C)$}]   Newton fan of $C$ relative to the cross $(L, L')$ 
       (see Definition \ref{def:Npolseries}). 
  \item[{$\fan^{reg}$}]   Regularization of the fan $\fan$
        (see Definition \ref{def:regulariz}). 
  \item[{$\Enriques(\cC)$}]   Enriques diagram of the finite constellation $\cC$ 
       (see Definition \ref{def:infnear}). 
  \item[{$H_{f, \rho}$}]   Supporting half-plane of the Newton polygon $\cN(f)$ 
      determined by the ray $\rho \subset \sigma_0$ (see Proposition \ref{prop:minrealiz}). 
  \item[{$\de_L$}]    Index function  (see Definition \ref{def:EW} and Notations \ref{not:finotEW}). 
  \item [{$k_x(\xi, \xi')$}] Order of coincidence of two Newton-Puiseux 
       series (see Definition \ref{def:charcont}).
  \item  [{$k_x(C, C')$}]   Order of coincidence of two distinct branches, 
       relative to a local coordinate system $(x,y)$ (see Definition \ref{def:charcont}).
  \item [{$l_{\Z}$}]  Integral length (see Definition \ref{def:intlength}). 
  \item [{$(L, L')$}]   Cross on a germ of smooth surface 
        (see Definition \ref{def:branchcross}). 
  \item [{$\Lambda(\cF)$}]   Lotus of the Newton fan $\cF$ 
        (see Definition \ref{def:deflot}). 
  \item [{$\Lambda(\lambda_1, \dots, \lambda_r)$}]   Lotus associated 
      to the finite set $\{ \lambda_1, \dots, \lambda_r\} \subset \Q_+ \cup \{ \infty\}$ 
         (see Definition \ref{def:deflot}). 
  \item [{$\Lambda_{\pi}(C)$}]   Lotus of the toroidal pseudo-resolution
       $\pi$ of $C$ (see Definition \ref{def:lotustoroid}). 
  \item  [{$\Lambda_{\pi}^{trunc}(C)$}]   Truncation of the 
       lotus  $\Lambda_{\pi}(C)$ (see Definition \ref{def:trunclot}).
  \item [{$m_o(C)$}]   Multiplicity of the plane curve singularity $C$ 
       at the point $o$ (see Definition \ref{def:multcurve}). 
  \item [{$M_{L, L'}$}]   Monomial lattice associated to the cross $(L, L')$, 
       (see Definition \ref{def:ilattice}). 
  \item [{$\N$}]   Set of non-negative integers.
  \item[{$\N^*$}]  Set of positive integers.
  \item [{$N_{L, L'}$}]   Weight lattice associated to the cross $(L, L')$
       (see Definition \ref{def:ilattice}). 
  \item[{$\cN(f)$}]   Newton polygon of the non-zero series $f \in \C[[x,y]]$
       (see Definition \ref{def:Npolalg}). 
  \item [{$\cN_{L, L'}(C)$}]   Newton polygon of $C$ relative to the cross $(L, L')$
       (see Definition \ref{def:Npolseries}). 
  \item [{$O_{\rho}$}]  Toric orbit associated to the cone $\rho$ of a fan 
      (see the relation (\ref{eq:partsgen})). 
  \item[{$\hat{\cO}_{S,o}$}]   Completed local ring of the complex  
      surface $S$ at the point $o$ (see Definition \ref{def:multcurve}).
  \item [{$\pi^*(C)$ }]   Total transform of a plane curve singularity $C$ 
       by a modification $\pi$ (see Definition \ref{def:modiftransf}). 
  \item [{$\psi^{\cF}_{\sigma}$}]  Toric morphism from 
       $X_{\cF}$ to $X_{\sigma}$ associated to any fan 
       $\fan$ which subdivides the cone $\sigma$ (see relation (\ref{eq:tormodif})). 
  \item [{$\psi_{L, L'}^{C}$}]   Newton modification defined by $C$ relative to the 
       cross $(L, L')$ (see Definition \ref{def:Npolseries}). 
  \item [{$\R_+$}]  Set of non-negative real numbers.
  \item [{$\Supp(f)$}]   Support of the power series  $f \in \C[[x,y]]$ 
      (see Definition \ref{def:seriesinv}). 
  \item [{$\slp_{\pi}$}]   Slope function of the toroidal pseudo-resolution 
       $\pi$ of $C$  (see Definition \ref{def:fantreetr}). 
  \item [{$\sigma_0$}]  Regular cone generated by the canonical 
      basis of the lattice $\Z^2$. 
  \item [{$\sigma_0^{L, L'}$}]   Regular cone generated by the canonical 
      basis of the lattice $N_{L, L'}$ (see Definition \ref{def:ilattice}). 
  \item [{$t^w$}]   One parameter subgroup of the algebraic torus $\cT_N$, 
      corresponding to the weight vector $w \in N$ (see the beginning of Subsection \ref{ssec:torsurf}). 
  \item [{$\cT_N$}]   Complex algebraic torus with weight lattice $N$ 
        (see formula (\ref{eq:Ntorus})). 
  \item [{$\mathrm{trop}^f$}]  Tropicalization of the non-zero power series $f \in \C[[x,y]]$ 
      (see Definition \ref{def:tropicaliz}). 
   \item [{$\mathrm{trop}^C_{L, L'}$}]   Tropical function of the curve singularity $C$ 
        relative to the cross $(L, L')$ (see Definition \ref{def:Npolseries}). 
  \item [{$\theta(\fan)$}]   Trunk of the fan $\fan$ (see Definition \ref{def:fantrunk}). 
  \item [{$\theta_{\pi}(C)$}]   Fan tree of the toroidal pseudo-resolution 
      $\pi$ of $C$ (see Definition \ref{def:fantreetr}). 
  \item [{$\Theta_L(C)$}]   Eggers-Wall tree of the plane curve singularity $C$ relative 
      to the smooth branch $L$ (see Definition \ref{def:EW} and Notations \ref{not:finotEW}). 
  \item [{$\Theta_L$}]   Universal Eggers-Wall tree (see Definition \ref{def:univEW}).
  \item [{$X_{\sigma}$}]   Affine toric variety defined by the fan consisting of the faces 
      of the cone $\sigma$ (see Definition \ref{def:afftoric}). 
  \item [{$X_{\cF}$}]  Toric variety defined by the fan $\cF$
       (see Definition \ref{def:gentoricvar}). 
  \item [{$\uplus$}]   Operation of the monoid of abstract lotuses 
       (see formula (\ref{eq:monoidlotus})).
  \item [{$\wedge$}]   Operation on the set $\Q_+^*$ 
     allowing to describe the intersection of Newton lotuses (see formula (\ref{eq:defwedge})). 
  \item [{$Z(f)$}]  Zero-locus of a holomorphic function $f$ or of a formal 
     germ $f \in \hat{\cO}_{S,o}$.
   \item  [{$\cZ_x(C)$}]   Set of Newton-Puiseux roots of a plane curve singularity 
       $C$ relative to a local coordinate system $(x,y)$ (see Definition \ref{def:charcont}).
\end{description}

\medskip
{\bf Acknowledgements.}  This research was partially supported by the French grants 
 ANR-17-CE40-0023-02 LISA  and Labex CEMPI (ANR-11-LABX-0007-01), 
 and also by the Spanish grants MTM2016-80659-P, MTM2016-76868-C2-1-P and SEV-2015-0554.
  We are grateful to Bernard Teissier for having invited us, when 
  we were his PhD students, to explore the fascinating combinatorial aspects of plane curve 
  singularities, and for everything he taught us about this world. We are grateful to  
  Ana Bel\'en de Felipe and Roberto Castellini for their remarks on lotuses. 
  We thank Patricio Almir\'on, Pierrette Cassou-Nogu\`es, Manuel Gonz\'alez Villa, 
  Carlos Guzm\'an Dur\'an,  L\^e D\~ung Tr\'ang and Camille Pl\'enat 
  for their comments on preliminary versions of this paper.  We are grateful 
  to Dale Cutkosky for several explanations concerning morphisms in the toroidal category. 
  We thank Jose Seade for the invitation to publish this work as a chapter of the 
  ``\emph{Handbook of Geometry and Topology of Singularities}'', 
  and for his remarks which were very helpful in adapting it to the spirit of this handbook. 
  We thank also the two anonymous referees for their useful recommendations and remarks.

\printindex

\medskip
\end{document}